\documentclass[letterpaper,12pt]{book}

\usepackage{booktabs}
\usepackage{tocbibind}
\usepackage[utf8]{inputenc}
\usepackage{amsmath, amssymb}
\usepackage{amsthm}
\usepackage{extarrows}
\usepackage{sectsty}
\usepackage{makeidx}
\usepackage[matrix,arrow,curve,2cell,cmtip]{xy}
\usepackage[width=14.4cm,height=25.2cm,includehead,includefoot,hcentering]{geometry}
\usepackage{fancyhdr}
\usepackage[english]{babel}
\usepackage[dvipsnames,usenames]{color}
\usepackage[colorlinks=true, linkcolor=RedOrange, citecolor=MidnightBlue, pagebackref]{hyperref}
\usepackage{breakurl}

\UseAllTwocells

\newcommand{\col}{\colon}

\newcommand{\Dis}{\caspar{Dis}}
\newcommand{\DisC}{\Dis(\C)}
\newcommand{\Con}{\caspar{Con}}
\newcommand{\ConC}{\Con(\C)}

\newcommand{\orth}{\downarrow}
\newcommand{\ra}{\rightarrow}
\newcommand{\rel}{\mathrel{\longrightarrow\hspace{-11pt}\mapstochar\hspace{11pt}}}
\newcommand{\Ra}{\Rightarrow}
\newcommand{\adj}{\dashv}
\newcommand{\tens}{\otimes}
\newcommand{\lan}{\langle}
\newcommand{\ran}{\rangle}
\newcommand{\incl}{\subseteq}

\newcommand{\eps}{\varepsilon}

\newcommand{\cat}{\mathbb}
\newcommand{\dcat}{\mc}
\newcommand{\fl}[1]{#1^\deux}
\newcommand{\caspar}{\mathsf}
\newcommand{\mc}{\mathcal}
\newcommand{\op}{{^{\mathrm{op}}}}

\newcommand{\A}{{\cat{A}}}
\newcommand{\B}{{\cat{B}}}
\newcommand{\C}{{\dcat{C}}}
\newcommand{\D}{{\dcat{D}}}
\newcommand{\G}{{\cat{G}}}
\renewcommand{\H}{{\cat{H}}}
\newcommand{\V}{{\mc{V}}}
\newcommand{\deux}{\mathsf{2}}
\newcommand{\flc}{{\C^\deux}}
\newcommand{\flv}{{\V^\deux}}
\newcommand{\flcp}{{\C^\deux_{+}}}

\newcommand{\Ens}{\caspar{Set}}
\newcommand{\Ensp}{{\Ens^*}}

\newcommand{\Ab}{\caspar{Ab}}
\newcommand{\CMon}{\caspar{AbMon}}

\newcommand{\Cat}{\caspar{Cat}}
\newcommand{\CGS}{\caspar{2}\text{-}\caspar{SGp}}
\newcommand{\CG}{\caspar{2}\text{-}\caspar{Gp}}
\newcommand{\dCMon}{\caspar{2}\text{-}\caspar{SMon}}
\newcommand{\dMon}{\caspar{2}\text{-}\caspar{Mon}}
\newcommand{\dMod}{\caspar{2}\text{-}\caspar{Mod}}
\newcommand{\Gpd}{\caspar{Gpd}}
\newcommand{\Gpdp}{\Gpd^*}
\newcommand{\Top}{\caspar{Top}}

\newcommand{\E}{\mc{E}}
\newcommand{\W}{{\mc{W}}}
\newcommand{\M}{\mc{M}}
\newcommand{\Inj}{\caspar{Inj}}
\newcommand{\zInj}{\mathrm{0}\text{-}\caspar{Inj}}
\newcommand{\Fid}{\caspar{Faith}}
\newcommand{\Fidnorm}{\caspar{NormFaith}}
\newcommand{\Cofid}{\caspar{Cofaith}}
\newcommand{\Cofidnorm}{\caspar{NormCofaith}}
\newcommand{\PlFid}{\caspar{FullFaith}}
\newcommand{\PlFidnorm}{\caspar{NormFullFaith}}
\newcommand{\PlCofid}{\caspar{FullCofaith}}
\newcommand{\PlCofidnorm}{\caspar{NormFullCofaith}}
\newcommand{\zFid}{\mathrm{0}\text{-}\caspar{Faith}}

\newcommand{\zPlFid}{\mathrm{0}\text{-}\caspar{FullFaith}}

\newcommand{\Mono}{\caspar{Mono}}
\newcommand{\Monoreg}{\caspar{RegMono}}
\newcommand{\Mononorm}{\caspar{NormMono}}
\newcommand{\zMono}{\mathrm{0}\text{-}\caspar{Mono}}
\newcommand{\KMono}{K\text{-}\Mono}

\newcommand{\Surj}{\caspar{Surj}}
\newcommand{\PlSurj}{\caspar{FullSurj}}
\newcommand{\Epi}{\caspar{Epi}}
\newcommand{\zEpi}{\mathrm{0}\text{-}\caspar{Epi}}
\newcommand{\Epireg}{\caspar{RegEpi}}
\newcommand{\Epinorm}{\caspar{NormEpi}}

\newcommand{\KEpireg}{K\text{-}\Epireg}
\newcommand{\Iso}{\caspar{Iso}}
\newcommand{\Bij}{\caspar{Bij}}
\newcommand{\Equ}{\caspar{Equ}}
\newcommand{\Cong}{\caspar{Cong}}
\newcommand{\EM}{(\E,\M)}

\DeclareMathOperator{\Ker}{\mathrm{Ker}}
\DeclareMathOperator{\KRel}{\mathrm{KRel}}
\DeclareMathOperator{\Quot}{\mathrm{Quot}}
\DeclareMathOperator{\im}{\mathrm{Im}}

\DeclareMathOperator{\Coker}{\mathrm{Coker}}
\DeclareMathOperator{\Pip}{\mathrm{Pip}}
\DeclareMathOperator{\Copip}{\mathrm{Copip}}
\DeclareMathOperator{\Root}{\mathrm{Root}}
\DeclareMathOperator{\Coroot}{\mathrm{Coroot}}
\newcommand{\Bimon}{\mathrm{Bimon}}
\newcommand{\dis}{{}_\mathrm{dis}}
\newcommand{\con}{{}_\mathrm{con}}
\newcommand{\Z}{\mathbb{Z}}
\newcommand{\spf}{(\Surj,\PlFid)}
\newcommand{\psf}{(\PlSurj,\Fid)}
\newcommand{\pl}{{}_{\mathrm{full}}}
\DeclareMathOperator{\impl}{\mathrm{Im}\pl}
\newcommand{\Hom}{\mathrm{Hom}}
\newcommand{\rf}{R[f]}
\newcommand{\vervec}[2]{{\left (\substack{{}^{\scriptstyle #1} \\ {}_{\scriptstyle #2}}\right )}}
\newcommand{\eqdef}{: =}

\newtheoremstyle{nmd}
     {\topsep}
     {\topsep}
     {\itshape}
     {}
     {\scshape}
     {.}
     {1em}
     {\thmnumber{#2}\hspace{8pt}\thmname{#1}\thmnote{ (#3)}}
\newtheoremstyle{nmdintro}
     {\topsep}
     {\topsep}
     {\itshape}
     {}
     {\scshape}
     {.}
     {1em}
     {\thmnote{#3}\hspace{8pt}\thmname{#1}}
\newtheoremstyle{dfi}
     {\topsep}
     {\topsep}
     {\upshape}
     {}
     {\scshape}
     {.}
     {1em}
     {\thmnumber{#2}\hspace{8pt}\thmname{#1}\thmnote{ (#3)}}
\newtheoremstyle{axi}
     {\topsep}
     {\topsep}
     {\itshape}
     {}
     {\scshape}
     {.}
     {1em}
     {\thmnumber{#2}\hspace{8pt}\thmname{#1}\thmnote{ #3}}
\theoremstyle{nmd}
\newtheorem{thm}{Theorem}
\newtheorem{pon}[thm]{Proposition}
\newtheorem{lemm}[thm]{Lemma}
\newtheorem{coro}[thm]{Corollary}
\theoremstyle{nmdintro}
\newtheorem{thmintro}{Theorem}
\theoremstyle{dfi}
\newtheorem{df}[thm]{Definition}
\newtheorem{rem}[thm]{Remark}
\newtheorem{ex}[thm]{Example}

\newcounter{eqnum}
\newenvironment{eqn}{\stepcounter{eqnum}\begin{equation}}{\end{equation}\ignorespacesafterend}

\newenvironment{xym}{\begin{eqn}\begin{gathered}}
                    {\end{gathered}\end{eqn}\ignorespacesafterend}
\newenvironment{xyml}{\begin{eqn}}
                    {\end{eqn}\ignorespacesafterend}

\allsectionsfont{\rmfamily\mdseries\upshape}

\makeatletter
\renewcommand{\cleardoublepage}
   {\clearpage\ifodd\c@page\else%
   \hbox{}%
   \thispagestyle{empty}%
   \newpage%
   \fi}
\makeatother

\makeindex

\begin{document}

\renewcommand{\proofname}{{\sc Proof}}

\CompileMatrices

\frontmatter
\thispagestyle{empty}
\mbox{ }\vspace{5.5cm}
\begin{center}
{\huge Abelian categories in dimension 2}
\end{center}
\mbox{}\vspace{4cm}
\begin{flushright}{\Large Mathieu Dupont}\end{flushright}
\mbox{}\vspace{7cm}

\noindent This is the English version of my PhD thesis (director: Enrico Vitale) defended on 30 June 2008 at the Université catholique de Louvain, in Louvain-la-Neuve. The original French version can be found at {\tt \url{http://edoc.bib.ucl.ac.be:81/ETD-db/collection/available/BelnUcetd-06112008-231800/}}.  Contact: {\tt \url{http://breckes.org}}.
\clearpage
\eject

\thispagestyle{empty}
\noindent {\large Abelian categories in dimension 2}
\par\bigskip\noindent {Mathieu Dupont}
\par\bigskip\noindent Dissertation pr\'esent\'ee en vue de l'obtention du grade de Docteur en Sciences, le 30 juin 2008, au D\'epartement de Math\'ematique de la
Facult\'e des Sciences de l'Universit\'e Catholique de Louvain, \`a
Louvain-la-Neuve.
\par\medskip\noindent Promoteur: Enrico Vitale
\par\medskip\noindent\hspace{-3.3mm}
\begin{tabular}{rl}
Jury: & Francis Borceux (UCL) \\
      & Dominique Bourn (Université du Littoral, Calais) \\
      & Yves Félix (UCL)\\
      & Jean Mawhin (UCL, président)\\
      & Isar Stubbe (Universiteit Antwerpen)\\
      & Enrico Vitale (UCL)
\end{tabular}
\par
\vfill
\par\noindent AMS Subject Classification (2000): 18A20 (epimorphisms, monomorphisms, special classes of morphisms, null morphisms), 18A22 (special properties of functors (faithful, full, etc.)), 18A32 (factorization of morphisms, substructures, quotient structures, congruences, amalgams), 18B40 (groupoids, semigroupoids, semigroups, groups), 18D05 (double categories, 2-categories, bicategories and generalizations), 18D10 (monoidal categories, symmetric monoidal categories, braided categories), 18E05 (preadditive, additive categories), 18E10 (exact categories, abelian categories), 18G35 (chain complexes)
\clearpage
\eject

\chapter*{Abstract}
\addcontentsline{toc}{chapter}{Abstract}
\markboth{Abstract}{Abstract}

In 2-dimensional algebra, symmetric 2-groups (symmetric monoidal groupoids in which every object has an inverse up to isomorphism) play a similar rôle to that of abelian groups in 1-dimensional algebra.  Since abelian categories are defined in the image of the category of abelian groups, a 2-dimensional version of the notion of abelian category should be a solution to the equation
\begin{equation*}
	\frac{\txt{abelian categories}}{\txt{abelian groups}}
	~=~ \frac{??}{\txt{symmetric 2-groups}}.
\end{equation*}
I give two solutions to this equation.

The first — \emph{abelian} groupoid enriched categories — is a generalisation of ordinary abelian categories.  In such a context, we can develop the theory of exact sequences and homology in a way close to homology in an abelian category: we prove several classical diagram lemmas as well as the existence of the long exact sequence of homology corresponding to an extension of chain complexes.  This generalises known results for symmetric 2-groups \cite{Bourn2002a,Rio2005a}.

The other solution — \emph{2-abelian} groupoid enriched categories, which are also abelian in the sense of the previous paragraph — mimics the specifically 2-dimensional properties of symmetric 2-groups, in particular the existence of two factorisation systems: surjective/full and faithful, and full and surjective/faithful \cite{Kasangian2000a}. Moreover, in a 2-abelian groupoid enriched category, the category of discrete objects is equivalent to that of connected objects, and these categories are abelian.  This is close to the project of Marco Grandis of “developing homotopical 
algebra as an enriched version of homological algebra” \cite{Grandis2001b}.

The examples include, in addition to symmetric 2-groups, the “2-modules” on a “2-ring”, which form a 2-abelian groupoid enriched category.  Moreover, internal groupoids, internal functors and internal natural transformations in an abelian category $\C$ (which includes as a special case Baez-Crans 2-vector spaces on a field $K$ \cite{Baez2004c}) form a 2-abelian groupoid-enriched category if and only if the axiom of choice holds in $\C$.

\chapter*{Acknowledgments}
\addcontentsline{toc}{chapter}{Acknowledgments}
\markboth{Acknowledgments}{Acknowledgments}

Je remercie tout d'abord Enrico, qui m'a proposé ce sujet de recherche, et qui a eu la patience d'abord d'attendre (de longues années) puis de lire minutieusement ce long travail.  Il est toujours de bonne humeur et toujours plus content que moi de ce que je lui montre, ce qui est rassurant.

Ensuite, je remercie les membres du jury pour avoir lu ce travail malgré sa longueur, son excès de définitions et les dangereux serpents qui s'y cachent, et pour leurs recommandations qui ont permis de rendre ce travail plus accessible.

Je remercie également mes collègues de bureau au Cyclotron. Le premier fut l'étrange insecte qui ornait le tableau de mon ancien bureau et qui périt lors de son invasion par Vincent et Julien.  Le deuxième fut Isar qui, suite à ces tragiques événements, m'hébergea dans son bureau mais surtout qui, dès mes débuts en tant qu'assistant, joua son rôle d'aîné en me guidant par ses conversations mathématiques et non-mathématiques qui se poursuivent encore.  Le troisième fut Nicolas, dont le célèbre lemme orne toujours la face cachée de mon armoire.  Enfin, apparut étonnamment un autre Mathieu Du… qui égaya cette dernière année et qui survécut à mon utilisation excessive du mot “catégorie”.

Je salue aussi tous mes collègues actuels et anciens, dont les rires chaleureux résonnent dans les couloirs de ma tête et du Cyclotron, en particulier ceux que j'ai le plus fréquentés : Julien, Mélanie, Vincent, Saji, et puis tous les jeunes, ainsi que les participants au séminaire de catégories de Louvain-la-Neuve et de Bruxelles : Claudia, Francis, Jean-Roger, Rudger, Tim, Tomas, …

Je salue également mes amis non-mathématiciens : Martin, toujours là après presque trois lustres de conversations dans la bibliothèque ou les cafés, Bertrand, Anna, les habitants et habitués (en particulier Françoise) anciens et actuels de la Ferme du Biéreau, ce lieu qui eut une grande importance pour moi ces sept dernières années, ma famille qui m'a vu de moins en moins au fur et à mesure que s'achevait ma thèse, en particulier la petite Méline (à qui, je n'en doute pas, ses parents liront pour l'endormir l'une ou l'autre page de ce livre) dont le grand sourire apparut sur Terre il y a un an, et tous les autres.

Pour terminer, je remercie pour leur fantaisie les animaux qui volent et/ou chantent dans l'espace aérien au-delà de la fenêtre de mon bureau : les insectes ciselés, la chauve-souris du soir, les dinoiseaures, et les montgolfières.

\tableofcontents

\mainmatter
\chapter*{Introduction}
\addcontentsline{toc}{chapter}{Introduction}
\markboth{Introduction}{Introduction}

\section*{Symmetric 2-groups}
\addcontentsline{toc}{section}{Symmetric 2-groups}
\markright{Symmetric 2-groups}

\subsection*{2-dimensional algebra}

Ordinary algebra (which we call here \emph{1-dimensional algebra}) is the study of algebraic structures on sets, i.e.\ of sets equipped with certain operations satisfying certain equations.  For example, we study pointed sets, monoids, commutative monoids, groups, abelian groups, rings, modules on a ring, etc.

This work takes place in the context of \emph{2-dimension algebra}, in the sense of the study of algebraic structures on groupoids.  A \emph{groupoid} is a category in which every arrow is invertible.  Groupoids generalise sets, because every set can be seen as a groupoid, with one arrow $a\ra b$ if and only if $a=b$.  Since equality is symmetric, we do get a groupoid.
In a set, there are as many objects as you want, but there are at most one arrow between two objects: there is only one degree of freedom.  On the other hand, in a groupoid, there are as many objects as you want and, between two objects $A$ and $B$, there are as many arrows as you want; groupoids are thus objects with two dimensions: the dimension of objects and the dimension of arrows between these objects.  That is why we speak of “2-dimensional algebra”.

There are two differences between algebraic structures on a groupoid and algebraic structures on a set.  Firstly, the operations must act not only on the objects, but also on the arrows of the groupoid; these are thus functors defined on the groupoid.  Secondly, these operations must satisfy some axioms, but only up to natural isomorphism.  These natural isomorphisms are themselves required to satisfy some equations. A well-known example of such a functorial operation is the tensor product of modules on a commutative ring $R$, which is associative up to natural isomorphism and has a neutral object (the ring $R$ itself) up to natural isomorphism.  In the following table, we compare 1-dimensional and 2-dimensional algebraic structures.

\begin{center}
  \begin{tabular}{@{} lcc @{}}
    \toprule
    & dimension 1 & dimension 2 \\ 
    \midrule
    basis & set & groupoid \\ 
    operations & functions & functors \\
    axioms & equations & natural iso. \\
    “2-axioms” & — & equations\\
    \bottomrule
  \end{tabular}
\end{center}

\subsection*{2-groups}

The more important 2-dimensional algebraic structure in this work is the notion of symmetric 2-group, which plays in dimension 2 the rôle that the notion of abelian group played in dimension 1.  A 2-group is a groupoid with a group structure or, in other words, a monoidal category where both the objects and the arrows are invertible.

A \emph{monoidal category} is a category $\cat{C}$ equipped with a constant $I$ (which is an object of $\cat{C}$), with a binary operation $-\tens -\col\cat{C}\times\cat{C}\ra\cat{C}$ (which is a functor) and with natural isomorphisms $A\tens I\simeq A\simeq I\tens A$ ($I$ is a unit up to isomorphism for $\tens$) and $(A\tens B)\tens C\simeq A\tens (B\tens C)$ ($\tens$ is associative up to isomorphism); these natural transformations must satisfy certain axioms (diagrams \ref{axmonassoc} and \ref{axmonunit}).  An example of monoidal category is the category of modules on a commutative ring $R$, equipped with the tensor product and the ring itself as unit $I$.

A \emph{2-group} (Definition \ref{defdgroup})
is then a monoidal groupoid $\A$ where, for each object $A$, there exists an object $A^*$ such that $A\tens A^*\simeq I$ and $A^*\tens A\simeq I$ ($A$ has an inverse $A^*$ up to isomorphism).   There is an introduction to 2-groups by John Baez and Aaron Lauda \cite{Baez2004a}.  See also \cite{Breen1992a}.

In dimension 2, there are two levels of commutativity.  Firstly, \emph{braided} 2-groups are 2-groups equipped with an isomorphism $c_{AB}\col A\tens B\ra B\tens A$ ($\tens$ is commutative up to isomorphism) which must satisfy certain axioms (diagrams \ref{axmontresi} and \ref{axmontresii}).  Secondly, there is a stronger notion of commutativity, the notion of \emph{symmetric} 2-group: this is a braided 2-group where $c_{AB}\circ c_{BA} = \mathrm{id}$.
We need symmetry to recover some properties of abelian groups.  For example, the set of morphisms between two abelian groups has itself an abelian group structure; this is not the case for the set of morphisms between two non-abelian groups.  In the same way, the groupoid of morphisms between two symmetric 2-groups has a symmetric 2-group structure; and this is not the case for 2-groups or even for braided 2-groups (see Subsection \ref{secdefbimon}).  Another example is that we need symmetry in order that every sub-2-group be the kernel of its cokernel.  For these reasons, we can consider that symmetric 2-groups play in dimension 2 the rôle that abelian groups play in dimension 1.  The bibliography of \cite{Vitale2002a} contains a list of references about symmetric 2-groups.  Symmetric 2-groups are also called \emph{symmetric $\Cat$-groups} or \emph{Picard categories}.

Here is a few examples of (symmetric) 2-groups.
\begin{enumerate}
	\item If $G$ is a group, we can see $G$ as a 2-group in the following way:
		we see $G$ as a groupoid with a unique arrow $a\ra b$ if $a=b$,
		and we take for $I$ the neutral element
		of the group and for $\tens$ the group operation.  In this situation,
		we say that we see $G$ as
		a \emph{discrete} 2-group (and we denote this 2-group by $G\dis$).  If $G$ is
		abelian, then $G\dis$ is symmetric.
	\item If $A$ is an abelian group, we can see $A$ as a 2-group in another way: this time
		the 2-group has only one object $I$ and the set of arrows $I\ra I$ is $A$; the
		composition (and the action of $\tens$ on the arrows) is the group operation;
		the identity on $I$ is the neutral element of the group. We say that we see
		$A$ as a \emph{connected} (or \emph{one-object}) 2-group
		 and we denote this 2-group by $A\con$.
\end{enumerate}

Conversely, if $\A$ is a 2-group, we can associate to it two groups corresponding to each of the dimensions of $\A$: on the one hand we have $\pi_0(\A)$, which is the group of objects of $\A$ up to isomorphism, equipped with $\tens$ as the operation and with $I$ as the neutral element.  If $\A$ is symmetric, then $\pi_0(\A)$ is abelian.  On the other hand, we have $\pi_1(\A)$, which is the group $\A(I,I)$ of arrows from $I$ to $I$ in $\A$, with the composition as the operation; this group is always abelian.  For any object $A$ of a (connected or not) 2-group $\A$, the abelian group $\A(A,A)$ is isomorphic to $\pi_1(\A)$.

The following examples are cases where one well-known group turns out to be the $\pi_0$ (or the $\pi_1$) of a 2-group.
\begin{enumerate}
	\item[3.] If $\cat{C}$ is a monoidal category, we can extract from it a 2-group by keeping only the objects which are invertible up to isomorphisms and the isomorphisms between them.  We call this 2-group the \emph{Picard 2-group}%
		\index{Picard 2-group}\index{2-group!Picard}
		of $\cat{C}$ (and we denote it by
		$\caspar{Pic}(\cat{C})$).\index{Pic(C)@$\caspar{Pic}(\cat{C})$}
		If $\cat{C}$ is a symmetric monoidal category, then
		$\caspar{Pic}(\cat{C})$ is symmetric.
		If $\cat{C}$ is the category of modules on a commutative ring $R$, then
		$\pi_0(\caspar{Pic}(\caspar{Mod}_R))$ is the ordinary Picard group of $R$
		and $\pi_1(\caspar{Pic}(\caspar{Mod}_R))$ is the group
		of invertible elements of $R$.  The Brauer group of $R$ can also be seen
		as the $\pi_0$ of a certain 2-group. See \cite{Vitale2002a} for more details.
	\item[4.] More generally, if $\C$ is a 2-category and $C$ is an objet
		of $\C$, we can define a 2-group $\caspar{Aut}(C)$ whose objects are
		the equivalences from $C$ to $C$ and whose arrows are the invertible
		2-arrows between them.
	\item[5.] If $\C$ is an abelian category and $A,B$ are objects of $\C$,
		then the extensions $0\ra A\ra E\ra B\ra 0$, with the morphisms of extensions
		whose first and last components are identities, form a symmetric 2-group 
		whose operation is the Baer sum.  The
		$\pi_0$ of this 2-group is the ordinary group of extensions $\mathrm{Ext}(A,B)$.
		See \cite{Grothendieck1968a,Vitale2003a}.
	\item[6.] If $X$ is a pointed topological space (with $x_0$ as the distinguished point),
		we can define the fundamental 2-group $\Pi_1(X)$\footnote{Usually, this 2-group is
		denoted by $\Pi_2(X)$; so the numbers of the following 2-groups are all shifted
		by 1.},
		whose objects are the paths from $x_0$ to
		$x_0$ and whose arrows are the path homotopies up to 2-homotopy. The
		$\pi_0$ of this 2-group is the ordinary fundamental group $\pi_1(X)$
		and its $\pi_1$ is the abelian group $\pi_2(X)$.  More generally, we can
		define $\Pi_k(X)\eqdef \Pi_1(\Omega^{k-1} X)$; then $\pi_0(\Pi_k(X))$ is
		$\pi_k(X)$ and $\pi_1(\Pi_k(X))$ is $\pi_{k+1}(X)$.
		As for the homotopy groups, there is a gradual increase of commutativity:
		$\Pi_1(X)$ is a 2-group, $\Pi_2(X)$ is a braided 2-group
		and $\Pi_k(X)$, for $k\geq 3$, is a symmetric 2-group \cite{Garzon2002a}.
	\item[7.] To each morphism of abelian groups $f\col A\ra B$, corresponds a symmetric
		2-group whose objects are the elements of $B$ and for which an arrow
		$b\ra b'$ is an element $a$ in $A$ such that $b = f(a)+b'$.
		The 2-group operation is given on objects by the addition of $B$.
		Its $\pi_0$ is $\Coker f$ and its $\pi_1$ is $\Ker f$.  More generally,
		to each crossed module of groups corresponds a (in general non symmetric) 2-group.
\end{enumerate}
We can also use 2-groups as coefficients for cohomology \cite{Ulbrich1984a,Carrasco2004b}; and stacks of symmetric 2-groups, often called Picard stacks, are used in algebraic geometry \cite{Deligne1973a} or in differential geometry \cite{Breen2005a}.

\subsection*{Groupoid enriched categories}

Sets equipped with some algebraic structure form a category, with the functions preserving that structure as arrows.  It is a set enriched category (or \emph{$\Ens$-category}), 
because the arrows between two sets with this structure form a set.  The groupoids equipped with a certain algebraic structure form a groupoid enriched category (or \emph{$\Gpd$-category}; Definition \ref{defgpdcat}), with the functors preserving the structure as arrows and the natural transformations compatible with this structure as arrows between arrows (called \emph{2-arrows}).

In a $\Gpd$-category, there are three levels: there are objects (in the case we are interested in, these are groupoids with a certain structure), between the objects there are arrows (which will be here in general functors preserving the structure), and between the arrows there are 2-arrows (which will be here natural transformations compatible with the structure).  The composition of arrows is a functor (which gives a second composition for 2-arrows).  A $\Gpd$-category is a 2-category in which every 2-arrow is invertible.

The main examples of $\Gpd$-categories we will talk about are:
\begin{enumerate}
	\item every $\Ens$-category (ordinary category) is a $\Gpd$-category,
		with exactly one 2-arrow $f\Ra g$ if and only if $f=g$
		(we will speak of a \emph{locally discrete} $\Gpd$-category);
	\item the $\Gpd$-category of groupoids, which we denote by $\Gpd$: the objects
		are the groupoids, the arrows are the functors and the 2-arrows are the natural
		transformations;
	\item the $\Gpd$-category of pointed groupoids, denoted by $\Gpdp$ (Definition
		\ref{deftoutpointe}): the objects are the groupoids equipped with a distinguished
		object $I$, the arrows are the functors preserving $I$ up to isomorphism and the
		2-arrows are the natural transformations whose component at $I$ is compatible
		with the isomorphisms of the functors;
	\item the $\Gpd$-category of monoidal groupoids (or \emph{2-monoids}), denoted
		by $\dMon$ (Proposition \ref{pondefdMon}): the objects are the monoidal groupoids,
		the arrows are the \emph{monoidal} functors, i.e.\ the functors preserving
		$\tens$ and $I$ up to isomorphism, and the 2-arrows are the \emph{monoidal} natural
		transformations, i.e.\ the natural transformations compatible with the isomorphisms
		of monoidal functor;
	\item the $\Gpd$-category of symmetric monoidal groupoids, denoted by $\dCMon$
		(Definition \ref{defdCMon}): the objects are the symmetric monoidal groupoids,
		the arrows are the \emph{symmetric monoidal} functors (the monoidal functors
		compatible with the symmetry) and the 2-arrows are the monoidal natural 
		transformations;
	\item the $\Gpd$-category of 2-groups, denoted by $\CG$, which is the full
		sub-$\Gpd$-category of $\dMon$ whose objects are the 2-groups;
	\item the $\Gpd$-category of symmetric 2-groups, denoted by $\CGS$, which
		is the full sub-$\Gpd$-category of $\dCMon$ whose objects are the symmetric
		2-groups.
\end{enumerate}
Other examples of $\Gpd$-categories are the $\Gpd$-category of topological spaces, with the continuous maps and the homotopies up to 2-homotopy, and the $\Gpd$-category of pointed topological spaces, with the continuous maps preserving the distinguished point and the homotopies (up to 2-homotopy) which are constant on the point.

Abelian categories (which appeared in the 1950's in a paper by David Buchsbaum \cite{Buchsbaum1955a} (with the name “exact categories”) and in the famous “Tôhoku” paper by Alexander Grothendieck \cite{Grothendieck1957a}) are categories sharing certain properties 
with the category of abelian groups, allowing to develop in them homology theory.
The categories of modules on a ring and the categories of sheaves of modules are abelian.  The goal of this work is to define a notion of {(2{-})a}belian $\Gpd$-category whose properties mimic those of the $\Gpd$-category of symmetric 2-groups, in such a way that we recover, on the one hand, the usual properties of abelian categories and, on the other hand, the specifically 2-dimensional properties of $\CGS$.

A category is abelian if it satisfies the following properties:
\begin{enumerate}
	\item it is additive (enriched in $\Ab$, with finite biproducts);
	\item it is Puppe-exact \cite{Grandis1992a} (every arrow factors as the cokernel of
		its kernel followed by the kernel of its cokernel).
\end{enumerate}
We will now review each of these properties and their 2-dimensional versions.

\section*{Additivity}
\addcontentsline{toc}{section}{Additivity}
\markright{Additivity}

\subsection*{Definition}

In dimension 1, abelian categories are in particular additive categories: they are enriched in the category $\Ab$ of abelian groups and they have all finite biproducts.

An $\Ab$-enriched category (or \emph{preadditive} category) is a category $\C$ such that, for all objects $A,B$ in $\C$, the set $\C(A,B)$ is equipped with an abelian group structure and such that, for every morphism $f$ the functions of composition with $f$ ($f\circ-$ and $-\circ f$) are group morphisms (this means that addition of arrows is distributive with respect to composition).  We know that the category of abelian groups itself is enriched in $\Ab$, as well as the categories of modules on a ring.  On the other hand, as it was recalled above, the set of morphisms between two non-abelian groups cannot be, in general, equipped we a natural group structure, and the category of groups is not additive.

In dimension 2, we will study categories enriched in $\CGS$ (which we call \emph{preadditive} $\Gpd$-categories; Definition \ref{defpreadd}): for all objects $A,B$ in $\C$, the groupoid $\C(A,B)$ is equipped with a symmetric 2-group structure, for every morphism $f$, the functors of composition with $f$ are symmetric monoidal functors and the natural transformations of the structures of $\Gpd$-category and symmetric monoidal functor are monoidal.
As we noticed above, it is necessary here to work with symmetric 2-groups in order that the $\Gpd$-category of symmetric 2-groups be itself preadditive.

In a preadditive $\Gpd$-category, the finite products coincide with the finite coproducts, if these (co)limits exist (then we speak of \emph{biproduct}) (Proposition \ref{caracbiprod}). We call \emph{additive} $\Gpd$-category a preadditive $\Gpd$-category with all finite biproducts. 

\subsection*{2-rings and 2-modules}

In dimension 1, a one-object preadditive category is in fact a ring.  Indeed, to give such a category $\C$ amounts to give an object $*$ and an abelian group of arrows $(\C(*,*),+,0)$,
with a composition $\C(*,*)\times\C(*,*)\ra\C(*,*)$ which is associative, has a neutral element $1_*$, and is such that addition is distributive with respect to composition.  Thus $R\eqdef\C(*,*)$ is a ring.

There is a notion of morphism between preadditive categories: an \emph{additive functor} is a functor preserving the abelian group structure at the level of arrows.  If $R$ is a ring seen as a one-object preadditive category, an $R$-module is nothing else than an additive functor from $R$ to $\Ab$ (the unique object $*$ is mapped to an abelian group $M$, and each scalar $r$ in $R$ is mapped to a group morphism $r\cdot -\col M\ra M$, which gives scalar multiplication).

In the same way, in dimension 2, we can call “\emph{2-ring}” a one-object preadditive $\Gpd$-category $\C$.  This is another kind of algebraic structure on a groupoid: we have, on the groupoid $\C(*,*)$, a symmetric 2-group structure (the addition, given by the enrichment in $\CGS$) and a monoidal groupoid structure (the multiplication, which is the composition of the $\Gpd$-category), and addition is distributive up to isomorphism with respect to multiplication.  The 2-rings are already well-known: they appear in Nguyen Tien Quang \cite{Quang2007a} or Mamuka Jibladze and Teimuraz Pirashvili \cite{Jibladze2007a} in the groupoidal case and it is a special case of categories with a semi-ring structure studied by
Miguel Laplaza \cite{Laplaza1972a} or Mikhail Kapranov and Vladimir Voevodsky \cite{Kapranov1994a}.

We can also define a notion of additive $\Gpd$-functor.  A \emph{$\Gpd$-functor} $F$ from a $\Gpd$-category $\C$ to a $\Gpd$-category $\D$ consists, for every object $C$ in $\C$, of an object $FC$ in $\D$ and, for all objects $C,C'$ in $\C$, of a functor $F_{C,C'}$ from the groupoid $\C(C,C')$ to the groupoid $\D(FC,FC')$ (so it acts on arrows and on 2-arrows); it must preserve composition and the identities up to isomorphism (Definition \ref{dfGpdfct}).  An \emph{additive} $\Gpd$-functor between preadditive $\Gpd$-categories is a $\Gpd$-functor such that each functor $F_{C,C'}$ between the symmetric 2-groups of arrows is symmetric monoidal (Definition \ref{dfgpdfctadd}).  Then we define a 2-module on a 2-ring $\cat{R}$ (seen as a one-object preadditive $\Gpd$-category) as an additive $\Gpd$-functor from $\cat{R}$ to $\CGS$, by analogy with the one-dimensional case.  This is thus a symmetric 2-group $\cat{M}$ equipped, for every $R$ in $\cat{R}$, with a symmetric monoidal functor $R\cdot -\col\cat{M}\ra\cat{M}$.  The axioms of modules hold, as usual, up to isomorphism. The 2-modules on a 2-ring $\cat{R}$ form an additive $\Gpd$-category $\dMod_{\cat{R}}$ which will be an example of 2-abelian $\Gpd$-category.

If $\C$ is an abelian category, the internal groupoids, internal functors and internal natural transformations in $\C$ form a $\Gpd$-category $\Gpd(\C)$, which is equivalent to the $\Gpd$-category $\flcp$ of arrows in $\C$, commutative squares between them, and chain 
complexes homotopies.  If $\C$ is $\caspar{Vect}_K$, the category of vector spaces on a field $K$, we get what John Baez and Alissa Crans \cite{Baez2004c} call “\emph{2-vector spaces on $K$}”.  We can also see $K$ as a discrete 2-ring $K\dis$, in the usual way.  Thus we can also study the $\Gpd$-category $\dMod_{K\dis}$ of 2-modules on $K\dis$.  There is actually a $\Gpd$-functor
\begin{eqn}
	(\caspar{Vect}_K)^{\deux}_+\ra \dMod_{K\dis},
\end{eqn}
which is defined in a way similar to the construction of a symmetric 2-group from a morphism of abelian groups (seventh example of symmetric 2-group above).

On the other hand, Kapranov-Voevodsky 2-vector spaces on $K$
\cite{Kapranov1994a} are 2-modules, but on the category of ordinary vector spaces on $K$, which is equipped with a semi-ring structure.  Thus they do not form a $\Gpd$-category and a fortiori they do not form an additive $\Gpd$-category.

\section*{Puppe-exactness}
\addcontentsline{toc}{section}{Puppe-exactness}
\markright{Puppe-exactness}

\subsection*{In dimension 1}

If $\C$ has all the kernels and cokernels, we can construct from an arrow $f$ in $\C$, on the one hand, the cokernel of the kernel of $f$ and, on the other hand, the kernel of the cokernel of $f$.  Then there exists a comparison arrow $w_f\col\Coker(\Ker f)\ra\Ker(\Coker f)$ which makes the following diagram commute.
\begin{xym}\xymatrix@=30pt{
	{\Ker f}\ar[r]^-{k_f}
	&A\ar[r]^f\ar[d]_{e_f}
	&B\ar[r]^-{q_f}
	&{\Coker f}
	\\ &{\Coker(\Ker f)}\ar[r]_-{w_f}
	&{\Ker(\Coker f)}\ar[u]_{m_f}
}\end{xym}
A \emph{Puppe-exact} category is a category with a zero object, all the kernels and cokernels, and where, for every arrow $f$, $w_f$ is an isomorphism.  In other words, every arrow factors as the cokernel of its kernel followed by the kernel of its cokernel.

In fact, we can deduce the additivity from the Puppe-exactness and the existence of finite products and coproducts.

\subsection*{Kernels and cokernels in dimension 2}

To speak of kernel and cokernel, we need a notion of zero arrow.  We work thus in a category enriched in pointed groupoids (or $\Gpdp$-category): this is a $\Gpd$-category $\C$ equipped, for all objects $A$, $B$, with an arrow $0\col A\ra B$ which is an absorbing (up to isomorphism) element for composition.

The \emph{kernel} of an arrow $f\col A\ra B$ in a $\Gpdp$-category consists of an objet $\Ker f$, an arrow $k_f$ and a 2-arrow $\kappa_f$, as in the following diagram, satisfying a universal property which characterises it up to equivalence (Definition \ref{defkernel}).
\begin{xym}
	\xymatrix@=40pt{K\ar[r]^k_f\rrlowertwocell_0<-9>{<2.7>\;\;\;\kappa_f} &A\ar[r]^f &B}
\end{xym}

In $\CGS$, the kernel of a symmetric monoidal functor $F\col\A\ra\B$ is, like for abelian groups, given by the objects of $\A$ mapped by $F$ to the neutral element $I$ of $B$.  The difference is that this must be the case up to isomorphism.  We can describe $\Ker F$ in the following way (see Definition \ref{dfdescker}):  
\begin{itemize}
	\item an object consists of an objet $A$ in $\A$ and an isomorphism
		$b\col FA\ra I$;
	\item an arrow from $(A,b)$ to $(A',b')$ consists of an arrow $f\col A\ra A'$
		in $\A$ such that $b'(Ff)=b$;
	\item the product of $(A,b)$ and $(A',b')$ is $(A\tens A',b'')$, where $b''$ is the composite
		$F(A\tens A')\simeq FA\tens FA'\xrightarrow{b\tens b'}I\tens I\simeq I$.
\end{itemize}
This construction is similar to the homotopy kernel of a map between pointed topological spaces.  But the homotopy kernel satisfies a weaker universal property than the one we use here.

Let us compute an example of kernel.  Let $f\col A\ra B$ be a morphism of abelian groups.  This morphism induces a symmetric monoidal functor $f\con\col A\con\ra B\con$ between the one-object symmetric 2-groups induced by the abelian groups.  By the above construction, an object of the kernel of $f\con$ consists of an object in $A\con$ (which is necessarily $I$) and of an arrow $b\col f\con(I)\ra I$ in $B\con$, i.e.\ an element $b$ in $B$.  An arrow $b\ra b'$ is an arrow $a\col I\ra I$ in $A\con$ (i.e.\ an element of $A$) such that $b'+ f(a) = b$.  The symmetric 2-group $\Ker f\con$ is thus the symmetric 2-group corresponding to the morphism of abelian groups $f$ (see the seventh example of symmetric 2-groups given above).

If we see the same morphism of abelian groups as a symmetric monoidal functor between discrete 2-groups $f\dis\col A\dis\ra B\dis$, its kernel is $(\Ker f)\dis$.  Indeed, an object of the kernel of $f\dis$ is simply an element $a$ of $A$ together with an arrow $f(a)\ra 0$ in $B\dis$, i.e.\ we have $f(a)=0$ in $B$.  It is thus an element of the kernel of $f$.  An arrow $a\ra a'$ is an arrow $a\ra a'$ in $A\dis$, i.e.\ we have $a=a'$ in $A$.

A last example, which is important for the following, is the symmetric mo\-no\-i\-dal functor $0\col 0\ra \A$ from the zero symmetric 2-group (with one object and one arrow) to any symmetric 2-group $\A$, which maps the unique object of $0$ to $I$.  An object of the kernel of this functor consists of an object of $0$ (which is necessarily $I$) together with an arrow $I\ra I$ in $\A$; this amounts to give an element $a$ of $\pi_1(\A)$.  An arrow $(I,a)\ra (I,a')$ in the kernel is an arrow $I\ra I$ in $0$ (this is necessarily $1_I$) such that $a'\circ F1_I = a$; so there is an arrow $a\ra a'$ if and only if $a=a'$ in $\pi_1\A$.  So the kernel of this functor is the symmetric 2-group $(\pi_1(\A))\dis$.  We denote it by $\Omega(\A)$.  More generally, in a $\Gpdp$-category $\C$ with zero object (object $0$ such that the groupoids $\C(X,0)$ and $\C(0,X)$ are equivalent to the groupoid $1$ for every object $X$ of $\C$), we denote by $\Omega A$ the kernel of the arrow $0\col 0\ra A$ (by analogy with the loop space of a topological space, which is constructed in a similar way with the help of the homotopy kernel).

The \emph{cokernel} of an arrow in a $\Gpdp$-category is defined by the dual universal property.  The principle to construct a quotient in dimension 2 is to add arrows between the objects we want to identify (in dimension 1, we add “equalities” between them).  In $\CGS$, the cokernel of $F\col\A\ra\B$ will have thus the same objects of $\B$ and there will be an arrow from $B$ to $B'$ if they are isomorphic up to an object of the form $FA$.  This leads to define $\Coker F$ in the following way (see Definition \ref{dfdesccoker}):
\begin{itemize}
	\item its objects are those of $\B$;
	\item an arrow $B\ra B'$ consists of an object $A$ of $\A$ and an arrow
		$g\col B\ra FA\tens B'$;
	\item two arrows $(A,g)$ and $(A',g')\col B\ra B'$ are equal if there exists
		an arrow $a\col A\ra A'$ compatible with $g$ and $g'$;
	\item the product of $B$ and $B'$ is their product in $\B$.
\end{itemize}

Let us compute an example of cokernel.  Let be again a morphism of abelian groups $f\col A\ra B$.  We consider this time the symmetric monoidal functors induced between the discrete symmetric 2-groups corresponding to the abelian groups, $f\dis\col A\dis\ra B\dis$.  An object of $\Coker f\dis$ is an object of $B\dis$, i.e.\ an element of $B$.  An arrow from $b$ to $b'$ consists of an object of $A\dis$ (i.e.\ an element of $A$) and of an arrow $b\ra fa\tens b'$ in $B\dis$, which means that we must have $b = f(a) + b'$.  Thus we find again the symmetric 2-group corresponding to the arrow $f$ from the seventh example above.

The cokernel of the symmetric monoidal functor $0\col \A\ra 0$, where $\A$ is any symmetric 2-group, is denoted by $\Sigma\A$.  Its objects are those of $0$, in other words there is only one object $I$.  An arrow $I\ra I$ is an object $A$ of $\A$, with an isomorphism $I\ra 0(A)\tens I$ in $0$ (this isomorphism is necessarily $1_I$).  Two arrows $A$ and $A'$ are equal if there exists an arrow $a\col A\ra A'$ in $\A$.  In other words, $\Sigma\A$ is $(\pi_0\A)\con$.  More generally, in a $\Gpdp$-category with zero object, we denote by $\Sigma A$ the cokernel of the arrow $0\col A\ra 0$, by analogy with the suspension of a topological space, which is constructed in a similar way using the homotopy cokernel.

\subsection*{2-Puppe-exactness}

We can now compute the cokernel of the kernel and the kernel of the cokernel of a morphism of symmetric 2-groups $F$.  But, as Enrico Vitale and Stefano Kasangian have noticed \cite{Kasangian2000a}, unlike what happens in dimension 1, there is in general no comparison functor $\Coker(\Ker F)\ra \Ker(\Coker F)$.
Thus we cannot hope that this functor be an equivalence.  There is in fact in general no equivalence between those two constructions: the cokernel of the kernel of $F$ determines a factorisation of $F$ as a surjective (up to isomorphism) functor followed by a full and faithful functor, and the kernel of the cokernel of $F$ determines a factorisation of $F$ as a full and surjective functor followed by a faithful functor.  Here appears a typical phenomenon of dimension 2: a symmetric monoidal functor has two images.

We will nevertheless be able to construct these images both as the quotient of a kernel and as the coquotient of a cokernel but, for that, we need to introduce new notions of kernel and quotient.

The new kind of kernel is called “pip” (see Definition \ref{defpepin}).  The \emph{pip} of an arrow $f\col A\ra B$ in a $\Gpdp$-category is an object $\Pip f$, with a loop $\pi_f\col0\Ra 0\col\Pip f\ra A$, universal for the property $f\pi_f=1_0$.  In $\CGS$, the pip of an arrow $F\col\A\ra\B$ is the abelian group $\pi_1(\Ker F)$ seen as a discrete symmetric 2-group: its objects are the arrows $a\col I\ra I$ in $\A$ such that $Fa=1_{FI}$ and there is one arrow $a\ra a'$ if $a=a'$ in $\A$.  For example, if $f\col A\ra B$ is morphism of abelian groups, the pip of $f\dis$ is $0$ (because the only arrow $I\ra I$ in $A\dis$ is $1_I$) and the pip of $f\con$ is $(\Ker f)\dis$, since the arrows $I\ra I$ in $A\con$ mapped by $f\dis$ to $1_I$ are the element of the kernel of $f$.

There is a notion of quotient corresponding to this new notion of kernel: the \emph{coroot} of a loop $\pi\col 0\Ra 0\col A\ra B$ is an object $\Coroot\pi$ and an arrow $r_\pi\col B\ra\Coroot\pi$ universal for the property $r_\pi\pi=1_0$ (Definition \ref{defroot}).
Dually, we define copips and roots.

Thanks to these new notions, we do have in $\CGS$ a coincidence between two constructions of the factorisations: by a colimit of a limit and a limit of a colimit.  But, unlike Puppe-exact categories where it is the dual constructions that coincide, we have here a crossed situation: on the one hand, the cokernel of the kernel coincides with the root of the copip and, on the other hand, the coroot of the pip coincides with the kernel of the cokernel.

In general, for every arrow $f$ in a $\Gpdp$-category with all kernels and cokernels (the (co)pips and (co)roots can be constructed using kernels and cokernels), there is a comparison arrow $\bar{w}_f$ between the cokernel of the kernel of $f$ and the root of the copip of $f$.
\begin{xym}\label{factcokerkerraccopeppourintro}\xymatrix@=30pt{
		{\Ker f}\ar[r]^-{k_f}
		& A\ar[r]^f\ar[d]_{\bar{e}^1_f}\rtwocell\omit\omit{_<4.75>\;\;\,\bar{\omega}_f}
		& B\rtwocell^0_0{\;\;\rho_f}
		&{}\save[]+<21pt,0pt>*{\Copip f}\restore
		\\ &{\Coker k_f}\ar[r]_-{\bar{w}_f}
		&{\Root\rho_f}\ar[u]_{\bar{m}^2_f}
	}\end{xym}
Dually, we can construct the coroot of the pip of $f$ and the kernel of the cokernel of $f$ and, again, there is a comparison arrow $w_f$ between them.
	\begin{xym}\label{factcoracpepkercokerpourintro}\xymatrix@=30pt{
		{~}\save[]+<-11pt,0pt>*{\Pip f}\restore\rtwocell^0_0{\;\,\pi_f}
		& A\ar[r]^f\ar[d]_{e^1_f}\rtwocell\omit\omit{_<4.75>\;\;\,\omega_f}
		& B\ar[r]^-{q_f}
		&{\Coker f}
		\\ &{\Coroot\pi_f}\ar[r]_-{w_f}
		&{\Ker q_f}\ar[u]_{m^2_f}
	}\end{xym}
By analogy with the definition of Puppe-exact category, we can thus define a \emph{2-Puppe-exact} $\Gpdp$-category as being a $\Gpdp$-category $\C$ with a zero object, all kernels and cokernels and where, for every arrow $f$, the comparison arrows $\bar{w}_f$ and $w_f$ are equivalences (see Proposition \ref{caractwopupex}).  If, moreover, $\C$ has all finite products and coproducts, we say that $\C$ is a \emph{2-abelian} $\Gpd$-category.

One of the main results of this work is the following theorem (Corollary \ref{2abPexadd}), whose version for abelian categories is well-known.
\begin{thmintro}[A]
	Every 2-abelian $\Gpd$-category is additive.
\end{thmintro}

The examples of 2-abelian $\Gpd$-categories known up to now are:
\begin{enumerate}
	\item the $\Gpd$-category of symmetric 2-groups (Proposition \ref{theocgsdab});
	\item the $\Gpd$-categories of 2-modules on a 2-ring and, more generally,
		the $\Gpd$-categories of (additive) $\Gpd$-functors to a 2-abelian $\Gpd$-category
		(Proposition \ref{pondmoddab});
	\item if $\C$ is an abelian category, the $\Gpd$-category $\flcp$ defined above is
		2-abelian if and only if $\C$ satisfies the axiom of choice (every epimorphism splits)
		 (Theorem \ref{caracdabflabplus});
	\item if $\D$ is a 2-abelian $\Gpd$-category, the full sub-$\Gpd$-category
		whose objects are the “discretely presentable” objects of $\D$ is also
		2-abelian (Proposition
		\ref{dispresdab}).
\end{enumerate}

On the other hand, $\Gpdp$ and $\Top^*$ are not 2-abelian $\Gpd$-categories.

\subsection*{Epimorphisms and monomorphisms in dimension 2}

A distinctive phenomenon of dimension 2 is that many notions which were unique in dimension 1 split.  The additional degree of freedom of groupoids with respect to sets allows more gradation in the definition of ordinary categorical notions.

We have already met this phenomenon in two forms: on the one hand, the unique factorisation of abelian groups splits into two non-equivalent factorisations in 2-abelian $\Gpd$-categories and, on the other hand, the unique notion of kernel splits into kernel and pip (and, consequently, the notion of cokernel (seen as the quotient corresponding to the kernel) splits into cokernel and coroot).  The same phenomenon happens also for monomorphisms and epimorphisms.

In dimension 1, these two notions are defined using the notion of injectivity for sets:
an arrow 
$f\col A\ra B$ in a category $\C$ is a \emph{monomorphism} if, for every object $X$ in $\C$, the function $f\circ - \col\C(X,A)\ra\C(X,B)$ is injective; dually, the arrow $f$ is an \emph{epimorphism} if, for every object $Y$ in $\C$, the function $-\circ f\col\C(B,Y)\ra\C(A,Y)$ is injective.

In dimension 2, there are two levels of injectivity for functors: a functor $F\col\A\ra\B$ (between groupoids or symmetric 2-groups) can be faithful (injective at the level of arrows) or fully faithful (bijective at the level of arrows\footnote{Surjectivity at the level of arrows can also be seen as a kind of injectivity at the level of objects; indeed, it implies that, if $FA\simeq FB$, then $A\simeq B$.}).  Thus we can also define two kinds of monomorphism in a $\Gpd$-category $\C$: an arrow $f\col A\ra B$ is
\begin{enumerate}
	\item \emph{faithful} if, for every object $X$ in $\C$, the functor
		$f\circ - \col\C(X,A)\ra\C(X,B)$ is faithful;
	\item \emph{fully faithful} if, for every object $X$ in $\C$, the functor
		$f\circ - \col\C(X,A)\ra\C(X,B)$ is full and faithful.
\end{enumerate}
In $\Gpd$ or $\CGS$, the faithful arrows are the faithful functors and the fully faithful arrows are the full and faithful functors.

Dually, we can define two notions of epimorphism: the \emph{cofaithful} arrows (when the composition functors $-\circ f$ are faithful) and the \emph{fully cofaithful} arrows (when they are full and faithful).  In $\Gpd$ and $\CGS$, the cofaithful functors are the surjective (up to isomorphism) functors and the fully cofaithful functors are the full and surjective (up to isomorphism) functors.  On the other hand, this is not the case in the 2-category of categories.

This splitting of the notions of monomorphism and epimorphism is connected to the other splittings already remarked: in the factorisation cokernel of the kernel/root of the copip (diagram \ref{factcokerkerraccopeppourintro}), $\bar{e}^1_f$ is cofaithful and $\bar{m}^2_f$ is fully faithful; and, in the factorisation coroot of the pip/kernel of the cokernel (diagram \ref{factcoracpepkercokerpourintro}), $e^1_f$ is fully cofaithful and $m^2_f$ is faithful.  In a 2-abelian $\Gpd$-category, every arrow factors thus, on the one hand, as a cofaithful arrow followed by a fully faithful arrow and, on the other hand, as a fully cofaithful arrow followed by a faithful arrow.

Moreover, in a 2-abelian $\Gpd$-category, every faithful arrow is the kernel of its cokernel, every fully faithful arrow is the root of its copip, every cofaithful arrow is the cokernel of its kernel, and every fully cofaithful arrow is the coroot of its pip.

In a 2-abelian $\Gpd$-category, we can also classify the properties of arrows with the help of the different kinds of kernels and cokernels, like in dimension 1: an arrow is faithful if and only if its pip is zero, an arrow is fully faithful if and only if its kernel is zero, and we also have the dual properties.

A last important property of 2-abelian $\Gpd$-categories concerning epimorphisms and monomorphisms is regularity: cofaithful and fully cofaithful arrows are stable under pullback, and the dual property also hold (Proposition \ref{gpdcatdabreg}).

We can also generalise the notion of full functor to a $\Gpd$-category (we speak then of a \emph{full} arrow; Subsection \ref{sectflechplen}) in such a way that, in $\Gpd$ and $\CGS$, the full arrows are exactly the full functors (but this is not the case in $\Cat$) and that, in a 2-abelian $\Gpd$-category, we have:
\begin{enumerate}
	\item fully faithful = full + faithful;
	\item fully cofaithful = full + cofaithful;
	\item equivalence = fully faithful + cofaithful = faithful + fully cofaithful = faithful + full + cofaithful.
\end{enumerate}

\subsection*{Discrete and connected objects}

In a $\Gpd$-category $\C$, a \emph{discrete} object is an object $D$ such that there is at most one 2-arrow between two arrows $d,d'\col X\rightrightarrows D$.
In other words, for every object $X$ in $\C$, the groupoid $\C(X,D)$ is discrete.  If $\C$ has a zero object, this is equivalent to the arrow $0\col D\ra 0$ being faithful.
In $\Gpd$, the discrete objects are the sets seen as discrete groupoids and, in $\CGS$, the discrete objects are the abelian groups seen as discrete symmetric 2-groups.  The discrete objects in $\C$ form a category (because the groupoid of arrows between two objects is in fact a set), which we denote by $\DisC$.

The notion of \emph{connected} object is defined dually.  This notion is relevant when there is a zero object.  Then an object $C$ is connected if the arrow $0\col 0\ra C$ is cofaithful.
In $\CGS$, the connected objects are the one-object symmetric 2-groups (the abelian groups seen as connected symmetric 2-groups).  The connected objects in $\C$ also form a category, denoted by $\ConC$.

An important question is: which is the link between the notion of 2-abelian $\Gpd$-category and the notion of abelian $\Ens$-category?  Every $\Ens$-category can be seen as a $\Gpd$-category (with exactly one 2-arrow $f\Ra g$ if $f=g$) and we could ask what are the 2-abelian $\Ens$-categories.  The answer is that the only 2-abelian $\Ens$-category is $1$.  Indeed, in a $\Ens$-category, every arrow is faithful (because the 2-arrows are all equal), so every object is discrete; dually, every arrow is cofaithful, so every object is connected.  But, in a 2-abelian $\Gpd$-category, an object which is both discrete and connected is necessarily zero.  Therefore, in a 2-abelian $\Ens$-category, every object is equivalent to $0$.

On the other hand, we will see that every 2-abelian $\Gpd$-category contains two equivalent abelian categories, $\DisC$ and $\ConC$ (they are equivalent as categories, but not as full sub-$\Gpd$-categories of $\C$, since their intersection is $\{0\}$).

Above we defined $\Omega A$, the kernel of the arrow $0\col 0\ra A$, which we can think of as being $A$ where the arrows have become the objects (in $\CGS$, $\Omega\A$ is the discrete 2-group whose objects are the elements of $\pi_1\A$).  This defines on objects a $\Gpd$-functor $\Omega\col\C\ra\C$.  Dually, we defined $\Sigma A$, the cokernel of the arrow $0\col A\ra 0$, which we can think of as being $A$ where the objects have become the arrows (in $\CGS$, $\Sigma\A$ is the one-object 2-group whose arrows are the elements of $\pi_0\A$).  This also defines a $\Gpd$-functor $\Sigma\col\C\ra\C$.  These two $\Gpd$-functors are adjoint to each other: $\Sigma\adj\Omega$ (Subsection \ref{sssectsigadjom}).

In a 2-abelian $\Gpd$-category, the connected objects are exactly the objects $C$ such that the counit of this adjunction at $C$ is an equivalence: $\Sigma\Omega C\simeq C$.  Dually, the discrete objects are exactly the objects $D$ such that the unit of this adjunction is an equivalence: $D\simeq\Omega\Sigma D$.  So the adjunction $\Sigma\adj\Omega$ restricts to an equivalence 
\begin{eqn}
	\DisC\simeq\ConC
\end{eqn}
(Subsection \ref{sssectequdiscconc}).
The sub-$\Gpd$-category $\DisC$ is reflective, with reflection $\pi_0\eqdef \Omega\Sigma$, and $\ConC$ is a coreflective sub-$\Gpd$-category, with coreflection $\pi_1\eqdef\Sigma\Omega$.

\begin{xym}\xymatrix@=50pt{
   	{\C}\ar@<-2mm>[r]_-{\Omega}\ar@{}[r]|-\perp
		\ar@<-2mm>[d]_{\pi_1}\ar@{}[d]|-\vdash
	&{\C}\ar@<-2mm>[l]_-{\Sigma}\ar@<-2mm>[d]_{\pi_0}\ar@{}[d]|-\dashv
    \\ {\ConC}\ar@<-2mm>[r]_-{\Omega}
		\ar@{}[r]|-\simeq\ar@<-2mm>[u]_i
	&{\DisC}\ar@<-2mm>[l]_-{\Sigma}\ar@<-2mm>[u]_i
}\end{xym}

In $\CGS$, we can take $\Ab$ with the inclusion $(-)\dis$ for $\Dis(\CGS)$ and $\Ab$ with the inclusion $(-)\con$ for $\Con(\CGS)$.

The link between the notion of 2-abelian $\Gpd$-category and the notion of abelian $\Ens$-category can now be expressed by the following theorem (Corollary \ref{discconcabel}).
\begin{thmintro}[B]
	If $\C$ is a 2-abelian $\Gpd$-category, 
	the category $\DisC\simeq\ConC$ is abelian.
\end{thmintro}

With the help of $\Omega$ and $\Sigma$ (which are (co)kernels), we can construct the (co)pip and the (co)root by using only the kernel and the cokernel.  This allows to construct the two canonical factorisations of an arrow in a $\Gpdp$-category without using pips and coroots.  Indeed, the root of the copip of an arrow $f$ coincides with the kernel of the $\pi_0$ of the cokernel of $f$ (see the following diagram).
	\begin{xym}\xymatrix@=30pt{
		{\Ker f}\ar[r]^-{k_f}
		& A\ar[r]^f\ar[d]_{\bar{e}^1_f}\rtwocell\omit\omit{_<4.75>\;\;\,\bar{\omega}_f}
		& B\ar[r]^-{q_f}
		&{\Coker f}\ar[r]^-{\eta}
		&{\pi_0\Coker f}
		\\ &{\Coker k_f}\ar[r]_-{\bar{w}_f}
		&{\Ker(\eta q_f)}\ar[u]_{\bar{m}^2_f}
	}\end{xym}
Dually, the coroot of the pip of $f$ coincides with the cokernel of the $\pi_1$ of the kernel
of $f$ (see the following diagram).
	\begin{xym}\xymatrix@=30pt{
		{\pi_1\Ker f}\ar[r]^-{\varepsilon}
		&{\Ker f}\ar[r]^-{k_f}
		& A\ar[r]^f\ar[d]_{e^1_f}\rtwocell\omit\omit{_<4.75>\;\;\,\omega_f}
		& B\ar[r]^-{q_f}
		&{\Coker f}
		\\ &&{\Coker(k_f\varepsilon)}\ar[r]_-{w_f}
		&{\Ker q_f}\ar[u]_{m^2_f}
	}\end{xym}
So we can define 2-Puppe-exact $\Gpdp$-categories (and thus also 2-abelian $\Gpd$-categories) purely in terms of kernel and cokernel, by asking (in addition to the existence of a zero object, all the kernels and cokernels) that the comparison arrows $\bar{w}_f$ and $w_f$ of the previous diagrams be equivalences.

\section*{Homology}
\addcontentsline{toc}{section}{Homology}
\markright{Homology}

\subsection*{Exact and relative exact sequences}

Another new phenomenon in dimension 2 is that the notion of exact sequence, which was unique in dimension 1, splits into two almost disjoint kinds of sequences: the exact sequences and the relative exact sequences.  The classical examples of exact sequences fall, in dimension 2, in one or the other class.  In both cases, we consider a sequence of the following shape.
\begin{xym}\label{biformgensuitex}\xymatrix@=30pt{
	\ldots\ar[r]
	&A_{n-2}\ar[r]_{a_{n-2}}\rruppertwocell<9>^0{^<-2.7>\alpha_{n-2}\;\;\;\;\;\;}
	&A_{n-1}\ar[r]^{a_{n-1}}\rrlowertwocell<-9>_0{_<2.7>\;\;\;\;\;\;\,\alpha_{n-1}}
	&A_n\ar[r]^{a_n}\rruppertwocell<9>^0{^<-2.7>\alpha_{n}\;\;}
	&A_{n+1}\ar[r]_{a_{n+1}}
	&A_{n+2}\ar[r]
	&\ldots
}\end{xym}

We say that this sequence is \emph{exact} at $A_n$ if the kernel of $a_n$ is canonically equivalent to the kernel of the cokernel of $a_{n-1}$ (Proposition \ref{defsequencex}).  Here is a few typical examples of exact sequences:
\begin{enumerate}
	\item
		the Puppe exact sequence (Proposition \ref{sequencePuppe}), well-known in homotopy
		theory, which is the 2-dimensional version of the sequence
		 $0\longrightarrow \Ker f\overset{k}\longrightarrow A
		\overset{f}\longrightarrow B\overset{q}\longrightarrow \Coker f\longrightarrow 0$
		(which, in dimension 2, is not exact at $\Ker f$ and $\Coker f$):
		\begin{multline}\stepcounter{eqnum}
			0\longrightarrow\Pip f\overset{\Omega k}\longrightarrow\Omega A\overset{\Omega f}
			\longrightarrow\Omega B\overset{d}\longrightarrow \Ker f\overset{k}\longrightarrow A\overset{f}
			\longrightarrow B\overset{q}\longrightarrow \\
			\longrightarrow \Coker f\overset{d'}\longrightarrow \Sigma A\overset{\Sigma f}
			\longrightarrow\Sigma B\overset{\Sigma q}\longrightarrow\Copip f\longrightarrow 0;
		\end{multline}
	\item the long exact sequence of homology constructed from an extension of chain
		complexes (which exists in an abelian $\Gpd$-category; see later);
	\item
		every continuous map between topological spaces $f\col X\ra Y$ induces
		a long exact sequence
		\cite{Grandis2002a}
		\begin{eqn}
			\cdots \Pi_{n+1}Kf\rightarrow\Pi_{n+1} X\rightarrow\Pi_{n+1}Y
			\rightarrow \Pi_{n}Kf\rightarrow\Pi_{n} X\rightarrow\Pi_{n}Y\cdots,
		\end{eqn}
		where $Kf$ is the homotopy kernel of $f$; this sequence is a sequence
		of symmetric 2-groups when $n\geq 3$.
\end{enumerate}
		
On the other hand, the notion of relative exact sequence (introduced in \cite{Rio2005a}) applies to a sequence of the form \ref{biformgensuitex} such that the adjacent 2-arrows
are compatible: for every $n$, the composite $\alpha_n a_{n-1}\circ a_{n+1}\alpha_{n-1}^{-1}\col 0\Ra 0$ must be equal to $1_0$.  This is the definition of chain complex used for symmetric 2-groups \cite{Takeuchi1983a,Rio2005a}.  In this case, we say that the sequence is \emph{relative exact} at $A_n$ if the relative kernel of $a_n$ and $\alpha_n$ is the kernel of the relative cokernel of $a_{n-1}$ and $\alpha_{n-2}$ (Proposition \ref{caracrelex}). The \emph{relative kernel} of $a_n$ and $\alpha_n$ consists of an arrow $k\col K\ra A_n$ and a 2-arrow $\kappa\col a_nk\Ra 0$ \emph{compatible} with $\alpha_n$, satisfying a universal property similar to that of the kernel. 

The typical examples of relative exact sequences are the extensions, i.e.\ the sequences of the form
\begin{xym}\xymatrix@=40pt{
	0\ar[r]^0\rrlowertwocell<-9>_0{_<2.7>\;\,1_0}
	&A\ar[r]_a\rruppertwocell<9>^0{^<-2.7>\alpha\,}
	&B\ar[r]_b\rrlowertwocell<-9>_0{_<2.7>\;\,1_0}
	&C\ar[r]^0
	&0
},\end{xym}
where $(a,\alpha)=\Ker b$ and $(b,\alpha)=\Coker a$.  Such a sequence is not exact, but relative exact.  Moreover, extensions compose to give relative exact sequences.  In particular, projective resolutions (for an appropriate notion of projective object) being created by joining extensions, would be relative exact sequences in dimension 2.

\subsection*{Good 2-abelian $\Gpd$-categories}

Between $\CGS$ and $\Ab$, the $\Gpd$-functors $\pi_0$ and $\pi_1$ preserve exact sequences (but not relative exact sequences) and we can characterise the different kinds of arrows of $\CGS$ by properties of their images under $\pi_0$ and $\pi_1$ \cite{Kasangian2000a}: if $F$ is a morphism of $\CGS$, then:
\begin{enumerate}
	\item $F$ is faithful if and only if $\pi_1 F$ is injective;
	\item $F$ is full if and only if $\pi_1 F$ is surjective and $\pi_0 F$ is injective;
	\item $F$ is surjective if and only if $\pi_0 F$ is surjective.
\end{enumerate}

It seems that it is not possible to deduce these properties from the definition of 2-abelian $\Gpd$-category.  Thus we define a \emph{good 2-abelian $\Gpd$-category} (see Subsection \ref{sectbondpex}) as being a 2-abelian $\Gpd$-category $\C$ where $\pi_0$ and $\pi_1$ preserve exact sequences.  We can deduce from that the characterisation of the properties of arrows of $\C$ in terms of their (co)reflection in $\DisC$ or $\ConC$.  The question of the independence of this condition from the others remains open, all known examples of 2-abelian $\Gpd$-category being good.

An important property of good 2-abelian $\Gpd$-categories is that, for every arrow $f$, the 2-arrow $\mu_f$, equal to the composite
	\begin{xym}\xymatrix@=40pt{
		{\Ker f}\ar[r]_-{k_f}\rrlowertwocell<10>^0{\;\;\;\;\,\kappa_f^{-1}}
		&A\ar[r]^f\rruppertwocell<-10>_0{_{\;\;\,\zeta_f}}
		&B\ar[r]^-{q_f}
		&{\Coker f},
	}\end{xym}
is exact (i.e.\ the sequence $\Ker f\ra 0\ra \Coker f$, with the 2-arrow $\mu_f$, is exact at $0$) (Proposition \ref{mufestexact}).
We can deduce from this result that, in a good 2-abelian $\Gpd$-category, the Puppe exact sequence and the long exact sequence of homology have an additional property:
for every $n$, the composite $\alpha_n a_{n-1}\circ a_{n+1}\alpha_{n-1}^{-1}\col
0\Ra 0$ of adjacent 2-arrows is exact.  If this is the case, we say that the sequence \ref{biformgensuitex} is \emph{perfectly exact}.

The perfectly exact sequences and the relative exact sequences are totally disjoint: in a 2-abelian $\Gpd$-category, a sequence being both perfectly exact and relative exact at each point is zero at each point (by Lemma \ref{lempourcaracfull}).

\subsection*{Abelian $\Gpd$-categories and diagram lemmas}

A way of presenting the notion of Puppe-exact category in dimension 1 is to ask that, for the exact sequence
\begin{eqn}
	0\ra K\overset{k}\longrightarrow A\overset{f}\longrightarrow B\overset{q}\longrightarrow Q\ra 0,
\end{eqn}
where $k$ is the kernel of $f$ and $q$ is the cokernel of $f$, the comparison arrow between the cokernel of $k$ and the kernel of $q$ be an isomorphism.  In dimension 2, as we have seen, this does not work any more: there is simply no comparison arrow between the cokernel of $k$ and the kernel of $q$, and the above sequence is exact only at $A$ and $B$ (to get an exact sequence at each point from $f$, we must take the Puppe sequence).

The solution we used above to define 2-Puppe-exact $\Gpdp$-categories was to compare the cokernel of the kernel with the root of the copip and the kernel of the cokernel with the coroot of the pip.  But the notion of relative exact sequence suggests another possibility: in a $\Gpdp$-category, if we start with a sequence like in the following diagram:
\begin{xym}\xymatrix@=40pt{
	0\ar[r]
	&A\ar[r]^a\ar@/^2pc/[rr]^0
		\rruppertwocell\omit{^<-2.7>\alpha}
	& B\ar[r]^b\ar@/_2pc/[rr]_0
		\rrtwocell\omit\omit{_<2.7>\beta}
	& C\ar[r]^c
	& D\ar[r]
	&0,
}\end{xym}
where $(a,\alpha)$ is the relative kernel of $b$ and $\beta$, and where $(c,\beta)$ is the relative cokernel of $b$ and $\alpha$, there always exists a comparison arrow between the cokernel of $a$ and the kernel of $c$.  We say that a $\Gpdp$-category is \emph{Puppe-exact} if, for every arrow $f$, this comparison arrow is an equivalence and if some kinds of arrows are stable under composition (Definition \ref{defpuppex}).

This is a generalisation of the usual notion of Puppe-exact category, because a $\Ens$-category is Puppe-exact in the 2-dimensional sense if and only if it is Puppe-exact in the 1-dimensional sense.  This is why we do not use here the prefix “2-”.

In a Puppe-exact $\Gpdp$-category, we recover the usual properties of exact sequences (distributed among exact sequences and relative exact sequences).  In particular, we can prove in such a $\Gpdp$-category that every relative exact sequence factors into extensions (Proposition \ref{suitrelexdecompenext}); the relative exact sequences in a Puppe-exact $\Gpdp$-category are thus exactly the sequences we get by combining extensions.

Moreover, in a Puppe-exact $\Gpdp$-category, we can prove that the two dual constructions of the \emph{homology} of a sequence (with compatible 2-arrows or not) at a point are equivalent (Proposition \ref{homology}).  This allows to study the (relative) exactness of a sequence by computing the homology objects of this sequence.

Besides, we can define a property of arrows of a $\Gpdp$-category which gives back in the 1-dimensional cases the monomorphisms and, in 2-abelian $\Gpd$-categories, the faithful arrows.  A \emph{monomorphism} is an arrow $f$ which is fully faithful for the 2-arrows compatible with the canonical 2-arrow $\zeta_f$ of its cokernel (Definition \ref{dfmonomorphs}).  In a Puppe-exact $\Gpdp$-category, we can prove that every monomorphism is the kernel of its cokernel and, dually, every epimorphism is the cokernel of its kernel.

We say that an additive $\Gpd$-category is \emph{abelian} if every monomorphism is the kernel of its cokernel, every epimorphism is the cokernel of its kernel, and if monomorphisms and fully faithful arrows are stable under pushouts and if epimorphisms and fully cofaithful arrows are stable under pullbacks (Proposition \ref{dfgpdab}).  Abelian $\Gpd$-categories are Puppe-exact.  The usual proofs that Puppe-exactness plus the existence of finite products and coproducts imply additivity and regularity do not work; the question remains open.  That is why we impose here additivity in the definition of abelian $\Gpd$-category; in the main text, the question is left undecided, and additivity is not incorporated into the definition, but it should be there eventually when the open questions will be answered.

We can generalise to abelian $\Gpd$-categories the known diagram lemmas of abelian categories, by replacing “exact sequence” by “exact sequence” or “relative exact sequence”, as the case may be.  In this way we prove the $3\times 3$ lemma (Proposition \ref{lemmtrxtr}), the short five lemma (Proposition \ref{petitlemmcinq}), the kernels lemma (Proposition \ref{lemmnoygen}), the snake lemma (which has here a long form, called “anaconda lemma”: Proposition \ref{lemmanacond}).  This allows to prove the following theorem (see Theorem \ref{lngsequencexhom} for a precise formulation) showing that an abelian $\Gpd$-category is an adequate context for homology.  It has been proved for symmetric 2-groups by Aurora Del Río, Juan Martínez-Moreno and Enrico Vitale \cite{Rio2005a}.
\begin{thmintro}[C]
	In an abelian $\Gpd$-category, every extension of chain complexes induces
	a long exact sequence of homology.
\end{thmintro}

As for Puppe-exact $\Gpdp$-categories, abelian $\Gpd$-categories generalise ordinary abelian categories: a $\Ens$-category is abelian in the 2-dimensional sense if and only if it is abelian in the 1-dimensional sense.  This fact, together with the following theorem (Corollary \ref{twoabimplab}) shows that the notion of abelian $\Gpd$-category is a generalisation both of the notion of ordinary abelian category and of the notion of 2-abelian $\Gpd$-category (let us recall that these notions are disjoint).  The question of the existence of abelian $\Gpd$-categories which are neither $\Ens$-categories nor 2-abelian remains open.

\begin{thmintro}[D]
	Every 2-abelian $\Gpd$-category is abelian.
\end{thmintro}

The following diagram represents the known implications between the different notions explained in this introduction.  Remark: additivity is included in the definition of abelian $\Gpd$-category, as it was in this introduction and unlike what is done in the main text.  Moreover, in this diagram appears the notion of Grandis homological $\Gpdp$-category, which did not appear in the introduction, but is defined in the main text (Definition \ref{defhomogran}).

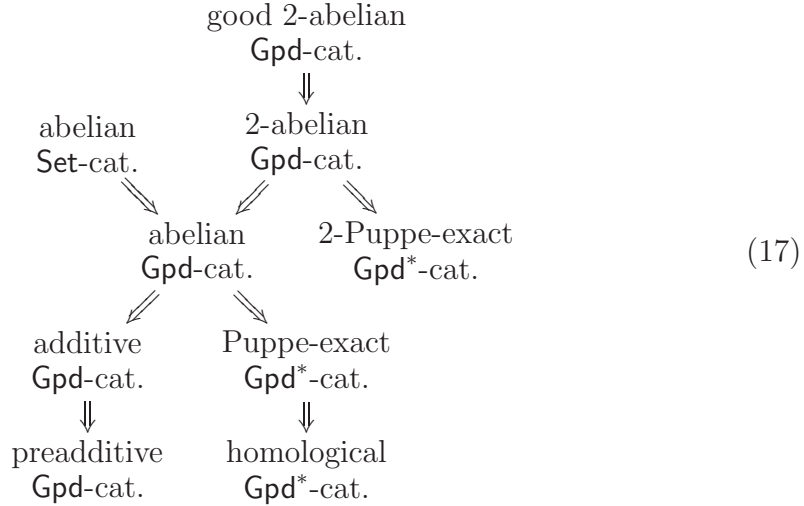
\begin{xym}\xymatrix@!@=-40pt{
	&&{\txt{good 2-abelian\\ $\Gpd$-cat.}}\ar@{=>}[d]
	\\ {\txt{abelian\\ $\Ens$-cat.}}\ar@{=>}[dr]
	&&{\txt{2-abelian\\ $\Gpd$-cat.}}\ar@{=>}[dl]\ar@{=>}[dr]
	\\ &{\txt{abelian\\ $\Gpd$-cat.}}\ar@{=>}[dl]\ar@{=>}[dr]
	&&{\txt{2-Puppe-exact\\ $\Gpdp$-cat.}}
	\\ {\txt{additive\\ $\Gpd$-cat.}}\ar@{=>}[d]
	&&{\txt{Puppe-exact\\ $\Gpdp$-cat.}}\ar@{=>}[d]
	\\ {\txt{preadditive\\ $\Gpd$-cat.}}&&{\txt{homological\\ $\Gpdp$-cat.}}
}\end{xym}

\section*{Open questions}
\addcontentsline{toc}{section}{Open questions}
\markright{Open questions}

Here are a few open questions.

\begin{enumerate}
	\item A first problem is to precise the relations between the different given
		definitions. For example, are there examples of non-good 2-abelian $\Gpd$-categories?
		Can we deduce additivity and regularity from the definition of Puppe-exact
		$\Gpd$-category and the existence of finite products and coproducts?
		What is the link between the condition of Remark \ref{remcondplusfortqueppex}
		and the notion of (good) 2-Puppe-exact $\Gpd$-category?
	\item Another question concerns Tierney theorem, according to which a category
		is abelian if and only if it is Barr exact and additive.  Can we prove an analogue
		of Tierney theorem for 2-abelian $\Gpd$-categories?  The last chapter
		gives the first steps towards an answer.
	\item A third problem is the search of additional examples.  Two kinds of
		$\Gpd$-categories could give new examples: on the one hand, stacks of symmetric
		2-groups on a (2-)site (or Picard stacks) and, on the other hand, internal
		symmetric 2-groups in a $\Gpd$-category \cite{Garzon2002a}.
		We can also mention the $\Gpd$-category of internal groupoids, internal
		anafunctors and internal ananatural transformations in an abelian category
		(see the end of Subsection \ref{secobjdiscpres}).
	\item A fourth question is: can we define (2-)-semi-abelian $\Gpd$-categories,
		which would be to (not symmetric in general) 2-groups what semi-abelian categories
		are to groups?  We can also ask a similar question with braided 2-groups.
	\item A last question concerns the subobjects of an object of a 2-abelian
		$\Gpd$-category. We know that the subobjects of an object of an abelian category
		form a modular lattice.  This is already the case for a Puppe-exact $\Ensp$-category
		\cite{Grandis1992a}.  In a 2-abelian $\Gpd$-category $\C$, we call subobjects of
		an object $A$ the faithful arrows with codomain $A$.  They form a full
		sub-$\Gpd$-category of $\C/A$, denoted by $\caspar{Sub}(A)$,
		which is in fact a category, because faithfulness implies the unicity of 2-arrows.
		In the special case where $A$ is a discrete object, $\caspar{Sub}(A)$
		is nothing else than the category $\DisC/A$ (where $\DisC$ is the abelian
		category of discrete objects in $\C$).
		
		But Aurelio Carboni \cite{Carboni1989a} introduced the notion of \emph{modular
		category}, which is a category with finite limits satisfying two conditions, one
		of them being a categorified version of the modularity law for lattices.  And he
		proves that, for every additive category $\C$ with kernel, the categories
		$\C/A$, for $A\col\C$, are modular.  Therefore, if $\C$ is a 2-abelian
		$\Gpd$-category and $A\col\DisC$, the category of subobjects of $A$ is modular.
		The question is thus: is $\caspar{Sub}(A)$ modular for any $A\col\C$?
\end{enumerate}

\section*{Warning}
\addcontentsline{toc}{section}{Warning}
\markright{Warning}

Two remarks about notations and terminology.

I use here the symbol “$\col$” to introduce variables: “let be $x\col A$” means “let $x$ be an object of $A$”, where $A$ is a set, a category, a groupoid, a 2-category, etc.; I use it also in quantifications: “for all $x\col A$” or “there exists $x\col A$”.  I reserve the use of the symbol “$\in$” to express the belonging relation between objects of a set and subsets of this set or between objects of a (2-)category and full sub(-2-)categories of this (2-)category.  This last relation is defined here in the following way: an object $C$ of a category $\C$ \emph{belongs to} a full subcategory $i\col\D\rightarrowtail\C$ (i.e.\ a fully faithful functor $i$) if there exist an object $D\col\D$ and an isomorphism $iD\simeq C$.  Therefore every full subcategory is “replete” by definition of the belonging relation.

The 2-categorical notions will be by default the weakest possible ones (i.e.\ the notions using equality only at the level of 2-arrows).  But we will use, when it is useful, strictified descriptions of these notions.  There is a consequence regarding terminology: I won't burden myself with the prefixes “pseudo-” or “bi-”.  I use the same name for the 2-dimensional notions as for the corresponding 1-dimensional notion when the 2-dimensional one reduces to the 1-dimensional one in a category seen as a locally discrete 2-category (in this case the 2-dimensional notion is a generalisation of the 1-dimensional), and I add the prefix “2-” when the 2-dimensional notion is only an analogue of the 1-dimensional one, but not a generalisation: for example, the kernel of an arrow in a category seen as a $\Gpd$-category coincides with the ordinary kernel, so I don't use the prefix “2-”; on the other hand, the $\Ens$-categories which are 2-abelian when seen as locally discrete $\Gpd$-categories are not the ordinary abelian $\Ens$-categories, so in this case I use the prefix “2-”.

\chapter{Kernel-quotient systems}

\begin{quote}{\it
	In this chapter, after basic definitions about groupoid enriched categories, the notions of
	factorisation system and kernel-quotient system are introduced.  They provide a general framework for the following chapters.  We study exactness conditions on a category equipped with a kernel-quotient system. 
}\end{quote}

\section{Factorisation systems}

\subsection{Groupoids and groupoid enriched categories}

	The 2-order of groupoids, functors and natural transformations will be denoted by $\Gpd$\index{Gpd@$\Gpd$}.  Since each arrow in a groupoid is invertible, the 2-cells in $\Gpd$ are all invertible.  Thus, for any two groupoids $\A$ and $\B$, $\Gpd(\A,\B)$ is a groupoid.  This turns $\Gpd$ into a groupoid enriched category in the sense of the following definition, i.e.\ a 2-category whose 2-arrows are invertible.  By definition, associativity holds only up to isomorphism (it's a weak 2-category or bicategory), but we will usually apply the coherence result according to which we can always assume that the associativity and neutrality natural transformations are identities. 
See \cite{Benabou1967a}, \cite{Borceux1994b},
	\cite{Kelly1974a} for more details about 2-categories.
	
	\begin{df}\label{defgpdcat}\index{Gpd-category@$\Gpd$-category}
		A \emph{groupoid enriched category} (or, for short, a \emph{$\Gpd$-category}) $\C$ is 		a collection of objects equipped with:
		\begin{enumerate}
			\item for all $A,B\col\C$, a groupoid $\C(A,B)$ (an object $f$ of this
groupoid
				will be written $f\col A\ra B$, whereas an arrow $\alpha$
				from $f$ to $g$ will be written $\alpha\col f\Ra g$);
			\item for all $A\col\C$, a functor $1\ra\C(A,A)$, i.e.
				an identity arrow $1_A\col A\ra A$;
			\item for all $A,B,C\col\C$, a functor
				$\mathrm{comp}\col \C(A,B)\times\C(B,C)\ra\C(A,C)$, which associates to each
				$A\overset{f}\ra B\overset{g}\ra C$ an arrow $g\circ f\col A\ra C$,
				and to $\alpha\col f\Ra f'$ and $\beta\col g\Ra g'$ a 2-arrow
				$\beta * \alpha \col gf\Ra g'f'$;
			\item for all $A,B,C,D\col\C$, a natural transformation
				\begin{xym}\label{diagassocgpdcat}\xymatrix@=30pt{
					{\C(A,B)\times\C(B,C)\times\C(C,D)}
						\drtwocell\omit\omit{\alpha}
						\ar[d]_{\mathrm{comp}\times 1}\ar[r]^-{1\times\mathrm{comp}\;}
					&{\C(A,B)\times\C(B,D)}\ar[d]^{\mathrm{comp}}
					\\ {\C(A,C)\times\C(C,D)}\ar[r]_-{\mathrm{comp}}
					& {\C(A,D)}
				}\end{xym}
				(thus we have, for $A\overset{f}\ra B\overset{g}\ra C\overset{h}\ra D$,
				a 2-arrow $\alpha_{hgf}\col(h\circ g)\circ f\Ra h\circ (g\circ f)$,
				natural in $f,g,h$);
			\item for all $A,B\col\C$, natural transformations:
					\begin{xym}\label{diagneutrdroitgpdcat}\xymatrix@=30pt{
						{1\times\C(A,B)}\ar[r]^-{\mathrm{id}\times 1}
							\ar@{=}[dr]\drtwocell\omit\omit{_<-2.5>{\rho}}
						&{\C(A,A)\times\C(A,B)}\ar[d]^{\mathrm{comp}}
						\\ &{\C(A,B)}
					}\end{xym}
					\begin{xym}\label{diagneutrgauchgpdcat}\xymatrix@=30pt{
						{\C(A,B)\times\C(B,B)}\ar[d]_{\mathrm{comp}}
						&{\C(A,B)\times 1}\ar[l]_-{1\times\mathrm{id}}
							\ar@{=}[dl]\dltwocell\omit\omit{^<2.5>{\lambda}}
						\\ {\C(A,B)}
					}\end{xym}
					(thus we have, for $A\overset{f}\ra B$, two 2-arrows
					$\rho_f\col f\circ 1_A\Ra f$ and $\lambda_f\col 1_B\circ f\Ra f$,
					natural in $f$).
		\end{enumerate}
		These data must make commute the following diagrams:
		\begin{xym}\label{axbicatassoc}\xymatrix@C=15pt@R=60pt{
			&(k\circ h)\circ(g\circ f)\ar[dr]^-{\alpha_{k,h,g\circ f}}
			\\ ((k\circ h)\circ g)\circ f\ar[ur]^-{\alpha_{k\circ h,g,f}}
			&& k\circ (h\circ(g\circ f))
			\\ {}\save[]+<1.5cm,0cm>*{(k\circ (h\circ g))\circ f\,}="a"
				\ar@{<-}[u]^-{\alpha_{khg} f}\restore
			&&{}\save[]+<-1.5cm,0cm>*{\,k\circ ((h\circ g)\circ f)}="b"
				\ar[u]_-{k \alpha_{hgf}}\restore
				\ar"a";"b"_-{\alpha_{k,h\circ g,f}}
		}\end{xym}
		\begin{xym}\label{axbicatunit}\xymatrix{
			(g\circ 1_B)\circ f\ar[rr]^{a_{g1_Bf}}\ar[dr]_{\rho_g f}
			&& g\circ(1_B\circ f)\ar[dl]^{g\lambda_f}
			\\ &g\circ f
		}\end{xym}
	\end{df}
	
	If we have $\A\overset{F}\ra\B\overset{G}\ra\cat{C}
	\overset{H}\ra\cat{D}$ in the $\Gpd$-category $\Gpd$ of groupoids,
	what are the functors $(H\circ G)\circ F$ and 
	$H\circ (G\circ F)$ ?  On the one hand, for each $A\col\A$, we have
	\begin{eqn}
		((H\circ G)\circ F) (A)\equiv H(G(F(A)))\equiv (H\circ (G\circ F))(A);
	\end{eqn}
	on the other hand, for each $a\col A_0\ra A_1$ in $\A$, we have
	\begin{eqn}
		((H\circ G)\circ F) (a)\equiv H(G(F(a)))\equiv (H\circ (G\circ F))(a).
	\end{eqn}
	Thus we have $(H\circ G)\circ F\equiv H\circ (G\circ F)$, and we take for $\alpha$
	the identity.  In the same way, we take for $\rho$ and $\lambda$
	the identity.
	
	The $\Gpd$-category $\Gpd$ is thus
	\emph{strictly described}\index{Gpd-category@$\Gpd$-category!strictly described},
	in the sense that
	we have $(k\circ h)\circ g\equiv k\circ (h\circ g)$, $f\circ 1_A\equiv f$ and
	$1_B\circ f\equiv f$, and that we have taken for $\alpha$, $\rho$ and $\lambda$
	the identity.
	
	\begin{df}\label{dfGpdfct}\index{Gpd-functor@$\Gpd$-functor}
		Let $\C$, $\D$ be two $\Gpd$-categories.
		A \emph{$\Gpd$-functor}\footnote{or pseudo-functor, or 
		homomorphism of bicategories} $F$ from $\C$ to $\D$ consists of the following data:
		\begin{enumerate}
			\item for all $A\col\C$, an object $FA\col\D$,
			\item for all $A,B\col\C$, a functor $F\col\C(A,B)\ra\D(FA,FB)$,
			\item for all $A\overset{f}\ra B\overset{g}\ra C$, a natural transformation
				$\varphi^F_{g,f}\col Fg\circ Ff\Ra F(g\circ f)$
				(we omit the superscript and subscripts
				when they can be implied from the context),
			\item for all $A\col\C$, a 2-arrow $\varphi_A\col 1_{FA}\Ra F1_A$,
		\end{enumerate}
		such that the following diagrams commute.
		\begin{xym}\label{axgpdfoncass}\xymatrix@=50pt{
			{}\save[]+<1.5cm,0cm>*{(Fh\circ Fg)\circ Ff\,}="m"
				\ar[d]_-{\varphi (Ff)}\restore
			&&{}\save[]+<-1.5cm,0cm>*{\,Fh\circ(Fg\circ Ff)}="n"
				\ar[d]^-{(Fh)\varphi}\restore
				\ar "m";"n" ^-{\alpha}
			\\ F(h\circ g)\circ Ff
			&& Fh\circ F(g\circ f)
			\\ {}\save[]+<1.5cm,0cm>*{F((h\circ g)\circ f)\,}="a"
				\ar@{<-}[u]^-{\varphi}\restore
			&&{}\save[]+<-1.5cm,0cm>*{\,F(h\circ (g\circ f))}="b"
				\ar@{<-}[u]_-{\varphi}\restore
				\ar "a";"b"_-{F\alpha}
		}\end{xym}
		\begin{xym}\label{axgpdfoncunite}\xymatrix@=40pt{
			Ff\circ 1_{FA}\ar[r]^-{\rho_{FA}}\ar[d]_{(Ff)\varphi_A}
			&Ff
			\\ Ff\circ F1_A\ar[r]_-{\varphi_{f,1_A}}
			&F(f\circ 1_A)\ar[u]_{F\rho_f}
		}\;\;\;
		\xymatrix@=40pt{
			1_{FB}\circ Ff\ar[r]^-{\lambda_{Ff}}\ar[d]_{\varphi_B (Ff)}
			&Ff
			\\ F1_B\circ Ff\ar[r]_-{\varphi_{1_B,f}}
			&F(1_B\circ f)\ar[u]_{F\lambda_f}
		}\end{xym}
	\end{df}
	
	\begin{df}\index{Gpd-natural transformation@$\Gpd$-natural transformation}
		Let $F,G\col\C\ra\D$ be two $\Gpd$-functors between $\Gpd$-categories.
		A \emph{$\Gpd$-natural transformation} $\mu$ from $F$ to $G$ consists of
		the following data:
		\begin{enumerate}
			\item for all $A\col\C$, an arrow $\mu_A\col FA\ra GA$ in $\D$;
			\item for all $A\overset{f}\ra B$ in $\C$, a 2-arrow
				$\mu_f\col\mu_B\circ Ff\Ra Gf\circ\mu_A$, natural in $f$,
		\end{enumerate}
		such that the following equations hold.
		\begin{xyml}\begin{gathered}\xymatrix@=40pt{
			FA\drtwocell\omit\omit{_{\;\;\,\mu_f}}\ar[r]^-{Ff}\ar[d]_{\mu_A}
			&FB\ar[r]^-{Fg}\ar[d]^{\mu_B}\drtwocell\omit\omit{\;\;\,\mu_g}
			&FC\ar[d]^{\mu_C}
			\\ GA\ar[r]_-{Gf}\rrlowertwocell<-9>_{G(gf)}{_<2.7>{\;\;\;\;\,\varphi^G_{g,f}}}
			&GB\ar[r]_{Gg}
			&GC
		}\end{gathered}\;\;=\;\;\begin{gathered}\xymatrix@=40pt{
			FA\ar[r]^-{Ff}\ar[d]_{\mu_A}
				\rrlowertwocell<-9>_{F(gf)}{_<2.7>{\;\;\;\;\,\varphi^F_{g,f}}}
			&FB\ar[r]^-{Fg}
			&FC\ar[d]^{\mu_C}
			\\ GA\rrlowertwocell<-9>_{G(gf)}{_<-0.8>{\;\;\;\,\mu_{gf}}}
			&&GC
		}\end{gathered}\end{xyml}
		\begin{xyml}\begin{gathered}\xymatrix@R=40pt@C=60pt{
			FA\ar@{=}[r]\ar[d]_{\mu_A}
			&FA\ar[d]^{\mu_A}
			\\ GA\ar@{=}[r]\rlowertwocell<-9>_{G1_A}{_<2.3>{\;\;\;\varphi^G_A}}
			&GA
		}\end{gathered}\;\;=\;\;\begin{gathered}\xymatrix@R=40pt@C=60pt{
			FA\ar@{=}[r]\ar[d]_{\mu_A}\rlowertwocell<-9>_{F1_A}{_<2.3>{\;\;\;\varphi^F_A}}
			&FA\ar[d]^{\mu_A}
			\\ GA\rlowertwocell<-9>_{G1_A}{_<-0.8>{\;\;\;\;\mu_{1_A}}}
			&GA
		}\end{gathered}\end{xyml}
	\end{df}
	
	\begin{df}\index{Gpd-modification@$\Gpd$-modification}\index{modification}
		Let $\mu,\nu\col F\Ra G\col\C\ra\D$ be two $\Gpd$-natural transformations.
		A \emph{$\Gpd$-modification} $\aleph\col\mu\Rrightarrow\nu$ associates to
		each $A\col\C$, a 2-arrow $\aleph_A\col\mu_A\Ra\nu_A$ such that, for all
		$A\overset{f}\ra B$, the following equation holds.
		\begin{xyml}\begin{gathered}\xymatrix@=40pt{
			FA\ar[r]^{Ff}\dtwocell_{\nu_A\;\;}^{\;\;\;\mu_A}{\aleph_A}
				\drtwocell\omit\omit{\;\;\;\mu_f}
			&FB\ar[d]^{\mu_B}
			\\ GA\ar[r]_{Gf}
			&GB
		}\end{gathered}\;\;=\;\;\begin{gathered}\xymatrix@=40pt{
			FA\ar[r]^{Ff}
				\drtwocell\omit\omit{\nu_f\;\;\;\;\;\;\;\;\;\;\;\;}\ar[d]_{\nu_A}
			&FB\dtwocell_{\nu_B\;\;}^{\;\;\;\mu_B}{\aleph_B}
			\\ GA\ar[r]_{Gf}
			&GB
		}\end{gathered}\end{xyml}
	\end{df}
	
	\begin{pon}
		Let $\C$, $\D$ be two $\Gpd$-categories.  The $\Gpd$-functors from $\C$ to $\D$,
		$\Gpd$-natural transformations between them and $\Gpd$-modifications
		between them form a $\Gpd$-category $[\C,\D]$, which is strictly described
		if $\D$ is.
	\end{pon}
	
	For each object $A$ in a $\Gpd$-category $\C$,
	there is a \emph{representable $\Gpd$-functor}\index{representable $\Gpd$-functor}%
	\index{Gpd-functor@$\Gpd$-functor!representable} $\C(A,-)\col\C\ra\Gpd$, which maps
	\begin{itemize}
		\item an object $B$ to the groupoid $\C(A,B)$;
		\item an arrow $b\col B\ra B'$ to the functor $\C(A,b)=b\circ -
			\col\C(A,B)\ra\C(A,B')$, which maps $f\col A\ra B$ to
			$bf$ and $\varphi\col f\Ra f'$ to $b\varphi\col bf\Ra bf'$;
		\item a 2-arrow $\beta\col b\Ra b'\col B\ra B'$ to the natural transformation
			$\C(A,\beta)=\beta*-\col b\circ -\Ra b'\circ -\col\C(A,B)\ra\C(A,B')$
			which is defined by $\beta f\col bf\Ra b'f$ at the point $f\col A\ra B$.
	\end{itemize}
	Dually, we define
	a representable $\Gpd$-functor $\C(-,B)\col\C\op\ra\Gpd$ for each object $B$ in $\C$.
	
	\begin{df}
		Let $\C$ be a $\Gpd$-category.  The \emph{Yoneda $\Gpd$-functor}%
		\index{Yoneda $\Gpd$-functor}\index{Gpd-functor@$\Gpd$-functor!Yoneda}
		$Y_{\C}\col\C\ra[\C\op,\Gpd]$ is the $\Gpd$-functor which maps an object $A\col\C$
		to the representable $\Gpd$-functor $\C(-,A)\col\C\op\ra\Gpd$,
		an arrow $a\col A\ra A'$ to the representable $\Gpd$-natural transformation
		$\C(-,a)\col\C(-,A)\ra \C(-,A')$ and a 2-arrow $\alpha\col a\Ra a'$
		to the $\Gpd$-modification $\C(-,\alpha)\col\C(-,a)\Ra\C(-,a')$.
	\end{df}
	
	\begin{pon}
		For any $\Gpd$-category $\C$, $Y_{\C}$ is fully faithful
		(in the sense that the functors $(Y_\C)_{A,A'}$ are equivalences).
	\end{pon}
	
	A well-known consequence (for example, see \cite[Theorem 1.5.15]{Leinster2004a})
	is that any $\Gpd$-category $\C$ is equivalent to its full image
	in $[\C\op,\Gpd]$ which, as we saw before, is strictly described.  Thus we can
	always assume that a $\Gpd$-category is strictly described.
	We will do that in the remaining of this text, except at a few places.
	
	Let us recall the definition of adjunctions (usually called
	biadjunctions \cite{Street1980a}) between $\Gpd$-functors.
	
	\begin{df}\index{adjunction}
		Let $\C$ and $\D$ be $\Gpd$-categories and $F\col\C\ra\D$ and $G\col\D\ra\C$
		be $\Gpd$-functors.  We call $F$ a \emph{left adjoint of $G$}
		and $G$ a \emph{right adjoint of $F$} (we denote this by $F\adj G$)
		if, for all $C\col\C$ and $D\col\D$, there is an equivalence
		\begin{eqn}
			\Phi_{C,D}\col \D(FC,D)\ra\C(C,GD)
		\end{eqn}
		$\Gpd$-natural in $C$ and $D$.
		The $\Gpd$-natural transformation $\eta\col 1_{\C}\Ra GF$
		whose component at an object $C\col\C$ is
		\begin{eqn}
			\eta_C\eqdef \Phi_{C,FC}(1_{FC})
		\end{eqn}
		is called the \emph{unit} of the adjunction. The $\Gpd$-natural transformation
		$\varepsilon\col FG\Ra 1_{\D}$	whose component at an object $D\col\D$ is
		\begin{eqn}
			\varepsilon_D\eqdef \Phi_{GD,D}^{-1}(1_{GD})
		\end{eqn}
		is called the \emph{counit} of the adjunction.
	\end{df}

	For such an adjunction, there exist modifications (the triangle identities)
	\begin{eqn}
		(F\xLongrightarrow{F\eta} FGF\xLongrightarrow{\varepsilon F}F)\;\Rrightarrow\;
		(F\xLongrightarrow{1_F} F)
	\end{eqn}
	and
	\begin{eqn}
		(G\xLongrightarrow{\eta G} GFG\xLongrightarrow{G\varepsilon}G)\;\Rrightarrow\;
		(G\xLongrightarrow{1_G} G).
	\end{eqn}

The following proposition is proved in \cite{Betti1999a} and can be found in \cite{Lambek1986a} (Lemma 4.3 of Part 0) for 1-categories.
The name “idempotent adjunction” has been suggested by Peter Johnstone on the \texttt{categories} list (21 November 2001);
it is justified by the fact that the monad and the comonad induced by the adjunction are idempotent.  The name “exact adjunction” is also used \cite{Grandis2001b}.

\begin{pon}\label{caracdefidemp}
	Let $F\adj G\col\D\ra\C$ be an adjunction of $\Gpd$-functors
	(with counit $\varepsilon$ and unit $\eta$).
 	The following conditions are equivalent. When they hold, we call
	the adjunction $F\adj G$ \emph{idempotent}\index{adjunction!idempotent}%
	\index{idempotent adjunction}.
	\begin{enumerate}
		\item $G\varepsilon $ is an equivalence;
		\item $\varepsilon F$ is an equivalence;
		\item $F\eta$ is an equivalence;
		\item $\eta G$ is an equivalence.
	\end{enumerate}
\end{pon}

	It is possible to define a notion of limit\index{limits} in groupoid enriched categories. 
	All limits will be here “pseudo-bilimits” \cite{Borceux1994b};
	they are unique up to equivalence.
	I won't give a general definition of limit, but only the definition of each special
	case of limit that will be used.  One of the most useful cases is the pullback.

	\begin{df}\label{defprodfib}\index{pullback}
		Let $\C$ be a $\Gpd$-category.
		Let us consider the solid part of the following diagram in $\C$.
		\begin{xym}\xymatrix@=20pt{
			P\ar@{-->}[rr]^{p_1}\ar@{-->}[dd]_{p_2}\ddrrtwocell\omit\omit{_{\pi}}
			&&A\ar[dd]^f
			\\ &&&{}
			\\ B\ar[rr]_g
			&&C
		}\end{xym}
		The dashed part (together with $\pi$) is a \emph{pullback 
		of $f$ and $g$} if
		\begin{enumerate}
			\item for all $X\col\C$, $a\col X\ra A$, $b\col X\ra B$, and 
				$\gamma\col fa\Ra gb$,
				there exist an arrow $(a,\gamma,b)\col X\ra P$, $\pi_1\col a\Ra p_1(a,\gamma,b)$
				and $\pi_2\col b\Ra p_2(a,\gamma,b)$ such that
				\begin{xyml}\begin{gathered}\xymatrix@=20pt{
					X\ar[dr]|{(a,\gamma,b)}\ar@/^1pc/[drrr]^a\ar@/_1pc/[dddr]_b 
					&&{}\ar@{}[dl]|(0.6){\pi_1\!\dir{=>}}
					\\ &P\ar[rr]^{p_1}\ar[dd]_{p_2}\ddrrtwocell\omit\omit{_{\pi}}
						\ar@{}[dl]|{\pi_2^{-1}\!\!\!\dir{=>}}
					&&A\ar[dd]^f
					\\ {}
					\\ &B\ar[rr]_g
					&&C
				}\end{gathered}\;\;=\;\;\gamma;
				\end{xyml}
			\item for all $X\col\C$, $x,x'\col X\ra P$, $\chi_1\col p_1x\Ra p_1x'$
				and $\chi_2\col p_2x\Ra p_2x'$ such that
				\begin{xym}\xymatrix@!{
					&&A\ar[dr]^f
					&&&P\ar[dr]^{p_1}
					\\ &P\ar[ur]^{p_1}\ar[dr]^{p_2}\rrtwocell\omit\omit{\pi}
					&&C
					&X\ar[ur]^{x}\ar[dr]_{x'}\rrtwocell\omit\omit{\;\;\chi_1}
					&&A\ar[dr]^f
					\\ X\ar[ur]^{x}\ar[dr]_{x'}\rrtwocell\omit\omit{\;\;\chi_2}
					&&B\ar[ur]_g \ar@{}[rrr]|-{}="a" \save "a"*{=}\restore
					&&&P\ar[ur]_{p_1}\ar[dr]_{p_2}\rrtwocell\omit\omit{\pi}
					&&C,
					\\ &P\ar[ur]_{p_2}
					&&&&&B\ar[ur]_g
				}\end{xym}
				there exists a unique $\chi\col x\Ra x'$ such that $p_1\chi=\chi_1$
				and $p_2\chi=\chi_2$.
		\end{enumerate}
	\end{df}
	
	In $\Gpd$, the pullback\index{pullback!in $\Gpd$}
	of $\A\overset{F}\longrightarrow\cat{C}
	\overset{G}\longleftarrow\B$ can be described in the following way.
	\begin{itemize}
		\item {\it Objects.} An object consists of $(A,\gamma,B)$ where $A\col\A$, $B\col\B$
			and $\gamma\col FA\ra GB$ in $\cat{C}$.
		\item {\it Arrows.} An arrow $(A,\gamma,B)\ra(A',\gamma',B')$ consists of
				$(a,b)$, where $a\col A\ra A'\col\A$ and $b\col B\ra B'\col\B$ are such that
				 $Gb\circ\gamma=\gamma'\circ Fa$.  The identities and composition are
				 defined pointwise.
		\item {\it Equality.}  Two arrows $(a,b)$ and $(a',b')$ are equal if
			$a=a'$ in $\A$ and $b=b'$ in $\B$.
	\end{itemize}
	A construction of the pullback and the pushout for symmetric 2-groups is given
	in \cite{Bourn2002a}.

\subsection{Factorisation systems on a $\Gpd$-category}

	If $\C$ is a $\Gpd$-category, we will denote by $\flc$\index{C2@$\flc$}
	the $\Gpd$-category of arrows in $\C$, described in the following way.
	\begin{itemize}
		\item {\it Objects. } An object of $\flc$ consists of two objects
				$A, B\col\C$ and an arrow $A\overset{f}\ra B$.
		\item {\it Arrows. } A morphism $(A,f,B)\ra(A',f',B')$ consists of
			two arrows $A\overset{a}\ra A'$ and $B\overset{b}\ra B'$ and a
			2-arrow $\psi\col bf\Ra f'a$.
		\item {\it 2-arrows.} A 2-arrow $(a,\psi,b)\Ra (a',\psi',b')
			\col(A,f,B)\ra(A',f',B')$ consists of two 2-arrows $\alpha\col a\Ra a'$
			and $\beta\col b\Ra b'$ such that $f'\alpha\circ\psi=\psi'\circ\beta f$.
	\end{itemize}
	We denote by $\partial_0\col\flc\ra \C$\index{d0@$\partial_0$} the domain $\Gpd$-functor
	 (which maps $(A,f,B)$ to $A$, $(a,\psi,b)$ to $a$ and $(\alpha,\beta)$ to
	$\alpha$), and $\partial_1\col\flc \ra\C$ the codomain $\Gpd$-functor.
	We denote by $\delta\col\partial_0\Ra\partial_1$\index{d1@$\partial_1$} the
	$\Gpd$-natural transformation defined by
	$\delta_{(A,f,B)}\eqdef f\col \partial_0(A,f,B)\ra\partial_1(A,f,B)$.
	We denote by $\Equ$ the full sub-$\Gpd$-category of $\flc$ whose objects
	are the equivalences in $\C$.

	We will take as the definition of factorisation system the following one, which is a variant of the notion of Eilenberg-Moore factorisation system of Mareli Korostenski and Walter Tholen \cite{Korostenski1993a}.
	Their proof that this definition is equivalent to the usual definition in terms of orthogonality (see \cite{Johnstone1981a} or \cite{Kasangian2000a} for the 2-dimensional case) will be generalized to dimension 2.

	\begin{df}\index{factorisation system}\index{system!factorisation|see{factorisation system}}
		Let $\C$ be a $\Gpd$-category.
		A \emph{factorisation system} on $\C$ consists of $(\E,\M,\im,e,m,\varphi)$, where
		\begin{itemize}
			\item $\E$ and $\M$ are full sub-$\Gpd$-categories of $\flc$;
			\item $\im$ is a $\Gpd$-functor, $e$ and $m$ are 
				$\Gpd$-natural transformations, and
				\begin{xyml}\label{eqfact}\varphi\col\begin{gathered}\xymatrix@=40pt{
					{\flc}\rtwocell^{\partial_0}_{\partial_1}{\delta}
					&{\C}
				}\end{gathered}\Rrightarrow\begin{gathered}\xymatrix@=40pt{
					{\flc}\ar[r]|-{\im}\ruppertwocell<8>^{\partial_0}{_<-2.2>e}
						\rlowertwocell<-8>_{\partial_1}{_<2.2>\,\;m}
					&{\C}
				}\end{gathered}.\end{xyml}
				Thus every arrow $f\col\flc$ factors in the following way.
				\begin{xym}\xymatrix@C=20pt@R=30pt{
					A\ar[rr]^f\ar[dr]_-{e_f}\rrtwocell\omit\omit{_<3.3>\;\;\varphi_f}
					&&B
					\\ &{\im f}\ar[ur]_-{m_f}
				}\end{xym}
		\end{itemize}
		These data must satisfy the following conditions: for all $f\col\flc$,
		\begin{enumerate}
			\item $e_f\in\E$ and $m_f\in\M$;
			\item if $f\in\E$, then $m_f\in\Equ$;
			\item if $f\in\M$, then $e_f\in\Equ$;
			\item if $f\in\Equ$, then $e_f, m_f\in\Equ$.
		\end{enumerate}
	\end{df}
	
	By the following proposition, one can replace condition 4 by
	\begin{enumerate}
		\item[4'.] $\Equ\incl\E\cap\M$.
	\end{enumerate}
	
	\begin{pon}\label{propfacdefem}
		If $(\E,\M,\im,e,m,\varphi)$ is a factorisation system on $\C$, 
		\begin{align}\stepcounter{eqnum}
			\E &=\{f\col\flc\,|\,m_f\in\Equ\}\text{ and}\\ \stepcounter{eqnum}
			\M &=\{f\col\flc\,|\,e_f\in\Equ\}.
		\end{align}
	\end{pon}
	
		\begin{proof}
			By condition 2, if $f\in\E$, then $m_f\in\Equ$.  Conversely, if
			$m_f$ is an equivalence, then $(1_A,\varphi_f^{-1},m_f)\col e_f\ra f$
			is an equivalence in $\flc$ and so $f\simeq e_f\in\E$, by
			condition 1.  The proof for $\M$ is dual.
		\end{proof}
	
	\begin{pon}\label{caracsimpsyf}
		To give a factorisation system $(\E,\M,\im,e,m,\varphi)$ on $\C$ is equivalent to
		give $(\im,e,m,\varphi)$ as in diagram \ref{eqfact} such that,
		if we set
		\begin{align}\stepcounter{eqnum}
			\E &\eqdef\{f\col\flc\,|\,m_f\in\Equ\}\text{ and}\\ \stepcounter{eqnum}
			\M &\eqdef\{f\col\flc\,|\,e_f\in\Equ\},
		\end{align}
		we have:
		\begin{enumerate}
			\item for all $f\col\flc$, $e_f\in\E$ and $m_f\in\M$;
			\item $\Equ\incl\E\cap\M$.
		\end{enumerate}
	\end{pon}
	
		\begin{proof}
			If we start with a factorisation system,
			it satisfies conditions 1 and 2, by Proposition \ref{propfacdefem}.
			
			Conversely, if we have $(\im,e,m,\varphi)$ satisfying conditions 1 and 2,
			then, by defining $\E$ and $\M$ as above, we get a factorisation system:
			condition 1 of factorisation system
			is condition 1 of this proposition, conditions 2 and 3 follow from the
			definition of $\E$ and $\M$, and condition 4' of factorisation system
			is condition 2 of this proposition.
		\end{proof}
		
	The factorisation systems on $\C$ form an order, denoted by
	$\caspar{Fac}(\C)$\index{Fac(C)@$\caspar{Fac}(\C)$},
	and defined in the following way.
	\begin{itemize}
		\item {\it Objects. } These are the factorisation systems on $\C$.
		\item {\it Order. } We will say that $(\E,\M,\im,e,m,\varphi)\leq
			(\E',\M',\im',e',m',\varphi')$ if the following conditions hold:
			\begin{enumerate}
				\item $\E\incl\E'$;
				\item $\M\supseteq \M'$;
				\item there exists a $\Gpd$-natural transformation $d\col\im\ra\im'$
					and modifications $\varepsilon\col de\Rrightarrow e'$
					and $\mu\col m\Rrightarrow m'd$ such that, for all $f\col\flc$,
					\begin{xyml}\label{eqordfac}\begin{gathered}\xymatrix@C=30pt@R=30pt{
						&{\im f}\ar[dr]^-{m_f}\ar[dd]_(0,4){d_f}
						&{}\ar@{}[dl]^(0.86){\mu_f\!\!\dir{=>}}
						\\A\ar[dr]_-{e'_f}\ar[ur]^{e_f}
						&{}\ar@{}[dl]_(0.3){\varepsilon_f\!\!\dir{=>}} &B
						\\ {}&{\im' f}\ar[ur]_-{m'_f}
					}\end{gathered}\;\;=\;\;\begin{gathered}\xymatrix@C=30pt@R=30pt{
						&{\im f}\ar[dr]^-{m_f}
						\\A\ar[rr]|f\ar[dr]_-{e'_f}
							\rrtwocell\omit\omit{_<3.8>\;\;\varphi'_f}\ar[ur]^-{e_f}
							\rrtwocell\omit\omit{_<-3.8>\;\;\;\;\,\varphi^{-1}_f}
						&&B
						\\ &{\im' f}\ar[ur]_-{m'_f}
					}\end{gathered}.\end{xyml}
			\end{enumerate}
	\end{itemize}

	Let us now define orthogonality in dimension 2 (see \cite{Street1982b},
	\cite{Kasangian2000a} and \cite{Dupont2003a}).
	Let us first remark, by using the construction of pullbacks in $\Gpd$
	given above, that if $\C$ is a $\Gpd$-category and $A\overset{f}\ra B$,
	$A'\overset{f'}\ra B'$ are arrows in $\C$,
	the following square is a pullback in $\Gpd$.
	\begin{xym}\label{flcprodfib}\xymatrix@=40pt{
		{\flc(f,f')}\ar[r]^-{\mathrm{proj}}\ar[d]_-{\mathrm{proj}}
			\drtwocell\omit\omit{^{}}
		&{\C(A,A')}\ar[d]^{f'\circ-}
		\\{\C(B,B')}\ar[r]_-{-\circ f}
		&{\C(A,B')}
	}\end{xym}
	We will denote by $\lan f,f'\ran\col\C(B,A')\ra\flc(f,f')$ the functor
	which maps 
	\begin{itemize}
		\item $c\col B\ra A'$ to $(cf,1_{f'cf},f'c)\col f\ra f'$, and
		\item $\gamma\col c\ra c'$ to $(\gamma f,f'\gamma)\col (cf,1_{f'cf},f'c)
			\Ra (c'f,1_{f'c'f},f'c')$.
	\end{itemize}

	\begin{pon}\label{deforthog}
		Let $\C$ be a $\Gpd$-category, $A\overset{f}\ra B$ and
		$A'\overset{f'}\ra B'$ be arrows in $\C$.  The following conditions are
		equivalent. When they hold, we say that \emph{$f$ is
		orthogonal to $f'$}\index{orthogonal arrows} (and we write $f\orth f'$).
		\begin{enumerate}
			\item The following square is a pullback in $\Gpd$.
				\begin{xym}\xymatrix@=40pt{
					{\C(B,A')}\ar[r]^-{-\circ f}\ar[d]_-{f'\circ -}
					&{\C(A,A')}\ar[d]^{f'\circ-}
					\\{\C(B,B')}\ar[r]_-{-\circ f}
					&{\C(A,B')}
				}\end{xym}
			\item The functor $\lan f,f'\ran\col\C(B,A')\ra\flc(f,f')$ is an equivalence.
			\item \begin{enumerate}
					\item For all $(a,\psi,b)\col f\ra f'$, there exist an arrow $c$
						and 2-arrows $\mu$ and $\nu$ such that
						\begin{xyml}\begin{gathered}\xymatrix@=40pt{
							A\ar[r]^f\ar[d]_a="a"
							&B\ar[d]^b="b"\ar[dl]_(0.4){c}
							\\ A'\ar[r]_{f'}
							&B'
							\ar@{}"b";"a"^(0.35){\overset{\;\;\text{$\nu$}}{\dir{=>}}}
							\ar@{}"b";"a"_(0.7){\underset{\;\;\text{$\mu$}}{\dir{=>}}}
						}\end{gathered}\;\;=\;\;\begin{gathered}\xymatrix@=40pt{
							A\ar[r]^f\ar[d]_a\drtwocell\omit\omit{\psi}
							&B\ar[d]^b
							\\ A'\ar[r]_{f'}
							&B'
						}\end{gathered}.\end{xyml}
					\item For all $c,c'\col B\ra A'$, $\alpha\col cf\Ra c'f$
						and $\beta\col f'c\Ra f'c'$ such that $f'\alpha=\beta f$, there exists
						a unique $\gamma\col c\Ra c'$ such that $\alpha=\gamma f$
						and $\beta = f'\gamma$.
				\end{enumerate}
		\end{enumerate}
	\end{pon}
	
		\begin{proof}
			Condition 2 is equivalent to condition 1, because the square
			\ref{flcprodfib} is a pullback.  Condition 3 expresses
			in elementary terms the fact that $\lan f,g\ran$ is surjective
			(a) and fully faithful (b), which is equivalent to condition 2.
		\end{proof}
		
	Here is now the
	(almost\footnote{We do not require here neither the stability of arrows of $\E$ and $\M$
	under composition with equivalences, nor their stability under isomorphism (these
	conditions amount to require that the full sub-$\Gpd$-categories $\E$ and $\M$ be
	“replete”).
	The reason is that the definition of the belonging relation between objects of $\C$ and full sub-$\Gpd$-categories $\mathcal{M}\overset{M}\ra \C$ implies that, if $A\simeq B\col\C$
	and $B\in M$, then necessarily $A\in\M$ (in other words, the full sub-($\Gpd$)-categories
	are always “replete”).  Moreover, we use here only categorical properties (i.e.\ properties of objects of a category stable under isomorphism or properties of objects of a $\Gpd$-category stable under equivalence)
	and so, if an arrow $f$ satisfies the property which defines $\E$ or $\M$,
	every arrow equivalent to $g$ also satisfies this property.}) usual definition of orthogonal factorisation system.

	\begin{df}\index{factorisation system!orthogonal}
		An \emph{orthogonal factorisation system} on a $\Gpd$-category $\C$
		consists of $(\E,\M)$, where $\E$ and $\M$ are two full sub-$\Gpd$-categories
		of $\flc$, such that
		\begin{enumerate}
			\item for all $f\col\flc$, there exist arrows $e_f\in\E$, $m_f\in\M$ and
				a 2-arrow $\varphi_f\col f\Ra m_fe_f$;
			\item for all $e\in\E$ and $m\in\M$, $e\orth m$.
		\end{enumerate}
	\end{df}

	Orthogonal factorisation systems possess the following properties
	(see \cite{Kasangian2000a}, \cite{Milius2001a}).
	
	\begin{pon}\label{propsyforth}
	Let $\EM$ be an orthogonal factorisation system. Then the following properties hold:
	\begin{enumerate}
		\item $f\in\E$ if and only if for all $m\in\M$, $f\orth m$; dually,
			$f\in\M$ if and only if for all $e\in\E$, $e\orth f$;
		\item $\E\cap\M=\Equ$;
		\item $\E$ and $\M$ are stable under composition;
		\item for all arrows $A\overset{f}\ra B\overset{g}\ra C$, if $gf, g\in\M$,
			then $f\in\M$, and if $gf,f\in\E$, then $g\in\E$;
		\item $\E$ is stable under weighted colimits and $\M$ is stable under weighted limits;
		\item $\E$ is stable under pushout and $\M$ is stable under pullback.
	\end{enumerate}
	\end{pon}
	
	Orthogonal factorisation systems on $\C$ form an order (denoted by
	$\caspar{OrthFac}(\C)$\index{OrthFac(C)@$\caspar{OrthFac}(\C)$}) described in the following way.
	\begin{itemize}
		\item {\it Objects. } These are the orthogonal factorisation systems on $\C$.
		\item {\it Order. } We say that $(\E,\M)\leq(\E',\M')$ if 
			$\E\incl\E'$ or, equivalently (by property 1 of Proposition
			\ref{propsyforth}),
			$\M\supseteq\M'$.
	\end{itemize}
	
	The proof of the following proposition is the 2-dimensional version of
	the proof of Korostenski and Tholen for the 1-dimensional case \cite[Theorem A]{Korostenski1993a}.
	
	\begin{pon}
		Let $\C$ be a $\Gpd$-category. Then
		\begin{eqn}
			\caspar{Fac}(\C)\simeq \caspar{OrthFac}(\C).
		\end{eqn}
	\end{pon}
	
		\begin{proof}
			{\it Definition of $\Phi\col\caspar{Fac}(\C)\ra
			\caspar{OrthFac}(\C)$. } Let $(\E,\allowbreak\M,\allowbreak\im,\allowbreak e,\allowbreak m,\allowbreak \varphi)$
			be a factorisation system. Then $(\E,\M)$ is an orthogonal factorisation system.
			The first condition obviously hold. It remains to prove that, for all arrows $A\overset{f}\ra B$ and $A'\overset{f'}\ra B'$,
			if $f\in\E$ and $f'\in\M$, then $f\orth f'$. We check the two parts
			of condition 3 of Proposition \ref{deforthog}.
			\renewcommand{\theenumi}{\alph{enumi}}\renewcommand{\labelenumi}{(\theenumi)}
			\begin{enumerate}
				\item Let be $(a,\psi,b)\col f\ra f'$ in $\flc$.  If we apply
					$\im$, $e$, $m$ and $\varphi$ to $f$, $f'$ and $(a,\psi,b)$, we get
					the following diagram, whose composite is equal to $\psi$.
					\begin{xym}\label{diagpremifact}\xymatrix@=40pt{
						A\ar[r]^-{e_f}\ar[d]_a\drtwocell\omit\omit{}
							\rruppertwocell<9>^f{_<-2.7>\;\;\,\varphi_f}
						&{\im f}\ar[r]^-{m_f}_-{\sim}\ar[d]|{\im(a,\psi,b)}
							\drtwocell\omit\omit{}\ar@{}[dl]_{e_{(a,\psi,b)}^{-1}}
						&B\ar[d]^b\ar@{}[dl]^{m_{(a,\psi,b)}^{-1}}
						\\ A'\ar[r]_-{e_{f'}}^-{\sim}
							\rrlowertwocell<-9>_{f'}{_<2.7>\;\;\;\;\;\varphi_{f'}^{-1}}
						&{\im f'}\ar[r]_-{m_{f'}}
						&B'
					}\end{xym}
					By conditions 2 and 3 of the definition of factorisation system, 
					$m_f$ is invertible (because $f\in\E$) and $e_{f'}$ is invertible
					(because $f'\in\M$).  Thus $c\eqdef e_{f'}^{-1}\im(a,\psi,b)m_f^{-1}$
					is an arrow $B\ra A'$,
					and we have 2-arrows $\mu\col cf\Ra a$ and $\nu\col b\Ra f'c$
					whose composite is $\psi$.
				\item Let be $c,c'\col B\ra A'$, $\alpha\col cf\Ra c'f$ and
					$\beta\col f'c\Ra f'c'$ such that $f'\alpha=\beta f$.  Then
					$(\alpha,\beta)$ is a 2-arrow
					$(cf,1_{f'cf},f'c)\Ra(c'f,1_{f'c'f},f'c')$
					in $\flc$. Moreover, $m_f$ and $e_{f'}$ are invertible, as well as
					$e_{1_B}$, $m_{1_B}$, $e_{1_{A'}}$ and $m_{1_{A'}}$, by condition
					4 of the definition of factorisation system.  We can define $\gamma$ as
					the composite of the following diagram.
					We check that $\gamma f=\alpha$ and $f'\gamma=\beta$.
					Moreover, $\gamma$ is unique because, if we have $\gamma'$ such that 
					$\gamma' f=\alpha$ and $f'\gamma'=\beta$, then $\gamma'$ is equal
					to the composite of the following diagram and so to $\gamma$.
			\end{enumerate}
					\begin{xym}\xymatrix@=40pt{
						&&A'\ar[dr]^-{m_{1_{A'}}^{-1}}\ar@{=}@/^3pc/[ddrrr]
							\ddrrtwocell\omit\omit{_<-5.5>}
						\\ &{}\rtwocell\omit\omit{_<-4>}
						&{\im 1_B}\ar[r]^-{\im(c,1_c,c)}
						&{\im 1_{A'}}\ar[dr]|-{\im(1_{A'},1_{f'},f')}
							\ar@/^/[drr]^-{e_{1_{A'}}^{-1}}\drrtwocell\omit\omit{_<1.5>}
						\\ B\ar[r]|-{m_f^{-1}}
							\ar@/^/[urr]^-{m_{1_B}^{-1}}\ar@/^2.3pc/[uurr]^-{c}
							\ar@/_/[drr]_-{m_{1_B}^{-1}}\ar@/_2.3pc/[ddrr]_-{c'}
							\urrtwocell\omit\omit{_<1.5>}\drrtwocell\omit\omit{_<-1.5>}
						&{\im f}\ar[ur]|-{\im(f,1_f,1_B)}\ar[dr]|-{\im(f,1_f,1_B)}
							\ar@/^1pc/[rrr]^-{\im(cf,1_{f'cf},f'c)}
							\ar@/_1pc/[rrr]_-{\im(c'f,1_{f'c'f},f'c')}
						&{}\rtwocell\omit\omit{\;\;\;\;\;\;\;\;\;\;\;\im(\alpha,\beta)}
							\rtwocell\omit\omit{_<-7.7>}\rtwocell\omit\omit{_<7.7>}
						&{}&{\im f'}\ar[r]|-{e_{f'}^{-1}}
						&A'
						\\ &{}\rtwocell\omit\omit{_<4>}
						&{\im 1_B}\ar[r]_-{\im(c',1_{c'},c')}
						&{\im 1_{A'}}\ar[ur]|-{\im(1_{A'},1_{f'},f')}
							\ar@/_/[urr]_-{e_{1_{A'}}^{-1}}\urrtwocell\omit\omit{_<-1.5>}
						\\ &&A'\ar[ur]_-{m_{1_{A'}}^{-1}}\ar@{=}@/_3pc/[uurrr]
							\uurrtwocell\omit\omit{_<5.5>}
					}\end{xym}
			\renewcommand{\theenumi}{\arabic{enumi}}\renewcommand{\labelenumi}{\theenumi .}
			It is obvious that the functor $\Phi$ which maps $(\E,\M,\im,e,m,\varphi)$
			to $\EM$ preserves the order.
			
			{\it $\Phi$ is surjective. } Let $(\E,\M)$ be an orthogonal factorisation system.
			We must define $\im$, $e$, $m$ and $\varphi$.  By the first condition
			of orthogonal factorisation system, they are already defined on objects.
			We extend this definition to arrows and 2-arrows thanks to orthogonality.
			So, if $(a,\psi,b)\col f\ra f'$ is an arrow in $\flc$, since
			$e_f\orth m_{f'}$, there exist an arrow $\im(a,\psi,b)\col\im f\ra \im f'$
			and 2-arrows $e_{(a,\psi,b)}$ and $m_{(a,\psi,b)}$ such that the composite
			of diagram \ref{diagpremifact} is equal to $\psi$.
			We check the $\Gpd$-functoriality
			of $\im$, the $\Gpd$-naturality of $e$ and $m$, and the fact that
			$\varphi$ is a modification also by using orthogonality.
			
			{\it $\Phi$ is fully faithful. } If $(\E,\M,\im,e,m,\varphi)$ and
			$(\E',\M',\im',e',m',\varphi')$ are two factorisation systems
			such that $\E\incl\E'$ and $\M'\supseteq\M$, we define a $\Gpd$-natural transformation
			$d\col\im\ra\im'$ and modifications $\varepsilon$
			and $\mu$ thanks to orthogonality: $e_f\in\E\incl\E'$ is orthogonal
			to $m'_f\in\M'$, so there exist $d_f$, $\varepsilon_f$ and $\mu_f$ such
			that equation \ref{eqordfac} hold.  We establish the $\Gpd$-naturality
			of $d$ and the fact that $\varepsilon$ and $\mu$ are modifications by using condition 3(b) of Proposition \ref{deforthog}.
		\end{proof}

\section{Kernel-quotient systems}\label{sectsysnoyquot}

\subsection{One-dimensional examples}

In this subsection, we will follow the same scheme to present some usual one-dimensional notions.  In each case, a notion of congruence and an adjunction between some kind of quotient and some kind of kernel are introduced.  Notions of monomorphism, regular epimorphism, etc., are then defined in terms of this adjunction.  This scheme will be generalized later in this section under the name of “kernel-quotient system”.  The proofs are omitted because the results are usually well-known and because they will be given in the general case.

The first example concerns $\Ens$-categories, with equivalence relations as congruences, the kernel relation (or kernel pair) as kernel, and the usual quotient of equivalence relations as quotient.

\begin{df}
	Let $\C$ be a $\Ens$-category.  We denote by $\Equ(\C)$ the category defined
	in the following way.
	\begin{itemize}
		\item {\it Objects. } These are the \emph{equivalence relations}
			$d_0,d_1\col R\rightrightarrows A$ in $\C$.
		\item {\it Arrows. } An arrow $(R,A,d_0,d_1)\ra(R',A',d'_0,d'_1)$ consists of
			two arrows $g\col R\ra R'$ and $f\col A\ra A'$ such that $d'_0g=fd_0$ and $d'_1g=fd_1$.
	\end{itemize}
\end{df}

The kernel of an arrow is its kernel relation (or kernel pair), which is the pullback of the arrow $f$ along itself.

\begin{df}\index{kernel relation}
	Let $A\overset{f}\ra B$ be an arrow in a $\Ens$-category.
	A \emph{kernel relation} of $f$
	consists of an object $\KRel(f)$\index{KRel(f)@$\KRel(f)$}
	and of two arrows $d_0,d_1\col \KRel(f)\rightrightarrows A$, such that $fd_0=fd_1$,
	and satisfying the universal property for this condition.
\end{df}

\begin{pon}
	The kernel relation of an arrow $f$ is an equivalence relation.
\end{pon}

The notion of quotient corresponding to this notion of kernel is the usual quotient of equivalence relations, which is the coequalizer of the pair of arrows.

\begin{df}\index{quotient!in $\Ens$-categories}
	Let $d_0,d_1\col R\rightrightarrows A$
	be an equivalence relation in a category $\C$.
	A \emph{quotient} of this relation consists of an
	object $\Quot (R)\col\C$ and an arrow $q\col A\ra \Quot (R)$ such that $qd_0=qd_1$,
	satisfying the universal property for this condition.
\end{df}

When $\C$ has all the kernel relations and quotients of equivalence relations, $\KRel$ and $\Quot$ are functors between $\Equ(\C)$ and $\flc$ which form an adjunction:
\begin{xym}\xymatrix@=50pt{
	{\Equ(\C)}\ar@<2mm>[r]^-{\Quot} \ar@{}[r]|-\perp
	&{\fl{\C}}\ar@<2mm>[l]^-{\KRel}
}.\end{xym}
We will denote the unit of this adjunction by $\eta_R\col R\ra \KRel(\Quot R)$ and its counit by $\varepsilon_f\col \Quot(\KRel f)\ra f$; we can describe the counit in more details: the arrow $f$ factors in the following way through the quotient of its kernel relation; then $\varepsilon_f = (1_A,m_f)$.  We call this factorisation the \emph{regular factorisation}\index{factorisation!regular}\index{regular factorisation} of $f$.
\begin{xym}\xymatrix@C=15pt@R=30pt{
	{\KRel f}\ar@<-1mm>[rr]_-{d_1}\ar@<1mm>[rr]^-{d_0}
	&& A\ar[rr]^-f\ar[dr]_-{e_f}
	&& B
	\\ &&&{\Quot (\KRel f)}\ar[ur]_-{m_f}
}\end{xym}

We can define the monomorphisms in terms of the adjunction $\Quot\adj \KRel$.
Let us remark that, for all arrow $A\overset{f}\ra B$ in a category $\C$, there is a canonical arrow $(1_A,f)\col 1_A\ra f$ in $\fl{\C}$.  We will say that $f$ has \emph{trivial kernel relation}\index{kernel relation!trivial} if $\KRel(1_A,f)$ is an isomorphism $\KRel(1_A)\simeq \KRel(f)$, i.e.\ if $1_A, 1_A\col A\rightrightarrows A$ is a kernel relation of $f$.

\begin{df}\index{monomorphism!in a category}
	An arrow $f$ in a category $\C$ is a \emph{monomorphism} if its kernel relation exists and is trivial. We denote by $\Mono(\C)$\index{Mono(C)@$\Mono(\C)$} the full subcategory of $\fl{\C}$ whose objects are the monomorphisms.
\end{df}

The regular epimorphisms corresponding to the adjunction $\Quot\adj\KRel$ are the ordinary regular epimorphisms; these are the arrows which are canonically the quotient of their kernel relation.

\begin{df}\index{epimorphism!regular}\index{regular epimorphism}
	An arrow $f$ in $\C$ is called a \emph{regular epimorphism} 
	if the counit $\varepsilon_f\col \Quot(\KRel f)\ra f$ is an isomorphism
	(i.e.\ if $m_f$ 
	is an isomorphism).   We denote by $\Epireg(\C)$\index{regepi(C)@$\Epireg(\C)$}
	the full subcategory of $\fl{\C}$ whose objects are the regular epimorphisms.
\end{df}

As Renato Betti, Dietmar Schumacher and Ross Street have observed
(see \cite{Betti1999a,Betti1997a}), it is possible to define various interesting exactness conditions on a category $\C$ in terms of the adjunction between quotient and kernel.

The following equivalences are well-known (\cite[Proposition 2.7]{Kelly1969a}).

\begin{pon}\label{caracfactorisable}
	Let $\C$ be a category which has all kernel relations and quotients of equivalence relations.  Then the following conditions are equivalent.
	When they hold, we will say that $\C$ is 
	\emph{$\KRel$-factorisable}\index{KRel-factorisable category@$\KRel$-factorisable category}%
	\index{factorisable!KRel@$\KRel$-}.
	\begin{enumerate}
		\item For all $f\col\fl{\C}$, $m_f$ is a monomorphism.
		\item Regular epimorphisms in $\C$ are stable under composition.
	\end{enumerate}
\end{pon}

In a $\KRel$-factorisable category, $(\Epireg(\C),\Mono(\C))$ is a factorisation system.

\begin{df}\index{KRel-preexact categorie@$\KRel$-preexact category}\index{preexact!KRel@$\KRel$-}
	A category $\C$ is \emph{$\KRel$-preexact} if $\eta$ is an isomorphism
	(i.e.\ if every equivalence relation in
	$\C$ is canonically the kernel relation of its quotient).
\end{df}

As every adjunction, the adjunction $\Quot\adj \KRel$ induces an equivalence between the objects of $\flc$ where the counit is an isomorphisme (i.e.\ the regular epimorphisms) and the objects of $\Equ(\C)$ where the unit is an isomorphism. So, for a $\KRel$-preexact category, we have an equivalence
\begin{xym}\xymatrix@=50pt{
	{\Equ(\C)}\ar@<2mm>[r]^-{\Quot} \ar@{}[r]|-\simeq
	&{\fl{\C}}\ar@<2mm>[l]^-{\KRel}
}.\end{xym}

An important property is that every quotient is an epimorphism and every coquotient is a monomorphism.
In every category which have all the involved limits and colimits, there is a comparison arrow $w_f$ between the quotient of the kernel relation of an arrow $f$ and the coquotient of its cokernel corelation (denoted by $\mathrm{CokRel}\,f$).
\begin{xym}\xymatrix@R=40pt@C=20pt{
	{\KRel f}\ar@<-1mm>[r]_-{d_1}\ar@<1mm>[r]^-{d_0}
	&A\ar[r]^-f\ar[dr]^(0.3){e_f^2}\ar[d]_{e_f^1}
	&B\ar@<-1mm>[r]_-{i_1}\ar@<1mm>[r]^-{i_0}
	&{\mathrm{CokRel}\, f}
	\\ &{\Quot(\KRel f)}\ar[r]_-{w_f}\ar[ur]^(0.7){m_f^1}
	&{ \mathrm{Coquot}(\mathrm{CokRel}\, f)}\ar[u]_{m_f^2}
}\end{xym}

The following proposition characterises the categories for which this comparison arrow is always an isomorphism.

\begin{pon}\label{propkrelparf}
	Let $\C$ be a category in which all the kernel relations, cokernel corelations,
	quotients of equivalence relations and coquotients of coequivalence corelations exist. The following conditions are equivalent:
	\begin{enumerate}
		\item for all $f$, $w_f$ is an isomorphism;
		\item each arrow factors as a regular epimorphism followed by a
			regular mo\-no\-mor\-phism;
		\item every monomorphism is regular and every epimorphism is regular.
	\end{enumerate}
	When these conditions hold, we will say that $\C$ is \emph{$\KRel$-$\mathrm{CokRel}$-perfect}%
	\index{perfect!KRel-CokRel@$\KRel$-$\mathrm{CokRel}$}%
	\index{KRel-CokRel-perfect category@$\KRel$-$\mathrm{CokRel}$-perfect category}.
\end{pon}

In a $\KRel$-$\mathrm{CokRel}$-perfect category, 
\begin{eqn}
	(\Epireg(\C),\Monoreg(\C)) = (\Epi(\C),\Mono(\C))
\end{eqn}
is a factorisation system; $\KRel$-$\mathrm{CokRel}$-perfect categories are both $\KRel$-fac\-to\-ri\-sable and $\mathrm{CokRel}$-factorisable (the dual of $\KRel$-factorisable).

The adjunction $\Quot\adj\KRel$ is directly connected to the factorisation system $(\Surj,\Inj)$ on the category of sets: we can recover the factorisation $(\Surj,\Inj)$ from the adjunction $\Quot\adj \KRel$.

\begin{pon}~
	\begin{enumerate}
		\item $\Inj = \Mono(\Ens)$;
		\item $\Surj = \Epireg(\Ens)$.
	\end{enumerate}
\end{pon}

In particular, $\Ens$ is $\KRel$-factorisable.
The notion of congruence which we use can be justified by the fact that $\Ens$ is $\KRel$-preexact. 

\begin{pon}
	For every equivalence relation $d_0,d_1\col R\rightrightarrows A$ in $\Ens$, the unit
	$\eta_R$ of the adjunction $\Quot\adj \Ker$ is an isomorphism.
\end{pon}

We can also directly define the monomorphisms in terms of the factorisation system $(\Surj,\Inj)$.

\begin{pon}\index{monomorphism!characterisation}
	An arrow $f$ in a category $\C$ is a monomorphism if and only if, for all
	$X\col\C$, $\C(X,f)\in\Inj$.
\end{pon}

The notions defined in terms of the dual adjunction, between the cokernel corelation and the coquotient, also induce the factorisation system $(\Surj,\Inj)$.

\begin{pon}~
	\begin{enumerate}
		\item $\Surj=\Epi(\Ens)$;
		\item $\Inj=\Monoreg(\Ens)$.
	\end{enumerate}
\end{pon}

\bigskip
In $\Ensp$-categories\index{Set*-category@$\Ensp$-category} (categories enriched in the category $\Ensp$\index{Set*@$\Ensp$} of pointed sets), we can study the same adjunction $\Quot\adj\KRel$ as in $\Ens$ and everything works in the same way as for ordinary $\Ens$-categories.
In the following of this subsection, we study the adjunction $\Coker\adj\Ker$ between the cokernel and the kernel.

For this adjunction, the congruences are the monomorphisms.  The congruences in $\C$ form thus a category $\Mono(\C)$, which is a full subcategory of the category $\flc$ of the arrows in $\C$.

\begin{df}\index{kernel!in a $\Ensp$-category}
	Let $A\overset{f}\ra B$ be an arrow in a $\Ensp$-category.  A \emph{kernel} of $f$
	consists of an object $\Ker f$ and of an arrow
	$k\col\Ker f\ra A$, such that $fk=0$, satisfying the universal property for this condition.
\end{df}

\begin{pon}
	The kernel of an arrow $f$ is a monomorphism.
\end{pon}

The notion of quotient corresponding to this notion of kernel is the cokernel, which is the dual of the kernel. If $\C$ has all kernels and the cokernels of monomorphisms, kernels and cokernels form functors $\Coker\adj\Ker$:
\begin{xym}\xymatrix@=50pt{
	{\Mono(\C)}\ar@<2mm>[r]^-{\Coker} \ar@{}[r]|-\perp
	&{\fl{\C}}\ar@<2mm>[l]^-{\Ker}
}.\end{xym}
We denote by $\eta_i\col i\ra \Ker(\Coker i)$ the unit of this adjunction and its counit
by $\varepsilon_f\col \Coker(\Ker f)\ra f$.  Like for the adjunction $\Quot\adj\KRel$, this adjunction gives a factorisation of each arrow, called the \emph{$\Ker$-regular factorisation}\index{factorisation!Ker-regular@$\Ker$-regular}\index{Ker-regular factorisation@$\Ker$-regular factorisation}, as in the following diagram. Then $\varepsilon_f= (1_A,m_f)$.
\begin{xym}\xymatrix@C=15pt@R=30pt{
	{\Ker f}\ar[rr]^-{k_f}
	&& A\ar[rr]^-f\ar[dr]_-{e_f}
	&& B
	\\ &&&{\Coker (\Ker f)}\ar[ur]_-{m_f}
}\end{xym}

We will say that $A\overset{f}\ra B$ has \emph{trivial kernel}\index{kernel!trivial!in a $\Ensp$-category} if $\Ker(1_A,f)\col \Ker(1_A)\ra \Ker(f)$ is an isomorphism, i.e.\ if $0$ is a kernel of $f$.  This allows us to define a notion of monomorphism in terms of the adjunction $\Coker\adj\Ker$.

\begin{df}\index{0-monomorphism!in a Ens*-category@in a $\Ensp$-category}
	An arrow $f$ in a $\Ensp$-category $\C$ is a \emph{0-monomorphism} if its
	kernel exists and is trivial (i.e.\ if $0$ is a kernel of $f$).
	We denote by $\zMono(\C)$ the full sub-category
	of $\fl{\C}$ whose objects are the 0-monomorphisms.
\end{df}

The dual notion will be called \emph{0-epimorphism}\index{0-epimorphism}.
We can also define a notion of regular epimorphism corresponding to the adjunction $\Coker\adj\Ker$.  This is an arrow which is canonically the cokernel of its kernel.

\begin{df}\index{epimorphism!normal}\index{normal!epimorphism}
	An arrow $f$ in $\C$ is a \emph{normal epimorphism}
	if the counit $\varepsilon_f\col \Coker(\Ker f)\ra f$ is a isomorphism (i.e.\ 
	if $m_f$ is a isomorphism).   We denote by
	$\Epinorm(\C)$ the full subcategory
	of $\fl{\C}$ whose objects are the normal epimorphisms.
\end{df}

The dual notion will be called \emph{normal monomorphism}\index{monomorphism!normal}\index{normal!monomorphism}.  We can now define some exactness conditions in terms of this adjunction, in the same way as for those defined in terms of the adjunction $\Quot\adj\KRel$.

\begin{pon}
	Let $\C$ be a $\Ensp$-category which has the kernels and the cokernels.
	Then the following conditions are equivalent; when they hold, we will say that $\C$ is \emph{$\Ker$-factorisable}%
	\index{factorisable!Ker-@$\Ker$-}\index{Ker-factorisable@$\Ker$-factorisable!Set*-category@$\Ensp$-category}:
	\begin{enumerate}
		\item for all $f\col\fl{\C}$, $m_f$ is a 0-monomorphism;
		\item normal epimorphisms in $\C$ are stable under composition.
	\end{enumerate}
\end{pon}

In a $\Ker$-factorisable $\Ensp$-category, $(\Epinorm(\C),\zMono(\C))$ is a factorisation system.

\begin{df}\index{preexact!Ker@$\Ker$-}\index{Ker-preexact@$\Ker$-preexact!Set*-category@$\Ensp$-category}
	A $\Ensp$-category $\C$ is \emph{$\Ker$-preexact} if every monomorphism in $\C$ is canonically the kernel of its cokernel (i.e.\ if $\eta$ is an isomorphism).
\end{df}

In a $\Ker$-preexact category, the adjunction $\Coker \adj\Ker$ restricts to an equivalence
\begin{xym}\xymatrix@=50pt{
	{\Mono(\C)}\ar@<2mm>[r]^-{\Coker}
		\ar@{}[r]|-\simeq
	&{\Epinorm(\C)}\ar@<2mm>[l]^-{\Ker}
}.\end{xym}

Let us turn to the dual constructions and to the comparison map between the factorisations.  In this case, there is a special phenomenon: the cokernel and the quotient coincide, as well as the coquotient and the kernel.

We can remark that every cokernel is a 0-epimorphism (even an epimorphism) and every kernel is a 0-monomorphism (even a monomorphism).  As for the kernel relation and the quotient, we can construct a comparison arrow $w_f$ between the cokernel of the kernel of an arrow $f$ and the kernel of its cokernel.  We get the following diagram.
\begin{xym}\xymatrix@R=40pt@C=20pt{
	{\Ker f}\ar[r]^-{k_f}
	&A\ar[r]^-f\ar[dr]^(0.3){e_f^2}\ar[d]_{e_f^1}
	&B\ar[r]^-{q_f}
	&{\Coker f}
	\\ &{\Coker(\Ker f)}\ar[r]_-{w_f}\ar[ur]^(0.7){m_f^1}
	&{ \Ker(\Coker f)}\ar[u]_{m_f^2}
}\end{xym}

As in the case of the adjunction $\Quot\adj\KRel$, we can characterise the $\Ensp$-categories for which $w_f$ is an isomorphism for all $f$.  This is Theorem 39.17
of \cite{Herrlich1973a}.  Puppe-exact categories appear in Puppe \cite{Puppe1962a} and in the book of Mitchell \cite{Mitchell1965a} and were studied by Marco Grandis \cite{Grandis1992a}.

\begin{pon}
	Let $\C$ be a $\Ensp$-category in which kernels and cokernels exist.
	The following conditions are equivalent; when they hold, we will say that $\C$ is \emph{$\Ker$-$\Coker$-perfect}
	or \emph{Puppe-exact}\index{Puppe-exact!Set*-category@$\Ensp$-category}%
	\index{Set*-category@$\Ensp$-category!Puppe-exact}%
	\index{Ker-Coker-perfect Set*-category@$\Ker$-$\Coker$-perfect $\Ensp$-category}%
	\index{perfect!Ker-Coker@$\Ker$-$\Coker$-}:
	\begin{enumerate}
		\item for all $f$, $w_f$ is a isomorphism;
		\item each arrow factors as a normal epimorphism followed by a normal mo\-no\-mor\-phism;
		\item every 0-monomorphism is a normal monomorphism
			and every 0-epi\-mor\-phism is a normal epimorphism.
	\end{enumerate}
\end{pon}

In a Puppe-exact $\Ensp$-category, 
\begin{eqn}
	(\Epinorm,\Mononorm)=(\Epi,\Mono)=(\zEpi,\zMono)
\end{eqn}
is a factorisation system.
With this definition, we have almost left the purely pointed world since, by adding the existence of products or coproducts (one of them suffices), $\C$ is abelian and thus an $\Ab$-category.  There are few examples of non-abelian Puppe-exact categories (see \cite{Carboni1996a} and \cite{Borceux2007a}).

We could ask that in a $\Ensp$-category $\KRel$-regular and  $\Ker$-regular factorisations exist and coincide.  This is the case for semi-abelian categories. 

The adjunction $\Coker\adj\Ker$ is closely related to the factorisation system $(\Bij^*,\zMono)$ on $\Ensp$, that we will describe now.

\begin{df}\index{0-injective}
	A morphism of pointed sets $A\overset{f}\ra B$ is \emph{0-injective} if for
	all $a\col A$, if $f(a)=0$, then $a=0$.  We denote by $\zInj$\index{0-inj@$\zInj$}
	the full subcategory of $\fl{(\Ensp)}$ whose objects are the 0-injective morphisms.
\end{df}

\begin{df}\index{bijective outside the kernel}
	A morphism of pointed sets $A\overset{f}\ra B$ is \emph{bijective outside the kernel}
	if the following conditions are satisfied:
	\begin{enumerate}
		\item $f$ is surjective;
		\item $f$ is injective outside the kernel:
			for all $a,a'\col A$, if $f(a)=f(a')$, then either $f(a)=0=f(a')$,
			or $a=a'$.
	\end{enumerate}
	We denote by $\Bij^*$\index{Bij*@$\Bij^*$} the full subcategory of $\fl{\Ensp}$
	whose objects are the morphisms bijective outside the kernel.
\end{df}

\begin{pon}\label{factzero}
	$(\Bij^*,\zInj)$ is a factorisation system.
\end{pon}

	\begin{proof}
		The image of $A\overset{f}\ra B$ is $\im^* f= A/{\sim}$,
		where $a\sim a'$ if $f(a)=0=f(a')$ or $a=a'$,
		with $0$ as the distinguished element. We define $m_f\col\im^* f\ra B\col a\mapsto f(a)$ and
		$e_f\col A\ra \im^* f\col a\mapsto a$.
		It is easy to see that $m_f$ is 0-injective, that $e_f$ is bijective outside the kernel and that this defines a factorisation system.
	\end{proof}

The following proposition shows that we can get back the factorisation system $(\Bij^*,\zInj)$ from the adjunction $\Coker\adj\Ker$.  

\begin{pon}~
	\begin{enumerate}
		\item $\zInj = \zMono(\Ensp)$;
		\item $\Bij^* = \Epinorm(\Ensp)$.
	\end{enumerate}
\end{pon}

Moreover, the following proposition justifies the definition of congruence we use.

\begin{pon}
	For every monomorphism $i$ in $\Ensp$, the unit
	$\eta_i$ of the adjunction $\Coker\adj\Ker$ is an isomorphism (every monomorphism is normal).
\end{pon}

We can also directly define 0-monomorphisms in terms of the factorisation system
 $(\Bij^*,\zInj)$.

\begin{pon}\index{0-monomorphism!characterisation}
	An arrow $f$ in a $\Ensp$-category $\C$ is a 0-monomorphism if and only if
	for all $X\col\C$, $\C(X,f)\in\zInj$.
\end{pon}

For the dual theory, there is a difference with the case of the adjunction $\Quot\adj\KRel$.  The $\Coker$-coregular factorisation in $\Ensp$ doesn't give back the factorisation system $(\Bij^*,\zMono)$, but the system $(\Surj,\Inj)$.

\begin{pon}\label{caraczepiensp}~
	\begin{enumerate}
		\item $\Surj=\zEpi(\Ensp)$;
		\item $\Inj=\Mononorm(\Ensp)$.
	\end{enumerate}
\end{pon}

\bigskip
A last example is the category $\Ab$\index{Ab@$\Ab$} of abelian groups, where the two factorisations of $\Ensp$ merge: 0-injections and injections coincide.  $\Ab$ is both $\Ker$-$\Coker$-perfect and $\KRel$-$\mathrm{CokRel}$-perfect:
the factorisation $(\Surj,\Inj)$ can be constructed by the cokernel of the kernel and the kernel of the cokernel as well as by the quotient of the kernel relation and the coquotient of the cokernel corelation.  More generally, in an $\Ab$-category, there is an equivalence between $\Equ(\C)$ and $\Mono(\C)$, and this equivalence commutes with, on the one hand, $\Ker$ and $\KRel$ and, on the other hand, $\Coker$ and $\Quot$.  This allows us to translate the notions of factorisability, preexactness or perfection in terms of one of the adjunctions from these notions in terms of the other adjunction.  This is what gives the theorem of Tierney.

\subsection{Definition of kernel-quotient system}\label{soussectdefsysnoyquot}

The definitions and propositions we've just recalled in a parallel way for the adjunction
$\Quot\adj\KRel$ in $\Ens$-categories and for the adjunction $\Coker\adj\Ker$ in $\Ensp$-categories can be developed in a general framework discovered by Renato Betti, Dietmar Schumacher and Ross Street (\cite{Betti1999a}, \cite{Betti1997a}).  
They present this framework in a 2-categorical context but, as they remark in their introduction, we can develop it in an enriched context (see  \cite{Dupont2008a}).

We start with a symmetric monoidal closed category (or, in dimension 2, with a $\Gpd$-category) $\V$, with a factorisation system $\EM$ on $\V$ and with $\W$, a full subcategory $\E$ which generates $\EM$ (in the orthogonality sense, i.e.\ for all $f\col\flv$, $f\in\M$ if and only if, for all $w\in\W$, $w\orth f$).  If the involved limits exist, we can define from $\W$, for a $\V$-category $\C$, an adjunction
	\begin{xym}\xymatrix@C=30pt@R=45pt{
		{[\W\op,\C]}\ar@<2mm>[rr]^-{Q_\W}\ar@{}[rr]|-\perp
		&&{\fl{\C},}\ar@<2mm>[ll]^-{K_\W}
	}\end{xym}
where $[\W\op,\C]$ is the $\V$-category of $\V$-functors $\W\op\ra\C$.  The functor $K_\W$ maps each arrow $f\col\flc$ to its \emph{$\W$-kernel}, which is a functor $\W\op\ra\C$, defined as the limit of $f\col\deux\ra\C$ weighted by the distributor $\phi\col\W\tens\deux\ra\V$ corresponding by adjunction to the inclusion $\W\hookrightarrow\flv$.  Dually, the functor $Q_\W$ maps each functor $H\col\W\op\ra\C$ to its \emph{$\W$-quotient}, which is an arrow in $\C$, defined as the colimit of $H$ weighted by $\phi$.

A functor $H\col\W\op\ra\C$ which becomes a $\W$-kernel in $\V$ under the action of the representable functors (i.e.\ such that, for all $X\col\C$, $\C(X,H)\col\W\op\ra\V$ is a $\W$-kernel) is called a \emph{$\W$-congruence} in $\C$.  We denote by $\W\text{-}\Cong(\C)$ the full sub-$\V$-category of $[\W\op,\C]$ whose objects are the $\W$-congruences.  If $\W$ contains the arrow $1_I$ (identity on the unit of the tensor product of $\V$), to each $\W$-congruence $H$ corresponds an object $\partial H\eqdef H(1_I)$; we say that $H$ is a $\W$-congruence \emph{on} $\partial H$.  

We have thus an adjunction $Q_\W\adj K_\W\col\flc\ra\W\text{-}\Cong(\C)$, with unit $\eta$ and counit $\varepsilon$.
	\begin{xym}\xymatrix@C=30pt@R=45pt{
		{\W\text{-}\Cong(\C)}\ar@<2mm>[rr]^-{Q_\W}\ar@<1mm>[dr]^-{\partial}\ar@{}[rr]|-\perp
		&&{\fl{\C}}\ar@<2mm>[ll]^-{K_\W}\ar@<2mm>[dl]^-{\partial_0}
		\\ &{\C}\ar@<0mm>[ur]^-{\mathrm{1_{-}}}\ar@<1mm>[ul]^-{K_\W1_{-}}
	}\end{xym}
This adjunction satisfies the following properties:
\begin{enumerate}
	\item there exist natural isomorphisms $\alpha\col\partial\overset{\sim}\Ra
		\partial_0 Q_\W$ and $\beta\col \partial_0\overset{\sim}\Ra\partial K_W$
		such that $\beta Q_\W\circ\alpha =\partial\eta$ and $\beta^{-1}\circ\alpha^{-1}K_\W
		= \partial_0\varepsilon$;
	\item $\varepsilon 1_{-}\col QK1_{-}\Ra 1_{-}$ is an isomorphism.
\end{enumerate}
Such an adjunction on $\C$ will be called here a \emph{kernel-quotient system} on $\C$.

In the case where $\V$ is $\Ens$, equipped with the factorisation system $(\Surj,\Inj)$, and where the objects of $\W$ are the unique arrows $2\ra 1$ and $1\ra 1$, the $\W$-kernel is the kernel relation, the $\W$-quotient is the quotient (coequalizer), and the $\W$-congruences are the equivalence relations (because in $\Ens$ every equivalence relation is the kernel relation of its quotient).  The functor $\partial$ maps $d_0,d_1\col R\rightrightarrows A$ to the object $A$ ($R$ is an equivalence relation \emph{on $A$}). We have thus the following kernel-quotient system.
	\begin{xym}\label{syskrelquot}\xymatrix@C=30pt@R=45pt{
		{\Equ(\C)}\ar@<2mm>[rr]^-{\Quot}\ar@<1mm>[dr]^-{\partial}\ar@{}[rr]|-\perp
		&&{\fl{\C}}\ar@<2mm>[ll]^-{\KRel}\ar@<2mm>[dl]^-{\partial_0}
		\\ &{\C}\ar@<0mm>[ur]^-{\mathrm{1_{-}}}\ar@<1mm>[ul]^-{\KRel 1_{-}}
	}\end{xym}

In the case where $\V$ is $\Ensp$, with the factorisation system $(\Bij^*,\zMono)$, and where the objects of $\W$ are the unique arrow $2\ra 1$ and the arrow $1_2\col 2\ra 2$ (where $2$ is the unit of the tensor product of pointed sets, with $0$
as distinguished element), the $\W$-kernel is the ordinary kernel, the $\W$-quotient is the cokernel, and the $\W$-congruences are the monomorphisms (because in $\Ensp$, every monomorphism is the kernel of its cokernel). The functor $\partial$ maps a monomorphism $M\overset{m}\ra A$ to the object $A$ ($m$ is a subobject \emph{of $A$}). Thus we get the following kernel-quotient system.
	\begin{xym}\label{sysnoyconoy}\xymatrix@C=30pt@R=45pt{
		{\Mono(\C)}\ar@<2mm>[rr]^-{\Coker}\ar@<1mm>[dr]^-{\partial_1}\ar@{}[rr]|-\perp
		&&{\fl{\C}}\ar@<2mm>[ll]^-{\Ker}\ar@<2mm>[dl]^-{\partial_0}
		\\ &{\C}\ar@<0mm>[ur]^-{\mathrm{1_{-}}}\ar@<1mm>[ul]^-{\Ker 1_{-}}
	}\end{xym}

From the kernel-quotient system $Q_\W\adj K_\W$, in parallel to what we did in the previous subsection for the systems $\Quot\adj\KRel$ and $\Coker\adj\Ker$, we can define notions of $K_\W$-monorphism (this is an arrow whose $K_\W$-kernel is trivial), of $K_\W$-regular epimorphism (this is  an arrow $f$ which is canonically isomorphic to $Q_\W K_\W f$), of $K_\W$-factorisability, of $K_\W$-preexactness, etc.
We can characterise the $K_\W$-monomorphisms in terms of the factorisation system $\EM$: $f$ is a $K_\W$-monomorphism if and only if, for all $X\col\C$, $\C(X,f)\in\M$.

It is always true that
\begin{eqn}
	\M = K_\W\text{-}\Mono(\V),
\end{eqn}
but it is not always the case that
\begin{eqn}
	\E = K_\W\text{-}\Epireg(\V)
\end{eqn}
(this is not the case, for example, for pointed groupoids, with the adjunction $\Coker\adj\Ker$).
This last property hold if and only if $\W$ is a \emph{dense} subcategory of $\E$.  In this case, we get back the factorisation system from the adjunction $Q_\W\adj K_\W$.

If we have two dense subcategories $\W_1$, $\W_2$ which generate the same factorisation system on $\V$, we can prove that the notions of monomorphism, regular epimorphism, factorisability, preexactness, etc., defined in terms of $\W_1$ and of $\W_2$ are equivalent (see \cite{Dupont2008a}).  This is the case for $\Ab$ where, for the two kernel-quotient systems \ref{syskrelquot} and \ref{sysnoyconoy}, $\Surj=K_\W\text{-}\Epireg(\Ab)$.  Since $\E$ is dense in $\E$, we can conclude that the notions defined in terms of a dense subcategory $\W$ in $\E$ are equivalent to those defined in terms of $\E$ itself; thus $\W$ play the role of an approximation of $\E$, easier to handle since it contains only a few objects.

In the following chapters, we will define congruences, kernels and quotients directly, without starting from a full sub-$\Gpd$-category $\W$.  That is the reason why we won't study in more details Betti-Schumacher-Street's theory.  We will merely develop a theory of kernel-quotient systems on  $\Gpd$-categories (this theory applies in particular to kernel-quotient systems on $\Ens$-categories). 
\bigskip

Let us define now kernel-quotient systems on a $\Gpd$-category.

\begin{df}\label{defsysnoyquot}\index{system!kernel-quotient|see{kernel-quotient system}}\index{kernel-quotient system}
	Let $\C$ be a $\Gpd$-category.  A \emph{kernel-quotient system} on $\C$ consists of:
	\begin{enumerate}
		\item a $\Gpd$-category $\Cong(\C)$ (whose objects will be called
			\emph{congruences}\index{congruence}) together with a $\Gpd$-functor
			$\partial\col\Cong(\C)\ra\C$
			(we will say that $H\col\Cong(\C)$ is a congruence \emph{on $\partial H$});
		\item an adjunction $Q\adj K\col\flc\ra\Cong(\C)$,
			with unit $\eta\col 1_{\Cong(\C)}\Ra KQ$ and counit
			$\varepsilon\col QK\Ra 1_{\flc}$ (we will call the congruence $Kf$
			the \emph{kernel}\index{kernel} of $f$ and the arrow $QH$ the \emph{quotient}%
			\index{quotient} of $H$);
		\item $\Gpd$-natural transformations $\alpha\col\partial\Ra \partial_0 Q$ and
			$\beta\col \partial_0\Ra\partial K$ such that
			$\partial_0\varepsilon\simeq\beta^{-1}\circ\alpha^{-1}K$
			and $\partial\eta\simeq\beta Q\circ\alpha$.
	\end{enumerate}
	Moreover, $\varepsilon 1_{-}$ must be an equivalence.
\end{df}

	\begin{xym}\xymatrix@C=30pt@R=45pt{
		{\Cong(\C)}\ar@<2mm>[rr]^-{Q}\ar@<1mm>[dr]^-{\partial}\ar@{}[rr]|-\perp
		&&{\fl{\C}}\ar@<2mm>[ll]^-{K}\ar@<2mm>[dl]^-{\partial_0}
		\\ &{\C}\ar@<0mm>[ur]^-{\mathrm{1_{-}}}\ar@<1mm>[ul]^-{\mathrm{triv}\eqdef K1_{-}}
	}\end{xym}
	The definition becomes more symmetrical if we see a kernel-quotient system as an adjunction which commutes, on the one hand, with the $\Gpd$-functors “domain”
	$\partial$ and $\partial_0$ to $\C$: we have
	\begin{itemize}
		\item $\partial_0 Q\simeq\partial$ (the domain of the quotient of a congruence on $A$ is $A$) and
		\item $\partial K\simeq \partial_0$ (the kernel of an arrow is a
			congruence on the domain of this arrow),
	\end{itemize}
	and, on the other hand, with the $\Gpd$-functors “trivial arrow” 
	$1_{-}$ (which maps an object $A$ to the “trivial” arrow on $A$, which is $1_A$)
	and “trivial congruence”\index{congruence!trivial} $\mathrm{triv}\eqdef K1_{-}$
	(which maps an object $A$ on the “trivial” congruence on $A$, which is the kernel of $1_A$): we have
	\begin{itemize}
		\item $K1_{-}\simeq \mathrm{triv}$ (the trivial congruence on $A$ is the kernel of $1_A$), and
		\item $Q\mathrm{triv}\simeq 1_{-}$ ($\varepsilon 1_A\col QK1_A\ra 1_A$
			is an equivalence: the quotient of the trivial congruence on $A$
			is $1_A$).
	\end{itemize}
	
	In the examples, the system will always be described semi-strictly, in the sense that
	$\partial\equiv\partial_0Q$ (and $\alpha$ is the identity),
	$\partial_0\equiv\partial K$ (and $\beta$ is the identity) and that
	\begin{eqn}
		\partial_0\varepsilon\equiv \mathrm{id}\text{ and }\partial\eta\equiv\mathrm{id}.
	\end{eqn}
	From now on, we assume that kernel-quotient systems are described strictly, to simplify calculations.
		
	For each arrow $A\overset{f}\ra B$ in $\C$, $\varepsilon_f$ is an arrow
	$QKf\ra f$ in $\flc$, whose first component is $1_A$, since we assume that
	$\partial_0\varepsilon$ is the identity.  We introduce the following notation:
	\begin{xyml}
		\varepsilon_f\;\;=\;\;\begin{gathered}\xymatrix@=40pt{
			A\ar@{=}[r]\ar[d]_{e_f}\drtwocell\omit\omit{\;\;\;\varphi_f}
			&A\ar[d]^f
			\\{\im f}\ar[r]_-{m_f}
			&B
		}\end{gathered} .
	\end{xyml}
	We call the factorisation of $f$ given by this diagram the \emph{$K$-regular factorisation}\index{K-regular factorisation@$K$-regular factorisation}%
	\index{factorisation!K-regular@$K$-regular} of $f$.
	
	In fact we have a $\Gpd$-functor $\im\eqdef \partial_1 QK\col \flc\ra\C$,
	$\Gpd$-natural transformations $e\eqdef \delta QK\col\partial_0\Ra \im$
	and $m\eqdef\partial_1\varepsilon\col\im\Ra \partial_1$,
	together with a modification $\varphi\col\delta\Rrightarrow me$.  Thus we are in the situation of diagram \ref{eqfact}.  But in general this does not define a factorisation system.  In Subsection \ref{sssectcondex} we study under which conditions we actually have a factorisation system.

\subsection{Monomorphisms and regular epimorphisms}

There is a canonical arrow $(1_A,1_f,f)\col 1_A\ra f$ in $\flc$. We use it to define monomorphisms relative to the kernel-quotient system $Q\adj K$ \cite[Definition 3.3]{Betti1999a}.

\begin{df}\label{defkmon}\index{monomorphism!K-@$K$-}\index{K-monomorphism@$K$-monomorphism}
	We say that an arrow $A\overset{f}\ra B$ in $\C$ is a
	\emph{$K$-monomorphism} if $K(1_A,1_f,f)\col K 1_A\ra Kf$ is an equivalence
	(i.e.\ if the kernel of $f$ is trivial).	
	We denote by $\KMono$\index{K-Mono@$\KMono$} the full sub-$\Gpd$-category of $\flc$ whose objects are the $K$-monomorphisms.
\end{df}

In a $\Ens$-category, the $\KRel$-monomorphisms are just the ordinary monomorphisms, whereas in a $\Ensp$-category, the $\Ker$-monomorphisms are the 0-mo\-no\-mor\-phisms.
We will now prove some properties of $K$-monomorphisms, similar to the usual properties of monomorphisms. To begin with, the kernels of two arrows connected by a $K$-monomorphism are equivalent \cite[Remark 3.4 2]{Betti1999a}.

\begin{pon}\label{simplnoysysteme}
	Let $A\overset{g}\ra B\overset{m}\ra C$ and $A\overset{f}\ra C$ be arrows
	in $\C$ and $\varphi\col f\Ra mg$ be a 2-arrow.  If $m$ is a $K$-monomorphism,
	then $K(1_A,\varphi^{-1},m)\col Kg\ra Kf$ is an equivalence.
\end{pon}

	\begin{proof}
		The following cube, where vertical arrows are thought of as objects of $\flc$, is a pullback in $\flc$ (because the upper and lower squares are pullbacks in $\C$).  
		\begin{xym}\xymatrix@=20pt{
			&A\ar[rr]^g\ar@{=}[dl]\ar'[d][dd]_(-0.4){g}
			&&B\ar@{=}[dd]\ar@{=}[dl]
			\\ A\ar[rr]_(0.3){g}\ar[dd]_f
			&&B\ar[dd]_(0.3){m}
			\\ &B\ar@{=}'[r][rr]\ar[dl]^-m
			&{}&B\ar[dl]^-{m}
			\\ C\ar@{=}[rr]
			&&C
		}\end{xym}
		Now, since $K$ is a right adjoint, it
		preserves limits.  Thus, by applying $K$ to this cube, we get the following pullback
		in $\Cong(\C)$.  As $m$ is a $K$-monomorphism, the right side of the square
		is an equivalence.  So, since the square is a pullback, the left side is also
		an equivalence.\qedhere
		\begin{xym}\xymatrix@=40pt{
			Kg\ar[r]^-{K(g,1_g,1_B)}\ar[d]_{K(1_A,\varphi^{-1},m)}\drtwocell\omit\omit{}
			&K1_B\ar[d]^{K(1_B,1_m,m)}
			\\ Kf\ar[r]_-{K(g,\varphi,1_C)}
			&Km
		}\end{xym}
	\end{proof}

A corollary is that $K$-monomorphisms are stable under composition and satisfy the cancellation law \cite[Theorem 3.5]{Betti1999a}.

\begin{coro}\label{kmonstabcompsimp}
	Let $A\overset{m}\ra B\overset{n}\ra C$ be arrows in $\C$.
	\begin{enumerate}
		\item If $m$ and $n$ are $K$-monomorphisms, then $nm$ is a $K$-monomorphism.
		\item If $nm$ and $n$ are $K$-monomorphisms, then $m$ is a
			$K$-monomorphism.
	\end{enumerate}
\end{coro}

	\begin{proof}
		By the previous proposition, if $n$ is a $K$-monomorphism, the upper arrow of the
		following diagram is an equivalence.  Thus $m$ is a $K$-monomorphism if and only
		if the left arrow is an equivalence if and only if the right arrow is an equivalence
		if and only if $nm$ is a $K$-monomorphism.\qedhere
		\begin{xym}\xymatrix@C=20pt@R=30pt{
			Km\ar[rr]^-{K(1_A,1_{nm},n)}_-{\sim}\ar@{}[drr]|(0.47){\underset{~}{\dir{=>}}}
			&&K(nm)
			\\ &K1_A\ar[ul]^-{K(1_A,1_m,m)}\ar[ur]_-{K(1_A,1_{nm},nm)}
			&{}
		}\end{xym}
	\end{proof}

A $K$-regular epimorphism is an arrow which is canonically the quotient of its kernel.  In general, quotients are not $K$-regular epimorphisms, unlike what happens in the 1-dimensional examples.  This will be the case when the adjunction $Q\adj K$ is idempotent (see the following subsection).

\begin{df}\label{defkregepi}\index{epimorphism!K-regular@$K$-regular}\index{K-regular epimorphism@$K$-regular epimorphism}
	An arrow $A\overset{f}\ra B$ in $\C$ is a
	\emph{$K$-regular epimorphism} if $m_f$ is an equivalence
	(i.e.\ if $\varepsilon_f$ is an equivalence).  We denote by
	$\KEpireg$\index{K-regepi@$\KEpireg$} the full sub-$\Gpd$-category of $\flc$
	whose objects are the $K$-regular epimorphism.
\end{df}

In a $\Ens$-category, the $\KRel$-regular epimorphisms are the ordinary regular epimorphisms and, in a $\Ensp$-category, the $\Ker$-regular epimorphisms are the normal epimorphisms.

\begin{pon}\label{kmonintkregestequ}
	$\KMono\cap\KEpireg = \Equ$.
\end{pon}

	\begin{proof}
		{\it $\KMono\cap\KEpireg \subseteq \Equ$. } Let $A\overset{f}\ra B$ be an arrow
		which is both a $K$-monomorphism and a $K$-regular epimorphism.
		Then, in the following square, the left arrow is an equivalence by the definition
		of kernel-quotient system, the upper arrow is an equivalence because $K(1_A,1_f,f)$
		is an equivalence, since $f$ is a $K$-monomorphism, and the right arrow is an
		equivalence because $f$ is a $K$-regular epimorphism.  So the upper arrow is
		an equivalence and, as $f\simeq 1_A$, $f$ is an equivalence.
		\begin{xym}\xymatrix@=40pt{
			QK1_A\ar[r]^-{QK(1_A,1_f,f)}\ar[d]_{\varepsilon_{1_A}}\drtwocell\omit\omit{}
			&QKf\ar[d]^{\varepsilon_f}
			\\ 1_A\ar[r]_-{(1_A,1_f,f)}
			&f
		}\end{xym}

		{\it $\Equ\incl\KMono$. } Let $A\overset{f}\ra B$ be an equivalence.
		Then $(1_A,1_f,f)$ is an equivalence in $\flc$ and so $K(1_A,1_f,f)$
		is an equivalence in $\Cong(\C)$ and $f$ is a $K$-monomorphism.
		
		{\it $\Equ\incl\KEpireg$. } Let $f$ be an equivalence. In the above diagram,
		the upper and lower arrow are equivalences because $f$ is and the left arrow
		is an equivalence by the definition of kernel-quotient system.  So $\varepsilon_f$
		is an equivalence and $f$ is a $K$-regular epimorphism.
	\end{proof}

The following proposition is Theorem 4.7 of \cite{Betti1999a}.

\begin{pon}\label{quotestorthamon}
	For every congruence $H$ and every $K$-monomorphism $A\overset{m}\ra B$,
	$QH\orth m$.  In particular, for every $K$-regular epimorphism $e$
	and every $K$-monomorphism $m$, $e\orth m$.
\end{pon}

	\begin{proof}
		In the following diagram, the left upper arrow is an equivalence
		(it comes from the adjunction $\partial_1\adj 1_{-}\col\C\ra\flc$),
		the right upper and right lower arrows are equivalences (they come from
		the adjunction $Q\adj K$), and the right arrow is an equivalence (since
		$m$ is a $K$-monomorphism).  Thus $\lan QH,m\ran$ is an equivalence, i.e.
		$QH\orth m$.\qedhere
		\begin{xym}\xymatrix@R=40pt@C=50pt{
			{\C(\partial_1QH,A)}\ar[r]^-{(-\circ QH,1,1_A)}\ar[dr]_-{\lan QH,m\ran}
			&{\flc(QH,1_A)}\ar[r]^-{\sim}\ar[d]|{(1_A,1_m,m)\circ -}\drtwocell\omit\omit{^{}}
			&{\Cong(\C)(H,K1_A)}\ar[d]^{K(1_A,1_m,m)\circ -}
			\\ &{\flc(QH,m)}\ar[r]_-{\sim}
			&{\Cong(\C)(H,Km)}
		}\end{xym}
	\end{proof}

\subsection{Exactness conditions}\label{sssectcondex}

\begin{df}\index{K-idempotent@$K$-idempotent}\index{idempotent!$K$-}
 	We say that $\C$ is \emph{$K$-idempotent} if the adjunction $Q\adj K$ is idempotent
	(Proposition \ref{caracdefidemp}).
\end{df}

In the 1-dimensional examples, $\Quot\!\adj\KRel$ and $\Coker\adj\Ker$, it is always true that the adjunction is idempotent.  The reason is that, as the adjunction $Q\adj K$ commutes with $\partial_0$ and $\partial$, it restricts to the fibres of these functors, and the adjunction $Q\adj K$ is idempotent if and only if, for all $A\col\C$, the adjunction
	\begin{xym}\xymatrix@C=30pt@R=45pt{
		{\Cong(A)}\ar@<2mm>[rr]^-{Q}\ar@{}[rr]|-\perp
		&&{A\backslash \C,}\ar@<2mm>[ll]^-{K}
	}\end{xym}
between the congruences on $A$ and the arrows of domain $A$ is idempotent.  Now, in the examples of dimension 1 we talked about, these adjunctions between fibres are adjunctions between \emph{orders}, and an adjunction between orders is always idempotent. On the other hand, in the examples of dimension 2, the fibre adjunctions are in general adjunctions between proper \emph{categories}, and so are not always idempotent.  For example, it happens that a cokernel is not the cokernel of its kernel, as shown in Proposition \ref{paskerkercoker}.

When $\C$ is $K$-idempotent, the properties of $K$-monomorphisms and $K$-regular epimorphisms are better. For example the cancellation law hold for $K$-regular epimorphisms.

\begin{pon}\label{loisimpkreg}
	Let us assume that $\C$ is $K$-idempotent and let 
	$A\overset{p}\ra B\overset{q}\ra C$ be arrows in $\C$.  If $p$ and $qp$ are
	$K$-regular epimorphisms, then $q$ is a $K$-regular epimorphism.
\end{pon}

	\begin{proof}
		Let us consider the following diagram in $\flc$.  The front face is a pushout.  
		Moreover, $\varepsilon_{1_B}$ is an equivalence (by the definition of kernel-quotient
		system), and $\varepsilon_p$ and $\varepsilon_{qp}$ are equivalences
		(since $p$ and $qp$ are $K$-regular epimorphisms).  So, by the universal property
		of the pushout, there exists $t\col q\ra QKq$ such that
		$\varepsilon_q t\simeq 1_q$.
		\begin{xym}\xymatrix@=20pt{
			&QKp\ar[rr]^-{QK(p,1_p,1_B)}\ar[dl]_-{\varepsilon_p}
				\ar'[d]|{\;\;\;\;QK(1_A,1_{qp},q)}[dd]
			&&QK1_B\ar[dd]^{QK(1_B,1_q,q)}\ar[dl]^-{\varepsilon_{1_B}}
			\\ p\ar[rr]_(0.25){(p,1_p,1_B)}\ar[dd]_{(1_A,1_{qp},q)}
			&&1_B\ar[dd]|(0.7){(1_B,1_q,q)}
			\\ &QK(qp)\ar'[r][rr]^(0.04){QK(p,1_{qp},1_C)}\ar[dl]_{\varepsilon_{qp}}
			&&QKq\ar[dl]^{\varepsilon_q}
			\\ qp\ar[rr]_-{(p,1_{qp},1_C)}
			&&q
		}\end{xym}
		
		By the idempotence of the adjunction $Q\adj K$, $K\varepsilon_q$ is invertible.
		Thus $Kt  K\varepsilon_q\simeq (K\varepsilon_q)^{-1}  K\varepsilon_q   Kt
		  K\varepsilon_q\simeq (K\varepsilon_q)^{-1}   K\varepsilon_q\simeq 1$.
		So we have $(QKt) (QK\varepsilon_q)\simeq 1$.  Now, in the following diagram,
		$\varepsilon_{QKq}$ is an equivalence, by idempotence.
		So $t\varepsilon_q\simeq 1$ and $\varepsilon_q$ is an equivalence.\qedhere
		\begin{xym}\xymatrix@=40pt{
			QKe_q\ar[r]^-{QK\varepsilon_q}\ar[d]_{\varepsilon_{QKq}}^{\wr}
				\drtwocell\omit\omit{}\ar@{=}@/^1.9pc/[rr]\rruppertwocell\omit{_<-2.7>{}}
			&QKq\ar[r]^-{QKt}\ar[d]^{\varepsilon_q}
				\drtwocell\omit\omit{}
			&QKe_q\ar[d]^{\varepsilon_{QKq}}_{\wr}
			\\ e_q\ar[r]_-{\varepsilon_q}
			&q\ar[r]_-{t}
			&e_q
		}\end{xym}
	\end{proof}

If we apply condition 2 of Proposition \ref{caracdefidemp} to the adjunction $Q\adj K$, we get that $\C$ is $K$-idempotent if and only if every quotient is a $K$-regular epimorphism.  In particular, we get the following proposition.

\begin{pon}\label{efkreg}
	If $\C$ is $K$-idempotent, for all $f\col\flc$, $e_f$ is a $K$-regular epimorphism.
\end{pon}

\begin{pon}\label{kmonmdesyf}
	If $\C$ is $K$-idempotent, then $A\overset{f}\ra B$ is a $K$-monomorphism if and only
	if $e_f$ is an equivalence.
\end{pon}

	\begin{proof}
		If $f$ is a $K$-monomorphism, then in the following diagram the upper arrow is an
		equivalence.  Morevoer, the right arrow is an equivalence because the adjunction $Q\adj K$ is idempotent.  We have thus $Ke_f\simeq K1_A$.
		So, by using once again the idempotence and the fact that
		$\varepsilon_{1_A}$ is an equivalence, we have
		$e_f\simeq QKe_f\simeq QK1_A\simeq 1_A$. So $e_f$ is an equivalence.
		\begin{xym}\xymatrix@C=20pt@R=30pt{
			K1_A\ar[rr]^-{K(1_A,1_{f},f)}\ar[dr]_-{K(1_A,1_{e_f},e_f)}
				\rrtwocell\omit\omit{_<3.5>{}}
			&&Kf
			\\ &Ke_f\ar[ur]_-{K\varepsilon_f=K(1_A,\varphi_f^{-1},m_f)}
		}\end{xym}
		
		Conversely, let us assume that $e_f$ is an equivalence.  Then, in the previous
		diagram, the left and right arrows are equivalences, so the upper one is
		also an equivalence and $f$ is a $K$-monomorphism.\qedhere
	\end{proof}

Proposition \ref{kmonmdesyf} and the definition of $K$-regular epimorphism show that, if $\C$ is $K$-idempotent, the full sub-$\Gpd$-categories $\E$ and $\M$ defined as in Proposition \ref{caracsimpsyf} from the $K$-regular factorisation are respectively $\KEpireg$ and $\KMono$.  Moreover, condition 2 of that proposition hold, by Proposition \ref{kmonintkregestequ}.  And Proposition \ref{efkreg} above gives the first half of condition 1 of Proposition \ref{caracsimpsyf}.  It remains thus only one condition to have a factorisation system: that, for all $f\col\flc$, $m_f\in\KMono$.  

\begin{df}\label{defkfact}\index{K-factorisable@$K$-factorisable}\index{factorisable!K-@$K$-}
	We say that $\C$ is \emph{$K$-factorisable} if for all $f\col\flc$, $m_f\in\KMono$.
\end{df}

\begin{pon}\label{KfactKidem}
	If $\C$ is $K$-factorisable, then $\C$ is $K$-idempotent.
\end{pon}

	\begin{proof}
		Let be $f\col\flc$.
		By applying Proposition \ref{simplnoysysteme}
		to $\varphi_f\col f\Ra m_fe_f$, since $m_f$ is a $K$-monomorphism,
		the arrow $K\varepsilon_f=K(1_A,\varphi_f,m_f)\col KQKf\ra Kf$ is an equivalence.
		So $K\varepsilon$ is an equivalence and $\C$ is $K$-idempotent.
	\end{proof}

Thus, if $\C$ is $K$-factorisable, $(\KEpireg,\KMono,\im,e,m,\varphi)$ is a factorisation system.

There is another point of view on factorisation systems, which we won't study in detail here: to give a factorisation system amounts to give a coreflexive full sub-$\Gpd$-category $\E\hookrightarrow\flc$ which is stable under composition and which contains all the identities (see \cite{Im1986a} for the 1-dimensional version).  From this point of view, what the regular factorisation is lacking to form a factorisation system is the stability under composition of $K$-regular epimorphisms. The following proposition shows that we can also use this condition to define $K$-factorisability.

\begin{lemm}\label{lemmpourcarfact}
	Let us assume that $\C$ is $K$-idempotent.
	Let $A\overset{f}\ra B$ be an arrow in $\C$ and $A\overset{e}\ra E$
	a $K$-regular epimorphism.
	Then
	\begin{eqn}
		\varepsilon_f\circ -\col (A\backslash\C)(e,e_f)\ra (A\backslash\C)(e,f)
	\end{eqn}
	is an equivalence.
\end{lemm}

	\begin{proof}
		In the following diagram, the left upper and lower arrows are equivalences, because
		$e$ is a $K$-regular epimorphisms; the right upper and lower arrows are equivalences,
		by the adjunction $Q\adj K$; the right arrow is an equivalence by idempotence
		of the adjunction $Q\adj K$.  So the left arrow is an equivalence.
		\begin{xym}\xymatrix@=40pt{
			{\flc(e,e_f)}\ar[d]_-{\varepsilon_f\circ-}\ar[r]^-{-\circ\varepsilon_e}
				\drtwocell\omit\omit{}
			&{\flc(QKe,e_f)}\ar[d]^-{\varepsilon_f\circ-}\ar[r]^-{\sim}\drtwocell\omit\omit{}
			&{\Cong(\C)(Ke,Ke_f)}\ar[d]^-{K\varepsilon_f\circ-}
			\\ {\flc(e,f)}\ar[r]_-{-\circ\varepsilon_e}
			&{\flc(QKe,f)}\ar[r]_-{\sim}
			&{\Cong(\C)(Ke,Kf)}
		}\end{xym}
		
		Then, in the following diagram, the front and back faces are pullbacks, and the
		three edges between the right lower corner of the back face and the right lower corner of the front face are equivalences.  So the edge from the left upper corner is an equivalence.\qedhere
		\begin{xym}\xymatrix@!@R=-30pt@C=-20pt{
			&{(A\backslash\C)(e,e_f)}\ar[rr]^-{}\ar[dl]_-{\varepsilon_f\circ -}
				\ar'[d][dd]
			&&{\flc(e,e_f)}\ar[dd]^{\partial_0}\ar[dl]_-{\varepsilon_f\circ -}
			\\ {(A\backslash\C)(e,f)}\ar[rr]_(0.25){}\ar[dd]
			&&{\flc(e,f)}\ar[dd]^(0.3){\partial_0}
			\\ &1\ar'[r]^{1_A}[rr]\ar@{=}[dl]
			&&{\C(A,A)}\ar@{=}[dl]
			\\ 1\ar[rr]_-{1_A}
			&&{\C(A,A)}
		}\end{xym}
	\end{proof}

\begin{pon}\label{defcaracfactor}\index{K-factorisable@$K$-factorisable!characterisation}
	The following conditions are equivalent:
	\begin{enumerate}
		\item $\C$ is $K$-factorisable;
		\item $\C$ is $K$-idempotent and $K$-regular epimorphisms are stable under composition.
	\end{enumerate}
\end{pon}

	\begin{proof}
		{1 $\Rightarrow$ 2. } We already know, by Proposition \ref{KfactKidem}, that
		$\C$ is $K$-idempotent.
		
		Then let $A\overset{p}\ra B\overset{q}\ra C$ be $K$-regular epimorphisms.
		Let us consider the following diagram.
		By $K$-idempotence, $e_p$ and $e_{qp}$ are $K$-regular epimorphisms
		(by Proposition \ref{efkreg}). Thus, by Proposition \ref{loisimpkreg},
		$\im(1_A,1_{qp},q)$ is a $K$-regular epimorphism.  Now, since
		$p$ is a $K$-regular epimorphism, $m_p$ is an equivalence. So
		$qm_p\simeq q$ is a $K$-regular epimorphism and, again by Proposition
		\ref{loisimpkreg}, $m_{qp}$ is a $K$-regular epimorphism.  As $\C$
		is $K$-factorisable, $m_{qp}$ is also a $K$-regular epimorphism.
		Thus $m_{qp}$ is an equivalence, by Proposition
		\ref{kmonintkregestequ}, and $qp$ is a $K$-regular epimorphism.
		\begin{xym}\xymatrix@=40pt{
			A\ar@{=}[d]\ar[r]^-{e_p}\drtwocell\omit\omit{}
			&{\im p}\ar[d]|{\im(1_A,1_{qp},q)}\ar[r]^-{m_p}_-{\sim}\drtwocell\omit\omit{}
			&B\ar[d]^q
			\\A\ar[r]_-{e_{qp}}
			&{\im(qp)}\ar[r]_-{m_{qp}}
			&C
		}\end{xym}
		
		{2 $\Rightarrow$ 1.} We apply the regular factoristion to $m\eqdef
		m_f$.  This gives the following diagram where, by idempotence, $e_f$ and $e_m$
		are $K$-regular epimorphisms.  By stability under composition of
		$K$-regular epimorphisms, $e\eqdef e_me_f$ is also a $K$-regular epimorphism.
		\begin{xym}\xymatrix@=40pt{
			A\ar[r]^-f\ar[d]_{e_f}
				\rtwocell\omit\omit{_<2.5>\varphi_f\;\;\;\;\;\;\;\;\;\;\;\;}
			&B
			\\ {\im f}\ar[ur]^(0.35){m_f}\ar[r]_-{e_m}
				\rtwocell\omit\omit{_<-2.5>\;\;\;\;\varphi_m}
			&{\im m}\ar[u]_{m_m}
		}\end{xym}
		
		We apply the previous lemma to $f$ and $e$.  To
		$(1_A,\varphi_f^{-1}\circ\varphi_m^{-1} e_f,m_m)\col e\ra f$ correspond
		$(1_A,\alpha,a)\col e\ra e_f$ and $\omega\col m_fa\Ra m_m$ such that
		$\omega e = \varphi_me_f\circ m_f\alpha$.  We will prove that
		$a$ is an inverse for $e_m$.
		
		We apply again the previous lemma to $f$ and $e_f$.
		Then $\varphi^{-1}_m\circ\omega e_m\col m_fae_m\Ra m_f$ which is a
		2-arrow $\varepsilon_f\circ (1_A,\alpha,ae_m)\Ra \varepsilon_f
		\circ (1_A,1_{e_f},1_{\im f})\col\allowbreak e_f\ra f$; since $\varepsilon_f\circ -$
		is fully faithful, there exists $\omega'\col ae_m\Ra 1$
		such that $\omega e_m=\varphi_m\circ m_f\omega'$.
		
		We apply once again the previous lemma, but to $m_f$ and $e_m$.
		We have $\omega\circ\varphi_m^{-1} a\col m_me_ma\Ra m_m$, which is a
		2-arrow $\varepsilon_m\circ(1_{\im f},e_m\omega',e_ma)\Ra
		\varepsilon_m\circ(1_{\im f},1_{e_m},1_{\im m})\col e_m\Ra m_f$.
		As $\varepsilon_m\circ -$ is fully faithful, there exists an isomorphism $e_ma\Ra 1$.
		
		So $e_m$ is an equivalence and, by Proposition \ref{kmonmdesyf},
		which we can use thanks to the $K$-idempotence, $m_f$ is a $K$-monomorphism.
	\end{proof}

\begin{df}\label{defkpreex}\index{K-preexact@$K$-preexact}\index{preexact!K-@$K$-}
	We say that $\C$ is \emph{$K$-preexact} if for every congruence $H$, the unit
	$\eta_H\col H\ra KQH$
	is an equivalence (every congruence is canonically the kernel of its quotient).
\end{df}

If $\C$ is $K$-preexact, the adjunction $Q\adj K$ restricts to an equivalence
	\begin{xym}\xymatrix@C=30pt@R=45pt{
		{\Cong(\C)}\ar@<2mm>[rr]^-{Q}\ar@{}[rr]|-\simeq
		&&{K\text{-}\Epireg.}\ar@<2mm>[ll]^-{K}
	}\end{xym}
If we restrict to the fibres in $A\col\C$, we have thus an equivalence
	\begin{xym}\xymatrix@C=30pt@R=45pt{
		{\Cong(A)}\ar@<2mm>[rr]^-{Q}\ar@{}[rr]|-\simeq
		&&{A\backslash K\text{-}\Epireg,}\ar@<2mm>[ll]^-{K}
	}\end{xym}
between the congruences on $A$ and the $K$-regular epimorphisms with domain $A$ (the quotients of $A$).

\subsection{Duality}\label{soussectdualsysnoyquot}

Let us come back to the situation of the beginning of Subsection \ref{soussectdefsysnoyquot}, where $\V$ is a symmetric monoidal closed category, $\EM$ is a factorisation system on $\V$, and $\W$ is a full subcategory of $\E$ generating $\EM$. We can define, in a $\V$-category, not only $\W$-kernels, $\W$-quotients, $\W$-congruences and the corresponding notions of monomorphism, regular epimorphism, etc., but also the dual notions: $\W$-cokernels $\mathring{K}_\W$, $\W$-coquotients $\mathring{Q}_\W$, $\W$-cocongruences, $\mathring{K}_\W$-epimorphisms, $\mathring{K}_\W$-regular monomorphisms, etc.

In the case of $\Ens$ with $(\Surj,\Inj)$ and $\W=\left\{\overset{2}{\underset{1}{\downarrow}},\overset{1}{\underset{1}{\downarrow}}\right\}$, there are two remarkable properties:
\begin{enumerate}
	\item in $\Ens$, $f$ is an epimorphism if and only if
		$f$ is a surjection;
	\item every quotient is an epimorphism; every coquotient is a monomorphism.
\end{enumerate}
In the case of $\Ensp$ with $(\Bij^*,\zInj)$ and $\W=\left\{\overset{2}{\underset{1}{\downarrow}},\overset{2}{\underset{2}{\parallel}}\right\}$, there are two slightly less remarkable properties (the first is not an equivalence):
\begin{enumerate}
	\item in $\Ensp$, if $f$ is a bijection outside the kernel,
		then $f$ is a 0-epimorphism (i.e.\ a surjection, by Proposition
		\ref{caraczepiensp});
	\item every cokernel is a 0-epimorphism and every kernel is a 0-monomorphism.
\end{enumerate}

These properties won't hold in dimension 2.  For example for $\Gpdp$-categories, the notion of monomorphism corresponding to the adjunction cokernel $\adj$ kernel is the notion of fully 0-faithful arrow, but it is not true that every cokernel is fully 0-faithful. 
But it was the second of these properties that allowed us to define the comparison arrow between the quotient of the kernel and the coquotient of the cokernel.

In dimension 2, instead of having one factorisation system which induces notions of kernel, quotient, etc., compatible with the dual notions induced by the same factorisation system, we have typically crossed situations: the kernels/quotients of system 1 are compatible with the cokernels/coquotients of system 2 (there is a comparison arrow between the quotient of the kernel 1 and the coquotient of the cokernel 2) and, conversely, the kernels/quotients of system 2 are compatible with the cokernels/coquotients of system 1.

Let us introduce the following definition \cite{Dupont2008a}.

\begin{df}\label{defsyfcouples}\index{factorisation system!precoupled}\index{precoupled!factorisation systems}
	Two factorisation systems $(\E_1,\M_1)$ and $(\E_2,\M_2)$ on $\V$ are
	\emph{precoupled} if the following implications hold\footnote{These two conditions are in fact equivalent.}:
	\begin{align}\stepcounter{eqnum}
		f\in\E_1\; &\Rightarrow\;\text{for all }Y\col\V\text{, }[f,Y]\in\M_2\text{ and}\\ 
		\stepcounter{eqnum}
		f\in\E_2\; &\Rightarrow\;\text{for all }Y\col\V\text{, }[f,Y]\in\M_1.
	\end{align}
	We say that the systems are \emph{coupled}\index{factorisation system!coupled}%
	\index{coupled factorisation systems} if these implications are equivalences.
\end{df}

The first property above for $(\Surj,\Inj)$ on $\Ens$ can be translated in the following way: the factorisation system $(\Surj,\Inj)$ is self-coupled, whereas the corresponding property for $(\Bij^*,\zInj)$ on $\Ensp$ means that $(\Bij^*,\zInj)$ is self-precoupled.  In the 
2-dimensional examples, we will meet situations where we have two different factorisation systems which are coupled or precoupled.  For examples, for groupoids or symmetric 2-groups, the systems $(\PlSurj,\Fid)$ and $(\Surj,\PlFid)$ are coupled (this is proved, for symmetric  2-groups, at the end of Section \ref{sectdeffactcgs} and, for groupoids, in \cite{Dupont2008b}).  In dimension 3, we should expect to have a chain of three factorisation systems, the first and the last ones being (pre)coupled and the central one being self-(pre)coupled.

If $(\E_1,\M_1)$ (with generator $\W_1$)
and $(\E_2,\M_2)$ (with generator $\W_2$) are two precoupled factorisation systems on $V$, we can prove the following properties (see \cite{Dupont2008a}), which are the second properties we talked about at the beginning of this subsection for the 1-dimensional examples:
\begin{enumerate}
	\item for every $\W_1$-congruence $H$, $Q_{\W_1}H$ is a
		$\mathring{K}_{\W_2}$-epimorphism;
	\item for every $\W_2$-cocongruence $H$, $\mathring{Q}_{\W_2}H$ is a
		$K_{\W_1}$-monomorphism.
\end{enumerate}
When a kernel-quotient system and a cokernel-coquotient system (dual of kernel-quotient system) satisfy these properties, we will say that they are \emph{precoupled}.  In such a context we will be able to define the comparison arrow between the regular and coregular factorisations and to develop the theory of perfect categories.  For a cokernel-coquotient system $\mathring{K}_2\adj\mathring{Q}_2$, we will call $\mathring{K}_2$-epimorphism\index{epimorphism!K-@$\mathring{K}$-}\index{K-epimorphism@$\mathring{K}$-epimorphism} the dual notion of $K$-monomorphism.

\begin{df}\label{defsysnoyquconcoqprec}\index{kernel-quotient system!precoupled}\index{precoupled!kernel-quotient systems}
	A kernel-quotient system $Q_1\adj K_1$ and a cokernel-coquotient system
	$\mathring{K}_2\adj\mathring{Q}_2$ are \emph{precoupled} if
	\begin{enumerate}
		\item every quotient $Q_1H$ is a $\mathring{K}_2$-epimorphism;
		\item every coquotient $\mathring{Q}_2H$ is a $K_1$-monomorphism.
	\end{enumerate}
\end{df}

	\begin{xym}\xymatrix@=50pt{
		{\Cong_1(\C)}\ar@<2mm>[r]^-{Q_1}\ar@<-3mm>[dr]_-{\partial}\ar@{}[r]|-\perp
		&{\fl{\C}}\ar@<2mm>[l]^-{K_1}\ar@<-2mm>[d]_-{\partial_0}
			\ar@<2mm>[d]^-{\partial_1}\ar@<2mm>[r]^-{\mathring{K}_2}\ar@{}[r]|-\perp
		&{\caspar{Cocong}_2(\C)}\ar@<2mm>[l]^-{\mathring{Q}_2}
			\ar@<3mm>[dl]^-{\partial}
		\\ &{\C}\ar@<0mm>[u]|-{\mathrm{1_{-}}}
			\ar@<1mm>[ul]|-{\;\;\;\;\,K_1 1_{-}}
			\ar@<-1mm>[ur]|-{\mathring{K}_2 1_{-}\;\;\;\;\;\;}
	}\end{xym}

For the remaining of this section, let us fix a kernel-quotient system and a cokernel-coquotient system which are precoupled.  Each arrow factors in two ways: $K_1$-regular factorisation $f\simeq m^1_fe^1_f$ for $Q_1\adj K_1$ and $\mathring{K}_2$-coregular factorisation $f\simeq m^2_fe^2_f$ for $\mathring{K}_2\adj\mathring{Q}_2$.  Thanks to the precoupling, there is a comparison arrow between these two factorisations.  In fact, $e_f^1$ is $Q_1K_1f$ and is thus a $\mathring{K}_2$-epimorphism, by the precoupling, and $m_f^2$ is the coquotient $\mathring{Q}_2\mathring{K}_2f$; so, by Proposition \ref{quotestorthamon}, $e_f^1\orth m_f^2$.

We have thus an arrow $w_f$ (which forms a $\Gpd$-natural transformation) and 2-arrows $\varepsilon_f$ and $\mu_f$ (which form modifications) such that
\begin{xyml}\begin{gathered}\xymatrix@C=30pt@R=30pt{
	&{\im^1\! f}\ar[dr]^-{m^1_f}\ar[dd]_(0,4){w_f}
	&{}\ar@{}[dl]^(0.86){\mu_f\!\!\dir{=>}}
	\\A\ar[dr]_-{e^2_f}\ar[ur]^{e^1_f}
	&{}\ar@{}[dl]_(0.3){\varepsilon_f\!\!\dir{=>}} &B
	\\ {}&{\im^2\! f}\ar[ur]_-{m^2_f}
}\end{gathered}\;\;=\;\;\begin{gathered}\xymatrix@C=30pt@R=30pt{
	&{\im^1\! f}\ar[dr]^-{m^1_f}
	\\A\ar[rr]|f\ar[dr]_-{e^2_f}
		\rrtwocell\omit\omit{_<3.8>\;\;\;\varphi^2_f}\ar[ur]^-{e^1_f}
		\rrtwocell\omit\omit{^<-3.8>\varphi^1_f\;\;}
	&&B
	\\ &{\im^2\! f}\ar[ur]_-{m^2_f}
}\end{gathered}.\end{xyml}

\begin{thm}\label{defcaracparf}\index{K1-K2-perfect@$K_1$-$\mathring{K}_2$-perfect}\index{perfect!K1-K2@$K_1$-$\mathring{K}_2$-}
	The following conditions are equivalent. When they hold, we say that $\C$ is \emph{$K_1$-$\mathring{K}_2$-perfect}.
	\begin{enumerate}
		\item For every $f\col\flc$, $w_f$ is an equivalence.
		\item Every arrow factors as a $K_1$-regular epimorphism followed by
			$\mathring{K}_2$-regular monomorphism.
		\item Every ${K}_1$-monomorphism is a $\mathring{K}_2$-regular monomorphism
			and every $\mathring{K}_2$-epimorphism is a $K_1$-regular epimorphism.
	\end{enumerate}
\end{thm}

	\begin{proof}
		{\it 1 $\Rightarrow$ 2. } For every $f\col\flc$, $m^2_f$ is a $K_1$-monomorphism,
		thanks to the precoupling.  Now, if $w_f$ is an equivalence, $m^1_f\simeq m^2_f$,
		so $m^1_f$ is also a $K_1$-monomorphism.  Thus $\C$ is $K_1$-factorisable.
		
		Hence, by Proposition \ref{KfactKidem}, $\C$ is $K_1$-idempotent
		and, by Proposition \ref{efkreg}, $e_f^1$ is a $K_1$-regular epimorphism.
		Dually, $m_f^2$ is a $\mathring{K}_2$-regular monomorphism.
		
		{\it 2 $\Rightarrow$ 3. } Let $f$ be a $K_1$-monomorphism.  By condition 2,
		there exist $e\in K_1\text{-}\Epireg$ and $m\in\mathring{K}_2\text{-}\Monoreg
		\incl K_1\text{-}\Mono$ and a 2-arrow $f\Ra me$.
		By the cancellation property of
		$K_1$-monomorphisms (Corollary \ref{kmonstabcompsimp}), $e$ is also a
		$K_1$-monomorphism.  Thus, by Proposition \ref{kmonintkregestequ}, 
		$e$ is an equivalence and $f\simeq m\in\mathring{K}_2\text{-}\Monoreg$.
		The other property is proved dually.
		
		{\it 3 $\Rightarrow$ 1. } Condition 3 implies that $\mathring{K}_2$-regular monomorphisms are stable under composition, because
		$K_1$-monomorphisms are (Corollary \ref{kmonstabcompsimp}).
		Dually, $K_1$-regular epimorphisms are stable under composition.
		
		Then, $\C$ is $K_1$-idempotent because every quotient is a
		$\mathring{K}_2$-epimorphism and thus a $K_1$-regular epimorphism
		($\varepsilon_{Q_1}\col Q_1K_1Q_1\Ra Q_1$ is an equivalence).  Dually, $\C$
		is $\mathring{K}_2$-idempotent.
		
		So, by Proposition \ref{defcaracfactor}, $m^1_f$ is a
		$K_1$-monomorphism and $e^2_f$ is a $\mathring{K}_2$-epimorphism.
		By the cancellation law (Corollary \ref{kmonstabcompsimp}), the arrow $w_f$
		belongs both to $K_1\text{-}\Mono\incl \mathring{K}_2\text{-}\Monoreg$
		and to $\mathring{K}_2\text{-}\Epi$.  Thus $w_f$ is an equivalence, by Proposition	\ref{kmonintkregestequ}.
	\end{proof}

If $\C$ is $K_1$-$\mathring{K}_2$-perfect, $\C$ is $K_1$-factorisable and $\mathring{K}_2$-factorisable and
\begin{eqn}
	(\mathring{K}_2\text{-}\Epi,K_1\text{-}\Mono) =
	(K_1\text{-}\Epireg,\mathring{K}_2\text{-}\Monoreg)
\end{eqn}
is a factorisation system.

\chapter{Kernels and pips}

\begin{quote}
	{\it This chapter introduces the basic notions which we use in the following
	chapters.  After the definition of pointed groupoid enriched categories, we introduce
	different kinds of arrows and objects, and different kinds of limits: (co)kernel,
	(co)pip and (co)root.}
\end{quote}

\section{Generalities}

\subsection{$\Gpdp$-categories}

	Pointed groupoids play in dimension 2 the rôle that pointed sets played in dimension 1. Pointed groupoid enriched categories will be the natural context where it makes sense to speak of kernel, cokernel and the corresponding notions of “monomorphism” and
	“epimorphism”.

	\begin{df}\label{deftoutpointe}~
		\begin{enumerate}
			\item A \emph{pointed groupoid}\index{pointed groupoid}\index{groupoid!pointed}
				consists of $(\A,I)$, where
				$\A$ is a groupoid and $I\col\A$.
			\item A \emph{pointed functor}\index{functor!pointed}\index{pointed functor}
				$(\A,I)\ra(\B,I)$ consists of a functor $F\col\A\ra\B$ and an arrow $\varphi_0\col I\ra FI$.
			\item A \emph{pointed natural transformation}%
				\index{natural transformation!pointed}%
				\index{pointed natural transformation}
				$(F,\varphi_0)\Ra(F',\varphi'_0)\col(\A,I)\ra(\B,I)$ is a 
				natural transformation $\alpha\col F\Ra F'$ such that
				$\alpha_I\varphi_0=\varphi'_0$.
		\end{enumerate}
	\end{df}
	
	\begin{pon}
		Pointed groupoids, pointed functors between them and poin\-ted natural transformations between them form a
		$\Gpd$-category $\Gpdp$\index{Gpd*@$\Gpdp$}, the composition
		of morphisms being defined by
		\begin{eqn}
			(G,\psi_0)\circ(F,\varphi_0)=(GF,G\varphi_0\circ \psi_0)
		\end{eqn}
		and the horizontal and vertical compositions of 2-morphisms being the usual compositions of natural transformations.
	\end{pon}

	We will often use the following coherence result.
	
	\begin{pon}\label{metthcoh}
		For any pointed functor $(F,\varphi_0)\col(\A,I)\ra(\B,I)$
		there exists a functor $F'\col\A\ra\B$ isomorphic to $F$ and described in such a way
		that $F'I$ coincides with $I$ and $\varphi_0$ is the identity
		(we say that $F'$ is
		\emph{strictly described}\index{pointed functor!strictly described}).
	\end{pon}
	
		\begin{proof}
			We define $F'\col\A\ra\B$, on objects by $F'I\eqdef I$
			and $F'A\eqdef FA$, if $A\not\equiv I$, and on arrows by the following
			composites (for $A, A'\not\equiv I$):
			\begin{eqn}
				F'_{II}\eqdef\A(I,I)\xrightarrow{F_{II}}\B(FI,FI)
				\xrightarrow{\varphi_0^{-1}\circ-\circ\varphi_0}\B(I,I);
			\end{eqn}
			\begin{eqn}
				F'_{AI}\eqdef\A(A,I)\xrightarrow{F_{AI}}\B(FA,FI)
				\xrightarrow{\varphi_0^{-1}\circ-}\B(FA,I);
			\end{eqn}
			\begin{eqn}
				F'_{IA'}\eqdef\A(I,A')\xrightarrow{F_{IA'}}\B(FI,FA')
				\xrightarrow{-\circ\varphi_0}\B(I,FA');
			\end{eqn}
			\begin{eqn}
				F'_{AA'}\eqdef\A(A,A')\xrightarrow{F_{AA'}}\B(FA,FA').
			\end{eqn}
			We define a pointed natural transformation $\omega\col F'\Ra F$
			by $\omega_I\eqdef\varphi_0$ and $\omega_A\eqdef 1_{FA}$, if $A\not\equiv I$.
		\end{proof}
		
	If $F,G\col(\A,I)\ra(\B,I)$ are strictly described pointed functors,
	a pointed natural transformation $\alpha\col F\Ra G$ is simply a
	natural transformation	such that $\alpha_I=1_I$.

	The internal $\Hom$ of $\Gpdp$ is given by the groupoids
	$[\A,\B]\eqdef\Gpdp(\A,\B)$
	with as distinguished object the constant functor $0\col\A\ra\B$, which
	maps every object of $\A$ to $I$ and every arrow of $\A$ to $1_I$.

	\begin{df}\index{functor!bipointed}\index{bipointed functor}
		Let be $\A, \B, \cat{Y}\col\Gpdp$.  A \emph{bipointed functor}
		$F\col\A\times\B\ra\cat{Y}$ is a functor $F\col\A\times\B\ra\cat{Y}$
		with isomorphisms natural in each variable
		\begin{align}\stepcounter{eqnum}
		\begin{split}
			\varphi_0^B &\col I\ra F(I,B),\\
			\psi^0_A &\col I\ra F(A,I),
		\end{split}
		\end{align}
		such that the natural transformations $\varphi_0^-\col 0\Ra F(I,-)$ 
		and $\psi^0_-\col 0\Ra F(-,I)$ are pointed (these two conditions are
		equivalent and mean that
		\begin{eqn}
			\varphi_0^I=\psi^0_I\col I\ra F(I,I)\text{).}
		\end{eqn}
	\end{df}
	
	\begin{df}
		Let be $\A, \B, \cat{Y}\col\Gpdp$. The pointed groupoid
		$\mathrm{Bipt}(\A\times\B,\cat{Y})$%
		\index{Bipt(AxB,Y)@$\mathrm{Bipt}(\A\times\B,\cat{Y})$} is defined in the following way.
		\begin{enumerate}
			\item {\it Objects.} These are the bipointed functors $\A\times\B\ra\cat{Y}$.
			\item {\it Arrows.}  These are the natural transformations
				$\alpha:F\Ra F'\col\A\times\B\ra\cat{Y}$ such that, for all $A\col\A$,
				$\alpha_{(A,-)}\col(F(A,-),\psi^0_A)\Ra (F'(A,-),\psi'^0_A)$ is
				pointed and, for all $B\col\B$, $\alpha_{(-,B)}\col(F(-,B),\varphi_0^B)
				\Ra (F'(-,B),\varphi'^B_0)$ is pointed
				(we say that such an $\alpha$ is a
				\emph{bipointed natural transformation}%
				\index{natural transformation!bipointed}%
				\index{bipointed natural transformation}).
			\item {\it Point.} This is the constant functor $0$.
		\end{enumerate}
	\end{df}
	
	It is very easy to check the following property, which shows that the bipointed functors
	 $\A\times\B\ra\cat{Y}$
	are equivalent to the pointed functors $\A\tens\B\ra\cat{Y}$; in this way we avoid the need to define $\A\tens\B$.
	
	\begin{pon}
		$\mathrm{Bipt}(\A\times\B,\cat{Y})\simeq [\A,[\B,\cat{Y}]]$
	\end{pon}
	
	We can now give a definition of $\Gpdp$-categories.

	\begin{pon}
		Let $\C$ be a $\Gpd$-category such that, for all $A,B\col\C$,
		the grou\-po\-id $\C(A,B)$ is pointed (we write the point $0^A_B$; we usually omit the superscript and subscripts),
		and equipped with natural transformations
		\begin{align}\begin{split}\stepcounter{eqnum}
			\varphi_0^h&\col 0\Ra h0;\\
			\psi_g^0&\col 0\Ra 0g.
		\end{split}\end{align}
		We say that $\C$ is a \emph{pointed groupoid enriched category} (for short \emph{$\Gpdp$-category}%
		\index{Gpd*-category@$\Gpdp$-category}) if the following equivalent
		conditions 1 and 2 hold.
		\begin{enumerate}
			\item \begin{enumerate}
				\item For all $A,B,C\col\C$, the composition functor
					$\C(A,B)\times\C(B,C)\ra\C(A,C)$, equipped with $\varphi_0^h$ and $\psi_g^0$,
					is bipointed.
				\item For all $A,B,C,D$, the associativity natural transformation $\alpha$ (diagram \ref{diagassocgpdcat}) is tripointed.
				\item For all $A,B$, the neutrality natural transformations $\rho$ and $\lambda$
					(diagrams
					\ref{diagneutrdroitgpdcat} and \ref{diagneutrgauchgpdcat})
					are pointed.
				\end{enumerate}
			\item \begin{enumerate}
				\item For all $A$ and $h\col B\ra C$ in $\C$, the functor
					$h\circ -\col\C(A,B)\ra\C(A,C)$, equipped with $\varphi_0^h$,
					is pointed.
				\item For all $g\col A\ra B$ and $C$ in $\C$, the functor
					$-\circ g\col\C(B,C)\ra\C(A,C)$, equipped with $\psi_g^0$, is pointed.
				\item The transformation $\psi_{-}^0$ is pointed
					(or, equivalently, the transformation
					$\varphi^{-}_0$ is pointed).
				\item For all $g\col B\ra C$ and $h\col C\ra D$, the following natural
						transformation, which expresses a part of the associativity,
					is pointed.
					\begin{xym}\xymatrix@=30pt{
						{\C(A,B)}\ar[r]_{g\circ -}\rruppertwocell^{hg\circ -}<10>
							{_<-3>{\;\;\;\;\;\;\,\alpha_{hg-}}}
						&{\C(A,C)}\ar[r]_{h\circ -}
						&{\C(A,D)}
					}\end{xym}
				\item For all $f\col A\ra B$ and $h\col C\ra D$, the following natural
						transformation, which expresses an other part of the associativity,
					is pointed.
					\begin{xym}\xymatrix@=30pt{
						{\C(B,C)}\ar[r]^{h\circ-}\ar[d]_{-\circ f}
							\drtwocell\omit\omit{\;\;\;\;\;\;\alpha_{h-f}}
						&{\C(B,D)}\ar[d]^{-\circ f}
						\\ {\C(A,C)}\ar[r]_{h\circ-}
						&{\C(A,D)}
					}\end{xym}
				\item For all $f\col A\ra B$ and $g\col B\ra C$, the following natural
						transformation, which expresses the last part of the associativity,
					is pointed.
					\begin{xym}\xymatrix@=30pt{
						{\C(C,D)}\ar[r]^{-\circ g}\rrlowertwocell_{-\circ gf}<-10>
							{_<3>{\;\;\;\;\;\;\,\alpha_{-gf}}}
						&{\C(B,D)}\ar[r]^{-\circ f}
						&{\C(A,D)}
					}\end{xym}
				\item For all $A,B\col\C$, the neutrality natural transformations are pointed.
					\begin{xym}\xymatrix@=30pt{
						{\C(A,B)}\ruppertwocell^{1_B\circ-}{\lambda}\ar@{=}[r]
							\rlowertwocell_{-\circ 1_A}{^{\rho}}
						&{\C(A,B)}
					}\end{xym}
				\end{enumerate}
		\end{enumerate}
	\end{pon}
	
	\begin{pon}\index{Gpd*-category@$\Gpdp$-category!strictly described}
		For every $\Gpdp$-category $\C$, we can construct an equivalent $\Gpdp$-category $\C'$
		which is strictly described as a $\Gpd$-category 
		($f(g h)\equiv (f g) h$, $f 1_A\equiv f$ and
		$1_B f\equiv f$) and where the composition functors are strictly described as
		bipointed functors:
		\begin{eqn}
			f 0 \equiv 0 \text{ and } 0 f\equiv 0;
		\end{eqn}
		\begin{eqn}
			\alpha 0 \equiv 1_0 \text{ and } 0\alpha \equiv 1_0.
		\end{eqn}
	\end{pon}
	
		\begin{proof}
			We know that every $\Gpd$-category $\C$ is equivalent to a 
			strictly described $\Gpd$-category $\C'$.  The bipointed structure of
			the composition functors can be transferred to $\C'$.  Then we apply
			Proposition \ref{metthcoh} to replace the composition functors by
			strictly described bipointed functors.  This doesn't break the strictness
			of associativity and neutrality.
		\end{proof}
	
	\begin{df}\index{object!zero}\index{zero object}
		A \emph{zero object} in a $\Gpdp$-category $\C$ is an object $0$ such that
		for every
		arrow $f\col X\ra 0$, there exists a unique $\varphi\col f\Ra 0$ and
		for every arrow $g\col 0\ra Y$, there exists a unique $\psi\col g\Ra 0$.
	\end{df}

	\begin{pon}
		An object $0$ is a zero object if and only if $1_0\simeq 0$.
	\end{pon}

	Usually we write $0^A\col A\ra 0$ instead of
	$0^A_0$ and $0_B\col 0\ra B$ instead of $0^0_B$.
	In the following of this work, unless explicitly stated otherwise, we always assume that $\Gpdp$-categories
	are strictly described.

\subsection{Taxonomy of arrows, objects and loops}

	In $\Gpd$-categories, there are two notions of monomorphism (faithful and
	fully faithful) and two notions of epimorphism (cofaithful and fully
	cofaithful); the definitions of the first two notions can be found in \cite{Street1982b}; the last two are dual and are studied in \cite{Dupont2003a}.
	
	\begin{df}\label{defsortfleches}
		Let $A\overset{f}\ra B$ be an arrow in a $\Gpd$-category $\C$.  We say that
		\begin{enumerate}
			\item $f$ is \emph{faithful}\index{arrow!faithful}\index{faithful arrow}
				if, for all $X\col \C$, the functor
				$f\circ - \col \C(X,A)\ra\C(X,B)$ is faithful in $\Gpd$,
				i.e.\ if, for all $X\col\C$ and for all
				$\alpha,\alpha'\col a\Ra a'\col X\ra A$, if $f\alpha
				=f\alpha'$, then $\alpha=\alpha'$;
			\item $f$ is \emph{fully faithful}\index{arrow!fully faithful}%
				\index{fully faithful arrow} if, for all $X\col \C$,
				$f\circ -\col \C(X,A)\ra\C(X,B)$ is full and faithful in $\Gpd$,
				i.e.\ if, for all $X\col\C$, for all $a,a'\col X\ra A$
				and for all $\beta\col fa\Ra fa'$, there exists a unique $\alpha\col a\Ra a'$
				such that $\beta=f\alpha$;
			\item $f$ is \emph{cofaithful}\index{arrow!cofaithful}%
				\index{cofaithful arrow} if, for all $Y\col \C$,
				$-\circ f\col \C(B,Y)\ra\C(A,Y)$ is faithful in $\Gpd$;
			\item $f$ is \emph{fully cofaithful}\index{arrow!fully cofaithful}%
				\index{fully cofaithful arrow} if, for all $Y\col \C$,
				$-\circ f\col \C(B,Y)\ra\C(A,Y)$ is full and faithful in $\Gpd$.
		\end{enumerate}
	\end{df}
	
	Since all 2-arrows are invertible in a $\Gpd$-category, we can
	simplify the definition of faithful arrow.
	
	\begin{pon}
		Let $A\overset{f}\ra B$ be an arrow in a $\Gpd$-category $\C$.
		Then $f$ is faithful
		if and only if, for all $\alpha\col a\Ra a\col X\ra A$, 
		if $f\alpha = 1_{fa}$, then $\alpha = 1_a$.
	\end{pon}
	
	Remark: we say \emph{fully faithful} and not \emph{full and faithful},
	because the condition that, for all $X\col\C$, $\C(X,f)$ be full is not equivalent in $\Gpd$ to $f$ being full. Moreover, in $\Gpd$, this condition implies faithfulness.  We will define (Definition \ref{deffullgpdcat})
	a notion of full arrow in a $\Gpd$-category which, in $\Gpd$ and $\CGS$,
	gives back the ordinary full functors.
	
	In $\Gpd$ and $\CGS$ (see \cite{Dupont2008b} for $\Gpd$ and Section 
	\ref{sectcarcflechcgs} for $\CGS$), the faithful arrows are the ordinary faithful functors,
	the fully faithful arrows are the ordinary fully faithful functors,
	the cofaithful arrows are the surjective functors, and the fully cofaithful arrows
	are the full and surjective functors (these last two properties are not
	true in the 2-category of categories, as the characterisation
	of cofaithful and fully cofaithful arrows in $\Cat$ given by
	Adámek, El Bashir, Sobral and Velebil \cite{Adamek2001b} shows).
	
	In a $\Ens$-category seen as a locally discrete $\Gpd$-category, 
	all arrows are faithful (because all 2-arrows between two given arrows are equal) and
	the fully faithful arrows are the monomorphisms.  Dually, all arrows
	are cofaithful and the fully cofaithful arrows are the epimorphisms.
	
	Faithful arrows possess a cancellation property similar to that of monomorphisms.  The proof is straightforward.
	
	\begin{pon}\label{propsimpfid}
		Let $A\overset{f}\ra B\overset{g}\ra C$ be arrows in a $\Gpd$-category.
		If $gf$ is faithful, then $f$ is faithful.
	\end{pon}
		
	\bigskip
	In a $\Gpdp$-category, we can define new kinds of morphisms,
	specific to the pointed case, which are to fully faithful or faithful arrows
	what 0-monomorphisms are to monomorphisms.
	We will give their definitions first in $\Gpdp$, where we assume
	that pointed functors are strictly described ($FI\equiv I$).

	\begin{pon}\label{fidsuro}
		Let $F\col\A\ra\B$ be a pointed functor in $\Gpdp$.
		The following conditions are equivalent:
		\begin{enumerate}
			\item for all $a\col I\ra I$ in $\A$, if $Fa=1_I$, then $a=1_I$;
			\item for all $a,a'\col A\ra I$ in $\A$, if $Fa=Fa'$, then $a=a'$;
			\item for all $a,a'\col A\ra A'$ in $\A$, where $A'\simeq I$,
				if $Fa=Fa'$, then $a=a'$.
		\end{enumerate}
		We call a functor which satisfies these conditions a \emph{0-faithful}%
		\index{functor!0-faithful}\index{0-faithful!functor} functor.
	\end{pon}
	
		\begin{proof}
			Condition 2 is a special case of condition 3 and condition 1
			is a special case of condition 2. It remains to prove that
			condition 1 implies condition 3.
			
			Let be $a,a'\col A\ra A'$ in $\A$ such that $Fa=Fa'$, with $\varphi\col A'\ra I$.
			We let $\hat{a}$ be equal to the following composite:
			\begin{eqn}
				I\overset{\varphi^{-1}}\longrightarrow A'\overset{a'^{-1}}
				\longrightarrow A\overset{a}\longrightarrow A'\overset{\varphi}
				\longrightarrow I.
			\end{eqn}
			Then $F\hat{a} = 1_I$, so $\hat{a}=1_I$, by condition 1,
			and $a=a'$.
		\end{proof}

	Faithful functors are 0-faithful, but the converse is not true in general.
	It will be the case in $\CGS$ (Proposition \ref{caracfidcgs}).
	Using the representable functors, we can now define a notion of 0-faithful arrow
	in any $\Gpdp$-category.  The previous proposition gives three equivalent
	versions of the definition.
	
	\begin{df}\label{caraczfid}\index{arrow!0-faithful}\index{0-faithful!arrow}
		Let $\C$ be a $\Gpdp$-category and $A\overset{f}\ra B$ be an arrow in $\C$.
		We say that $f$ is \emph{0-faithful} if, for all $X\col\C$, $f\circ -
		\col\C(X,A)\ra\C(X,B)$ is 0-faithful in $\Gpdp$, i.e.
		if the following equivalent conditions hold:
		\begin{enumerate}
			\item for all $X\col\C$ and for all $\alpha\col 0\Ra 0\col X\ra A$, if 
				$f\alpha=1_0$, then $\alpha=1_0$;
			\item for all $X\col\C$ and for all $\alpha, \alpha'\col a\Ra 0\col X\ra A$,
				if $f\alpha=f\alpha'$, then $\alpha=\alpha'$;
			\item for all $X\col\C$ and for all $\alpha, \alpha'\col a\Ra a'\col
				X\ra A$, where
				$a'\simeq 0$, if $f\alpha=f\alpha'$, then $\alpha=\alpha'$.
		\end{enumerate}
		We will denote by $\zFid$\index{0-fid@$\zFid$}
		the full sub-$\Gpd$-category of $\flc$ whose objects are the 0-faithful arrows.
	\end{df}

	It is obvious that every faithful arrow is 0-faithful. We will see that, in every 2-Puppe-exact $\Gpd$-category, the 0-faithful arrows are precisely the faithful arrows (Proposition \ref{fidfidzer}).
	
	In $\Gpdp$, the 0-faithful arrows are the 0-faithful functors.
	In a $\Ensp$-category seen as a locally discrete $\Gpdp$-category,
	all arrows are 0-faithful.
		
	The 0-cofaithful arrows\index{arrow!0-cofaithful}\index{0-cofaithful arrow}
	are defined dually.

	We define now the fully 0-faithful arrows, first in $\Gpdp$
	and then in any $\Gpdp$-category.
	The first three conditions of the following proposition correspond
	to the same conditions of Proposition \ref{fidsuro}.
	
	\begin{pon}\label{carplfidsurzer}
		Let $F\col\A\ra\B$ be a pointed functor in $\Gpdp$.  The following conditions
		are equivalent:
		\begin{enumerate}
			\item \begin{enumerate}
					\item for all $b\col FA\ra I$ in $\B$, there exists
						$a\col A\ra I$ in $\A$ such that $b=Fa$;
					\item for all $a\col I\ra I$ in $\A$, if $Fa=1_I$, then
						$a = 1_I$ (i.e.\ $F$ is 0-faithful);
				\end{enumerate}
			\item for all $b\col FA\ra I$ in $\B$, there exists a unique
				$a\col A\ra I$ in $\A$ such that $b=Fa$;
			\item for all $b\col FA\ra FA'$ in $\B$, where $FA'\simeq I$, there
				exists a unique $a\col A\ra A'$ in $\A$ such that $b=Fa$;
			\item \begin{enumerate}
					\item for all $A\col\A$, if $FA\simeq I$, then $A\simeq I$; 
					\item for all $b\col I\ra I$ in $\B$, there exists a unique
						$a\col I\ra I$ in $\A$ such that $b=Fa$.
				\end{enumerate}
		\end{enumerate}
		When these conditions hold, we say that $F$
		is \emph{fully 0-faithful}\index{functor!fully 0-faithful}%
		\index{fully 0-faithful!functor}.
	\end{pon}
	
		\begin{proof}
			{\it 3 $\Rightarrow$ 2. } Condition 2 is a special case of
			condition 3.
			
			{\it 2 $\Rightarrow$ 1. } Condition 1(a) is part of condition 2,
			and condition 1(b) is a special case of the unicity part of condition 2.
			
			{\it 1 $\Rightarrow$ 3. } Let be $A,A'\col\A$ and $b\col FA\ra FA'$,
			with $\psi\col FA'\ra I$.  By condition 1(a), there exist
			$\varphi\col A\ra I$ such that $F\varphi=\psi\circ b$ and $\varphi'\col A'\ra I$
			such that $F\varphi'=\psi$.  If we set $a\eqdef\varphi'^{-1}\circ\varphi$,
			we have $b=Fa$.
			The arrow $a$ is unique because, if $a'\col A\ra A'$ is such that
			$Fa'=b$, then $F(\varphi'\circ a'\circ\varphi^{-1})=1_I$ and, by
			condition 1(b), $\varphi'\circ a'\circ\varphi^{-1}=1_I$, hence
			$a'=a$.
			
			{\it 2 $\Rightarrow$ 4. } Condition 4 follows immediately from condition 2.
			
			{\it 4 $\Rightarrow$ 2. } Let be $A\col\A$ and $b\col FA\ra I$.
			By condition 4(a), there exists an arrow $\varphi\col A\ra I$.
			Then, by condition 4(b), there exists a unique $a'\col I\ra I$
			such that $Fa'=b\circ F\varphi^{-1}$.  If we set
			$a\eqdef a'\circ\varphi$, we have thus $Fa=b$. Moreover,
			$a$ is unique because, if we have $a''\col A\ra I$ such that $Fa''=b$,
			then $F(a''\circ\varphi^{-1})=b\circ F\varphi^{-1}$, so
			$a''\circ\varphi^{-1}=a'$ and $a''=a$.
		\end{proof}
	
	The fully faithful functors are fully 0-faithful,
	but the converse doesn't hold in general.
	It will be the case in $\CGS$ (Proposition \ref{caracplfidcgs}).
	
	\begin{df}\label{caracplzfid}\index{arrow!fully 0-faithful}%
		\index{fully 0-faithful!arrow}
		Let $\C$ be a $\Gpdp$-category and $A\overset{f}\ra B$ an arrow in $\C$.
		We say that $f$ is \emph{fully 0-faithful} if, for all $X\col\C$,
		the pointed functor $f\circ-\col\C(X,A)\ra\C(X,B)$ is fully 0-faithful,
		i.e.\ if
		if the following equivalent conditions hold:
		\begin{enumerate}
			\item \begin{enumerate}
				\item for all $X\col\C$, for all $a\col X\ra A$ and for all 
					$\beta\col fa\Ra 0$, there exists $\alpha\col a\Ra 0$ such that
					$\beta = f\alpha$;
				\item for all $X\col\C$ and for all $\alpha\col 0\Ra 0\col X\ra A$,
					if $f\alpha=1_0$, then $\alpha=1_0$ ($f$ is 0-faithful);
				\end{enumerate}
			\item for all $X\col\C$, for all $a\col X\ra A$ and for all
				$\beta\col fa\Ra 0$, there exists a unique $\alpha\col a\Ra 0$ such that
				$\beta = f \alpha $;
			\item for all $X\col\C$, for all $a,a'\col X\ra A$ and for all
				$\beta\col fa\Ra fa'$, where $fa'\simeq 0$, there exists a unique
				$\alpha\col a\Ra a'$ such that $\beta = f\alpha$;
			\item \begin{enumerate}
					\item for all $X\col\C$, for all $a\col X\ra A$, if $fa\simeq 0$,
						then $a\simeq 0$;
					\item for all $X\col\C$, for all $\beta\col 0\Ra 0\col X\ra B$,
						there exists a unique $\alpha\col 0\Ra 0\col X\ra A$ such that 
						$\beta = f \alpha $.
				\end{enumerate}
		\end{enumerate}
		We denote by $\zPlFid$\index{0-fullfaith@$\zPlFid$} the full sub-$\Gpdp$-category
		of $\flc$ whose objects are the fully 0-faithful arrows.
	\end{df}
	
	Il is obvious that every fully faithful arrow is fully 0-faithful. 
	We will see that in any 2-Puppe-exact
	$\Gpd$-category the fully 0-faithful arrows
	are exactly the fully faithful arrows (Proposition \ref{fidfidzer}). 
	
	In $\Gpdp$, the fully 0-faithful arrows are the fully 0-faithful functors.
	In a $\Ensp$-category seen as a locally discrete $\Gpdp$-category,
	the fully 0-faithful arrows are the 0-monomorphisms.
		
	The fully 0-cofaithful arrows\index{arrow!fully 0-cofaithful}%
	\index{fully 0-cofaithful arrow} are defined dually.

	\bigskip
	Now let us turn to the properties of objects. First recall the notion of discrete (or “bidiscrete” \cite{Street1980a}) object
	 and the dual notion of connected object.

	\begin{df}
		Let $\C$ be a $\Gpd$-category and $A\col\C$.
		\begin{enumerate}
			\item  We say that $A$ is \emph{discrete}\index{object!discrete}%
				\index{discrete object} if, for all $X\col\C$, $\C(X,A)$
				is a set (a discrete groupoid); we denote by $\DisC$\index{Dis(C)@$\DisC$}
				the full sub-$\Gpd$-category of $\C$ of the discrete objects.
			\item We say that $A$ is \emph{connected}\index{object!connected}%
				\index{connected object} if, for all $Y\col\C$, $\C(A,Y)$
				is a set; we denote by $\ConC$\index{Con(C)@$\ConC$}
				the full sub-$\Gpd$-category of $\C$ of the connected objects.
		\end{enumerate}
	\end{df}

	$\DisC$ and $\ConC$ are $\Ens$-categories, since their Hom are sets.
	In $\Gpd$, the discrete objects are the discrete groupoids, i.e.
	the sets, and the only connected object is the empty groupoid (the notion of connected object
	connected is relevant only in the pointed case).
	In $\Gpdp$, the discrete objects are the pointed sets
	seen as discrete pointed groupoids and the connected objects
	are the groups seen as one-object groupoids.
	In $\CG$, the $\Gpd$-category of 2-groups, the discrete objects are the groups seen
	as discrete 2-groups and the connected objects are the abelian groups seen as one-object
	2-groups.

	If $\C$ has a zero object, we can characterise the properties of an object in terms of the properties of the arrow from the object to $0$ or from $0$ to the object.
	
	\begin{pon}
		Let $A$ be an object in a $\Gpdp$-category $\C$ with a zero object.
		The following conditions are equivalent:
		\begin{enumerate}
			\item $A$ is discrete;
			\item for all arrows $a_1, a_2\col X\ra A$, there is at most one
				2-arrow between $a_1$ and $a_2$;
			\item $0^A\col A\ra 0$ is faithful.
		\end{enumerate}
	\end{pon}

	\begin{lemm}\label{fiddisc}
		Let $\C$ be a $\Gpdp$-category with a zero object.
		\begin{enumerate}
			\item If $A$ is discrete, then every arrow $f\col A\ra B$ is faithful.
			\item If $f\col A\ra B$ is faithful and $B$ is discrete,
				then $A$ is discrete.
			\item If $0\col A\ra B$ is faithful, then $A$ is discrete.
		\end{enumerate}
	\end{lemm}
	
		\begin{proof}
			{\it 1. } If $A$ is discrete, $0^A = 0^B\circ f$ is faithful and, by
			Proposition \ref{propsimpfid}, $f$ is faithful.
			
			{\it 2. } If $B$ is discrete, then $0^B$ is faithful and so
			$0^A=0^B\circ f$ is faithful, because faithful arrows are stable under composition.
			
			{\it 3. } If $0=0_B\circ 0^A\col A\ra B$ is faithful, then $0^A$ is faithful,
			by Proposition \ref{propsimpfid}.
		\end{proof}

\bigskip
We close this subsection by monomorphism- or epimorphism-like properties for loops in a $\Gpdp$-category $\C$ (i.e.\ for 2-arrows $0\Ra 0$ in $\C$).  Marco Grandis \cite{Grandis1994a} call the monoloops “monics on morphisms”.

	\begin{df}
		Let $\C$ be a $\Gpdp$-category and $\pi$ be a loop in $\C$:
		\begin{xym}\label{diagdhorm}
			\xymatrix@=40pt{A\rtwocell^0_0{\pi} &B.}
		\end{xym}
		We say that:
		\begin{enumerate}
			\item 	
				$\pi$ is a \emph{monoloop}\index{monoloop}  if,
				for all $a_1,a_2\col X\ra A$ such that $\pi a_1=\pi a_2$,
				there exists a unique 2-arrow $a_1\Ra a_2$;
			\item $\pi$ is an \emph{epiloop}\index{epiloop}  if,
				for all $b_1,b_2\col B\ra Y$ such that $b_1\pi=b_2\pi$,
				there exists a unique 2-arrow $b_1\Ra b_2$.
		\end{enumerate}
	\end{df}

The monoloops in $\C$ form a $\Gpd$-category $\caspar{MonoLoop}(\C)$\index{MonoLoop(C)@$\caspar{MonoLoop}(\C)$}, described in the following way.
\begin{itemize}
	\item {\it Objects. } The objects are the monoloops in $\C$.
	\item {\it Arrows. } An arrow $(A,\pi,B)\ra(A',\pi',B')$ consists of
		$a\col A\ra A'$ and $b\col B\ra B'$ such that $b\pi=\pi' a$.
	\item {\it 2-arrows. } A 2-arrow $(a,b)\Ra (a',b')\col (A,\pi,B)\ra(A',\pi',B')$
		consists of $\alpha\col a\Ra a'$ and $\beta\col b\Ra b'$.
\end{itemize}

Dually, we define a $\Gpd$-category $\caspar{EpiLoop}(\C)$\index{EpiLoop(C)@$\caspar{EpiLoop}(\C)$}, whose objects are the epiloops in $\C$.

	\begin{pon}\label{domonobodis}
		In any $\Gpdp$-category $\C$, if $\pi\col 0\Ra 0\col A\ra B$ is a mo\-no\-loop,
		then $A$ is discrete.  
	\end{pon}
	
		\begin{proof}
			If we have $\alpha\col a_1\Ra a_2\col X\ra A$, then $\pi a_1=\pi a_2$, because
			$\pi a_1 = 1_0\circ\pi a_1 = (0\alpha)\circ(\pi a_1)
			=(\pi a_2)\circ(0\alpha)=\pi a_2\circ 1_0 = \pi a_2$.
			So there exists a unique 2-arrow $a_1\Ra a_2$.
		\end{proof}

\section{Kernel-quotient systems on a $\Gpdp$-category}

\subsection{$\Coker\adj\Ker$}

	Let $\C$ be a $\Gpdp$-category $\C$ (strictly described, as we always assume).
	Let us first introduce convenient terminology.
	
	\begin{df}
		Let $\varphi$ and $\beta$ be 2-arrows in $\C$, as in the following diagram.
		 We say that $\beta$ is \emph{compatible with $\varphi$}%
		\index{compatibility of 2-arrows} if the following equation hold.
		\begin{xyml}\label{condphiplfid}\begin{gathered}\xymatrix@=25pt{
			&A\ar[dr]^f\drruppertwocell^0<5>{_<-0.0>\;\;\;\;\varphi^{-1}}
			\\ X\ar[ur]^{a_1}\ar[dr]_{a_2}\rrtwocell\omit\omit{\beta}
			&&B\ar[r]_y &Y
			\\ &A\ar[ur]_f\urrlowertwocell_0<-5>{_<0.0>\varphi}
		}\end{gathered}\;\;=\;\;1_0\end{xyml}
	\end{df}
 
 	In the special case where $a_2$ is $0$, this definition becomes simpler: in the
	situation of the following diagram, $\beta$ is compatible with $\varphi$ if $y\beta=\varphi a$.
	\begin{xym}\label{diagcompati}\xymatrix@=40pt{
		X\ar[r]^a\rrlowertwocell<-9>_0{_<2.7>\beta}
		&A\ar[r]^f\rruppertwocell<9>^0{^<-2.7>\varphi\,}
		&B\ar[r]_y
		&Y
	}\end{xym}

	The kernel of an arrow $f\col\flc$ appears, with the following universal property,
	in \cite{Vitale2002a}.
	
	\begin{df}\label{defkernel}\index{kernel!in a Gpd*-category@in a $\Gpdp$-category}
		Let be $A\overset{f}\ra B$ in $\C$.  We call $K$, $k$ and $\kappa$, as in the following diagram:
		\begin{xym}\label{diagdefkernel}
			\xymatrix@=40pt{K\ar[r]^k\rrlowertwocell_0<-9>{<2.7>\kappa} &A\ar[r]^f &B}
		\end{xym}
		a \emph{kernel} of $f$ if the following properties hold:
		\begin{enumerate}
			\item for all $X\col\C$, $a\col X\ra A$ and $\beta\col fa\Ra 0$,
				there exist $a'\col X\ra K$ and a 2-arrow $\alpha\col a\Ra ka'$
				such that
            	\begin{xyml}\label{kernel1}
				\beta\;\;=\;\;\begin{gathered}\xymatrix@=40pt{
					X\ar[r]_{a'}\rruppertwocell^a<9>{<-2.7>\alpha}
					&K\ar[r]^k\rrlowertwocell_0<-9>{<2.7>\kappa} 
					&A\ar[r]^f 
					&B
				}\end{gathered};\end{xyml}
			\item for all $a_1,a_2\col X\ra K$ and for all $\alpha\col ka_1\Ra ka_2$
				compatible with $\kappa$, there exists a unique $\alpha'\col a_1\Ra a_2$
				such that $\alpha=k\alpha'$.
		\end{enumerate}
		We write this property “$(K,k,\kappa)=\Ker f$”\index{Ker@$\Ker$}.
	\end{df}

	In $\Gpdp$ the kernel of a pointed functor $F\col\A\ra\B$
	(which we assume to be strictly described) can be constructed  in the following way
	(see \cite{Gabriel1967a} for a weaker universal property
	and \cite{Grandis2002a}).
	\begin{itemize}
		\item {\it Objects. } An object consists of $(A,b)$, where $A\col\A$
			and $b\col FA\ra I$ in $\B$.
		\item {\it Arrow. } An arrow $f\col(A,b)\ra(A',b')$ is an
			arrow $f\col A\ra A'$ in $\A$ such that $b'(Ff)=b$.
		\item {\it Point. } This is $(I,1_I)$.
	\end{itemize}

	\begin{pon}
		If $(K,k,\kappa)=\Ker f$, then $k$ is faithful.
	\end{pon}
	
		\begin{proof}
			This is the unicity of condition 2 of the definition of kernel.
		\end{proof}

	\begin{lemm}\label{noycodomdisplfid}
		Let $A\overset{f}\ra D$ be an arrow in $\C$ whose codomain is discrete.  If
		$(K,k,\kappa)=\Ker f$, then $k$ is fully faithful.
	\end{lemm}
		
		\begin{proof}
			Let be $a_1,a_2\col X\ra K$ and $\alpha\col k a_1\Ra ka_2$.  Since 
			$D$ is discrete, all 2-arrows with codomain $D$ are equal.
			In particular, $\alpha$ is necessarily compatible with $\kappa$.
			Then, by the universal property of the kernel, there exists a unique $\alpha'\col a_1
			\Ra a_2$ such that $\alpha = k\alpha'$.
		\end{proof}

	In dimension 1, the kernel classifies the 0-monomorphisms; in dimension 2, it
	classifies the fully 0-faithful arrows.  We won't give a proof of this proposition here, because we will prove later a more general version (Proposition 
	\ref{clasproprelker}).
	
	\begin{pon}\label{claspropker}
		Let $\C$ be a $\Gpdp$-category with a zero object.
		In the situation of diagram \ref{diagdefkernel} the following conditions are equivalent:
		\begin{enumerate}
			\item $f$ is fully
				0-faithful\index{arrow!fully 0-faithful!characterisation};
			\item if $(K,k,\kappa)=\Ker f$, then there exists
				$\kappa'\col k\Ra 0$ such that $f\kappa'=\kappa$;
			\item $(0,0_A,1_{0_B})=\Ker f$.
		\end{enumerate}
	\end{pon}

	\begin{pon}\label{kerneltrivial}
		Let $\C$ be a $\Gpdp$-category with a zero object.
		In the following diagram, $(0,0_A,1_{0_A})=\Ker 1_A$ and $(A,1_A,1_{0_A})=\Coker 0_A$.
		\begin{xym}
			\xymatrix@=40pt{0\ar[r]^{0_A}\rrlowertwocell_{0_A}<-9>{<2.7>\;\;\;\,1_{0_A}}
				&A\ar[r]^{1_A} &A}
		\end{xym}
	\end{pon}
	
	Now we prove that, if $\C$ has the kernels and cokernels, they extend to $\Gpd$-functors $\Ker$ and $\Coker\col\flc\ra\flc$
	which form an adjunction.
	
	\begin{pon}
		Let $\C$ be a $\Gpdp$-category which has the kernels and cokernels.  Then
		the kernels and cokernels extend to
		$\Gpd$-functors $\Ker, \Coker\col\flc\ra\flc$ such that
		\begin{eqn}
			\Coker\adj\Ker.
		\end{eqn}
	\end{pon}
	
		\begin{proof}
			{\it Preliminary construction. } Let be $X\overset{x}\ra Y\overset{y}\ra Z$,
			$\chi\col yx\Ra 0$ and $A\overset{f}\ra B$.  We define a functor
			\begin{eqn}
				\Phi_f^{x,\chi,y}\col\flc(y,f)\Ra \flc(x,\Ker f)
			\end{eqn}
			in the following way.
			Let be $(a,\varphi,b)\col y\ra f$ in $\flc$. By the universal property of the kernel, there exist $\hat{a}\col X\ra\Ker f$ and
			$\hat{\varphi}\col ax\Ra k_f\hat{a}$
			such that $\kappa_f \hat{a}\circ f\hat{\varphi}\circ\varphi x = b\chi$.  We set
			$\Phi_f^{x,\chi,y}(a,\varphi,b)\eqdef (\hat{a},\hat{\varphi},a)\col x\Ra k_f$.
			\begin{xym}\xymatrix@=40pt{
				{X}\ar[r]^-{x}\ar[d]_{\hat{a}}\rruppertwocell<9>^0{^<-2.7>\chi\,}
					\drtwocell\omit\omit{\hat{\varphi}}
				&Y\ar[r]^{y}\ar[d]_{a}\drtwocell\omit\omit{_{\varphi}}
				&Z\ar[d]^{b}
				\\ {\Ker f}\ar[r]_-{k_f}\rrlowertwocell<-9>_0{_<2.7>\;\;\;\;\kappa_f}
				&A\ar[r]_{f}
				&B
			}\end{xym}
			Next, let be $(\alpha,\beta)\col(a,\varphi,b)\Ra (a',\varphi',b')\col y\ra f$
			in $\C$.  Then $\hat{\varphi}'\circ\alpha x\circ\hat{\varphi}^{-1}
			\col k_f\hat{a} \Ra k_f\hat{a}'$ is compatible
			with $\kappa_f$ and, by the universal property of the kernel, there exists
			a unique $\hat{\alpha}\col\hat{a}\Ra \hat{a}'$ such that $k_f\hat{\alpha}
			\circ\hat{\varphi} = \hat{\varphi}'\circ\alpha x$.
			Then we set $\Phi_f^{x,\chi,y}(\alpha,\beta)\eqdef (\hat{\alpha},\alpha)$.
			This defines a functor by unicity of $\hat{\alpha}$.
			
			Dually we construct a functor
			\begin{eqn}
				\Psi_f^{x,\chi,y}\col\flc(f,x)\Ra \flc(\Coker f,y).
			\end{eqn}

			{\it Construction of $\Ker$. } The $\Gpd$-functor $\Ker\col\flc\ra\flc$ is already
			defined on objects by the existence of kernels. If
			$(a,\varphi,b)\col f\ra f'$, we set $\Ker(a,\varphi,b)\eqdef 
			\Phi_{f'}^{k_f,\kappa_f,f}(a,\varphi,b)\col k_f\Ra k_{f'}$. Moreover, if
			$(\alpha,\beta)\col(a,\varphi,b)\Ra (a',\varphi',b')$, we set
			$\Ker(\alpha,\beta)\eqdef \Phi_{f'}^{k_f,\kappa_f,f}(\alpha,\beta)$.
			We construct the $\Gpd$-functor structure by using condition 2 of the definition
			of kernel and we check that we get a $\Gpd$-functor by using the unicity
			of this condition 2.
			
			{\it Construction of $\Coker$. } The construction of $\Coker$ is dual.
			
			{\it Adjunction. } Let $A\overset{f}\ra B$ and $C\overset{g}\ra D$
			be two arrows in $\C$. We have functors
			\begin{eqn}
				\Phi_{f,g}\eqdef \Phi_g^{f,\zeta_f,q_f}\col\flc(\Coker f, g)
				\ra\flc(f,\Ker g)
			\end{eqn}
			and
			\begin{eqn}
				\Psi_{f,g}\eqdef \Psi_f^{k_g,\kappa_g,g}\col\flc(f,\Ker g)\ra
				\flc(\Coker f, g).
			\end{eqn}
			By using condition 2 of the definitions of kernel and cokernel, we prove
			that $\Phi_{f,g}$ and $\Psi_{f,g}$ are inverse to each other and that
			$\Phi$ is natural in $f$ and $g$.
		\end{proof}
	
	Let be $A\overset{f}\ra B$.  We construct the kernel of $f$, and then the cokernel of this
	kernel. By the universal property of the cokernel, there exist $\bar{e}^1_f:=\Coker k_f$,
	$\bar{m}^1_f\col\im_{\mathrm{full}}^1 f\eqdef\Coker(\Ker f)\ra B$%
	\index{imfull1f@$\im_{\mathrm{full}}^1 f$} and $\bar{\varphi}^1_f\col f
	\Ra \bar{m}^1_f \bar{e}^1_f$ such that
	$\bar{m}^1_f\zeta_{k_f}\circ\bar{\varphi}^1_fk_f=\kappa_f$.
	\begin{xym}\xymatrix@C=20pt@R=40pt{
			{\Ker f}\ar[rr]^-{k_f}\ar@/_0.78pc/[drrr]_0\ar@/^2pc/[rrrr]^0
			&{}\rrtwocell\omit\omit{^<-2.7>\kappa_f\;\,}
			&A\ar[rr]^f\ar[dr]_-{\bar{e}^1_f}\rrtwocell\omit\omit{_<4>\;\;\;\;\bar{\varphi}^1_f}
				\ar@{}[dll]|(0.32){\zeta_{k_f}\!\!\dir{=>}}
			&&B
			\\ &&&{\im_{\mathrm{full}}^1} f\ar[ur]_-{\bar{m}^1_f}
		}\end{xym}
	The counit of the adjunction $\Coker\adj\Ker$ is then
	\begin{eqn}
		\varepsilon_f\eqdef(1_A,(\bar{\varphi}^1_f)^{-1},\bar{m}^1_f)\col \Coker(\Ker f)\ra f.
	\end{eqn}
	The unit of the adjunction
	is defined dually.
	
	By taking as congruences the faithful arrows (we denote by
	$\Fid(\C)$\index{Faith(C)@$\Fid(\C)$} the full sub-$\Gpd$-category of $\flc$ whose
	objects are the faithful arrows), we get a kernel-quotient system.
	
	\begin{pon}
		Let $\C$ be a $\Gpdp$-category in which all kernels and cokernels exist.  The faithful arrows
		in $\C$, together with the $\Gpd$-functor codomain $\partial_1\col\Fid(\C)\ra\C$,
		and the adjunction $\Coker\adj\Ker$ form a kernel-quotient system.
		\begin{xym}\xymatrix@C=30pt@R=45pt{
			{\Fid(\C)}\ar@<2mm>[rr]^-{\Coker}\ar@<1mm>[dr]^-{\partial_1}\ar@{}[rr]|-\perp
			&&{\fl{\C}}\ar@<2mm>[ll]^-{\Ker}\ar@<2mm>[dl]^-{\partial_0}
			\\ &{\C}\ar@<0mm>[ur]^-{\mathrm{1_{-}}}\ar@<1mm>[ul]^-{\Ker 1_{-}}
		}\end{xym}
	\end{pon}
	
		\begin{proof}
			The $\Gpd$-functors $\Ker$ and $\Coker$ were constructed in such a way that
			$\partial_0\Coker\equiv \partial_1$ and $\partial_1\Ker\equiv\partial_0$.
			So we can take $\alpha$ and $\beta$ equal to the identity in the definition
			of kernel-quotient system (Definition \ref{defsysnoyquot}).
			Morevoer, $\varepsilon_{1_A}$ is an equivalence, by Proposition
			\ref{kerneltrivial}. 
		\end{proof}

	We can thus apply the theory of Section \ref{sectsysnoyquot}.
	By Proposition \ref{claspropker}, the $\Ker$-monomorphisms are the
	 fully 0-faithful arrows and, dually, the $\Coker$-epi\-mor\-phisms
	are the fully 0-cofaithful arrows.
	Moreover, we call the $\Ker$-regular epimorphisms \emph{normal cofaithful arrows}\index{arrow!normal cofaithful}\index{normal cofaithful arrow}
	 and, dually, we call the $\Coker$-regular monomorphisms \emph{normal faithful arrows}\index{arrow!normal faithful}\index{normal faithful arrow}.  We denote by
	$\Cofidnorm$\index{NormCofaith@$\Cofidnorm$}
	the full sub-$\Gpd$-category of $\flc$ whose objects are the normal cofaithful arrows
	and $\Fidnorm$\index{NormFaith@$\Fidnorm$} the full sub-$\Gpd$-category whose objects
	are the normal faithful arrows.
	
	We can also define $\Ker$-idempotent $\Gpdp$-categories, which
	Marco Grandis (\cite{Grandis2001b}) call \emph{h-semistable}.
	
	\begin{df}
		We say that a $\Gpdp$-category $\C$ is \emph{$\Ker$-idempotent}%
		\index{idempotent!Ker-@$\Ker$-}%
		\index{Gpd*-category@$\Gpdp$-category!Ker-idempotent@$\Ker$-idempotent}
		\index{Ker-idempotent Gpd*-category@$\Ker$-idempotent $\Gpdp$-category}
		if the adjunction $\Coker\adj\Ker$
		is idempotent, i.e.\ if the following equivalent conditions hold:
		\begin{enumerate}
			\item if $k$ is a kernel and $(q,\zeta)=\Coker k$, then $(k,\zeta)=\Ker q$
				(every kernel is a normal faithful arrow);
			\item if $q$ is a cokernel and $(k,\kappa)=\Ker q$, then $(q,\kappa)=\Coker k$
				(every cokernel is a normal cofaithful arrow).
		\end{enumerate}
	\end{df}
	
	We can express in this context Proposition \ref{defcaracfactor}.
	\begin{pon}\label{carackerfac}
		The following conditions are equivalent (when they hold, we say that
		 $\C$ is $\Ker$-factorisable%
		\index{factorisable!Ker-@$\Ker$-}%
		\index{Gpd*-category@$\Gpdp$-category!Ker-factorisable@$\Ker$-factorisable}%
		\index{Ker-factorisable@$\Ker$-factorisable!Gpd*-catrgorie@$\Gpdp$-category}):
		\begin{enumerate}
			\item for every $f\col\flc$, $\bar{m}^1_f$ is fully 0-faithful;
			\item $\C$ is $\Ker$-idempotent and
				normal cofaithful arrows are stable under composition.
		\end{enumerate}
	\end{pon}
	
	In a $\Ker$-factorisable $\Gpdp$-category, $(\Cofidnorm,\zPlFid)$
	is a factorisation system.  We can also apply Definition
	\ref{defkpreex}.
	
\begin{df}%
		\index{preexact!Ker-@$\Ker$-}%
		\index{Gpd*-category@$\Gpdp$-category!Ker-preexact@$\Ker$-preexact}
		\index{Ker-preexact@$\Ker$-preexact!Gpd*-category@$\Gpdp$-category}
	We say that $\C$ is \emph{$\Ker$-preexact} if for every faithful arrow $m$,
	$\eta_m\col m\ra \Ker(\Coker m)$
	is an equivalence (every faithful arrow is normal).
\end{df}

En dimension 1, the adjunction $\Coker\adj\Ker$ gave back in $\Ensp$ the factorisation system $(\Bij^*,\zMono)$.  This doesn't work any more in dimension 2: there is a factorisation system $\EM$ on $\Gpdp$ with $\M=\zFid$, but it is not true that every arrow of $\E$ is canonically the cokernel of its kernel; so $\E\neq\Cofidnorm$.  Actually the adjunction $\Coker\adj\Ker$ is not idempotent in $\Gpdp$, as the following example shows.

\begin{pon}\label{paskerkercoker}
	In $\Gpdp$, it doesn't hold that every kernel is canonically the kernel of its cokernel.
\end{pon}
	
		\begin{proof}
			Here is a counter-example: we start with the unique pointed functor
			$0\ra(\mathbb{Z}_2)\con$ (where $(\mathbb{Z}_2)\con$ is the one-object pointed groupoid whose group of arrows is $\mathbb{Z}_2$).
			The objects of its kernel are $(I,0)$ (which is the distinguished object) and $(I,1)$ (where $I$ is the unique object of $(\mathbb{Z}_2)\con$)
			and the only arrows are the identities on these objects;
			so the kernel is equivalent to the set $2$.
			The unique object of the cokernel of the kernel of this functor is the unique object of $0$, and its arrows are generated by $\gamma_{(I,0)}$,
			which must be equal to $1_I$ and by $\gamma_{(I,1)}$
			(and its inverse), which has \emph{no} condition to satisfy.  So
			this cokernel is $\mathbb{Z}\con$.  Finally, the kernel of $0\ra\mathbb{Z}\con$
			is $\mathbb{Z}\dis$ (where $\mathbb{Z}\dis$ is the discrete pointed groupoid
			whose objects are the integers, with distinguished element $0$), and not $2$.\qedhere
			\begin{xym}\xymatrix@=40pt{
				{2}\ar[r]
				&0\ar[r]^-F\ar[dr]
				&{(\mathbb{Z}_2)\con}
				\\ {\cat{Z}\dis}\ar[ur]&&{\mathbb{Z}\con}\ar[u]
			}\end{xym}
		\end{proof}

To give a notion of kernel which works well  in $\Gpdp$ and in any $\Gpdp$-category, we should extend the kernel by adjoining objects and arrows containing the missing informations to recover the image from the cokernel of the kernel. The problem is similar to the case of $\Gpd$ where the ordinary kernel-pair doesn't suffice to recover the image by taking the coequalizer; we need to add an object to this kernel-pair (see \cite{Dupont2008b}).  

But we would lose the coincidence of the kernel and the coquotient.  Moreover, our final goal is to work in $\CGS$-categories (see Chapter 5) and, there, the kernel works as we want: the full image of every morphism of symmetric 2-groups is canonically the cokernel of its kernel.  

In the same way, the notion of congruence we use is not justified by the kernels in the basis: it is not true that any faithful functor is the kernel of its cokernel in $\Gpdp$, since the adjunction is not idempotent; but it will be the case in $\CGS$.

\subsection{$\Coroot\adj\Pip$}\label{sectcoradjpep}

	(Co)pips and (co)roots have been introduced in \cite{Dupont2003a}.

	\begin{df}\label{defpepin}\index{pip}\index{copip}
		Let be $A\overset{f}\ra B$ in $\C$.  We call an object $P\col\C$
		equipped with a loop $\pi\col 0\Ra 0\col P\ra A$ such that $f\pi=1_0$
		a \emph{pip} of $f$ if the following conditions hold:
		\begin{enumerate}
			\item
				for all $X\col\C$ and for all $\alpha\col 0\Ra 0\col X\ra A$
				such that $f\alpha = 1_0$, there exists $x\col X\ra P$ such that $\pi x=\alpha$;
			\item
				$\pi$ is a monoloop.
		\end{enumerate}
		We write this property “$(P,\pi)=\Pip f$”.\index{Pip@$\Pip$}
		\begin{xym}\label{diagpep}
		\xymatrix@=40pt{P\rtwocell^0<5>_0<5>{\pi} &A\ar[r]^f &B}
		\end{xym}
	\end{df}
		We can remove condition 2 and add instead to condition 1: “… and, for any other $x'\col X\ra P$ such that $\pi x'=\alpha$, there exists a unique $\chi\col x'\Ra x$”.
		As $\pi$ is a monoloop, the object $P$ is discrete, by Proposition
	\ref{domonobodis}.
	
	We have seen that the kernel classifies the fully 0-faithful arrows.  The pip classifies the 0-faithful arrows.
	
	\begin{pon}\label{pepclaszfid}\index{arrow!0-faithful!characterisation}
		In the situation of diagram \ref{diagpep}, the following conditions are equivalent:
		\begin{enumerate}
			\item $f$ is 0-faithful;
			\item if $(P,\pi)=\Pip f$, then $\pi=1_0$;
			\item $(0,1_0)=\Pip f$.
		\end{enumerate}
	\end{pon}
	
		\begin{proof}
			{\it 1 $\Rightarrow$ 3. } Let be $\alpha\col 0\Ra 0\col X\ra A$ such that
			$f\alpha=1_0$.  As $f$ is 0-faithful, $\alpha=1_0$, we have thus
			$0\col X\ra 0$ such that $1_0 0 =\alpha$.  Moreover,
			$1_0\col 0\Ra 0\col 0\ra A$ is obviously a monoloop, since $0$ is a zero object.
			
			{\it 3 $\Rightarrow$ 2. } By the universal property of the pip, there exists
			$x\col P\ra 0$ such that $\pi=1_0 x = 1_0$.
			
			{\it 2 $\Rightarrow$ 1. } Let be $\alpha\col 0\Ra 0\col X\ra A$ such that
			$f\alpha=1_0$. By the universal property of the pip, there exists
			$x\col X\ra P$ such that $\alpha = \pi x = 1_0$.
		\end{proof}

	\begin{df}\label{defroot}\index{root}
		Let $\pi\col 0\Ra 0\col A\ra B$ be a loop in $\C$. We call an object $R\col\C$
		equipped with an arrow $R\overset{r}\ra A$ such that $\pi r=1_0$ 
		a \emph{root} of $\pi$ if the following conditions hold:
		\begin{enumerate}
			\item for all $a\col X\ra A$ such that $\pi a=1_0$, there exists
				$a'\col X\ra R$ and $\alpha\col a\Ra ra'$;
			\item $r$ is fully faithful.
		\end{enumerate}
		We write this property 
		“$(R,r)=\Root \pi$”.\index{root@$\Root$}
	\end{df}
	
	We define coroots\index{coroot} dually.
	
	\begin{pon}\label{pepintrivial}
		In the situation of the following diagram, $(0,1_0) =\Pip 1_A$ and
		$(A,1_A) = \Coroot 1_0$.
		\begin{xym}
		\xymatrix@=40pt{0\rtwocell^0<5>_0<5>{\,1_0} &A\ar[r]^{1_A} &A}
		\end{xym}
	\end{pon}
	
	\begin{pon}
		Let $\C$ be a $\Gpdp$-category which has all pips and coroots of monoloops.
		Then the pips form a $\Gpd$-functor $\Pip\col\flc\ra\caspar{MonoLoop}(\C)$
		and the coroots form a $\Gpd$-functor
		$\Coroot\col\caspar{MonoLoop}(\C)\ra\flc$. Moreover,
		\begin{eqn}
			\Coroot\adj\Pip.
		\end{eqn}
	\end{pon}
	
		\begin{proof}
			{\it Preliminary constructions. } We construct first, for all
			2-arrow $\chi\col 0\Ra 0\col X\ra Y$ and arrow
			$Y\overset{y}\ra Z$ such that $y\chi=1_0$
			and for all $A\overset{f}\ra B$, a functor
			\begin{eqn}
				\Phi_f^{\chi,y}\col \flc(y,f)\ra\caspar{MonoLoop}(\C)(\chi,\pi_f).
			\end{eqn}
			Let be $(a,\varphi,b)\col y\ra f$ in $\flc$.  Since $f(a\chi)=1_0$, by
			the universal property of the pip, there exists $\hat{a}\col X\ra \Pip f$
			such that $\pi_f\hat{a} = a\chi$.  We set $\Phi_f^{\chi,y}(a,\varphi,b)
			\eqdef (\hat{a},a)$.
			
			Then, let be $(\alpha,\beta)\col(a,\varphi,b)\Ra (a',\varphi',b')\col y\ra f$.
			Thanks to the existence of $\alpha$, we have $\pi_f\hat{a}=a\chi=a'\chi=\pi_f\hat{a}'$.
			So, as $\pi_f$ is a monoloop, there exists a unique $\hat{\alpha}
			\col\hat{a}\Ra\hat{a}'$. We set $\Phi_f^{\chi,y}(\alpha,\beta)
			\eqdef(\hat{\alpha},\alpha)$.  Moreover, $\Phi_f^{\chi,y}$ is a functor
			thanks to the unicity of $\hat{\alpha}$.
			
			We construct next, for all
			$\chi\col 0\Ra 0\col X\ra Y$, $Y\overset{y}\ra Z$ such that $y\chi=1_0$
			and, for all $\pi\col 0\Ra 0\col A\ra B$ (with coroot
			$r_\pi\col B\ra \Coroot\pi$), a functor
			\begin{eqn}
				\Psi_{\pi}^{\chi,y}\col\caspar{MonoLoop}(\C)(\pi,\chi)\ra
				\flc(\Coroot\pi,y).
			\end{eqn} 
			Let $(a,b) \col \pi\ra\chi$ be an arrow in $\caspar{MonoLoop}(\C)$.  Since $yb\pi=1_0$,
			by the universal property of the coroot, there exist $\check{b}\col\Coroot\pi
			\ra Z$ and $\varphi\col\check{b}r_{\pi}\Ra yb$. We set $\Psi_{\pi}^{\chi,y}
			(a,b)\eqdef (b,\varphi,\check{b})$.
			
			Then, let be $(\alpha,\beta)\col(a,b)\Ra(a',b')\col\pi\ra\chi$.  Since $r_\pi$
			is fully cofaithful, there exists a unique $\check{\beta}\col\check{b}
			\Ra\check{b}'$ such that $\varphi'\circ\check{\beta}r_\pi=y\beta\circ\varphi$.
			We set $\Psi_{\pi}^{\chi,y}(\alpha,\beta)\eqdef(\beta,\check{\beta})$.
			This is a functor thanks to the unicity of $\check{\beta}$.
			
			{\it Construction of $\Pip$. }  It is already defined on objects by the existence of pips. We define it on objects in the following way:
			\begin{eqn}
				\Pip_{f,f'}\eqdef\Phi_{f'}^{\pi_f,f}\col
				\flc(f,f')\ra\caspar{MonoLoop}(\C)(\Pip f,\Pip f')
			\end{eqn}
			We get the $\Gpd$-fonctoriality by using the fact that pips are monoloops.
			
			{\it Construction of $\Coroot$. } It is already defined on objects by the existence of coroots.  We define it on arrows in the following way:
			\begin{eqn}
				\Coroot_{\pi,\pi'}\eqdef\Psi_{\pi}^{\pi', r_{\pi'}}\col
				\caspar{MonoLoop}(\C)(\pi,\pi')\ra \flc(\Coroot\pi,\Coroot\pi').
			\end{eqn}
			We get the $\Gpd$-fonctoriality by using the fact that coroots are fully cofaithful.
			
			{\it Adjunction. } Let $\pi\col 0\Ra 0\col A\ra B$ be a monoloop and $C\overset{f}\ra D$.  We define two functors
			\begin{eqn}
				\Phi_{\pi,f}\eqdef\Phi_f^{\pi,r_{\pi}}
				\col\flc(\Coroot\pi,f)\ra\caspar{MonoLoop}(\C)(\pi,\Pip f)
			\end{eqn}
			and
			\begin{eqn}
				\Psi_{\pi,f}\eqdef\Psi_\pi^{\pi_f,f}
				\col\caspar{MonoLoop}(\C)(\pi,\Pip f)\ra\flc(\Coroot\pi,f).
			\end{eqn}
			By using the fact that pips are monoloops and coroots are fully cofaithful, we prove that these two functors are inverse to each other and we prove the $\Gpd$-naturality of $\Phi$ in $\pi$ and in $f$.
		\end{proof}

	Let be $A\overset{f}\ra B$.  We construct the pip of $f$, then the coroot
	$e^1_f\col A\ra\im^1 f$ of this pip. By the universal property of the coroot,
	there exist $m^1_f\col\im^1 f\ra B$\index{im1f@$\im^1 f$}
	and $\varphi^1_f\col f\Ra m^1_f e^1_f$.
	\begin{xym}\xymatrix@C=20pt@R=40pt{
			{\Pip f}\rrtwocell^0<5>_0<5>{\;\;\pi_f}
			&&A\ar[rr]^f\ar[dr]_-{e^1_f}\rrtwocell\omit\omit{_<4>\;\;\,\varphi^1_f}
			&&B
			\\ &&&\im^1 f\ar[ur]_-{m^1_f}
		}\end{xym}
	The counit of the adjunction $\Coroot\adj\Pip$ is then
	\begin{eqn}
		\varepsilon_f\eqdef
		(1_A,(\varphi^1_f)^{-1},m^1_f)\col \Coroot(\Pip f)\ra f.
	\end{eqn}
	The unit of the adjunction is defined dually.
	
	By taking as congruences the monoloops, we get a kernel-quotient system.
	
	\begin{pon}
		Let $\C$ be a $\Gpdp$-category which has all pips and coroots of monoloops.
		The monoloops in $\C$, together with the $\Gpd$-functor codomain
		$\partial\col\caspar{MonoLoop}(\C)\ra\C$, which maps $\pi\col 0\Ra 0\col A\ra B$
		to $B$, and the adjunction $\Coroot\adj\Pip$ form a kernel-quotient system.
		\begin{xym}\xymatrix@C=30pt@R=45pt{
			{\caspar{MonoLoop}(\C)}\ar@<2mm>[rr]^-{\Coroot}
				\ar@<1mm>[dr]^-{\partial}\ar@{}[rr]|-\perp
			&&{\fl{\C}}\ar@<2mm>[ll]^-{\Pip}\ar@<2mm>[dl]^-{\partial_0}
			\\ &{\C}\ar@<0mm>[ur]^-{\mathrm{1_{-}}}\ar@<1mm>[ul]^-{\Pip 1_{-}}
		}\end{xym}
	\end{pon}
	
		\begin{proof}
			The $\Gpd$-functors $\Pip$ and $\Coroot$ have been constructed in such a way that
			$\partial_0\Coroot\equiv \partial$ and $\partial\Pip\equiv\partial_0$.
			So we can take the identity for $\alpha$ and $\beta$ in the definition
			of kernel-quotient system (Definition \ref{defsysnoyquot}).
			Moreover, $\varepsilon_{1_A}$ is an equivalence, by Proposition
			\ref{pepintrivial}. 
		\end{proof}

	So we can apply here the theory of Section \ref{sectsysnoyquot}.
	By Proposition \ref{pepclaszfid}, the $\Pip$-monomorphisms are the
	0-faithful arrows and, dually, the $\Copip$-epimorphisms
	are the 0-cofaithful arrows.
	Moreover, we call \emph{normal fully cofaithful
	arrows}\index{arrow!normal fully cofaithful}%
	\index{normal fully cofaithful arrow} the $\Pip$-regular epimorphisms and,
	dually, we call \emph{normal fully faithful arrows}\index{arrow!normal fully faithful}%
	\index{normal fully faithful arrow} the $\Copip$-regular monomorphisms.
	We denote by $\PlCofidnorm$\index{NormFullCofaith@$\PlCofidnorm$}
	the full sub-$\Gpd$-category of $\flc$ of normal fully cofaithful arrows
	and $\PlFidnorm$\index{NormFullFaith@$\PlFidnorm$} the full sub-$\Gpd$-category 
	of normal fully faithful arrows.
	
	There is a simplification with respect to the kernel-quotient system $\Coker\adj\Ker$:
	every $\Gpdp$-category which has all pips and coroots is
	$\Pip$-idempotent. So, in the Proposition \ref{defcaracfactor}
	applied to $\Coroot\adj\Pip$, we can omit this condition.
	\begin{pon}
		The following conditions are equivalent (when they hold,
		we say that $\C$ is \emph{$\Pip$-factorisable}\index{factorisable!Pip-@$\Pip$-}%
		\index{Gpd*-category@$\Gpdp$-category!Pip-factorisable@$\Pip$-factorisable}%
		\index{Pip-factorisable Gpd*-categorie@$\Pip$-factorisable $\Gpdp$-category}):
		\begin{enumerate}
			\item for every $f\col\flc$, $m^1_f$ is 0-faithful;
			\item normal fully cofaithful arrows are stable under composition.
		\end{enumerate}
	\end{pon}
	
	In a $\Pip$-factorisable $\Gpdp$-category, $(\PlCofidnorm,\zFid)$
	is a factorisation system.  We can also apply Definition
	\ref{defkpreex}.
	
\begin{df}
	We say that a $\Gpdp$-category $\C$ in which all pips and coroots of monoloops exist is \emph{$\Pip$-preexact}\index{preexact!Pip-@$\Pip$-}%
		\index{Gpd*-category@$\Gpdp$-category!Pip-preexact@$\Pip$-preexact}%
		\index{Pip-preexact Gpd*-category@$\Pip$-preexact $\Gpdp$-category}
		if for every monoloop $\pi$,
	$\eta_{\pi}\col \pi\ra \Pip(\Coroot\pi)$
	is an equivalence (every monoloop is canonically the pip of its coroot).
\end{df}

\subsection{$\Sigma\adj\Omega$}\label{sssectsigadjom}

	From the adjunction between kernel and cokernel follows an adjunction between
	two functors $\C\ra\C$ which play the rôle of the suspension and loop space functors
	in algebraic topology.  Gabriel and Zisman \cite{Gabriel1967a} define them in general
	in  $\Gpdp$-categories for a notion of kernel with a weaker and stricter universal property (without
	unicity at the level of 2-arrows).  In the case of the $\Gpdp$-category of pointed
	topological spaces with continuous maps and homotopies up to 2-homotopies, 
	they recover the usual suspension and loop space functors.  In our case, because of the unicity
	in the universal property of the kernel, we always have $\Omega\Omega A\simeq 0$
	and $\Sigma\Sigma A\simeq 0$.
	
	Marco Grandis have studied this adjunction in a more general context in
	\cite{Grandis2001b}.  In the beginning, we follow his presentation, but in $\Gpdp$-categories and with the notion of kernel defined above and not with the standard homotopy kernel.
	
	Let $\C$ be a $\Gpdp$-category which has all kernels and cokernels and a zero object.  The adjunction
	$\Coker\adj\Ker$ induces
    adjunctions between the fibres at the point $C\col\C$
    of the domain and codomain $\Gpd$-functors:
	\begin{xym}\xymatrix@=50pt{
         	{C\backslash\C}\ar@<-2mm>[r]_-{\Ker} \ar@{}[r]|-\perp
			&{\C\slash C.}\ar@<-2mm>[l]_-{\Coker}
    }\end{xym}
    In particular, for $C\eqdef 0$, we get, since
    $\C\slash 0\simeq 0\backslash\C\simeq\C$ an adjunction $\Sigma\adj\Omega$
    between $\C$ and $\C$, where $\Sigma C \eqdef \Coker 0^C$\index{s@$\Sigma$}
    and $\Omega C \eqdef \Ker 0_C$\index{o@$\Omega$}.
    The object $\Omega C$ is discrete, since the arrow to $0$ is faithful (it is a
    kernel); dually, $\Sigma C$ is connected.
	\begin{xym}\label{adjsigmomeg}\xymatrix@=50pt{
         	{0\backslash\C}\ar@<-2mm>[r]_-{\Ker} \ar@{}[r]|-\perp\ar@{-}[d]_\wr
			&{\C\slash 0}\ar@<-2mm>[l]_-{\Coker}\ar@{-}[d]^\wr
         	\\ {\C}\ar@<-2mm>[r]_-{\Omega} \ar@{}[r]|-\perp
			&{\C}\ar@<-2mm>[l]_-{\Sigma}
    }\end{xym}
	The unit of the adjunction
    $\Sigma\adj\Omega $ is denoted by $\eta_C$ and its counit by $\eps_C$.  They satisfy the equations $\omega_{\Sigma C}\eta_C = \sigma_C$ and $\eps_C\sigma_{\Omega C} = \omega_C$.
    We set $\pi_0 \eqdef\Omega\Sigma$ and $\pi_1 \eqdef\Sigma\Omega$.%
    \index{p0@$\pi_0$}\index{p1@$\pi_1$}
    \begin{xym}\label{diagunitcounitpiopiu}\xymatrix@C=20pt@R=40pt{
			C\ar[rr]^-{0^C}\ar[dr]_-{\eta_C}
				\ar@/^2pc/[rrrr]^0
			&{}\rrtwocell\omit\omit{^<-2.7>\sigma_C\;\;}
			&0\ar[rr]^-{0_{\Sigma C}}\ar@{}[drr]|(0.32){\dir{=>}\!\omega_{\Sigma_C}}
			&{}
			&\Sigma C
			\\ &\Omega\Sigma C\ar[ur]_-{0}\ar@/_0.76pc/[urrr]_0 &&&{}
		}\;\;\;\;\xymatrix@C=20pt@R=40pt{
			\Omega C\ar[rr]^-{0^{\Omega C}}\ar@/_0.76pc/[drrr]_0\ar@/^2pc/[rrrr]^0
			&{}\rrtwocell\omit\omit{^<-2.7>\omega_C\;\;\;}
			&0\ar[rr]^-{0_C}\ar[dr]_-{0}
				\ar@{}[dll]|(0.32){\sigma_{\Omega C}\!\!\dir{=>}}
			&&C
			\\ {}
			&&&\Sigma\Omega C\ar[ur]_-{\varepsilon_C}
		}\end{xym}

	The $\Gpd$-functors $\Omega$ and $\Sigma$ allow us to reduce loops to arrows and, in particular, to reduce the properties expressed in terms of pip and coroot to properties expressed in terms of kernel and cokernel.

	From the following proposition follows the fact that $\omega_C$ is a monoloop and that $\sigma_C$ is an epiloop.

	\begin{lemm}\label{omegpepzc}
    	$(\Omega C,\omega_C)=\Pip 0^C$ and $(\Sigma C,\sigma_C)=\Copip 0_C$.
    \end{lemm}
    
    	\begin{proof}
			We will prove the first property; the proof of the second is dual.
			Let be $\gamma\col 0\Ra 0\col X\ra C$ (which necessarily satisfies $0^C\gamma =1_0$).
			By the universal property of the kernel, there exists $x\col X\ra\Omega C$ such
			that $\omega_C x = \gamma$.  Moreover, $\omega_C$ is a monoloop
			since, if we have $x,x'\col X\ra\Omega C$ such that $\omega_C x =\omega_C x'$,
			by the second part of the universal property of the kernel, there exists
			a unique $\chi\col x\ra x'$ such that $0\chi = 1_0$.
		\end{proof}
	
	If $\Omega B$ and $\Sigma A$ exist, the universal properties of the kernel and the cokernel give the following equivalences:
	\begin{eqn}\label{equadjomsi}
		\C(\Sigma A,B)\simeq\C(A,B)(0,0)\simeq\C(A,\Omega B).
	\end{eqn}
	Let us make these equivalences explicit: if $\pi\col 0\Ra 0\col A\ra B$, we have,
	on the one hand, $\bar{\pi}\col \Sigma A\ra B$	such that $\pi=\bar{\pi}\sigma_A$
	and, on the other hand, $\tilde{\pi}\col A\ra\Omega B$ such that $\pi=\omega_B \tilde{\pi}$.
	The following diagram illustrates the situation.
	Conversely, if $f\col \Sigma A\ra B$, we have
	$\hat{f}\eqdef f\sigma_A \col  0\Ra 0\col A\ra B$
	and if $g\col A\ra\Omega B$, we have $\check{g}\eqdef\omega_B g\col 0\Ra 0\col A\ra B$.
	\begin{xym}\label{diagnotationsomegsigm}\xymatrix@=40pt{
		&\Omega B\dtwocell_0^0{\omega_B}
		\\ A\rtwocell^0_0{\pi}\dtwocell_0^0{\sigma_A}\ar@<1mm>[ur]^{\tilde{\pi}}
		&B
		\\ \Sigma A\ar@<-2mm>[ur]_{\bar{\pi}}
	}\end{xym}

   \begin{lemm}\label{deuflunfl}
    	In the situation of the following diagram, we have
		\begin{eqn}
			f\alpha = 1_0 \;\Leftrightarrow\; f\bar{\alpha}\simeq 0.
		\end{eqn}
		(The isomorphism $f\bar{\alpha}\simeq 0$ is necessarily unique, because $\Sigma X$
		is connected.)
	\begin{xym}\label{chibarchi}\xymatrix@=40pt{
		X\rtwocell^0_0{\alpha}\dtwocell_0^0{\sigma_X}
		&A\ar[r]^f
		&B
		\\ \Sigma X\ar@<-2mm>[ur]_{\bar{\alpha}}
	}\end{xym}
    \end{lemm}
    
    	\begin{proof}
			If $f\bar{\alpha}\simeq 0$, then
			$f\alpha = f\bar{\alpha}\sigma_X = 0\sigma_X=1_0$. Conversely,
			if $f\alpha = 1_0$, then $f\bar{\alpha}\sigma_X=f\alpha = 1_0=0\sigma_X$
			and, since $\sigma_X$ is an epiloop, $f\bar{\alpha}\simeq 0$.
		\end{proof}

    \begin{pon}\label{kerpep}
    	Let us consider the situation of the following diagram in a $\Gpdp$-category.
		If $(K,k,\kappa)=\Ker f$, then $(\Omega K, k\omega_{K})=\Pip f$.
		So, if $\C$ has all kernels, $\C$ has all
		pips\index{pip!construction by kernels}%
		\index{copip!construction by cokernels}.
		\begin{xym}\xymatrix@=40pt{
			\Omega {K}\rtwocell^0_0{\;\;\;\,\omega_{K}}
			&K\rruppertwocell^0<9>{^<-2.7>\kappa}\ar[r]_{k}
			&A\ar[r]
			&B
		}\end{xym}
    \end{pon}

    	\begin{proof}
			Let us consider the following diagram.
			Let $\alpha\col 0\Ra 0\col X\ra A$ be such that $f\alpha = 1_0$.
			By Lemma \ref{deuflunfl}, $f\bar{\alpha}\simeq 0$.
			So, by the universal property of the kernel, there exists a
			factorisation $a\col\Sigma X\ra K$ such that $ka\simeq\bar{\alpha}$.
			To $a$ corresponds under the adjunction $\Sigma\adj\Omega$ an arrow
			$\tilde{a}\col X\ra\Omega K$ such that
			$a\sigma_X = \omega_K\tilde{a}$.  Therefore
			$k\omega_K\tilde{a}=k a\sigma_X =
			\bar{\alpha}\sigma_X = \alpha$.
			\begin{xym}\xymatrix@=40pt{
				X\rtwocell^0_0{{\;\;\;\sigma_X}}\ar[d]_{\tilde{a}}
				& \Sigma X\ar[dr]^{\bar{\alpha}}\ar[d]^a
				\\ {\Omega K}\rtwocell^0_0{{\;\;\;\,\omega_K}}
				&K\ar[r]_{k}
				&A\ar[r]_f
				&B
			}\end{xym}
			
			It remains to prove that $k\omega_K$ is a monoloop. Let
			$x,x'\col X\ra\Omega K$ be such that $k\omega_K x=k\omega_K x'$.
			Since $k$ is faithful (because it is a kernel), we have $\omega_K x=\omega_K x'$.
			And since $\omega_K$ is a monoloop, there exists a unique 2-arrow
			$x\Ra x'$.
		\end{proof}

	In a $\Gpdp$-category which has all kernels and cokernels, we will define  the pip and the copip in the following way:
	\begin{align}\stepcounter{eqnum}
		\Pip f &\eqdef \Omega\Ker f,\\ \stepcounter{eqnum}
		\Copip f &\eqdef \Sigma\Coker f,
	\end{align}
	equipped with, respectively, $\pi_f\eqdef k_f\omega_{\Ker f}$ and
	$\rho_f\eqdef \sigma_{\Coker f}q_f$.
  
	\begin{pon}\label{coraccoker}
		In the situation of diagram \ref{chibarchi}, $f$ is the coroot
		of $\alpha$ if and only if $f$ is the cokernel of $\bar{\alpha}$.
		Therefore, if $\C$ has all cokernels, $\C$ has all coroots.
	\end{pon}
	
		\begin{proof}
			Lemma \ref{deuflunfl} says that a coroot-candidate of $\alpha$ is the same
			as a cokernel-candidate of $\bar{\alpha}$.
		\end{proof}

	If $\C$ has all kernels and cokernels, we will define the coroot of a  2-arrow $\pi$
	by $\Coroot \pi\eqdef \Coker \bar{\pi}$ and its root by
	$\Root\pi\eqdef \Ker\tilde{\pi}$.\index{coroot!construction by cokernels}
	\index{root!construction by kernels}
    To sum up, if $\C$ has all kernels and cokernels, $\C$ also has
    all pips, copips, roots and coroots.
     
	\begin{rem}\label{remfactregpepkerraccoker}
	By joining the las two propositions, we get a new construction
	of the $\Pip$-regular factorisation  of an arrow $f$, which is normally constructed by
	taking the coroot of the pip of $f$.  By Proposition \ref{kerpep},
	$k_f\omega_{Kf}$ is a pip of $f$ and the coroot of this pip can be constructed, by
	Proposition \ref{coraccoker}, as the cokernel of $k_f\varepsilon_{Kf}\col\pi_1 Kf\ra A$.
	\begin{xym}\newdir{t}{{}*!/-13pt/\dir{}}\xymatrix@=40pt{
		\Omega Kf\rtwocell^0_0{\;\;\;\;\;\omega_{Kf}}\dtwocell_0<5>^0<5>{\sigma_{\Omega Kf}}
		&Kf\ar[r]^-{k_f}
		&A\ar[r]^f
		&B
		\\ \Sigma\Omega Kf\ar@<-2mm>@{t->}[ur]_{\varepsilon_{Kf}}
	}\end{xym}
	So, in particular, an arrow is normal fully cofaithful (in other words
	a $\Pip$-regular epimorphism) if and only if is
	is canonically the cokernel of the $\pi_1$ of its kernel.
	\end{rem}

	\bigskip
	We can also give characterisations of the 0-faithful and fully 0-faithful arrows
	without using loops $0\Ra 0$.
    The following proposition is an immediate consequence of Lemma
    \ref{deuflunfl}.
    
    \begin{pon}\label{fidconnected}
    	Let $A\overset{f}\ra B$ be an arrow in $\C$ which has all $\Sigma$s.
		The following conditions are equivalent:
		\begin{enumerate}
			\item $f$ is 0-faithful;
			\item for all $X\col\C$ and for all $a\col\Sigma X\ra A$, if
				$fa\simeq 0$, then $a\simeq 0$.
		\end{enumerate}
    \end{pon}
    
    We deduce from this proposition that, if $\Sigma$ exists, in condition 1 of
    the definition of fully 0-faithful arrows (\ref{caracplzfid}), we can remove
    point (b), i.e.\ 0-faithfulness, because it follows from point (a).
    
    \begin{pon}\label{simpldefplzfidsigm}
    	If $\C$ has all $\Sigma$s, then an arrow $A\overset{f}\ra B$ in $\C$
		is fully 0-faithful
		if and only if, for all $X\col\C$, for all $a\col X\ra A$
		and for all $\beta\col fa\Ra 0$, there exists
		$\alpha\col a\Ra 0$ such that $\beta=f\alpha$.
    \end{pon}
    
    	\begin{proof}
			Point (b) of characterisation 1 of the definition
			\ref{caracplzfid}, which is 0-faith\-ful\-ness, follows from point (a), if we apply the previous proposition: if $fa\simeq 0$ (where the domain of $a$ is $\Sigma X$),
			then there exists $\alpha\col a\Ra 0$.
		\end{proof}
 
  \begin{pon}\label{lemmhormzfid}
    		Let be $A\overset{f}\ra B$ and $\pi\col 0\Ra 0\col\Omega A\ra B$ in $\C$
		such that $\pi=f\omega_A$.
		If $\pi$ is a monoloop, then $f$ is $0$-faithful.
    \end{pon}
    
    	\begin{proof}
			Let be $\alpha\col 0\Ra 0\col X\ra A$
			such that $f\alpha = 1_0$.  By the equivalences \ref{equadjomsi},
			to $\alpha$ corresponds an arrow
			$\bar{\alpha}\col X\ra\Omega A$ (such that
			$\alpha=\omega_A \bar{\alpha}$).  Then $\pi\bar{\alpha} =
			f\omega_A\bar{\alpha} = f\alpha = 1_0$ and,
			since $\pi$ is a monoloop,
			$\bar{\alpha}\simeq 0$.  Therefore $\alpha=\omega_A\bar{\alpha}=
			\omega_A 0=1_0$.
		\end{proof}

\chapter{Abelian $\Gpd$-categories and homology}\label{chapgpdcatabhom}

\begin{quote}
	{\it In this chapter, after basic definitions and lemmas, we introduce
	Puppe-exact $\Gpdp$-categories, which form a context in which we can study
	exact and relative exact sequences as well as their homology.  Next, in the
	context of abelian $\Gpd$-categories, we prove several classical diagram lemmas and
	the existence of the long exact sequence of homology corresponding to an
	extension of chain complexes (Theorem \ref{lngsequencexhom}).}
\end{quote}

\section{Relative kernel and relative pullback}

For this section we fix a $\Gpdp$-category $\C$, which has all kernels, cokernels and a zero object.

\subsection{Relative fully faithful arrows}

	The notion of fully faithfulness relative to a 2-arrow was introduced
	in \cite{Rio2005a}.

	\begin{df}\index{arrow!relative fully faithful}%
	\index{relative fully faithful arrow}
		Let us consider the following diagram in $\C$.
		We say that $f$ is \emph{$\varphi$-fully faithful} if, for every 
		$X\col\C$, for all $a_1,a_2\col X\ra A$	 and for all $\beta\col fa_1\Ra fa_2$
		compatible with $\varphi$,
		there exists a unique $\alpha\col a_1\Ra a_2$ such that $f\alpha=\beta$.
		\begin{xym}\label{diagwithphi}\xymatrix@=40pt{
			A\ar[r]^f\rrlowertwocell<-9>_0{_<2.7>\varphi}
			&B\ar[r]^y
			&Y
		}\end{xym}
	\end{df}
	
	The second part of the universal property of the kernel of an arrow $f$
	(Definition \ref{defkernel}) says precisely that $k_f$ is $\kappa_f$-fully faithful.
	
	For a given arrow $A\overset{f}\ra B$, we can apply this definition to two canonical
	2-arrows $\varphi$: on the one hand, $1_{0^A}\col 0^Bf\Ra 0^A$ and, 
	on the other hand, the canonical 2-arrow of the cokernel of $f$.
	
	\begin{pon}
		Let $A\overset{f}\ra B$ be an arrow in $\C$.  Then $f$ is 
		$1_{0^A}$-fully faith\-ful if and only if $f$ is fully faithful.
	\end{pon}
	
		\begin{proof}
			Every 2-arrow $\beta$ is trivially compatible with $1_{0^A}$.
		\end{proof}
	
	\begin{df}\label{dfmonomorphs}%
	\index{monomorphism!in a Gpd*-category@in a $\Gpdp$-category}
		An arrow $A\overset{f}\ra B$ in $\C$ is called a \emph{monomorphism}
		if it is $\zeta_f$-fully faithful (where $\zeta_f\col q_ff\Ra 0$ is the
		canonical arrow of the cokernel of $f$).
	\end{df}
	
	As the following proposition shows, these two canonical cases are the extreme cases.
		
	\begin{pon}\label{implicrelplfid}
		Let $f$ be an arrow and $\varphi$ be a 2-arrow
		as in diagram \ref{diagwithphi}.  We have the following implications:
		\begin{quote}
			$f$ is fully faithful $\Rightarrow$ $f$ is $\varphi$-fully faithful
			$\Rightarrow $	$f$ is a monomorphism $\Rightarrow $ $f$ is faithful.
		\end{quote}
	\end{pon}
	
		\begin{proof}
			The first and last implications are obvious.  It remains
			to prove that, if $f$ is $\varphi$-fully faithful, then $f$ is 
			$\zeta_f$-fully faithful.  By the universal property of the cokernel, 
			there exists an arrow $y'\col Qf\ra Y$ and a 2-arrow
			$\iota\col y\Ra y'q_f$ such that $y'\zeta_f\circ\iota f=\varphi$.
			Therefore, if $\beta \col fa_1\Ra fa_2$ is compatible with $\zeta_f$,
			it is also compatible with $\varphi$ and, as $f$ is
			$\varphi$-fully faithful, there exists a unique $\alpha$ such that
			$\beta=f\alpha$.\qedhere
			\begin{xym}\xymatrix@C=20pt@R=40pt{
				A\ar[rr]^-f\ar@/_0.78pc/[drrr]_0\ar@/^2pc/[rrrr]^0
				&{}\rrtwocell\omit\omit{^<-2.7>\varphi\,}
				&B\ar[rr]^-y\ar[dr]_-{q_f}\rrtwocell\omit\omit{_<4>\iota}
					\ar@{}[dll]|(0.32){\zeta_f\!\dir{=>}}
				&{}
				&Y
				\\ {}
				&&&{Qf}\ar[ur]_-{y'}	
			}\end{xym}
		\end{proof}

	In a $\Ensp$-category seen as a locally discrete $\Gpdp$-category, all 2-arrows
	are equal and so several of these notions merge: $f$ is fully faithful if and only if
	$f$ is a monomorphism (in the sense of the above definition), if and only if $f$ is
	a monomorphism (in the usual sense). Moreover, all arrows are faithful.  It is thus in general not
	true that all faithful arrows are monomorphisms.

	On the other hand, Lemma	 5.1 of \cite{Kasangian2000a} says precisely that the monomorphisms in $\CGS$ are the faithful arrows.  More generally, if in $\C$ all faithful
	arrows are canonically the kernel of their cokernel (i.e.\ if $\C$ is $\Ker$-exact, 
	in particular if $\C$ is 2-Puppe-exact), the monomorphisms are the faithful arrows.
	The interest of this notion of monomorphisms is thus that it is relevant both
	in dimension 1 and in dimension 2.
	
	We can also define a relative version of the notion of fully 0-faithful arrow.		
	\begin{df}\index{arrow!relative fully 0-faithful}%
	\index{relative fully 0-faithful arrow}
		Let $f$, $y$ and $\varphi$ be as in diagram \ref{diagwithphi}.
		We say that $f$ is \emph{$\varphi$-fully 0-faithful} if for every 2-arrow
		$\beta\col fa\Ra 0\col X\ra B$ compatible with $\varphi$, there exists a unique
		$\alpha\col a\Ra 0$ such that $f\alpha = \beta$.
		We say that $f$ is a \emph{0-monomorphism}%
		\index{0-monomorphism!in a Gpd*-category@in a $\Gpdp$-category}
		if $f$ is $\zeta_f$-fully 0-faithful.
	\end{df}
	
	If $f$ is $\varphi$-fully faithful, $f$ is also $\varphi$-fully
	0-faithful, and $f$ is $1_0$-fully 0-faithful if and only if $f$
	is fully 0-faithful.  Moreover, the implications corresponding to those of Proposition \ref{implicrelplfid} hold.

\subsection{Relative kernel}

	The notion of relative kernel have been introduced in \cite{Rio2005a} for symmetric 2-groups.
	The difference with the usual kernel is that we ask that $\kappa_f$ be
	compatible with a given 2-arrow $\varphi$.
	
	\begin{df}
		Let $f$, $y$ and $\varphi$ be as in the following diagram.
		\begin{xym}\label{diagrelker}\xymatrix@=40pt{
			{K}\ar[r]^-{k}
				\rrlowertwocell<-9>_0{_<2.7>\kappa}
			&A\ar[r]^f\rruppertwocell<9>^0{^<-2.7>\varphi\,}
			&B\ar[r]_y
			&Y
		}\end{xym}
		We call an object $K$, equipped with an arrow $k$ and a 2-arrow $\kappa$
		compatible with $\varphi$ a \emph{relative kernel}\index{kernel!relative}%
		\index{relative!kernel} of $f$, $y$, $\varphi$
		if the following conditions hold:
		\begin{enumerate}
			\item for all $X\col\C$, $a\col X\ra A$, $\beta\col fa\Ra 0$ such that
				$\beta$ is compatible with $\varphi$, there exist an arrow
				$a'\col X\ra K$ and a 2-arrow $\alpha\col a\Ra ka'$
				such that $\beta = \kappa a'\circ f\alpha$:
				\begin{xyml}\xymatrix@C=20pt{
					X\ar[rr]^a\ar@/_2pc/[rrrr]_0
					&{}\rrtwocell\omit\omit{_<2.7>\,\beta}
					&A\ar[rr]^f
					&{} &B
				}\;=\;\xymatrix@C=20pt@R=40pt{
					X\ar[rr]^a\ar[dr]_{a'}\rrtwocell\omit\omit{_<4>\alpha}
					&&A\ar[rr]^f\ar@{}[drr]|(0.32){\dir{=>}\!\kappa}
					&&B
					\\ &K\ar[ur]_k\ar@/_0.78pc/[urrr]_0 &&&{}
				};\end{xyml}
			\item $k$ is $\kappa$-fully faithful.		
		\end{enumerate}
		We write $(K,k,\kappa)=\Ker(f,\varphi)$ the fact that $(K,k,\kappa)$
		is a kernel of $f$ relative to $\varphi$ (we often omit the object $K$).
	\end{df}
	
	We can construct the kernel of $f$ relative to $\varphi$ from the kernel of $f$, by
	a root.  So $\C$ has all relative kernels, since it has all kernels (and so all roots,
	by the dual of Proposition \ref{coraccoker}).
	
	\begin{pon}
		Let be $A\overset{f}\ra B$ and $\varphi\col yf\Ra 0$ in $\C$.
		Let us denote by $\pi$ the composite $y\kappa_f\circ \varphi^{-1}k_f$ and $r\col\Root\pi\ra Kf$
		the root of $\pi$. Then
		\begin{eqn}
			(\Root \pi,k_fr,\kappa_fr)=\Ker(f,\varphi).
		\end{eqn}
		We denote by $(\Ker(f,\varphi),k_{f,\varphi},\kappa_{f,\varphi})$ this particular
		construction of the relative kernel.
		\begin{xym}\xymatrix@=40pt{
			{\Root \pi}\ar[r]^-r
			&{Kf}\ar[r]^-{k_f}
				\rrlowertwocell<-9>_0{_<2.7>\;\;\;\kappa_f}
			&A\ar[r]^f\rruppertwocell<9>^0{^<-2.7>\varphi\,}
			&B\ar[r]_y
			&Y
		}\end{xym}
	\end{pon}
	
		\begin{proof}
			To give $X\overset{a}\ra A$ and $\beta\col fa\Ra 0$ compatible with $\varphi$
			amounts to give $X\overset{a'}\ra Kf$ such that
			$\pi a'=1_0$.
		\end{proof}

	We recover as a special case of the relative kernel the usual kernel, the root and the pip; these are the cases where, respectively, $Y$, $B$ and $A$ are $0$ in diagram
	\ref{diagwithphi}.
	\begin{pon}\label{caspartickerrel}
		Let be $A\overset{f}\ra B$ in $\C$.
		\begin{enumerate}
			\item In the situation of the following diagram, $(K,k,\kappa)=\Ker f$
				if and only if $(K,k,\kappa)=\Ker(f,1_{0^A})$.
				\begin{xym}\xymatrix@=40pt{
					{K}\ar[r]^-{k}
						\rrlowertwocell<-9>_0{_<2.7>\;\;\;\kappa}
					&A\ar[r]^f\rruppertwocell<9>^0{\omit}
					&B\ar[r]_0
					&0
				}\end{xym}
			\item In the situation of the following diagram, $(P,\pi)=\Pip f$
				if and only if $(P,0^P,\pi) =\Ker(0_A,1_{0_B})$.
				\begin{xym}\xymatrix@=40pt{
					{P}\ar[r]
						\rrlowertwocell<-9>_0{_<2.7>\;\;\;\pi}
					&0\ar[r]\rruppertwocell<9>^0{\omit}
					&A\ar[r]_f
					&B
				}\end{xym}
		\end{enumerate}
		Let be $\pi\col 0\Ra 0\col A\ra B$.
		\begin{enumerate}
			\item[3.] In the situation of the following diagram, 
				$(R,r)=\Root \pi$ if and only if $(R,r,1_{0^R})=\Ker(0^A,\pi)$.
				\begin{xym}\xymatrix@=40pt{
					{R}\ar[r]^-{r}
						\rrlowertwocell<-9>_0{\omit}
					&A\ar[r]\rruppertwocell<9>^0{^<-2.7>\pi\,}
					&0\ar[r]_0
					&B
				}\end{xym}
		\end{enumerate}
	\end{pon}
	
	The kernel classifies the fully 0-faithful arrows; the relative kernel
	classifies the relative fully 0-faithful arrows.
	
	\begin{pon}\label{clasproprelker}
		In the situation of diagram \ref{diagrelker} the following conditions are equivalent:
		\begin{enumerate}
			\item $f$ is $\varphi$-fully 0-faithful%
				\index{arrow!relative fully 0-faithful!characterisation};
			\item if $(K,k,\kappa)=\Ker(f,\varphi)$, then there exists
				$\kappa'\col k\Ra 0$ such that $f\kappa'=\kappa$;
			\item $(0,0_A,1_{0_B})=\Ker(f,\varphi)$.
		\end{enumerate}
	\end{pon}
	
		\begin{proof}
			{\it 1 $\Rightarrow$ 3. }
			For the first part of the universal property of the kernel,
			let be $X\col\C$, $a\col X\ra A$ and $\beta\col fa\Ra 0$ compatible
			with $\varphi$. As $f$ is $\varphi$-fully 0-faithful, there exists
			a unique $\alpha\col a\Ra 0$ such that $\beta=f\alpha$.  We have thus
			$0^X\col X\ra 0$ and $\alpha\col a\Ra 0\equiv 0_A0^X$ such that
			$1_{0_B}0^X\circ f\alpha = \beta$.
			
			It remains to prove that $0_A$ is $1_{0_B}$-fully faithful.
			Let be $x_1,x_2\col X\ra 0$ and $\alpha\col 0_A x_1\Ra 0_Ax_2
			\col X\ra A$ compatible with $1_{0_B}$, i.e.\ such that $f\alpha=1_0$.
			As $f$ is 0-faithful, $\alpha=1_0$.
			So, if $\chi$ is the unique 2-arrow $x_1\Ra x_2$ ($0$ is 
			terminal), $\alpha = 0_A\chi$.
			
			{\it 3 $\Rightarrow$ 2. }
			By the universal property of the kernel, there exist $k'\col K\ra 0$ and
			$\kappa'\col k\Ra 0_Ak'$ such that $\kappa = f\kappa'$.
			
			{\it 2 $\Rightarrow$ 1. }
			We use Proposition \ref{simpldefplzfidsigm}.
			Let be $a\col X\ra A$ and $\beta\col fa\Ra 0$ compatible with $\varphi$.
			By the universal property of the kernel, there exist $a'\col X\ra K$ and
			$\alpha\col a\Ra ka'$ such that $\kappa a'\circ f\alpha=\beta$.
			So $\beta=f(\kappa'a'\circ\alpha)$.
		\end{proof}
	
	This proposition contains as special cases the already known facts that the ordinary
	kernel classifies the fully 0-faithful arrows and that the pip classifies the 0-faithful
	arrows; another special case is that the root classifies the 0-monoloops (the loops $\pi\col 0\Ra 0\col A\ra B$
	such that, if $\pi a=1_0$, then there is a unique 2-arrow $a\Ra 0$):
	\begin{enumerate}
		\item $f$ is fully 0-faithful $\Leftrightarrow$ $f$ is $1_{0^A}$-fully 0-faithful
			$\Leftrightarrow$ $0=\Ker(f,1_{0^A})$ $\Leftrightarrow$ $0=\Ker f$;
		\item $f$ is 0-faithful $\Leftrightarrow$ $0_A$ is $1_{0_B}$-fully 0-faithful
			$\Leftrightarrow$ $0=\Ker(0_A,1_{0_B})$ $\Leftrightarrow$ $0=\Pip f$;
		\item $\pi$ is a 0-monoloop $\Leftrightarrow$ $0^A$ is $\pi$-fully
			0-faithful $\Leftrightarrow$ $0=\Ker(0^A,\pi)$ $\Leftrightarrow$ $0=\Root \pi$.
	\end{enumerate}
	
	In dimension 1, the kernel of the kernel of an arrow $f$ is always $0$.  As
	Lemma \ref{kerkeromegcodom} shows, this is not true any more in dimension 2:
	the kernel of the kernel of $f\col A\ra B$ is $\Omega B$.  But if we take the kernel
	of the kernel relative to the canonical 2-arrow of the kernel, we get $0$.
	
	\begin{coro}\label{kerkerzer}
		If, in the situation of the following diagram, $(K,k,\kappa)=\Ker (f,\varphi)$,
		then $(0,0_K,1_{0_A})=\Ker(k,\kappa)$.
		\begin{xym}\label{diagkerkerzer}\xymatrix@=40pt{
			{0}\ar[r]\rruppertwocell<9>^0{\omit}
			&{K}\ar[r]_-{k}
				\rrlowertwocell<-9>_0{_<2.7>\kappa}
			&A\ar[r]^f\rruppertwocell<9>^0{^<-2.7>\varphi\,}
			&B\ar[r]_y
			&Y
		}\end{xym}
	\end{coro}
	
		\begin{proof}
			This is a direct application of the previous proposition, since
			$k$ is $\kappa$-fully faithful.
		\end{proof}
	
\subsection{Relative pullback}

	The relative pullback is a generalisation of the pullback, which we recover by
	taking $0$ for $Y$ in the following definition.

	\begin{df}
		Let us consider the solid part of the following diagram.
		\begin{xym}\xymatrix@=20pt{
			P\ar@{-->}[rr]^{p_1}\ar@{-->}[dd]_{p_2}\ddrrtwocell\omit\omit{_{\pi}}
			&&A\ar[dd]^f\ar@/^1pc/[dddr]^0
			\\ &&&{}
			\\ B\ar[rr]_g\ar@/_1pc/[drrr]_0
			&&C\ar[dr]^y\ar@{}[ur]|{\varphi\dir{=>}}\ar@{}[dl]|{\psi\!\dir{=>}}
			\\ &{} &&Y
		}\end{xym}
		The dashed part is a \emph{pullback 
		of $f$ and $g$ relative to $\varphi$ and $\psi$}\index{relative!pullback}%
		\index{pullback!relative}
		if the composite of the diagram is $1_0$\footnote{i.e.\ if
		$\psi p_2\circ y\pi\circ\varphi^{-1} p_1=1_0$; we say then that
		$\pi$ is \emph{compatible with $\varphi$ and $\psi$}%
		\index{compatibility of 2-arrows}}
		and if
		\begin{enumerate}
			\item for all $X\col\C$, $a\col X\ra A$, $b\col X\ra B$, and 
				$\gamma\col fa\Ra gb$ compatible with $\varphi$ and $\psi$,
				there exist $(a,\gamma,b)\col X\ra P$, $\pi_1\col a\Ra p_1(a,\gamma,b)$
				and $\pi_2\col b\Ra p_2(a,\gamma,b)$ such that
				\begin{xyml}\begin{gathered}\xymatrix@=20pt{
					X\ar[dr]|{(a,\gamma,b)}\ar@/^1pc/[drrr]^a\ar@/_1pc/[dddr]_b 
					&&{}\ar@{}[dl]|(0.6){\pi_1\!\dir{=>}}
					\\ &P\ar[rr]^{p_1}\ar[dd]_{p_2}\ddrrtwocell\omit\omit{_{\pi}}
						\ar@{}[dl]|{\pi_2^{-1}\!\!\!\dir{=>}}
					&&A\ar[dd]^f
					\\ {}
					\\ &B\ar[rr]_g
					&&C
				}\end{gathered}\;\;=\;\;\gamma;
				\end{xyml}
			\item condition 2 of Definition \ref{defprodfib} hold.
		\end{enumerate}
	\end{df}
	
	The pullback-candidates of $f$ and $g$ relative to $\varphi$ and $\psi$
	form a $\Gpd$-category $\mathrm{PBCand}(f,g;\varphi,\psi)$ (if $Y$ is $0$,
	we write simply $\mathrm{PBCand}(f,g)$):\index{PBCand@$\mathrm{PBCand}$}
	\begin{itemize}
		\item {\it objects:} an object consists of $X\col\C$, $a\col X\ra A$,
			$b\col X\ra B$ and $\gamma\col fa\Ra gb$ compatible with $\varphi$ and $\psi$;
		\item {\it arrows:} an arrow $(X,a,b,\gamma)\ra(X',a',b',\gamma')$
			consists of $x\col X\ra X'$, $\zeta\col a\Ra a'x$ and
			$\xi\col b\Ra b'x$
			such that $g\xi^{-1}\circ\gamma'x\circ f\zeta=\gamma$;
			the composition and identities are defined in an obvious way;
		\item {\it 2-arrows:} a 2-arrow $(x,\zeta,\xi)\Ra(x',\zeta',\xi')$
			consists of $\chi\col x\Ra x'$ such that $\zeta'= a'\chi \circ \zeta$
			and $\xi'= b'\chi\circ \xi$; the compositions and identities
			are defined as in $\C$.
	\end{itemize}

	\begin{pon}
		A pullback of $f$ and $g$ relative to $\varphi$ and $\psi$ is
		a terminal object in $\mathrm{PBCand}(f,g;\varphi,\psi)$.
	\end{pon}
	
	The relative kernel can be expressed in terms of relative pullback,
	in the same way as the kernel is expressed as a pullback.
	
	\begin{pon}
		Let us consider the situation of the following diagram.  Then $(K,k,\kappa)=\Ker(f,\varphi)$
		if and only if $\kappa$ is a pullback relative to $\varphi$
		and $1_0$.
		\begin{xym}\xymatrix@=20pt{
			K\ar[rr]^{k}\ar[dd]_{0}\ddrrtwocell\omit\omit{_{\kappa}}
			&&A\ar[dd]^f\ar@/^1pc/[dddr]^0
			\\ &&&{}
			\\ 0\ar[rr]_0\ar@/_1pc/[drrr]_0
			&&B\ar[dr]^y\ar@{}[ur]|{\varphi\dir{=>}}\ar@{}[dl]|{1_0\!\dir{=>}}
			\\ &{} &&Y
		}\end{xym}
	\end{pon}

	Let us prove now the cube lemma, which generalises the cancellation property of
	pullbacks and from which follow all the diagram lemmas of the following section.
	
	\begin{pon}[Cube lemma]\label{lemmcub}\index{lemma!cube}\index{cube lemma}
		Let us consider the following diagram.
		\begin{xym}\xymatrix@=20pt{
			&A_1\ar[rr]^e\ar[dl]_{k_1}\ar'[d][dd]^(-0.4){l_1}
			&&A_2\ar[dd]^{l_2}\ar[dl]^{k_2}
			\\ B_1\ar[rr]_(0.3)f\ar[dd]_{m_1}
			&&B_2\ar[dd]^(0.3){m_2}
			\\ &C_1\ar'[r][rr]^(-0.4)g\ar[dl]_(0.4){n_1}
			&&C_2\ar[dl]_{n_2}\ar@/^/[dd]^0\ddtwocell\omit\omit{^<2>\varphi}
			\\ D_1\ar[rr]_h\ar@/_1pc/[drrr]_0
			&&D_2\ar[dr]_y\ar@{}[dl]|{\psi\dir{=>}}
			\\ &&&Y
		}\end{xym}
		The six faces of the cube contain each a 2-arrow, and these 2-arrows satisfy the
		equation
		\begin{xyml}\begin{gathered}\xymatrix@=20pt{
			&A_1\ar[rr]^e\ar[dl]_{k_1}\ar[dd]^{l_1}\ddrrtwocell\omit\omit{\lambda}
			&&A_2\ar[dd]^{l_2}
			\\ B_1\ar[dd]_{m_1}\dtwocell\omit\omit{_<-3.5>\;\;\;\;\;\;\omega_1^{-1}}
			\\ &C_1\ar[rr]^g\ar[dl]^(0.4){n_1}\drtwocell\omit\omit{\nu}
			&&C_2\ar[dl]^{n_2}
			\\ D_1\ar[rr]_h
			&&D_2
		}\end{gathered}\;\;=\;\;\begin{gathered}\xymatrix@=20pt{
			&A_1\ar[rr]^e\ar[dl]_{k_1}\drtwocell\omit\omit{\kappa}
			&&A_2\ar[dd]^{l_2}\ar[dl]_{k_2}
			\\ B_1\ar[rr]_f\ar[dd]_{m_1}\ddrrtwocell\omit\omit{\mu}
			&&B_2\ar[dd]_{m_2}\dtwocell\omit\omit{_<-3.5>\;\;\;\;\;\;\omega_2^{-1}}
			\\ &&{}&C_2\ar[dl]^{n_2}
			\\ D_1\ar[rr]_h
			&&D_2
		}\end{gathered}.\end{xyml}
		If $\nu$ is a pullback relative to $\varphi$ and $\psi$
		and if $\mu$ is a pullback, then the following conditions are equivalent:
		\begin{enumerate}
			\item $\kappa$ is a pullback relative to 
				$\varphi l_2\circ y\omega_2$ and $\psi m_1\circ y\mu$;
			\item $\lambda$ is a pullback.
		\end{enumerate}
	\end{pon}
	
		\begin{proof}
			Set $\varphi'\eqdef\varphi l_2\circ y\omega_2$
			and $\psi'\eqdef\psi m_1\circ y\mu$. We define a $\Gpd$-functor
			\begin{eqn}
				\Psi\col\mathrm{PBCand}(k_2,f;\varphi',\psi')\ra\mathrm{PBCand}(l_2,g).
			\end{eqn}
			
			{\it Objects.}  Let $(X,a_2,b_1,\beta_2)$ be an object of
			$\mathrm{PBCand}(k_2,f;\varphi',\psi')$
			(we have thus $X\col\C$, $a_2\col X\ra A_2$, $b_1\col X\ra B_1$ and
			$\beta_2\col k_2a_2\Ra fb_1$ compatible with $\varphi'$
			and $\psi'$).  By the universal property of relative pullback
			for $\nu$, there exist $c_1\col X\ra C_1$, $\gamma_2\col l_2a_2\Ra gc_1$ and
			$\delta_1\col m_1b_1\Ra n_1c_1$ (see diagram \ref{tutsuipr}) such that
			\begin{eqn}\label{lahaut}
				h\delta_1\circ \mu b_1\circ m_2\beta_2
				= \nu c_1\circ n_2\gamma_2\circ\omega_2a_2.
			\end{eqn}
			We set $\Psi(X,a_2,b_1,\beta_2)\eqdef (X,a_2,c_1,\gamma_2)$.
			
			{\it Arrows.} Let be $(x,\zeta,\xi)\col(X,a_2,b_1,\beta_2)
			\ra(X',a'_2,b'_1,\beta'_2)$ (i.e.\ $x\col X\ra X'$, 
			$\zeta\col a_2\Ra a'_2x$ and $\xi\col b_1\Ra b'_1x$).
			By the universal property of $\nu$, there exists a
			unique $\theta\col c_1\Ra c'_1x$ such that
			\begin{eqn}\label{eqlahautbis}
				n_1\theta^{-1}\circ\delta'_1 x\circ m_1\xi = \delta_1,
			\end{eqn}
			\begin{eqn}\label{eqlaunpeu}
				g\theta^{-1}\circ\gamma'_2 x\circ l_2\zeta = \gamma_2.
			\end{eqn}
			We set $\Psi(x,\zeta,\xi)\eqdef (x,\zeta,\theta)$, which is
			a morphism in $\mathrm{PBCand}(l_2,g)$, by this last equation.
			
			{\it 2-arrows.} Let be $\chi\col(x,\zeta,\xi)\Ra (x',\zeta',\xi')$.  We
			set $\Psi(\chi)\eqdef\chi$.  We check that it is a 2-arrow
			of $\mathrm{PBCand}(l_2,g)$ by testing it with $g$ and $n_1$, which
			are jointly faithful by the universal property of $\nu$.
			
			\vspace{1em}
			Next, we prove that $\Psi$ is an equivalence.
			
			{\it Surjective.} Let be $(X,a_2,c_1,\gamma_2)\col \mathrm{PBCand}(l_2,g)$.
			By the universal property of pullback of $\mu$, there exist
			$b_1\col X\ra B_1$, $\beta_2\col k_2a_2\Ra fb_1$ and $\delta_1\col m_1b_1
			\Ra n_1c_1$ (see diagram \ref{tutsuipr}) such that equation
			\ref{lahaut} hold.  If we apply the above construction of $\Psi$
			to $(X,a_2,b_1,\beta_2)$, we get $c'_1\col X\ra C_1$,
			$\gamma'_2\col l_2a_2\Ra gc'_1$ and $\delta'_1\col m_1b_1\Ra n_1c'_1$
			satisfying the same equation.  Then, by the universal property
			of relative pullback of $\nu$, there exists a 2-arrow
			$\gamma_1\col c_1\Ra c'_1$ which gives us an isomorphism
			\begin{eqn}
				(1_X,1_{a_2},\gamma_1)\col(X,a_2,c_1,\gamma_2)\Ra
				(X,a_2,c_1',\gamma'_2)=\Psi(X,a_2,b_1,\beta_2).
			\end{eqn}
						
			{\it Full.} Let $(X,a_2,b_1,\beta_2)$ and $(X',a'_2,b'_1,\beta'_2)$
			be objects in $\mathrm{PBCand}(k_2,f;\varphi',\psi')$ and let
			$(x,\zeta,\theta)$ be an arrow between their images by $\Psi$ in 
			$\mathrm{PBCand}(l_2,g)$.  Then, by the universal property of $\mu$,
			there exists $\xi\col b_1\Ra b'_1x$ satisfying equation \ref{eqlahautbis}.
			If we apply $\Psi$ to $(x,\zeta,\xi)$, we get a unique
			$\theta'$ satisfying equations \ref{eqlahautbis} and \ref{eqlaunpeu}.
			As $\theta$ also satisfies these equations, $\theta'=\theta$ and 
			$1_x$ is a isomorphism between $(x,\zeta,\theta')=\Psi(x,\zeta,\xi)$
			and $(x,\zeta,\theta)$.
			
			{\it Faithful.} If $\chi\col\Psi(x,\zeta,\xi)\Ra\Psi(x',\zeta',\xi')$
			is a 2-arrow in $\mathrm{PBCand}(l_2,g)$, then $\chi$ is also
			a 2-arrow $(x,\zeta,\xi)\Ra(x',\zeta',\xi')$: we can check it by testing
			with $f$ and $m_1$, which are jointly faithful since
			$\mu$ is a pullback.
			
			\vspace{1em}
			Finally, since $\Psi(A_1,e,k_1,\kappa)\simeq(A_1,e,l_1,\lambda)$
			(we are in the situation of the proof of the surjectivity of $\Psi$) and
			since $\Psi$ is an equivalence, $(A_1,e,k_1,\kappa)$
			is an initial object if and only if $(A_1,e,l_1,\lambda)$ is an initial object,
			in other words $\kappa$ is a relative pullback if and only
			if $\lambda$ is a pullback.\qedhere
			\begin{xym}\label{tutsuipr}\xymatrix@=20pt{
				&X\ar[rr]^{a_2}\ar[dl]_{b_1}\ar'[d][dd]_(-0.4){c_1}
					\drtwocell\omit\omit{\;\beta_2}
				&&A_2\ar[dd]^{l_2}\ar[dl]^{k_2}
				\\ B_1\ar[rr]_(0.7)f\ar[dd]_{m_1}
					\dtwocell\omit\omit{^<-4>\;\;\;\;\;\;\delta_1}
				&&B_2
				&{}\ar@{}[dl]|{\gamma_2\!\!\dir{=>}~~~~~~~}
				\\ &C_1\ar[rr]_g\ar[dl]^{n_1}
				&{}&C_2
				\\ D_1
			}\end{xym}
		\end{proof}

\subsection{Small diagram lemmas}

	\begin{lemm}\label{lemmun}
		Let be the following diagram in $\C$, where
		\begin{eqn}
			 \varphi=\varphi'b\circ y\nu
			 \text{ and }\kappa = \kappa'a\circ g'\mu\circ\nu f.
		\end{eqn}
		\begin{xym}\xymatrix@C=40pt@R=15pt{
			A\ar[dd]_a\ar[r]^f\ddrtwocell\omit\omit{\mu}
				\drruppertwocell^0<15>{^<-5.6>{\kappa}}
			&B\ar[dr]^g\ar[dd]^b\drruppertwocell^0{^<-1.3>{\varphi}}
				\drtwocell\omit\omit{_<3.5>\nu}
			\\ &&C\ar[r]^y
			&Y
			\\ A'\ar[r]_{f'}\urrlowertwocell_0<-15>{_<5.6>{\;\,\kappa'}}
			&B'\ar[ur]_{g'}\urrlowertwocell_0{_<1.3>{\;\,\varphi'}}
		}\end{xym}
		If $(f',\kappa')=\Ker(g',\varphi')$, then the following conditions are
		equivalent:
		\begin{enumerate}
			\item $\mu$ is a pullback;
			\item $(f,\kappa)=\Ker(g,\varphi)$.
		\end{enumerate}
	\end{lemm}
	
		\begin{proof}
			We apply the cube lemma (Proposition \ref{lemmcub}) to the situation of the following
			diagram. The front face is obviously a pullback.  The bottom face of the cube
			is a pullback relative to $\varphi'$ and $1_0$,
			by the hypothesis $(f',\kappa')
			=\Ker(g',\varphi')$.  Thus $\kappa$ is a pullback relative
			to $\varphi'b\circ y\nu(=\varphi)$ and $1_0$
			(i.e.\ $(f,\kappa)=\Ker(g,\varphi)$) if and only if $\mu$
			is a pullback.\qedhere
			\begin{xym}\xymatrix@=20pt{
				&A\ar[rr]^f\ar[dl]_{0}\ar'[d][dd]_(-0.4){a}\drtwocell\omit\omit{\kappa}
				&&B\ar[dd]^{b}\ar[dl]^{g}
				\\ 0\ar[rr]_(0.3)0\ar@{=}[dd]
				&&C\ar@{=}[dd]\dtwocell\omit\omit{^<-4>\nu}
				\\ &A'\ar'[r][rr]^(-0.4){f'}\ar[dl]_(0.4){0}
					\drtwocell\omit\omit{\kappa'}
				&{}&B'\ar[dl]_{g'}\ar@/^/[dd]^0\ddtwocell\omit\omit{^<2>\varphi'}
				\\ 0\ar[rr]_0\ar@/_1pc/[drrr]_0
				&&C\ar[dr]_y
				\\ &&&Y
			}\end{xym}
		\end{proof}

	\begin{lemm}\label{lemmtwo}
		Let us consider the situation of the following diagram, where $k'\eqdef ak$ and where
		$\kappa'\eqdef b\kappa\circ\mu k$.
		\begin{xym}\label{diaglemtwo}\xymatrix@=40pt{
			K\ar[r]^{k}\ar@{=}[d]\rruppertwocell<9>^0{^<-2.7>\kappa\,}
			&A\ar[r]^{f}\ar[d]_{a}\drtwocell\omit\omit{^{\mu}}
			&B\ar[d]^{b}
			\\ K\ar[r]_{k'}\rrlowertwocell<-9>_0{_<2.7>\;\kappa'}
			&A'\ar[r]_{f'}
			&B'
		}\end{xym}
		If $\mu$ is a pullback, then the following conditions are equivalet:
		\begin{enumerate}
			\item $(k,\kappa)=\Ker f$;
			\item $(k',\kappa')=\Ker f'$.
		\end{enumerate}
	\end{lemm}
	
		\begin{proof}
			We apply the cube lemma to the situation of the following diagram.
			The front face is obviously a pullback, whereas the bottom face
			 is a pullback by hypothesis.  Therefore $\kappa$
			is a pullback if and only if $\kappa'$ is a pullback.
			\qedhere
			\begin{xym}\xymatrix@=20pt{
				&K\ar[rr]^0\ar[dl]_{k'}\ar'[d][dd]_(-0.4){k}
					\drtwocell\omit\omit{^\kappa'}
				&&0\ar[dd]^{0}\ar[dl]^{0}
				\\ A'\ar[rr]_(0.3){f'}\ar@{=}[dd]
				&&B'\ar@{=}[dd]
				\\ &A\ar'[r][rr]^(-0.4){f}\ar[dl]_(0.4){a}\drtwocell\omit\omit{^\mu}
				&{}&B\ar[dl]^{b}\ar@/^/[dd]^0
				\\ A'\ar[rr]_{f'}\ar@/_1pc/[drrr]_0
				&&B'\ar[dr]_0
				\\ &&&0
			}\end{xym}
		\end{proof}

	\begin{coro}\label{lemmtrois}
		Let $\mu\col f'a\Ra bf$ be a pullback, where $f$ and $f'$ are 
		normal cofaithful arrows.
		Then $\mu$ is also a pushout.
		\begin{xym}\xymatrix@=40pt{
			A\ar@{->>}[r]^{f}\ar[d]_{a}\drtwocell\omit\omit{^{\mu}}
			&B\ar[d]^{b}
			\\ A'\ar@{->>}[r]_{f'}
			&B'
		}\end{xym}
	\end{coro}
	
		\begin{proof}
			Let us consider diagram \ref{diaglemtwo}, where $(K,k,\kappa)$ is
			the kernel of $f$.  As $\mu$ is a pullback, by Lemma
			\ref{lemmtwo}, $(k',\kappa')$ (where $k'\eqdef ak$ and
			$\kappa'\eqdef b\kappa\circ\mu k$) is a kernel of $f'$.
			Then, since $f$ and $f'$ are normal cofaithful,
			$(f,\kappa)=\Coker k$ and $(f',\kappa')=\Coker k'$. So, by the dual of
			Lemma \ref{lemmun} (with $Y\eqdef 0$), $\mu$ is a pushout.
		\end{proof}
	
	Let us prove now a lemma from which will follow the cancellation property of kernels
	and the restricted kernel lemma.
	
	\begin{lemm}\label{lemmquatre}
		Let be the situation of the following diagram, where
		$\kappa'a\circ g'\mu=c\kappa\circ\nu f$.
		\begin{xym}\xymatrix@=40pt{
			A\ar[r]^{f}\ar[d]_a\rruppertwocell<9>^0{^<-2.7>\kappa\,}
				\drtwocell\omit\omit{_{\mu}}
			&B\ar[r]^{g}\ar[d]_{b}\drtwocell\omit\omit{^{\nu}}
			&C\ar[d]^{c}\dduppertwocell<9>^0{^<-2.7>\varphi}
			\\ A'\ar[r]_{f'}\rrlowertwocell<-9>_0{_<2.7>\;\kappa'}
			&B'\ar[r]_{g'}
			&C'\ar[d]_{y}
			\\ &&Y
		}\end{xym}
		If $c$ is $\varphi$-fully 0-faithful and $(f',\kappa')=\Ker g'$,
		then the following conditions are equivalent:
		\begin{enumerate}
			\item $\mu$ is a pullback relative to $\varphi g\circ y\nu$
				and $y\kappa'$;
			\item $(f,\kappa)=\Ker g$.
		\end{enumerate}
	\end{lemm}
	
		\begin{proof}
			We apply the cube lemma (Proposition \ref{lemmcub}) to the situation of the 
			following diagram.  The front face is a pullback, by the hypothesis
			$(f',\kappa')=\Ker g'$; since $c$ is $\varphi$-fully 0-faithful,
			the bottom face is a relative pullback, by Proposition 
			\ref{clasproprelker}. So conditions 1 and 2 are equivalent.\qedhere
			\begin{xym}\xymatrix@=20pt{
				&A\ar[rr]^f\ar[dl]_{a}\ar'[d][dd]_(-0.4){0}\drtwocell\omit\omit{\mu}
				&&B\ar[dd]^{g}\ar[dl]^{b}
				\\ A'\ar[rr]_(0.3){f'}\ar[dd]_0
				&&B'\ar[dd]_(0.3){g'}\dtwocell\omit\omit{^<-4>\nu}
				\\ &0\ar'[r][rr]_(-0.4){0}\ar@{=}[dl]
				&{}&C\ar[dl]_{c}\ar@/^/[dd]^0\ddtwocell\omit\omit{^<2>\varphi}
				\\ 0\ar[rr]_0\ar@/_1pc/[drrr]_0
				&&C'\ar[dr]_y
				\\ &&&Y
			}\end{xym}
		\end{proof}

	We know that in dimension 1, if $f=mf'$, where $m$ is a 0-monomorphism,
	then $\Ker f = \Ker f'$.  Here is the corresponding property for relative kernels.
		
	\begin{coro}\label{lemmsimplnoy}
		Let us consider the following situation, where $\varphi = \varphi' f'\circ  y\mu$,
		$\kappa=m\kappa'\circ\mu k$, and where $m$ is $\varphi'$-fully 0-faithful.
		\begin{xym}\xymatrix@C=20pt@R=40pt{
			K\ar[rr]^k\ar@/_0.78pc/[drrr]_0\ar@/^2pc/[rrrr]^0
			&{}\rrtwocell\omit\omit{^<-2.7>\kappa\,}
			&A\ar[rr]^f\ar[dr]_{f'}\ar@/^2pc/[rrrr]^0\rrtwocell\omit\omit{_<4>\mu}
				\ar@{}[dll]|(0.32){\kappa'\!\!\!\dir{=>}}
			&{}\rrtwocell\omit\omit{^<-2.7>\varphi\,}
			&B\ar[rr]^y\ar@{}[drr]|(0.32){\dir{=>}\!\varphi'}
			&{}
			&Y
			\\ &&&B'\ar[ur]_m\ar@/_0.78pc/[urrr]_0 &&&{}
		}\end{xym}
		Then the following conditions are equivalent:
		\begin{enumerate}
			\item $(k,\kappa)=\Ker(f,\varphi)$;
			\item $(k,\kappa')=\Ker f'$.
		\end{enumerate}
	\end{coro}
	
		\begin{proof}
			It suffices to apply the previous lemma to the situation of the following
			diagram.\qedhere
			\begin{xym}\xymatrix@=40pt{
				K\ar[r]^{k}\ar[d]_0\rruppertwocell<9>^0{^<-2.7>\kappa'\,}
					\drtwocell\omit\omit{_{\kappa}}
				&A\ar[r]^{f'}\ar[d]_{f}\drtwocell\omit\omit{^{\mu}}
				&B'\ar[d]^{m}\dduppertwocell<9>^0{^<-2.7>\varphi'}
				\\ 0\ar[r]_{0}\rrlowertwocell<-9>_0{_<2.7>\;1_0}
				&B\ar@{=}[r]
				&B\ar[d]_{y}
				\\ &&Y
			}\end{xym}
		\end{proof}
		
	A corollary of this corollary is that relative kernels are kernels.
	
	\begin{coro}\label{coronoyrelnoy}
		Let $f$, $y$, $\varphi$ be as in the following diagram and $f'\col A\ra Ky$ be
		the induced arrow to the kernel of $y$.  Then $(K,k,\kappa)=\Ker(f,\varphi)$
		if and only if $(K,k,\kappa')=\Ker f'$.  Therefore $f$ is
		$\varphi$-fully 0-faithful if and only if $f'$ is fully 0-faithful.
		\begin{xym}\xymatrix@C=20pt@R=40pt{
			K\ar[rr]^k\ar@/_0.78pc/[drrr]_0\ar@/^2pc/[rrrr]^0
			&{}\rrtwocell\omit\omit{^<-2.7>\kappa\,}
			&A\ar[rr]^f\ar[dr]_{f'}\ar@/^2pc/[rrrr]^0\rrtwocell\omit\omit{_<4>\mu}
				\ar@{}[dll]|(0.32){\kappa'\!\!\!\dir{=>}}
			&{}\rrtwocell\omit\omit{^<-2.7>\varphi\,}
			&B\ar[rr]^y\ar@{}[drr]|(0.32){\dir{=>}\!\kappa_y}
			&{}
			&Y
			\\ &&&Ky\ar[ur]_{k_y}\ar@/_0.78pc/[urrr]_0 &&&{}
		}\end{xym}
	\end{coro}
	
	Here is a last lemma concerning the relative kernel and pullback; it will be
	used to prove the $3\times 3$ lemma (Proposition \ref{lemmtrxtr}).
	
	\begin{lemm}\label{lemmcinq}
		Let be the following diagram in $\C$, where $\kappa$ is compatible with $\varphi$
		and $\kappa'a\circ g'\mu = \kappa\circ\nu f$.
		\begin{xym}\xymatrix@C=40pt@R=15pt{
			A\ar[dd]_a\ar[r]^f\ddrtwocell\omit\omit{\mu}
				\drruppertwocell^0<15>{^<-5.6>{\kappa}}
			&B\ar[dr]^g\ar[dd]_b\drruppertwocell^0{^<-1.3>{\varphi}}
				\drtwocell\omit\omit{^<3.5>\nu}
			\\ &&C\ar[r]^y
			&Y
			\\ A'\ar[r]_{f'}\urrlowertwocell_0<-15>{_<5.6>{\;\,\kappa'}}
			&B'\ar[ur]_{g'}
		}\end{xym}
		If $(f',\kappa')=\Ker g'$, then the following properties are equivalent:
		\begin{enumerate}
			\item $\mu$ is a pullback relative to $\varphi\circ y\nu$
				and $y\kappa'$;
			\item $(f,\kappa)=\Ker(g,\varphi)$.
		\end{enumerate}
	\end{lemm}
	
		\begin{proof}
			Let us consider the following diagram. We construct the kernel
			of $y$.  There is an induced arrow $\hat{g}\col B\ra Ky$
			and a 2-arrow $\hat{\varphi}\col g\Ra k_y\hat{g}$ such that
			$\kappa_y\hat{g}\circ y\hat{\varphi}=\varphi$. We set
			$\hat{\nu}\eqdef \hat{\varphi}\circ\nu$.  There is also an induced
			2-arrow $\hat{\kappa}\col\hat{g}f\Ra 0$ such that
			$k_y\hat{\kappa}\circ\hat{\nu}f=\kappa' a\circ g'\mu$.
			Then, by Lemma \ref{lemmquatre}, since $(f',\kappa')=\Ker g'$
			and $k_y$ is $\kappa_y$-fully faithful (by the universal property
			of the kernel), $\mu$ is a pullback relative to $\kappa_y\hat{g}\circ
			y\hat{\nu}(=\varphi\circ y\nu)$ and $y\kappa'$ if and only if
			$(f,\hat{\kappa})=\Ker\hat{g}$.  By the previous corollary, this last
			property is equivalent to $(f,\kappa)=\Ker(g,\varphi)$.
			\qedhere
			\begin{xym}\xymatrix@=40pt{
				A\ar[r]^{f}\ar[d]_a\rruppertwocell<9>^0{^<-2.7>\hat{\kappa}\,}
					\drtwocell\omit\omit{_{\mu}}
				&B\ar[r]^{\hat{g}}\ar[d]_{b}\drtwocell\omit\omit{^{\hat{\nu}}}
				&Ky\ar[d]^{k_y}\dduppertwocell<9>^0{^<-2.7>\kappa_y}
				\\ A'\ar[r]_{f'}\rrlowertwocell<-9>_0{_<2.7>\;\kappa'}
				&B'\ar[r]_{g'}
				&C\ar[d]_{y}
				\\ &&Y
			}\end{xym}
		\end{proof}

\subsection{Restricted kernels lemma}

	To complete this collection of lemmas always true in a $\Gpdp$-category
	with kernels, here is the restricted kernels lemma and the relative kernels lemma.
	Let us first introduce the following terminology: we say that a
	diagram of the shape of diagram \ref{diagtyptroitroi} \emph{commutes} if the following
	equations hold:
	\begin{align}\stepcounter{eqnum}
		f_3\alpha_1\circ\varphi_2 a_1\circ b_2\varphi_1 &= \beta_1f_1; \\ \stepcounter{eqnum}
		g_3\beta_1\circ\psi_2 b_1\circ c_2\psi_1 &= \gamma_1g_1; \\ \stepcounter{eqnum}
		\eta_2 a_1\circ g_2\varphi_1\circ\psi_1 f_1 &=c_1\eta_1; \\ \stepcounter{eqnum}
		\eta_3 a_2\circ g_3\varphi_2\circ\psi_2 f_2 &=c_2\eta_2.
	\end{align}
	\begin{xym}\label{diagtyptroitroi}\xymatrix@=40pt{
			A_1\ar[r]^{f_1}\ar[d]_{a_1}\rruppertwocell<9>^0{^<-2.7>\eta_1\;}
				\drtwocell\omit\omit{_{\;\,\varphi_1}}
				\ddlowertwocell<-9>_0{_<2.7>\alpha_1}
			&B_1\ar[r]^{g_1}\ar[d]^{b_1}\ar@/^1.89pc/[dd]^(0.44)0
				\drtwocell\omit\omit{_{\;\,\psi_1}}
			&C_1\ar[d]^{c_1}\dduppertwocell<9>^0{^<-2.7>\gamma_1}
			\\ A_2\ar[r]_{f_2}\ar[d]_{a_2}\ar@/_1.87pc/[rr]_(0.45)0
				\drtwocell\omit\omit{_{\;\,\varphi_2}}
			&B_2\ar[r]_(0.62){g_2}\ar[d]^(0.62){b_2}\drtwocell\omit\omit{_{\;\,\psi_2}}
				\ar@{}[d]_(0.25){\eta_2\,\dir{=>}}
				\ar@{}[r]^(0.25){\txt{$\overset{\beta_1\;\;}{\dir{=>}}$}}
			&C_2\ar[d]^{c_2}
			\\ A_3\ar[r]_{f_3}\rrlowertwocell<-9>_0{_<2.7>\;\,\eta_3}
			&B_3\ar[r]_{g_3}
			&C_3
		}\end{xym}
	
	\begin{pon}[Restricted kernels lemma]\label{lemmnoyrestr}%
	\index{lemma!kernels!restricted}\index{kernels lemma!restricted}
		Let us consider the situation of diagram \ref{diagtyptroitroi}.  Let us assume
		that the diagram commutes and that the following conditions hold:
		\begin{enumerate}
			\item $(a_1,\alpha_1)=\Ker a_2$;
			\item $(b_1,\beta_1)=\Ker b_2$;
			\item $c_1$ is $\gamma_1$-fully 0-faithful;
			\item $(f_2,\eta_2)=\Ker g_2$;
			\item $f_3$ is $\eta_3$-fully 0-faithful.
		\end{enumerate}
		Then $(f_1,\eta_1)=\Ker g_1$ and $\varphi_1$ is a pullback
		relative to $\gamma_1g_1\circ c_2\psi_1^{-1}$ and $c_2\eta_2$.
	\end{pon}
	
		\begin{proof}
			By Lemma \ref{lemmquatre}, conditions 1, 2 and 5 imply
			that $\varphi_1$ is a pullback relative to 
			$g_3\beta_1$ and $\eta_3 a_2\circ g_3\varphi_2$. Thanks to the isomorphism
			$\psi_2$, $\varphi_1$ is also a pullback relative to 
			$\gamma_1g_1\circ c_2\psi_1^{-1}$ and $c_2\eta_2$ and, by Lemma
			\ref{lemmquatre}, conditions
			3 and 4 imply that $(f_1,\eta_1)=\Ker g_1$.
		\end{proof}
		
	\begin{coro}[Relative kernels lemma]\label{lemmnoyrel}%
	\index{lemma!relative kernels}\index{relative kernels lemma}
		Let be the following diagram, where the lower two thirds and the
		upper two thirds commute.
		
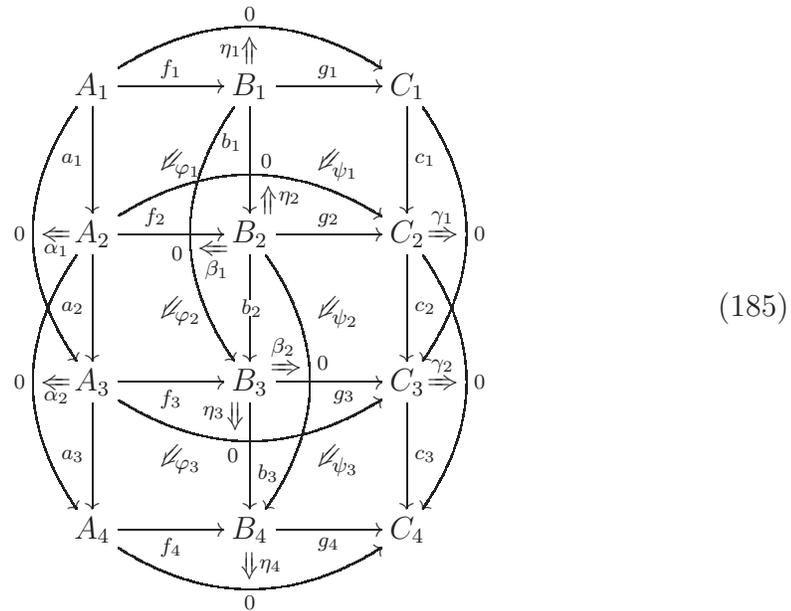
\begin{xym}\xymatrix@=40pt{
			A_1\ar[r]^{f_1}\ar[d]_{a_1}\rruppertwocell<9>^0{^<-2.7>\eta_1\;}
				\drtwocell\omit\omit{_{\;\,\varphi_1}}
				\ddlowertwocell<-9>_0{_<2.7>\alpha_1}
			&B_1\ar[r]^{g_1}\ar[d]_(0.38){b_1}\ar@/_1.89pc/[dd]_(0.56)0
				\drtwocell\omit\omit{_{\;\,\psi_1}}
			&C_1\ar[d]^{c_1}\dduppertwocell<9>^0{^<-2.7>\gamma_1}
			\\ A_2\ar[r]^(0.4){f_2}\ar[d]_{a_2}\ar@/^1.87pc/[rr]^(0.55)0
				\drtwocell\omit\omit{_{\;\,\varphi_2}}
				\ddlowertwocell<-9>_0{_<2.7>\alpha_2}
			&B_2\ar[r]^{g_2}\ar[d]|{\,b_2}\drtwocell\omit\omit{_{\;\,\psi_2}}
				\ar@{}[u]_(0.25){\,\dir{=>}\,\eta_2}\ar@/^1.89pc/[dd]^(0.44)0
				\ar@{}[l]^(0.25){\txt{$\underset{\;\;\beta_1}{\dir{=>}}$}}
			&C_2\ar[d]^{c_2}\dduppertwocell<9>^0{^<-2.7>\gamma_2}
			\\ A_3\ar[r]_{f_3}\ar[d]_{a_3}\ar@/_1.87pc/[rr]_(0.45)0
				\drtwocell\omit\omit{_{\;\,\varphi_3}}
			&B_3\ar[r]_(0.6){g_3}\ar[d]^(0.62){b_3}\drtwocell\omit\omit{_{\;\,\psi_3}}
				\ar@{}[d]_(0.25){\eta_3\,\dir{=>}}
				\ar@{}[r]^(0.25){\txt{$\overset{\beta_2\;\;}{\dir{=>}}$}}
			&C_3\ar[d]^{c_3}
			\\ A_4\ar[r]_{f_4}\rrlowertwocell<-9>_0{_<2.7>\;\,\eta_4}
			&B_4\ar[r]_{g_4}
			&C_4
		}\end{xym}
		Let us assume that the following conditions hold:
		\begin{enumerate}
			\item $(a_1,\alpha_1)=\Ker (a_2,\alpha_2)$;
			\item $(b_1,\beta_1)=\Ker (b_2,\beta_2)$;
			\item $(c_1,\gamma_1)=\Ker (c_2,\gamma_2)$;
			\item $(f_2,\eta_2)=\Ker g_2$;
			\item $(f_3,\eta_3)=\Ker g_3$;
			\item $(f_4,\eta_4)=\Ker g_4$.
		\end{enumerate}
		Then $(f_1,\eta_1)=\Ker g_1$.
	\end{coro}
	
		\begin{proof}
			We decompose the diagram in two parts to wich we apply
			the kernels lemma.  We factor $a_2$ by taking
			the kernel of $a_3$, which gives the following diagram, where
			$a''\bar{\alpha}_1\circ\alpha a_1=\alpha_1$
			and $\bar{\alpha}_2a'\circ a_3\alpha=\alpha_2$.
			\begin{xym}\label{diagsittesuu}\xymatrix@C=20pt@R=40pt{
				A_1\ar[rr]^-{a_1}\ar@/_0.78pc/[drrr]_0\ar@/^2pc/[rrrr]^0
				&{}\rrtwocell\omit\omit{^<-2.7>\alpha_1\;\,}
				&A_2\ar[rr]^-{a_2}\ar[dr]_{a'}\ar@/^2pc/[rrrr]^0
					\rrtwocell\omit\omit{_<4>\alpha}
					\ar@{}[dll]|(0.32){\bar{\alpha}_1\!\dir{=>}}
				&{}\rrtwocell\omit\omit{^<-2.7>\alpha_2\;\;}
				&A_3\ar[rr]^-{a_3}\ar@{}[drr]|(0.32){\dir{=>}\!\!\bar{\alpha}_2}
				&{}
				&A_4
				\\ &&&Ka_3\ar[ur]_{a''}\ar@/_0.78pc/[urrr]_0 &&&{}
			}\end{xym}
			We do the same for the other two columns, which gives respectively
			$b'$, $b''$, $\bar{\beta}_1$, $\bar{\beta}_2$, $\beta$, and
			$c'$, $c''$, $\bar{\gamma}_1$, $\bar{\gamma}_2$, $\gamma$,
			satisfying the corresponding conditions.  By the universal property
			of kernels, these constructions induce $f$, $\varphi'$
			and $\varphi''$ (see the following diagrams) which make commute the left
			half of diagram \ref{diagsittesy} and such that
			\begin{eqn}
				\varphi''a'\circ b''\varphi'\circ\beta f_2
				= f_3\alpha\circ\varphi_2,
			\end{eqn}
			as well as $g$, $\psi'$ and $\psi''$, which make commute the right half
			of this diagram, and such that
			\begin{eqn}
				\psi''b'\circ c''\psi'\circ\gamma g_2
				= g_3\beta\circ\psi_2.
			\end{eqn}
			Finally, the universal property of the kernel of $c_3$ induces a 2-arrow
			$\eta\col gf\Ra 0$ such that $\eta_3 a''\circ g_3\varphi''\circ\psi'' f
			= c''\eta$.
			Therefore the restricted kernels lemma applies to the following diagram, 
			and $(f,\eta)=\Ker g$.
			\begin{xym}\label{diagsittesy}\xymatrix@=40pt{
				Ka_3\ar[r]^{f}\ar[d]_{a''}\rruppertwocell<9>^0{^<-2.7>\eta\;}
					\drtwocell\omit\omit{_{\;\,\varphi''}}
					\ddlowertwocell<-9>_0{_<2.7>\bar{\alpha}_2}
				&Kb_3\ar[r]^{g}\ar[d]^{b''}\ar@/^1.89pc/[dd]^(0.44)0
					\drtwocell\omit\omit{_{\;\,\psi''}}
				&Kc_3\ar[d]^{c''}\dduppertwocell<9>^0{^<-2.7>\bar{\gamma}_2}
				\\ A_3\ar[r]_{f_3}\ar[d]_{a_3}\ar@/_1.87pc/[rr]_(0.45)0
					\drtwocell\omit\omit{_{\;\,\varphi_3}}
				&B_3\ar[r]_(0.62){g_3}\ar[d]^(0.62){b_3}
					\drtwocell\omit\omit{_{\;\,\psi_3}}
					\ar@{}[d]_(0.25){\eta_3\,\dir{=>}}
					\ar@{}[r]^(0.22){\txt{$\overset{\bar{\beta}_2\;\;}{\dir{=>}}$}}
				&C_3\ar[d]^{c_3}
				\\ A_4\ar[r]_{f_4}\rrlowertwocell<-9>_0{_<2.7>\;\,\eta_4}
				&B_4\ar[r]_{g_4}
				&C_4
			}\end{xym}
			Finally, by Corollary \ref{coronoyrelnoy} applied to diagram
			\ref{diagsittesuu} and to the corresponding diagrams for the other two
			columns, $(a_1,\bar{\alpha_1})=\Ker a'$, $(b_1,\bar{\beta_1})=\Ker b'$
			and $(c_1,\bar{\gamma_1})=\Ker c'$.  So the restricted kernels lemma applies
			to the following diagram (which commutes, thanks to the faithfulness of $b''$ and $c''$)
			and $(f_1,\eta_1)=\Ker g_1$.\qedhere
			\begin{xym}\xymatrix@=40pt{
				A_1\ar[r]^{f_1}\ar[d]_{a_1}\rruppertwocell<9>^0{^<-2.7>\eta_1\;}
					\drtwocell\omit\omit{_{\;\,\varphi_1}}
					\ddlowertwocell<-9>_0{_<2.7>\bar{\alpha}_1}
				&B_1\ar[r]^{g_1}\ar[d]_(0.38){b_1}\ar@/_1.89pc/[dd]_(0.56)0
					\drtwocell\omit\omit{_{\;\,\psi_1}}
				&C_1\ar[d]^{c_1}\dduppertwocell<9>^0{^<-2.7>\bar{\gamma}_1}
				\\ A_2\ar[r]^(0.4){f_2}\ar[d]_{a'}\ar@/^1.87pc/[rr]^(0.55)0
					\drtwocell\omit\omit{_{\;\,\varphi'}}
				&B_2\ar[r]^{g_2}\ar[d]_{b'}\drtwocell\omit\omit{_{\;\,\psi'}}
					\ar@{}[u]_(0.25){\,\dir{=>}\,\eta_2}
					\ar@{}[l]^(0.22){\txt{$\underset{\;\;\bar{\beta}_1}{\dir{=>}}$}}
				&C_2\ar[d]^{c'}
				\\ Ka_3\ar[r]_{f}\rrlowertwocell<-9>_0{_<2.7>\;\,\eta}
				&Kb_3\ar[r]_{g}
				&Kc_3
			}\end{xym}
		\end{proof}

\section{Exact sequences and Puppe-exact $\Gpdp$-categories}

\subsection{Exact sequences}

	The notion of exact sequence for symmetric 2-groups appears in
	\cite{Vitale2002a}, in the form of conditions 3 and 4 of Proposition
	\ref{caracdexbis}.  Marco Grandis also defined a notion of (homotopical) exactness
	in a more general context \cite{Grandis2001b}.
	
	In order that the two dual definitions of exact sequence be equivalent in dimension 2,
	we must assume that in $\C$ the adjunction $\Coker\adj\Ker$ is idempotent
	(as Marco Grandis \cite{Grandis2001b} shows), i.e.\ that $\C$ be
	$\Ker$-idempotent (Marco Grandis uses the term \emph{semistable} for a $\Ker$-idempotent
	h-category). As we noticed at the beginning of Subsection
	\ref{soussectdualsysnoyquot}, this adjunction is always idempotent in dimension 1,
	but this is not the case any more in dimension 2.
	
	\begin{pon}\label{defsequencex}
		Let us assume that $\C$ is $\Ker$-idempotent. Let us consider the following diagram,
		where we have constructed $(k,\kappa)=\Ker b$ and $(q,\zeta)=\Coker a$;
		this induces arrows $a'$ and $b'$ and 2-arrows $\mu$ and $\nu$ such that
		$\kappa a'\circ b\mu=\alpha = b'\zeta\circ\nu a$.
		\begin{xym}\label{diagsuitdex}\xymatrix@R=40pt@C=20pt{
			A\ar[rr]^a\ar[dr]_-{a'}\ar@/^2pc/[rrrr]^0\rrtwocell\omit\omit{_<4>\mu}
			& {}\rruppertwocell\omit{^<-2.7>\alpha}
			& B\ar[rr]^b\ar[dr]_-{q}\rrtwocell\omit\omit{_<4>\nu}
			& {}
			& C
			\\ &{Kb}\ar[ur]_-{k}
			&&{Qa}\ar[ur]_-{b'}
		}\end{xym}
		The following conditions are equivalent; when they hold, we say that
		the sequence $(a,\alpha,b)$ is \emph{exact (at $B$)}.\index{exact!sequence}%
		\index{sequence!exact}
		\begin{enumerate}
			\item There exists $\omega\col qk\Ra 0$ such that $(k,\omega)=\Ker q$,
				$\omega a'\circ q\mu = \zeta$ and $b'\omega\circ \nu k=\kappa$.
			\item There exists $\omega\col qk\Ra 0$ such that $(q,\omega)=\Coker k$,
				$\omega a'\circ q\mu = \zeta$ and $b'\omega\circ \nu k=\kappa$.
		\end{enumerate}
	\end{pon}

		\begin{proof}
			Condition 1 implies condition 2 because, if $(k,\omega)=\Ker q$, as
			$q$ is a cokernel, $(q,\omega)=\Coker k$, by $\Ker$-idempotence.
			Dually, condition 2 implies condition 1.
		\end{proof}
	
	The first condition means that the kernel of $b$ is the (faithful) image of $a$
	(the kernel of the cokernel of $a$). The second condition means that the cokernel of $a$ is
	the full image of $b$ (the cokernel of the kernel of $b$).
	
	The special case where $B$ is $0$ gives an exactness notion for loops.
	
	\begin{df}\index{loop!exact}\index{exact!loop}
		We say that a loop $\pi\col 0\Ra 0\col A\ra B$ is \emph{exact} if the sequence
		\begin{xym}\xymatrix@=40pt{
			A\ar[r]\rruppertwocell<9>^0{^<-2.7>\pi}
			&0\ar[r] &B
		}\end{xym}
		is exact.
	\end{df}

	As in dimension 1, if the first arrow of the sequence is the kernel of the second,
	then the sequence is exact.
	
	\begin{pon}\label{qquessequencex}
		Let us assume that $\C$ is $\Ker$-idempotent. In the situation of diagram
		\ref{diagsuitdex}, if $(a,\alpha)=\Ker b$, then the sequence $(a,\alpha,b)$ is exact.
		In particular, the loop $\omega_A\col 0\Ra 0\col \Omega A\ra A$ is exact.
	\end{pon}
	
		\begin{proof}
			In this case, $a'$ is an equivalence.  Now, by $\Ker$-idempotence,
			$(a,\zeta)=\Ker q$. So there is a 2-arrow $\omega$ satisfying the required
			conditions such that $(k,\omega)=\Ker q$.
		\end{proof}

	On the other hand, if we start with an arrow $f$, we don't get an exact sequence
	$0\ra \Ker f\overset{k}\ra A\overset{f}\ra B\overset{q}\ra \Coker f\ra 0$ any more.
	This sequence is exact at $A$ and $B$ but not at $\Ker f$ nor at $\Coker f$,
	because the kernel of the kernel of $f$ is not in general $0$ (since $\Ker k$ is $0$ if and only if $k$ is fully 0-faithful).

	\begin{lemm}\label{kerkeromegcodom}
		Let be $f\col A\ra B$.  If $(K,k,\kappa)$ is a kernel of $f$, then
		there exist $d\col\Omega B\ra K$ and $\delta\col kd\Ra 0$ such that 
		\begin{eqn}\label{eqlajusosu}
			\omega_B = f\delta\circ \kappa^{-1}d.
		\end{eqn}
		Moreover, for all $d$, $\delta$ satisfying the previous equation,
		$(\Omega B,d,\delta)$ is a kernel of $k$.
		\begin{xym}\xymatrix@=40pt{
			{\Omega B}\ar[r]^{d}\rrlowertwocell<-9>_0{_<2.7>\,\delta}
			&{K}\ar[r]^{k}\rruppertwocell<9>^0{^<-2.7>\kappa\;}
			&A\ar[r]_f
			&B
		}\end{xym}
	\end{lemm}
	
		\begin{proof}
			First, the universal property of $(k,\kappa)=\Ker f$ induces an arrow
			$d$ and a 2-arrow $\delta$ satisfying \ref{eqlajusosu}.
			Then, by Lemma \ref{lemmun} applied to the following diagram,
			$\delta$ is a pullback, since $(0^B,\omega_B^{-1})=\Ker 0_B$ and
			$(k,\kappa)=\Ker f$.\qedhere
			\begin{xym}\xymatrix@C=40pt@R=15pt{
				{\Omega B}\ar@{-->}[dd]_d\ar[r]^0\ddrtwocell\omit\omit{^\delta}
					\drruppertwocell^0<15>{^<-5.6>{\omega_B^{-1}\;\;\;\;\;}}
				&0\ar[dr]^0\ar[dd]^0
				\\ &&B
				\\ K\ar[r]_{k}\urrlowertwocell_0<-15>{_<5.6>{\;\,\kappa}}
				&A\ar[ur]_{f}
			}\end{xym}
		\end{proof}
	
	The 2-dimensionnel equivalent of the 1-dimensional exact sequence
	$0\ra \Ker f\overset{k}\ra A\overset{f}\ra B\overset{q}\ra \Coker f\ra 0$ is the sequence of the following proposition.
	It is a truncated version (because we always have $\Omega\Omega A\simeq 0$ and
	$\Sigma\Sigma A\simeq 0$) of the Puppe exact sequence
	in homotopy theory.  It was studied in the context of $\Gpdp$-categories
	with a weaker notion of kernel by Gabriel and Zisman
	\cite[Chapter~5]{Gabriel1967a} and, in a more general context, for the standard
	homotopy kernel, by Marco Grandis 
	\cite{Grandis1992a,Grandis1994a,Grandis1997a}. The Puppe sequence for
	symmetric 2-groups is described by Dominique Bourn and Enrico
	Vitale \cite[Proposition 2.6]{Bourn2002a}.
	Let us recall that we construct the pip of $f$ as 
	$\Pip f\eqdef\Omega\Ker f$ and, dually, the copip of $f$ as
	$\Copip f\eqdef\Sigma \Coker f$.  In the following proposition, we abbreviate
	$\Pip f$ by $Pf$, $\Ker f$ by $Kf$, $\Coker f$ by $Qf$ and $\Copip f$ by $Rf$,
	and we denote by $\mu_f$ the composite $\zeta k\circ q\kappa^{-1}$.

	\begin{pon}[Puppe long exact sequence]\label{sequencePuppe}%
	\index{long exact sequence!Puppe}\index{Puppe long exact sequence}
		Let $f\col A\ra B$ be in $\C$, that we assume to be $\Ker$-idempotent.
		Then there exist $d$, $\delta$, $\varepsilon$
		and $d'$, $\delta'$, $\varepsilon'$ such that the following sequence is exact.
		Each arrow to the left of $f$,
		with the adjacent 2-arrow to $0$, is the kernel of the following arrow;
		dually, each arrow to the right of $f$, with the adjacent 2-arrow to $0$,
		is the cokernel of the previous arrow.
		Moreover, the following equalities hold:
		\begin{enumerate}
			\item $\varepsilon\Omega k\circ d\Omega\kappa^{-1}=\omega_{Kf}^{-1}$;
			\item $\delta\Omega f\circ k\varepsilon^{-1}=\omega_A$;
			\item $\kappa d\circ f\delta^{-1} = \omega_B^{-1}$;
			\item $\zeta k\circ q\kappa^{-1} =: \mu_f$;
			\item $\delta' f\circ of \zeta^{-1}=\sigma_A^{-1}$;
			\item $\varepsilon' q\circ (\Sigma f)\delta'^{-1}=\sigma_B$;
			\item $(\Sigma\zeta) of \circ(\Sigma q)\varepsilon'^{-1}=\sigma_{Qf}^{-1}$.
		\end{enumerate}
		\begin{xym}\label{gabzis}\xymatrix@=30pt{
			0\ar[r]^-0\rrlowertwocell<-9>_0{_<2.7>\;\,1_0}
			&Pf\ar[r]^-{\Omega k}\rruppertwocell<9>^0{^<-2.7>\Omega\kappa\;\;\,}
			&\Omega A\ar[r]^-{\Omega f}\rrlowertwocell<-9>_0{_<2.7>\,\varepsilon}
			&\Omega B\ar[r]^-d\rruppertwocell<9>^0{^<-2.7>\delta}
			&Kf\ar[r]^-k\rrlowertwocell<-9>_0{_<2.7>\,\kappa}
			&A\ar[r]^-f
			&{B\cdots}
		}\end{xym}
		\begin{xym}\label{gabzistwo}\xymatrix@=30pt{
			{\cdots A}\ar[r]^-f\rruppertwocell<9>^0{^<-2.7>\zeta}
			&B\ar[r]^-q\rrlowertwocell<-9>_0{_<2.7>\;\,\delta'}
			&Qf\ar[r]^-{d'}\rruppertwocell<9>^0{^<-2.7>\varepsilon'\,}
			&\Sigma A\ar[r]^-{\Sigma f}\rrlowertwocell<-9>_0{_<2.7>\;\;\Sigma\zeta}
			&\Sigma B\ar[r]^-{\Sigma q}\rruppertwocell<9>^0{^<-2.7>1_0\;}
			&Rf\ar[r]^-0
			&0
		}\end{xym}
	\end{pon}
	
		\begin{proof}
			Let us deal with the left half; the argument for the right half is dual.
			We will prove that each morphism of the sequence is the kernel of the following morphism;
			the exactness will follow by Proposition \ref{qquessequencex}.
			
			First, $(k,\kappa)=\Ker f$, by definition.  
			Then, by the previous lemma, there exist $d$ and $\delta$ such that
			$(d,\delta)=\Ker k$ and $\omega_B =  f\delta\circ \kappa^{-1}d$.  
			
			Next, $\omega_A^{-1}\circ\delta\Omega f\col kd\Omega f\Ra 0$ is compatible
			with $\kappa$, because $\omega_B\Omega f=f\omega_A$.  So, by the universal
			property of the kernel of $f$, there exists a unique $\varepsilon\col d\Omega f\Ra 0$
			such that
			\begin{eqn}
				\delta\Omega f\circ k\varepsilon^{-1}=\omega_A.
			\end{eqn}
			Then $(\Omega A,\Omega f,\varepsilon)=\Ker d$, by the previous lemma.
			
			Next, $(Pf,\Omega k,\Omega \kappa)=\Ker(\Omega f)$
			because $\Omega$ preserves limits (being right adjoint to $\Sigma$).
			
			Then, $\Omega k$ is fully faithful because the kernel
			of an arrow with discrete co\-do\-main is fully faithful 
			(by Lemma \ref{noycodomdisplfid});
			therefore $0$ is a kernel of $\Omega f$.
			
			To end with, we have $k(d\Omega\kappa\circ\varepsilon^{-1}\Omega k)
			= \omega_A\Omega k = k\omega_{Kf}$.  Since $k$ is faithful, we have 
			$d\Omega\kappa\circ\varepsilon^{-1}\Omega k
			= \omega_{Kf}$.
		\end{proof}
		
	The loops we get by composing the neighbouring 2-arrows in this exact sequence
	are all exact, except perhaps $\mu_f$, which is exact in a good 2-Puppe-exact
	$\Gpdp$-category (Proposition \ref{mufestexact}).
	So, in a good 2-Puppe-exact $\Gpdp$-category, this sequence is
	\emph{perfectly exact}\index{sequence!perfectly exact}%
	\index{perfectly exact sequence} (it is exact at each point and each 2-arrow
	$0\Ra 0$ is exact).

\subsection{Relative exact sequences}

	The notion of relative exact sequence has been introduced 
	by Aurora Del Río, Juan Martínez-Moreno and Enrico Vitale in \cite{Rio2005a}.
	
	\begin{pon}\label{caracrelex}
		Let us assume that $\C$ is $\Ker$-idempotent. Let us consider the
		following diagram, where $\varphi$ and $\alpha$ are compatible, as well as
		$\alpha$ and $\psi$, and where we have constructed $(k,\kappa)=\Ker(b,\psi)$
		and $(q,\zeta)=\Coker(a,\varphi)$.  This induces arrows $a'$ and $b'$
		and 2-arrows $\mu$, $\nu$, $\varphi'$, $\psi'$ such that
		$\kappa a'\circ b\mu=\alpha = b'\zeta\circ\nu a$, $k\varphi'\circ\mu x=\varphi$
		and $\psi'q\circ y\nu=\psi$.
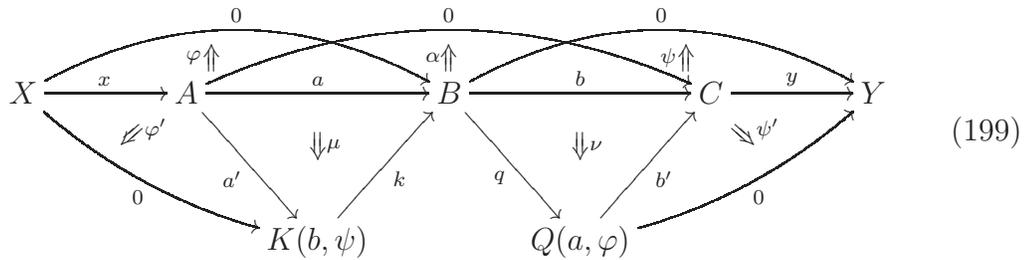
		\begin{xym}\label{diagsuitrelex}\xymatrix@R=40pt@C=20pt{
			X\ar[rr]^x\ar@/^2pc/[rrrr]^0\ar@/_0.78pc/[drrr]_-0
			&{}\rrtwocell\omit\omit{^<-2.7>\varphi\,}
			&A\ar[rr]^a\ar[dr]_-{a'}\ar@/^2pc/[rrrr]^0
				\ar@{}[dll]|(0.28){\dir{=>}\,\varphi'}\rrtwocell\omit\omit{_<4>\mu}
			& {}\rruppertwocell\omit{^<-2.7>\alpha}
			& B\ar[rr]^b\ar[dr]_-{q}\rrtwocell\omit\omit{_<4>\nu}
				\ar@/^2pc/[rrrr]^0
			& {}\rrtwocell\omit\omit{^<-2.7>\psi\,}
			& C\ar[rr]^y\ar@{}[drr]|(0.28){\dir{=>}\!\psi'}
			& {}
			& Y
			\\ {} &&&{K(b,\psi)}\ar[ur]_-{k}
			&&{Q(a,\varphi)}\ar[ur]_-{b'}\ar@/_0.78pc/[urrr]_-0 &&&{}
		}\end{xym}
		The following conditions are equivalent;
		when they hold, we say that the
		 sequence $(x,\varphi,a,\alpha,b,\psi,y)$ is \emph{relative exact at $B$}.%
		\index{sequence!relative exact}\index{relative exact sequence}
		\begin{enumerate}
			\item There exists $\omega\col qk\Ra 0$ such that $(k,\omega)=\Ker q$,
				$\omega a'\circ q\mu = \zeta$ and $b'\omega\circ \nu k=\kappa$.
			\item There exists $\omega\col qk\Ra 0$ such that $(q,\omega)=\Coker k$,
				$\omega a'\circ q\mu = \zeta$ and $b'\omega\circ \nu k=\kappa$.
		\end{enumerate}
	\end{pon}
	
		\begin{proof}
			The proof is the same as the proof of Proposition \ref{defsequencex}
			because, by Corollary \ref{coronoyrelnoy}, $k$
			is a kernel (thus is the kernel of its cokernel, since $\C$
			is $\Ker$-idempotent), and $q$
			is a cokernel (thus is the cokernel of its kernel).
		\end{proof}
	
	A related notion is defined by Hans-Joachim Baues and Mamuka Jibladze
	in \cite{Baues2006b}: they fix a sub-$\Gpd$-category $\dcat{P}$ of $\C$
	(whose objects are to be thought of as projectives), and say that the sequence of the
	upper row of diagram \ref{diagsuitrelex}
	is \emph{$\dcat{P}$-exact} if, for all $P\col\dcat{P}$,
	$p\col P\ra B$ and $\pi\col bp\Ra 0$ compatible with $\psi$, there exist
	$p'\col P\ra A$, $\pi'\col p\Ra ap'$ such that $\alpha p'\circ b\pi' =\pi$.
		
	The notion of exactness is a special case of the notion of relative exactness (diagram \ref{diagsuitdex}
	is the special case of diagram \ref{diagsuitrelex} where $X$ and $Y$
	are $0$).
		
	\begin{pon} Let us consider the following diagram.
		\begin{xym}\label{sequencesuiv}\xymatrix@=40pt{
			0\ar[r]_0\rruppertwocell<9>^0{^<-2.7>{1_0\;}}
			&A\ar[r]^a\rrlowertwocell<-9>_0{_<2.7>\alpha}
			&B\ar[r]^b\rruppertwocell<9>^0{^<-2.7>{1_0\;}}
			&C\ar[r]_0
			&0
		}\end{xym}
		The sequence $(a,\alpha,b)$ is exact at $B$ if and only if
		the sequence $(0,1_0,a,\alpha,b,1_0,0)$ is relative exact at $B$.
	\end{pon}

	On the other hand, long exact sequences are not in general relative exact,
	because there is no reason for the 2-arrows to be compatible.  For example, 
	the loops we get by composing the adjacent 2-arrows of the sequence
	\ref{gabzis}/\ref{gabzistwo} are not in general identities, as we noticed in Proposition \ref{sequencePuppe}.  In particular, the sequence
	$Kf\overset{k}\ra A\overset{f}\ra B\overset{q}\ra Qf$ cannot be relative
	exact, because $\kappa_f$ and $\zeta_f$ are not in general compatible (in a
	good 2-Puppe-exact $\Gpdp$-category, this is the case if and only if $f$
	is full (Proposition \ref{caracfullbondpex})).

	\begin{pon}\label{kerrelrelex}
		If $\C$ is $\Ker$-idempotent and if $(k,\kappa)=\Ker(f,\varphi)$, then
		the sequence of diagram \ref{diagkerkerzer} is relative exact at $0$,
		$K$ and $A$.
	\end{pon}
	
		\begin{proof}
			First, the sequence is clearly relative exact at $0$.
			Next, by Proposition \ref{kerkerzer}, $(0_K,1_0)=\Ker (k,\kappa)$;
			moreover $(1_K,1_0)=\Coker (0_K,1_0)$.  Hence $(0_K,1_0)=\Ker 1_K$,
			so the sequence is relative exact at $K$.
			
			Finally, by hypothesis, $(k,\kappa)=\Ker (f,\varphi)$, so $k$ is a kernel
			(by Corollary \ref{coronoyrelnoy}) and, by $\Ker$-idempotence,
			$k$ is the kernel of its cokernel.  So the sequence is relative exact at $A$.
		\end{proof}
	
	The importance of relative exact sequences comes from the fact that extensions
	(short exact sequences) are not in general exact sequences
	(for that $a$ should be 
	fully 0-faithful and $b$ fully 0-cofaithful), but are
	relative exact sequences.
	
	\begin{df}\index{extension}
		The sequence \ref{sequencesuiv} is called an \emph{extension} if 
		the following conditions hold:
		\begin{enumerate}
			\item $(a,\alpha)=\Ker b$;
			\item $(b,\alpha)=\Coker a$.
		\end{enumerate}
	\end{df}
	
	In an extension, $a$ is normal faithful and $b$ is normal cofaithful.
	
	\begin{pon}
		If diagram \ref{sequencesuiv} is an extension,
		then it is a relative exact sequence at each point.
	\end{pon}
	
		\begin{proof}
			By the previous proposition, the sequence is relative exact at $0$
			(on the left side), $A$ and $B$.  Dually, it is exact at $C$ and at $0$ (on the
			right side).
		\end{proof}
	
	In dimension 1, if we join two short exact sequences, we get a three-term
	exact sequence, the image of the middle morphism being the meeting point
	of the short exact sequences.  In dimension 2, if we join two extensions, we get
	a relative exact sequence and the meeting point is in general none of the images
	of the middle arrow.
	Therefore a projective resolution in dimension 2 should be a
	relative exact sequence rather than an exact sequence.
	
	\begin{pon}\label{twoextrelex}
		In the following diagram, if $(a,\alpha,b)$ and $(c,\gamma,d)$ are
		extensions, then the upper sequence is relative exact.
		
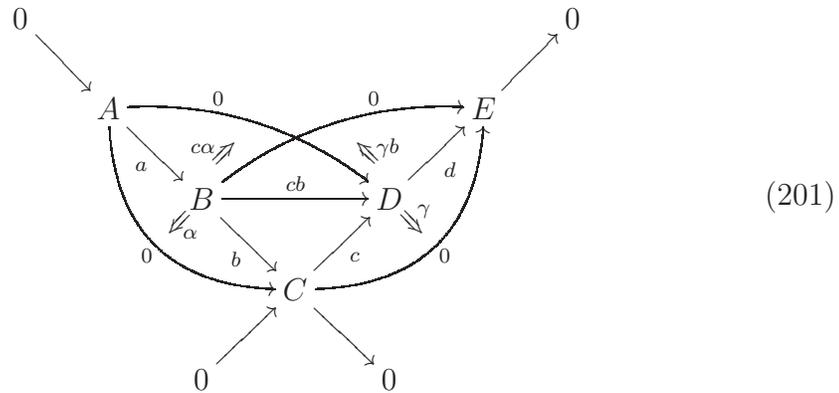
\begin{xym}\xymatrix@C=20pt@R=20pt{
			0\ar[dr]
			&&&&&&0
			\\ &A\ar[dr]_a\ddrrlowertwocell_0<-10>{\alpha}\ar@/^1pc/[drrr]^(0.35)0
			&&{}
			&&E\ar[ur]
			\\ &&B\ar[rr]^{cb}\ar[dr]_b\ar@/^1pc/[urrr]^(0.65)0
				\ar@{}[ur]|(0.55){c\alpha\dir{=>}\;\;\;\;\;\;\;\;\;\;\;\;}
			&&D\ar[ur]_d\ar@{}[ul]|(0.55){\;\;\;\;\;\;\;\;\;\;\;\;\dir{=>}\gamma b}
			&{}
			\\ {}
			&&&C\ar[ur]_c\uurrlowertwocell_0<-10>{\gamma}\ar[dr]
			&&&{}
			\\&&0\ar[ur]
			&&0
		}\end{xym}
	\end{pon}
	
		\begin{proof}
			First, it is easy to see that $c\alpha$ and
			$\gamma b$ are compatibles.  Since $(c,\gamma)=\Ker d$, $c$ is
			$\gamma$-fully faithful and so, by Corollary \ref{lemmsimplnoy},
			the fact that $(a,\alpha)=\Ker b$ implies that
			$(a,c\alpha)=\Ker(cb,\gamma b)$. Dually, $(d,\gamma b)=\Coker(cb,c\alpha)$.
			By applying Proposition \ref{kerrelrelex} and its dual, we get
			the relative exactness of the upper sequence.
		\end{proof}

\subsection{Homological $\Gpdp$-categories and homology}

	In the following proposition, the compatibility of the 2-arrows $\alpha$ and $\beta$ is
	necessary to define the comparison arrow $w$.  
	In particular, in dimension 2, we cannot in general define a comparison arrow
	between $\Coker(\Ker f)$ and $\Ker(\Coker f)$, because the 2-arrows
	$\kappa_f$ and $\zeta_f$ are not compatible (in a good 2-abelian $\Gpd$-category,
	they are compatible if and only if $f$ is full, by
	Proposition \ref{caracfullbondpex}).
	On the other hand, as we will see (diagrams \ref{diagcokerkerraccopep} and 
	\ref{diagcoracpepkercoker}), we can define
	a comparison arrow between $\Coker(\Ker f)$ and $\Root(\Copip f)$ and dually.
	
	\begin{lemm}\label{exenunfl}
		Let $\C$ be a $\Gpdp$-category with all the kernels and cokernels.
		Let be the situation of the upper row of the following diagram, where
		the 2-arrows are compatible.
		\begin{xym}\label{diagsuitmorphex}\xymatrix@=40pt{
			X\ar[r]_x\ar@/^2pc/[rr]^0
				\rrtwocell\omit\omit{^<-2.7>\varphi\,}
			&A\ar[r]^a\ar@/^2pc/[rr]^0\ar@/_0.78pc/[dr]_-0
				\drtwocell\omit\omit{_<-2>{\zeta}}
				\rruppertwocell\omit{^<-2.7>\alpha\,}
			& B\ar[r]^b\ar[d]_{q}
				\ar@/^2pc/[rr]^0\ar[dr]^(0.3){q'}
				\rrtwocell\omit\omit{^<-2.7>\beta\,}
			& C\ar[r]^c
				\ar@/^2pc/[rr]^0
				\rrtwocell\omit\omit{^<-2.7>\psi\,}
			& D\ar[r]_y
			& Y
			\\ &&{Q(a,\varphi)}\ar[r]_-w\ar[ur]^(0.7){k'}
			&{K(c,\psi)}\ar[u]_-{k}
				\ar@/_0.78pc/[ur]_-0\urtwocell\omit\omit{_<-2>{\,\kappa}}
		}\end{xym}
		We construct the relative cokernel $Q(a,\varphi)$ and the relative kernel $K(c,\psi)$.
		Then there exist arrows $q'$, $k'$, $w$, and 2-arrows $\mu$, $\nu$, $\mu'$,
		$\nu'$ such that
		\begin{xyml}\label{eqsusz}\begin{gathered}\xymatrix@=40pt{
			B\ar[r]^b\ar[d]_{q}\ar[dr]^(0.7){q'}
				\drtwocell\omit\omit{_<-3.5>{\mu}}\drtwocell\omit\omit{_<3>{\,\mu'}}
			&C
			\\ {Q(a,\varphi)}\ar[r]_w
			&{K(c,\psi)}\ar[u]_{k}
		}\end{gathered}\; = \; \begin{gathered}\xymatrix@=40pt{
			B\ar[r]^b\ar[d]_{q}
			&C
			\\ {Q(a,\varphi)}\ar[r]_w\ar[ur]_(0.7){k'}
				\urtwocell\omit\omit{_<-3.5>{\;\,\nu}}
				\urtwocell\omit\omit{_<3>{\;\;\,\nu'}}
			&{K(c,\psi)}\ar[u]_{k}
		}\end{gathered}\end{xyml}
		and such that
		\begin{eqn}\label{eqsusu}
			k'\zeta\circ \nu a = \alpha,
		\end{eqn}
		and
		\begin{eqn}\label{eqsusd}
			\kappa q'\circ c\mu = \beta.
		\end{eqn}

		We set $\zeta'\eqdef w\zeta\circ \mu'a\col q'a\Ra 0$ and
		$\kappa'\eqdef\kappa w\circ c\nu'\col ck'\Ra 0$.
		Then the following properties are equivalent:
		\begin{enumerate}
			\item $w$ is an equivalence;
			\item $(k',\kappa')=\Ker(c,\psi)$;
			\item $(q',\zeta')=\Coker(a,\varphi)$.
		\end{enumerate}
	\end{lemm}
	
		\begin{proof}
			Since $\alpha$ is compatible with $\varphi$, by the universal property
			of the relative cokernel, we have $k'\col Q(a,\varphi)\ra C$
			and $\nu\col b\Ra k'q$ satisfying equation \ref{eqsusu}.
			In the same way, since $\beta$ is compatible with $\psi$, by the universal
			property of the relative kernel, we have $q'\col B\ra K(c,\psi)$
			and $\mu\col b\Ra kq'$ satisfying equation \ref{eqsusd}.
			
			Then $\alpha\circ \mu^{-1}a$ is compatible with $\kappa$
			(thanks to the compatibility of $\alpha$ and $\beta$:
			$c\alpha = \beta a = \kappa q'a\circ c\mu a$) and, by the
			universal property of the relative kernel, there exists
			$\zeta'\col q'a\Ra 0$ such that
			\begin{eqn}
				k\zeta'\circ\mu a=\alpha.
			\end{eqn}
			
			Next, $\zeta'$ is compatible with $\varphi$ and, by the universal
			property of the relative cokernel, there exist $w\col Q(a,\varphi)
			\ra K(c,\psi)$ and $\mu'\col q'\Ra wq$ such that
			\begin{eqn}
				\zeta'=w\zeta\circ \mu'a.
			\end{eqn}
				
			Then the 2-arrow $k\mu'\circ\mu\circ\nu^{-1}\col k'q\Ra kw q$
			is compatible with $\zeta$ and so there exists
			$\nu'\col k'\Ra kw$ satisfying equation \ref{eqsusz}.
		\end{proof}	
	
	The subquotient axiom have been introduced by Marco Grandis in dimension 1
	 \cite{Grandis1992a}.
	
	\begin{pon}\label{propsemihom}
		Let $\C$ be a $\Ker$-idempotent $\Gpdp$-category.
		Then the following conditions are equivalent.
		\begin{enumerate}
			\item For all $B\overset{g}\rightarrowtail C\overset{f}
				\twoheadleftarrow A$ in $\C$, where $g$ is normal faithful
				and $f$ is normal cofaithful,
				the pullback $\pi\col fg'\Ra gf'$ of $f$ and $g$ exists,
				$g'$ is normal faithful, and $f'$ is normal cofaithful.
			\item
				\begin{enumerate}
					\item \emph{Subquotient axiom. }\index{subquotient axiom}%
					\index{axiom!subquotient}%
						In the situation of the following diagram (solid arrows),
						where we set $i\eqdef jk$ and $q\eqdef sr$, where
						$\alpha\col rjk\Ra 0$ and $\beta\col srj\Ra 0$ are compatible,
						and where $(i,\alpha,r)$ and $(j,\beta,q)$ are extensions,
						the comparison arrow $w\col Qk\ra Ks$ given by Lemma
						\ref{exenunfl} is an equivalence.
						\begin{xym}\label{diagsousquot}\xymatrix@=40pt{
							N\ar@{=}[r]\ar[d]_k\ddlowertwocell<-9>\omit{_<2.7>\alpha'}
								\ar@/_1.87pc/@{-->}[dd]_0
							&N\ar[d]^(0.35)i\ar@/^1.89pc/[dd]^(0.56)0
							\\ M\ar[r]^j\ar@{-->}[d]_{q'}\drtwocell\omit\omit{\mu}
								\ar@/^1.87pc/[rr]^(0.45)0
							&A\ar[r]^(0.6)q\ar[d]^r
								\ar@{}[u]^(0.25){\beta\,\dir{=>}}
								\ar@{}[r]_(0.25){\txt{$\underset{\alpha\;\;}{\dir{=>}}$}}
							&Q\ar@{=}[d]
							\\ Ks\ar@{-->}[r]_{k_s}\ar@/_1.87pc/@{-->}[rr]_0
								\rrlowertwocell<-9>\omit{_<2.7>\;\,\kappa_s}
							&R\ar[r]_s
							&Q
						}\end{xym}
					\item Normal cofaithful arrows are stable under
						composition.
				\end{enumerate}
		\end{enumerate}
	\end{pon}
	
		\begin{proof}
			{\it 1 $\Rightarrow$ 2(a).} In the situation of diagram
			\ref{diagsousquot}, we construct the kernel of $s$, which induces
			$q'\col M\ra Ks$ and $\mu\col rj\Ra k_sq'$
			such that $\kappa_sq'\circ s\mu=\beta$, as well as $\alpha'\col q'k\Ra 0$
			such that $k_s\alpha'\circ\mu k=\alpha$.
			
			By Lemma \ref{lemmun}, $\mu$ is a pullback.  So, by condition
			1, $q'$ is normal cofaithful. Now $(k,\alpha')=\Ker q'$,
			by Lemma \ref{lemmtwo}. So $(q',\alpha')=\Coker k$, which implies that
			$w$ is an equivalence, by the previous lemma.
			
			{\it 1 $\Rightarrow$ 2(b).} Let $A\overset{f}\twoheadrightarrow
			B\overset{g}\twoheadrightarrow C$ be normal cofaithful arrows.
			Let us consider the following diagram, where we have constructed
			the kernels of $g$ and $gf$.
			\begin{xym}\xymatrix@=40pt{
				K(gf)\rruppertwocell<9>^0{^<-2.7>\kappa_{gf}\;\;\,}
					\drtwocell\omit\omit{_{\tau}}\ar[r]_-{k_{gf}}\ar@{->>}[d]_{t}
				&A\ar[r]_-{gf}\ar@{->>}[d]_{f}
				&C\ar@{=}[d]
				\\ Kg\ar[r]_-{k_g}\rrlowertwocell<-9>_0{_<2.7>{\;\;\kappa_g}}
				&B\ar@{->>}[r]_{g}
				&C
			}\end{xym}
			By Lemma \ref{lemmun}, $\tau$ is a pullback.  So, by
			condition 1, $t$ is normal cofaithful.  Therefore, by
			Corollary \ref{lemmtrois}, $\tau$ is also a pushout.
			Finally, by the dual of Lemma \ref{lemmtwo},
			$(gf,\kappa_{gf})=\Coker k_{gf}$.
			
			{\it 2 $\Rightarrow$ 1.} From $f$ and $g$, we construct 
			$(Kf,i,\iota)=\Ker f$, $(q,\zeta)=\Coker g$ and $(g',\beta)=\Ker(qf)$. 
			By the universal property of the kernel of $qf$, there exist $k\col Kf\ra K(qf)$
			and $\theta\col g'k\Ra i$ such that $\beta k = q\alpha$, where
			$\alpha\eqdef\iota \circ f\theta$.	Then $(g'k,\alpha)=\Ker f$.
			(See the following diagram.)
			
			Next, by the universal property of $(g,\zeta)=\Ker q$
			(since $g$ is normal faithful), there exist $f'\col K(qf)\ra B$ and
			$\varphi\col fg'\Ra gf'$ such that $\zeta f'\circ q\varphi=\beta$.
			Finally, since $g$ is $\zeta$-fully faithful, there exists $\kappa\col f'k\Ra 0$
			such that $g\kappa\circ\varphi k = \alpha$.
			\begin{xym}\xymatrix@=40pt{
				Kf\ar@{=}[r]\ar[d]_k\ddlowertwocell<-13>_0{_<4.7>{\;\;\kappa}}
				&Kf\ar[d]_{g'k}\ar@/^1.89pc/[dd]^(0.44)0
				\\ K(qf)\ar[r]^-{g'}\ar[d]_{f'}\drtwocell\omit\omit{_{\varphi}}
					\ar@/_1.87pc/[rr]_(0.48)0
				&A\ar[r]^(0.65){qf}\ar@{->>}[d]^(0.62){f}
					\ar@{}[d]_(0.25){\beta\,\dir{=>}}
					\ar@{}[r]^(0.25){\txt{$\overset{\alpha\;\;}{\dir{=>}}$}}
				&Qg\ar@{=}[d]
				\\ B\ar@{>->}[r]_-{g}\rrlowertwocell<-9>_0{_<2.7>{\;\zeta}}
				&C\ar@{->>}[r]_-{q}
				&Qg
			}
			\end{xym}
			
			By condition 2(b),
			$qf$ is normal cofaithful and $(qf,\beta)=\Coker g'$.
			Moreover, since $f$ is normal cofaithful,
			$(f,\alpha)=\Coker(g'k)$. We are thus in the situation of diagram
			\ref{diagsousquot} and, by condition 2(a), $(f',\kappa)=\Coker k$.
			By Lemma \ref{lemmun}, $\varphi$ is a pullback.
		\end{proof}
		
	The subquotient axiom and homological categories\footnote{This is a different notion
	from Borceux-Bourn homological categories \cite{Borceux2004c}.
	For example, the category of groups is a Borceux-Bourn homological category
	but not a Grandis homological category, because normal monomorphisms are not
	stable under composition.}
	have been introduced by Marco Grandis \cite{Grandis1992a}.  Here is a 2-dimensional
	version.
	
	\begin{df}\label{defhomogran}\index{homological $\Gpdp$-category}%
	\index{Gpd*-category@$\Gpdp$-category!Grandis homological}
		Let $\C$ be a $\Gpdp$-category with all the kernels and cokernels.
		We say that $\C$ is \emph{Grandis homological} if the following conditions 
		(which are equivalent by the previous proposition) hold.
		\begin{enumerate}
			\item
				\begin{enumerate}
					\item For all $B\overset{g}\rightarrowtail C\overset{f}
						\twoheadleftarrow A$ in $\C$, where $g$ is normal faithful
						and $f$ is normal cofaithful,
						the pullback $\pi\col fg'\Ra gf'$ of $f$ and $g$ exists,
						$g'$ is normal faithful, and $f'$ is normal cofaithful;
					\item for all $B\overset{f}\twoheadleftarrow A\overset{g}
						 \rightarrowtail C$ in $\C$, where $g$ is normal faithful
						  and $f$ is normal cofaithful,
						the pushout $\pi\col g'f\Ra f'g$ of $f$ and $g$ exists,
						$g'$ is normal faithful, and $f'$ is normal cofaithful;
					\item $\C$ is $\Ker$-idempotent.
				\end{enumerate}
			\item
				\begin{enumerate}
					\item $\C$ satisfies the subquotient axioms;
					\item normal cofaithful arrows are stable under composition;
					\item normal faithful arrows are stable under
						composition;
					\item $\C$ is $\Ker$-idempotent.
				\end{enumerate}
		\end{enumerate}
	\end{df}
	
	The homological $\Gpdp$-categories satisfy the first isomorphism theorem: we can read
	the subquotient axiom in the following way: if
	$N\overset{i}\ra A$ and $M\overset{j}\ra A$ are normal subobjects of $A$
	(i.e.\ $i$ and $j$ are normal faithful) and if we have
	$N\overset{k}\ra M$ such that $jk\simeq i$ (“$N\incl M$”), then
	\begin{eqn}
		(A/N)/(M/N)\simeq A/M.
	\end{eqn}
	
	Conditions 2(b) and 2(d) mean that $\C$ is $\Ker$-factorisable
	(Proposition \ref{carackerfac}), which implies that in the
	$\Ker$-regular factorisation $f\simeq mq$,
		where $q$ is the cokernel of the kernel of $f$, the arrow
		$m$ is fully 0-faithful. 
	We will now deduce from this
	a more general version of this property.
	Dually, conditions 2(c) and 2(d) mean that $\C$ is $\Coker$-factorisable.
	
	\begin{pon}\label{flcomprel}
		Let $\C$ be a $\Ker$-factorisable $\Gpdp$-category (in particular, $\C$ can
		be homological).  Let be the situation of the following diagram in $\C$, where
		$(k,\kappa)=\Ker(f,\varphi)$, $(q,\zeta)=\Coker k$, $m\zeta\circ\mu k=\kappa$
		and $\varphi'q\circ y\mu=\varphi$. Then $m$ is $\varphi'$-fully 0-faithful.
		
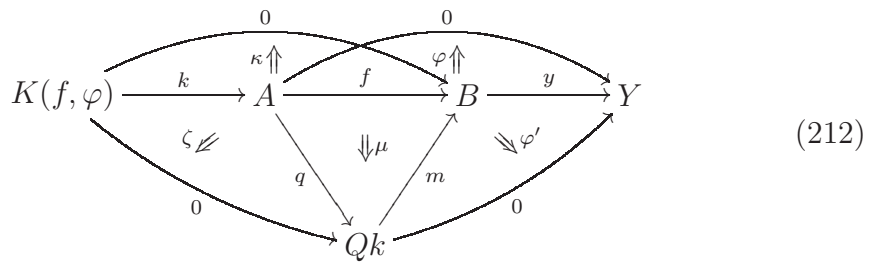
\begin{xym}\label{diagflcomprel}\xymatrix@C=20pt@R=40pt{
			K(f,\varphi)\ar[rr]^-k\ar@/_0.76pc/[drrr]_0\ar@/^2pc/[rrrr]^0
			&{}\rrtwocell\omit\omit{^<-2.7>\kappa\,}
			&A\ar[rr]^-f\ar[dr]_-{q}\rrtwocell\omit\omit{_<4>\mu}
				\ar@{}[dll]|(0.32){\zeta\!\dir{=>}}\ar@/^2pc/[rrrr]^0
			&{}\rrtwocell\omit\omit{^<-2.7>\varphi\,}
			&B\ar[rr]^-y\ar@{}[drr]|(0.32){\dir{=>}\!\varphi'}
			&{}
			&Y
			\\ {}
			&&&Qk\ar[ur]_-m\ar@/_0.76pc/[urrr]_0 &&&{}
		}\end{xym}
	\end{pon}

		\begin{proof}
			We construct $(k_y,\kappa_y)=\Ker y$. By the universal property of kernel,
			there exist $m'$ and $\omega$ such that $\kappa_ym'\circ y\omega = \varphi'$.
			\begin{xym}\xymatrix@C=20pt@R=40pt{
				Qk\ar[rr]^-m\ar[dr]_-{m'}\rrtwocell\omit\omit{_<4>\,\omega}
					\ar@/^2pc/[rrrr]^0
				&{}\rrtwocell\omit\omit{^<-2.7>\varphi'\;}
				&B\ar[rr]^-y\ar@{}[drr]|(0.32){\dir{=>}\!\kappa_y}
				&{}
				&Y
				\\ &Ky\ar[ur]_-{k_y}\ar@/_0.76pc/[urrr]_0 &&&{}
			}\end{xym}
			
			By Corollary \ref{lemmsimplnoy}, since $k_y$ is $\kappa_y$-fully
			faithful, $(k,m'\zeta)=\Ker(m'q)$.  Therefore, since $\C$ is $\Ker$-factorisable,
			 $m'$ is fully 0-faithful.  Finally, by Corollary
			\ref{coronoyrelnoy} applied to the previous diagram, $m$ is $\varphi'$-fully
			0-faithful.
		\end{proof}
		
	In a homological $\Gpdp$-category, the two dual constructions of
	the homology become equivalent.  For symmetric 2-groups, the homology
	was constructed in \cite{Rio2005a}.
	
	\begin{pon}\label{homology}
		Let $\C$ be a homological $\Gpdp$-category.  In the situation of
		Proposition \ref{caracrelex}, there exist an object $H$, arrows
		$q'\col K(b,\psi)\ra H$ and $k'\col H\ra Q(a,\varphi)$ and 2-arrows
		$\zeta'\col q'a'\Ra 0$, $\kappa'\col b'k'\Ra 0$ and $\eta\col qk\Ra k'q'$,
		such that $k'\zeta'\circ\eta a'\circ q\mu = \zeta$,
		$\kappa'q'\circ b'\eta\circ\nu k=\kappa$ and
		\begin{align}\stepcounter{eqnum}
			(H,q',\zeta')&=\Coker (a',\varphi');\\ \stepcounter{eqnum}
			(H,k',\kappa')&=\Ker (b',\psi').
		\end{align}
		We call this object $H$ the \emph{homology at $B$} of the sequence of the upper
		row of the following diagram.\index{homology}
		\begin{xym}\label{grandiaghomo}\xymatrix@R=40pt@C=20pt{
			X\ar[rr]^x\ar@/^2pc/[rrrr]^0\ar@/_0.78pc/[drrr]_-0
			&{}\rrtwocell\omit\omit{^<-2.7>\varphi\,}
			&A\ar[rr]^a\ar[dr]_-{a'}\ar@/^2pc/[rrrr]^0
				\ar@{}[dll]|(0.28){\dir{=>}\,\varphi'}\rrtwocell\omit\omit{_<4>\mu}
				\ddrrlowertwocell\omit{_<3.5>\zeta'}
				\ar@/_2.5pc/@{-->}[ddrr]_0
			& {}\rruppertwocell\omit{^<-2.7>\alpha}
			& B\ar[rr]^b\ar[dr]_-{q}\rrtwocell\omit\omit{_<4>\nu}
				\ar@/^2pc/[rrrr]^0
			& {}\rrtwocell\omit\omit{^<-2.7>\psi\,}
			& C\ar[rr]^y\ar@{}[drr]|(0.28){\dir{=>}\!\psi'}
			& {}
			& Y
			\\ {} &&&K(b,\psi)\ar[ur]_-{k}\ar@{-->}[dr]_-{q'}\rrtwocell\omit\omit{\eta}
			&&Q(a,\varphi)\ar[ur]_-{b'}\ar@/_0.78pc/[urrr]_-0 &&&{}
			\\ &&&&H\ar@{-->}[ur]_-{k'}
				\uurrlowertwocell\omit{_<3.8>\;\;\kappa'}
				\ar@/_2.5pc/@{-->}[uurr]_0
		}\end{xym}
	\end{pon}
	
		\begin{proof}
			We construct $(Kq,k_q,\kappa_q)=\Ker q$, which induces $j\col Kq\ra K(b,\psi)$
			and $\iota\col k_q\Ra kj$ such that $\kappa j\circ b\iota = 
			b'\kappa_q\circ \nu k_q$. Dually, we construct $(Qk,q_k,\zeta_k)=\Coker k$,
			which induces $p\col Q(a,\varphi)\ra Qk$ and $\pi\col q_k\Ra pq$ such that
			$p\zeta\circ\pi a = \zeta_k a'\circ q_k\mu$.

			Then, by Ker-idempotence, the subquotient axiom applies to the
			following diagram. So there exist $H\eqdef Qj\simeq Kp$,
			$q'\col K(b,\psi)\ra H$ and $\zeta''\col q'j\Ra 0$ such that 
			$(q',\zeta'')=\Coker j$, as well as $k'\col H\ra Q(a,\varphi)$,
			$\kappa''\col pk'\Ra 0$
			such that $(k',\kappa'')=\Ker p$, and also $\eta\col qk\Ra k'q'$.
			\begin{xym}\xymatrix@=40pt{
				Kq\ar@{=}[r]\ar[d]_j\drtwocell\omit\omit{\iota}
					\ddlowertwocell<-13>_0{_<5>\zeta''}
				&Kq\ar[d]^{k_q}\ar@/^1.89pc/[dd]^(0.44)0
				\\ K(b,\psi)\ar[r]^-{k}\ar[d]_{q'}\ar@/_1.87pc/[rr]_(0.48)0
					\drtwocell\omit\omit{_{\eta}}
				&B\ar[r]^(0.62){q_k}\ar[d]^(0.62){q}\drtwocell\omit\omit{\pi}
					\ar@{}[d]_(0.25){\zeta_k\,\dir{=>}}
					\ar@{}[r]^(0.20){\txt{$\overset{\kappa_q\;\;}{\dir{=>}}$}}
				&Qk\ar@{=}[d]
				\\ H\ar[r]_-{k'}\rrlowertwocell<-9>_0{_<2.7>\;\,\kappa''}
				&Q(a,\varphi)\ar[r]_-{p}
				&Qk
			}\end{xym}
			
			Next, by the universal property of the kernel of $q$, there exists an arrow
			$t\col A\ra Kq$, with $\tau\col a\Ra k_q t$ such that
			$\kappa_q t\circ q\tau=\zeta$ and $\tau'\col a'\Ra jt$ such that
			$\iota t\circ \tau=k\tau'\circ \mu$, as well as
			$\varphi''\col tx\Ra 0$ such that $j\varphi''\circ\tau' x=\varphi'$.
			By the dual of Proposition \ref{flcomprel} applied to the following diagram,
			$t$ is $\varphi''$-fully 0-cofaithful.
			\begin{xym}\xymatrix@C=20pt@R=40pt{
				X\ar[rr]^-x\ar@/_0.76pc/[drrr]_0\ar@/^2pc/[rrrr]^0
				&{}\rrtwocell\omit\omit{^<-2.7>\varphi\,}
				&A\ar[rr]^-a\ar[dr]_-{t}\rrtwocell\omit\omit{_<4>\tau}
					\ar@{}[dll]|(0.32){\varphi''\!\dir{=>}}\ar@/^2pc/[rrrr]^0
				&{}\rrtwocell\omit\omit{^<-2.7>\zeta\,}
				&B\ar[rr]^-q\ar@{}[drr]|(0.32){\dir{=>}\!\kappa_q}
				&{}
				&Q(a,\varphi)
				\\ {}
				&&&Kq\ar[ur]_-{k_q}\ar@/_0.76pc/[urrr]_0 &&&{}
			}\end{xym}
			
			Since $(q',\zeta'')=\Coker j$,
			by the dual of Proposition \ref{lemmsimplnoy} applied to the following diagram,
			$(q',\zeta')=\Coker (a',\varphi')$, where 
			$\zeta'\eqdef \zeta'' t\circ q'\tau'$.
			\begin{xym}\xymatrix@C=20pt@R=40pt{
				X\ar[rr]^-x\ar@/_0.76pc/[drrr]_0\ar@/^2pc/[rrrr]^0
				&{}\rrtwocell\omit\omit{^<-2.7>\varphi'\,}
				&A\ar[rr]^-{a'}\ar[dr]_-{t}\rrtwocell\omit\omit{_<4>\;\,\tau'}
					\ar@{}[dll]|(0.32){\varphi''\!\!\dir{=>}}\ar@/^2pc/[rrrr]^0
				&{}\rrtwocell\omit\omit{^<-2.7>\zeta'\,}
				&K(b,\psi)\ar[rr]^-{q'}\ar@{}[drr]|(0.32){\dir{=>}\!\zeta''}
				&{}
				&H
				\\ {}
				&&&Kq\ar[ur]_-{j}\ar@/_0.76pc/[urrr]_0 &&&{}
			}\end{xym}
			
			Dually, we have $\kappa'$
			such that $(k',\kappa')=\Ker (b',\psi')$.
		\end{proof}
		
	This proposition allows us to prove new characterisations of the relative exactness
	of a sequence, which extends Proposition
	\ref{caracrelex}.

	\begin{pon}\label{caracrelexbis}\index{sequence!relative exact!characterisation}%
	\index{relative exact sequence!characterisation}
		Let us consider the situation of diagram \ref{grandiaghomo}, where the 2-arrows of
		the upper row are compatible.
		If $\C$ is $\Ker$-factorisable, then the following conditions are
		equivalent to those of Proposition \ref{caracrelex} (i.e.
		to the relative exactness of the sequence):
		\begin{enumerate}
			\item[3.] $a'$ is $\varphi'$-fully 0-cofaithful;
			\item[4.] $b'$ is $\psi'$-fully 0-faithful.
		\end{enumerate}
		If $\C$ is homological, we can add the following condition:
		\begin{enumerate}
			\item[5.] $H\simeq 0$.
		\end{enumerate}
	\end{pon}
	
		\begin{proof}
			{\it 1 $\Rightarrow$ 4. } If we have $\omega$ satisfying the two required
			equations and such that $(k,\omega)=\Ker q$, then $a'$ is $\varphi'$-fully
			0-cofaithful by Proposition \ref{flcomprel} (which applies because
			$\C$ is $\Ker$-factorisable).
			
			{\it 4 $\Rightarrow$ 2. } If $a'$ is $\varphi'$-fully 0-cofaithful,
			there exists $\omega\col qk\Ra 0$ such that $\omega a'\circ q\mu = \zeta$;
			by using the cofaithfulness of $a'$, we also prove that
			$b'\omega\circ\nu k=\kappa$.
			Then, since $(q,\zeta)=\Coker(a,\varphi)$, we also have
			$(q,\omega)=\Coker k$, by the dual of Corollary \ref{lemmsimplnoy}.
			
			{\it 2  $\Rightarrow$ 3 and 3 $\Rightarrow$ 1. }
			The proof is dual.
			
			{\it 3 $\Leftrightarrow$ 5. } Let us assume that $\C$
			is homological. Since $(H,q',\zeta')=\Coker (a',\varphi')$,
			$H\simeq 0$ if and only if $a'$ is $\varphi'$-fully 0-cofaithful,
			by Proposition \ref{clasproprelker}.
		\end{proof}
		
For the (non relative) exactness, the homology is constructed in the same way as for the relative exactness, with $0$ for $X$ and $Y$.  So we can extend Proposition \ref{defsequencex}.

	\begin{coro}\label{caracdexbis}\index{sequence!exact!characterisation}%
	\index{exact!sequence!characterisation}
		Let us consider the situation of diagram \ref{grandiaghomo}, with $X\eqdef 0$
		and $Y\eqdef 0$.
		If $\C$ is $\Ker$-factorisable, then the following conditions are
		equivalent to those of Proposition \ref{defsequencex} (i.e.
		to the exactness of the sequence):
		\begin{enumerate}
			\item[3.] $a'$ is fully 0-cofaithful;
			\item[4.] $b'$ is fully 0-faithful.
		\end{enumerate}
		If $\C$ is homological, we can add the following condition:
		\begin{enumerate}
			\item[5.] $H\simeq 0$.
		\end{enumerate}
	\end{coro}

\subsection{Puppe-exact $\Gpdp$-categories}

	It is important to distinguish Puppe-exact $\Gpdp$-categories from 2-Puppe-exact
	$\Gpdp$-categories, defined in the following chapter.
	2-Puppe-exactness is a 2-dimensional analogue of Puppe-exactness for $\Ensp$-categories, 
	whereas
	Puppe-exactness for $\Gpdp$-categories is a \emph{generalisation} of
	Puppe-exactness for $\Ensp$-category (a $\Ensp$-category seen as a
	$\Gpdp$-category is Puppe-exact if and only if it is Puppe-exact
	in the 1-dimensional sense), which allows us to give them the same name.

	\begin{df}\label{defpuppex}\index{Puppe-exact!Gpd*-category@$\Gpdp$-category}%
	\index{Gpd*-category@$\Gpdp$-category!Puppe-exact}
		We call a $\Gpdp$-category \emph{Puppe-exact} if the following conditions
		hold:
		\begin{enumerate}
			\item in the following situation, if $(a,\alpha)=\Ker(b,\beta)$ and
				$(c,\beta)=\Coker(b,\alpha)$, then the arrow $w\col Qa\ra Kc$
				given by Lemma \ref{exenunfl} (with $X$ and $Y$ equal to $0$)
				is an equivalence;
				\begin{xym}\xymatrix@=40pt{
					A\ar[r]^a\ar@/^2pc/[rr]^0
						\rruppertwocell\omit{^<-2.7>\alpha}
					& B\ar[r]^b\ar@/_2pc/[rr]_0
						\rrtwocell\omit\omit{_<2.7>\beta}
					& C\ar[r]^c
					& D
				}\end{xym}
			\item 0-monomorphisms and 0-epimorphisms are stable under
				composition.
		\end{enumerate}
	\end{df}

	For a $\Ensp$-category seen as a $\Gpdp$-category, condition 2 is
	always true, and condition 1 becomes ordinary Puppe-exactness: for the sequence
	\begin{eqn}
		0\longrightarrow Kf\overset{k_f}\longrightarrow A\overset{f}\longrightarrow B\overset{q_f}\longrightarrow Qf\longrightarrow 0,
	\end{eqn}
	we have $Q (k_f)\simeq K(q_f)$.
	In dimension 2, we cannot express this condition in a similar form because,
	in general, the 2-arrows $\kappa_f$ and $\zeta_f$
	are not compatible and the comparison arrow $w$ does not exist.

	\begin{pon}\label{zermonregmon}
		In a $\Gpdp$-category which satisfies condition 1 of the definition
		of Puppe-exactness, every 0-monomorphism is normal faithful.
	\end{pon}
	
		\begin{proof}
			Let $f$ be a 0-monomorphism.  
			On the one hand, by Proposition \ref{clasproprelker}, $(0,0,1_0)=\Ker(f,\zeta)$
			because $f$ is $\zeta$-fully 0-faithful, and on the other hand $(Qf,q,\zeta)
			=\Coker(f,1_0)$, by Proposition \ref{caspartickerrel}.
			So condition 1 of the definition
			of Puppe-exactness applies to the upper row of the following diagram
			and, since $(A,1_A,1_0)=\Coker 0_A$, the comparison $f\col A\ra B$, with $\zeta$,
			is the kernel of $q$.\qedhere
			\begin{xym}\xymatrix@C=20pt@R=30pt{
				0\ar[rr]^-0\ar@/_0.78pc/[drrr]_0\ar@/^2pc/[rrrr]^0
				&{}\rrtwocell\omit\omit{^<-2.7>1_0\,}
				&A\ar[rr]^-f\ar@{=}[dr]\ar@{}[dll]|(0.32){1_0\!\dir{=>}}
					\ar@/^2pc/[rrrr]^0
				&{}\rrtwocell\omit\omit{^<-2.7>\zeta\,}
				&B\ar[rr]^-q\ar@{}[drr]|(0.32){\dir{=>}\!\zeta}
				&{}
				&Qf
				\\ {}
				&&&A\ar[ur]_-f\ar@/_0.78pc/[urrr]_0 &&&{}
			}\end{xym}
		\end{proof}
		
	\begin{rem}\label{remcondplusfortqueppex}
		A stronger property could be interesting.  In the situation of the
		following diagram, we could ask that $f$ be $\varphi$-fully 0-faithful if
		and only if $(f,\varphi)=\Ker(y,\zeta)$.
		\begin{xym}\xymatrix@=40pt{
			A\ar[r]^f\rrlowertwocell<-9>_0{_<2.7>\,\varphi}
			&B\ar[r]^y\rruppertwocell<9>^0{^<-2.7>\zeta\,}
			&Y\ar[r]_-q
			&Q(y,\varphi)
		}\end{xym}
		This condition contains the property of the previous proposition,
		because the relative cokernel of a cokernel is $0$.
		Moreover, this condition implies that $f$ is fully 0-faithful
		if and only if $f=\Root(\Copip f)$ (this is the case where $Y\equiv 0$), and that
		$\varphi$ is a 0-monoloop if and only if $\varphi=\Pip(\Coroot \varphi)$
		(this is the case where $B\equiv 0$).
		These two properties are false in general in a Puppe-exact $\Gpdp$-category,
		as the example of Puppe-exact $\Ensp$-categories shows.
		But they are true in a 2-Puppe-exact $\Gpdp$-category.
		It remains to establish the link between $\Gpdp$-categories
		satisfying this property and its dual and 2-Puppe-exact $\Gpdp$-categories .
	\end{rem}
	
	\begin{coro}\label{carottesue}
		In a Puppe-exact $\Gpdp$-category, the following properties are
		equivalent for an arrow $f$:
		\begin{enumerate}
			\item $f$ is a 0-monomorphism;
			\item $f$ is a monomorphism;
			\item $f$ is a kernel;
			\item $f$ is a normal faithful arrow.
		\end{enumerate}
	\end{coro}
	
		\begin{proof}
			The implications $4 \Rightarrow 3 \Rightarrow 2 \Rightarrow 1$ are always
			true.  The previous proposition completes the proof.
		\end{proof}
	
	Here is a characterisation of Puppe-exact $\Gpdp$-categories.

	\begin{pon}\label{caracpuppex}
		Let $\C$ be a $\Gpdp$-category. The following conditions are equivalent:
		\begin{enumerate}
			\item $\C$ is Puppe-exact;
			\item \begin{enumerate}
					\item $\C$ is Grandis homological;
					\item every arrow which is both fully 0-faithful and fully
						0-cofaithful is an equivalence.
				\end{enumerate}
		\end{enumerate}
	\end{pon}
	
		\begin{proof}
			{\it 1 $\Rightarrow$ 2(a). }
			Normal cofaithful arrows are stable under composition since
			the normal cofaithful arrows coincide with the 0-epimorphisms,
			which are stable under composition by hypothesis.  Dually, normal faithful
			arrows are stable under composition.
			
			Next, by the equivalence between conditions 3 and 4 of
			Corollary \ref{carottesue}, $\C$ is $\Ker$-idempotent.
			
			It remains to prove the subquotient axiom. Let us consider diagram
			\ref{diagsousquot}. Since $(k_s,\kappa_s)=\Ker s$ and $(j,\beta)=\Ker q$,
			$\mu$ is a pullback, by Lemma \ref{lemmun}.  And since $(i,\alpha)
			=\Ker r$, we have $(k,\alpha')=\Ker q'$, by Lemma \ref{lemmtwo}.
			Finally, by Corollary \ref{lemmsimplnoy}, $(k,\alpha)=\Ker(rj,\beta)$.
			Dually, $(s,\beta)=\Coker(rj,\alpha)$.  So condition 1 of the definition
			of Puppe-exactness applies to the sequence $(k,\alpha,rj, \beta, s)$, and
			the comparison arrow $w\col Qk\ra Ks$ is an equivalence.
			  
			{\it 1 $\Rightarrow$ 2(b). }
			Let $f\col A\ra B$ be a fully 0-faithful and fully 0-cofaithful arrow.
			Since $f$ is fully 0-faithful, $\Ker f\simeq 0$ and, by Proposition
			\ref{zermonregmon}, the fact that $f$ is a 0-epimorphism implies
			that $f$ is normal cofaithful, i.e.\ that $f$
			is a cokernel of $0_A$. Thus $f$ is an equivalence.
			
			{\it 2 $\Rightarrow$ condition 1 of Definition \ref{defpuppex}. }
			Let us consider the situation of Lemma \ref{exenunfl}, with
			$X$ and $Y$ equal to $0$, and $\varphi$ and $\psi$ equal to $1_0$.
			We assume that $(a,\alpha)=\Ker (b,\beta)$ and $(c,\beta)=\Coker(b,\alpha)$.
			We construct the kernel of $w$, which induces $t$ and $\tau$ such that
			$k\kappa_w t\circ kw\tau\circ\eta a=\alpha.$
			\begin{xym}\xymatrix@=40pt{
				A\ar[r]^a\ar@/^2pc/[rr]^0\ar[d]_t
					\drtwocell\omit\omit{\tau}
					\rruppertwocell\omit{^<-2.7>\alpha}
				& B\ar[r]^b\ar[d]^{q}
					\ar@/^2pc/[rr]^0\rtwocell\omit\omit{_<5.7>\,\eta}
					\rrtwocell\omit\omit{^<-2.7>\beta\,}
				& C\ar[r]^c
				& D
				\\ Kw\ar@/_2pc/[rr]_0\ar[r]_-{k_w}
					\rrlowertwocell\omit{_<2.7>\;\;\,\kappa_w}
				&{Qa}\ar[r]_-w
				&{Kc}\ar[u]_-{k}\urlowertwocell_0<-3>{_<-2>{\kappa}}
			}\end{xym}
			By Corollary \ref{lemmsimplnoy}, $(k_w,k\kappa_w)=
			\Ker(kw,\kappa w)$, since $k$ is 
			$\kappa$-fully faithful.  So, by Lemma
			\ref{lemmun}, since we also have $(a,\alpha)=\Ker(b,\beta)$, $\tau$
			is a pullback.
			
			Now, since $\C$ is $\Ker$-idempotent, $q$ is a normal cofaithful arrow
			and $k_w$ is a normal faithful arrow.  So,
			by condition 1(a) of the definition of homological $\Gpdp$-category,
			$t$ is also normal cofaithful.
			
			Next, $\zeta\circ\tau^{-1}\col k_w t\Ra 0$ is compatible
			with $\kappa_w$ (we see that by testing the condition with $k$, which is faithful).
			Thus, since $k_w$ is $\kappa_w$-fully faithful, there exists
			$\sigma\col t\Ra 0$ such that $k_w\sigma\circ\tau = \zeta$.
			
			Then Lemma \ref{lemmtwo} applies to the following diagram
			and $(1_A,\sigma)=\Ker t$.  So, since $t$ is normal cofaithful,
			$(t,\sigma)=\Coker 1_A$ and $Kw\simeq 0$.
			Thus $w$ is fully 0-faithful.
			\begin{xym}\xymatrix@=40pt{
				A\ar@{=}[r]\ar@{=}[d]\rruppertwocell<9>^0{^<-2.7>\sigma\,}
				&A\ar[r]^-{t}\ar[d]_{a}\drtwocell\omit\omit{^{\tau}}
				&Kw\ar[d]^{k_w}
				\\ A\ar[r]_-{a}\rrlowertwocell<-9>_0{_<2.7>\,\zeta}
				&B\ar[r]_-{q}
				&Qa
			}\end{xym}			
			Dually, we prove that $w$ is fully 0-cofaithful, which allows us
			to conclude that $w$ is an equivalence, by condition 2(b).
			
			{\it 2 $\Rightarrow$ condition 2 of Definition \ref{defpuppex}. }
			By the previous part of this proof, we know that $\C$
			satisfies condition 1 of the definition of Puppe-exactness. So,
			by Proposition \ref{zermonregmon}, the 0-monomorphisms
			and the normal faithful arrows coincide.  Thus conditions
			2(b) and 2(c) of the definition of homological $\Gpdp$-category
			imply condition 2 of the definition of Puppe-exactness.
		\end{proof}

	\begin{coro}
		In a Puppe-exact $\Gpdp$-category, the epimorphisms and the
		fully 0-faithful arrows form a factorisation system.
	\end{coro}
	
	We can generalise condition 1 of the definition of Puppe-exactness
	to the general situation of Lemma \ref{exenunfl}.
	
	\begin{pon}
		Let us consider the situation of Lemma \ref{exenunfl} in a Puppe-exact 
		$\Gpdp$-category.  If the upper row is relative exact at $B$ and $C$, then
		$w$ is an equivalence.
	\end{pon}
	
		\begin{proof}
			We construct $K(b,\beta)$ and $Q(b,\alpha)$; this induces $a'$,
			$b'$, $\mu$, $\nu$, $\varphi'$ and $\psi'$, as in the following diagram,
			such that $\kappa'a'\circ b\mu=\alpha$, $c'\zeta'\circ\nu b=\beta$,
			$k'\varphi'\circ\mu x=\varphi$ and $\psi'q'\circ y\nu=\psi$.
			\begin{xym}\xymatrix@R=35pt@C=17.5pt{
				X\ar[rr]^x\ar@/^2pc/[rrrr]^0\ar@/_0.78pc/[drrr]_-0
				&{}\rrtwocell\omit\omit{^<-2.7>\varphi\,}
				&A\ar[rr]^a\ar[dr]_-{a'}\ar@/^2pc/[rrrr]^0
					\ar@{}[dll]|(0.28){\dir{=>}\,\varphi'}\rrtwocell\omit\omit{_<4>\mu}
				& {}\rruppertwocell\omit{^<-2.7>\alpha}
				& B\ar[rr]^b\ar@/^2pc/[rrrr]^0\ar@{}[drr]_(0.28){\dir{=>}\!\kappa'}
					\ar@/_0.78pc/[drrr]_-0
				& {}\rrtwocell\omit\omit{^<-2.7>\beta\,}
				& C\ar[rr]^c\ar[dr]_-{q'}\rrtwocell\omit\omit{_<4>\nu}
					\ar@/^2pc/[rrrr]^0\ar@{}[dll]^(0.28){\dir{=>}\,\zeta'}
				& {}\rrtwocell\omit\omit{^<-2.7>\psi\,}
				& D\ar[rr]^y\ar@{}[drr]|(0.28){\dir{=>}\!\psi'}
				& {}
				& Y
				\\ {} &&&K(b,\beta)\ar[ur]_-{k'}\ar@/_0.78pc/[urrr]_-0
				&&{}&&Q(b,\alpha)\ar[ur]_-{c'}\ar@/_0.78pc/[urrr]_-0 &&&{}
			}\end{xym}
			
			As the sequence is relative exact at $B$, by Proposition
			\ref{caracrelexbis}, $a'$ is $\varphi'$-fully
			0-cofaithful and, as it is relative exact at $C$, $c'$ is
			$\psi'$-fully 0-faithful.
			Then $Q(a,\varphi)$ is also the cokernel of $k'$ (as $a'$ is
			$\varphi'$-fully 0-cofaithful) and, dually, $K(c,\psi)$ is also the kernel
			of $q'$.
			
			Moreover, since $c'$ is 0-faithful, the relative kernel-candidates of $(b,\beta)$ and of $(b,\zeta')$ coincide. So $(k',\kappa')=\Ker(b,\zeta')$. Dually,
			$(q',\zeta')=\Coker(b,\kappa')$.  Thus condition 1 of the definition of
			Puppe-exactness applies and $w\col Q(a,\varphi)\ra K(c,\psi)$ is an
			equivalence.
		\end{proof}
	
	Now we can prove the converse of Proposition \ref{twoextrelex}: 
	in a Puppe-exact $\Gpdp$-category, every
	relative exact sequence factors into extensions (short exact sequences).
	
	\begin{pon}\label{suitrelexdecompenext}
		Let $(A_n,a_n,\alpha_n)_{n\col\mathbb{Z}}$ be a relative exact sequence,
		where $a_n\!:\allowbreak A_n\ra A_{n+1}$ and $\alpha_n\col a_{n+1}a_n\Ra 0$.
		Then for all $n\col\mathbb{Z}$, there exist an object $I_n$, arrows
		$k_n\col I_n\ra A_{n+1}$ and $q_n\col A_n\ra I_n$, and 2-arrows
		$\varphi_n\col q_{n+1}k_n\Ra 0$ and $\eta_n\col a_n\Ra k_nq_n$ such that
		\begin{enumerate}
			\item $0\longrightarrow I_n\overset{k_n}\longrightarrow A_{n+1}\overset{q_{n+1}}\longrightarrow I_{n+1}\longrightarrow 0$, with
				$\varphi_n$, is an extension;
			\item $k_{n+1}\varphi_n q_n\circ (\eta_{n+1} * \eta_n) = \alpha_n$.
		\end{enumerate}
		\begin{xym}\xymatrix@!@=10pt{
			&0\ar[dr]
			&&0
			&&0\ar[dr]
			&&0
			\\ &&I_{n-1}\ar[dr]^{k_{n-1}}\ar[ur]
			&&{}
			&&I_{n+1}\ar[ur]\ar[dr]^{k_{n+1}}
			\\ {}\ar[r]
			&A_{n-1}\ar[rr]_{a_{n-1}}\ar[ur]^{q_{n-1}}
			&&A_{n}\ar[rr]_{a_n}\ar[dr]_{q_n}
			&&A_{n+1}\ar[rr]_{a_{n+1}}\ar[ur]^{q_{n+1}}
			&{}
			&A_{n+2}\ar@{-}[dr]\ar@{-}[r] &
			\\ {}\ar[ur] &{}
			&&&I_n\ar[ur]_{k_{n}}\ar[dr]
			&&&{} &
			\\&&&0\ar[ur]
			&&0
		}\end{xym}
	\end{pon}
	
		\begin{proof}
			Let $n$ be an integer.
			By applying the previous proposition to the sequence at $A_n$ and $A_{n+1}$,
			we get
			$I_n$, $q_n$, $k_n$, $\eta_n$, as well as $\zeta_n\col q_na_{n-1}\Ra 0$
			and $\kappa_n\col a_{n+1}k_n\Ra 0$ such that
			\begin{enumerate}
				\item $(q_n,\zeta_n)=\Coker (a_{n-1},\alpha_{n-2})$;
				\item $(k_n,\kappa_n)=\Ker (a_{n+1},\alpha_{n+1})$.
			\end{enumerate}

			Then we construct a new relative exact sequence
			$\ldots, A_{n-1},\allowbreak a_{n-1},\allowbreak 	A_n,\allowbreak q_n,\allowbreak I_n, 0$
			(with $\ldots, \alpha_{n-3}, \alpha_{n-2},	\zeta_n$ as 2-arrows).
			By applying again the previous proposition to this sequence
			at $A_{n-1}$ and $A_n$, we get $I_{n-1}$, $q_{n-1}$, $k_{n-1}$, $\eta_{n-1}$,
			as well as $\zeta_{n-1}\col q_{n-1}a_{n-2}\Ra 0$
			and $\varphi_{n-1}\col q_nk_{n-1}\Ra 0$ such that
			\begin{enumerate}
				\item $(q_{n-1},\zeta_{n-1})=\Coker (a_{n-2},\alpha_{n-3})$;
				\item $(k_{n-1},\varphi_{n-1})=\Ker q_n$.
			\end{enumerate}
			And so on.  For indices greater than $n$, the construction is dual.
		\end{proof}
	
	In Puppe-exact $\Gpdp$-categories, we can give a simple characterisation of the relative exactness of some sequences.
	
	\begin{pon}
		Let $\C$ be a Puppe-exact $\Gpdp$-category. Let us consider the sequence of diagram
		\ref{diagkerkerzer}.  Then
		\begin{enumerate}
			\item this sequence is exact at $K$ if and only if $k$ is
				$\kappa$-fully 0-faithful;
			\item this sequence is exact at $K$ and $A$ if and only if
				$(k,\kappa)=\Ker (f,\varphi)$.
		\end{enumerate}
	\end{pon}
		
		\begin{proof}
			\begin{enumerate}
				\item (This first property is always true.)
					This follows immediately from the fact that $k$ is $\kappa$-fully
					0-faithful if and only if $\Ker(k,\kappa)\simeq 0$.
				\item We have already proved that, if $(k,\kappa)=\Ker (f,\varphi)$, then
					the sequence is relative exact (Proposition \ref{kerrelrelex}).
					Let us assume that the sequence is relative exact at $K$ and $A$. By
					point 1, $k$ is $\kappa$-fully 0-faithful
					and is thus a 0-monomorphism.
					
					Let us consider the following diagram, where $(q,\zeta)=\Coker k$,
					$f'\zeta\circ\mu k =\kappa$ and $\varphi'q\circ y\mu=\varphi$.
					\begin{xym}\xymatrix@C=20pt@R=40pt{
						K\ar[rr]^-k\ar@/_0.76pc/[drrr]_0\ar@/^2pc/[rrrr]^0
						&{}\rrtwocell\omit\omit{^<-2.7>\kappa\,}
						&A\ar[rr]^-f\ar[dr]_-{q}\rrtwocell\omit\omit{_<4>\mu}
							\ar@{}[dll]|(0.32){\zeta\!\dir{=>}}\ar@/^2pc/[rrrr]^0
						&{}\rrtwocell\omit\omit{^<-2.7>\varphi\,}
						&B\ar[rr]^-y\ar@{}[drr]|(0.32){\dir{=>}\!\varphi'}
						&{}
						&Y
						\\ {}
						&&&Qk\ar[ur]_-{f'}\ar@/_0.76pc/[urrr]_0 &&&{}
					}\end{xym}
					Since the sequence is exact at $A$,
					by Proposition \ref{caracrelexbis}, $f'$ is $\varphi'$-fully
					0-faithful.  Now, since $\C$ is Puppe-exact and $k$ is a 0-monomorphism,
					$(k,\zeta)=\Ker q$.  So, by Lemma \ref{lemmsimplnoy},
					$(k,\kappa)=\Ker(f,\varphi)$.\qedhere
				\end{enumerate}
		\end{proof}
	
	\begin{pon} Let $\C$ be a Puppe-exact $\Gpdp$-category. The sequence \ref{sequencesuiv}
		is relative exact if and only if it is an extension.
	\end{pon}
		\begin{proof}
			We apply the previous proposition and its dual.
		\end{proof}

\subsection{Abelian $\Gpd$-categories}

	The notion of abelian $\Gpd$-category, which is a \emph{generalisation} of the usual notion of abelian category, must not be confused with the notion of 2-abelian $\Gpd$-category  (Definition \ref{deftwoab}).  This notion is provisional, because it is not sure that
	we can deduce from it that $\C$ is additive, or even that every fully 0-faithful arrow is fully faithful.
	It will be perhaps necessary to add the condition that $\C$ is enriched in the $\Gpd$-category $\CGS$ of symmetric 2-groups to have an appropriate notion of abelian $\Gpd$-category.

	\begin{df}\label{dfgpdab}\index{abelian!$\Gpd$-category}%
	\index{Gpd-category@$\Gpd$-category!abelian}
		An \emph{abelian $\Gpd$-category} is a $\Gpdp$-category $\C$ with zero object,
		finite products and coproducts, kernels and cokernels, such that:
		\begin{enumerate}
			\item every 0-monomorphism is normal faithful;
			\item every 0-epimorphism is normal cofaithful;
			\item fully 0-faithful arrows and 0-monomorphisms are stable
				under pushout;
			\item fully 0-cofaithful arrows and 0-epimorphisms are stable
				under pullback.
		\end{enumerate}
	\end{df}
	
	A $\Ens$-category is abelian in the usual sense if and only if it is abelian
	as a locally discrete $\Gpd$-category; for a 
	$\Ens$-category, conditions 3 and 4 follow from the first two.
	We will prove in Corollary \ref{twoabimplab} that 2-abelian $\Gpd$-categories
	are also abelian.  Thus the notion of abelian $\Gpd$-category
	covers both the 1-dimensional notion of abelian category and the 2-dimensional notion of 2-abelian  $\Gpd$-category.  A natural question is: does it exist other natural examples of abelian $\Gpd$-category?
	
	\begin{pon}
		Every abelian $\Gpd$-category $\C$ is Puppe-exact.
	\end{pon}
	
		\begin{proof}
			We use characterisation \ref{caracpuppex}. First, we prove
			that $\C$ is Grandis homological.  Conditions 1(a) and 1(b)
			of Definition \ref{defhomogran} hold, by conditions 3 and
			4 of the definition of abelian $\Gpd$-category, since the 0-monomorphisms
			and the normal faithful arrows coincide, by condition 1, and since
			the 0-epimorphisms and the normal cofaithful arrows coincide, by
			condition 2.  Moreover, since the kernels are 0-monomorphisms and so, by
			condition 1, normal faithful arrows, $\C$ is $\Ker$-idempotent.
			
			Next, let us prove condition 2(b) of Proposition \ref{caracpuppex}.
			Let $A\overset{f}\ra B$ be a fully 0-faithful and fully 0-cofaithful arrow.
			By Proposition \ref{claspropker},
			$\Ker f \simeq 0$.  And since $f$ is fully 0-cofaithful,
			$f$ is a 0-epimorphism and so, by condition 2, a normal cofaithful arrow.
			Thus $f$ is the cokernel of $0_A$ and is an equivalence.
		\end{proof}

\section{Diagram lemmas}

\subsection{$3 \times 3$ lemma}

	\begin{pon}[$3 \times 3$ lemma]\label{lemmtrxtr}\index{lemma!3 x 3@$3\times 3$}%
	\index{3 x 3 lemma@$3\times 3$ lemma}
		Let $\C$ be a Puppe-exact $\Gpdp$-category and let be the situation of
		diagram \ref{diagtyptroitroi} in $\C$.
		\begin{enumerate}
			\item If the three columns and the two last rows are relative
				exact (are extensions), then the first row is relative
				exact.
			\item If the middle row and the middle column are relative exact, all
				rows and columns are relative exact if and only if
				$\varphi_1$ is a pullback relative to
				$\gamma_1g_1\circ c_2\psi_1^{-1}$ and $c_2\eta_2$,
				$\psi_2$ is a pushout relative
				to $c_1\eta_1\circ \psi_1^{-1}f_1$ and $\beta_1 f_1$, 
				$g_1$ is $\eta_1$-fully 0-cofaithful, $c_1$ is $\gamma_1$-fully
				0-faithful, $a_2$ is $\alpha_1$-fully 0-co\-faith\-ful and $f_3$
				is $\eta_3$-fully 0-faithful.
		\end{enumerate}
	\end{pon}

		\begin{proof}
			{\it 1. } By the restricted kernels lemma
			(Proposition \ref{lemmnoyrestr}), $\varphi_1$
			is a relative pullback and $(f_1,\eta_1)=\Ker g_1$.
			Then Lemma \ref{lemmcinq} applies to the following diagram
			and $(f_1,c_1\eta_1)=\Ker(c_1g_1,\gamma_1g_1)$.
			\begin{xym}\xymatrix@C=40pt@R=15pt{
				A_1\ar[dd]_{a_1}\ar[r]^{f_1}\ddrtwocell\omit\omit{\;\;\varphi_1}
					\drruppertwocell^0<15>{^<-5.6>{c_1\eta_1\;\;\;\;}}
				&B_1\ar[dr]^{c_1g_1}\ar[dd]_{b_1}\drtwocell\omit\omit{_<3.5>\;\;\psi_1}
					\drruppertwocell^0{^<-1.3>{\;\,\;\;\;\;\;\;\;\;\;\;\;\;\;\gamma_1g_1}}
				\\ &&C_2\ar[r]^{c_2}
				&C_3
				\\ A_2\ar[r]_{f_2}\urrlowertwocell_0<-15>{_<5.6>{\;\,\eta_2}}
				&B_2\ar[ur]_{g_2}
			}\end{xym}
			
			By Lemma
			\ref{lemmquatre}, since $(g_2,\eta_2)=\Coker f_2$, $(g_3,\eta_3)=\Coker f_3$
			and $a_2$ is $\alpha_1$-fully 0-cofaithful, $\psi_2$ is a relative pushout.
			By applying Lemma \ref{lemmcinq} to the dual of the previous diagram,
			we prove that $(c_2,\gamma_1g_1)=\Coker(c_1g_1,c_1\eta_1)$.
			
			Now $(c_1,\gamma_1)=\Ker c_2$.  So, by condition 1 of the definition
			of Puppe-exactness, $(g_1,\eta_1)=\Coker f_1$.
			
			{\it 2. }  If all the rows and columns are exact, then $\varphi_1$
			and $\psi_2$ are relative pullback or pushout, as we have seen in the proof
			 of point 1.  The other properties come from the fact that the
			 outer rows and columns are extensions.
			
			Conversely, we know that
			$(f_2,\eta_2)=\Ker g_2$, $c_1$ is $\gamma_1$-fully
			0-faithful and $\varphi_1$ is a relative pullback. So, by Lemma
			\ref{lemmquatre}, $(f_1,\eta_1)=\Ker g_1$.  And since, moreover, $g_1$
			is $\eta_1$-fully 0-cofaithful, the first row is an extension.
			The proof is similar for the third row and the first and third columns.
		\end{proof}
	
	Thanks to the $3\times 3$ lemma and to the subquotient axiom, we can imitate in a Puppe-exact $\Gpdp$-category the proof of the snake lemma in a Puppe-exact category of Schubert's book \cite{Schubert1972a}.  We get then the following result: in the situation of diagram \ref{diagserp},
	if the two middle rows are extensions and the three columns (extended on each side
	by $0$s) are \emph{relative} exact sequences, then there exist
	an arrow $d$ and 2-arrows $\delta\col d\bar{g}\Ra 0$
	and $\delta'\col\bar{f}'d\Ra 0$ which make the dashed sequence (extended on each side
	by $0$s) a \emph{relative} exact sequence.
	But this property is not the snake lemma we need to construct the
	long exact sequence of homology.  The reason is that, in this case, the columns
	(not extended by $0$s) are not relative exact, but exact.
	And the sequence we get is not relative
	exact, but exact.  The genuine snake lemma is Proposition \ref{lemmserp}.

\subsection{Short 5 lemma}

	The short 5 lemma is always true for (fully) 0-faithful or 0-cofaithful arrows.
	
	\begin{lemm}\label{petitlemmcinqtjrvrai}
		Let $\C$ be a $\Gpdp$-category.  In the situation of the following diagram,
		 where $\eta' a\circ g'\varphi\circ\psi f=c\eta$ and where
		the rows are extensions, we have the following properties:
		\begin{enumerate}
			\item if $a$ and $c$ are 0-faithful, then $b$ is 0-faithful;
			\item if $a$ and $c$ are 0-cofaithful, then $b$ is 0-cofaithful;
			\item if $a$ and $c$ are fully 0-faithful, then $b$ is fully 
				0-faithful;
			\item if $a$ and $c$ are fully 0-cofaithful, then $b$ is
				fully 0-cofaithful.
		\end{enumerate}
		\begin{xym}\label{diagpetitlemmcinq}\xymatrix@=40pt{
			0\ar[r]
			&A\rruppertwocell<9>^0{^<-2.7>\eta\,}
				\drtwocell\omit\omit{_{\,\varphi}}\ar[r]_-{f}\ar[d]_{a}
			&B\ar[r]_-{g}\ar[d]^{b}\drtwocell\omit\omit{\,\psi}
			&C\ar[d]^c\ar[r]
			&0
			\\ 0\ar[r] 
			&A'\ar[r]_-{f'}\rrlowertwocell<-9>_0{_<2.7>{\;\,\eta'}}
			&B'\ar[r]_{g'}
			&C'\ar[r]
			&0
		}\end{xym}
	\end{lemm}
	
		\begin{proof}
			We prove properties 1 and 3; the other two are dual.
			
			{\it Property 1. } Let us assume that $a$ and $c$ are 0-faithful. Let
			$\chi\col 0\Ra 0\col X\ra B$ be such that $b\chi =1_0$.  We have
			\begin{eqn}
				b\chi=1_0 \Rightarrow g'b\chi=1_0 \Rightarrow cg\chi=1_0
				\Rightarrow g\chi=1_0,
			\end{eqn}
			because $g'b\simeq cg$ and $c$ is 0-faithful.  As $g\chi=1_0$,
			$\chi\col f0\Ra f0$ is compatible with $\eta$ and, as
			$f$ is $\eta$-fully faithful, there exists $\hat{\chi}\col 0\Ra 0\col X
			\ra A$ such that $f\hat{\chi}=\chi$.  We have then $bf\hat{\chi}=b\chi=1_0$ and
			\begin{eqn}
				bf\hat{\chi}=1_0 \Rightarrow f'a\hat{\chi}=1_0
				\Rightarrow a\hat{\chi}=1_0 \Rightarrow \hat{\chi}=1_0,
			\end{eqn}
			since $f'a\simeq bf$ and since $f'$ and $a$ are 0-faithful.  So
			$\chi=f\hat{\chi}=1_0$.
			
			{\it Property 3. } Let us assume that $a$ and $b$ are fully 0-faithful.
			Let be $X\col\C$, $x\col X\ra B$ and $\chi\col bx\Ra 0$.  Since $c$ is fully
			0-faithful, there exists $\chi'\col gx\Ra 0$ such that
			\begin{eqn}
				g'\chi\circ\psi x=c\chi'.
			\end{eqn}
			Then, since $(f,\eta)=\Ker g$, there exist $x'\col X\ra A$,
			$\chi''\col x\Ra fx'$ such that
			\begin{eqn}
				\eta x'\circ g\chi''=\chi'.
			\end{eqn}
			Then $\chi\circ b\chi''^{-1}\circ\varphi^{-1}x'$ is compatible
			with $\eta'$ and, since $f'$ is $\eta'$-fully faithful, 
			there exists $\bar{\chi}\col ax'\Ra 0$ such that $f'\bar{\chi}$ is equal
			to this composite.
			
			Since $a$ is fully 0-faithful, there exists $\check{\chi}\col x'\Ra 0$
			such that $\bar{\chi}=a\check{\chi}$.  Finally we have $\hat{\chi}
			\eqdef f\check{\chi}\circ\chi''$ such that $b\hat{\chi}=\chi$.
			
			Unicity follows from property 1.
		\end{proof}

	\begin{pon}[Short five lemma]\label{petitlemmcinq}\index{lemma!short five}%
	\index{short five lemma}\index{five, short … lemma}
		In a Puppe-exact $\Gpdp$-category, in the situation of diagram
		\ref{diagpetitlemmcinq}, if $a$ and $c$ are equivalences, then $b$ is an
		equivalence.
	\end{pon}
	
		\begin{proof}
			Since every arrow in a Puppe-exact $\Gpdp$-category which is both
			fully 0-faithful and fully 0-cofaithful is an equivalence
			(by Proposition \ref{caracpuppex}), this follows from the previous lemma.
		\end{proof}
	
	The short five lemma has been proved for symmetric 2-groups by Dominique Bourn
	and Enrico Vitale \cite[Proposition 2.8]{Bourn2002a}.
	We will prove (Proposition \ref{petitlemmcinqraffine}) a refined version of the short
	5 lemma for good 2-Puppe-exact $\Gpdp$-categories.
	
\subsection{Triangle lemma and two-square lemma}

	\begin{pon}[Generalised kernels lemma]\label{lemmnoygen}%
	\index{lemma!kernels!generalised}\index{kernels lemma!generalised}
		Let $\C$ be an abelian $\Gpd$-category.
		Let us consider the situation of diagram \ref{diagtyptroitroi}.  Let us assume that
		the diagram commutes and that the following conditions hold:
		\begin{enumerate}
			\item $(a_1,\alpha_1)=\Ker a_2$;
			\item $(b_1,\beta_1)=\Ker b_2$;
			\item $c_1$ is $\gamma_1$-fully 0-faithful;
			\item the sequence $(f_2,\eta_2,g_2)$ is exact;
			\item $(f_3,\eta_3)=\Ker g_3$.
		\end{enumerate}
		Then the sequence $(f_1,\eta_1,g_1)$ is exact.
	\end{pon}
	
		\begin{proof}
			We construct the following diagram, where:
			\begin{enumerate}
				\item $(k_2,\kappa_2)=\Ker g_2$;
				\item by the universal property of the kernel of $g_3$, there exist
					$\hat{a}_2$ and $\hat{\varphi}_2$ such that
					$\eta_3\hat{a}_2\circ g_3\hat{\varphi}_2\circ\psi_2k_2=c_2\kappa_2$;
				\item $(\hat{a}_1,\hat{\alpha}_1)=\Ker\hat{a}_2$;
				\item by the universal property of the kernel of $b_2$, there exist
					$k_1$ and $\hat{\varphi}_1$ such that $f_3\hat{\alpha}_1\circ
					\hat{\varphi}_2\hat{a}_1\circ b_2\hat{\varphi}_1 = \beta_1k_1$;
				\item since $c_1$ is $\gamma_1$-fully 0-faithful, there exists
					$\kappa_1$ such that $\kappa_2\hat{a}_1\circ g_2\hat{\varphi}_1
					\circ\psi_1k_1=c_1\kappa_1$ (so the right two thirds commute);
				\item by the universal property of the kernel of $g_2$, there exist
					$d_2$ and $\omega_2$ such that $\eta_2\circ g_2\omega_2=\kappa_2d_2$;
				\item since $f_3$ is $\eta_3$-fully faithful, there exists
					$\delta_2$ such that $f_3\delta_2\circ\hat{\varphi}_2d_2
					=\varphi_2\circ b_2\omega_2$;
				\item by the universal property of the kernel of $\hat{a}_2$, we
					have $d_1$ and $\delta_1$ such that $\alpha_1\circ\delta_2 a_1
					\circ \hat{a}_2\delta_1 = \hat{\alpha}_1d_1$;
				\item since $b_1$ is $\beta_1$-fully faithful, there exists
					$\omega_1$ such that $\omega_2a_1\circ k_2\delta_1\circ\hat{\varphi}_1d_1
					=b_1\omega_1$.
			\end{enumerate}
			\begin{xym}\xymatrix@=40pt{
				A_1\ar[r]^{d_1}\ar[d]_{a_1}\rruppertwocell<9>^{f_1}{^<-2.7>\omega_1\;}
					\drtwocell\omit\omit{_{\;\,\delta_1}}
					\ddlowertwocell<-9>_0{_<2.7>\alpha_1}
				&K_1\ar[r]^{k_1}\ar[d]_{\hat{a}_1}\rruppertwocell<9>^0{^<-2.7>\kappa_1\;}
					\drtwocell\omit\omit{_{\;\,\hat{\varphi}_1}}
					\ar@/_1.89pc/[dd]_(0.44)0
				&B_1\ar[r]^{g_1}\ar[d]^(0.4){b_1}\ar@/^1.89pc/[dd]^(0.56)0
					\drtwocell\omit\omit{_{\;\,\psi_1}}
				&C_1\ar[d]^{c_1}\dduppertwocell<9>^0{^<-2.7>\gamma_1}
				\\ A_2\ar[r]_(0.4){d_2}\ar[d]_{a_2}\ar@/_1.87pc/[rr]_(0.55){f_2}
					\drtwocell\omit\omit{_{\;\,\delta_2}}
				&K_2\ar[r]|{k_2}\ar[d]_(0.6){\hat{a}_2}\ar@{}[d]^(0.25){\dir{=>}\,\omega_2}
					\drtwocell\omit\omit{_{\;\,\hat{\varphi}_2}}\ar@/^1.87pc/[rr]^(0.45)0
					\ar@{}[l]_(0.25){\txt{$\overset{\;\;\hat{\alpha}_1}{\dir{=>}}$}}
				&B_2\ar[r]^(0.6){g_2}\ar[d]^{b_2}\drtwocell\omit\omit{_{\;\,\psi_2}}
					\ar@{}[u]^(0.25){\kappa_2\,\dir{=>}}
					\ar@{}[r]_(0.25){\txt{$\underset{\beta_1\;\;}{\dir{=>}}$}}
				&C_2\ar[d]^{c_2}
				\\ A_3\ar@{=}[r]
				&A_3\ar[r]_{f_3}\rrlowertwocell<-9>_0{_<2.7>\;\,\eta_3}
				&B_3\ar[r]_{g_3}
				&C_3
			}\end{xym}
			
			Then the restricted kernels lemmma (Proposition \ref{lemmnoyrestr})
			applies to the right two thirds of the above diagram and it follows that
			$(k_1,\kappa_1)=\Ker g_1$.  Now, since the sequence $(f_2,\eta_2,g_2)$
			is exact, $d_2$ is fully 0-cofaithful, by Corollary
			\ref{caracdexbis}.  Moreover, by Lemma \ref{lemmun}, $\delta_1$
			is a pullback.  Therefore, since fully 0-cofaithful arrows 
			are stable under pullback (since $\C$ is abelian),
			$d_1$ is fully 0-cofaithful and, by Corollary
			\ref{caracdexbis}, the sequence $(f_1,\eta_1,g_1)$ is exact.
		\end{proof}
		
	To prove the snake lemma, we will follow the 1-dimensional proof of F. Rudolf Beyl \cite{Beyl1979a} and Temple H. Fay, Keith A. Hardie and Peter J. Hilton
	\cite{Fay1989a}.  It is based on two lemmas: the triangle lemma (also called
	“Produktlemma”), and the two-square lemma.
		
	\begin{pon}[Triangle lemma]\label{lemmtriangle}\index{triangle lemma}%
	\index{lemma!triangle}
		Let $\C$ be an abelian $\Gpd$-category.
		Let us consider the following diagram, where for $i=1$, $2$ or $3$ the 2-arrows
		$\kappa_i\col f_ik_i\Ra 0$ and $\zeta_i\col q_if_i\Ra 0$ are not shown.
		\begin{xym}\xymatrix@R=20pt@C=40pt{
			&K_2\ar[dd]^{k_2}\ar[ddrr]^{m_2}
			\\ K_1\ar[dr]_{k_1}\ar[ur]^{m_1}\drtwocell\omit\omit{^<-4.5>\mu_1\;\,}
			&{}\ddrtwocell\omit\omit{^<-4>\mu_2\;\;}
			\\ &A\ar[dd]_{f_2}\ar[dr]^{f_1}\ddtwocell\omit\omit{^<-4>\omega}
			&&K_3\ar[dl]_{k_3}\ar[dd]^d
			\\ &&B\ar[dl]^{f_3}\ar[dr]_{q_1}\ddltwocell\omit\omit{^<-4>\,\nu_1}
			\\ &C\ar[dl]_{q_3}\ar[dd]^{q_2}\dltwocell\omit\omit{^<-4.5>\;\;\nu_2}
			&& Q_1\ar[ddll]^{n_1}
			\\ Q_3 &{}
			\\ &Q_2\ar[ul]^{n_2}
		}\end{xym}
		Let us assume that the following conditions hold:
		\begin{enumerate}
			\item for all $i=1$, $2$ or $3$, $(k_i,\kappa_i)=\Ker f_i$;
			\item for all $i=1$, $2$ or $3$, $(q_i,\zeta_i)=\Coker f_i$;
			\item $f_3\kappa_1\circ \omega k_1 = \kappa_2 m_1\circ f_2\mu_1$;
			\item $\kappa_2 = \kappa_3 m_2\circ f_3\mu_2\circ\omega k_2$;
			\item $n_1\zeta_1\circ\nu_1f_1\circ q_2\omega = \zeta_2$;
			\item $n_2\zeta_2\circ \nu_2 f_2=\zeta_3 f_1\circ q_3\omega$.
		\end{enumerate}
		Then, if we set $d\eqdef q_1k_3$, there exist 2-arrows
		$\mu$, $\delta$, $\delta'$ and $\nu$ such that the following sequence
		is exact.
		\begin{xym}\xymatrix@=40pt{
			K_1\ar[r]^-{m_1}\rrlowertwocell<-9>_0{_<2.7>\,\mu}
			&K_2\ar[r]^-{m_2}\rruppertwocell<9>^0{^<-2.7>\delta}
			&K_3\ar[r]^-{d}\rrlowertwocell<-9>_0{_<2.7>\;\delta'}
			&Q_1\ar[r]^-{n_1}\rruppertwocell<9>^0{^<-2.7>\nu}
			&Q_2\ar[r]^-{n_2}
			&Q_3
		}\end{xym}
		Moreover, the following conditions hold:
		\begin{enumerate}
			\item $(m_1,\mu)=\Ker m_2$;
			\item $(n_2,\nu)=\Coker n_1$;
			\item $\delta m_1\circ d\mu^{-1}= \zeta_1k_1\circ q_1\kappa_1^{-1}$;
			\item $\delta' m_2\circ n_1\delta^{-1}= q_2\kappa_2\circ \zeta_2^{-1}k_2$;
			\item $\nu d\circ n_2\delta'^{-1}= \zeta_3k_3\circ q_3\kappa_3^{-1}$.
		\end{enumerate}
	\end{pon}
	
		\begin{proof}
			Since $k_3$ is $\kappa_3$-fully faithful, we get $\mu$
			such that $k_3\mu\circ\mu_2 m_1\circ f_1\mu_1 = \kappa_1$. Dually,
			since $q_1$ is $\zeta_1$-fully cofaithful, we have $\nu$
			such that $\nu q_1\circ n_2\nu_1\circ\nu_2 f_3 = \zeta_3$.  Moreover,
			we set $\delta\eqdef \zeta_1 k_2\circ q_1\mu_2^{-1}$
			and $\delta'\eqdef q_2\zeta_3\circ \nu_1^{-1}k_3$.
			
			We apply then the generalised kernels lemma (Proposition \ref{lemmnoygen})
			to the following diagram, which allows us to conclude that the sequence $(m_2,\delta,d)$
			is exact. Dually, the sequence $(d,\delta',n_1)$ is exact.
			\begin{xym}\xymatrix@=40pt{
				K_2\ar[r]^{m_2}\ar[d]_{k_2}\rruppertwocell<9>^0{^<-2.7>\delta\,}
					\drtwocell\omit\omit{^{\mu_2\;\,}}
					\ddlowertwocell<-9>_0{_<2.7>\kappa_2}
				&K_3\ar[r]^{d}\ar[d]^{k_3}\ar@/^1.89pc/[dd]^(0.44)0
				&Q_1\ar@{=}[d]
				\\ A\ar[r]_(0.44){f_1}\ar[d]_{f_2}\ar@/_1.87pc/[rr]_(0.45)0
					\drtwocell\omit\omit{^{\omega}}
				&B\ar[r]_(0.6){q_1}\ar[d]^(0.62){f_3}
					\ar@{}[d]_(0.25){\zeta_1\,\dir{=>}}
					\ar@{}[r]^(0.25){\txt{$\overset{\kappa_3\;\;}{\dir{=>}}$}}
				&Q_1\ar[d]^0
				\\ C\ar@{=}[r]
				&C\ar[r]_0
				&0
			}\end{xym}
			
			Next, by Lemma \ref{lemmun}, $\mu_2$ is a pullback. So,
			by Lemma \ref{lemmtwo}, $(m_1,\mu)=\Ker m_2$, which proves 
			exactness at $K_2$. Dually, $(n_2,\nu)=\Coker n_1$, which proves
			exactness at $Q_2$.
			
			Finally, we easily check equations 3 to 5 using the conditions
			which the involved 2-arrows satisfy.
		\end{proof}
	
	Here is now an ad hoc 2-dimensional version of the two-square lemma
	of Fay, Hardie and Hilton
	\cite{Fay1989a}, which is stated without name by Beyl \cite{Beyl1979a}, and was already
	in Mitchell's book \cite[section VII.1]{Mitchell1965a}.
	
	\begin{pon}[Two-square lemma]\index{two-square lemma}%
	\index{lemma!two-square}
		Let $\C$ be an abelian $\Gpd$-category.  Let us consider
		diagram \ref{diagpetitlemmcinq} in $\C$.  This morphism
		of extensions factors as in the following diagram, where
		\begin{enumerate}
			\item $(i,\iota,j)$ is an extension;
			\item $\varphi_1$ is a pushout and a pullback;
			\item $\psi_2$ is a pullback and a pushout;
			\item $\iota a\circ j\varphi_1\circ\psi_1f=\eta$;
			\item $\eta'\circ g'\varphi_2\circ\psi_2 i=c\iota$;
			\item $\varphi\circ\beta f=\varphi_2a\circ c'\varphi_1$;
			\item $g'\beta\circ\psi_2 a'\circ c\psi_1=\psi$.
		\end{enumerate}
		\begin{xym}\label{diagtwocarres}\xymatrix@=40pt{
			0\ar[r]
			&A\ar[r]^{f}\ar[d]_{a}\rruppertwocell<9>^0{^<-2.7>\eta\,}
				\drtwocell\omit\omit{_{\;\,\varphi_1}}
			&B\ar[r]^{g}\ar[d]^{a'}\ar@/^1.89pc/[dd]^(0.44)b
				\drtwocell\omit\omit{_{\;\,\psi_1}}
			&C\ar@{=}[d]\ar[r]
			&0
			\\ 0\ar[r]
			&A'\ar[r]_{i}\ar@{=}[d]\ar@/_1.87pc/[rr]_(0.45)0
				\drtwocell\omit\omit{_{\;\,\varphi_2}}
			&I\ar[r]_(0.62){j}\ar[d]^(0.62){c'}\drtwocell\omit\omit{_{\;\,\psi_2}}
				\ar@{}[d]_(0.25){\iota\,\dir{=>}}
				\ar@{}[r]^(0.25){\txt{$\overset{\beta\;\;}{\dir{=>}}$}}
			&C\ar[d]^{c}\ar[r]
			&0
			\\ 0\ar[r]
			&A'\ar[r]_{f'}\rrlowertwocell<-9>_0{_<2.7>\;\,\eta'}
			&B'\ar[r]_{g'}
			&C'\ar[r]
			&0
		}\end{xym}
	\end{pon}
	
		\begin{proof}
			First, we construct the pushout $(I,a',i,\varphi_1)$, which
			induces an arrow $j$ and 2-arrows $\iota$ and $\psi_1$ satisfying
			condition 4, as well as an arrow $c'$ and
			2-arrows $\beta$ and $\varphi_2$ satisfying condition 6.
			The universal property of the pushout
			gives also a 2-arrow $\psi_2$ satisfying conditions 5 and 7.
			
			By the dual of Lemma \ref{lemmtwo}, since $(g,\eta)=\Coker f$, we also have
			$(j,\iota)=\Coker i$.  Since normal faithful arrows are stable
			under pushout (because $\C$ is abelian), the arrow
			$i$ is normal faithful, because $f$ is.  So the sequence $(i,\iota, j)$ is an
			extension.
			
			Next, we construct the following diagram.  First, $(P,p_1,p_2,\pi)$
			is a pullback, which induces an arrow $s$ and 2-arrows $\sigma_1$
			and $\sigma_2$ such that $\eta'\circ g'\sigma_2\circ\pi s=c\sigma_1$, as well
			as an arrow $w$ and 2-arrows $\omega_1$ and $\omega_2$ such that
			$g'\omega_2\circ \pi w\circ c\omega_1 = \psi_2$.  We also have a 2-arrow
			$\omega$ such that $\varphi_2\circ\omega_2 i=\sigma_2\circ p_2\omega$
			and $\sigma_1\circ p_1\omega\circ\omega_1 i =\iota$.
			\begin{xym}\xymatrix@=40pt{
				0\ar[r]
				&A\ar[r]^{f}\ar[d]_{a}\rruppertwocell<9>^0{^<-2.7>\eta\;}
					\drtwocell\omit\omit{_{\;\,\varphi_1}}
				&B\ar[r]^{g}\ar[d]_(0.38){a'}
					\drtwocell\omit\omit{_{\;\,\psi_1}}
				&C\ar@{=}[d]\ar[r]
				&0
				\\ 0\ar[r]
				&A'\ar[r]^{i}\ar@{=}[d]\ar@/^1.87pc/[rr]^(0.55)0
					\drtwocell\omit\omit{_{\;\,\omega}}
				&I\ar[r]^{j}\ar[d]^{w}\drtwocell\omit\omit{_{\;\;\omega_1}}
					\ar@{}[u]_(0.25){\,\dir{=>}\,\iota}\ar@/^1.89pc/[dd]^(0.44){c'}
				&C\ar@{=}[d]\ar[r]
				&0
				\\ 0\ar[r]
				&A'\ar[r]_{s}\ar@{=}[d]\ar@/_1.87pc/[rr]_(0.45)0
					\drtwocell\omit\omit{_{\;\,\sigma_2}}
				&P\ar[r]_(0.6){p_1}\ar[d]^(0.62){p_2}\drtwocell\omit\omit{_{\;\,\pi}}
					\ar@{}[d]_(0.25){\sigma_1\,\dir{=>}}
					\ar@{}[r]^(0.25){\txt{$\overset{\omega_2\;\;}{\dir{=>}}$}}
				&C\ar[d]^{c}\ar[r]
				&0
				\\ 0\ar[r]
				&A'\ar[r]_{f'}\rrlowertwocell<-9>_0{_<2.7>\;\,\eta'}
				&B'\ar[r]_{g'}
				&C'\ar[r]
				&0
			}\end{xym}
			
			Then, by Lemma \ref{lemmtwo}, $(s,\sigma_1)=\Ker p_1$. And, since
			normal cofaithful arrows are stable under pullback
			(since $\C$ is abelian), $p_1$ is
			a normal cofaithful arrow. So the sequence $(s,\sigma_1,p_1)$ is
			an extension. Therefore, by the short 5 lemma
			(Proposition \ref{petitlemmcinq}),
			$w$ is an equivalence and so $\psi_2$ is a pullback.
			Finally, by Lemma \ref{lemmtrois}, $\varphi_1$ is also a pullback
			and $\psi_2$ is also a pushout.
		\end{proof}
	
\subsection{Ophiology}

	To prove the snake lemma\index{lemma!snake}\index{snake lemma},
	we will follow the proofs of
	\cite{Beyl1979a} and \cite{Fay1989a}.  The first version we give is a short one (the snake
	lemma).  The exact sequence we get here
	cannot be extended by $0$ at each ends, because there is no reason in general
	for the kernel of $\bar{f}$ to be $0$ or for the cokernel of $\bar{g}'$
	to be $0$.

	\begin{pon}[Snake lemma]\label{lemmserp}%
	\index{lemma!snake}\index{snake lemma}
		Let $\C$ be an abelian $\Gpd$-category.
		Let us consider the following commutative diagram in $\C$.
		\begin{xym}\label{diagserp}\xymatrix@=40pt{
			Ka\ar@{-->}[r]^{\bar{f}}\ar[d]_{k_a}\rruppertwocell<9>^0{^<-2.7>\bar{\eta}\;}
				\drtwocell\omit\omit{_{\;\,}}
				\ddlowertwocell<-9>_0{_<2.7>\kappa_a}
			&Kb\ar@{-->}[r]^{\bar{g}}\ar[d]_(0.38){k_b}\ar@/_1.89pc/[dd]_(0.56)0
				\drtwocell\omit\omit{_{\;\,}}\ar@{}[dl]_{\bar{\varphi}}
			&Kc\ar[d]^{k_c}\dduppertwocell<9>^0{^<-2.7>\kappa_c}\ar@{-->}[dddll]_(0.1){d}
				\ar@{}[dl]_{\bar{\psi}}
			\\ A\ar[r]^(0.4){f}\ar[d]_{a}\ar@/^1.87pc/[rr]^(0.55)0
				\drtwocell\omit\omit{_{\;\,\varphi}}
				\ddlowertwocell<-9>_0{_<2.7>\;\,\zeta_a}
			&B\ar[r]^(0.6){g}\ar[d]^(0.65){b}\drtwocell\omit\omit{_{\;\,\psi}}
				\ar@{}[u]_(0.25){\,\dir{=>}\,\eta}\ar@/^1.89pc/[dd]^(0.44)0
				\ar@{}[l]^(0.25){\txt{$\underset{\;\;\kappa_b}{\dir{=>}}$}}
			&C\ar[d]^{c}\dduppertwocell<9>^0{^<-2.7>\zeta_c}
			\\ A'\ar[r]_(0.4){f'}\ar[d]_{q_a}\ar@/_1.87pc/[rr]_(0.45)0
				\drtwocell\omit\omit{_{\;\,\bar{\varphi}'}}
			&B'\ar[r]_(0.6){g'}\ar[d]^(0.62){q_b}\drtwocell\omit\omit{_{\;\,\bar{\psi}'}}
				\ar@{}[d]_(0.25){\eta'\,\dir{=>}}
				\ar@{}[r]^(0.25){\txt{$\overset{\zeta_b\;\;}{\dir{=>}}$}}
			&C'\ar[d]^{q_c}
			\\ Qa\ar@{-->}[r]_{\bar{f}'}\rrlowertwocell<-9>_0{_<2.7>\;\,\bar{\eta}'}
			&Qb\ar@{-->}[r]_{\bar{g}'}
			&Qc
		}\end{xym}
		Let us assume that the following conditions hold:
		\begin{enumerate}
			\item $(f,\eta,g)$ is an extension;
			\item $(f',\eta',g')$ is an extension;
			\item $(k_a,\kappa_a)=\Ker a$;
			\item $(k_b,\kappa_b)=\Ker b$;
			\item $(k_c,\kappa_c)=\Ker c$;
			\item $(q_a,\zeta_a)=\Coker a$;
			\item $(q_b,\zeta_b)=\Coker b$;
			\item $(q_c,\zeta_c)=\Coker c$.
		\end{enumerate}
		Then there exist an arrow $d\col Kc\ra Qa$ and 2-arrows $\delta$ and $\delta'$
		such that the following sequence is exact.
		\begin{xym}\xymatrix@=40pt{
			Ka\ar[r]^-{\bar{f}}\rrlowertwocell<-9>_0{_<2.7>\,\bar{\eta}}
			&Kb\ar[r]^-{\bar{g}}\rruppertwocell<9>^0{^<-2.7>\delta}
			&Kc\ar[r]^-{d}\rrlowertwocell<-9>_0{_<2.7>\;\delta'}
			&Qa\ar[r]^-{\bar{f}'}\rruppertwocell<9>^0{^<-2.7>\bar{\eta}'}
			&Qb\ar[r]^-{\bar{g}'}
			&Qc
		}\end{xym}
		Moreover, the following conditions (where $\mu_x\eqdef\zeta_x k_x\circ q_x\kappa_x^{-1}$)
		hold:
		\begin{enumerate}
			\item $(\bar{f},\bar{\eta})=\Ker \bar{g}$;
			\item $(\bar{g}',\bar{\eta}')=\Coker \bar{f}'$;
			\item $\delta\bar{f}\circ d\bar{\eta}^{-1}= \mu_a$;
			\item $\delta' \bar{g}\circ \bar{f}'\delta^{-1} = \mu_b^{-1}$;
			\item $\bar{\eta}' d\circ \bar{g}'\delta'^{-1}=  \mu_c$.
		\end{enumerate}
	\end{pon}
	
		\begin{proof}
			First, by the two-square lemma applied to the two central rows of 
			diagram \ref{diagserp}, we get the situation of diagram
			\ref{diagtwocarres}, where $\varphi_1$ and $\psi_2$ are pullbacks and
			pushouts, and where the central row is an extension.
			
			By the universal property of the pullback $\psi_2$, there exist an arrow
			$k'_c$ and 2-arrows $\kappa'_c$ and
			$\xi$ such that $g'\kappa'_c\circ\psi_2 k'_c\circ c\xi
			=\kappa_c$.  By Lemma \ref{lemmtwo} applied to the following diagram, 
			$(k'_c,\kappa'_c)=\Ker c'$.
			\begin{xym}\xymatrix@=40pt{
				Kc\ar[r]^{k'_c}\ar@{=}[d]\rruppertwocell<9>^0{^<-2.7>\kappa'_c\,}
					\drtwocell\omit\omit{^\xi}
				&I\ar[r]^{c'}\ar[d]_{j}\drtwocell\omit\omit{^{\psi_2\;\;}}
				&B'\ar[d]^{g'}
				\\ Kc\ar[r]_{k_c}\rrlowertwocell<-9>_0{_<2.7>\;\;\kappa_c}
				&C\ar[r]_{c}
				&C'
			}\end{xym}
			Moreover, we have a 2-arrow $\xi'\col a'k_b\Ra k'_c\bar{g}$
			such that $\kappa'_c\bar{g}\circ c'\xi'=\kappa_b\circ \beta k_b$
			and $j\xi'\circ \psi_1k_b\circ\bar{\psi}=\xi\bar{g}$.
			
			Dually, there is $q'_a$, $\zeta'_a$ and $\pi$ such that 
			$\zeta_a\circ \pi a\circ q'_a\varphi_1=\zeta'_a f$.
			By the dual of Lemma \ref{lemmtwo} applied to the following diagram,
			$(q'_a,\zeta'_a)=\Coker a'$.
			\begin{xym}\xymatrix@=40pt{
				A\ar[r]^-{a}\ar[d]_f\rruppertwocell<9>^0{^<-2.7>\zeta_a\,}
					\drtwocell\omit\omit{^\varphi_1\;}
				&A'\ar[r]^-{q_a}\ar[d]_{i}\drtwocell\omit\omit{^{\pi}}
				&Qa\ar@{=}[d]
				\\ B\ar[r]_-{a'}\rrlowertwocell<-9>_0{_<2.7>\;\;\zeta'_a}
				&I\ar[r]_-{q'_a}
				&Qa
			}\end{xym}
			Moreover, we have $\pi'\col q_bc'\Ra\bar{f}' q'_a$ such that
			$\bar{f}'\zeta'_a\circ\pi' a'=\zeta_b\circ q_b\beta$ and 
			$\bar{\varphi}'\circ q_b\varphi_2=\bar{f}'\pi\circ \pi' i$.
			
			Next, as $\varphi_1$ is a pullback, by Lemma \ref{lemmtwo},
			$(fk_a,i\kappa_a\circ \varphi_1k_a)=\Ker a'$.  Dually,
			$(q_cg',\zeta_c j\circ q_c\psi_2^{-1})=\Coker c'$.
			
			We apply then the triangle lemma (Proposition \ref{lemmtriangle})
			to the following diagram.  We get the required exact sequence; properties 
			3 to 5 follow from properties 3 to 5 of the triangle lemma.\qedhere
			\begin{xym}\xymatrix@R=20pt@C=40pt{
				&Kb\ar[dd]^{k_b}\ar[ddrr]^{\bar{g}}
				\\ Ka\ar[dr]_{fk_a}\drtwocell\omit\omit{^<-4.5>\bar{\varphi}^{-1}\;\,}
					\ar[ur]^{\bar{f}}
				&{}\ddrtwocell\omit\omit{^<-4>\xi'}
				\\ &B\ar[dd]_{b}\ar[dr]^{a'}\ddtwocell\omit\omit{^<-4>\;\beta^{-1}}
				&&Kc\ar[dl]_{k'_c}\ar[dd]^d
				\\ &&I\ar[dl]^{c'}\ar[dr]_{q'_a}\ddltwocell\omit\omit{^<-4>\;\;\pi'}
				\\ &B'\ar[dl]_{q_cg'}\ar[dd]^{q_b}\dltwocell\omit\omit{^<-4.5>\;\;\bar{\psi}'}
				&& Qa\ar[ddll]^{\bar{f}'}
				\\ Qc &{}
				\\ &Qb\ar[ul]^{\bar{g}'}
			}\end{xym}
		\end{proof}

	If we extend the columns of diagram \ref{diagserp} at both ends by $0$,
	we have neither relative exact sequences ($\kappa_x$ and $\zeta_x$ are not in general
	compatible) nor exact sequences (the kernel of $k_x$ is in general not $0$).
	We get exact sequences by extending the columns in such a way to get the Puppe
	exact sequence (Proposition \ref{sequencePuppe}).
	In this way appears a longer snake (anaconda), which is a
	$\Pip$-$\Ker$-$\Coker$-$\Copip$ exact sequence.
	
	\begin{pon}[Anaconda lemma]\label{lemmanacond}\index{anaconda lemma}\index{lemma!anaconda}
		Let $\C$ be an abelian $\Gpd$-category.
		In the situation of the following diagram (where we omit the 2-arrows to $0$),
		where the two central rows are extensions, the dashed sequence
		(equipped with omitted 2-arrows) is exact.  Moreover, the composite
		$0\Ra 0$ of the adjacent 2-arrows of this sequence are successively (up to the
		sign): $\omega_{\Ker a}$, $\omega_{\Ker b}$, $\omega_{\Ker c}$, $\mu_a$,
		$\mu_b$, $\mu_c$, $\sigma_{\Coker a}$, $\sigma_{\Coker b}$
		and $\sigma_{\Coker c}$.
		\begin{xym}\xymatrix@=40pt{
			&0\ar[d] &0\ar[d] &0\ar[d]
			\\ 0\ar@{-->}[r]
			&\Pip a\ar@{-->}[r]^{\Omega \bar{f}}\ar[d]_{\Omega k_a}\drtwocell\omit\omit{}
			&\Pip b\ar@{-->}[r]^{\Omega \bar{g}}\ar[d]^{\Omega k_b}\drtwocell\omit\omit{}
			&\Pip c\ar[d]^{\Omega k_c}\ar@{-->}[dddll]_(0.1){d'}
			\\ &\Omega A\ar[r]^{\Omega f}\ar[d]_{\Omega a}
				\drtwocell\omit\omit{}
			&\Omega B\ar[r]^(0.6){\Omega g}\ar[d]^(0.65){\Omega b}\drtwocell\omit\omit{}
			&\Omega C\ar[d]^{\Omega c}
			\\ &\Omega A'\ar[r]^(0.4){\Omega f'}\ar[d]_{d_a}
				\drtwocell\omit\omit{}
			&\Omega B'\ar[r]^{\Omega g'}\ar[d]^{d_b}\drtwocell\omit\omit{}
			&\Omega C'\ar[d]^{d_c}
			\\ &\Ker a\ar@{-->}[r]^{\bar{f}}\ar[d]_{k_a}\drtwocell\omit\omit{}
			&\Ker b\ar@{-->}[r]^{\bar{g}}\ar[d]^{k_b}\drtwocell\omit\omit{}
			&\Ker c\ar[d]^{k_c}\ar@{-->}[dddll]_(0.1){d}
			\\ 0\ar[r] &A\ar[r]^{f}\ar[d]_{a}
				\drtwocell\omit\omit{}
			&B\ar[r]^(0.6){g}\ar[d]^(0.65){b}\drtwocell\omit\omit{}
			&C\ar[d]^{c}\ar[r] &0
			\\ 0\ar[r] &A'\ar[r]_(0.4){f'}\ar[d]_{q_a}
				\drtwocell\omit\omit{}
			&B'\ar[r]_{g'}\ar[d]^{q_b}\drtwocell\omit\omit{}
			&C'\ar[d]^{q_c}\ar[r] &0
			\\ &\Coker a\ar@{-->}[r]_{\bar{f}'}\ar[d]_{d'_a}\drtwocell\omit\omit{}
			&\Coker b\ar@{-->}[r]_{\bar{g}'}\ar[d]^{d'_b}\drtwocell\omit\omit{}
			&\Coker c\ar[d]^{d'_c}\ar@{-->}[dddll]_(0.1){d''}
			\\ &\Sigma A\ar[r]_{\Sigma f}\ar[d]_{\Sigma a}
				\drtwocell\omit\omit{}
			&\Sigma B\ar[r]_(0.6){\Sigma g}\ar[d]^(0.65){\Sigma b}\drtwocell\omit\omit{}
			&\Sigma C\ar[d]^{\Sigma c}
			\\ &\Sigma A'\ar[r]_(0.4){\Sigma f'}\ar[d]_{\Sigma q_a}
				\drtwocell\omit\omit{}
			&\Sigma B'\ar[r]_{\Sigma g'}\ar[d]^{\Sigma q_b}\drtwocell\omit\omit{}
			&\Sigma C'\ar[d]^{\Sigma q_c}
			\\ &\Copip a\ar@{-->}[r]_{\Sigma \bar{f}'}\ar[d]
			&\Copip b\ar@{-->}[r]_{\Sigma \bar{g}'}\ar[d]
			&\Copip c\ar@{-->}[r]\ar[d]
			&0
			\\ &0 &0 &0
		}\end{xym}
	\end{pon}

		\begin{proof}
			We join the snake lemma (Proposition \ref{lemmserp}) and the 
			Puppe exact sequence (Proposition \ref{sequencePuppe}) constructed on the left
			side starting from $\bar{g}\col \Ker b\ra\Ker c$ (by using the fact
			that $(\bar{f},\bar{\eta})=\Ker\bar{g}$), and the Puppe sequence constructed
			on the right side starting from $\bar{f}'\col\Coker a\ra\Coker b$ (by using
			the fact that $(\bar{g}',\bar{\eta}')=\Coker\bar{f}'$).
		\end{proof}

	To construct the long exact sequence of homology, we need a more general version
	of the snake lemma, where the second row of diagram \ref{diagserp} is not left relative
	exact any more, and the third row is not right relative exact any more.  Then, in the conclusion, we lose the facts that $\bar{f}=\Ker\bar{g}$
	and $\bar{g}'=\Coker\bar{f}'$.

	\begin{pon}[Generalised snake lemma]\label{lemmserpgen}%
	\index{snake lemma!generalised}\index{lemma!snake!generalised}
		Let $\C$ be an abelian $\Gpd$-category.
		Let us consider the commutative diagram \ref{diagserp}.
		Let us assume that the following conditions hold:
		\begin{enumerate}
			\item $(g,\eta)=\Coker f$;
			\item $(f',\eta')=\Ker g'$;
			\item $(k_a,\kappa_a)=\Ker a$;
			\item $(k_b,\kappa_b)=\Ker b$;
			\item $(k_c,\kappa_c)=\Ker c$;
			\item $(q_a,\zeta_a)=\Coker a$;
			\item $(q_b,\zeta_b)=\Coker b$;
			\item $(q_c,\zeta_c)=\Coker c$.
		\end{enumerate}
		Then there exist an arrow $d\col Kc\ra Qa$ and 2-arrows $\delta$ and $\delta'$
		such that the following sequence is exact.
		\begin{xym}\xymatrix@=40pt{
			Ka\ar[r]^-{\bar{f}}\rrlowertwocell<-9>_0{_<2.7>\,\bar{\eta}}
			&Kb\ar[r]^-{\bar{g}}\rruppertwocell<9>^0{^<-2.7>\delta}
			&Kc\ar[r]^-{d}\rrlowertwocell<-9>_0{_<2.7>\;\delta'}
			&Qa\ar[r]^-{\bar{f}'}\rruppertwocell<9>^0{^<-2.7>\bar{\eta}'}
			&Qb\ar[r]^-{\bar{g}'}
			&Qc
		}\end{xym}
		Moreover, the following conditions (where $\mu_x=\zeta_x k_x\circ q_x\kappa_x^{-1}$)
		hold:
		\begin{enumerate}
			\item $\delta\bar{f}\circ d\bar{\eta}^{-1}= \mu_a$;
			\item $\delta' \bar{g}\circ \bar{f}'\delta^{-1} = \mu_b^{-1}$;
			\item $\bar{\eta}' d\circ \bar{g}'\delta'^{-1}=  \mu_c$.
		\end{enumerate}
	\end{pon}
	
		\begin{proof}
			We factor the squares $\bar{\varphi}$ and $\varphi$ as in the proof
			of the generalised kernels lemma (Proposition \ref{lemmnoygen}), and we factor
			the squares $\psi$ and $\bar{\psi}'$ in a dual way.  We get the following
			diagram.
			\begin{xym}\xymatrix@=40pt{
				Ka\ar[r]^{\bar{m}}\ar[d]_{k_a}\drtwocell\omit\omit{_{\;\,}}
					\rruppertwocell<9>^{\bar{f}}{^<-2.7>\bar{\omega}\;}
					\ddlowertwocell<-9>_0{_<2.7>\kappa_a}
				&K\hat{a}\ar[r]^{\hat{\bar{f}}}\ar[d]_{k_{\hat{a}}}
					\drtwocell\omit\omit{_{\;\,}}
					\rruppertwocell<9>^0{^<-2.7>\hat{\bar{\eta}}\;}
					\ar@{}[dl]_{\bar{\mu}}\ar@/_1.89pc/[dd]_(0.44)0
				&Kb\ar[r]^{\bar{g}}\ar[d]_(0.38){k_b}\ar@/_1.89pc/[dd]_(0.44)0
					\drtwocell\omit\omit{_{\;\,}}\ar@{}[dl]_{\hat{\bar{\varphi}}}
				&Kc\ar@{=}[r]\ar[d]^{k_c}\ar@/^1.89pc/[dd]^(0.44)0
					\ar@{}[dl]_{\bar{\psi}}
				&Kc\ar[d]^(0.4){k_c}\dduppertwocell<9>^0{^<-2.7>\kappa_c}
				\\ A\ar[r]_(0.4){m}\ar[d]_{a}\ar@/_1.87pc/[rr]_(0.55)f
					\ddlowertwocell<-9>_0{_<2.7>\;\,\zeta_a}\drtwocell\omit\omit{_{\;\,\mu}}
				&\hat{A}\ar[r]_(0.4){\hat{f}}\ar[d]_(0.62){\hat{a}}\ar@/^1.87pc/[rr]^(0.55)0
					\drtwocell\omit\omit{_{\;\,\hat{\varphi}}}\ar@/_1.89pc/[dd]_(0.56)0
					\ar@{}[d]^(0.25){\dir{=>}\,\omega}
					\ar@{}[l]_(0.24){\txt{$\overset{\;\;\kappa_{\hat{a}}}{\dir{=>}}$}}
				&B\ar[r]^{g}\ar[d]^{b}\drtwocell\omit\omit{_{\;\,\hat{\psi}}}
					\ar@{}[u]_(0.25){\,\dir{=>}\,\hat{\eta}}\ar@/^1.89pc/[dd]^(0.56)0
					\ar@{}[l]_(0.25){\txt{$\overset{\;\;\kappa_b}{\dir{=>}}$}}
				&C\ar@{=}[r]\ar[d]^(0.38){\hat{c}}\ar@/^1.89pc/[dd]^(0.56)0
					\drtwocell\omit\omit{\nu}
					\ar@{}[r]^(0.25){\txt{$\overset{\kappa_{\hat{c}}\;\;}{\dir{=>}}$}}
				&C\ar[d]^{c}\dduppertwocell<9>^0{^<-2.7>\zeta_c}
				\\ A'\ar@{=}[r]\ar[d]_{q_a}
				&A'\ar[r]_{f'}\ar[d]_(0.6){q_a}\ar@/_1.87pc/[rr]_(0.45)0
					\drtwocell\omit\omit{_{\;\,\bar{\varphi}'}}
					\ar@{}[l]^(0.24){\txt{$\underset{\;\;\zeta_{\hat{a}}}{\dir{=>}}$}}
				&B'\ar[r]^(0.6){\hat{g}'}\ar[d]^(0.62){q_b}
					\drtwocell\omit\omit{_{\;\,\hat{\bar{\psi}}'}}
					\ar@{}[d]_(0.25){\hat{\eta}'\,\dir{=>}}\ar@/^1.87pc/[rr]^(0.45){g'}
					\ar@{}[r]_(0.24){\txt{$\underset{\zeta_b\;\;}{\dir{=>}}$}}
				&\hat{C}'\ar[r]^(0.6){n'}\ar[d]^{q_{\hat{c}}}\drtwocell\omit\omit{\bar{\nu}'}
					\ar@{}[u]^(0.25){\omega'\,\dir{=>}}
					\ar@{}[r]_(0.24){\txt{$\underset{\zeta_{\hat{c}}\;\;}{\dir{=>}}$}}
				&C'\ar[d]^{q_c}
				\\ Qa\ar@{=}[r]
				&Qa\ar[r]_{\bar{f}'}\rrlowertwocell<-9>_0{_<2.7>\;\,\hat{\bar{\eta}}'}
				&Qb\ar[r]_{\hat{\bar{g}}'}
					\rrlowertwocell<-9>_{\bar{g}'}{_<2.7>\;\,\bar{\omega}'}
				&Q\hat{c}\ar[r]_{\bar{\nu}'}
				&Qc
			}\end{xym}
			We apply the snake lemma to the two central squares
			(since $(\hat{f},\hat{\eta},g)$ and $(f',\hat{\eta}',\hat{g}')$
			are extensions), and we get an exact sequence
			\begin{xym}\xymatrix@=40pt{
				K\hat{a}\ar[r]^-{\hat{\bar{f}}}\rrlowertwocell<-9>_0{_<2.7>\,\hat{\bar{\eta}}}
				&Kb\ar[r]^-{\bar{g}}\rruppertwocell<9>^0{^<-2.7>\delta}
				&Kc\ar[r]^-{d}\rrlowertwocell<-9>_0{_<2.7>\;\delta'}
				&Qa\ar[r]^-{\bar{f}'}\rruppertwocell<9>^0{^<-2.7>\hat{\bar{\eta}}'}
				&Qb\ar[r]^-{\hat{\bar{g}}'}
				&Q\hat{c}.
			}\end{xym}
			Finally, by the generalised kernels lemma and its dual, the sequences 
			$(\bar{f},\bar{\eta},\bar{g})$ and $(\bar{f}',\bar{\eta}',\bar{g}')$ are exact.
		\end{proof}

\section{Long exact sequence of homology}

\subsection{Chain complexes}

	Chain complexes of symmetric 2-groups appear in \cite{Takeuchi1983a}
	and \cite{Rio2005a}.  In such a complex, adjacent 2-arrows have to be compatible; thus they are
	\emph{pre}relative exact sequences.

	\begin{df}\index{chain complex}\index{complex!chain}
		A \emph{chain complex} in $\C$ is a sequence
		$A_\bullet = (A_n,a_n,\allowbreak\alpha_n)_{n\col\mathbb{Z}}$,
		where $a_n\col A_n\ra A_{n+1}$ and $\alpha_n\col a_{n+1}a_n\Ra 0$, such that, for all $n$,
		the 2-arrows $\alpha_{n}$ and $\alpha_{n+1}$ are compatible.
	\end{df}
	
	The chain complexes in $\C$ form a $\Gpdp$-category.
	\begin{itemize}
		\item {\it Objects. } These are the chain complexes in $\C$.
		\item {\it Arrows. } A \emph{morphism of chain complexes}
			from $A_\bullet = (A_n,a_n,\alpha_n)_{n\col\mathbb{Z}}$ to $B_\bullet = (B_n,b_n,\beta_n)_{n\col\mathbb{Z}}$
			is a sequence $f_\bullet=(f_n,\varphi_n)$, where $f_n\col A_n\ra B_n$
			and $\varphi_n\col b_nf_n\Ra f_{n+1}a_n$, such that, for all $n\col\mathbb{Z}$,
			$f_{n+2}\alpha_n\circ\varphi_{n+1}a_n\circ b_{n+1}\varphi_n=\beta_n f_n$.
			We compose these morphisms in the obvious way.
		\item {\it 2-arrows. } A 2-morphism of chain complexes
			$(f_n,\varphi_n)\Ra (f'_n,\varphi'_n)$ is a sequence
			$(\gamma_n)_{n\col\mathbb{Z}}$, where $\gamma_n\col f_n\Ra f'_n$, such that,
			for all $n\col\mathbb{Z}$, $\gamma_{n+1}a_n\circ\varphi_n=
			\varphi'_n\circ b_n\gamma_n$.
		\item {\it Zero arrows. } The zero arrows are the arrows
			$(0,1_0)_{n\col\mathbb{Z}}$.
	\end{itemize}

	\vspace{1em}
	Let us introduce notations for the construction of the homology of such a
	sequence at $A_n$.
	
	\begin{enumerate}
		\item We construct $(K_n(A_\bullet),a'_n,\alpha'_n)=\Ker (a_{n+1},\alpha_{n+1})$,
			which induces an arrow $\hat{a}_n$ and 2-arrows $\mu_n$ and
			$\hat{\alpha}_{n-1}$ such that $\alpha'_n\hat{a}_n\circ a_{n+1}\mu_n=\alpha_n$
			and $a'_n\hat{\alpha}_{n-1}\circ \mu_n a_{n-1}=\alpha_{n-1}$.
	\end{enumerate}
			\begin{xym}\xymatrix@R=40pt@C=20pt{
				A_{n-1}\ar[rr]^-{a_{n-1}}\ar@/^2pc/[rrrr]^0\ar@/_0.78pc/[drrr]_-0
				& {}\rruppertwocell\omit{^<-2.7>\alpha_{n-1}\;\;\;\;\;\;}
				& A_n\ar[rr]^-{a_n}\ar[dr]_-{\hat{a}_n}\rrtwocell\omit\omit{_<4>\;\,\mu_n}
					\ar@/^2pc/[rrrr]^0\ar@{}[dll]|(0.28){\dir{=>}\,\hat{\alpha}_{n-1}}
				& {}\rrtwocell\omit\omit{^<-2.7>\alpha_{n}\;\,}
				& A_{n+1}\ar[rr]^-{a_{n+1}}\ar@{}[drr]|(0.28){\dir{=>}\!\alpha'_n}
					\ar@/^2pc/[rrrr]^0
				& {}\rrtwocell\omit\omit{^<-2.7>\alpha_{n+1}\;\;\;\;\;\;}
				& A_{n+2}\ar[rr]^-{a_{n+2}}
				&{}
				& A_{n+3}
				\\ &&&K_n(A_\bullet)\ar[ur]_-{a'_n}\ar@/_0.78pc/[urrr]_-0 &&&{}
			}\end{xym}
	\begin{enumerate}
		\item[2.] We construct $(Q_n(A_\bullet),a''_n,\alpha''_{n-1})=
			\Coker(a_{n-1},\alpha_{n-2})$, which induces an arrow $\check{a}_n$ and
			2-arrows $\nu_n$ and $\check{\alpha}_n$ such that
			$\check{a}_n\alpha''_{n-1}\circ\nu_n a_{n-1}=\alpha_{n-1}$
			and $\check{\alpha}_n a''_n\circ a_{n+1}\nu_n = \alpha_n$.
	\end{enumerate}
			\begin{xym}\xymatrix@R=40pt@C=20pt{
				A_{n-2}\ar[rr]^-{a_{n-2}}\ar@/^2pc/[rrrr]^0
				&{}\rrtwocell\omit\omit{^<-2.7>\alpha_{n-2}\;\;\;\;\;\;}
				&A_{n-1}\ar[rr]^-{a_{n-1}}\ar@/^2pc/[rrrr]^0\ar@/_0.78pc/[drrr]_-0
				& {}\rruppertwocell\omit{^<-2.7>\alpha_{n-1}\;\;\;\;\;\;}
				& A_n\ar[rr]^-{a_n}\ar[dr]_-{a''_n}\rrtwocell\omit\omit{_<4>\;\;\nu_n}
					\ar@/^2pc/[rrrr]^0\ar@{}[dll]|(0.28){\dir{=>}\,\alpha''_{n-1}}
				& {}\rrtwocell\omit\omit{^<-2.7>\alpha_{n}\;\,}
				& A_{n+1}\ar[rr]^-{a_{n+1}}\ar@{}[drr]|(0.28){\dir{=>}\!\check{\alpha}_n}
				& {}
				& A_{n+2}
				\\ &&&&&Q_n(A_\bullet)\ar[ur]_-{\check{a}_n}\ar@/_0.78pc/[urrr]_-0 &&&{}
			}\end{xym}
	\begin{enumerate}
		\item[3.] If $\C$ is homological, by Proposition \ref{homology}, there exist an object
			$H_n(A_\bullet)$, arrows $q_n$ and $k_n$ and 2-arrows $\zeta_n$,
			$\eta_n$ and $\kappa_n$ such that
			\begin{enumerate}
				\item $(q_n,\zeta_n)=\Coker(\hat{a}_{n-1},\hat{\alpha}_{n-2})$,
				\item $(k_n,\kappa_n)=\Ker(\check{a}_n,\check{\alpha}_n)$,
			\end{enumerate}
			and such that $k_n\zeta_n\circ \eta_n\hat{a}_{n-1}\circ a''_n\mu_{n-1}
			=\alpha''_{n-1}$ and $\kappa_nq_n\circ\check{a}_n\eta_n\circ\nu_na'_{n-1}
			=\alpha'_{n-1}$. 
	\end{enumerate}
			\begin{xym}\xymatrix@R=30pt@C=9.6pt{
				A_{n-2}\ar[rr]^-{a_{n-2}}\ar@/^2pc/[rrrr]^0\ar@/_0.78pc/[drrr]_-0
				&{}\rrtwocell\omit\omit{^<-2.7>\alpha_{n-2}\;\;\;\;\;\;}
				&A_{n-1}\ar[rr]^-{a_{n-1}}\ar[dr]|-{\hat{a}_{n-1}}\ar@/^2pc/[rrrr]^0
					\ar@{}[dll]|(0.25){\dir{=>}\,\hat{\alpha}_{n-2}}\ar@/_2.5pc/[ddrr]_0
					\rrtwocell\omit\omit{_<3.7>\;\;\;\;\;\;\,\mu_{n-1}}
					\ddrrlowertwocell\omit{_<3.5>\;\;\zeta_n}
				& {}\rruppertwocell\omit{^<-2.7>\alpha_{n-1}\;\;\;\;\;\,}
				& A_{n}\ar[rr]^-{a_{n}}\ar[dr]_-{a''_n}\rrtwocell\omit\omit{_<3.7>\;\;\nu_n}
					\ar@/^2pc/[rrrr]^0
				& {}\rrtwocell\omit\omit{^<-2.7>\alpha_n\;\;}
				& A_{n+1}\ar[rr]^-{a_{n+1}}
					\ar@{}[drr]|(0.33){\dir{=>} \!\check{\alpha}_n}
				& {}
				& A_{n+2}
				\\ {} &&&K_{n-1}(A_\bullet)\ar[ur]_-{a'_{n-1}}
					\ar[dr]_-{q_n}\rrtwocell\omit\omit{\;\;\eta_n}
				&&Q_n(A_\bullet)\ar[ur]_-{\check{a}_n}\ar@/_0.78pc/[urrr]_-0 &&&{}
				\\ &&&&H_n(A_\bullet)\ar[ur]_-{k_n}\ar@/_2.5pc/[uurr]_0
					\uurrlowertwocell\omit{_<3.8>\;\;\kappa_n}
			}\end{xym}
	\begin{enumerate}
		\item[4.] As in Proposition \ref{exenunfl}, there exists an arrow
			$h_n\col Q_n(A_\bullet)\ra K_n(A_\bullet)$ and 2-arrows
			$\mu'_n$ and $\nu'_n$ such that $a'_n\mu'_n\circ \mu_n=\nu'_n a''_n\circ\nu_n$,
			$h_n\alpha''_{n-1}\circ\mu'_na_{n-1}=\hat{\alpha}_{n-1}$
			and $\alpha'_nh_n\circ a_{n+1}\nu'_n=\check{\alpha}_n$.
			Moreover, there exist 2-arrows $\kappa'_n$ and $\zeta'_{n+1}$ such
			that $a'_n\kappa'_n\circ\nu'_n k_n=\kappa_n$ and
			$\zeta'_{n+1}a''_n\circ q_{n+1}\mu'_n=\zeta_{n+1}$.  And, by Corollary
			\ref{lemmsimplnoy}, since $a'_n$ is $\alpha'_n$-fully faithful
			and $a''_n$ is $\alpha''_{n-1}$-fully cofaithful, we have
			\begin{enumerate}
				\item $(H_n(A_\bullet),k_n,\kappa'_n)=\Ker h_n$ and
				\item $(H_{n+1}(A_\bullet),q_{n+1},\zeta'_{n+1})=\Coker h_n$.
			\end{enumerate}
			\begin{xym}\xymatrix@=40pt{
				A_{n-1}\ar[r]^-{a_{n-1}}\ar@/^2pc/[rr]^0\ar@/_0.78pc/[dr]_-0
					\drtwocell\omit\omit{_<-2>{\alpha''_{n-1}\;\;\;\;\;\;\;\;\;\;\;\;\;\;\;}}
					\rruppertwocell\omit{^<-2.7>\alpha_{n-1}\;\;\;\;\;\,}
				&A_{n}\ar[r]^-{a_{n}}\ar[d]_{a''_n}\ar@{}[dr]^(0.72){\dir{=>}\nu'_n}
					\ar@/^2pc/[rr]^0\ar[dr]^(0.3){\hat{a}_n}
					\rrtwocell\omit\omit{^<-2.7>\alpha_{n}\;\;}
				&A_{n+1}\ar[r]^-{a_{n+1}}\ar@{}[dl]_(0.72){\mu'_n\dir{=>}}
				&A_{n+2}
				\\ H_n(A_\bullet)\ar[r]_-{k_n}\rrlowertwocell\omit{_<2.7>\;\;\,\kappa'_n}
					\ar@/_2pc/[rr]_0
				&{Q_n(A_\bullet)}\ar[r]_-{h_n}\ar[ur]^(0.7){\check{a}_n}
					\rrlowertwocell\omit{_<2.7>\;\;\;\;\;\;\,\zeta'_{n+1}}\ar@/_2pc/[rr]_0
				&{K_n(A_\bullet)}\ar[u]_-{a'_n}\ar[r]_-{q_{n+1}}
					\ar@/_0.78pc/[ur]_-0\urtwocell\omit\omit{_<-2>{\;\;\,\alpha'_n}}
				&H_{n+1}(A_\bullet)
			}\end{xym}
	\end{enumerate}

	Now, let $f_\bullet\col A_\bullet\ra B_\bullet$ be a morphism of chain complexes.
	\begin{enumerate}
		\item For any $n\col\mathbb{Z}$, we have
			an arrow $f'_{n-1}$ and a 2-arrow $\varphi_{n-1}'\col b'_{n-1} f'_{n-1}
			\Ra f_na'_{n-1}$ such that $f_{n+1}\alpha'_{n-1}\circ\varphi_n a'_{n-1}
			\circ b_n\varphi'_{n-1}=\beta'_{n-1} f'_{n-1}$.
		\item Then we have, for all $n$, an arrow $f''_{n+1}$ and a 2-arrow
			$\varphi''_{n+1}\col b''_{n+1} f_{n+1}\Ra f''_{n+1}a''_{n+1}$ such
			that $f''_{n+1}\alpha''_n\circ\varphi''_{n+1}a_n\circ b''_{n+1}\varphi_n
			= \beta''_n f_n$.
		\item This induces, by the universal property of $b'_n$, which is the
				relative kernel of $b_{n+1}$,
			a 2-arrow $\hat{\varphi}_n$ such that
	\end{enumerate}
			\begin{xyml}\begin{gathered}\xymatrix@R=25pt@C=0pt{
				&A_n\ar[rr]^{f_n}\ar[dl]_{\hat{a}_n}\ar[dd]^{a_n}
					\ddrrtwocell\omit\omit{\;\,\varphi_n}\ddtwocell\omit\omit{_<2.2>\;\mu_n}
				&&B_n\ar[dd]^{b_n}
				\\ K_n(A_\bullet)\ar[dr]_{a'_n} &&{~~~~~~~~}
				\\ &A_{n+1}\ar[rr]_{f_{n+1}}
				&&B_{n+1}
			}\end{gathered}\;\;=\;\;\begin{gathered}\xymatrix@R=25pt@C=0pt{
				&A_n\ar[rr]^{f_n}\ar[dl]_{\hat{a}_n}\drtwocell\omit\omit{\;\;\hat{\varphi}_n}
				&&B_n\ar[dd]^{b_n}\ar[dl]_{\hat{b}_n}\ddtwocell\omit\omit{_<2.2>\;\mu_n}
				\\ K_n(A_\bullet)\ar[dr]_{a'_n}\ar[rr]_{f'_n}
				&&K_n(B_\bullet)\ar[dr]_{b'_n}\ar@{}[dl]|{\dir{=>}\,\varphi'_n}
				\\ &A_{n+1}\ar[rr]_{f_{n+1}}
				&&B_{n+1}
			}\end{gathered}.\end{xyml}
	\begin{enumerate}
		\item[4.] This induces, by the universal property of $a''_n$, which is the
			 relative cokernel of
			$a_{n-1}$, a 2-arrow $\tilde{\varphi}_n$ such that
	\end{enumerate}
			\begin{xyml}\begin{gathered}\xymatrix@R=25pt@C=-7pt{
				&A_n\ar[rr]^{f_n}\ar[dl]_{a''_n}\ddtwocell\omit\omit{_<2.2>\mu'_n}
					\ddrrtwocell\omit\omit{\;\,\hat{\varphi}_n}\ar[dd]^{\hat{a}_n}
				&&B_n\ar[dd]^{\hat{b}_n}
				\\ Q_n(A_\bullet)\ar[dr]_{h_n}&&{~~~~~~~~~}
				\\ &K_n(A_\bullet)\ar[rr]_-{f'_n}
				&&K_n(B_\bullet)
			}\end{gathered}\;\;=\;\;\begin{gathered}\xymatrix@R=25pt@C=-7,pt{
				&A_n\ar[rr]^{f_n}\ar[dl]_{a''_n}\drtwocell\omit\omit{\;\;\varphi''_n}
				&&B_n\ar[dd]^{\hat{b}_n}\ar[dl]_{b''_n}\ddtwocell\omit\omit{_<2.2>\;\mu'_n}
				\\ Q_n(A_\bullet)\ar[dr]_{h_n}\ar[rr]_{f''_n}
				&&Q_n(B_\bullet)\ar[dr]_{h_n}\ar@{}[dl]|{\dir{=>}\,\tilde{\varphi}_n}
				\\ &K_n(A_\bullet)\ar[rr]_-{f'_n}
				&&K_n(B_\bullet)
			}\end{gathered}.\end{xyml}
	\begin{enumerate}
		\item[5.] Finally, this induces for every $n$ an arrow $H_n(f_\bullet)\col H_n(A_\bullet)
			\ra H_n(B_\bullet)$ and 2-arrows $\dot{\varphi}_n$ and $\ddot{\varphi}_n$
			such that $f'_n\kappa'_n\circ\tilde{\varphi}_n k_n\circ h_n\dot{\varphi}_n
			=\kappa'_n H_n(f_\bullet)$ and $H_{n+1}(f_\bullet)\zeta'_{n+1}\circ
			\ddot{\varphi}_{n+1}h_n\circ q_{n+1}\tilde{\varphi}_n=\zeta'_{n+1}f''_n$
			(see diagram \ref{diagpresequencexhom}).
	\end{enumerate}

\subsection{The long sequence}

	Let $0\longrightarrow A_\bullet\overset{f_\bullet}\longrightarrow B_\bullet\overset{g_\bullet}\longrightarrow C_\bullet
	\longrightarrow 0$, with $\omega_\bullet\col g_\bullet f_\bullet\Ra 0_\bullet$, be an extension
	in the $\Gpdp$-category of chain complexes in $\C$.  Since $\C$ has all the kernels
	and cokernels, the kernels and cokernels in the $\Gpdp$-category of chain complexes
	in $\C$ are computed pointwise. Thus the fact that $(f_\bullet, \omega_\bullet,
	g_\bullet)$ is an extension means that for every $n\col\mathbb{Z}$,
	$(f_n,\omega_n,g_n)$ is an extension in $\C$.
	
	Then we make the following constructions.
	\begin{enumerate}
		\item We apply the relative kernels lemma (Corollary \ref{lemmnoyrel})
			and its dual to the central rows of the following diagram.  Then, for every $n$,
			the following diagram commutes and
			$(f'_{n-1},\omega'_{n-1})=\Ker g'_{n-1}$ and $(g''_{n+2},\omega''_{n+2})
			=\Coker f''_{n+2}$.
			\begin{xym}\xymatrix@R=40pt@C=30pt{
				K_{n-1}(A_\bullet)\ar[r]^{f'_{n-1}}\ar[d]|{a'_{n-1}}
					\rruppertwocell<9>^0{^<-2.7>{\omega'_{n-1}\;\;\;\;\;}}
					\drtwocell\omit\omit{_{\;\;\;\;\;\,\varphi'_{n-1}}}
					\ddlowertwocell<-12>_0{_<4>{\alpha'_{n-1}}}
				&K_{n-1}(B_\bullet)\ar[r]^{g'_{n-1}}\ar@/_1.89pc/[dd]_(0.56)0
					\drtwocell\omit\omit{_{\;\;\;\;\;\;\,\psi'_{n-1}}}\ar[d]|(0.38){b'_{n-1}}
				&K_{n-1}(C_\bullet)\ar[d]|{c'_{n-1}}
					\dduppertwocell<12>^0{^<-4>{\gamma'_{n-1}\;}}
				\\ A_n\ar[r]^(0.4){f_n}\ar[d]_{a_n}\ar@/^1.87pc/[rr]^(0.55)0
					\drtwocell\omit\omit{_{\;\;\,\varphi_n}}
					\ddlowertwocell<-12>_0{_<4>{\alpha_n}}
				&B_n\ar[r]^{g_n}\ar[d]|{\,b_n}\drtwocell\omit\omit{_{\;\;\psi_n}}
					\ar@{}[u]_(0.25){\,\dir{=>}\,\omega_n}\ar@/^1.89pc/[dd]^(0.44)0
					\ar@{}[l]^(0.16){\txt{$\underset{\;\;\,\beta'_{n-1}}{\dir{=>}}$}}
				&C_n\ar[d]^{c_n}\dduppertwocell<12>^0{^<-4>{\gamma_n}}
				\\ A_{n+1}\ar[r]^{f_{n+1}}\ar[d]|{a_{n+1}}\ar@/_1.87pc/[rr]_(0.45)0
					\drtwocell\omit\omit{_{\;\;\;\;\;\;\varphi_{n+1}}}
					\ddlowertwocell<-12>_0{_<4>{\;\;\alpha''_{n+1}}}
				&B_{n+1}\ar[r]_(0.6){g_{n+1}}\ar[d]|(0.62){b_{n+1}}
					\drtwocell\omit\omit{_{\;\;\;\;\;\;\psi_{n+1}}}\ar@/^1.89pc/[dd]^(0.56)0
					\ar@{}[d]_(0.25){\omega_{n+1}\,\dir{=>}}
					\ar@{}[r]^(0.17){\txt{$\overset{\beta_n\;\;}{\dir{=>}}$}}
				&C_{n+1}\ar[d]|{c_{n+1}}\dduppertwocell<12>^0{^<-4>{\gamma''_{n+1}}}
				\\ A_{n+2}\ar[r]^{f_{n+2}}\ar@/_1.87pc/[rr]_(0.4)0\ar[d]|{a''_{n+2}}
					\drtwocell\omit\omit{\varphi''_{n+2}\;\;\;\;\;\;\;\;\;\;\;\;\;\;\;}
				&B_{n+2}\ar[r]^{g_{n+2}}\ar@{}[d]_(0.25){\omega_{n+2}\,\dir{=>}\;\;\;\;}
					\drtwocell\omit\omit{\;\;\;\;\;\;\psi''_{n+2}}\ar[d]|(0.6){b''_{n+2}}
					\ar@{}[r]_(0.17){\txt{$\underset{\beta''_{n+1}\;\;}{\dir{=>}}$}}
				&C_{n+2}\ar[d]|{c''_{n+2}}
				\\ Q_{n+2}(A_\bullet)\ar[r]_{f''_{n+2}}
					\rrlowertwocell<-9>_0{_<2.7>{\;\;\;\;\;\;\;\omega''_{n+2}}}
				&Q_{n+2}(B_\bullet)\ar[r]_{g''_{n+2}}
				&Q_{n+2}(C_\bullet)
			}\end{xym}
		\item We apply the generalised snake lemma (Proposition \ref{lemmserpgen})
			to the following diagram, which gives us the following theorem.
			\begin{xym}\label{diagpresequencexhom}\xymatrix@R=40pt@C=30pt{
				H_n(A_\bullet)\ar[r]^{H_n(f_\bullet)}\drtwocell\omit\omit{_{\;\,}}
					\rruppertwocell<10.9>^0{^<-3.5>H_n(\omega_\bullet)\;\;\;\;\;\;\;\;\;\;}
					\ddlowertwocell<-15>_0{_<5.4>\kappa'_n}\ar[d]_{k_n}
				&H_n(B_\bullet)\ar[r]^{H_n(g_\bullet)}\drtwocell\omit\omit{_{\;\,}}
					\ar@/_2.5pc/[dd]_(0.56)0\ar[d]_(0.38){k_n}\ar@{}[dl]_{\dot{\varphi}_n}
				&H_n(C_\bullet)\ar[d]^{k_n}\dduppertwocell<15>^0{^<-5.4>\;\kappa'_n}
					\ar@{}[dl]_{\dot{\psi}_n}
				\\ Q_n(A_\bullet)\ar[r]^(0.4){f''_n}\ar[d]_{h_n}\ar@/^1.87pc/[rr]^(0.55)0
					\drtwocell\omit\omit{_{\;\,\tilde{\varphi}_n}}
					\ddlowertwocell<-15>_0{_<5.4>\;\,\zeta'_{n+1}}
				&Q_n(B_\bullet)\ar[r]^{g''_n}\ar[d]_{h_n}\ar@/^2.5pc/[dd]^(0.44)0
					\drtwocell\omit\omit{_{\;\,\tilde{\psi}_n}}
					\ar@{}[u]_(0.25){\,\dir{=>}\,\omega''_n}
					\ar@{}[l]^(0.2){\txt{$\underset{\;\;\kappa'_n}{\overset{~}{\dir{=>}}}$}}
				&Q_n(C_\bullet)\ar[d]^{h_n}\dduppertwocell<15>^0{^<-5.4>\zeta'_{n+1}}
				\\ K_n(A_\bullet)\ar[r]_{f'_n}\ar[d]_{q_{n+1}}\ar@/_1.87pc/[rr]_(0.45)0
					\drtwocell\omit\omit{_{\;\;\;\;\;\;\,\ddot{\varphi}_{n+1}}}
				&K_n(B_\bullet)\ar[r]_(0.6){g'_n}\ar@{}[d]_(0.25){\omega'_n\,\dir{=>}}
					\drtwocell\omit\omit{_{\;\;\;\;\;\,\ddot{\psi}_{n+1}}}\ar[d]|(0.62){q_{n+1}}
				   \ar@{}[r]^(0.2){\txt{$\overset{\zeta'_{n+1}\;\;}{\underset{~}{\dir{=>}}}$}}
				&K_n(C_\bullet)\ar[d]^{q_{n+1}}
				\\ H_{n+1}(A_\bullet)\ar[r]_{H_{n+1}(f_\bullet)}\rrlowertwocell<-10.9>_0
					{_<3.5>\;\;\;\;\;\;\;\;\;\;\;\;\;\,H_{n+1}(\omega_\bullet)}
				&H_{n+1}(B_\bullet)\ar[r]_{H_{n+1}(g_\bullet)}
				&H_{n+1}(C_\bullet)
			}\end{xym}
	\end{enumerate}

	\begin{thm}\label{lngsequencexhom}\index{homology!long exact sequence of}%
	\index{long exact sequence!of homology}
		Let $(f_\bullet,\omega_\bullet,g_\bullet)$ be a chain complexes extension  
		in an abelian $\Gpd$-category. For every $n\col\mathbb{Z}$,
		there exist an arrow $d_n$ and 2-arrows $\delta_n$ and $\delta'_n$
		such that the following sequence (where the left lower 2-arrow
		is $H_n(\omega_\bullet)$ and the right upper 2-arrow
		is $H_{n+1}(\omega_{\bullet})$) is exact.
		\begin{xym}\xymatrix@=10pt{
			\cdots H_n(A_\bullet)\ar[r]^-{}
				\rrlowertwocell<-9>_0{_<2.7>{}}
			&H_n(B_\bullet)\ar[r]^-{}\rruppertwocell<9>^0{^<-2.7>\delta_n\;}
			&H_n(C_\bullet)\ar[r]^-{d_n\;\,}\rrlowertwocell<-9>_0{_<2.7>\;\;\,\delta'_n}
			&H_{n+1}(A_\bullet)\ar[r]^-{}
				\rruppertwocell<9>^0{^<-2.7>{}}
			&H_{n+1}(B_\bullet)\ar[r]_-{}
			&H_{n+1}(C_\bullet)\cdots
		}\end{xym}
	\end{thm}
	
	In a good 2-abelian $\Gpd$-category, we can add to this theorem that the loops
	$\delta_n H_n(f_\bullet)\circ d_n H_n(\omega_\bullet)^{-1}$,
	$\delta'_n H_n(g_\bullet)\circ H_{n+1}(f_\bullet)\delta_n^{-1}$ and
	$H_{n+1}(\omega_\bullet)d_n\circ H_{n+1}(g_\bullet)\delta_n^{\prime-1}$ are exact,
	since the snake lemma tells us that these composites are equal
	to $\mu_{h_n}$ (ou $\mu_{h_n}^{-1}$).

\chapter{2-abelian $\Gpd$-categories}

\begin{quote}
	{\it In this chapter, we define 2-Puppe-exact and 2-abelian $\Gpdp$-ca\-te\-go\-ries, 
	which share many properties with the $\Gpd$-category
	of symmetric 2-groups.  We prove that in a 2-abelian $\Gpd$-category,
	the category of discrete objects is equivalent to the category of connected objects and
	is abelian (Corollary \ref{discconcabel}).  Next, in the context of a
	good 2-Puppe-exact $\Gpdp$-category $\C$, we classify the properties of
	arrows in $\C$ in terms of the properties of their (co)reflexions in $\DisC$
	and $\ConC$.  We also define an internal notion of full arrow which
	generalises the full functors (Definition \ref{deffullgpdcat}).}
\end{quote}

\section{Definition}

We define now a second 2-dimensional version of the notion of Puppe-exact category.  Unlike Puppe-exact $\Gpdp$-categories of Definition \ref{defpuppex}, which was a generalisation of the notion of Puppe-exact category, this one is rather an analogue of Puppe-exact categories: the 2-Puppe-exact $\Gpdp$-categories are the $\Gpdp$-categories where, for an arrow $f$, the coroot of the pip of $f$ and the kernel of the cokernel of $f$ coincide, and dually.  But we will see in the following sections that there is a strong link between 2-Puppe-exact $\Gpdp$-categories and Puppe-exact categories, because the discrete (or connected) objects in a 2-Puppe-exact $\Gpdp$-category form a Puppe-exact category.

We can notice the following properties:
\begin{enumerate}
	\item every cokernel is cofaithful and is thus 0-cofaithful (i.e.\ is a
		$\Copip$-epi\-mor\-phism);
	\item every kernel is faithful and is thus 0-faithful (i.e.\ is a
		 $\Pip$-monomorphism);
	\item every coroot is fully cofaithful and is thus fully 0-cofaithful
		(i.e.\ is a $\Coker$-epimorphism);
	\item every root is fully faithful and is thus fully 0-faithful
		(i.e.\ is a $\Ker$-mo\-no\-mor\-phism).
\end{enumerate}
We have thus the following proposition.

\begin{pon}
	Let $\C$ be a $\Gpdp$-category with all the kernels and cokernels.
	\begin{enumerate}
		\item The kernel-quotient system $\Coker\adj\Ker$
			and the cokernel-coquotient system $\Copip\adj\Root$ are precoupled.
		\item The kernel-quotient system $\Coroot\adj\Pip$ and the
			cokernel-coquotient system $\Coker\adj\Ker$ are precoupled.
	\end{enumerate}
\end{pon}

There exist thus comparison arrows between, on the one hand, the cokernel of the kernel of an arrow and the root of its copip and, on the other hand, the coroot of the pip of an arrow and the kernel of its cokernel.
	\begin{xym}\label{diagcokerkerraccopep}\xymatrix@=30pt{
		{\Ker f}\ar[r]^-{k_f}
		& A\ar[r]^f\ar[d]_{\bar{e}^1_f}\rtwocell\omit\omit{_<4.75>\;\;\,\bar{\omega}_f}
		& B\rtwocell^0_0{\;\;\rho_f}
		&{}\save[]+<21pt,0pt>*{\Copip f}\restore
		\\ &{\Coker k_f}\ar[r]_-{\bar{w}_f}
		&{\Root\rho_f}\ar[u]_{\bar{m}^2_f}
	}\end{xym}
	\begin{xym}\label{diagcoracpepkercoker}\xymatrix@=30pt{
		{~}\save[]+<-11pt,0pt>*{\Pip f}\restore\rtwocell^0_0{\;\,\pi_f}
		& A\ar[r]^f\ar[d]_{e^1_f}\rtwocell\omit\omit{_<4.75>\;\;\,\omega_f}
		& B\ar[r]^-{q_f}
		&{\Coker f}
		\\ &{\Coroot\pi_f}\ar[r]_-{w_f}
		&{\Ker q_f}\ar[u]_{m^2_f}
	}\end{xym}
	
By Remark \ref{remfactregpepkerraccoker}, we can also construct these factorisations by taking the cokernel of the $\pi_1$ of the kernel of $f$  or the kernel of the
	$\pi_0$ of the cokernel of $f$.
	\begin{xym}\label{diagcokerkerraccopepalt}\xymatrix@=30pt{
		{\Ker f}\ar[r]^-{k_f}
		& A\ar[r]^f\ar[d]_{\bar{e}^1_f}\rtwocell\omit\omit{_<4.75>\;\;\,\bar{\omega}_f}
		& B\ar[r]^-{q_f}
		&{\Coker f}\ar[r]^-{\eta}
		&{\pi_0\Coker f}
		\\ &{\Coker k_f}\ar[r]_-{\bar{w}_f}
		&{\Ker(\eta q_f)}\ar[u]_{\bar{m}^2_f}
	}\end{xym}
	\begin{xym}\label{diagcoracpepkercokeralt}\xymatrix@=30pt{
		{\pi_1\Ker f}\ar[r]^-{\varepsilon}
		&{\Ker f}\ar[r]^-{k_f}
		& A\ar[r]^f\ar[d]_{e^1_f}\rtwocell\omit\omit{_<4.75>\;\;\,\omega_f}
		& B\ar[r]^-{q_f}
		&{\Coker f}
		\\ &&{\Coker(k_f\varepsilon)}\ar[r]_-{w_f}
		&{\Ker q_f}\ar[u]_{m^2_f}
	}\end{xym}

	By applying Proposition \ref{defcaracparf} to these two pairs of precoupled systems,
	we get the following proposition, which defines 2-Puppe-exact $\Gpdp$-categories
	as the $\Gpdp$-categories which are both
	$\Ker$-$\Copip$-perfect and $\Pip$-$\Coker$-perfect.

	\begin{pon}\label{caractwopupex}
		Let $\C$ be a $\Gpdp$-category with zero object and all
		the kernels and cokernels.
		The following properties are equivalent (if they
		hold, we say that $\C$ is \emph{2-Puppe-exact}):%
		\index{2-puppe-exact $\Gpdp$-category}%
		\index{Gpd*-category@$\Gpdp$-category!2-Puppe-exact}
		\begin{enumerate}
			\item for every arrow $f\col\flc$, $\bar{w}_f$
				and $w_f$ are equivalences;
			\item every arrow factors as a normal cofaithful arrow followed by
				a normal fully faithful arrow, and as a normal fully
				cofaithful arrow followed by a normal faithful arrow;
			\item
				\begin{enumerate}
					\item every 0-cofaithful arrow is canonically the cokernel of its kernel
						(i.e.\ is normal cofaithful);
					\item every 0-faithful arrow is canonically the kernel of its cokernel
						(i.e.\ is normal faithful);
					\item every fully 0-cofaithful arrow is canonically the coroot
					 	of its pip (i.e.\ is normal fully cofaithful);
					\item every fully 0-faithful arrow is canonically the root
					 	of its copip (i.e.\ is normal fully faithful).
				\end{enumerate}
		\end{enumerate}
	\end{pon}
	
	A first important property of  2-Puppe-exact $\Gpdp$-categories
	is that the (fully) 0-faithful arrows coincide with the
	(fully) faithful arrows.
	
	\begin{pon}\label{fidfidzer}
		Let be $f\col\flc$, where $\C$ is a 2-Puppe-exact $\Gpdp$-category. We have:
		\begin{enumerate}
			\item $f$ is faithful if and only if $f$ is 0-faithful;
			\item $f$ is cofaithful if and only if $f$ is 0-cofaithful;
			\item $f$ is fully faithful if and only if $f$ is fully
				0-faithful;
			\item $f$ is fully cofaithful if and only if $f$ is fully
				0-cofaithful.
		\end{enumerate}
	\end{pon}
	
		\begin{proof}
			If $f$ is 0-faithful, then $f$ is a kernel, and is thus faithful.
			If $f$ is fully 0-faithful, then $f$ is a root, and is
			thus fully faithful.
		\end{proof}
		
	\begin{pon}\label{fidplcofidequ}
		In a 2-Puppe-exact $\Gpdp$-category,
		\begin{enumerate}
			\item every faithful and fully cofaithful arrow is an equivalence;
			\item every fully faithful and cofaithful arrow is an equivalence.
		\end{enumerate}
	\end{pon}

	\begin{pon}
		Let $\C$ be a 2-Puppe-exact $\Gpdp$-category. Then:
		\begin{enumerate}
			\item $\C$ is $\Ker$-preexact and $\Coker$-preexact;
			\item $\C$ is $\Ker$-factorisable and $\Coker$-factorisable;
			\item $\C$ is $\Pip$-factorisable and $\Copip$-factorisable.
		\end{enumerate}
	\end{pon}
	
		\begin{proof}
			Condition 3(b) of Proposition \ref{caractwopupex}
			implies immediately that $\C$ is $\Ker$-preexact.
			Dually, by condition 3(a), $\C$ is $\Coker$-preexact.
			Next, $\C$ is $\Ker$-factorisable and $\Copip$-factorisable
			since it is $\Ker$-$\Copip$-perfect, and it is $\Pip$-factorisable
			and $\Coker$-factorisable
			since it is $\Pip$-$\Coker$-perfect.
		\end{proof}
	
	We will prove later the $\Pip$-preexactness and the $\Copip$-preexactness
	(Proposition \ref{depupexpippreex}).

	In a 2-Puppe-exact $\Gpdp$-category, there are two factorisations of each
	arrow.   Let us fix for each of these factorisations a construction:
	\begin{enumerate}
		\item $f$ factors as $A\xrightarrow{\hat{e}_f} \impl f
			\xrightarrow{\hat{m}_f} B$, with $\hat{\varphi}_f\col f\Ra \hat{m}_f\hat{e}_f$,
			where $\hat{e}_f$ is cofaithful and $\hat{m}_f$ is fully faithful;
		\item $f$ factors as $A\xrightarrow{e_f} \im f
			\xrightarrow{m_f} B$, with $\varphi_f\col f\Ra m_fe_f$,
			where $e_f$ is fully cofaithful and $m_f$ is faithful.
	\end{enumerate}
	Since $e_f$ is fully cofaithful and $\hat{m}_f$ is faithful, $e_f\orth\hat{m}_f$,
	because $(\PlCofid,\Fid)$ is a factorisation system.  There exist thus an arrow
	$l_f\col\im f\ra\impl f$ and 2-arrows such that
	\begin{xyml}\begin{gathered}\xymatrix@C=30pt@R=30pt{
		&{\im f}\ar[dr]^-{m_f}\ar[dd]_(0,4){l_f}
		&{}\ar@{}[dl]^(0.86){\dir{=>}}
		\\A\ar[dr]_-{\hat{e}_f}\ar[ur]^{e_f}
		&{}\ar@{}[dl]_(0.3){\dir{=>}} &B
		\\ {}&{\impl f}\ar[ur]_-{\hat{m}_f}
	}\end{gathered}\;\;=\;\;\begin{gathered}\xymatrix@C=30pt@R=30pt{
		&{\im f}\ar[dr]^-{m_f}
		\\A\ar[rr]|f\ar[dr]_-{\hat{e}_f}
			\rrtwocell\omit\omit{_<3.8>\;\;\,\hat{\varphi}_f}\ar[ur]^-{e_f}
			\rrtwocell\omit\omit{^<-3.8>\varphi_f\;\;}
		&&B
		\\ &{\impl f}\ar[ur]_-{\hat{m}_f}
	}\end{gathered}.\end{xyml}
	The arrow $l_f$ is faithful, because $m_f$ is faithful, and cofaithful, because $\hat{e}_f$
	is cofaithful.  Thus every arrow in a 2-Puppe-exact $\Gpdp$-category factors in three parts:
	\begin{eqn}\label{diaglf}
		A\xrightarrow{e_f}\im f\xrightarrow{l_f}\impl f\xrightarrow{\hat{m}_f} B,
	\end{eqn}
	where $e_f$ is fully cofaithful, $l_f$ is faithful and cofaithful, and $\hat{m}_f$
	is fully faithful.

	In dimension 1, an abelian category is a Puppe-exact category which has the finite
	products and coproducts.  This explains that there are few examples of non-abelian
	Puppe-exact categories. Up to now, there is no known 2-dimensional example.
	
	\begin{df}\label{deftwoab}\index{Gpd-category@$\Gpd$-category!2-abelian}%
	\index{2-abelian $\Gpd$-category}
		A \emph{2-abelian $\Gpd$-category} is a 2-Puppe-exact $\Gpdp$-category
		which has all the finite products and coproducts.
	\end{df}

\section{Discrete and connected objects}

\subsection{Equivalence of $\DisC$ and $\ConC$}\label{sssectequdiscconc}

	In 2-Puppe-exact $\Gpdp$-categories, there is an alternative construction of $\pi_0$ and $\pi_1$.
	
	\begin{pon}\label{etacoracepsrac}
		Let $C$ be an object in a 2-Puppe-exact $\Gpdp$-category $\C$.
		Then
		\begin{align}\stepcounter{eqnum}
			\eta_C&=\Coroot(\omega_C)\text{ and}\\ \stepcounter{eqnum}
			\varepsilon_C&=\Root(\sigma_C)
		\end{align}
		(where $\eta_C$ is the unit and $\varepsilon_C$, the counit of the adjunction
		$\Sigma\adj\Omega$; see diagrams \ref{diagunitcounitpiopiu}).
		\begin{xym}\xymatrix@=40pt{
			&\pi_1 C\ar[d]^{\varepsilon_C}
			\\ \Omega C\rtwocell^0_0{\;\;\;\,\omega_C}
			&C\ar[r]^-{\eta_C}\dtwocell^0_0{\sigma_C}
			&\pi_0 C
			\\ &\Sigma C
		}\end{xym}\index{skittles lemma}
	\end{pon}
	
		\begin{proof}
			The first equation comes from the fact that, in a 2-Puppe-exact $\Gpdp$-category,
			the kernel of the cokernel of $0^C\col C\ra 0$, which is in fact
			$\Omega\Sigma C = \pi_0 C$, coincides with the coroot of the pip of
			$0^C$, which is $\Omega C$, by Proposition \ref{omegpepzc}.  The second
			equation is proved dually.
		\end{proof}

	\begin{pon}\label{extenspizpiu}
		Let $C$ be an object of a 2-Puppe-exact $\Gpdp$-category $\C$.  Then there
		exists a 2-arrow $\eta_C\varepsilon_C\Ra 0$ which makes
		the following sequence an extension (a relative exact sequence).
		\begin{xym}\xymatrix@=40pt{
			0\ar[r]
			&\pi_1 C\rruppertwocell<-9>_0{_<2.7>{}}\ar[r]^-{\varepsilon_C}
			&C\ar[r]^-{\eta_C}
			&\pi_0 C\ar[r]
			&0
		}\end{xym}
	\end{pon}
	
		\begin{proof}
			Proposition \ref{coraccoker} applies to the following diagram, because
			$\varepsilon_C=\bar{\omega}_C$.  Since $\eta_C=\Coroot\omega_C$, by the
			previous proposition, it follows that there is a 2-arrow
			$\eta_C\varepsilon_C\Ra 0$ (unique, because $\pi_1 C=\Sigma\Omega C$
			is connected) such that
			$\eta_C=\Coker \varepsilon_C$.  We prove dually that
			$\varepsilon_C=\Ker \eta_C$.\qedhere
			\begin{xym}\xymatrix@=50pt{
				\Omega C\rtwocell^0_0{\;\;\;\omega_C}
					\dtwocell_0<4.5>^0<4.5>{\sigma_{\Omega C}}
				&C\ar[r]^-{\eta_C}
				&\pi_0C
				\\ \Sigma\Omega C\ar@<-0.5mm>[ur]_{\varepsilon_C}
			}\end{xym}
		\end{proof}

	If we apply conditions 3(a) and 3(b) of Proposition \ref{caractwopupex} to the case
	of arrows with zero (co)domain, we get a characterisation of discrete objects
	as the objects $D$ which are canonically  equivalent to
	$\pi_0(D)=\Omega\Sigma D$,
	and of connected objects as the objects $C$ which are canonically equivalent
	to $\pi_1(C)=\Sigma\Omega C$.

	\begin{pon}\label{caracdispreadd}
		Let $\C$ be a 2-Puppe-exact $\Gpdp$-category
		 and $D\col\C$.  The following conditions are equivalent:
		\begin{enumerate}
			\item $D$ is discrete;
			\item $\eta_D\col D\ra \pi_0 D$ is an equivalence;
			\item $\pi_1 D\simeq 0$;
			\item $0^D\col D\ra 0$ is 0-faithful (for every
				 $\alpha\col 0\Ra 0\col X\ra D$, $\alpha=1_0$);
			\item $0_D\col 0\ra D$ is fully (0-)faithful.
		\end{enumerate}
	\end{pon}
	
		\begin{proof}
			{\it 1 $\Leftrightarrow$ 2 $\Leftrightarrow$ 4. } 
			Condition 1 means that $0^D$ is faithful; condition 2 means
			that $0^D$ is the kernel of its cokernel. So conditions 1, 2 and 4
			are equivalent by condition 3(b) of Proposition \ref{caractwopupex}.
			
			{\it 2 $\Leftrightarrow$ 3. } 
			This is an immediate consequence of Proposition \ref{extenspizpiu}.
			
			{\it 3 $\Leftrightarrow$ 5. }
			The arrow $0_D$ factors, by taking the cokernel of its kernel,
			as $0\longrightarrow \pi_1D\overset{\varepsilon_D}\longrightarrow D$,
			where $0_{\pi_1D}$ is cofaithful and $\varepsilon_D$ is fully faithful.
			By the properties of factorisation systems, $0_D$ is fully faithful
			if and only if $0_{\pi_1D}$ is an equivalence.
		\end{proof}
	
	Let us return to the adjunction $\Sigma\adj\Omega$ (diagram \ref{adjsigmomeg}).  The
	previous proposition tells us that in a 2-Puppe-exact $\Gpdp$-category,
	the objects of $\C$ where the unit of this adjunction
	is an equivalence are the discrete objects and that the objects of $\C$ where
	the co\-unit is an equivalence are the connected objects.
	So the adjunction $\Sigma\adj\Omega$ restricts to an equivalence between $\ConC$
	and $\DisC$ (let us recall that $\ConC$ is a $\Ens$-category, because there is
	at most one 2-arrow between two arrows, since the objects are connected,
	and that, dually, $\DisC$ is also a $\Ens$-category).
	
	\begin{pon}
		For a 2-Puppe-exact $\Gpdp$-category $\C$,
		\begin{enumerate}
			\item $\pi_1\col\C\ra\ConC$ is right adjoint
				to the inclusion
				$i\col\ConC\hookrightarrow\C$, with counit
				$\varepsilon\col\pi_1\Ra 1_{\C}$;
			\item $\pi_0\col\C\ra\DisC$ is left adjoint to the inclusion
				$i\col\DisC\hookrightarrow\C$, with unit
				$\eta\col 1_{\C}\Ra \pi_0$.
		\end{enumerate}
	\end{pon}
	
		\begin{proof}
			We prove point 1; point 2 is dual.
			We will prove that for all $C\col\ConC$ and $A\col\C$, the
			functor $\varepsilon_A\circ -\col\ConC(C,\pi_1A)\ra\C(C,A)$
			is an equivalence (then we get the counit of the adjunction by taking
			$C\eqdef\pi_1 A$ and by applying this functor $1_{\pi_1A}$; in this way,
			we do get $\varepsilon_A$).
			Let us consider the following diagram.  Since $C$ is connected,
			$\varepsilon_C$ is an equivalence (by the dual of Proposition
			\ref{caracdispreadd}), thus the left horizontal arrows are
			equivalences.  The right horizontal arrows are
			the equivalences of the adjunction $\Sigma\adj\Omega$.  The right vertical arrow
			is an equivalence because the adjunction $\Sigma\adj\Omega$
			is idempotent (since $\C$ is $\Ker$-idempotent, by 2-Puppe-exactness).
			So the left arrow is an equivalence.\qedhere
			\begin{xym}\xymatrix@R=40pt@C=30pt{
				{\ConC(C,\Sigma\Omega A)}\ar[d]_{\varepsilon_A\circ-}
					\ar[r]^-{-\circ\varepsilon_C}
				&{\ConC(\Sigma\Omega C,\Sigma\Omega A)}\ar[d]_{\varepsilon_A\circ-}
					\drtwocell\omit\omit{{}}\ar[r]^-{\sim}
				&{\DisC(\Omega C,\Omega\Sigma\Omega A)}\ar[d]^{\Omega\varepsilon_A\circ -}
				\\ {\C(C,A)}\ar[r]_-{-\circ\varepsilon_C}
				&{\C(\Sigma\Omega C,A)}\ar[r]_-{\sim}
				&{\C(\Omega C,\Omega A)}
			}\end{xym}
		\end{proof}
	
	To sum up, in a 2-Puppe-exact $\Gpdp$-category, we are in the situation
	 of the following diagram.
	
	\begin{xym}\label{diagequcondis}\xymatrix@=50pt{
         	{\C}\ar@<-2mm>[r]_-{\Omega}\ar@{}[r]|-\perp
				\ar@<-2mm>[d]_{\pi_1}\ar@{}[d]|-\vdash
			&{\C}\ar@<-2mm>[l]_-{\Sigma}\ar@<-2mm>[d]_{\pi_0}\ar@{}[d]|-\dashv
         	\\ {\ConC}\ar@<-2mm>[r]_-{\Omega}
				\ar@{}[r]|-\simeq\ar@<-2mm>[u]_i
			&{\DisC}\ar@<-2mm>[l]_-{\Sigma}\ar@<-2mm>[u]_i
    }\end{xym}

Let us recall \cite[Section 1.12]{Borceux1994a} that a \emph{torsion theory}%
\index{torsion theory}\index{theory, torsion} in an abelian
category $\C$ consists of two full subcategories $\mc{T}$, $\mc{F}$ of $\C$ such that
\begin{enumerate}
	\item for all $T\in\mc{T}$ and $F\in\mc{F}$, 
		\begin{eqn}
			\C(T,F)\simeq 0;
		\end{eqn}
	\item for each $C\col\C$, there exists an extension
		\begin{eqn}
			0\longrightarrow T\longrightarrow C\longrightarrow F\longrightarrow 0,
		\end{eqn}
		where $T\in\mc{T}$ and $F\in\mc{F}$.
\end{enumerate}
We will use provisionally the same definition in dimension 2.  But it is possible that in a later study of 2-dimensional torsion theories we will need to add conditions to get the expected properties.

\begin{pon}
	In a 2-Puppe-exact $\Gpdp$-category $\C$, the sub-$\Gpdp$-ca\-te\-go\-ries
	$\ConC$ and $\DisC$ form a torsion theory.
\end{pon}

	\begin{proof}
		The first condition hold because, if $C$ is connected and $D$ is discrete,
		by the adjunction $i\adj\pi_1\col\C\ra\ConC$ and by the fact
		that $\pi_1D\simeq 0$ (by Proposition \ref{caracdispreadd}),
		we have $\C(C,D)\simeq\ConC(C,\pi_1D)\simeq\ConC(C,0)\simeq 0$.
		
		The second condition follows from Proposition \ref{extenspizpiu},
		since $\pi_1 C$ is connected and $\pi_0 C$ is discrete.
	\end{proof}

To end this section, let us give two other consequences of Proposition \ref{caracdispreadd}.

	\begin{pon}\label{depupexpippreex}
		Every 2-Puppe-exact $\Gpdp$-category $\C$ is
		$\Pip$-preexact and $\Copip$-preexact.
	\end{pon}
	
		\begin{proof}
			We prove the $\Pip$-preexactness; the other property is dual.
			Let $\pi\col 0\Ra 0\col A\ra B$ be a monoloop in $\C$.
			Let us consider the following diagram, where $r=\Coroot(\pi)$.  By Proposition
			\ref{coraccoker}, there exists a 2-arrow $\rho\col r\bar{\pi}\Ra 0$
			(unique since $\Sigma A$ is connected) such that $(r,\rho)=\Coker \bar{\pi}$.
			Moreover, as $\pi$ is a monoloop, its domain $A$ is discrete
			(by Proposition \ref{domonobodis}) and thus $\eta_A$ is an equivalence
			(by Proposition \ref{caracdispreadd}).  Then, by Proposition
			\ref{lemmhormzfid}, $\bar{\pi}$ is 0-faithful and, by 2-Puppe-exactness,
			$(\bar{\pi},\rho)=\Ker r$.  Finally, by Proposition \ref{kerpep},
			$\pi=\Pip r$.\qedhere
			\begin{xym}\xymatrix@=50pt{
				&A\rtwocell^0_0{\pi}\dtwocell^0_0{\sigma_A}
				&B\ar[r]^-r
				&{\Coroot\pi}
				\\ \Omega\Sigma A\ar[ur]^-{\eta_A^{-1}}
					\rtwocell^0_0{\;\;\;\;\;\,\omega_{\Sigma A}}
				& \Sigma A\ar@<-0.7mm>[ur]_-{\bar{\pi}}\ar@/_0.5pc/[urr]_0
					\urrtwocell\omit\omit{_<-1.7>\rho}
			}\end{xym}
		\end{proof}

	\begin{pon}\label{fidconnectedplus}
		Let $\C$ be a 2-Puppe-exact $\Gpdp$-category and $A\overset{f}\ra B$ be an arrow
		in $\C$.  Then the following conditions are equivalent:
		\begin{enumerate}
			\item $f$ is faithful;
			\item for every $C\col\ConC$ and for every $a\col C\ra A$, if $fa\simeq 0$,
				then $a\simeq 0$;
			\item for every $C\col\ConC$, $\C(C,f)$ is fully 0-faithful.
		\end{enumerate}
	\end{pon}
	
		\begin{proof}
			{\it 1 $\Leftrightarrow$ 2. } 
			This is an immediate consequence of Proposition \ref{fidconnected}
			since, if $D$ is connected, $D\simeq\pi_1 D = \Sigma\Omega D$,
			by Proposition \ref{caracdispreadd}.
			
			{\it 2 $\Leftrightarrow$ 3. } Condition 2 is part (a) of property 4 of
			Definition \ref{caracplzfid}.
			Part (b) is always true when
			$X$ is connected, which is the case here.
		\end{proof}

\subsection{2-Puppe-exactness of $\DisC\simeq\ConC$}

	We will prove that, if $\C$ is 2-Puppe-exact (or 2-abelian),
	then the $\Ens$-category $\DisC\simeq\ConC$ is Puppe-exact (or abelian).
	In order to do that, we use the following facts:
	\begin{enumerate}
		\item since $\DisC$ is a reflective sub-$\Gpd$-category of $\C$,
			limits in it are computed as in $\C$ (in particular, the kernel
			of an arrow in $\DisC$ is its kernel in $\C$) and colimits
			are computed by applying $\pi_0$ to the colimit in $\C$ (in particular,
			the cokernel of $A\overset{f}\ra B$ is the composite
			\begin{eqn}
				B\xrightarrow{q_f}\Coker f\xrightarrow{\eta_{\Coker f}}\pi_0\Coker f\text{);}
			\end{eqn}
		\item dually, colimits in $\ConC$ are computed as in $\C$
			and limits are computed by applying $\pi_1$ to the limit in $\C$.
	\end{enumerate}

	\begin{thm}\label{puppepuppe}
		If $\C$ is a 2-Puppe-exact $\Gpdp$-category, then $\DisC\simeq\ConC$
		is a Puppe-exact category.
	\end{thm}
	
		\begin{proof}
			First, $\DisC$ has a zero object, the kernels and the cokernels, by the remark
			preceding this theorem.  Next, let $A\overset{f}\ra B$ be an arrow in
			$\DisC$.  The cokernel of the kernel of $f$ in $\DisC$ is
			$\eta_{\Coker k_f}\bar{e}^1_f$; the kernel of the cokernel of $f$ is $\bar{m}^2_f$,
			which is the kernel of $\eta q_f$ (the cokernel of $f$ in $\DisC$).  There is a
			comparison arrow $w_f\col\pi_0\Coker k_f\ra \Ker (\eta q_f)$ such
			that $w_f\eta_{\Coker k_f}\simeq \bar{w}_f$.
			
			Then, as $\C$ is 2-Puppe-exact, $\bar{w}_f$ is an equivalence.
			As $\Ker(\eta q_f)$ is discrete, $\Coker k_f$ is also discrete
			and, by Proposition \ref{caracdispreadd}, $\eta_{\Coker k_f}$ is
			an equivalence. This allows to conclude that $w_f$ is an equivalence.\qedhere
			\begin{xym}\xymatrix@=30pt{
				{\Ker f}\ar[r]^-{k_f}
				& A\ar[r]^f\ar[d]_{\bar{e}^1_f}\rtwocell\omit\omit{_<5>\;\;\,\bar{\omega}_f}
				& B\ar[r]^-{q_f}
				&{\Coker f}\ar[r]^{\eta}
				&{\pi_0\Coker f}
				\\ &{\Coker k_f}\ar[r]_-{\bar{w}_f}\ar[d]_{\eta}
				&{\Ker(\eta q_f)}\ar[u]_{\bar{m}^2_f}
				\\ &\pi_0\Coker k_f\ar[ur]_{w_f}
			}\end{xym}
		\end{proof}

	\begin{coro}\label{discconcabel}
		If $\C$ is a 2-abelian $\Gpd$-category, then $\DisC$ (and thus also $\ConC$)
		is an abelian category.
	\end{coro}
	
		\begin{proof}
			By the remarks preceding the previous theorem,
			the category $\DisC$ has all finite products and coproducts,
			because $\C$ has them.  And by the previous theorem $\DisC$
			is Puppe-exact.
		\end{proof}

	We can also characterise the monomorphisms and epimorphisms in $\DisC$
	and $\ConC$ in terms of the (fully) (co)faithful arrows in $\C$.

	\begin{pon}\label{epimonconc}
		Let $\C$ be a 2-Puppe-exact $\Gpdp$-category.  
		We have the following equivalences:
		\begin{enumerate}
			\item $f$ is a monomorphism in $\DisC$ if and only if $f$ is
				fully faithful in $\C$;
			\item $f$ is an epimorphism in $\DisC$ if and only if $f$ is
				cofaithful in $\C$;
			\item $f$ is a monomorphism in $\ConC$ if and only if $f$ is
				faithful in $\C$;
			\item $f$ is an epimorphism in $\ConC$ if and only if $f$ is
				fully cofaithful in $\C$.
		\end{enumerate}
	\end{pon}
	
		\begin{proof}
			We give the proof for $\DisC$; the proof for $\ConC$ is dual.
			Let $A\overset{f}\ra B$ be an arrow in $\DisC$.
			
			First, $f$ is a monomorphism in $\DisC$ if and only if
			$\Ker f=0$ in $\DisC$.  Since the kernel in $\DisC$ is computed as in $\C$,
			this is equivalent to $\Ker f= 0$ in $\C$ and thus to $f$ being
			 fully faithful in $\C$.
			
			Next, $f$ is an epimorphism in $\DisC$ if and only if
			$\Coker f= 0$ in $\DisC$.  Since the cokernel of $f$ in $\DisC$ is
			$\pi_0\Coker f$ in $\C$ and since $\pi_0\Coker f =  \Omega\Copip f$, this
			is equivalent to $\Copip f= 0$ in $\C$ and thus to $f$ being
			cofaithful in $\C$.
		\end{proof}

\section{Good 2-Puppe-exact $\Gpdp$-categories}\label{sectbondpex}

	In this section, we study a property which, added to the definition of 2-Puppe-exactness, allows to recover certain properties of the
	$\Gpd$-category $\CGS$ of symmetric 2-groups.  The question of the independance
	of this property with respect to the axioms of 2-Puppe-exact $\Gpd$-category
	remains open, all known examples of 2-Puppe-exact $\Gpd$-categories being good.

\subsection{Good 2-Puppe-exact $\Gpdp$-categories and exactness of $\mu_f$}
	
	Let us begin by introducing some notation.  Let $\C$ be a 2-Puppe-exact $\Gpdp$-category
	and $f\col A\ra B$ be an arrow of $\C$.
	Let us construct the kernel $\Ker\pi_0 f$ in $\DisC$ (which is also the kernel of $\pi_0f$
	in $\C$, since the inclusion $\DisC\hookrightarrow\C$ is a right adjoint) and
	let us denote by $a_f\col \pi_0\Ker f\ra\Ker\pi_0 f$ the comparison arrow
	such that $k_{\pi_0 f}\circ a_f = \pi_0(k_f)$; we set
	$c_f\eqdef a_f\circ\eta_{\Ker f}\col\Ker f\ra\Ker\pi_0 f$.
	We denote by $b_f\col \Coker\pi_1 f\ra\pi_1\Coker f$ the comparison arrow in
	$\ConC$ constructed dually and $d_f\eqdef \varepsilon_{\Coker f}\circ b_f\col
	\Coker\pi_1 f\ra\Coker f$.
	\begin{xym}\label{diagbonnquatr}\xymatrix@=40pt{
		{Kf}\ar[r]^-{k_f}\ar[d]_{\eta_{K f}}\drtwocell\omit\omit{}\ddlowertwocell_{c_f}<-12>{\omit}
		&A\ar[r]^f\ar[d]_{\eta_A}\drtwocell\omit\omit{}
		&B\ar[d]^{\eta_B}
		\\ {\pi_0 K f}\ar[r]_-{\pi_0 k_{f}}\ar[d]_{a_f}
		&\pi_0 A\ar[r]_{\pi_0 f}
		&\pi_0 B
		\\ {K\pi_0 f}\ar[ur]_-{k_{\pi_0 f}}
	}\end{xym}	
	
	A good 2-Puppe-exact $\Gpdp$-category is a 2-Puppe-exact $\Gpdp$-category
	where $\pi_0$ and $\pi_1$ are exact, in the sense that they preserve exact sequences.
	On the other hand, they do not preserve in general relative exact sequences. This would be equivalent to the preservation of kernels and cokernels; but, in general,
	$\pi_0$ preserves cokernels but not kernels and $\pi_1$ preserves
	kernels, but not cokernels.
	
	These properties hold in the $\Gpd$-category of symmetric 2-groups.
	Conditions 1, 2 and 3 are proved in \cite{Vitale2002a}.
	
	\begin{pon}\label{defcaracgood}
		Let $\C$ be a 2-Puppe-exact $\Gpdp$-category.  The following conditions
		are equivalent.  When they hold, we say that $\C$
		is a \emph{good 2-Puppe-exact $\Gpdp$-category}.%
		\index{good 2-Puppe-exact Gpd*-category@good 2-Puppe-exact $\Gpdp$-category}%
		\index{2-Puppe-exact $\Gpdp$-category!good}%
		\index{Gpd*-category@$\Gpdp$-category!2-Puppe-exact!good}
		\begin{enumerate}
			\item $\pi_0\col\C\ra\DisC$ and $\pi_1\col\C\ra\ConC$
				preserve exact sequences.
			\item $\pi_0(\PlFid)\incl\caspar{Mono}$ and
				$\pi_1(\PlCofid)\incl\caspar{Epi}$.
			\item For each $f$, $a_f\col \pi_0\Ker f\ra\Ker\pi_0 f$ is an epimorphism
				in $\DisC$ and
				$b_f\col \Coker\pi_1 f\ra\pi_1\Coker f$ is a monomorphism in $\ConC$.
			\item For each $f$, $c_f$ is cofaithful and $d_f$ is faithful.
		\end{enumerate}
	\end{pon}

		\begin{proof}
			{\it 1 $\Rightarrow$ 3. } We apply $\pi_0$ to the sequence
			\begin{xym}\xymatrix@=40pt{
				{K f}\ar[r]_-{k_f}\rruppertwocell<9>^0{^<-2.7>\kappa_f\;\;}
				&A\ar[r]_-f &B,
			}\end{xym}
			which is exact, by Proposition \ref{qquessequencex}. By condition
			1, the central row of diagram \ref{diagbonnquatr}
			is thus an exact sequence in the Puppe-exact category $\DisC$.
			By the properties of exact sequences in Puppe-exact categories,
			$a_f$ is an epimorphism in $\DisC$.  The proof is dual for $b_f$.
			
			{\it 3 $\Rightarrow$ 2. }  If $f$ is fully faithful,
			then $\Ker f\simeq 0$, thus $\pi_0\Ker f\simeq 0$ and $a_f$
			is a monomorphism.  By condition 3, $a_f$ is also an epimorphism.
			It is thus un isomorphism in the Puppe-exact category $\DisC$.
			Therefore $\Ker\pi_0 f\simeq 0$ and $\pi_0 f$ is a monomorphism in $\DisC$.
			The proof is dual for $\pi_1$.
			
			{\it 2 $\Rightarrow$ 1. }  Let be an exact sequence in $\C$, as in the upper
			part of diagram \ref{diagsuitdex}.  By Corollary
			\ref{caracdexbis} (which applies because $\C$ is $\Ker$-factorisable), the arrow
			$b'\col\Coker a\ra C$ is fully faithful.  Then, by condition 2,
			the comparison arrow
			\begin{eqn}
				\Coker\pi_0 a\overset{\sim}\longrightarrow
				 \pi_0\Coker a\overset{\pi_0 b'}\longrightarrow\pi_0 C
			\end{eqn}
			is a monomorphism, thus the sequence $(\pi_0 a,\pi_0 b)$
			is exact in $\DisC$.  The proof of the second
			part of condition 1 is dual.
						
			{ 3 $\Leftrightarrow$ 4. }  
			Given that $c_f\equiv a_f\circ\eta_{\Ker f}$ and that
			$\eta_{\Ker f}$ is fully cofaithful, since it is a coroot
			(by Proposition \ref{etacoracepsrac}), by the cancellation property 
			of cofaithful arrows, $c_f$ is cofaithful if
			and only if $a_f$ is cofaithful.  And, by Proposition \ref{epimonconc},
			$a_f$ is cofaithful in $\C$ if and only if $a_f$ is an epimorphism
			in $\DisC$.
		\end{proof}

	At the end of the following subsection, we will learn that
	in a good 2-Puppe-exact $\Gpdp$-category, we have also
	\begin{itemize}
		\item $\pi_0(\Cofid)\incl\caspar{Epi}$,
		\item $\pi_1(\Fid)\incl\caspar{Mono}$,
		\item $\pi_0(\PlCofid)\incl\Iso$, and
		\item $\pi_1(\PlFid)\incl\Iso$.
	\end{itemize}

	An important property of good 2-Puppe-exact $\Gpdp$-categories is that,
	for every arrow $f\col\C$, the 2-arrow $\mu_f\eqdef$ 
	\begin{xym}\xymatrix@=40pt{
		{\Ker f}\ar[r]_-{k_f}\rrlowertwocell<10>^0{\;\;\;\;\,\kappa_f^{-1}}
		&A\ar[r]^f\rruppertwocell<-10>_0{_{\;\;\,\zeta_f}}
		&B\ar[r]^-{q_f}
		&{\Coker f}
	}\end{xym}
	is exact. Let us recall that a 2-arrow $\pi\col 0\Ra 0\col A\ra B$
	is exact if the sequence $A\longrightarrow 0\longrightarrow B$, equipped with the 2-arrow
	$\pi$, is exact.
	
	Starting from a loop $\pi\col 0\Ra 0\col A\ra B$, we construct the following diagram,
	by using the notations of diagram \ref{diagnotationsomegsigm}.
	\begin{xym}\label{loopxacte}\xymatrix@R=40pt@C=20pt{
		A\ar[rr]^0\ar[dr]_-{\tilde{\pi}}\ar@/^2pc/[rrrr]^0\rrtwocell\omit\omit{_<4>{}}
		& {}\rruppertwocell\omit{^<-2.7>\pi}
		& 0\ar[rr]^0\ar[dr]_-{0}\rrtwocell\omit\omit{_<4>{}}
		& {}
		& B
		\\ &{\Omega B}\ar[ur]_-{0}
		&&{\Sigma A}\ar[ur]_-{\bar{\pi}}
	}\end{xym}

	Since 2-Puppe-exact $\Gpdp$-categories are $\Ker$- and $\Coker$-factorisable,
	we can apply Corollary \ref{caracdexbis} to the exactness of $\pi$.
	The following conditions are thus equivalent:
		\begin{enumerate}
			\item $\pi$ is exact;
			\item $\tilde{\pi}$ is fully 0-cofaithful;
			\item $\bar{\pi}$ is fully 0-faithful.
		\end{enumerate}

The following lemma appears in \cite{Grandis1994a}.

\begin{lemm}
	Let $\C$ be a $\Gpdp$-category and $A\overset{f}\ra B$ be an arrow in $\C$. Let us consider
	the Puppe exact sequence constructed from $f$ (Proposition \ref{sequencePuppe}).
	There exist 2-arrows (unique because $\Sigma\Omega B$ and $\Sigma Kf$ are connected),
	shown in the following diagram.
	\begin{xym}\xymatrix@=40pt{
		\Sigma\Omega B\ar[r]^-{\Sigma d}\ar[d]_{\varepsilon_B}\drtwocell\omit\omit{}
		&\Sigma Kf\ar[r]^{\Sigma k}\ar[d]^{\bar{\mu}_f}\drtwocell\omit\omit{}
		&\Sigma A\ar@{=}[d]
		\\ B\ar[r]_q
		&Qf\ar[r]_-{d'}
		&\Sigma A
	}\end{xym}
\end{lemm}

	\begin{proof}
		For the 2-arrow of the left square, it suffices, as $\sigma_{\Omega B}$ is
		an epiloop, to prove that $q\varepsilon_B\sigma_{\Omega B} 
		= \bar{\mu}_f(\Sigma d)\sigma_{\Omega B}$.
		By using the definition of $\bar{\mu}_f$ and the equations which
		the 2-arrows of the Puppe sequence satisfy, we have the succession of equalities
		\begin{eqn}
			q\varepsilon_B\sigma_{\Omega B} = q\omega_B = qf\delta\circ q\kappa^{-1} d
			= \zeta kd\circ q\kappa^{-1} d = \mu_f d = \bar{\mu}_f\sigma_{Kf}d
			=\bar{\mu}_f(\Sigma d)\sigma_{\Omega B}.
		\end{eqn}
		
		To get the 2-arrow of the right square, we check in a similar way that
		$d'\bar{\mu}_f\sigma_{Kf}=(\Sigma k)\sigma_{Kf}$.
	\end{proof}
	
	\begin{pon}\label{mufestexact}
		If $\C$ is a good 2-Puppe-exact $\Gpdp$-category, the 2-arrow
		$\mu_f\col 0\Ra 0\col \Ker f\ra\Coker f$ is exact.
	\end{pon}
	
		\begin{proof}
			We apply $\pi_0\col\C\ra\DisC$ and $\Omega\col\C\ra\DisC$
			to the Puppe exact sequence (Proposition \ref{sequencePuppe}).
			Given that $\pi_0$ maps connected objects to zero
			and $\Omega$ maps discrete objects to zero, we get two shorter sequences
			that we can merge in the following way.
			\begin{xym}\label{diagpourmuex}\xymatrix@C=12pt@R=25pt{
				0\ar[r]
				&{\pi_0\Omega Kf}\ar[r]\ar[d]_{\Omega\varepsilon_{Kf}}^\wr
				&\pi_0\Omega A\ar[r]\ar[d]_{\Omega\varepsilon_{A}}^\wr
				&\pi_0\Omega B\ar[r]\ar[d]_{\Omega\varepsilon_{B}}^\wr
				&{\pi_0 Kf}\ar[r]\ar[d]^{\Omega \bar{\mu}_f}
				&\pi_0 A\ar[r]\ar@{=}[d]
				&\pi_0 B\ar[r]\ar@{=}[d]
				&{\pi_0 Qf}\ar[r]\ar@{=}[d]
				&0
				\\ 0\ar[r]
				&{\Omega Kf}\ar[r]
				&\Omega A\ar[r]
				&\Omega B\ar[r]
				&{\Omega Qf}\ar[r]
				&\Omega\Sigma A\ar[r]
				&\Omega\Sigma B\ar[r]
				&{\Omega\Sigma Qf}\ar[r]
				&0
			}\end{xym}
			Since $\C$ is good, $\pi_0$ and $\Omega$ (which is the
			composite of $\pi_1\col \C\ra\ConC$ with the equivalence $\Omega\col \ConC\ra\DisC$)
			preserve exact sequences, thus the two rows are exact
			in $\DisC$. The two left squares commute by $\Gpd$-naturality of $\varepsilon$, the two
			central squares commute by the previous lemma, and the two right
			squares obviously commute.
			
			The three left vertical arrows are isomorphisms,
			because the adjunction $\Sigma\adj\Omega$ is idempotent, and the three right
			vertical arrows are obviously isomorphisms.
			Since $\DisC$ is Puppe-exact, the “7 lemma” (which can be proved from the
			5 lemma by factoring the second and fifth arrows of each
			sequence) holds and ensures that $\Omega\bar{\mu}_f$ is an isomorphism.
			
			Therefore, in the following diagram (which exists by $\Gpd$-naturality
			of $\varepsilon$), the left and upper sides are
			equivalences, whereas $\varepsilon_{Qf}$ is fully faithful,
			since it is a root. So $\bar{\mu}_f$ is fully faithful
			and $\mu_f$ is exact.\qedhere
			\begin{xym}\xymatrix@=40pt{
				\Sigma\Omega\Sigma Kf\ar[r]^-{\varepsilon_{\Sigma Kf}}
					\ar[d]_{\Sigma\Omega\bar{\mu}_f}\drtwocell\omit\omit{^{}}
				&\Sigma Kf\ar[d]^{\bar{\mu}_f}
				\\ \Sigma\Omega Qf\ar[r]_-{\varepsilon_{Qf}}
				&Qf
			}\end{xym}
		\end{proof}

	Diagram \ref{diagpourmuex} gives us an exact sequence (the upper or lower row, as you prefer) in $\DisC$ which is the generalisation
	of the following sequence in $\CGS$, constructed from a symmetric monoidal functor
	$F\col \A\ra\B$ \cite[Corollary 2.7]{Bourn2002a}.
		\begin{multline}\stepcounter{eqnum}
			0\ra \pi_1(\Ker F)\longrightarrow\pi_1(\A)\overset{\pi_1(F)}
			\longrightarrow\pi_1(\B) \longrightarrow\pi_1(\Coker F)\\
			\simeq\pi_0(\Ker F) \longrightarrow\pi_0(\A)\overset{\pi_0(F)} \longrightarrow
			\pi_0(\B) \longrightarrow\pi_0(\Coker F)\ra 0.
		\end{multline}
	In particular, if we start with an extension $(f,\eta,g)$ (like the upper row of
	diagram \ref{diagpetitlemmcinq}), we get the following exact sequence in $\DisC$,
	where the central arrow measures the non-preservation by $\pi_0$ and $\pi_1$ (or $\Omega$)
	of relative exact sequences:
	\begin{eqn}\label{suitdiscextens}
		0\longrightarrow \Omega A\overset{\Omega f}\longrightarrow \Omega B\overset{\Omega g}
		\longrightarrow\Omega C\longrightarrow\pi_0A\overset{\pi_0f}\longrightarrow
		\pi_0B\overset{\pi_0g}\longrightarrow \pi_0C\longrightarrow 0.
	\end{eqn}
	It corresponds under the equivalence $\Sigma\adj\Omega$ to the following exact sequence 
	in $\ConC$:
	\begin{eqn}
		0\longrightarrow \pi_1 A\overset{\pi_1 f}\longrightarrow \pi_1 B\overset{\pi_1 g}
		\longrightarrow\pi_1 C\longrightarrow\Sigma A\overset{\Sigma f}\longrightarrow
		\Sigma B\overset{\Sigma g}\longrightarrow \Sigma C\longrightarrow 0.
	\end{eqn}
	
	We also get characterisations of the $\pi_0$ and the $\pi_1$ of the kernel and of the 
	cokernel of an arrow $f$.  The first and the third properties are always
	true in a 2-Puppe-exact $\Gpdp$-category (for $\pi_0$ is a left adjoint
	and $\pi_1$ a right adjoint), whereas the second is a translation of the fact that
	$\Omega\bar{\mu}_f$ is an equivalence, as we have seen in the proof of the
	previous proposition.
	\begin{enumerate}
		\item $\pi_0\Coker f\simeq\Coker\pi_0 f$.
		\item $\pi_1\Coker f$ is equivalent, under the equivalence $\Sigma\adj\Omega$,
			to $\pi_0\Ker f$.
		\item $\pi_1\Ker f\simeq\Ker\pi_1 f$.
	\end{enumerate}
	These three properties give the three levels of a 3-dimensional object:
	the objects (1) are the objects of the cokernel of $f$; the arrows (2)
	are both the arrows (from $0$ to $0$) of the cokernel of $f$ and the objects
	of the kernel of $f$; the 2-arrows (3) are the arrows of the kernel of $f$.
	In the case of symmetric 2-groups, we can be more specific: the kernel of a
	morphism $F\col\A\ra\B$ combines with its cokernel to form a pointed 2-groupoid
	$\Ker(F\con)$, which is the kernel of the 2-functor $F\con\col\A\con\ra\B\con$ obtained
	by “suspending” the functor $F$ ($\A\con$ is the one-object 2-groupoid whose
	2-group of arrows is $\A$).  We recover the cokernel of $F$
	by taking locally the $\pi_0$ of $\Ker(F\con)$, and we recover the kernel
	of $F$ by taking the 2-group of arrows from $I$ to $I$ in $\Ker(F\con)$.
	
\subsection{Full arrows}\label{sectflechplen}

	Now, let us introduce a notion of full arrow in a $\Gpd$-category
	(which is called a \emph{prefull} arrow in \cite{Dupont2003a}).
	In $\Gpd$ (see \cite{Dupont2008b}),
	$\CGS$ (Proposition \ref{caracfullcgs}),
	and the $\Gpd$-categories of 2-modules on a 2-ring,
	the full arrows in the 2-categorical sense are the full functors in the ordinary
	elementary sense. On the other hand, this is not the case in the 2-category of categories
	$\Cat$, but here we work only with $\Gpd$-categories.

	\begin{df}\label{deffullgpdcat}\index{full arrow}\index{arrow!full}
		Let $A\overset{f}\ra B$ be an arrow in a $\Gpd$-category.
		We say that $f$ is \emph{full}
		if, for all 2-arrows $\alpha$ and $\beta$ as in the following diagram,
		the following equation holds.
		\begin{xym}\label{deffull}\xymatrix@!{
			&&B\ar[dr]^{v_0}
			&&&A\ar[dr]^f
			\\ &A\ar[ur]^f\ar[dr]_f\rrtwocell\omit\omit{\beta}
			&&Y
			&X\ar[ur]^{u_0}\ar[dr]_{u_1}\rrtwocell\omit\omit{\alpha}
			&&B\ar[dr]^{v_0}
			\\ X\ar[ur]^{u_0}\ar[dr]_{u_1}\rrtwocell\omit\omit{\alpha}
			&&B\ar[ur]_{v_1} \ar@{}[rrr]|-{}="a" \save "a"*{=}\restore
			&&&A\ar[ur]^f\ar[dr]_f\rrtwocell\omit\omit{\beta}
			&&Y
			\\ &A\ar[ur]_f
			&&&&&B\ar[ur]_{v_1}
		}\end{xym}
	\end{df}
	
	We can define a pointed version, by taking $u_1=0$ and $v_0=0$. 
	
	\begin{df}\index{0-full arrow}\index{arrow!0-full}
		Let $A\overset{f}\ra B$ be an arrow in a $\Gpd$-category.
		We say that $f$ is \emph{0-full}
		if for all 2-arrows $\alpha$ and $\beta$ as in the following diagram,
		the following equation holds (i.e.\ $\alpha$ and $\beta$ are
		compatible: $v\alpha = \beta u$).
		\begin{xym}\label{deffullpoint}\xymatrix@=40pt{
			X\ar[r]^-{u}\rrlowertwocell<-10>_0{\alpha}
			&A\ar[r]^f\rruppertwocell<10>^0{_{\;\;\;\;\beta^{-1}}}
			&B\ar[r]_-{v}
			&Y
		}\;\;\;=\;\;\;\;\, 1_0\end{xym}
	\end{df}

		In 2-Puppe-exact $\Gpdp$-categories, full arrows behave
		as we expect: fully faithful = full + faithful.
							
		\begin{pon}\label{plmentfideplpamentfid}
			In a 2-Puppe-exact $\Gpdp$-category, $f$ is fully faithful
			if and only if $f$ is full and faithful.  Dually, $f$ is
			fully cofaithful if and only if $f$ is full and cofaithful.
		\end{pon}
		
			\begin{proof}
				Let us assume that $f$ is fully faithful. We already know that $f$
				is faithful.  It remains to prove that $f$ is full.  Let us
				consider diagram \ref{deffull}.  Since $f$ is fully
				faithful, there exists $\gamma\col u_0\Ra u_1$ such that $\alpha = f\gamma$.
				Then the two sides of equation \ref{deffull} are equal to
				$\beta *\gamma$ and are thus equal to each other.
				
				Conversely, let us assume that $f$ is full and faithful.
				Let be $\beta\col fa\Ra 0$. Since $f$
				is full, $\beta$ is compatible with $\zeta_f$:
				\begin{xym}\xymatrix@=40pt{
					X\ar[r]^-{a}\rrlowertwocell<-10>_0{\beta}
					&A\ar[r]^f\rruppertwocell<10>^0{_{\;\;\;\;\,\zeta^{-1}_f}}
					&B\ar[r]_-{q_f}
					&{\Coker f}
				}\;\;\;=\;\;\;\;\, 1_0.\end{xym}
				Now, $(f,\zeta_f)=\Ker q_f$, since $f$ is faithful and $\C$ is 2-Puppe-exact.
				So, by the universal property of the kernel of $q_f$, there exists a unique
				$\alpha\col a\Ra 0$ such that $\beta = f\alpha$.
			\end{proof}
			
		We proceed to the main result of this section, which gives
		in a \emph{good} 2-Puppe-exact $\Gpdp$-category a list of
		characterisations of full arrows, some of them being known
		in $\CGS$ (see \cite{Vitale2002a,Kasangian2000a}).
	
	\begin{lemm}\label{lempourcaracfull}
		Let $\pi\col 0\Ra 0\col A\ra B$ be an exact loop.
		The following conditions are equivalent:
		\begin{enumerate}
			\item $\pi=1_0$;
			\item $A$ is connected;
			\item $B$ is discrete.
		\end{enumerate}
	\end{lemm}
	
		\begin{proof}
			{\it 2 or 3 $\Rightarrow$ 1. }
			If $A$ is connected or $B$ discrete, there is only one 2-arrow
			$0\Ra 0\col A\ra B$, which can only be $1_0$.
			
			{\it 1 $\Rightarrow$ 2. }
			Since $\pi$ is exact, $\tilde{\pi}$ (see diagram \ref{loopxacte})
			is fully cofaithful.  But $\pi=\omega_B\tilde{\pi}$. If $\pi = 1_0$,
			we have thus $\omega_B=1_0$, which implies that $\Omega B\simeq 0$, by
			Proposition \ref{pepclaszfid} (since $\omega_B=\Pip 0^B$).
			So $\pi_1B=\Sigma\Omega B\simeq 0$ and $B$ is discrete, by Proposition
			\ref{caracdispreadd}.
			
			{\it 1 $\Rightarrow$ 3. } The proof is dual.
		\end{proof}

		\begin{pon}\label{caracfullbondpex}
			Let $\C$ be a good 2-Puppe-exact $\Gpdp$-category and $A\overset{f}\ra B$
			be an arrow in $\C$.
			The following conditions are equivalent:
			\begin{enumerate}
				\item $f$ is full;
				\item $f$ is 0-full;
				\item $\mu_f=1_0$;
				\item $\Ker f$ is connected ($\pi_0\Ker f\simeq 0$);
				\item $\Coker f$ is discrete ($\pi_1\Coker f\simeq 0$);
				\item $f$ factors as a fully cofaithful arrow followed by
					a fully faithful arrow;
				\item $l_f\col\im f\ra\im\pl f$ is an equivalence;
				\item $\pi_0 f$ is a monomorphism in $\DisC$
					and $\pi_1 f$ is an epimorphism in $\ConC$.
			\end{enumerate}
		\end{pon}
		
			\begin{proof}
				{\it 1 $\Rightarrow$ 2 $\Rightarrow$ 3. }
				Condition 3 is condition 2 applied to $\kappa_f$ and $\zeta_f$,
				whereas condition 2 is a special case of condition 1.
				
				{\it 3 $\Leftrightarrow$ 4 $\Leftrightarrow$ 5. }
				By Proposition \ref{mufestexact}, $\mu_f$ is exact because $\C$ is good.
				Thus the previous lemma applies and gives us the equivalence between
				these properties.
				
				{\it 4 $\Rightarrow$ 6. }
				As $\C$ is 2-Puppe-exact, we have a
				factorisation $f\simeq \hat{e}_f\circ \hat{m}_f$,
				where $\hat{e}_f$ is cofaithful and $\hat{m}_f$ is fully
				faithful.  By construction, $\hat{e}_f$ is the cokernel of $\Ker f$.
				Since $\Ker f$
				is connected, by the dual of Lemma \ref{noycodomdisplfid},
				$\hat{e}_f$ is fully cofaithful.

				{\it 6 $\Rightarrow$ 1. } Let us assume that we have $\varphi\col f\Ra me$, where $e$
				is fully cofaithful and $m$ is fully faithful and let us consider the
				situation of diagram \ref{deffull}. Since $m$ is fully faithful,
				there exists $\gamma\col eu_0\Ra eu_1$ such that
				$\alpha = \varphi^{-1} u_1\circ m\gamma\circ\varphi u_0$ and, since $e$
				is fully cofaithful, there exists
				$\delta\col v_0m\Ra v_1m$ such that $\beta= v_1\varphi^{-1}\circ\delta e
				\circ v_0\varphi$.  Then the two sides of equation \ref{deffull}
				are equal to the composite of the following diagram.
				\begin{xym}\xymatrix@=30pt{
					&A\ar[rr]^f\ar[dr]_e\rrtwocell\omit\omit{_<3.3>\varphi}
					&&B\ar[dr]^{v_0}
					\\ X\ar[ur]^{u_0}\ar[dr]_{u_1}\rrtwocell\omit\omit{\gamma}
					&&I\ar[ur]_m\ar[dr]^m\rrtwocell\omit\omit{\delta}
					&&Y
					\\ &A\ar[rr]_f\ar[ur]^e\rrtwocell\omit\omit{_<-3.3>\;\;\;\;\,\varphi^{-1}}
					&&B\ar[ur]_{v_1}
				}\end{xym}

				{\it 7 $\Rightarrow$ 6. } If the arrow $l_f$ of diagram \ref{diaglf}
				is an equivalence, then $f$ factors as the fully cofaithful arrow
				$l_f e_f$ followed by the fully faithful arrow $\hat{m}_f$.
				
				{\it 6 $\Rightarrow$ 7. } Let us assume that $f$ factors as
				$A\overset{g}\ra I\overset{h}\ra B$, with $f\simeq hg$,
				$g$ fully cofaithful and $h$ fully faithful.
				Since the factorisation of $f$ as a
				cofaithful arrow followed by a fully faithful arrow is unique,
				we have an equivalence $l_1\col I\ra\impl f$.  And since the factorisation
				of $f$ as a fully cofaithful arrow followed by a faithful arrow is
				unique, we have an equivalence $l_2\col I\ra\im f$.  We can check that
				$l_1\simeq l_fl_2$. So $l_f$ is an equivalence.
								
				{\it 4 $\Rightarrow$ 8. }
				Condition 4 implies that $\pi_0 f$ is a monomorphism: 
				if $\Ker f$ is connected, by the dual of Proposition \ref{caracdispreadd},
				$\pi_0\Ker f\simeq 0$; the comparison arrow
				$a_f\col \pi_0\Ker f\ra\Ker\pi_0 f$ has thus a zero domain;
				since it is an epimorphism (for $\C$
				is good), its codomain is also zero and 
				$\pi_0 f$ is a monomorphism.  Dually, condition
				5 (which is equivalent to condition 4) implies that $\pi_1 f$ is an
				epimorphism.
				
				{\it 8 $\Rightarrow$ 4. } We construct the following exact sequence,
				where the left part is the left part of the lower row of diagram
				\ref{diagpourmuex}, and the right part is the right part of
				the upper row of this diagram.  The arrow $\Omega B\ra\pi_0 Kf$
				is the composite $\Omega B\ra \Omega Qf
				\xrightarrow{\Omega\bar{\mu}_f^{-1}}\pi_0 Kf$.
				\begin{xym}\xymatrix@!@C=-14pt@R=-10pt{
					&&&&&&&&0\ar[dr]&&0
					\\ &&&&&&&&&0\ar[ur]\ar[dr]
					\\0\ar[rr]
					&&{\Omega Kf}\ar[rr]
					&&\Omega A\ar[rr]^{\Omega f}
					&&\Omega B\ar[rr]\ar[dr]
					&&{\pi_0 Kf}\ar[rr]\ar[ur]
					&&\pi_0 A\ar[rr]^{\pi_0 f\;}
					&&\pi_0 B\ar[rr]
					&&{\pi_0 Qf}\ar[rr]
					&&0
					\\ &&&&&&&0\ar[dr]\ar[ur]
					\\ &&&&&&0\ar[ur]&&0
				}\end{xym}
				
				Since the sequence is exact at $\Omega B$ and since
				$\Omega f$ is an epimorphism in $\DisC$
				(because $\Sigma\Omega f=\pi_1 f$ is an epimorphism in $\ConC$), the image
				of the arrow $\Omega B\ra\pi_0 Kf$ is $0$. In the same way, since the sequence
				is exact at $\pi_0 A$ and since
				$\pi_0 f$ is a monomorphism, the image of the arrow $\pi_0Kf\ra\pi_0A$
				is $0$.  Finally, since the sequence is exact at $\pi_0 Kf$,
				the oblique sequence is exact, and $\pi_0\Ker f\simeq 0$.
			\end{proof}

	Without assuming that $\C$ is good, we
	can prove characterisations of faithfulness and cofaithfulness
	corresponding to points 4, 5 and 8 of the previous proposition.
	
	\begin{pon}\label{equfid}
		Let $\C$ be a 2-Puppe-exact $\Gpdp$-category and $A\overset{f}\ra B$ be an arrow
		in $\C$. The following conditions are equivalent:
		\begin{enumerate}
			\item $f$ is faithful;
			\item $\Ker f$ is discrete;
			\item $\pi_1 f$ is a monomorphism in $\ConC$.
		\end{enumerate}
	\end{pon}
	
		\begin{proof}
			The arrow $f$ is faithful if and only if $\Omega\Ker f\equiv
			\Pip f\simeq 0$, which is the case if and only if 
			$ \Ker\pi_1 f\simeq\pi_1\Ker f\simeq 0$.  On the one hand, by Proposition
			\ref{caracdispreadd}, $\pi_1\Ker f\simeq 0$ if and only if
			 $\Ker f$ is discrete.  On the other hand, $\Ker \pi_1f\simeq 0$
			 if and only if $\pi_1 f$ is a monomorphism in $\ConC$.
		\end{proof}

	By combining the two previous propositions, we get a similar characterisation
	for fully faithful arrows (which are the full and faithful arrows
	by Proposition \ref{plmentfideplpamentfid}).

	\begin{pon}\label{equfulldddfid}
		Let $\C$ be a good 2-Puppe-exact $\Gpdp$-category and $A\overset{f}\ra B$
		be an arrow in $\C$. The following conditions are equivalent:
		\begin{enumerate}
			\item $f$ is fully faithful;
			\item $\Ker f\simeq 0$;
			\item $\pi_0 f$ is a monomorphism in $\DisC$
				and $\pi_1 f$ is an isomorphism in $\ConC$.
		\end{enumerate}
	\end{pon}

	The previous proposition and its dual tell us that $\pi_0\col\C\ra\DisC$
	makes invertible fully cofaithful arrows and $\pi_1\col\C\ra\ConC$
	makes invertible fully faithful arrows:
	\begin{align}\stepcounter{eqnum}
		\pi_0(\PlCofid)&\incl\Iso,\\ \stepcounter{eqnum}
		\pi_1(\PlFid)&\incl\Iso.
	\end{align}
	We can go further: $\DisC$ is the $\Gpd$-category of fractions of $\C$
	for the fully cofaithful arrows and $\ConC$ is the $\Gpd$-category of fractions of
	$\C$ for the fully faithful arrows.  Let us recall first the definition
	of $\Gpd$-categories of fractions \cite{Pronk1996a}.
	
	\begin{df}\index{fractions, $\Gpd$-category of}%
	\index{Gpd-category@$\Gpd$-category!of fractions}
		Let $\C$ be a $\Gpd$-category and $\Sigma$ be a full sub-$\Gpd$-category
		of $\flc$. A \emph{$\Gpd$-category of fractions
		for $\Sigma$} consists of a $\Gpd$-category $\C[\Sigma^{-1}]$ together with a
		$\Gpd$-functor $P\col\C\ra\C[\Sigma^{-1}]$ such that, for every $f\in\Sigma$,
		$Pf$ is an equivalence, satisfying the following universal property:
		for every $\Gpd$-category $\D$, the $\Gpd$-functor
		\begin{eqn}
			-\circ P\col [\C[\Sigma^{-1}],\D]\ra[\C,\D]_{\Sigma^{-1}}
		\end{eqn}
		is an equivalence
		(where $[\C,\D]_{\Sigma^{-1}}$ is the full sub-$\Gpd$-category of $[\C,\D]$
		whose objects are the $\Gpd$-functors which map every $f\in\Sigma$ to an equivalence).
	\end{df}

	\begin{pon}
		If $\C$ is a good 2-Puppe-exact $\Gpd$-category, then $\pi_0\col\allowbreak\C\ra\DisC$
		is a $\Gpd$-category of fractions for $\PlCofid$
		and $\pi_1\col\C\ra\ConC$ is a $\Gpd$-category of fractions for $\PlFid$.
	\end{pon}
	
		\begin{proof}
			We prove the property for $\pi_0$; the proof for $\pi_1$ is dual.
			First, the dual of Proposition \ref{equfulldddfid} shows that $\pi_0$
			maps fully cofaithful arrows to equivalences in $\DisC$. 
			Next, let $\D$ be a $\Gpd$-category.  We have two $\Gpd$-functors
			\begin{xym}\xymatrix@=40pt{
				{[\DisC,\D]}\ar@<-1mm>[r]_-{-\circ\pi_0}
				&{[\C,\D]_{\PlCofid^{-1}}.}\ar@<-1mm>[l]_-{-\circ i}
			}\end{xym}
			We check that they are inverse to each other.  First, as
			$\pi_0 i\simeq 1_{\DisC}$, $(-\circ i)(-\circ\pi_0)\simeq 1$.  Next, we have
			a $\Gpd$-natural transformation $1\Ra (-\circ\pi_0)(-\circ i)$
			which is defined at $F\col[\C,\D]_{\PlCofid^{-1}}$
			by $F\eta\col F\Ra Fi\pi_0$.  Now, since $\eta$
			is a coroot (Proposition \ref{etacoracepsrac}),
			$\eta$ is fully cofaithful.
			So, since $F$ maps fully cofaithful arrows to equivalences,
			$F\eta$ is an equivalence. Thus $-\circ \pi_0$ is an equivalence.
		\end{proof}

	We conclude this chapter by giving a refined version of the short 5 lemma,
	using the same method as in the proof of the short 5 lemma for symmetric 2-groups
	by Dominique Bourn and Enrico
	Vitale \cite[proposition 2.8]{Bourn2002a}.
	
	\begin{pon}\label{petitlemmcinqraffine}\index{lemma!short five}%
	\index{short five lemma}\index{five, short … lemma}
		Let be a good 2-Puppe exact $\Gpdp$-category.
		In the situation of diagram
		\ref{diagpetitlemmcinq}, where $\eta' a\circ g'\varphi\circ\psi f=c\eta$ and where
		the rows are extensions, we have the following properties:
		\begin{enumerate}
			\item if $a$ and $c$ are faithful, then $b$ is faithful;
			\item if $a$ and $c$ are full, then $b$ is full;
			\item if $a$ and $c$ are cofaithful, then $b$ is cofaithful.
		\end{enumerate}
	\end{pon}
	
		\begin{proof}
			Properties 1 and 3 are immediate consequences of Lemma
			\ref{petitlemmcinqtjrvrai} since, in a 2-Puppe-exact $\Gpdp$-category,
			(co)faithful and 0-(co)faithful arrows coincide.
			
			For property 2, let us construct the exact sequence \ref{suitdiscextens}
			starting from the two rows of diagram \ref{diagpetitlemmcinq}.
			We get the following diagram in $\DisC$.
			\begin{xym}\newdir{ >}{{}*!/-5pt/\dir{>}}\xymatrix@=25pt{
				0\ar[r]
				&\Omega A\ar[r]^{\Omega f}\ar@{->>}[d]_{\Omega a}
				&\Omega B\ar[r]^{\Omega g}\ar[d]_{\Omega b}
				&\Omega C\ar[r]\ar@{->>}[d]_{\Omega c}
				&\pi_0 A\ar[r]^{\pi_0 f}\ar@{ >->}[d]^{\pi_0 a}
				&\pi_0 B\ar[r]^{\pi_0 g}\ar[d]^{\pi_0 b}
				&\pi_0 C\ar[r]\ar@{ >->}[d]^{\pi_0 c}
				&0
				\\ 0\ar[r]
				&\Omega A'\ar[r]_{\Omega f'}
				&\Omega B'\ar[r]_{\Omega g'}
				&\Omega C'\ar[r]
				&\pi_0 A'\ar[r]_{\pi_0 f'}
				&\pi_0 B'\ar[r]_{\pi_0 g'}
				&\pi_0 C'\ar[r]
				&0
			}\end{xym}
			If $a$ and $c$ are full, by Proposition \ref{caracfullbondpex},
			$\pi_0 a$ and $\pi_0 c$ are monomorphisms 
			and $\pi_1 a$ and $\pi_1 c$ are epimorphisms (thus, by the equivalence
			$\Sigma\adj\Omega$, $\Omega a$ and $\Omega c$ are epimorphisms).
			
			Since $\DisC$ is Puppe-exact, the 4 lemma holds in it
			(see \cite[Proposition 13.5.4]{Schubert1972a}).  By applying this lemma
			and its dual to the above diagram, we have that $\pi_0 b$ is
			a monomorphism and $\Omega b$ is an epimorphism (thus $\pi_1 b$
			is an epimorphism).  By applying Proposition
			\ref{caracfullbondpex} again, we can conclude that $b$ is full.
		\end{proof}

\chapter{Symmetric 2-groups and additive $\Gpd$-categories}\label{CGS}

\begin{quote}{\it This chapter is devoted to the study of additivity in dimension 2. 
	First, we define symmetric 2-groups and 
	$\Gpd$-categories enriched in symmetric 2-groups (or preadditive $\Gpd$-categories), as well as additive $\Gpd$-categories.
	We prove that 2-abelian $\Gpd$-categories are additive (Corollary \ref{2abPexadd}),
	regular (Proposition \ref{gpdcatdabreg}) and abelian in the sense of
	Chapter 3 (Corollary \ref{twoabimplab}),
	which implies that all diagram lemmas of Chapter 3
	(among which the long exact sequence of homology) hold in a 2-abelian
	$\Gpd$-category.
}\end{quote}

\section{Symmetric 2-groups and preadditive $\Gpd$-categories}
\subsection{Symmetric 2-monoids}\label{defdgrab}

	Let us recall first the definition of 2-monoids, which are the monoidal groupoids\index{monoidal!groupoid}, and of their morphisms and 2-morphisms.

	\begin{df}\index{2-monoid}
		A \emph{2-monoid} consists of
		\begin{enumerate}
			\item a groupoid $\G$,
			\item a functor $-\tens - \col \G\times\G\ra\G$,
			\item an object $I\col \G$,
			\item a transformation $a_{A,B,C}\col (A\tens B)\tens C
				 \ra A\tens(B\tens C)$ natural in $A,B,C$,
			\item transformations $l_A\col I\tens A \ra A$ and $r_A\col A\tens I\ra A$
				natural in $A$,
		\end{enumerate}
		such that the following diagrams commute.
		\begin{xym}\label{axmonassoc}\xymatrix@C=15pt@R=60pt{
			&(A\tens B)\tens(C\tens D)\ar[dr]^-{a_{A,B,C\tens D}}
			\\ ((A\tens B)\tens C)\tens D\ar[ur]^-{a_{A\tens B,C,D}}
			&& A\tens (B\tens(C\tens D))
			\\ {}\save[]+<1.5cm,0cm>*{(A\tens (B\tens C))\tens D\,}="a"
				\ar@{<-}[u]^-{a_{ABC}\tens D}\restore
			&&{}\save[]+<-1.5cm,0cm>*{\,A\tens ((B\tens C)\tens D)}="b"
				\ar[u]_-{A\tens a_{BCD}}\restore
				\ar"a";"b"_-{a_{A,B\tens C,D}}
		}\end{xym}
		\begin{xym}\label{axmonunit}\xymatrix{
			(A\tens I)\tens B\ar[rr]^{a_{AIB}}\ar[dr]_{r_A\tens B}
			&& A\tens(I\tens B)\ar[dl]^{A\tens l_B}
			\\ &A\tens B
		}\end{xym}
	\end{df}
	
	As a general rule, in this chapter, we omit the subscripts and superscripts
	of natural transformations when they are clearly determined by the context.
	
	\begin{df}\index{monoidal!functor}\index{functor!monoidal}
		Let $\G$, $\H$ be 2-monoids.
		A \emph{monoidal functor} from $\G$ to $\H$ consists of:
		\begin{enumerate}
			\item a functor $F\col \G\ra\H$,
			\item a transformation	$\varphi^F_{AB}\col FA\tens FB\ra F(A\tens B)$
				natural in $A$ and $B$ (we usually omit the superscript if it 
				can be implied from the context),
			\item a morphism $\varphi_0\col I\ra FI$,
		\end{enumerate}
		such that the following diagrams commute.
		\begin{xym}\label{axfoncmonass}\xymatrix@=50pt{
			{}\save[]+<1.5cm,0cm>*{(FA\tens FB)\tens FC\,}="m"
				\ar[d]_-{\varphi\tens 1}\restore
			&&{}\save[]+<-1.5cm,0cm>*{\,FA\tens(FB\tens FC)}="n"
				\ar[d]^-{1\tens\varphi}\restore
				\ar "m";"n" ^-{a}
			\\ F(A\tens B)\tens FC
			&& FA\tens F(B\tens C)
			\\ {}\save[]+<1.5cm,0cm>*{F((A\tens B)\tens C)\,}="a"
				\ar@{<-}[u]^-{\varphi}\restore
			&&{}\save[]+<-1.5cm,0cm>*{\,F(A\tens (B\tens C))}="b"
				\ar@{<-}[u]_-{\varphi}\restore
				\ar "a";"b"_-{Fa}
		}\end{xym}
		\begin{xym}\label{axfoncmonunite}\xymatrix@=40pt{
			FA\tens I\ar[r]^-{r_{FA}}\ar[d]_{1\tens\varphi_0}
			&FA
			\\ FA\tens FI\ar[r]_-{\varphi_{AI}}
			&F(A\tens I)\ar[u]_{Fr_A}
		}\;\;\;
		\xymatrix@=40pt{
			I\tens FA\ar[r]^-{l_{FA}}\ar[d]_{\varphi_0\tens 1}
			&FA
			\\ FI\tens FA\ar[r]_-{\varphi_{IA}}
			&F(I\tens A)\ar[u]_{Fl_A}
		}\end{xym}
	\end{df}
	
	\begin{df}\index{natural transformation!monoidal}%
	\index{monoidal!natural transformation}
		Let $F,G\col \G\ra\H$ be two monoidal functors between 2-mo\-no\-ids.
		A \emph{monoidal natural transformation} from $F$ to $G$ is a
		natural transformation $\alpha\col F\Ra G$ such that the following diagrams
		commute.
		\begin{xym}\label{axtnmonass}\xymatrix@=40pt
			{FA\tens FB\ar[r]^{\varphi^F_{A,B}}
				\ar[d]_{\alpha_A\tens\alpha_B}
			&F(A\tens B)\ar[d]^{\alpha_{A\tens B}}
			\\ GA\tens GB\ar[r]_{\varphi^G_{A,B}}
			&G(A\tens B)}
		\end{xym}
		\begin{xym}\label{axtnmonunit}\xymatrix@R=20pt@C=40pt{
			&FI\ar[dd]^{\alpha_I}
			\\ I\ar[ur]^{\varphi_0}\ar[dr]_{\varphi_0}
			\\ &GI
		}\end{xym}
	\end{df}
	
	\begin{pon}\label{pondefdMon}\index{2-mon@$\dMon$}
		The 2-monoids, monoidal functors between them and mo\-no\-i\-dal natural transformations
		between them form a $\Gpdp$-category, denoted by $\dMon$.
	\end{pon}
	
		\begin{proof}
			Horizontal composition of monoidal functors is defined by
			$(G,\varphi^G,\varphi_0^G)\circ(F,\varphi^F,\varphi_0^F)
			=(GF,\varphi^{GF},\varphi_0^{GF})$, where $\varphi^{GF}_{A,B}=$
			\begin{eqn}
				GFA\tens GFB\overset{\varphi^G}\longrightarrow G(FA\tens FB)
				\xrightarrow{G\varphi^F} GF(A\tens B),
			\end{eqn}
			and $\varphi_0^{GF} = I\overset{\varphi_0^G}\longrightarrow GI\xrightarrow{G\varphi_0^F} GFI$.
			Horizontal and vertical compositions of monoidal natural transformations
			are the ordinary compositions of the underlying natural transformations.
			The zero morphism $0\col \G\ra\H$ is the constant functor
			mapping all objects to $I$ and all arrows to $1_I$.
		\end{proof}
	
	The 2-monoids are actually the one-object $\Gpd$-categories
	(Definition \ref{defgpdcat}) and the monoidal functors are the $\Gpd$-functors
	between these one-object $\Gpd$-ca\-te\-go\-ries.
	The coherence metatheorem allowing to replace
	every $\Gpd$-category by an equivalent strictly described $\Gpd$-category
	takes the following form for 2-monoids: every 2-monoid
	can be \emph{strictly described} \index{2-monoid!strictly described}%
	up to equivalence (in $\dMon$)
	 i.e.\ in such a way that
	$A\tens (B\tens C)$ coincides with $(A\tens B)\tens C$, that $A\tens I$ and $I\tens A$
	coincide with $A$, and that the transformations $a$, $l$ and $r$ are identities.
	
	In the same way, every monoidal functor $F\col \G\ra\H$ can be \emph{normally}
	described up to isomorphism
	(in $\dMon(\G,\H)$), i.e.\ in such a way that
	$FI$ coincides with $I$;
	we proceed as in $\Gpdp$ (Proposition \ref{metthcoh}).
	But it is not true that every monoidal functor can be strictly described up to
	isomorphism (an example is given by diagram \ref{exampleequpasequ}).
	
	Let us recall now the definition of braided and symmetric 2-monoids and their
	morphisms \cite{Joyal1986a,Joyal1993a}.
		
	\begin{df}\index{braided 2-monoid}\index{2-monoid!braided}
		A \emph{braided} 2-monoid is a 2-monoid $\G$ equipped with
		\begin{enumerate}
			\item[6.] a transformation $c_{A,B}\col A\tens B\ra B\tens A$ natural
				in $A, B$
		\end{enumerate}
		such that the following diagrams commute.
		\begin{xym}\label{axmontresi}\xymatrix@=60pt{
			{}\save[]+<1.5cm,0cm>*{(A\tens B)\tens C\,}="m"
				\ar[d]_-{a_{ABC}}\restore
			&&{}\save[]+<-1.5cm,0cm>*{\,(B\tens A)\tens C}="n"
				\ar[d]^-{a_{BAC}}\restore
				\ar "m";"n" ^-{c_{AB}\tens C}
			\\ A\tens (B\tens C)
			&& B\tens (A\tens C)
			\\ {}\save[]+<1.5cm,0cm>*{(B\tens C)\tens A\,}="a"
				\ar@{<-}[u]^-{c_{A,B\tens C}}\restore
			&&{}\save[]+<-1.5cm,0cm>*{\,B\tens (C\tens A)}="b"
				\ar@{<-}[u]_-{B\tens c_{AC}}\restore
				\ar "a";"b"_-{a_{BCA}}
		}\end{xym}
		\begin{xym}\label{axmontresii}\xymatrix@=60pt{
			{}\save[]+<1.5cm,0cm>*{(A\tens B)\tens C\,}="m"
				\ar[d]_-{a_{ABC}}\restore
			&&{}\save[]+<-1.5cm,0cm>*{\,(B\tens A)\tens C}="n"
				\ar[d]^-{a_{BAC}}\restore
				\ar "n";"m" _-{c_{BA}\tens C}
			\\ A\tens (B\tens C)
			&& B\tens (A\tens C)
			\\ {}\save[]+<1.5cm,0cm>*{(B\tens C)\tens A\,}="a"
				\ar@{<-}[u]^-{c_{A,B\tens C}}\restore
			&&{}\save[]+<-1.5cm,0cm>*{\,B\tens (C\tens A)}="b"
				\ar[u]_-{B\tens c_{CA}}\restore
				\ar "a";"b"_-{a_{BCA}}
		}\end{xym}
		If moreover 
		\begin{eqn}\label{axmonsym}
			c_{BA}\circ c_{AB}=1_{A\tens B},
		\end{eqn}
		we call $\G$ a \emph{symmetric 2-monoid}.%
		\index{symmetric!2-monoid}\index{2-monoid!symmetric}
	\end{df}
	
	For a symmetric 2-monoid, axiom \ref{axmontresii} follows from the others.
	
	\begin{df}
		Let $\G$, $\H$ be two braided 2-monoids.  A \emph{symmetric monoidal functor}
		from $\G$ to $\H$ is a monoidal functor $(F,\varphi,\varphi_0)\col 
		\G\ra\H$ such that the following diagram commutes.
		\begin{xym}\label{axfonmonsym}\xymatrix@=40pt
			{FA\tens FB\ar[r]^{\varphi_{A,B}}
			\ar[d]_{c_{FA,FB}}
			&F(A\tens B)\ar[d]^{F(c_{A,B})}
			\\ FB\tens FA\ar[r]_{\varphi_{B,A}}
			&F(B\tens A)}
		\end{xym}
	\end{df}
	
	\begin{df}\label{defdCMon}
		The symmetric 2-monoids, symmetric monoidal functors and 
		monoidal natural transformations form a sub-$\Gpdp$-category
		of $\dMon$, denoted by $\dCMon$.\index{2-smon@$\dCMon$}
	\end{df}
	
	There is in general
	no equivalent description (in $\dCMon$) where $A\tens B$ and $B\tens A$	coincide, neither for braided 2-monoids nor for symmetric 2-monoids.

	There is a forgetful $\Gpd$-functor $U\col \dCMon\ra\Gpdp$.
	
\subsection{Symmetric bimonoidal functors}\label{secdefbimon}

	First, let us define the internal $\Hom$ of $\dCMon$.  In dimension 1,
	commutativity was necessary to define the sum of two homomorphisms
	of monoids (or of groups); in dimension 2, we need symmetry.
	
	Indeed, if $F,G\col \A\ra\B$ are monoidal functors, we would define
	their tensor product as the composite $F\tens G \eqdef$
	\begin{eqn}\label{tensfonctmon}
		\A\overset{\Delta}\longrightarrow\A\times\A\xrightarrow{F\times G}\B\times\B
		\xrightarrow{-\tens-}\B,
	\end{eqn}
	where $\Delta$ is the diagonal functor.
	But for this composite to be a monoidal functor, $-\tens-$ should be
	itself a monoidal functor.  Joyal and Street \cite{Joyal1993a}
	have proved that to give a monoidal functor structure on $-\tens -$ amounts to
	give a braiding on $\B$.  In this case, the monoidality natural transformation
	$\varphi_{(A,B),(C,D)}\col (A\tens B)\tens (C\tens D)\ra
	(A\tens C)\tens (B\tens D)$ (which we denote more simply by $\bar{c}_{ABCD}$) is given
	by the composite
	\begin{xym}\label{assocsym}\xymatrix@=20pt{
		(A\tens B)\tens (C\tens D)\ar[d]^{a_{A,B,C\tens D}}
		\\ A\tens (B\tens (C\tens D))\ar[d]^{1\tens a^{-1}_{B,C,D}}
		\\ A\tens ((B\tens C)\tens D)\ar[d]^{1\tens (c_{BC}\tens 1)}
		\\ A\tens ((C\tens B)\tens D)\ar[d]^{1\tens a_{C,B,D}}
		\\ A\tens (C\tens (B\tens D))\ar[d]^{a^{-1}_{A,C,B\tens D}}
		\\ (A\tens C)\tens (B\tens D).
	}\end{xym}
	
	So we must work in the $\Gpd$-category of braided monoidal groupoids and symmetric 
	monoidal functors.  But, in this context, we want that, for two symmetric monoidal 
	functors $F$ and $G$, the composite \ref{tensfonctmon} be also
	a \emph{symmetric} monoidal functor.  But, in order that $-\tens -$ preserves the braiding,
	Joyal and Street have shown that it is necessary and sufficient that $\B$ be symmetric
	(this is a 2-dimensional version of the characterisation of the commutativity
	of an algebraic theory by the fact that the operations are algebra morphisms).
	That is why we can define the internal $\Hom$ only for symmetric 2-monoids.
	
	\begin{df}\label{homintdcmon}
		Let $\A,\B\col \dCMon$.  The symmetric 2-monoid $[\A,\B]$ is the grou\-po\-id of
		symmetric monoidal functors $\A\ra\B$ and monoidal natural transformations
		between them, equipped with the following structure.
		\begin{enumerate}
			\item The tensor product of $F,G\col \A\ra\B$ is defined by the composite
				\ref{tensfonctmon}, i.e.\ on the objects
				by $(F\tens G)(A)=FA\tens GA$ and on the arrows by $(F\tens G)(f)
				=Ff\tens Gf$.  If $A,A'\col \A$, the natural transformation
				$\varphi^{F\tens G}_{A,A'}$ is defined (according to the definition
				of the composite of monoidal functors) by the composite:
				\begin{xym}\label{phiprodtens}\xymatrix@=20pt{
					(FA\tens GA)\tens (FA'\tens GA')\ar[d]^-{\bar{c}}
					\\ (FA\tens FA')\tens (GA\tens GA')\ar[d]^-{\varphi^F\tens\varphi^G}
					\\ F(A\tens A')\tens G(A\tens A'),
				}\end{xym}
				where $\bar{c}$ is the composite \ref{assocsym}.
				And $\varphi_0^{F\tens G}$ is defined by the composite
				$I\xrightarrow{l_I^{-1}=r_I^{-1}} I\tens I
				\xrightarrow{\varphi^F_0\tens \varphi^G_0}FI\tens GI$.
			\item The unit is the zero functor $0\col \A\ra\B$.
			\item The natural transformations $a,l,r$ are defined pointwise.
		\end{enumerate}
	\end{df}
	
	Now we define symmetric bimonoidal functors, which play the same rôle as bilinear
	functions play for abelian monoids (or groups).  This notion have been introduced,
	for stacks of symmetric 2-groups in \cite{Deligne1973a}, with the name
	\emph{biadditive functor}.
	It is possible to define a tensor product on $\dCMon$
	and to show that a symmetric bimonoidal functor
	$\G\times\H\ra\cat{K}$ is the same as a symmetric monoidal functor
	$\G\tens\H\ra\cat{K}$.
		
	\begin{df}\label{defbimon}\index{symmetric bimonoidal functor}%
	\index{functor!symmetric bimonoidal}
		Let be $\G$, $\H$, $\cat{K}\col \dCMon$.  A  \emph{symmetric
		bimonoidal} functor $F\col \G\times\H\ra\cat{K}$ is a functor
		$F\col \G\times\H\ra\cat{K}$, with isomorphisms natural in each
		variable
		\begin{align}\stepcounter{eqnum}
		\begin{split}
			\varphi_{G_1,G_2}^H &\col F(G_1,H)\tens F(G_2,H)\ra F(G_1\tens G_2,H), \\
			\psi_G^{H_1,H_2}&\col F(G,H_1)\tens F(G,H_2)\ra F(G,H_1\tens H_2),\\
			\varphi_0^H &\col I\ra F(I,H),\\
			\psi^0_G &\col I\ra F(G,I),
		\end{split}
		\end{align}
		satisfying the following conditions (conditions 3 and 3' are
		equivalent):
		\begin{enumerate}
			\item for every $G\col \G$, $F(G,-)$, with $\psi_G^{-,-}$ and $\psi_G^0$, is a
			 		symmetric monoidal functor;
			\item for every $H\col \H$, $F(-,H)$, with $\varphi^H_{-,-}$ and $\varphi^H_0$,
				is a symmetric monoidal functor;
			\item for all $G_1,G_2\col \H$, the natural transformations
				$\varphi^{-}_{G_1,G_2}\col 
				F(G_1,-)\tens F(G_2,-)\allowbreak\Ra F(G_1\tens G_2,-)$ and $\varphi^-_0\col 0\Ra F(I,-)$
				are monoidal;
			\item[3'.] for all $H_1,H_2\col \H$, the natural transformations
				$\psi_{-}^{H_1,H_2}\col 	F(-,H_1)\tens F(-,H_2)\allowbreak\Ra F(-,H_1\tens H_2)$
				and $\psi_{-}^0\col 0\Ra F(-,I)$ are monoidal.
		\end{enumerate}
	\end{df}

	The naturality in $G$ of $\psi$ and $\psi^0$ is equivalent
	to the fact that, for every arrow $g$ in $\G$, $F(g,-)$ is a monoidal transformation,
	and the naturality in $H$ of $\varphi$ and $\varphi_0$ is
	equivalent to the fact that, for each arrow $h$ in $\cat{H}$, $F(-,h)$ is a
	monoidal transformation.
	
	Conditions 3 and 3' express the compatibility between
	$\varphi$ and $\psi$.  For example, axiom \ref{axtnmonass} for
	$\psi^{H_1,H_2}_{-}$ can be expressed by
	the commutativity of the following diagram, for each $G_1,G_2\col \G$ and $H_1,H_2\col \H$.
	The left column is the “$\varphi$”
	of the tensor product $F(-,H_1)\tens F(-,H_2)$, as defined in
	diagram \ref{phiprodtens}.  The commutativity of this diagram is equivalent to
	axiom \ref{axtnmonass} for $\varphi^{-}_{G_1,G_2}$, since $\bar{c}^{-1}=\bar{c}$.
	\begin{xym}\xymatrix@R=40pt@C=60pt{
		{\txt{$(F(G_1,H_1)\tens F(G_1,H_2))$
			\\$\tens (F(G_2,H_1)\tens F(G_2,H_2))$}}
			\ar[r]^-{\psi_{G_1}^{H_1,H_2}\tens\psi_{G_2}^{H_1,H_2}}
			\ar[d]_-{\bar{c}}
		&{\txt{$F(G_1,H_1\tens H_2)$
			\\ $\tens F(G_2,H_1\tens H_2)$}}
			\ar[dd]^-{\varphi_{G_1,G_2}^{H_1\tens H_2}}
		\\ {\txt{$(F(G_1,H_1)\tens F(G_2,H_1))$
			\\ $\tens (F(G_1,H_2)\tens F(G_2,H_2))$}}
			\ar[d]_-{\varphi_{G_1,G_2}^{H_1}\tens \varphi_{G_1,G_2}^{H_2}}
		\\ {\txt{$F(G_1\tens G_2,H_1)$
			\\$\tens F(G_1\tens G_2,H_2)$}}
			\ar[r]_-{\psi_{G_1\tens G_2}^{H_1,H_2}}
		&F(G_1\tens G_2,H_1\tens H_2)
	}\end{xym}
	In the same way, we can check that axiom \ref{axtnmonunit} for $\psi^{H_1,H_2}_{-}$
	is equivalent to axiom \ref{axtnmonass} for $\varphi_0^{-}$, that axiom
	\ref{axtnmonass} for $\psi^0_{-}$ is equivalent to axiom
	\ref{axtnmonunit} for $\varphi_{G_1,G_2}^{-}$, and that axioms
	\ref{axtnmonunit} for $\psi^0_{-}$ and $\varphi_0^{-}$ are equivalent, which shows
	that conditions 3' and 3 are equivalent.
		
	\begin{df}\index{Bimon(G x H,K)@$\Bimon(\G\times\H,\cat{K})$}
		The symmetric 2-monoid $\Bimon(\G\times\H,\cat{K})$ is defined in the
		following way:
		\begin{itemize}
			\item {\it Objects.} These are the symmetric bimonoidal functors
				$\G\times\H\ra\cat{K}$.
			\item {\it Arrows.} These are the natural transformations
				$\alpha\col F\Ra F'\col \G\times\H\ra\cat{K}$ such that, for each
				$G\col \G$, $\alpha_{(G,-)}\col F(G,-)\Ra F'(G,-)$ is monoidal
				and, for each $H\col \H$, $\alpha_{(-,H)}\col F(-,H)\Ra F'(-,H)$ is monoidal.
				We will call them \emph{bimonoidal natural transformations}.
			\item {\it Tensor.} $F\tens F'\col \G\times\H\ra\cat{K}$ is defined on
				objects by $(F\tens F')(G,H)\eqdef F(G,H)\tens F'(G,H)$ and on
				arrows by $(F\tens F')(g,h)\eqdef F(g,h)\tens F'(g,h)$.  This is a
				symmetric bimonoidal functor (thanks to symmetry). Moreover,
				$(\alpha\tens\alpha')_{(G,H)}\eqdef \alpha_{(G,H)}\tens\alpha'_{(G,H)}$.
			\item {\it Unit.} This is the constant functor on $I$.
		\end{itemize}
	\end{df}
	
	The following proposition justifies the definition of symmetric bimonoidal functors 
	and of bimonoidal natural transformations.
	
	\begin{pon}
		$\Bimon(\G\times\H,\cat{K})\simeq [\G,[\H,\cat{K}]]$
	\end{pon}
	
		\begin{proof}
			I give only the construction of the symmetric monoidal equivalences
			$\Phi\col \Bimon(\G\times\H,\cat{K})\ra [\G,[\H,\cat{K}]]$
			and $\Psi\col [\G,[\H,\cat{K}]]\ra\Bimon(\G\times\H,\cat{K})$.
			
			In one direction, if $F\col \Bimon(\G\times\H,\cat{K})$, the functor
			$\Phi F\col \G\ra [\H,\cat{K}]$ maps $G$
			to $F(G,-)$ and $g\col G\ra G'$ to $F(g,-)$; $\Phi F$ is symmetric monoidal,
			with the transformations
			$\varphi^{-}_{G_1,G_2}\col  F(G_1,-)\tens F(G_2,-)
			\Ra F(G_1\tens G_2,-)$ and $\varphi^{-}_0\col  0 \Ra F(I,-)$, which are
			morphisms of $[\H,\cat{K}]$,
			i.e.\ monoidal transformations, by condition 3 of
			definition \ref{defbimon}.
			If $\alpha\col F\Ra F'\col \G\times\H\ra\cat{K}$ is a bimonoidal transformation,
			then $(\Phi\alpha)_G \eqdef \alpha_{(G,-)}$.
			
			In the other direction, if $F\col [\G,[\H,\cat{K}]]$, the functor
			$\Psi F\col \G\times\H\ra\cat{K}$ maps $(G,H)$ to $(FG)(H)$
			and $(g,h)\col (G,H)\ra (G',H')$ to $(FG')(h)\circ(Fg)_H$.  And if
			$\alpha\col F\Ra F'\col \G\ra[\H,\cat{K}]$ is a monoidal transformation,
			we set $(\Psi\alpha)_{(G,H)}\eqdef (\alpha_G)_H$.
		\end{proof}

\subsection{$\dCMon$-categories}

	The goal of this section is to define what means for a
	$\Gpd$-category to be \emph{enriched in $\dCMon$}.	 This is a 2-dimensional
	version of preadditivity (enrichment in the category
	of commutative monoids). In dimension 1,
	we can take two points of view on a $\Ens$-category enriched in $\CMon$:
	\begin{enumerate}
		\item it is a category $\C$ whose $\Hom$s are equipped with an abelian monoid
			structure such that, for all $A,B,C\col \C$, the composition function
			$\C(A,B)\times\C(B,C)\ra\C(A,C)$ is bilinear;
		\item it is a category $\C$ whose $\Hom$s are equipped with an abelian monoid
			structure such that, for all $A\col \C$ and $g\col \C(B,C)$,
			the composition function $\C(A,B)\xrightarrow{g\circ -}\C(A,C)$ is linear and,
			 for all $f\col \C(A,B)$ and $C\col \C$, the function
			 $\C(B,C)\xrightarrow{-\circ f}\C(A,C)$ is linear.
	\end{enumerate}
	
	In dimension 2, these two points of view are also available; this gives
	the following definition.
	
	\begin{df}\label{semipreadd}
		Let $\C$ be a (weak in general)  $\Gpd$-category such that, for all $A,B\col \C$,
		the groupoid $\C(A,B)$ is equipped with a symmetric 2-monoid structure
		(the tensor is denoted by $+$, the unit is denoted by $0$, the associativity
		transformation is denoted by $\alpha$,
		the unit transformations are denoted by $\lambda$ and $\rho$, and the symmetry
		transformation is denoted by $\gamma$)
		and with transformations natural in each variable
		\begin{align}\begin{split}\stepcounter{eqnum}
			\varphi_{g_1,g_2}^h&\col hg_1+hg_2\Ra h(g_1+g_2);\\
			\psi_g^{h_1,h_2}&\col h_1g+h_2g\Ra(h_1+h_2)g;\\
			\varphi_0^h&\col 0\Ra h0;\\
			\psi_g^0&\col 0\Ra 0g.
		\end{split}\end{align}
		We say that $\C$ is \emph{presemiadditive}%
		\index{Gpd-category@$\Gpd$-category!presemiadditive}%
		\index{presemiadditive Gpd-category@presemiadditive $\Gpd$-category}
		(or that it is a
		\emph{$\dCMon$-category})\index{2-smon-categorie@$\dCMon$-category}
		 if the equivalent conditions 1 and 2 hold
		 (we put a dot above the natural transformations
		of the structure of $\Gpd$-category to distinguish them from their analogues
		of the 2-monoid structure).
		\begin{enumerate}
			\item \begin{enumerate}
				\item For all $A,B,C\col \C$, the composition functor
					$\C(A,B)\times\C(B,C)\ra\C(A,C)$, with $\varphi_{g_1,g_2}^h$,
					$\psi_g^{h_1,h_2}$, $\varphi_0^h$ and $\psi_g^0$,
					is symmetric bimonoidal.
				\item For all $A,B,C,D$, the associativity natural transformation
					 is trimonoidal.
					\begin{xym}\xymatrix@=30pt{
						{\C(A,B)\times\C(B,C)\times\C(C,D)}
							\drtwocell\omit\omit{\dot{\alpha}}
							\ar[d]_{\mathrm{comp}\times 1}\ar[r]^-{1\times\mathrm{comp}}
						&{\C(A,B)\times\C(B,D)}\ar[d]^{\mathrm{comp}}
						\\ {\C(A,C)\times\C(C,D)}\ar[r]_-{\mathrm{comp}}
						& {\C(A,D)}
					}\end{xym}
				\item For all $A,B$, the unit natural transformations are monoidal.
					\begin{xym}\xymatrix@=30pt{
						{1\times\C(A,B)}\ar[r]^-{\mathrm{id}\times 1}
							\ar@{=}[dr]\drtwocell\omit\omit{_<-2.5>{\dot{\rho}}}
						&{\C(A,A)\times\C(A,B)}\ar[d]^{\mathrm{comp}}
						\\ &{\C(A,B)}
					}\end{xym}
					\begin{xym}\xymatrix@=30pt{
						{\C(A,B)\times\C(B,B)}\ar[d]_{\mathrm{comp}}
						&{\C(A,B)\times 1}\ar[l]_-{1\times\mathrm{id}}
							\ar@{=}[dl]\dltwocell\omit\omit{^<2.5>{\dot{\lambda}}}
						\\ {\C(A,B)}
					}\end{xym}
				\end{enumerate}
			\item \begin{enumerate}
				\item For all $A$ and $h\col B\ra C$ in $\C$, the functor
					$h\circ -\col \C(A,B)\ra\C(A,C)$, with
					$\varphi_{g_1,g_2}^h$ and $\varphi_0^h$,
					is symmetric monoidal.
				\item For all $g\col A\ra B$ and $C$ in $\C$, the functor
					$-\circ g\col \C(B,C)\ra\C(A,C)$, with
					$\psi_g^{h_1,h_2}$ and $\psi_g^0$, is symmetric monoidal.
				\item For all parallel $h_1,h_2$ in $\C$, the transformation
					$\psi_{-}^{h_1,h_2}$ is monoidal and $\psi_{-}^0$ is monoidal
					(or, equivalently,
					for all parallel $g_1,g_2$ in $\C$, the transformation
					$\varphi_{g_1,g_2}^{-}$ is monoidal, and $\varphi^{-}_0$ is monoidal).
				\item For all $g\col B\ra C$ and $h\col C\ra D$, the natural transformation
					expressing a first part of associativity
					is monoidal.
					\begin{xym}\xymatrix@=30pt{
						{\C(A,B)}\ar[r]_{g\circ -}\rruppertwocell^{hg\circ -}<10>
							{_<-3>{\;\;\;\;\;\;\,\dot{\alpha}_{hg-}}}
						&{\C(A,C)}\ar[r]_{h\circ -}
						&{\C(A,D)}
					}\end{xym}
				\item For all $f\col A\ra B$ and $h\col C\ra D$, the natural transformation
					expressing another part of associativity
					is monoidal.
					\begin{xym}\xymatrix@=30pt{
						{\C(B,C)}\ar[r]^{h\circ-}\ar[d]_{-\circ f}
							\drtwocell\omit\omit{\;\;\;\;\;\;\dot{\alpha}_{h-f}}
						&{\C(B,D)}\ar[d]^{-\circ f}
						\\ {\C(A,C)}\ar[r]_{h\circ-}
						&{\C(A,D)}
					}\end{xym}
				\item For all $f\col A\ra B$ and $g\col B\ra C$, the natural transformation
					expressing a last part of associativity
					is monoidal.
					\begin{xym}\xymatrix@=30pt{
						{\C(C,D)}\ar[r]^{-\circ g}\rrlowertwocell_{-\circ gf}<-10>
							{_<3>{\;\;\;\;\;\;\,\dot{\alpha}_{-gf}}}
						&{\C(B,D)}\ar[r]^{-\circ f}
						&{\C(A,D)}
					}\end{xym}
				\item For all $A,B\col \C$, the unit natural transformations
					are monoidal.
					\begin{xym}\xymatrix@=30pt{
						{\C(A,B)}\ruppertwocell^{1_B\circ-}{\dot{\lambda}}\ar@{=}[r]
							\rlowertwocell_{-\circ 1_A}{^{\dot{\rho}}}
						&{\C(A,B)}
					}\end{xym}
				\end{enumerate}
		\end{enumerate}
	\end{df}
	
		\begin{proof}
			It is clear that conditions (a), (b) and (c) of version 2 are
			equivalent to point (a) of version 1, by Definition \ref{defbimon}.
			It is also clear that point (g) of version 2 is equivalent
			to point (c) of version 1.
	
			Condition (b) of version 1, the trimonoidality of
			$\dot{\alpha}_{fgh}\col (hg)f\Ra h(gf)$ means that the three natural
			transformations that we get by fixing two of the variables of $\dot{\alpha}$ are
			monoidal:
			$\dot{\alpha}_{-gh}\col (hg)\circ -\Ra h(g\circ -)$ must be monoidal
			(condition (d) of version 2), $\dot{\alpha}_{f-h}\col (h\circ-)f\Ra h(-\circ f)$
			must be monoidal (condition 2(e)),	and
			$\dot{\alpha}_{fg-}\col (-\circ g)f\Ra -\circ (gf)$ must be monoidal
			(condition 2(f)).  Thus we see that condition 1(b) is
			equivalent to the conjunction of conditions 2(d), 2(e) and 2(f).
		\end{proof}
		
	The advantage of the second version is that it avoids bimonoidal or trimonoidal
	functors and natural transformations.
	
	Here is an elementary translation of the conditions of version 2.  
	
	\renewcommand{\theenumi}{\alph{enumi}}\renewcommand{\labelenumi}{(\theenumi)}
	\begin{enumerate}
		\item \begin{xym}\label{preadda1}\xymatrix@R=40pt@C=50pt{
						{}\save[]+<1.5cm,0cm>*{(hg_1 + hg_2) + hg_3\,}="m"
							\ar[d]_-{\varphi^h_{g_1,g_2}+1}\restore
						&&{}\save[]+<-1.5cm,0cm>*{\,hg_1+(hg_2+hg_3)}="n"
						\ar[d]^-{1+\varphi^h_{g_2,g_3}}\restore
						\ar "m";"n" ^-{\alpha}
						\\ h(g_1 + g_2) + hg_3
						&& hg_1 + h(g_2 + g_3)
						\\ {}\save[]+<1.5cm,0cm>*{h((g_1+g_2)+g_3)\,}="a"
							\ar@{<-}[u]^-{\varphi^h_{g_1+g_2,g_3}}\restore
						&&{}\save[]+<-1.5cm,0cm>*{\,h(g_1+(g_2+g_3))}="b"
							\ar@{<-}[u]_-{\varphi^h_{g_1,g_2+g_3}}\restore
							\ar "a";"b"_-{h\alpha}
					}\end{xym}
			\begin{xym}\label{preadda2}\xymatrix@=40pt{
				hg_1+ hg_2\ar[r]^-{\varphi^h_{g_1,g_2}}\ar[d]_\gamma
				&h(g_1+ g_2)\ar[d]^{h\gamma}
				\\ hg_2+ hg_1\ar[r]_-{\varphi^h_{g_2,g_1}}
				&h(g_2+ g_1)
			}\end{xym}
			\begin{xym}\xymatrix@=40pt{
				hg+ 0\ar[r]^-{\rho_{hg}}\ar[d]_{1+\varphi^h_0}
				&hg
				\\ hg+ h0\ar[r]_-{\varphi^h_{g,0}}
				&h(g+0)\ar[u]_{h\rho_g}
			}\;\;\;
			\xymatrix@=40pt{
				0+hg\ar[r]^-{\lambda_{hg}}\ar[d]_{\varphi^h_0+1}
				&hg
				\\ h0+ hg\ar[r]_-{\varphi^h_{0,g}}
				&h(0+g)\ar[u]_{h\lambda_g}
			}\end{xym}
		\item \begin{xym}\label{preaddb1}\xymatrix@R=40pt@C=50pt{
						{}\save[]+<1.5cm,0cm>*{(h_1g + h_2g) + h_3g\,}="m"
							\ar[d]_-{\psi_g^{h_1,h_2}+1}\restore
						&&{}\save[]+<-1.5cm,0cm>*{\,h_1g+(h_2g+h_3g)}="n"
						\ar[d]^-{1+\psi_g^{h_2,h_3}}\restore
						\ar "m";"n" ^-{\alpha}
						\\ (h_1 + h_2)g + h_3g
						&& h_1g + (h_2 + h_3)g
						\\ {}\save[]+<1.5cm,0cm>*{((h_1+h_2)+h_3)g\,}="a"
							\ar@{<-}[u]^-{\psi_g^{h_1+h_2,h_3}}\restore
						&&{}\save[]+<-1.5cm,0cm>*{\,(h_1+(h_2+h_3))g}="b"
							\ar@{<-}[u]_-{\psi_g^{h_1,h_2+h_3}}\restore
							\ar "a";"b"_-{\alpha g}
					}\end{xym}
			\begin{xym}\label{preaddb2}\xymatrix@=40pt{
				h_1g+ h_2g\ar[r]^-{\psi_g^{h_1,h_2}}\ar[d]_\gamma
				&(h_1+ h_2)g\ar[d]^{\gamma g}
				\\ h_2g+ h_1g\ar[r]_-{\psi_g^{h_2,h_1}}
				&(h_2+ h_1)g
			}\end{xym}
			\begin{xym}\xymatrix@=40pt{
				hg+ 0\ar[r]^-{\rho_{hg}}\ar[d]_{1+\psi_g^0}
				&hg
				\\ hg+ 0g\ar[r]_-{\psi_g^{h,0}}
				&(h+0)g\ar[u]_{\rho_hg}
			}\;\;\;
			\xymatrix@=40pt{
				0+hg\ar[r]^-{\lambda_{hg}}\ar[d]_{\psi_g^0+1}
				&hg
				\\ 0g+ hg\ar[r]_-{\psi_g^{0,h}}
				&(0+h)g\ar[u]_{\lambda_hg}
			}\end{xym}
		\item \begin{xym}\label{preaddc}\xymatrix@R=40pt@C=70pt{
				(h_1g_1+ h_2g_1)+ (h_1g_2+ h_2g_2)
					\ar[r]^-{\psi_{g_1}^{h_1,h_2}+\psi_{g_2}^{h_1,h_2}}
					\ar[d]_-{\bar{\gamma}}
				&(h_1+ h_2)g_1+ (h_1+ h_2)g_2
					\ar[dd]^-{\varphi_{g_1,g_2}^{h_1+ h_2}}
				\\ (h_1g_1+ h_1g_2)+ (h_2g_1+ h_2g_2)
					\ar[d]_-{\varphi_{g_1,g_2}^{h_1}+ \varphi_{g_1,g_2}^{h_2}}
				\\ h_1(g_1+ g_2)+ h_2(g_1+ g_2)
					\ar[r]_-{\psi_{g_1+ g_2}^{h_1,h_2}}
				&(h_1+ h_2)(g_1+ g_2)
			}\end{xym}
			\begin{xym}\label{preaddc2}\xymatrix@=40pt{
				0+0\ar[r]^-{\psi^0_{g_1}+\psi^0_{g_2}}\ar[d]_{\lambda_I=\rho_I}
				&0g_1+0g_2\ar[d]^{\varphi^0_{g_1,g_2}}
				\\ 0\ar[r]_-{\psi^0_{g_1+g_2}}
				&0(g_1+g_2)
			}\;\;\;\xymatrix@=40pt{
				0+0\ar[r]^-{\varphi_0^{h_1}+\varphi_0^{h_2}}\ar[d]_{\lambda_I=\rho_I}
				&h_10+h_20\ar[d]^{\psi_0^{h_1,h_2}}
				\\ 0\ar[r]_-{\varphi_0^{h_1+h_2}}
				&(h_1+h_2)0
			}\end{xym}
			\begin{eqn}\label{preaddc3}
				0\xrightarrow{\varphi_0^0=\psi^0_0} 0\circ 0
			\end{eqn}
		\item \begin{xym}\label{preaddd}\xymatrix@R=40pt@C=50pt{
				{(hg)f_1+(hg)f_2}\ar[dd]_-{\varphi^{hg}_{f_1,f_2}}
					\ar[r]^-{\dot{\alpha}_{hgf_1}+\dot{\alpha}_{hgf_2}\,}
				&{h(gf_1)+h(gf_2)}\ar[d]^-{\varphi^h_{gf_1,gf_2}}
				\\ & h(gf_1+gf_2)
				\\ {(hg)(f_1+f_2)}\ar[r]_-{\dot{\alpha}_{h,g,f_1+f_2}}
				&{h(g(f_1+f_2))}\ar@{<-}[u]_-{h\varphi^g_{f_1,f_2}}
			}\end{xym}
			\begin{xym}\xymatrix@=40pt{
				0\ar[r]^{\varphi^h_0}\ar[d]_{\varphi^{hg}_0}
				&h0\ar[d]^{h\varphi^g_0}
				\\ (hg)0\ar[r]_{\dot{\alpha}_{hg0}}
				&h(g0)
			}\end{xym}
		\item \begin{xym}\label{preadde}\xymatrix@R=40pt@C=50pt{
				{(hg_1)f+(hg_2)f}\ar[d]_-{\psi^{hg_1,hg_2}_{f}}
					\ar[r]^-{\dot{\alpha}_{hg_1f}+\dot{\alpha}_{hg_2f}\,}
				&{h(g_1f)+h(g_2f)}\ar[d]^-{\varphi^h_{g_1f,g_2f}}
				\\ (hg_1+hg_2)f\ar[d]_{\varphi^h_{g_1,g_2}f}
				& h(g_1f+g_2f)
				\\ {(h(g_1+g_2))f}\ar[r]_-{\dot{\alpha}_{h,g_1+g_2,f}}
				&{h((g_1+g_2)f)}\ar@{<-}[u]_-{h\psi^{g_1,g_2}_f}
			}\end{xym}
			\begin{xym}\xymatrix@=40pt{
				0\ar[d]_{\psi^0_f}\ar@/^1.5pc/[dr]^{\varphi^h_0}
				\\ 0f\ar[d]_{\varphi^h_0f}
				&h0\ar[d]^{h\psi^0_f}
				\\ (h0)f\ar[r]_{\dot{\alpha}_{h0f}}
				&h(0f)
			}\end{xym}
		\item \begin{xym}\label{preaddf}\xymatrix@R=40pt@C=50pt{
				{(h_1g)f+(h_2g)f}\ar[d]_-{\psi^{h_1g,h_2g}_{f}}
					\ar[r]^-{\dot{\alpha}_{h_1gf}+\dot{\alpha}_{h_2gf}\,}
				&{h_1(gf)+h_2(gf)}\ar[dd]^-{\psi^{h_1,h_2}_{gf}}
				\\ (h_1g+h_2g)f\ar[d]_{\psi^{h_1,h_2}_{g}f}
				\\ {((h_1+h_2)g)f}\ar[r]_-{\dot{\alpha}_{h_1+h_2,g,f}}
				&{(h_1+h_2)(gf)}
			}\end{xym}
			\begin{xym}\xymatrix@=40pt{
				0f\ar[d]_{\psi_g^0f}
				&0\ar[l]_{\psi^0_f}\ar[d]^{\psi^0_{gf}}
				\\ (0g)f\ar[r]_{\dot{\alpha}_{0gf}}
				&0(gf)
			}\end{xym}
		\item \begin{xym}\label{preaddg}\xymatrix@=40pt{
				1_Bf_1+1_Bf_2\ar[dr]^{\dot{\lambda}_{f_1}+\dot{\lambda}_{f_2}}
					\ar[d]_{\varphi^{1_B}_{f_1,f_2}}
				&&f_11_A+f_21_A\ar[dl]_{\dot{\rho}_{f_1}+\dot{\rho}_{f_2}}
					\ar[d]^{\psi^{f_1,f_2}_{1_A}}
				\\ 1_B(f_1+f_2)\ar[r]_-{\dot{\lambda}_{f_1+f_2}}
				&f_1+f_2
				&(f_1+f_2)1_A\ar[l]^-{\dot{\rho}_{f_1+f_2}}
			}\end{xym}
			\begin{xym}\xymatrix@=40pt{
				0\ar[d]_{\varphi^{1_B}_0}\ar@{=}[dr]
				&&0\ar@{=}[dl]\ar[d]^{\psi^0_{1_A}}
				\\ 1_B0\ar[r]_{\dot{\lambda}_0}
				&0
				&01_A\ar[l]^{\dot{\rho}_0}
			}\end{xym}
	\end{enumerate}
	\renewcommand{\theenumi}{\arabic{enumi}}\renewcommand{\labelenumi}{\theenumi .}

	In the one-object case, we recover exactly the axioms of Laplaza \cite{Laplaza1972a}
	for a category with two monoidal structures, one being distributive
	with respect to the other, except that he does not work with groupoids and
	adds axioms concerning the symmetry of the product (here the composition), which
	we do not assume here; Kapranov and Voevodsky
	\cite{Kapranov1994a}
	call that a “ring category”; they use the category
	of vector spaces on a field $K$, equipped with the tensor product and the direct sum
	to define their 2-vector spaces on $K$;  these categories are also called
	bimonoidal categories or rig-like categories.

	\begin{ex}\begin{enumerate}
		\item A $\Ens$-category seen as a discrete $\Gpd$-category is
			presemiadditive in this 2-dimensional sense if and only if it is
			presemiadditive in the usual 1-dimensional sense.  So there is no possible
			terminological ambiguity.
		\item The $\Gpd$-category $\dCMon$ is itself presemiadditive.
			The symmetric 2-monoid structure of $\dCMon(\G,\H)=[\G,\H]$ has been described
			in Definition \ref{homintdcmon}.  If $G\col \G\ra\H$ and $H_1,H_2\col \H
			\ra\cat{K}$ are morphisms in $\dCMon$, then $\psi_G^{H_1,H_2}$ is
			the identity at each point, whereas if $G_1,G_2\col \G\ra\H$ and
			$H\col \H\ra\cat{K}$, $\varphi_{G_1,G_2}^H$ is $\varphi^H_{G_1-,G_2-}$.
	\end{enumerate}\end{ex}

\subsection{Symmetric 2-groups and $\CGS$-categories}

	2-groups, also called $\Cat$-groups or Gr-categories,
	are 2-monoids where each object is an
	equivalence (if we see the 2-monoid as a one-object $\Gpd$-category).
	See \cite{Vitale2002a,Baez2004a} for general results about 2-groups.

	\begin{df}\index{2-group}\label{defdgroup}
		A \emph{2-group} is a 2-monoid (monoidal groupoid) such that, for
		each object $A\col \G$, there exist an object $A^*\col \G$ and arrows
		$\eps_A\col A\tens A^*\ra I$ and $\eta_A\col I\ra A^*\tens A$ satisfying the
		triangular identities: the two following composites
		are equal respectively to $1_{A^*}$ and to $1_A$.
		\begin{eqn}
			A^*\overset{l^{-1}}\longrightarrow I\tens A^*\xrightarrow{\eta_A\tens 1}
			(A^*\tens A)\tens A^*\overset{a}\longrightarrow
			A^*\tens(A\tens A^*)\xrightarrow{1\tens\eps_A} A^*\tens I\overset{r}\longrightarrow A^*
		\end{eqn}
		\begin{eqn}
			A\overset{r^{-1}}\longrightarrow A\tens I\xrightarrow{1\tens\eta_A}
			A\tens(A^*\tens A)\overset{a^{-1}}\longrightarrow
			(A\tens A^*)\tens A\xrightarrow{\eps_A\tens 1} I\tens A\overset{l}\longrightarrow A
		\end{eqn}
	\end{df}

	A very useful result is that from $\varepsilon_A$ alone we can construct 
	an arrow $\eta_A$ which satisfies the triangular identities with $\varepsilon_A$
	\cite{Joyal1986a}.
	So to turn a 2-monoid into a 2-group, it is sufficient to give for each object $A$
	an object $A^*$ such that $A\tens A^*\simeq I$.
	
	If $\G$ is a 2-group, we can define a functor $(-)^*\col \G\op\ra\G$, which maps
	$A$ to $A^*$ and an arrow $f\col B\ra A$ to the composite, denoted by $f^*$.
	\begin{multline}\stepcounter{eqnum}
		A^*\xrightarrow{l^{-1}}I\tens A^*\xrightarrow{\eta_B\tens 1} (B^*\tens B)\tens A^*
		\xrightarrow{(1\tens f)\tens 1} (B^*\tens A)\tens A^*\\
		\overset{a}\longrightarrow B^*\tens (A\tens A^*)
		\xrightarrow{1\tens\eps_A} B^*\tens I\overset{r}\longrightarrow B^*
	\end{multline}
	Then we can check that $\eta$ and $\eps$ are dinatural transformations.
	
	We can also prove that every monoidal functor preserves the inverse of an object
	(see for example \cite[Proposition 2.3]{Baez2004a}).
	So monoidal functors do not have to satisfy additional properties
	to be morphisms of 2-groups.
	
	\begin{df}
		We denote by $\CG$ the full sub-$\Gpdp$-category of $\dMon$ whose objects
		are the 2-groups.\index{2-gp@$\CG$}
	\end{df}
	
	If $\G$ is equipped with a symmetry, we have a \emph{symmetric 2-group}\index{2-group!symmetric}\index{symmetric!2-group}
	(or symmetric $\caspar{Cat}$-group or Picard category \cite{Deligne1973a}).
	We denote by $\CGS$\index{2-sgp@$\CGS$}
	the full sub-$\Gpdp$-category of $\dCMon$ whose objects are the symmetric 2-groups.
		
	We can simplify the description of arrows and 2-arrows of $\CGS$ in the following
	way.  First, for a symmetric monoidal functor
	$F\col \G\ra\H$, we can prove, thanks to the invertibility of objects, that $\varphi_0 =$
	\begin{multline}\stepcounter{eqnum}
		I\xrightarrow{\eps^{-1}_{FI}}FI\tens FI^*\xrightarrow{Fr^{-1}\tens 1}
		F(I\tens I)\tens FI^*\xrightarrow{\varphi^{-1}\tens 1}
		(FI\tens FI)\tens FI^*\\
		\overset{a}\longrightarrow
		FI\tens (FI\tens FI^*)\xrightarrow{1\tens \eps_{FI}}
		FI\tens I\overset{r}\longrightarrow FI.
	\end{multline}
	And, if we assume that we have only $\varphi_{AB}$, we can recover $\varphi_0$
	by defining it as this composite; then we can check that axiom
	\ref{axfoncmonunite} holds.  We can thus remove from the definition
	$\varphi_0$ and axiom \ref{axfoncmonunite}.	
	Moreover, by using this definition of $\varphi_0$, we can deduce
	axiom \ref{axtnmonunit} of monoidal natural transformation from the axiom
	\ref{axtnmonass} and from the naturality; we can thus also remove this axiom.
	
	In the same way, when $\G,\H, \cat{K}$ are symmetric 2-groups, in the definition
	of symmetric bimonoidal functor $\G\times\H\ra\cat{K}$, we can remove
	$\varphi_0^H$ and $\psi^0_G$, and the parts of condition 3
	(or of the equivalent condition 3') concerning these transformations.
	These simplifications allow to reduce significantly the number of axioms of
	preadditive $\Gpd$-categories.

	\begin{df}
		Let be $\A,\B\col \CGS$.  The symmetric 2-group $[\A,\B]$ is the symmetric
		2-monoid $[\A,\B]$ of Definition \ref{homintdcmon}, whose objects are the
		symmetric monoidal functors $\A\ra\B$ and whose arrows are the
		monoidal natural transformations between them.  If $F\col \A\ra\B$, its inverse
		$F^*\col \A\ra\B$ maps an object $A\col \A$ to $(FA)^*$ and an arrow $f\col A\ra B$
		to $(Ff^*)^{-1}\col FA^*\ra FB^*$; we define $\eps_F \col  F\tens F^*\ra 0$ pointwise.
	\end{df}

	In the same way, when $\G,\H, \cat{K}$ are symmetric 2-groups,
	$\Bimon(\G\times\H,\cat{K})$ is a symmetric 2-group; we can prove that by transfering
	the inverses of objects from $[\G,[\H,\cat{K}]]$, which is
	a symmetric 2-group, as we know by the previous proposition.

	We can now give the definition of preadditive $\Gpd$-categories.
	This notion of preadditivity for 2-categories was introduced (with an axiom missing) by 
	Benjamin Drion in \cite{Drion2002a}.
	
	\begin{df}\label{defpreadd}\index{preadditive Gpd-categorie@preadditive $\Gpd$-category}%
	\index{Gpd-category@$\Gpd$-category!preadditive}%
	\index{2-sgp-categorie@$\CGS$-category}
		A \emph{preadditive} $\Gpd$-category (or a \emph{$\CGS$-category})
		is a presemiadditive $\Gpd$-category $\C$
		such that, for all $A,B\col \C$, $\C(A,B)$ is a symmetric 2-group.
	\end{df}
	
	By using the simplifications above-mentioned, we see that we can remove from this
	definition $\varphi^0$ and $\psi_0$ and the conditions involving them.
	From the list of elementary conditions following Definition \ref{semipreadd} 
	remain only equations \ref{preadda1}, \ref{preadda2}, \ref{preaddb1}, 
	\ref{preaddb2}, \ref{preaddc}, \ref{preaddd}, \ref{preadde}, \ref{preaddf}
	and \ref{preaddg}.

	Version 2 of Definition \ref{semipreadd} minus axiom (c) is equivalent
	to the definition given in \cite{Drion2002a}.  There are a few differences 
	in the presentation:
	instead of asking for the naturality of $\varphi$ in $h$, he asked that
	for each $\beta\col h\Ra h'\col B\ra C$, the natural transformation
	$\beta * -\col h\circ -\Ra h'\circ -$ be monoidal, which is equivalent;
	and dually, instead of asking for the naturality of $\psi$ in $g$, he asked that for each	
	$\alpha\col g\Ra g'\col B\ra C$, the natural transformation
	$\alpha * -\col g\circ -\Ra g'\circ -$ be monoidal.  Moreover, conditions
	(d), (f) and (g) were replaced by the equivalent requirement
	that $\C(A,-)\col \C\ra\CGS$ and $\C(-,D)\col \C\ra\CGS$ be
	$\Gpd$-functors, for each $A, D\col \C$.
		
	A one-object preadditive $\Gpd$-category is what could be called a
	2-ring\index{2-ring}, by analogy with the 1-dimensional case, where a one-object
	additive $\Ens$-category is a ring.  In fact one-object preadditive $\Gpd$-categories
	coincide with \emph{$\mathrm{Ann}$-categories}
	introduced by Nguyen Tien Quang \cite{Quang2007c}. 
	The \emph{categorical rings} defined by Mamuka Jibladze and Teimuraz Pirashvili
	\cite{Jibladze2007a} are according to them equivalent to the $\mathrm{Ann}$-categories of Quang, but it seems to be necessary to add a condition to recover the $\mathrm{Ann}$-categories
	\cite{Quang2007a}.
	
	\begin{ex}~
	\begin{enumerate}
		\item A $\Ens$-category seen as a locally discrete $\Gpd$-category is
			preadditive in this 2-dimensional sense if and only if it is preadditive
			in the usual 1-dimensional sense.
		\item The $\Gpd$-category $\CGS$ is itself preadditive, with the same
			structure as the one of $\dCMon$.
		\item We will prove (Corollary \ref{2abPexadd})
			that every 2-abelian $\Gpd$-category is preadditive.
		\item It seems
			that the $2\text{-}\Ab\text{-}$categories introduced
			by Nelson Martins-Ferreira \cite{Martins-Ferreira2004a}
			coincide with the $\CGS$-categories
			$\C$ where for each $A,B\col \C$, the symmetric 2-group $\C(A,B)$
			is strict (i.e.\ the transformations $\alpha$, $\lambda$, $\rho$,
			$\gamma$ and $\eta$ are identities) and the composition is strict bimonoidal
			(i.e.\ the transformations $\psi$ and $\varphi$ are identities).
	\end{enumerate}\end{ex}

\section{Additive $\Gpd$-categories}

\subsection{First elements of matrix algebra}\label{premelcalmat}

	Let us assume that $\C$ is a $\Gpd$-category with finite products and
	coproducts.  We have thus for each family of objects $(A_k)_{1\leq k\leq n}$
	a product $\prod_{k=1}^n A_k$, with projections $p_k$, and a coproduct
	$\sum_{k=1}^n A_k$, with injections $i_k$.  We assume that, in
	the case of families with one object $A$, we take as product $A$ itself with
	projection $1_A$, and as coproduct $A$ with injection $1_A$.
	
	The universal property of the product gives us,
	for each family of arrows $(X\overset{a_k}\ra A_k)_{1\leq k\leq n}$ an arrow
	\begin{eqn}\begin{pmatrix}
		a_1\\
		a_2\\
		\cdots\\
		a_n
	\end{pmatrix}\col X\ra\prod_{k=1}^n A_k,\end{eqn}
 	with 2-arrows $\pi_{k_0}\col p_{k_0}(a_k)\Ra a_{k_0}$.
	Dually, for each family of arrows
	$(A_k\overset{a_k}\ra Y)_{1\leq k\leq n}$, we have an arrow
	\begin{eqn}\begin{pmatrix}
		a_1 &a_2 &\cdots &a_n
	\end{pmatrix}\col \sum_{k=1}^n A_k\ra Y,\end{eqn}
	and 2-arrows $\iota_{k_0}\col (a_k) i_{k_0}\Ra a_{k_0}$.
	In the case where $n=1$, we take in the horizontal and vertical cases
	$(a_k)_{1\leq k\leq 1}= a_1$, with $\pi_1$ or $\iota_1$ equal to the identity.
	
	Let $(A_k)_{1\leq k\leq n}$ and $(B_j)_{1\leq j\leq m}$ be two finite families
	of objects of $\C$. By the universal property of the product and of the coproduct, the functor
	\begin{eqn}\label{equcoprodprod}
		\C(\sum_{k=1}^n A_k,\prod_{j=1}^m B_j)
		\xrightarrow{(p_j\circ -\circ i_k)_{jk}}
		\prod_{\substack{1\leq k\leq n\\ 1\leq j\leq m}}\C(A_k,B_j)
	\end{eqn}
	is an equivalence of groupoids; let us denote by $\Phi$ its inverse. We denote the image
	of $(f_{jk})_{\substack{1\leq j\leq m \\ 1\leq k\leq n}}$ under $\Phi$
	also by $(f_{jk})$.\index{matrix}  We have thus an isomorphism
	\begin{eqn}
		\eta_{j_0k_0}\col p_{j_0}(f_{jk})i_{k_0}\Ra f_{j_0k_0}
	\end{eqn}
	such that
	\begin{xyml}\begin{gathered}\xymatrix@=50pt{
		{\sum_{k=1}^n A_k}\ar[r]^-{(f_{jk})}\ar[dr]_{(f_{j_0k})}
			\drtwocell\omit\omit{_<-3.5>{\;\;\;\,\pi_{j_0}}}
		&{\prod_{j=1}^m B_j}\ar[d]^{p_{j_0}}
		\\ A_k\ar[u]^{i_{k_0}}\ar[r]_{f_{j_0k_0}}
			\rtwocell\omit\omit{_<-2.5>{\;\;\,\iota_{k_0}}}
		& B_j
	}\end{gathered}=\eta_{j_0k_0}=\begin{gathered}\xymatrix@=50pt{
		{\sum_{k=1}^n A_k}\ar[r]^-{(f_{jk})}
		&{\prod_{j=1}^m B_j}\ar[d]^{p_{j_0}}
		\\ A_k\ar[u]^{i_{k_0}}\ar[r]_{f_{j_0k_0}}
			\rtwocell\omit\omit{_<-2.5>{\;\;\;\pi_{j_0}}}
			\ar[ur]_{(f_{j_0k})}\urtwocell\omit\omit{_<-3.5>{\;\;\;\iota_{k_0}}}
		& B_j
	}\end{gathered}.\end{xyml}
	In the case $m=1$, we take $(f_{jk})=(f_{1k})$, with $\pi_1=1$.
	In the case $n=1$, we take $(f_{jk})=(f_{j1})$, with $\iota_1=1$.
	
\subsection{Finite biproducts}

	From now on, we assume that all $\Gpd$-categories are strictly described and that
	all $\Gpdp$-categories are strictly described.

	The equivalence of conditions 5 to 7 of the following proposition has been
	proved by Benjamin Drion \cite{Drion2002a}.
	
	\begin{pon}\label{caracbiprod}
		Let $\C$ be a $\Gpdp$-category and $A_1, A_2\col \C$.
		The following conditions are equivalent. If they hold,
		we say that the data of condition 4 form a
		\emph{biproduct}\index{biproduct} of $A_1$ and $A_2$.  If for each pair of
		objects of $\C$ these conditions hold, we say that $\C$ \emph{has all
		binary biproducts}.
		\begin{enumerate}
			\item The product $A_1\overset{p_1}\longleftarrow A_1\times A_2
				\overset{p_2}\longrightarrow A_2$ exists and 
				\begin{eqn}
					A_1\overset{\vervec{1}{0}} \longrightarrow A_1\times A_2
					\overset{\vervec{0}{1}}\longleftarrow A_2
				\end{eqn}
				is a coproduct.
			\item The coproduct $A_1\overset{i_1}\longrightarrow A_1 + A_2
				\overset{i_2}\longleftarrow A_2$ exists and 
				\begin{eqn}
					A_1\overset{(1\; 0)}\longleftarrow A_1 + A_2
					\overset{(0\; 1)} \longrightarrow A_2
				\end{eqn}
				is a product.
			\item The product and the coproduct of $A_1$ and $A_2$ exist
				and the identity matrix
				\begin{eqn}\label{matrid}
					\begin{pmatrix}
						1_{A_1} & 0 \\
						0 &1_{A_2}
					\end{pmatrix}\col A_1+A_2 \longrightarrow A_1\times A_2
				\end{eqn}
				is an equivalence.
			\item For $1\leq j,k\leq 2$, there exist an object, arrows and 2-arrows
				\begin{xym}\xymatrix@=20pt{
					&{A_1\oplus A_2}\ar[dr]^{p_j}
					\\ A_k\ar[ur]^{i_k}\ar[rr]_{\delta_{jk}}
						\rrtwocell\omit\omit{_<-2.7>{\;\;\;\eta_{jk}}}
					&& A_j,
				}\end{xym}
				where $\delta_{jk} = 1_{A_k}$ if $j=k$ and $0$ otherwise,
				such that $p_1, p_2$ is a product and $i_1, i_2$ is a coproduct.
		\end{enumerate}
		If $\C$ is presemiadditive, we can add the following conditions.
		\begin{enumerate}
			\item[5.] The product $A_1\times A_2$ exists.
			\item[6.] The coproduct $A_1 + A_2$ exists.
			\item[7.] For $1\leq j,k\leq 2$, there exist an object, arrows and 2-arrows
				\begin{xym}\xymatrix@=20pt{
					&{A_1\oplus A_2}\ar[dr]^{p_j}
					\\ A_k\ar[ur]^{i_k}\ar[rr]_{\delta_{jk}}
						\rrtwocell\omit\omit{_<-2.7>{\;\;\;\eta_{jk}}}
					&& A_j,
				}\end{xym}
				as well as a 2-arrow
				\begin{eqn}
					\omega\col i_1p_1 + i_2p_2 \Ra 1_{A_1\oplus A_2},
				\end{eqn}
				such that the following diagrams commute.
		\end{enumerate}
				\begin{xym}\label{diagomegp}\xymatrix@R=40pt@C=25pt{
					p_1i_1p_1+p_1i_2p_2\ar[r]^-{\varphi}
						\ar[d]_{\substack{\eta_{11}p_1\\ +\eta_{12}p_2}}
					& p_1(i_1p_1+i_2p_2)\ar[d]^{p_1\omega}
					\\ p_1+0\ar[r]_\rho
					&p_1
				}\;\;\;\xymatrix@R=40pt@C=25pt{
					p_2i_1p_1+p_2i_2p_2\ar[r]^-{\varphi}
						\ar[d]_{\substack{\eta_{21}p_1\\ +\eta_{22}p_2}}
					& p_2(i_1p_1+i_2p_2)\ar[d]^{p_2\omega}
					\\ 0+p_2\ar[r]_\lambda
					&p_2
				}\end{xym}
		\begin{enumerate}
			\item[8.] Like condition 7, but with the following diagrams.
		\end{enumerate}
				\begin{xym}\label{diagomegpbis}\xymatrix@R=40pt@C=25pt{
					i_1p_1i_1+i_2p_2i_1\ar[r]^-{\psi}
						\ar[d]_{\substack{i_1\eta_{11}\\ +i_2\eta_{21}}}
					& (i_1p_1+i_2p_2)i_1\ar[d]^{\omega i_1}
					\\ i_1+0\ar[r]_\rho
					&i_1
				}\;\;\;\xymatrix@R=40pt@C=25pt{
					i_1p_1i_2+i_2p_2i_2\ar[r]^-{\psi}
						\ar[d]_{\substack{i_1\eta_{12}\\ +i_2\eta_{22}}}
					& (i_1p_1+i_2p_2)i_2\ar[d]^{\omega i_2}
					\\ 0+i_2\ar[r]_\lambda
					&i_2
				}\end{xym}
	\end{pon}
	
		\begin{proof}
			{\it 4 $\Rightarrow$ 3. } Indeed, $(\eta_{jk})\col 
			(p_j\circ 1_{A_1\oplus A_2}\circ i_k)\Ra (\delta_{jk})$ is an isomorphism in 
			$\prod_{\substack{1\leq k\leq 2\\ 1\leq j\leq 2}}\C(A_k,A_j)$ and, since
			functor \ref{equcoprodprod} is an equivalence, $1_{A_1\oplus A_2}$
			is isomorphic to the identity matrix.
			
			{\it 3 $\Rightarrow$ 2. } Indeed,
			\begin{align}\begin{split}\stepcounter{eqnum}
				p_1 \begin{pmatrix}1_{A_1}&0 \\ 0 &1_{A_2}\end{pmatrix} &\simeq (1_{A_1}\; 0)\\
				p_2 \begin{pmatrix}1_{A_1}&0 \\ 0 &1_{A_2}\end{pmatrix} &\simeq (0\; 1_{A_2})
			\end{split}\end{align}
			and, since the identity matrix is an equivalence, these two composites also form
			a product.
			
			{\it 3 $\Rightarrow$ 1. } The proof is dual.
			
			{\it 2 $\Rightarrow$ 4. }  It suffices to take $p_1=(1\; 0)$,
			$p_2=(0\; 1)$ and $\eta_{jk}=\pi_j$.
			
			{\it 1 $\Rightarrow$ 4. } The proof is dual.
			
			\vspace{1em}
			Let us assume now that $\C$ is presemiadditive.

			{\it 2 $\Rightarrow$ 6. } Obvious.

			{\it 6 $\Rightarrow$ 8. } 
			We set $p_1 \eqdef (1\; 0)$ and $p_2 \eqdef (0\; 1)$ and
			$\eta_{jk} \eqdef \iota_k$, and diagrams \ref{diagomegpbis}
			induce $\omega$ by the universal property of the coproduct.
			
			{\it 8 $\Rightarrow$ we have a coproduct $A_1\overset{i_1}
			\longrightarrow A_1\oplus A_2\overset{i_2}\longleftarrow A_2$. }
			Indeed, if $A_1\overset{y_1}\longrightarrow Y\overset{y_2}\longleftarrow A_2$
			is a rival, we have $A_1\oplus A_2\xrightarrow{y_1p_1+y_2p_2} Y$, with
			$(y_1p_1+y_2p_2)i_1\xLongrightarrow{\psi^{-1}} y_1p_1i_1+y_2p_2i_1 
			\xLongrightarrow{y_1\eta_{11}+y_2\eta_{21}} y_1+0\overset{\rho}\Longrightarrow y_1$.
			The proof is similar for $i_2$ and $y_2$.  Next, if we have $u,v\col A_1\oplus A_2\ra Y$
			with $\alpha_1\col ui_1\Ra vi_1$ and $\alpha_2\col ui_2\Ra vi_2$,
			we set $\alpha$ equal to the composite
			\begin{multline}\stepcounter{eqnum}
				u\circ 1\xLongrightarrow{u\omega^{-1}}u(i_1p_1+i_2p_2)
				\xLongrightarrow{\varphi^{-1}} ui_1p_1+ui_2p_2 
				\xLongrightarrow{\alpha_1p_1+\alpha_2p_2}vi_1p_1+vi_2p_2
				\\ \overset{\varphi}\Longrightarrow v(i_1p_1+i_2p_2)\xLongrightarrow{v\omega}v\circ 1
			\end{multline}
			and we check, by using two diagrams of condition 8 and of the axioms
			of presemiadditive $\Gpd$-category, that $\alpha i_1 = \alpha_1$ and $\alpha i_2
			=\alpha_2$.  Unicity is easy to check with the help of $\omega$.
			
			{\it 1 $\Rightarrow$ 5 $\Rightarrow$ 7 $\Rightarrow$ 
			we have a product $p_1$, $p_2$. } The proof is dual.
			
			{\it 8 $\Leftrightarrow$ 7. }   We have just proved that
			 condition 8 implies that $i_1$ and $i_2$ form a coproduct.  They are thus
			jointly cofaithful.  Then we test the two diagrams of
			condition 7 with $i_1$ and $i_2$ to check their commutativity.
			Dually, condition 7 implies condition 8.
							
			{\it 8 $\Rightarrow$ 4. } Indeed, $i_1$ and $i_2$
			form a coproduct and, by condition 7,
			$p_1$ and $p_2$ form a product.
		\end{proof}

	A first example of $\Gpdp$-category with biproducts is $\CGS$; the 
	biproduct of two symmetric 2-groups is simply their cartesian product.
	More generally, 2-abelian $\Gpd$-categories
	have all finite biproducts.  To prove that, we need the following lemma.

	\begin{lemm}
		Let $\C$ be a $\Gpdp$-category and $A_1\overset{i_1}\longrightarrow A_1 + A_2
		\overset{i_2}\longleftarrow A_2$ be a coproduct.  
		\begin{enumerate}
			\item The arrows $i_1$ and $i_2$ are faithful.
			\item $(0\;1_{A_2})= \Coker i_1$:
				\begin{xym}\xymatrix@=40pt{
					A_1\ar[r]^-{i_1}\rrlowertwocell_0<-11>{\;\,\iota_1}
					& A_1+A_2\ar[r]^-{(0\;1_{A_2})}
					&A_2
				}\end{xym}
				 By symmetry, $(1_{A_1}\; 0)= \Coker i_2$.
		\end{enumerate}
	\end{lemm}
	
		\begin{proof}
			As we have $(1_{A_1}\;0)i_1\overset{\iota_1}\Longrightarrow 1_{A_1}$, 
			$i_1$ is faithful; by symmetry, $i_2$ is faithful.
			
			Let us prove point 2.
			Let be $y\col A_1+A_2\ra Y$, with $\varphi\col yi_1\Ra 0$.  We set
			$y'\eqdef y i_2\col A_2\ra Y$.
			The 2-arrows $y'(0\;1)i_1\overset{y'\iota_1\;\,\,}\Longrightarrow y'0=0
			\overset{\varphi^{-1}}\Longrightarrow yi_1$ and
			$y'(0\;1)i_2\overset{y'\iota_2\;\;\,}\Longrightarrow y'=yi_2$ induce, by the
			universal property of the coproduct, a 2-arrow $\gamma\col y'(0\;1)\Ra y$
			such that $\gamma i_1=\varphi^{-1}(y'\iota_1)$ and $\gamma i_2=y'\iota_2$.
			Then $\varphi(\gamma i_1) = \varphi\varphi^{-1}(y'\iota_1)=y'\iota_1$.
			
			Next, if $\gamma\col u(0\;1)\Ra v(0\;1)$ is such that
			$v\iota_1\circ\gamma i_1\circ u\iota_1^{-1} = 1_0$, we set
			$\alpha$ equal to the composite $u\overset{u\iota_2^{-1}}\Longrightarrow
			u(0\;1)i_2\overset{\gamma i_2\;\,}\Longrightarrow v(0\;1)i_2
			\overset{v\iota_2\;\,}\Longrightarrow v$.  We check that
			we have $\gamma=\alpha(0\;1)$ by testing with $i_1$ and $i_2$
			(which are jointly cofaithful, by the universal property
			of the coproduct).
		\end{proof}

	A corollary of this lemma is that, like in dimension 1,
	biproducts form extensions.

	\begin{coro}\label{biprodextens}
		Let $\C$ be a $\Gpdp$-category and $A_1\oplus A_2$ be a biproduct in $\C$.  Then
		the following sequences are extensions.
		\begin{xym}\xymatrix@=40pt{
			A_1\ar[r]^-{i_1}\rrlowertwocell_0<-11>{\;\;\,\eta_{21}}
			& {A_1\oplus A_2}\ar[r]^-{p_2}
			&A_2
		}\end{xym}
		\begin{xym}\xymatrix@=40pt{
			A_2\ar[r]^-{i_2}\rrlowertwocell_0<-11>{\;\;\,\eta_{12}}
			& {A_1\oplus A_2}\ar[r]^-{p_1}
			&A_1
		}\end{xym}
	\end{coro}
	
	\begin{pon}\label{gpd2abbiprod}
		Let $\C$ be a 2-abelian $\Gpd$-category.  Then $\C$ has all binary biproducts.
	\end{pon}
	
		\begin{proof}
			We prove condition 3 of Proposition \ref{caracbiprod}.
			Let us denote by $I$ the identity matrix. We will prove that $I$ is fully 0-faithful.
			By duality, $I$ will also be fully 0-cofaithful, and thus an equivalence.
			
			Let be $x\col X\ra A_1+A_2$ and $\chi\col Ix\Ra 0$.  By the previous lemma, we
			know that $(0\;1_{A_2})=\Coker i_1$ and that $i_1$ is faithful thus, by
			2-Puppe-exactness of $\C$, $i_1$ with $\iota_1$ is the kernel of $(0\;1_{A_2})$.
			Since we have $\pi_{2}\col p_2 I\Ra (0\;1_{A_2})$, we can deduce
			that $i_1$, with $\eta_{21}=\iota_1\circ\pi_{2}i_1$ is the kernel of $p_2 I$.
			In the same way, $i_2$, with $\eta_{12}$ is the kernel of $p_1 I$.
			
			Therefore $x$, with $p_2\chi$, is a rival of $i_1$ and there exist, by the
			universal property of the kernel, $x_1\col X\ra A_1$ and $\chi_1\col x\Ra i_1x_1$
			such that the lower part of the following diagram is equal to $p_2\chi$.
			In the same way, we have $x_2\col X\ra A_2$ and $\chi_2$ such that the upper part
			of the diagram is equal to $p_1\chi$.
			\begin{xym}\xymatrix@=40pt{
				&A_2\ar[r]^0\ar[d]^{i_2}
				&A_1
				\\ X\ar[r]^-{x}\ar@/^1.4pc/[ur]^{x_2}\ar@/_1.4pc/[dr]_{x_1}
					\rtwocell\omit\omit{_<5>{\,\;\chi_1}}
					\rtwocell\omit\omit{^<-5>{\chi_2}\;\,}
				&A_1+A_2\ar[r]^-I
					\rtwocell\omit\omit{_<5.5>{\,\;\;\;\eta_{21}}}
					\rtwocell\omit\omit{^<-5.5>{\eta_{12}}\,\;\;\;}
				&A_1\times A_2\ar[u]_{p_1}\ar[d]^{p_2}
				\\ &A_1\ar[r]_0\ar[u]_{i_1}
				&A_2
			}\end{xym}
			We set then $\alpha\col x_1\Ra 0$ equal to the composite of the following diagram.
			\begin{xym}\xymatrix@=40pt{
				&A_1\ar@{=}[dr]\drtwocell\omit\omit{<3>{\;\;\,\iota_1^{-1}}}\ar[d]_{i_1}
				\\ X\ar[r]|-x\ar[ur]^{x_1}\rtwocell\omit\omit{<3>{\;\;{\chi_2}}}
					\ar[dr]_{x_2}\rtwocell\omit\omit{<-3>{\;\;\;\;\chi_1^{-1}}}
				&A_1+A_2\ar[r]|-{(1\;0)}
				&A_1
				\\ &A_2\ar[ur]_0\urtwocell\omit\omit{<-3>{\,\iota_2}}\ar[u]^{i_2}
			}\end{xym}
			Then $\gamma$, defined by the composite
			$x\overset{\chi_1\,}\Longrightarrow i_1x_1\overset{i_1\alpha\;\;}
			\Longrightarrow i_1 0=0$, is such that $\chi = I\gamma$; to check that,
			we test this equation with $p_1$ and $p_2$, which are jointly faithful.  
			Since $\C$ has all $\Sigma$s, we can conclude that $I$ is fully
			0-faithful, by Proposition \ref{simpldefplzfidsigm}.
		\end{proof}
	
	Products in $\Gpd$ are usually described with a strict universal property:
	if $\A$ and $\B$ are groupoids and $F\col\cat{X}\ra\A$
	and $G\col\cat{X}\ra\B$ are functors, we have an induced functor
	$(F,G)\col\cat{X}\ra\A\times\B$ such that $p_1(F,G)\equiv F$ and $p_2(F,G)\equiv G$;
	we can thus take the identity for the natural transformations $\pi_1$ and $\pi_2$.
	
	Let $\C$ be a $\Gpdp$-category with zero object and biproducts.
	By taking the full image of the Yoneda embedding
	\begin{eqn}
		\C\overset{Y}\ra[\C\op,\Gpd],
	\end{eqn}
	which is locally an equivalence, we can replace $\C$ by an equivalent $\Gpd$-category
	(the full sub-$\Gpd$-category of $[\C\op,\Gpd]$ whose objects are the functors
	equivalent to a representable, which is stable under products because $\C$
	has all products) where the product also has a strictly described universal property
	because, in $\Gpd$, and
	thus in $[\C\op,\Gpd]$, the product is strictly described.  (On the other hand, we
	cannot require at the same time that the coproduct be strictly described; see
	Baues, Jibladze and Pirashvili \cite[section 5]{Baues2006a}.)
	
	Henceforth we assume that the universal property of the product is strictly described
	in $\C$. For every object $A$ we fix an exponentiation $A\times \ldots \times A$ of $A$
	for each natural exponent $n$.  We take for $A^1$ the object $A$ itself
	and for $A^0$ the terminal object $1$.  And we set
	$i_1 \eqdef \vervec{1}{0}$ and $i_2 \eqdef \vervec{0}{1}$.

\subsection{Bimonoids}
	
	The goal of this subsection and of the following subsections is to prove that
	when a $\Gpd$-category $\C$ has all finite biproducts, it is presemiadditive.
	To do that, in this subsection, we prove that the existence of finite biproducts
	implies that each object of $\C$ is equipped with a bimonoid structure and that
	the $\Gpd$-functor $\Phi\col \C\ra\caspar{Bimon}(\C)$,
	mapping an object of $\C$ to the bimonoid constructed on this object,
	has good properties.
	
	To begin with, let us recall the definition of internal (symmetric) monoids
	in a $\Gpd$-category
	with products (they are defined with the name of (symmetric) pseudomonoids in the context
	of Gray-monoids by Day and Street \cite{Day1997a}; the pseudomonoids in
	$\Cat$ are the monoidal categories; see \cite{Baez2004a} for internal
	(2-)groups in a 2-category with products), and next the definition of internal bimonoids.
	
	\begin{df}\index{symmetric!monoid}\index{monoid!symmetric}
		Let $\C$ be a $\Gpd$-category with finite products.
		A \emph{symmetric monoid} in $\C$ consists of an object $A\col \C$,
		a multiplication $m\col A\times A\ra A$, a unit $e\col 1\ra A$,
		and 2-arrows
		\begin{xym}\xymatrix@=40pt{
			{A\times A\times A}\ar[d]_{1\times m}\ar[r]^-{m\times 1}
				\drtwocell\omit\omit{\alpha}
			&A\times A\ar[d]^{m}
			\\ A\times A\ar[r]_-{m}
			& A
		}\end{xym}
		\begin{xym}\xymatrix@=40pt{
			A\ar[r]^-{e\times 1}\ar@{=}[dr]\drtwocell\omit\omit{_<-2.5>\rho}
			&A\times A\ar[d]_m
			&A\ar[l]_-{1\times e}\ar@{=}[dl]
				\dltwocell\omit\omit{^<2.5>\,\lambda}
			\\ &A
		}\end{xym}
		\begin{xym}\xymatrix@C=40pt@R=20pt{
			A\times A\ar[dd]_{c_{AA}}\ar[dr]^{m}
			\\ {}\rtwocell\omit\omit{\gamma}
			&A
			\\ A\times A\ar[ur]_-m
		}\end{xym}
		(where $c_{AB}\col A\times B\ra B\times A$ is such that $p_1c\equiv p_2$ and $p_2c\equiv p_1$).
		These data must satisfy the following conditions 
		(the regions without 2-arrow commute because the product is strictly described;
		the 2-arrows $\alpha\times 1$ and $1\times\alpha$ (in the first
		condition) and other 2-arrows constructed in a similar way
		do have the indicated domain and codomain thanks to the same strictness; the symbol
		$\times$ is often omitted for objects, to save space).
				\begin{xyml}\label{axmonintassoc}\begin{gathered}\xymatrix@R=20pt@C=10pt{
					AAAA\ar[rr]^-{m\times 1\times 1}
						\ar[dd]_{1\times 1\times m}\ar[dr]|-{1\times m\times 1}
						\ddrrtwocell\omit\omit{_<-4.5>{\;\;\;\;\,\alpha\times 1}}
						\ddrrtwocell\omit\omit{_<6>{\substack{1\times\alpha
						\;\;\;\;\;\;\;\;\;\;\;\;\;\;\; \\ {} \\ {}}}}
					&&AAA\ar[dr]^-{m\times 1}
					\\ &AAA\ar[rr]_-{m\times 1}\ar[dd]^{1\times m}
						\ddrrtwocell\omit\omit{\alpha}
					&&AA\ar[dd]^m
					\\ AAA\ar[dr]_-{1\times m} &&{}
					\\ &AA\ar[rr]_-m
					&&A
				}\end{gathered} \;\;= \;\;
				\begin{gathered}\xymatrix@R=20pt@C=10pt{
					AAAA\ar[rr]^-{m\times 1\times 1}
						\ar[dd]_{1\times 1\times m}
					&&AAA\ar[dr]^-{m\times 1}\ar[dd]_{1\times m}
					\\ &{}\ddrrtwocell\omit\omit{_<5>{\alpha}}
						\ddrrtwocell\omit\omit{_<-6>{\alpha}}
					&&AA\ar[dd]^m
					\\ AAA\ar[dr]_-{1\times m} \ar[rr]^{m\times 1}
					&&AA\ar[dr]^m
					\\ &AA\ar[rr]_-m
					&&A
				}\end{gathered}\end{xyml}
				\begin{xyml}\label{axmonintunit}\begin{gathered}\xymatrix@=40pt{
					AA\ar[r]^-{1\times e\times 1}\ar@{=}[dr]
						\drtwocell\omit\omit{_<-2.8>{\;\;\;\;1\times \rho}}
					&AAA\ar[r]^-{m\times 1}\ar[d]^(0.3){1\times m}
						\drtwocell\omit\omit{\alpha}
					&AA\ar[d]^{m}
					\\ &AA\ar[r]_-m
					&A
				}\end{gathered}\;\;=\;\;\begin{gathered}\xymatrix@=40pt{
					AA\ar[r]^-{1\times e\times 1}\ar@{=}[dr]
						\drtwocell\omit\omit{_<-2.8>{\;\;\;\;\lambda\times 1}}
					&AAA\ar[d]^{m\times 1}\ar[r]^{m\times 1}
					&AA\ar[d]^m
					\\ &AA\ar[r]^-m  &A
				}\end{gathered}\end{xyml}
				\begin{xyml}\label{axmoninttresi}\begin{gathered}\xymatrix@=26pt{
					AAA\ar[dd]_{1\times m}
					&&AAA\ar[ll]_{c_{A,A\times A}}\ar[dd]^{m\times 1}\ar[dl]_{c_{AA}\times 1}
						\ddtwocell\omit\omit{_<4>\gamma}
					\\ &AAA\ar[ul]_{1\times c_{AA}}\ar[dr]_{m\times 1}\ar[dl]^{1\times m}
						\ddtwocell\omit\omit{\alpha}\dltwocell\omit\omit{_<4>\gamma}
					\\AA\ar[dr]_{m}
					&&AA\ar[dl]^{m}
					\\ &A
				}\end{gathered}\;\;=\;\;	\begin{gathered}\xymatrix@=30pt{
					AAA\ar[dd]_{1\times m}
					&&AAA\ar[ll]_{c_{A,A\times A}}\ar[dd]^{m\times 1}
					\\ {}\drtwocell\omit\omit{_<3.5>\alpha}
					&{}\ar@{}[dl]|{\txt{\normalsize \it AA}}="c"
						\ar@{}[dr]|{\txt{\normalsize \it AA}}="d"
						\ddtwocell\omit\omit{\gamma}\ar"d";"c"_{c_{AA}}
					&{}\dltwocell\omit\omit{_<-3.5>\alpha}
					\\AA\ar[dr]_{m}
					&&AA\ar[dl]^{m}
					\\ &A\ar@{<-}"c"^m\ar@{<-}"d"_m\ar@{<-}"c";[uuul]_{m\times 1}
						\ar@{<-}"d";[uuur]^{1\times m}
				}\end{gathered}\end{xyml}
				\begin{xyml}\label{axmonintsym}\begin{gathered}\xymatrix@=40pt{
					A\times A\ar[dr]^-{m}\ar[d]_c\ar@/_3pc/@{=}[dd]
						\drtwocell\omit\omit{_<3>\gamma}
					\\ A\times A\ar[r]^-m\ar[d]_c\rtwocell\omit\omit{_<2.5>\gamma}
					&A
					\\ A\times A\ar[ur]_-m
				}\end{gathered}\;\;= \;\;1_m\end{xyml}
	\end{df}

	In the case $\C=\Gpd$, these four conditions become respectively
	conditions \ref{axmonassoc}, \ref{axmonunit}, \ref{axmontresi} and 
	\ref{axmonsym}.  The symmetric monoids in $\Gpd$ are thus the
	symmetric 2-monoids defined above.  If $\C$ is a $\Ens$-category
	with finite products, a symmetric monoid in $\C$ is simply a commutative monoid
	in $\C$.
	
	\begin{df}\index{morphism!of symmetric monoids}
		Let $A, B$ be two symmetric monoids in $\C$.  A morphism of symmetric monoids
		$A\ra B$ consists of an arrow $f\col A\ra B$ and two 2-arrows
		\begin{xym}\xymatrix@=40pt{
			A\times A\ar[r]^-{f\times f}\ar[d]_{m}\drtwocell\omit\omit{\varphi}
			&B\times B\ar[d]^m
			\\ A\ar[r]_-f
			&B
		}\;\;\;\;\xymatrix@=40pt{
			1\ar[d]_-e\ar[dr]^-e\drtwocell\omit\omit{_<2.3>{\varphi^0}}
			& 
			\\ A\ar[r]_-f
			&B
		}\end{xym}
		satisfying the following conditions.
				\begin{xyml}\begin{gathered}\xymatrix@R=20pt@C=10pt{
					AAA\ar[rr]^-{f\times f\times f}
						\ar[dd]_{1\times m}\ar[dr]|-{m\times 1}
						\ddrrtwocell\omit\omit{_<-4.5>{\;\;\;\;\,\varphi\times f}}
						\ddrrtwocell\omit\omit{_<6>{\substack{\alpha
						\;\;\;\;\;\;\;\;\; \\ {} \\ {}}}}
					&&BBB\ar[dr]^-{m\times 1}
					\\ &AA\ar[rr]_-{f\times f}\ar[dd]^{m}
						\ddrrtwocell\omit\omit{\varphi}
					&&BB\ar[dd]^m
					\\ AA\ar[dr]_-{m} &&{}
					\\ &A\ar[rr]_-f
					&&B
				}\end{gathered}\;\;=\;\;
				\begin{gathered}\xymatrix@R=20pt@C=10pt{
					AAA\ar[rr]^-{f\times f\times f}\ar[dd]_{1\times m}
						\ddrrtwocell\omit\omit{\substack{f\times\varphi
						\;\;\;\;\;\;\;\;\;\;\;\;\;\; \\ {}\\ {}}}
					&&BBB\ar[dr]^-{m\times 1}\ar[dd]_{1\times m}
					\\ &{}\ddrrtwocell\omit\omit{_<5>{\varphi}}
						\ddrrtwocell\omit\omit{_<-6>{\alpha}}
					&&BB\ar[dd]^m
					\\ AA\ar[dr]_-{m}\ar[rr]^{f\times f}
					&&BB\ar[dr]^m
					\\ &A\ar[rr]_-f
					&&B
				}\end{gathered}\end{xyml}
				\stepcounter{eqnum}\begin{subequations}
				\begin{equation}\begin{gathered}\xymatrix@R=20pt@C=10pt{
					A\ar[rr]^-f\ar[dr]^-{1\times e}\ar@{=}@/_1.1pc/[dddr]
					&{}\drtwocell\omit\omit{\;\;\;\;\;\;f\times \varphi^0}
					&B\ar[dr]^-{1\times e}
					\\ {}\drtwocell\omit\omit{\lambda}
					&A\times A\ar[rr]_-{f\times f}\ar[dd]^{m}
						\ddrrtwocell\omit\omit{\varphi}
					&{}&B\times B\ar[dd]^m
					\\ &{}
					\\ &A\ar[rr]_-f
					&&B
				}\end{gathered}\;\;=\;\;\begin{gathered}\xymatrix@R=20pt@C=10pt{
					A\ar[rr]^-f
					&&B\ar[dr]^-{1\times e}\ar@{=}@/_1.1pc/[dddr]
					\\ &&{}\drtwocell\omit\omit{\lambda}
					&B\times B\ar[dd]^m
					\\ {}&&&{}
					\\ &&&B
				}\end{gathered}\end{equation}
				\begin{equation}\begin{gathered}\xymatrix@R=20pt@C=10pt{
					A\ar[rr]^-f\ar[dr]^-{e\times 1}\ar@{=}@/_1.1pc/[dddr]
					&{}\drtwocell\omit\omit{\;\;\;\;\;\;\varphi^0\times f}
					&B\ar[dr]^-{e\times 1}
					\\ {}\drtwocell\omit\omit{\rho}
					&A\times A\ar[rr]_-{f\times f}\ar[dd]^{m}
						\ddrrtwocell\omit\omit{\varphi}
					&{}&B\times B\ar[dd]^m
					\\ &{}
					\\ &A\ar[rr]_-f
					&&B
				}\end{gathered}\;\;=\;\;\begin{gathered}\xymatrix@R=20pt@C=10pt{
					A\ar[rr]^-f
					&&B\ar[dr]^-{e\times 1}\ar@{=}@/_1.1pc/[dddr]
					\\ &&{}\drtwocell\omit\omit{\rho}
					&B\times B\ar[dd]^m
					\\ {}&&&{}
					\\ &&&B
				}\end{gathered}\end{equation}\end{subequations}
				\ignorespacesafterend
				\begin{xyml}\begin{gathered}\xymatrix@=40pt{
					A\times A\ar[r]^{f\times f}\ar[d]_c
					&B\times B\ar[d]_c\dduppertwocell^m<13>{_<-4.7>\gamma}
					\\ A\times A\ar[r]^{f\times f}\ar[d]_m\drtwocell\omit\omit{\varphi}
					&B\times B\ar[d]_m
					\\ A\ar[r]_f
					&B
				}\end{gathered} =\;\; \begin{gathered}\xymatrix@=40pt{
					A\times A\ar[r]^{f\times f}\ar[d]_c\dduppertwocell^m<13>{_<-4.7>\gamma}
					&B\times B\dduppertwocell^m<13>{\omit}\ddtwocell\omit\omit{_<2>\varphi}
					\\ A\times A\ar[d]_m
					\\ A\ar[r]_f
					&B
				}\end{gathered}\end{xyml}
	\end{df}
	
	In the case $\C=\Gpd$, these conditions become respectively
	conditions \ref{axfoncmonass}, \ref{axfoncmonunite} and \ref{axfonmonsym}.
	The morphisms of symmetric monoids in $\Gpd$ are thus the
	symmetric monoidal functors.
	
	\begin{df}
		Let $f,g\col A\ra B$ in $\C$ be morphisms of symmetric monoids.
		A 2-morphism between these morphisms consists of a 2-arrow $\alpha\col f\Ra g$
		such that the following conditions hold.
		\begin{xyml}\begin{gathered}\xymatrix@=40pt{
			A\times A\ar[r]_-{g\times g}\ar[d]_{m}\drtwocell\omit\omit{\varphi}
				\ruppertwocell^{f\times f}{\;\;\;\;\;\;\alpha\times\alpha}
			&B\times B\ar[d]^m
			\\ A\ar[r]_-g
			&B
		}\end{gathered}\;\;=\;\;\begin{gathered}\xymatrix@=40pt{
			A\times A\ar[r]^-{f\times f}\ar[d]_{m}\drtwocell\omit\omit{\varphi}
			&B\times B\ar[d]^m
			\\ A\ar[r]^-f\rlowertwocell_g{\alpha}
			&B
		}\end{gathered}\end{xyml}
		\begin{xyml}\begin{gathered}\xymatrix@=40pt{
			1\ar[d]_-e\ar[dr]^-e\drtwocell\omit\omit{_<2.3>{\varphi^0}}
			& 
			\\ A\ar[r]^-f\rlowertwocell_g{\alpha}
			&B
		}\end{gathered}\;\;=\;\;\begin{gathered}\xymatrix@=40pt{
			1\ar[d]_-e\ar[dr]^-e\drtwocell\omit\omit{_<2.3>{\varphi^0}}
			& 
			\\ A\ar[r]_-g
			&B
		}\end{gathered}\end{xyml}
	\end{df}

	In the case $\C=\Gpd$, these conditions become conditions \ref{axtnmonass}
	and \ref{axtnmonunit}; the
	2-morphisms of symmetric monoids in $\Gpd$ are thus the monoidal natural transformations.
	
\subsection{Structure of bimonoid induced by biproducts}

	We will now prove that in a $\Gpd$-category with biproducts
	the diagonal and the codiagonal give on each object a structure
	of symmetric bimonoid, i.e.\ a structure of symmetric monoid and a structure 
	of symmetric comonoid such that the multiplication and the unit are morphisms
	of comonoids, and the comultiplication and the counit are morphisms of monoids.
	(To really define bimonoids, we should also ask that the 2-arrows
	of associativity and coassociativity, unit and counit, symmetry and cosymmetry be
	2-morphisms of monoid or of comonoid, but we won't need these properties in the following.)
	
	\begin{pon}
		Let be an object $A$ in a $\Gpd$-category with biproducts (we always assume that
		the product is strictly described).  The diagonal $A\overset{\Delta}\longrightarrow A\oplus A$
		and the arrow $A\overset{0}\ra 0$ define a strictly described comonoid structure
		on $A$.
	\end{pon}
	
		\begin{proof}
			The strictness of the product implies that we can take the identity
			for $\alpha$, $\lambda$, $\rho$ and $\gamma$.
		\end{proof}
		
	\begin{pon}\label{codiagmon}
		Let be $A$ in a $\Gpd$-category with biproducts.
		The codiagonal $A\oplus A\overset{\nabla}\ra A$
		and the arrow $0\overset{0}\ra A$ determine a monoid structure on $A$.
	\end{pon}
	
		\begin{proof}
			The codiagonal comes from the universal property of the coproduct and comes with
			two 2-arrows $\delta_1\col \nabla i_1\Ra 1_A$ and
			$\delta_2\col \nabla i_2\Ra 1_A$.
			
			Let us define the coherence 2-arrows.  The 2-arrow
			$\alpha_A\col \nabla(\nabla\oplus A)\Ra \nabla(A\oplus\nabla)$ is defined 
			by the universal property of the coproduct, as the unique 2-arrow such that
			\begin{xyml}\alpha_A i_1 =\begin{gathered}\xymatrix@=40pt{
				A\ar[r]^-{i_1}
				&{A\oplus A}\ar[r]^-{i_1\oplus A}\ar@{=}[dr]
					\drtwocell\omit\omit{_<-3>{\;\;\;\;\;\;\,\delta_1\oplus A}}
				&{A\oplus A\oplus A}\ar[d]^{\nabla\oplus A}
				\\ &&{A\oplus A}\ar[d]^\nabla
				\\ &&A
			}\end{gathered}\end{xyml}
			\begin{xyml}\alpha_A i_2 =\begin{gathered}\xymatrix@=40pt{
				A\ar[r]^-{i_1}\ar@{=}[ddrr]\ar[d]_{i_2}
					\ddrrtwocell\omit\omit{_<-4.5>{\;\delta_1}}
					\ddrrtwocell\omit\omit{_<4.5>{\;\;\;\,\delta_2^{-1}}}
				&{A\oplus A}\ar[r]^-{i_2\oplus A}\ar@{=}[dr]
					\drtwocell\omit\omit{_<-3>{\;\;\;\;\;\;\,\delta_2\oplus A}}
				&{A\oplus A\oplus A}\ar[d]^{\nabla\oplus A}
				\\ {A\oplus A}\ar@{=}[dr]\ar[d]_{A\oplus i_1}
					\drtwocell\omit\omit{_<3.5>{\;\;\;\;\;\;\;\;A\oplus\delta_1^{-1}}}
				&&{A\oplus A}\ar[d]^\nabla
				\\ {A\oplus A\oplus A}\ar[r]_-{A\oplus\nabla}
				&{A\oplus A}\ar[r]_-{\nabla}
				&A
			}\end{gathered}\end{xyml}
			\begin{xyml}\text{and }\alpha_A i_3 =\begin{gathered}\xymatrix@=40pt{
				A\ar[d]_{i_2}
				\\ {A\oplus A}\ar@{=}[dr]\ar[d]_{A\oplus i_2}
					\drtwocell\omit\omit{_<3.5>{\;\;\;\;\;\;\;\;A\oplus\delta_2^{-1}}}
				\\ {A\oplus A\oplus A}\ar[r]_-{A\oplus\nabla}
				&{A\oplus A}\ar[r]_-{\nabla}
				&A
			}\end{gathered}\end{xyml}
			We set $\lambda_A \eqdef \delta_2$ and $\rho_A\eqdef\delta_1$.
			Finally, $\gamma_A\col \nabla
			\Ra\nabla c_{A,A}$ is uniquely determined by the conditions
			\begin{xyml}\gamma_A i_1 = \begin{gathered}\xymatrix@=40pt{
				A\ar[r]^-{i_1}\ar@{=}[dr]\ar[d]_{i_2}\drtwocell\omit\omit{_<-3>{\;\delta_1}}
					\drtwocell\omit\omit{_<3>{\;\;\;\,\delta_2^{-1}}}
				&{A\oplus A}\ar[d]^\nabla
				\\{A\oplus A}\ar[r]_-\nabla &A
			}\end{gathered}\end{xyml}
			\begin{xyml}\text{and }\gamma_A i_2 = \begin{gathered}\xymatrix@=40pt{
				A\ar[r]^-{i_2}\ar@{=}[dr]\ar[d]_{i_1}\drtwocell\omit\omit{_<-3>{\;\delta_2}}
					\drtwocell\omit\omit{_<3>{\;\;\;\,\delta_1^{-1}}}
				&{A\oplus A}\ar[d]^\nabla
				\\{A\oplus A}\ar[r]_-\nabla &A
			}\end{gathered}\end{xyml}
			We check the axioms by testing them with the inclusions of the biproduct.
		\end{proof}
	
		\begin{pon}
			Let $A$ be an object in a $\Gpd$-category with biproducts.  Then
			the diagonal, the codiagonal and the zero arrows from or to $0$ are
			the basis of a “bimonoid” structure on $A$.
		\end{pon}
		
			\begin{proof}
				The previous propositions give the monoid and the comonoid structures. It
				remains to check that $\nabla$ and $0^A$ are morphisms of
				comonoids and that $\Delta$ and $0_A$ are morphisms of monoids.
				The object $A\oplus A$ is naturally equipped with a comonoid structure
				induced by $\Delta$: it consists of the zero arrow and of
				\begin{eqn}
					A\oplus A\xrightarrow{\Delta\oplus\Delta} A\oplus A\oplus A\oplus A
					\xrightarrow{A\oplus c_{AA}\oplus A} A\oplus A\oplus A\oplus A.
				\end{eqn}
				Dually, $A\oplus A$ is naturally equipped with a monoid structure
				induced by $\nabla$:
				\begin{eqn}\label{monAplusA}
					A\oplus A\oplus A\oplus A\xrightarrow{A\oplus c_{AA}\oplus A}
					 A\oplus A\oplus A\oplus A\xrightarrow{\nabla\oplus\nabla} A\oplus A.
				\end{eqn}
				The 2-arrow which expresses both the fact that $\Delta$ and that $\nabla$
				are morphisms is then the identity (thanks to the strictness of the product):
				\begin{xym}\label{bimoneq}\xymatrix@R=40pt@C=25pt{
					{A\oplus A}\ar[dd]_\nabla\ar[r]^-{\Delta\oplus\Delta}
					&{A\oplus A\oplus A\oplus A}\ar[dr]^-{A\oplus c_{AA}\oplus A}
					\\&&{A\oplus A\oplus A\oplus A}\ar[d]^{\nabla\oplus\nabla}
					\\A\ar[rr]_-\Delta &&{A\oplus A}
				}\end{xym}
				It is easy to check the axioms of morphism of monoids for
				$\Delta$ (we use the fact that, for a 2-arrow $\mu$, 
				$\Delta\mu=(\mu\oplus\mu)\Delta$)
				and of comonoids for $\nabla$ (only identities are involved).
			\end{proof}
			
		Moreover, each arrow of $\C$ has a structure of morphism of bimonoids
		between the bimonoids so defined, and each 2-arrow is a 2-morphism
		between these morphisms.
		
		\begin{pon}\label{tteflechmorph}
			Let be $f\col A\ra B$ in a $\Gpd$-category with biproducts.  We can
			equipped $f$ with a structure of morphism of bimonoids
			from $(A,\Delta_A,\allowbreak\nabla_A,\allowbreak 0^A,\allowbreak 0_A,\ldots)$ to
			$(B,\Delta_B,\nabla_B,0^B,0_B,\ldots)$.
		\end{pon}
		
			\begin{proof}
				On the one hand, by the strictness of the product, we have
				$(f\oplus f)\Delta = \Delta f$. On the other hand, the universal property of 
				the coproduct implies the existence of a unique
				$\varphi^f\col \nabla(f\oplus f)\Ra f\nabla$ such that, for $k=1$ or $2$,
				\begin{xyml}\varphi^f i_k =\begin{gathered}\xymatrix@=40pt{
					A\ar[r]^f\ar@{=}[dr]\ar[d]_{i_k}
						\drtwocell\omit\omit{_<3>{\;\;\;\,\delta_k^{-1}}}
					&B\ar@{=}[dr]\ar[r]^-{i_k}\drtwocell\omit\omit{_<-3>{\;\delta_k}}
					&{B\oplus B}\ar[d]^{\nabla_B}
					\\{A\oplus A}\ar[r]_-{\nabla_A}
					&A\ar[r]_f &B
				}\end{gathered}\end{xyml}
				We set $\varphi^0 \eqdef 1_0$.
				We check the axioms by testing them with the inclusions of the biproduct.
			\end{proof}
			
		\begin{pon}
			Let be $\mu\col f\Ra g\col A\ra B$ in a $\Gpd$-category with biproducts.
			Then $\mu$ is a 2-morphism between the structures of morphisms of
			bimonoids defined on $f$ and $g$ in the previous proposition.
		\end{pon}
		
			\begin{proof}
				Since $(\mu\oplus\mu)\Delta=\Delta\mu$, $\mu$ is a
				2-morphism of comonoids.  To prove that
				it is a 2-morphism of monoids, it suffices to test the axioms with $i_1$
				and $i_2$.
			\end{proof}
			
		The following propositions mainly show that
		the constructions of the previous propositions define a $\Gpd$-functor
		preserving biproducts
		\begin{eqn}
			\Phi \col  \C\ra\caspar{Bimon}(\C).
		\end{eqn}
		Some details needed to actually establish that $\Phi$ is a $\Gpd$-functor
		preserving biproducts are missing, but they won't be needed in the following.
		
		\begin{pon}\label{PhiGpdfonc}
			The following conditions hold ($\Phi$ is a “$\Gpd$-functor”).
			\begin{align}\stepcounter{eqnum}\begin{gathered}
			\xymatrix@=40pt{
				{A\oplus A}\ar[r]^{f\oplus f}\ar[d]_{\nabla_A}
					\drtwocell\omit\omit{\;\varphi^f}
				&{B\oplus B}\ar[r]^{g\oplus g}\ar[d]^{\nabla_B}
					\drtwocell\omit\omit{\;\,\varphi^g}
				&{C\oplus C}\ar[d]^{\nabla_C}
				\\ A\ar[r]_f &B\ar[r]_g &C
			}\end{gathered} &=\begin{gathered}\xymatrix@=40pt{
				{A\oplus A}\ar[r]^{gf\oplus gf}\ar[d]_{\nabla_A}
					\drtwocell\omit\omit{\;\;\,\varphi^{gf}}
				&{C\oplus C}\ar[d]^{\nabla_C}
				\\ A\ar[r]_{gf} &C
			}\end{gathered}\\
			\stepcounter{eqnum} 1_{\nabla_A} &=\begin{gathered}\xymatrix@=40pt{
				{A\oplus A}\ar@{=}[r]\ar[d]_{\nabla_A}
					\drtwocell\omit\omit{\;\;\,\varphi^{1_A}}
				&{A\oplus A}\ar[d]^{\nabla_A}
				\\ A\ar@{=}[r] &A
			}\end{gathered}\end{align}
		\end{pon}
		
			\begin{proof}
				It suffices to test these equations with $i_1$ and $i_2$.
			\end{proof}
			
		\begin{pon}\label{Phipresbiprod}
			There exists $\omega\col \nabla_{A\oplus A}\Ra
			(\nabla_A\oplus\nabla_A)(A\oplus c_{AA}\oplus A)$ ($\Phi$ “preserves the biproduct”)
			such that the following conditions hold (“$\Phi(f_1\oplus f_2)=
			\Phi(f_1)\oplus \Phi(f_2)$” and “$\Phi(0)=0$”).
			\begin{multline}\stepcounter{eqnum}\begin{gathered}
			\xymatrix@R=40pt@C=50pt{
				{A\oplus A\oplus A\oplus A}\ar[r]^-{f_1\oplus f_2\oplus f_1\oplus f_2}
					\ar[d]_{\nabla_{A\oplus A}}
					\drtwocell\omit\omit{\;\;\;\;\;\;\;\;\,\varphi^{f_1\oplus f_2}}
				&{B\oplus B\oplus B\oplus B}\ar[d]^{\nabla_{B\oplus B}}
				\\ {A\oplus A}\ar[r]_-{f_1\oplus f_2} &{B\oplus B}
			}\end{gathered}\\ =\begin{gathered}\xymatrix@R=40pt@C=50pt{
				{A\oplus A\oplus A\oplus A}\ar[r]^-{f_1\oplus f_2\oplus f_1\oplus f_2}
					\ar[d]^{A\oplus c\oplus A}
					\ar@/_43pt/[dd]_{\nabla_{A\oplus A}}
				&{B\oplus B\oplus B\oplus B}\ar[d]_{B\oplus c\oplus B}
					\ar@/^43pt/[dd]^{\nabla_{B\oplus B}}
				\\ {A\oplus A\oplus A\oplus A}\ar[r]^{f_1\oplus f_1\oplus f_2\oplus f_2}
					\ar[d]^{\nabla_A\oplus\nabla_A}
					\drtwocell\omit\omit
					{\;\;\;\;\;\;\;\;\;\;\;\,\varphi^{f_1}\oplus\varphi^{f_2}}
					\dtwocell\omit\omit{_<3>{\;\,\omega^{-1}}}
				&{B\oplus B\oplus B\oplus B}\ar[d]_(0.3){\nabla_B\oplus \nabla_B}
					\dtwocell\omit\omit{_<-3>\omega}
				\\ {A\oplus A}\ar[r]_{f_1\oplus f_2} &{B\oplus B}
			}\end{gathered}\end{multline}
			\begin{eqn}\label{phizero}
				\varphi^0 = 1_0 \col  \nabla(0\oplus 0)\Ra 0\nabla
			\end{eqn}
		\end{pon}
		
			\begin{proof}
				By the universal property of the coproduct, we have a unique $\omega$
				such that the following conditions hold.
				\begin{xym}\omega i_1=\begin{gathered}\xymatrix@=40pt{
					A\ar[r]^-{i_1}
					&{A\oplus A}\ar[r]^-{i_{12}}\ar@{=}[dr]\ar[d]_{i_1\oplus A}
						\drtwocell\omit\omit{_<-3>{\;\delta_1}}
						\drtwocell\omit\omit{_<3>{\;\;\;\;\;\;\;\;\;\delta_1^{-1}\oplus A}}
					&{A\oplus A\oplus A\oplus A}\ar[d]^{\nabla_{A\oplus A}}
					\\ &{A\oplus A\oplus A}\ar[r]_-{\nabla_A\oplus A}
					&{A\oplus A}
				}\end{gathered}\end{xym}
				\begin{xym}\omega i_2=\begin{gathered}\xymatrix@=40pt{
					A\ar[r]^-{i_2}
					&{A\oplus A}\ar[r]^-{i_{12}}\ar@{=}[dr]\ar[d]_{A\oplus i_1}
						\drtwocell\omit\omit{_<-3>{\;\delta_1}}
						\drtwocell\omit\omit{_<3>{\;\;\;\;\;\;\;\;\,A\oplus\delta_1^{-1}}}
					&{A\oplus A\oplus A\oplus A}\ar[d]^{\nabla_{A\oplus A}}
					\\ &{A\oplus A\oplus A}\ar[r]_-{A\oplus\nabla_A}
					&{A\oplus A}
				}\end{gathered}\end{xym}
				\begin{xym}\omega i_3=\begin{gathered}\xymatrix@=40pt{
					A\ar[r]^-{i_1}
					&{A\oplus A}\ar[r]^-{i_{34}}\ar@{=}[dr]\ar[d]_{i_2\oplus A}
						\drtwocell\omit\omit{_<-3>{\;\delta_2}}
						\drtwocell\omit\omit{_<3>{\;\;\;\;\;\;\;\;\;\delta_2^{-1}\oplus A}}
					&{A\oplus A\oplus A\oplus A}\ar[d]^{\nabla_{A\oplus A}}
					\\ &{A\oplus A\oplus A}\ar[r]_-{\nabla_A\oplus A}
					&{A\oplus A}
				}\end{gathered}\end{xym}
				\begin{xym}\omega i_4=\begin{gathered}\xymatrix@=40pt{
					A\ar[r]^-{i_2}
					&{A\oplus A}\ar[r]^-{i_{34}}\ar@{=}[dr]\ar[d]_{A\oplus i_2}
						\drtwocell\omit\omit{_<-3>{\;\delta_2}}
						\drtwocell\omit\omit{_<3>{\;\;\;\;\;\;\;\;\,A\oplus\delta_2^{-1}}}
					&{A\oplus A\oplus A\oplus A}\ar[d]^{\nabla_{A\oplus A}}
					\\ &{A\oplus A\oplus A}\ar[r]_-{A\oplus\nabla_A}
					&{A\oplus A}
				}\end{gathered}\end{xym}
				We check the conditions by testing them with $i_1$, $i_2$, $i_3$ and $i_4$.
			\end{proof}

	Since $(A,\nabla,\ldots)$ is a symmetric monoid, $\nabla$ is a morphism of
	monoids from $A\oplus A$ (with multiplication given by the composite
	\ref{monAplusA}) to $A$.  The composites of the two following diagrams are equal
	and define the 2-arrow $\bar{\gamma}_A$ expressing that $\nabla$ is a morphism
	(in the case $\C=\Gpd$, we recover composite \ref{assocsym}).
	
	\begin{xym}\xymatrix@=40pt{
		{A\oplus A\oplus A\oplus A}\ar[r]^-{A\oplus A\oplus\nabla}\ar[dd]_{A\oplus c\oplus A}
			\ar[dr]^-{A\oplus\nabla\oplus A}
			\drtwocell\omit\omit{_<6.5>{\;\;\;\;\;\;\;\;\;\;\;\;{A\oplus\gamma_A\oplus A}}}
			\drrtwocell\omit\omit{\;\;\;\;\;\;\;\;\;A\oplus\alpha_A^{-1}}
		&{A\oplus A\oplus A}\ar[r]^-{\nabla\oplus A}\ar[dr]^{A\oplus\nabla}
			\drrtwocell\omit\omit{\;\;\;\alpha_A}
		&{A\oplus A}\ar[dr]^\nabla
		\\ &{A\oplus A\oplus A}\ar[r]_-{A\oplus\nabla}
		&{A\oplus A}\ar[r]^-\nabla &A
		\\{A\oplus A\oplus A\oplus A}\ar[ur]_-{A\oplus\nabla\oplus A}
			\ar[r]_-{A\oplus A\oplus\nabla}
			\urrtwocell\omit\omit{\;\;\;\;\;\;\;\;\;A\oplus\alpha_A}
		&{A\oplus A\oplus A}\ar[r]_-{\nabla\oplus A}\ar[ur]_{A\oplus\nabla}
			\urrtwocell\omit\omit{\;\;\;\;\;\alpha_A^{-1}}
		&{A\oplus A}\ar[ur]_{\nabla}
	}\end{xym}
	\begin{xym}\xymatrix@=40pt{
		{A\oplus A\oplus A\oplus A}\ar[r]^-{\nabla\oplus A\oplus A}\ar[dd]_{A\oplus c\oplus A}
			\ar[dr]^-{A\oplus\nabla\oplus A}
			\drtwocell\omit\omit{_<6.5>{\;\;\;\;\;\;\;\;\;\;\;\;{A\oplus\gamma_A\oplus A}}}
			\drrtwocell\omit\omit{\;\;\;\;\;\;\;\;\alpha_A\oplus A}
		&{A\oplus A\oplus A}\ar[r]^-{A\oplus\nabla}\ar[dr]^{\nabla\oplus A}
			\drrtwocell\omit\omit{\;\;\;\;\,\alpha_A^{-1}}
		&{A\oplus A}\ar[dr]^\nabla
		\\ &{A\oplus A\oplus A}\ar[r]_-{\nabla\oplus A}
		&{A\oplus A}\ar[r]^-\nabla &A
		\\{A\oplus A\oplus A\oplus A}\ar[ur]_-{A\oplus\nabla\oplus A}
			\ar[r]_-{\nabla\oplus A\oplus A}
			\urrtwocell\omit\omit{\;\;\;\;\;\;\;\;\;\;\,\alpha_A^{-1}\oplus A}
		&{A\oplus A\oplus A}\ar[r]_-{A\oplus\nabla}\ar[ur]_{\nabla\oplus A}
			\urrtwocell\omit\omit{\;\;\;\alpha_A}
		&{A\oplus A}\ar[ur]_{\nabla}
	}\end{xym}
	
	In the same way, since $(A,\nabla,\Delta,\ldots)$ is a bimonoid, $\Delta$ is equipped with a
	structure of morphism of monoids, given by diagram \ref{bimoneq}.
	
	But, by Proposition \ref{tteflechmorph}, $\nabla$ and $\Delta$ are morphisms
	of monoids between $(A\oplus A,\nabla_{A\oplus A},\ldots)$ and $(A, \nabla_A,\ldots)$.
	The following proposition shows that, both for $\Delta$ and for $\nabla$,
	these two structures of morphisms of monoids coincide modulo $\omega$.
	
	\begin{pon}\label{Phipresdelnab}
		The following equations hold.
		\begin{xyml}\begin{gathered}\xymatrix@=40pt{
			{\txt{$A\oplus A$\\ ${}\oplus A\oplus A$}}\ar[r]^-{\nabla_A\oplus\nabla_A}
				\ar[d]^{A\oplus c\oplus A}\ddrtwocell\omit\omit{_<-1.5>{\;\;\,\bar{\gamma}_A}}
				\ar@/_43pt/[dd]_{\nabla_{A\oplus A}}\ddtwocell\omit\omit{_<6>{\omega^{-1}}}
			&{A\oplus A}\ar[dd]^{\nabla_A}
			\\ {\txt{$A\oplus A$\\ ${}\oplus A\oplus A$}}\ar[d]^{\nabla_A\oplus\nabla_A}
			\\ {A\oplus A}\ar[r]_-{\nabla_A} &A
		}\end{gathered}=\begin{gathered}\xymatrix@=40pt{
			{\txt{$A\oplus A$\\ ${}\oplus A\oplus A$}}\ar[r]^-{\nabla_A\oplus\nabla_A}
				\drtwocell\omit\omit{\;\;\;\;\varphi^{\nabla}}
				\ar[d]_{\nabla_{A\oplus A}}
			&{A\oplus A}\ar[d]^{\nabla_A}
			\\ {A\oplus A}\ar[r]_-{\nabla_A} &A
		}\end{gathered}\end{xyml}
		\begin{xyml}\begin{gathered}\xymatrix@=40pt{
				{A\oplus A}\ar[r]^-{\Delta_A\oplus\Delta_A}\ar[dd]_{\nabla_A}
				&{\txt{$A\oplus A$\\ ${}\oplus A\oplus A$}}\ar[d]_{A\oplus c\oplus A}
					\ar@/^43pt/[dd]^(0.3){\nabla_{A\oplus A}}\ddtwocell\omit\omit{_<-6>\omega}
				\\ &{\txt{$A\oplus A$\\ ${}\oplus A\oplus A$}}\ar[d]_{\nabla_A\oplus \nabla_A}
				\\ A\ar[r]_-{\Delta_A}
				&{A\oplus A}
		}\end{gathered}=\begin{gathered}\xymatrix@=40pt{
			{A\oplus A}\ar[r]^-{\Delta_A\oplus\Delta_A}
				\drtwocell\omit\omit{\;\;\;\;\varphi^{\Delta}}
				\ar[d]_{\nabla_{A\oplus A}}
			&{\txt{$A\oplus A$\\ ${}\oplus A\oplus A$}}\ar[d]^{\nabla_A}
			\\ A\ar[r]_-{\Delta_A} &{A\oplus A}
		}\end{gathered}\end{xyml}
	\end{pon}
	
		\begin{proof}
			It suffices to test these conditions with the inclusions of the biproduct.
		\end{proof}

\subsection{The existence of finite biproducts implies semiadditivity}

	The structures of comonoid on $A$ and of monoid on $B$ induce a monoid structure
	on $\C(A,B)$, which turns it into a symmetric 2-monoid
	(see Day and Street \cite{Day1997a}).
	We apply this principle to the constructions of the previous subsection to
	get the following proposition.

	\begin{pon}
		If $\C$ is a $\Gpd$-category with all finite biproducts, then,
		for $A,B\col \C$, we define a structure of symmetric 2-monoid on $\C(A,B)$
		in the following way:
		\begin{enumerate}
			\item $0\col A\ra B$ is the zero arrow;
			\item if $f,g\col A\ra B$, $f+g\eqdef$
				\begin{eqn}
					A\overset{\Delta}\longrightarrow A\oplus A\xrightarrow{f\oplus g}B\oplus B
					\overset{\nabla}\longrightarrow B,
				\end{eqn}
				and if we have $\alpha\col f\Ra f'\col A\ra B$ and $\beta\col g\Ra g'\col A\ra B$,
				$\alpha+\beta\eqdef\nabla(\alpha\oplus\beta)\Delta$;
			\item if $f,g,h\col A\ra B$, $\alpha_{fgh}$ is the composite of the following diagram;
		\end{enumerate}
				\begin{xym}\xymatrix@R=30pt@C=25pt{
					&{A\oplus A}\ar[rrr]^{(f+g)\oplus h}\ar[dr]^{\Delta\oplus A}
					&&&{B\oplus B}\ar[dr]^-\nabla
					\\ A\ar[ur]^-\Delta\ar[dr]_-\Delta
					&&{A\oplus A\oplus A}\ar[r]^{f\oplus g\oplus h}
					&{B\oplus B\oplus B}\ar[ur]^-{\nabla\oplus B}\ar[dr]_-{B\oplus\nabla}
						\rrtwocell\omit\omit{\;\;\,\alpha_B}
					&&B
					\\ &{A\oplus A}\ar[rrr]_{f\oplus(g+h)}\ar[ur]_-{A\oplus\Delta}
					&&&{B\oplus B}\ar[ur]_-\nabla
				}\end{xym}
		\begin{enumerate}
			\item[4.] if $f\col A\ra B$, $\lambda_f$ is the composite of the following diagram;
				\begin{xym}\xymatrix@C=40pt@R=20pt{
					&{A\oplus A}\ar[r]^{0\oplus f}\ar[dd]^{p_2}
					&{B\oplus B}\ar[dr]^-\nabla
					\\ A\ar@{=}[dr]\ar[ur]^-\Delta
					&&&B
					\\ &A\ar[r]_f
					&B\ar[uu]^{i_2}\ar@{=}[ur]\urtwocell\omit\omit{_<-4>{\;\;\;\lambda_B}}
				}\end{xym}
			\item[5.] if $f\col A\ra B$, $\rho_f$ is the composite of the following diagram;
				\begin{xym}\xymatrix@C=40pt@R=20pt{
					&{A\oplus A}\ar[r]^{f\oplus 0}\ar[dd]^{p_1}
					&{B\oplus B}\ar[dr]^-\nabla
					\\ A\ar@{=}[dr]\ar[ur]^-\Delta
					&&&B
					\\ &A\ar[r]_f
					&B\ar[uu]^{i_1}\ar@{=}[ur]\urtwocell\omit\omit{_<-4>{\;\;\;\rho_B}}
				}\end{xym}
			\item[6.] if $f,g\col A\ra B$, $\gamma_{fg}$ is the composite of the following diagram.
				\begin{xym}\xymatrix@C=40pt@R=20pt{
					&{A\oplus A}\ar[r]^{f\oplus g}\ar[dd]^{c_{AA}}
					&{B\oplus B}\ar[dr]^-\nabla\ar[dd]_{c_{BB}}
						\drtwocell\omit\omit{_<4>{\;\;\gamma_B}}
					\\ A\ar[dr]_-{\Delta}\ar[ur]^-\Delta
					&&&B
					\\ &{A\oplus A}\ar[r]_{g\oplus f}
					&{B\oplus B}\ar[ur]_-\nabla
				}\end{xym}
		\end{enumerate}
	\end{pon}
	
		\begin{proof}
			The axioms of symmetric 2-monoid of $\C(A,B)$ follows automatically
			from the axioms of internal symmetric monoid of $B$ with multiplication
			$\nabla$ (Proposition \ref{codiagmon}).
		\end{proof}
		
	\begin{pon}
		If $\C$ is a $\Gpd$-category with finite biproducts, $\C$ is presemiadditive,
		with the structure on the $\Hom$s described in the previous proposition, and
		the distributivity 2-arrows defined in the following way:
		\begin{enumerate}
			\item if $f_1,f_2\col A\ra B$ and $g\col B\ra C$, $\varphi^g_{f_1,f_2}$
				is defined by the following composite;
				\begin{xym}\xymatrix@=40pt{
					A\ar[r]^-\Delta
					&{A\oplus A}\ar[r]^{gf_1\oplus gf_2}\ar[dr]_-{f_1\oplus f_2}
					&{C\oplus C}\ar[r]^-\nabla
					&C
					\\ &&{B\oplus B}\ar[u]^-{g\oplus g}\ar[r]_-\nabla
						\urtwocell\omit\omit{\;\;\;\varphi^g}
					&B\ar[u]_g
				}\end{xym}
			\item if $f\col A\ra B$ and $g_1,g_2\col B\ra C$, $\psi^{g_1,g_2}_f$ is defined
				by the following composite (of identities);
				\begin{xym}\xymatrix@=40pt{
					A\ar[r]^-\Delta\ar[d]_f
					&{A\oplus A}\ar[r]^{g_1f\oplus g_2f}\ar[d]_{f\oplus f}
					&{C\oplus C}\ar[r]^-\nabla
					&C
					\\ B\ar[r]_-\Delta
					&{B\oplus B}\ar[ur]_-{g_1\oplus g_2}
				}\end{xym}
			\item $\varphi^h_0\col 0\Ra h0$ and $\psi_g^0\col 0\Ra 0g$ are the identity
				(we have assumed that the $\Gpdp$-category is strictly described).
		\end{enumerate}
	\end{pon}
	
		\begin{proof}
			We follow the numbering of the conditions of point 2 of Definition \ref{semipreadd}
			(and thus of their elementary translation, which follows the proposition).
			\renewcommand{\theenumi}{\alph{enumi}}\renewcommand{\labelenumi}{(\theenumi)}
			\begin{enumerate}
				\item These conditions follow automatically from the fact that
					$h$, with $\varphi^h$, is a morphism of internal symmetric monoids
					(Proposition \ref{tteflechmorph}).
				\item It suffices to write the two sides of these conditions 
					to realize that they are equal.
				\item For the first condition (diagram \ref{preaddc}), we 
					have the following succession of equalities:
					$\varphi^{h_1+h_2}_{g_1,g_2}
						\circ(\psi_{g_1}^{h_1,h_2}+\psi_{g_2}^{h_1,h_2})$
					is equal to diagram \ref{diagun}.  By Proposition
					\ref{PhiGpdfonc}, this diagram is equal to the composite of
					diagram \ref{diagtwo}. By Propositions \ref{Phipresbiprod}
					and \ref{Phipresdelnab}, diagram \ref{diagtwo} is equal to
					diagram \ref{diagthree}, which is equal to
					$\psi^{h_1,h_2}_{g_1+g_2}\circ(\varphi^{h_1}_{g_1,g_2}+
					\varphi^{h_2}_{g_1,g_2})\circ\bar{\gamma}_C$.
			\end{enumerate}
					\begin{xyml}\label{diagun}\begin{gathered}\xymatrix@C=60pt@R=40pt{
						A\ar[r]^-\Delta
						&{A\oplus A}\ar[r]^-{g_1\oplus g_2}
						&{B\oplus B}\ar[r]^-{(h_1+h_2)\oplus (h_1+h_2)}\ar[d]_\nabla
							\drtwocell\omit\omit{\;\;\;\;\;\;\;\;\;\varphi^{h_1+h_2}}
						&{C\oplus C}\ar[d]^\nabla
						\\ &&B\ar[r]_{h_1+h_2} &C
					}\end{gathered}\end{xyml}
					\begin{xyml}\label{diagtwo}\begin{gathered}\xymatrix@C=23pt@R=40pt{
						A\ar[r]^-\Delta
						&{A\oplus A}\ar[r]^-{g_1\oplus g_2}
						&{B\oplus B}\ar[r]^-{\Delta\oplus\Delta}\ar[d]_\nabla
							\drtwocell\omit\omit{\;\;\;\;\;\varphi^\Delta}
						&{\txt{$B\oplus B$\\ ${}\oplus B\oplus B$}}\ar[d]^{\nabla_{B\oplus B}}
							\ar[rr]^-{\substack{h_1\oplus h_2\\ {}\oplus h_1\oplus h_2}}
							\drrtwocell\omit\omit{\;\;\;\;\;\;\;\;\;\varphi^{h_1\oplus h_2}}
						&&{\txt{$C\oplus C$\\ ${}\oplus C\oplus C$}}\ar[d]^{\nabla_{C\oplus C}}
							\ar[r]^-{\nabla\oplus\nabla}
							\drtwocell\omit\omit{\;\;\;\;\;\varphi^\nabla}
						&{C\oplus C}\ar[d]^\nabla
						\\ &&B\ar[r]_-{\Delta}
						&{B\oplus B}\ar[rr]_{h_1\oplus h_2}
						&&{C\oplus C}\ar[r]_-\nabla &C
					}\end{gathered}\end{xyml}
					\begin{xyml}\label{diagthree}\begin{gathered}\xymatrix@C=23pt@R=40pt{
						A\ar[r]^-\Delta
						&{A\oplus A}\ar[r]^-{g_1\oplus g_2}
						&{B\oplus B}\ar[r]^-{\Delta\oplus\Delta}
						&{\txt{$B\oplus B$\\ ${}\oplus B\oplus B$}}\ar[d]_{B\oplus c\oplus B}
							\ar[rr]^-{\substack{h_1\oplus h_2\\ {}\oplus h_1\oplus h_2}}
						&&{\txt{$C\oplus C$\\ ${}\oplus C\oplus C$}}\ar[d]^{C\oplus c\oplus C}
							\ar[r]^-{\nabla\oplus\nabla}
							\ddrtwocell\omit\omit{_<-2>{\;\;\,\bar{\gamma}_C}}
						&{C\oplus C}\ar[dd]^\nabla
						\\&&&{\txt{$B\oplus B$\\ ${}\oplus B\oplus B$}}\ar[d]_{\nabla\oplus\nabla}
							\ar[rr]^-{\substack{h_1\oplus h_1\\ {}\oplus h_2\oplus h_2}}
							\drrtwocell\omit\omit
							{\;\;\;\;\;\;\;\;\;\;\;\;\varphi^{h_1}\oplus\varphi^{h_2}}
						&&{\txt{$C\oplus C$\\ ${}\oplus C\oplus C$}}\ar[d]^{\nabla\oplus\nabla}
						\\ &&&{B\oplus B}\ar[rr]_{h_1\oplus h_2}
						&&{C\oplus C}\ar[r]_-\nabla &C
					}\end{gathered}\end{xyml}
			\begin{enumerate}
				\item[]For the second condition (the left side of diagram \ref{preaddc2}), we have
					on one side $\varphi^0(f_1\oplus f_2)\Delta$, which is the identity
					on $0$, by equation \ref{phizero}.  On the other side, we have
					$\rho_C(0\oplus 0)i_1$ which is the identity on $0$, because
					$0\oplus 0=0$.
					For the third and fourth conditions, all involved 2-arrows
					are identities.
				\item[(d)]The first of the conditions (diagram \ref{preaddd})
					is a direct consequence of Proposition \ref{PhiGpdfonc}.
					The second involves only identities.
				\item[(e)]For the first condition, it suffices to write the two
					terms of the equation to see that they are equal; the second condition
					involves only identities.
				\item[(f)]These conditions involve only identities.
				\item[(g)]For the left part of diagram \ref{preaddg}, we have
					$\varphi^{1_B}_{f_1,f_2}=\varphi^{1_B}(f_1\oplus f_2)\Delta$,
					which is equal to the identity, because $\varphi^{1_B}=1_\nabla$, by
					Proposition \ref{PhiGpdfonc}.  The other conditions involve
					only identities.\qedhere
			\end{enumerate}
			\renewcommand{\theenumi}{\arabic{enumi}}\renewcommand{\labelenumi}{\theenumi .}
		\end{proof}

	Semiadditive $\Gpd$-categories have been defined by Baues and Pirashvili
	\cite{Baues2004a} and additive $\Gpd$-categories by Drion \cite{Drion2002a}.

	\begin{df}
		Let $\C$ a $\Gpd$-category.
		\begin{enumerate}
			\item We say that $\C$ is \emph{semiadditive} if it is presemiadditive
				and has all finite biproducts.\index{semiadditive $\Gpd$-category}%
				\index{Gpd-category@$\Gpd$-category!semiadditive}
			\item We say that $\C$ is \emph{additive} if it is preadditive and has
				all finite biproducts.\index{additive!Gpd-categorie@$\Gpd$-category}%
				\index{Gpd-category@$\Gpd$-category!additive}
		\end{enumerate}
	\end{df}
	
	A corollary of the previous proposition is that, in the definition of semiadditivity,
	we can remove presemiadditivity, which follows from the existence of
	biproducts (that is how Baues and Pirashvili have defined it).
	We can also deduce the following corollary, which shows that
	preadditive $\Gpd$-categories coincide with what Baues and Pirashvili
	\cite{Baues2004a} call “additive track theories” (condition 3 of the corollary)
	and Baues, Jibladze and Pirashvili \cite{Baues2006a}
	call “2-additive track categories” (condition 2).
	We denote by $\mathrm{Ho}\,\C$ the \emph{homotopy category}\index{homotopy category}%
	\index{category!homotopy}\index{Ho(C)@$\mathrm{Ho}\,\C$}
	of $\C$, which has the same objects as $\C$ and such that 
	\begin{eqn}\label{defhochomotop}
		\mathrm{Ho}\,\C(A,B)=\pi_0(\C(A,B)).
	\end{eqn}
	In $\mathrm{Ho}\,\C$ the limits of $\C$ become in general weak limits
	but the biproduct of $\C$ does remain a biproduct in $\mathrm{Ho}\,\C$.

	\begin{coro}
		Let $\C$ be a $\Gpd$-category.  The following conditions are equivalent.
		\begin{enumerate}
			\item $\C$ is additive.
			\item $\C$ has all finite biproducts and, for each $A\overset{f}\ra B$, there
				exists $A\overset{f^*}\ra B$ with an isomorphism
				$\eps\col \nabla(f\oplus f^*)\Delta\Ra 0$.
			\item $\C$ has all finite biproducts and $\mathrm{Ho}\,\C$ is additive.
		\end{enumerate}
	\end{coro}
	
		\begin{proof}
			{\it 3 $\Rightarrow$ 2. } In $\mathrm{Ho}\,\C$, the
			biproduct is the same as in $\C$, and $f+g$ is equal to
			\begin{eqn}
				A\overset{\Delta}\longrightarrow A\oplus A\xrightarrow{f\oplus g}
				B\oplus B\overset{\nabla} \longrightarrow B.
			\end{eqn}
			As $\mathrm{Ho}\,\C$ is additive, there exists $f^*$ such that $f+f^*=0$.
			But two arrows are equal in $\mathrm{Ho}\,\C$ if there exists an
			isomorphism between them in $\C$.  So
			there is an isomorphism $\nabla(f\oplus f^*)\Delta\Ra 0$.
			
			{\it 2 $\Rightarrow$ 1. } By the previous proposition,
			$\C$ is presemiadditive and the existence of $f^*$ and $\eps$ shows
			that $f^*$ is the opposite of $f$ for the addition of $\C(A,B)$ and thus
			that each $\C(A,B)$ is a symmetric 2-group.
			
			{\it 1 $\Rightarrow$ 3. } The 2-arrows become equalities
			for the axioms of group and of distributivity.
		\end{proof}

\section{Additivity and regularity of 2-abelian $\Gpd$-categories}\label{sectaddregdab}

\subsection{Matrix product}

	In this subsection we will see that in dimension 2 also the matrix product
	corresponds to the composition of arrows.
	
	We work in a semiadditive $\Gpd$-category.  We assume,
	by Mac Lane's coherence theorem  that $\alpha$, $\lambda$ and $\rho$ are
	identities.  In the same way we assume that $\varphi_0$ and $\psi^0$ are
	identities.
	
	\begin{df}
		Let $(A_j)_{1\leq j\leq m}$ and $(B_k)_{1\leq k\leq n}$ be two families
		of objects in a $\Gpd$-category $\C$.  The groupoid of matrices between
		these two families of objects is\index{Mat(C)@$\caspar{Mat}(\C)$}
		\begin{eqn}
			\caspar{Mat}(\C)((A_j),(B_k)) \eqdef
			\prod_{\substack{1\leq k\leq n\\ 1\leq j\leq m}}\C(A_j,B_k).
		\end{eqn}
	\end{df}
	
	\begin{df}
		Let $(A_j)_{1\leq j\leq m}$, $(B_k)_{1\leq k\leq n}$ and  $(C_l)_{1\leq l\leq o}$
		be families of objects in a presemiadditive $\Gpd$-category $\C$.
		The matrix product\index{matrix product}
		\begin{eqn}
			\mathrm{prod}\col \caspar{Mat}(\C)((A_j),(B_k))\times\caspar{Mat}(\C)((B_k),(C_l)) 
			\ra\caspar{Mat}(\C)((A_j),(C_l)) 
		\end{eqn}
		is defined on objects by
		\begin{eqn}
				(g_{lk})_{\substack{1\leq l\leq o\\ 1\leq k\leq n}}
				(f_{kj})_{\substack{1\leq k\leq n\\ 1\leq j\leq m}}
				\eqdef\left(\sum_{k=1}^n g_{lk}\circ f_{kj}\right)_{\substack{1\leq l\leq o\\ 
				1\leq j\leq m}}.
		\end{eqn}
	\end{df}
	
	We will see that, for a semiadditive $\Gpd$-category, this matrix product
	corresponds under the equivalence \ref{equcoprodprod} to the composition 
	\begin{eqn}
		\C\left(\bigoplus_{j=1}^m A_j,\bigoplus_{k=1}^n B_k\right)\times
		\C\left(\bigoplus_{k=1}^n B_k,\bigoplus_{l=1}^o C_l\right)\ra
		\C\left(\bigoplus_{j=1}^m A_j,\bigoplus_{l=1}^o C_l\right).
	\end{eqn}

	\begin{pon}\label{tradprodmatcomp}
		Let $\C$ be a presemiadditive $\Gpd$-category.
		There exists a natural isomorphism $\xi$ as in the following diagram.
		\begin{xym}\xymatrix@C=60pt@R=40pt{
			{\txt{$\C(\bigoplus_{j=1}^m A_j,\bigoplus_{k=1}^n B_k)$\\ ${}\times
				\C(\bigoplus_{k=1}^n B_k,\bigoplus_{l=1}^o C_l)$}}
				\ar[d]_-{\substack{(p_k\circ - \circ i_j)_{k,j}\\ 
				{}\times(p_l\circ - \circ i_k)_{l,k}}}
				\ar[r]^-{\mathrm{comp}}\drtwocell\omit\omit{^\xi}
			&{\C(\bigoplus_{j=1}^m A_j,\bigoplus_{l=1}^o C_l)}
				\ar[d]^-{(p_l\circ - \circ i_j)_{l,j}}
			\\ {\txt{$\caspar{Mat}(\C)((A_j),(B_k))$\\ ${}\times
				\caspar{Mat}(\C)((B_k),(C_l))$}}\ar[r]_-{\mathrm{prod}}
			&{\caspar{Mat}(\C)((A_j),(C_l))}
		}\end{xym}
	\end{pon}
		
		\begin{proof}
			We prove this for the case $n=2$, which is the only case that we
			use in the following.
			We define $\xi^{f,g}_{lj}$ as being the following composite.
			\begin{multline}\stepcounter{eqnum}
				p_lgi_1p_1fi_j+p_lgi_2p_2fi_j\overset{\varphi}\ra
				p_lg(i_1p_1fi_j+i_2p_2fi_j)\\ \xrightarrow{p_lg\psi}
				p_lg(i_1p_1+i_2p_2)fi_j\xrightarrow{p_lg\omega fi_j} p_lgfi_j
			\end{multline}
			The naturality of $\xi$ follows from that of $\psi$ and $\varphi$.
		\end{proof}

	We also need the natural isomorphism going in the other direction.
	\begin{pon}
		Let $\C$ a presemiadditive $\Gpd$-category.
		There exists a natural isomorphism $\theta$ as in the following diagram.
		\begin{xym}\xymatrix@C=60pt@R=40pt{
			{\txt{$\C(\bigoplus_{j=1}^m A_j,\bigoplus_{k=1}^n B_k)$\\ ${}\times
				\C(\bigoplus_{k=1}^n B_k,\bigoplus_{l=1}^o C_l)$}}
				\ar[r]^-{\mathrm{comp}}
			&{\C(\bigoplus_{j=1}^m A_j,\bigoplus_{l=1}^o C_l)}
			\\ {\txt{$\caspar{Mat}(\C)((A_j),(B_k))$\\ ${}\times
				\caspar{Mat}(\C)((B_k),(C_l))$}}\ar[r]_-{\mathrm{prod}}\ar[u]^{\Phi}
				\urtwocell\omit\omit{^\theta}
			&{\caspar{Mat}(\C)((A_j),(C_l))}\ar[u]_{\Phi}
		}\end{xym}
	\end{pon}
	
		\begin{proof}
			We define $\theta$ for two special cases with $m,n,o\leq 2$.
			We define first
			\begin{xym}\xymatrix@=40pt{
				A\ar[r]^-{\vervec{b_1}{b_2}}\rrlowertwocell_{c_1b_1+c_2b_2}<-9>{_<2.7>\theta}
				&{B_1\oplus B_2}\ar[r]^-{(c_1\; c_2)}
				&C
			}\end{xym}
			as the composite
			\begin{multline}\stepcounter{eqnum}
				(c_1\; c_2)\vervec{b_1}{b_2}
				\xLongrightarrow{(c_1\; c_2)\omega^{-1}\vervec{b_1}{b_2}}
				(c_1\; c_2)(i_1p_1+i_2p_2)\vervec{b_1}{b_2}
				\xLongrightarrow{(c_1\; c_2)\psi^{-1}}
				\\ (c_1\; c_2)i_1p_1\vervec{b_1}{b_2}+(c_1\; c_2)i_2p_2\vervec{b_1}{b_2}
				\xLongrightarrow{\iota_1*\pi_1+\iota_2*\pi_2\,} c_1b_1+c_2b_2.
			\end{multline}
			Next we define
			\begin{xym}\xymatrix@=40pt{
				A\ar[r]^-{b}\rrlowertwocell_{\overset{~}{\overset{~}{\overset{~}{\overset{~}{\overset{~}{\overset{~}{\vervec{c_1b}{c_2b}}}}}}}}<-9>{_<2.7>\theta}
				&B\ar[r]^-{\vervec{c_1}{c_2}}
				&{C_1\oplus C_2}
			}\end{xym}
			such that $p_1\theta = \pi_1^{-1}\circ\pi_1b$ and $p_2\theta = 
			\pi_2^{-1}\circ\pi_2b$.
		\end{proof}
		
		We will need two little “associativity” properties of $\theta$.
		
	\begin{lemm}\label{lemmforlaregul}
		\begin{xyml}\begin{gathered}\xymatrix@=50pt{
			B\ar[r]^-{\vervec{b_1}{b_2}}\rtwocell\omit\omit{_<3>{\theta}}
			&{C_1\oplus C_2}\ar[d]^{(c_1\; c_2)}
			\\ A\ar[u]^a\ar[ur]|(0.3){\vervec{b_1a}{b_2a}}\ar[r]_-{c_1b_1a+c_2b_2a}
				\rtwocell\omit\omit{_<-3>{\theta}}
			&D
		}\end{gathered}\;\;=\;\;\begin{gathered}\xymatrix@=50pt{
			B\ar[r]^-{\vervec{b_1}{b_2}}\rtwocell\omit\omit{_<3>{\theta}}
				\ar[dr]|{\;c_1b_1+c_2b_2}
			&{C_1\oplus C_2}\ar[d]^{(c_1\; c_2)}
			\\ A\ar[u]^a\ar[r]_-{c_1b_1a+c_2b_2a}
				\rtwocell\omit\omit{_<-3>{\psi}}
			&D
		}\end{gathered}\end{xyml}
		\begin{xyml}\begin{gathered}\xymatrix@=50pt{
			B\ar[r]^-{b}
			&C\ar[d]^{\vervec{c_1}{c_2}}
			\\ A\ar[u]^a\ar[r]_-{\vervec{c_1ba}{c_2ba}}
				\rtwocell\omit\omit{_<-3>{\theta}}
				\ar[ur]^{ba}
			&{D_1\oplus D_2}
		}\end{gathered}	\;\;=\;\;\begin{gathered}\xymatrix@=50pt{
			B\ar[r]^-{b}\rtwocell\omit\omit{_<3>{\theta}}
				\ar[dr]|(0.3){\vervec{c_1b}{c_2b}}
			&C\ar[d]^{\vervec{c_1}{c_2}}
			\\ A\ar[u]^a\ar[r]_-{\vervec{c_1ba}{c_2ba}}
				\rtwocell\omit\omit{_<-3>{\theta}}
			&{D_1\oplus D_2}
		}\end{gathered}\end{xyml}
	\end{lemm}

\subsection{Additivity of 2-abelian $\Gpd$-categories}

	In this subsection, we prove that every 2-abelian $\Gpd$-category
	is additive, by following the proof of \cite{Freyd2003a}: we have already proved
	that every 2-abelian $\Gpd$-category has all finite biproducts (Proposition
	\ref{gpd2abbiprod}) and that, if a $\Gpd$-category has all finite biproducts,
	it is presemiadditive.  Thus we already know that every 2-abelian $\Gpd$-category
	is semiadditive. It remains to prove that every arrow $A\overset{f}\ra B$
	has an opposite for the addition of the symmetric 2-monoid $\C(A,B)$.
	
	\begin{lemm}
		Let $\C$ be a semiadditive $\Gpdp$-category where every fully 0-faithful
		and fully 0-cofaithful arrow is an equivalence.  Then for each object
		$A\col \C$, the matrix
		\begin{eqn}\label{matr1101}
			\begin{pmatrix}1_A &1_A \\0 &1_A\end{pmatrix}\col A\oplus A\ra A\oplus A
		\end{eqn}
		is an equivalence.
	\end{lemm}
	
		\begin{proof}
			We prove that this matrix is fully 0-faithful.  Dually, it will be
			fully 0-cofaithful and thus an equivalence.
			
			Let be $X\col \C$.  We must prove
			that, for each $a\col X\ra A\oplus A$
			and for each $\alpha\col \begin{pmatrix}1_A &1_A \\0 &1_A\end{pmatrix}a\Ra 0$,
			there exists a unique $\alpha'\col a\Ra 0$ such that 
			$\alpha=\begin{pmatrix}1_A &1_A \\0 &1_A\end{pmatrix}\alpha'$.  
			By Proposition
			\ref{tradprodmatcomp}, we can transfer this situation to the matrix side.
			We must prove that for each
			$\begin{pmatrix}a_1\\ a_2\end{pmatrix}\col X\ra (A,A)$ and for each 
			$\begin{pmatrix}\alpha_1\\ \alpha_2\end{pmatrix}\col 
			\begin{pmatrix}1_A &1_A \\0 &1_A\end{pmatrix}\begin{pmatrix}a_1\\ a_2\end{pmatrix}
			\Ra \begin{pmatrix}0\\ 0\end{pmatrix}$, there exists a unique
			$\begin{pmatrix}\alpha'_1\\ \alpha'_2\end{pmatrix}\col 
			\begin{pmatrix}a_1\\ a_2\end{pmatrix}\Ra \begin{pmatrix}0\\ 0\end{pmatrix}$
			such that $\begin{pmatrix}1_A &1_A \\0 &1_A\end{pmatrix}
			\begin{pmatrix}\alpha'_1\\ \alpha'_2\end{pmatrix}=
			\begin{pmatrix}\alpha_1\\ \alpha_2\end{pmatrix}$.
			
			Let be such a situation.  We have $\alpha_1\col a_1+a_2\Ra 0$ and
			$\alpha_2\col a_2\equiv 0+a_2\Ra 0$.  We set $\alpha'_2\eqdef\alpha_2$ and $\alpha'_1$ equal to the
			composite $a_1\equiv a_1+0\xrightarrow{a_1+\alpha_2^{-1}} a_1+a_2\overset{\alpha_1}\longrightarrow 0$.
			Then
			\begin{eqn}
				\begin{pmatrix}1_A &1_A \\0 &1_A\end{pmatrix}
				\begin{pmatrix}\alpha'_1\\ \alpha'_2\end{pmatrix}
				=\begin{pmatrix}\alpha'_1+\alpha'_2\\ \alpha'_2\end{pmatrix}
				=\begin{pmatrix}\alpha_1\\ \alpha_2\end{pmatrix},
			\end{eqn}
			because the following diagram commutes.
			\begin{xym}\xymatrix@=40pt{
				a_1+a_2\ar@{=}[r]\ar@/^2pc/[rrrr]^-{\alpha'_1+\alpha'_2}
				&a_1+0+a_2\ar[r]_-{a_1+0+\alpha_2}\ar@{=}[d]
				&a_1+0+0\ar[r]_-{a_1+\alpha_2^{-1}+0}\ar@{=}[d]
				&a_1+a_2+0\ar[r]_-{\alpha_1}\ar@{=}[d]
				&0
				\\ &a_1+a_2\ar[r]^-{a_1+\alpha_2}\ar@/_1.5pc/@{=}[rr]_{}
				&a_1+0\ar[r]^-{a_1+\alpha_2^{-1}}
				&a_1+a_2
			}\end{xym}
			
			For unicity, let be
			$\begin{pmatrix}\alpha_1\\ \alpha_2\end{pmatrix}\col \begin{pmatrix}0\\ 0\end{pmatrix}
			\Ra\begin{pmatrix}0\\ 0\end{pmatrix}\col X\ra(A,A)$ such that
			\begin{eqn}
				\begin{pmatrix}1_A &1_A \\0 &1_A\end{pmatrix}
				\begin{pmatrix}\alpha_1\\ \alpha_2\end{pmatrix}
				=\begin{pmatrix}1_0\\ 1_0\end{pmatrix}.
			\end{eqn}
			Then we have $\alpha_2=1_0$ and $\alpha_1=\alpha_1+1_0=\alpha_1+\alpha_2=1_0$.
		\end{proof}
		
	\begin{coro}
		Under the hypotheses of the previous lemma, for each $A\col \C$, there exists
		$n_A\col A\ra A$ such that $n_A+1_A\simeq 0$.
	\end{coro}
	
		\begin{proof}
			By the previous lemma, matrix \ref{matr1101} is an equivalence.
			There is thus an inverse matrix:
			\begin{eqn}
				\begin{pmatrix}1_A &1_A \\0 &1_A\end{pmatrix}
				\begin{pmatrix}a &b \\c &d\end{pmatrix}
				\simeq \begin{pmatrix}1_A &0 \\0 &1_A\end{pmatrix}.
			\end{eqn}
			This gives a system of four isomorphisms:
			\begin{align}\stepcounter{eqnum}\begin{split}
				a+c &\simeq 1_A; \\
				b+d &\simeq 0; \\
				c &\simeq 0; \\
				d &\simeq 1_A.
			\end{split}\end{align}
			It follows that $b+1_A\simeq b+d\simeq 0$. We can thus take $n_A\eqdef b$.
		\end{proof}

	The arrow $n_A$ plays the rôle of an antipode for the bimonoid
	$A\overset{\Delta}\ra A\oplus A\overset{\nabla}\ra A$, which becomes a Hopf monoid 
	(the isomorphisms $n_A+1_A\simeq 0$
	and $1_A+n_A\simeq 0$ should satisfy some additional conditions for $n_A$ to be a genuine antipode).
		
	\begin{coro}
		Under the hypotheses of the previous lemma, $\C$ is additive.
	\end{coro}
	
		\begin{proof}
			If $A\overset{f}\ra B$ is an arrow in $\C$, then
			$n_Bf+f=n_Bf+1_Bf\simeq(n_B+1_B)f\simeq 0f\equiv 0$.  Thus $n_Bf$
			is an opposite for $f$.
		\end{proof}
	
	\begin{coro}\label{2abPexadd}
		A $\Gpd$-category is 2-abelian if and only if it
		is 2-Puppe-exacte and additive.
	\end{coro}
	
\subsection{Regularity of 2-abelian $\Gpd$-categories}

	In this subsection, we prove that every 2-abelian $\Gpd$-category
	is regular, in the sense that cofaithful and fully
	cofaithful arrows are stable under pullback and, dually, faithful
	and fully faithful arrows are stable under pushout.  We will deduce from that that
	2-abelian $\Gpd$-categories are also abelian.
	In this subsection, we denote by $f^*\col A\ra B$ the opposite of $f\col A\ra B$
	for the addition of arrows.  We assume that symmetric 2-groups are
	strictly described.
	
	We follow the proof of \cite{Mitchell1965a} or \cite{Borceux1994a}
	in dimension 1.
	First, we prove the correspondence between pullbacks and kernels.
	
	\begin{lemm}
		Let $\C$ be an additive $\Gpd$-category and
		let $B_1\overset{g_1}\ra C\overset{g_2}\leftarrow B_2$ be arrows in $\C$,
		to which corresponds an arrow $(g_1\; g_2^*)\col B_1\oplus B_2\ra C$.
		There is an equivalence
		\begin{eqn}
			\Phi\col\mathrm{PBCand}(g_1,g_2)\ra\mathrm{KerCand}((g_1\; g_2^*))
		\end{eqn}
		from the pullback-candidates of $g_1$ and $g_2$ to the kernel-candidates of
		$(g_1\; g_2^*)$.
	\end{lemm}
	
		\begin{proof}
			The details of the following proof are checked
			by using Lemma \ref{lemmforlaregul} and the naturality of $\theta$.
			
			{\it Construction of $\Phi$. } The $\Gpd$-functor $\Phi$ maps a square
			$\gamma\col g_1b_1\Ra g_2b_2$, where $b_1\col X\ra B_1$ and $b_2\col X\ra B_2$,
			to $\vervec{b_1}{b_2}\col X\ra B_1\oplus B_2$, equipped with the 2-arrow
			$\bar{\gamma}\col (g_1\; g_2^*)\vervec{b_1}{b_2}\Ra 0$ defined as the 
			following composite:
			\begin{eqn}
				(g_1\; g_2^*)\vervec{b_1}{b_2}\overset{\theta}\Longrightarrow g_1b_1+g_2^*b_2
				\overset{\gamma + 1\;\;\,}\Longrightarrow g_2b_2+g_2^*b_2\overset{\psi\,}
				\Longrightarrow (g_2+g_2^*)b_2\overset{\varepsilon b_2\;\,}\Longrightarrow 0.
			\end{eqn}
			
			An arrow $(x,\beta_1,\beta_2)\col(X,b_1,b_2,\gamma)\ra (X',b'_1,b'_2,\gamma')$
			between pullback-can\-di\-dates ($x\col X\ra X'$, $\beta_1\col b_1\Ra b'_1x$
			and $\beta_2\col b_2\Ra b'_2x$ are such that $\gamma'x\circ g_1\beta_1=g_2\beta_2
			\circ\gamma$) is mapped to $x\col X\ra X'$ equipped with the 2-arrow
			$\bar{\beta}$ equal to the composite
			\begin{eqn}
				\vervec{b_1}{b_2}\xLongrightarrow{\vervec{\beta_1}{\beta_2}}
				\vervec{b'_1x}{b'_2x}\overset{\theta^{-1}}\Longrightarrow
				\vervec{b'_1}{b'_2}x.
			\end{eqn}
			
			A 2-arrow $\chi\col(x,\beta_1,\beta_2)\Ra (x',\beta'_1,\beta'_2)$
			is mapped to itself.
			
			The natural transformations of the structure
			of $\Gpd$-functor are identities.
			
			{\it $\Phi$ is surjective. } Let be $X\overset{b}\ra B_1\oplus B_2$
			equipped with a 2-arrow $\bar{\gamma}\col (g_1\; g_2^*)b\Ra 0$.
			We set $\tilde{\gamma}$ equal to the composite
			\begin{multline}\stepcounter{eqnum}
				g_1p_1b+g_2^*p_2b\xLongrightarrow{\iota_1^{-1}p_1b+\iota_2^{-1}p_2b}
				(g_1\; g_2^*)i_1p_1b + (g_1\; g_2^*)i_2p_2b\overset{\varphi}\Longrightarrow
				(g_1\; g_2^*)(i_1p_1b+i_2p_2b)\\ \xLongrightarrow{(g_1\; g_2^*)\psi}
				(g_1\; g_2^*)(i_1p_1+i_2p_2)b \xLongrightarrow{(g_1\; g_2^*)\omega b}
				(g_1\; g_2^*)b\overset{\bar{\gamma}}\Longrightarrow 0.
			\end{multline}
			and $\gamma$ equal to the composite
			\begin{multline}\stepcounter{eqnum}
				g_1p_1b\equiv g_1p_1b + 0p_2b\xLongrightarrow{1+\eta p_2b}
				g_1p_1b+(g_2^*+g_2)p_2b\\ \xLongrightarrow{1+\psi^{-1}}
				g_1p_1b+g_2^*p_2b+g_2p_2b \xLongrightarrow{\tilde{\gamma}+1}
				0+g_2p_2b\equiv g_2p_2b.
			\end{multline}
			Then $B_1\xleftarrow{p_1b} X\xrightarrow{p_2b} B_2$, equipped with $\gamma$,
			is a pullback-candidate of $g_1$ and $g_2$ and $\Phi(X,p_1b,p_2b,\gamma)
			\simeq(X,b,\bar{\gamma})$.
			
			{\it $\Phi$ is locally surjective. } Let $(X,b_1,b_2,\gamma)$
			and $(X',b'_1,b'_2,\gamma')$ be pullback-candidates and
			let be $(x,\beta)\col (X,\vervec{b_1}{b_2},\bar{\gamma})\ra 
			(X',\vervec{b'_1}{b'_2},\bar{\gamma'})$
			in $\mathrm{KerCand}((g_1\; g_2^*))$.
			For $i=1,2$, we set $\beta_i$ equal to the composite
			\begin{eqn}
				b_i\xLongrightarrow{\pi_i^{-1}}p_i\vervec{b_1}{b_2}\xLongrightarrow{p_i\beta}
				p_i\vervec{b'_1}{b'_2}x\overset{\pi_i x}\Longrightarrow b'_ix.
			\end{eqn}
			Then $\Phi(x,\beta_1,\beta_2)\simeq (x,\beta)$.
			
			{\it $\Phi$ is locally full and faithful. }  Local faithfulness is obvious.
			Moreover, if $\chi$ is a 2-arrow in
			$\mathrm{KerCand}((g_1\; g_2^*))$, then it is a 2-arrow in
			$\mathrm{PBCand}(g_1,g_2)$.
		\end{proof}
	
	\begin{pon}\label{tradpfnoy}
		Let $\C$ be an additive $\Gpd$-category.  The square
		\begin{xym}\label{pfregul}\xymatrix@=40pt{
			A\ar[r]^-{f_1}\drtwocell\omit\omit{\gamma}\ar[d]_{f_2}
			&B_1\ar[d]^{g_1}
			\\ B_2\ar[r]_{g_2}
			&C
		}\end{xym}
		is a pullback if and only if $\vervec{f_1}{f_2}\col A\ra B_1\oplus B_2$,
		equipped with $\bar{\gamma}$, is a kernel of $(g_1\; g_2^*)$.
	\end{pon}
	
		\begin{proof}
			We use the equivalence $\Phi$ of the previous lemma.
			Since
			$\left(A,\vervec{f_1}{f_2},\bar{\gamma}\right)=\Phi(A,f_1,f_2,\gamma)$,
			the one is an initial object if and only if the other is an initial object.
		\end{proof}
	
	\begin{lemm}
		Let $\C$ be a $\Gpdp$-category and let be the following diagram in $\C$,
		where the row and the column are extensions.
		\begin{xym}\xymatrix@=40pt{
			&A\ar[d]^{f}\ar@/^1.89pc/[dd]^(0.44)0
			\\ C\ar[r]_{h}\ar@/_1.87pc/[rr]_(0.45)0
			&E\ar[r]_(0.62){k}\ar[d]^(0.62){g}
				\ar@{}[d]_(0.25){\nu\,\dir{=>}}
				\ar@{}[r]^(0.22){\txt{$\overset{\mu\;\;}{\dir{=>}}$}}
			&D
			\\ &B
		}\end{xym}
		Then:
		\begin{enumerate}
			\item $gh$ is 0-faithful if and only if $kf$ is 0-faithful;
			\item $gh$ is 0-cofaithful if and only if $kf$ is 0-cofaithful;
			\item $gh$ is fully 0-faithful if and only if $kf$ is
				fully 0-faithful;
			\item $gh$ is fully 0-cofaithful if and only if $kf$ is
				fully 0-cofaithful.
		\end{enumerate}
	\end{lemm}
	
		\begin{proof}
			Properties 2 and 4 are the dual of properties 1 and 3, thus it suffices
			to prove the latter.  Moreover, for each case the
			situation is symmetric with respect to the line with slope $-1$
			passing through $E$.
			
			{\it $gh$ 0-faithful $\Rightarrow$ $kf$ 0-faithful. } 
			Let $\alpha\col 0\Ra 0\col X\ra A$ be such that $kf\alpha = 1_0$. Then
			$f\alpha\col h0\Ra 0$ is compatible with $\nu$ and, by the
			universal property of the kernel ($h$ is $\nu$-fully faithful),
			there exists $\gamma\col 0\Ra 0\col X\ra C$ such that $h\gamma=f\alpha$.
			Then $gh\gamma=gf\alpha=1_0$, because $gf\simeq 0$.  So $\gamma = 1_0$ since,
			by hypothesis, $gh$ is 0-faithful.  Then $f\alpha = h1_0=1_0$ and, since
			$f$ is faithful, $\alpha=1_0$.
			
			{\it $gh$ fully 0-faithful $\Rightarrow$ $kf$ fully 0-faithful. }
			We already know by the previous part of the proof that $kf$
			is 0-faithful.  It remains to prove that for each
			$\delta\col kfa\Ra 0$,
			where $a\col X\ra A$, there exists $\alpha\col a\Ra 0$ such that $kf\alpha=\delta$.
			First, as $(h,\nu)=\Ker k$, there exist $c\col X\ra C$
			and $\varepsilon\col fa\Ra hc$ such that $\nu c\circ k\varepsilon = \delta$.
			
			Since $gh$ is fully 0-faithful, there exists $\gamma\col c\Ra 0$
			such that $gh\gamma = \mu a\circ g\varepsilon^{-1}$.  Then $h\gamma
			\circ \varepsilon$ is compatible with $\mu$ and there exists
			$\alpha\col a\Ra 0$ such that $f\alpha = h\gamma \circ \varepsilon$.
			Hence we have $kf\alpha=kh\gamma\circ k\varepsilon =
			\nu c\circ k\varepsilon	=\delta$.
		\end{proof}
	
	We can now prove the regularity of 2-abelian $\Gpd$-categories.
	For symmetric 2-groups, this result and its dual have been proved
	by Dominique Bourn and Enrico Vitale \cite[Propositions 5.1 and 5.2]{Bourn2002a}.
	
	\begin{pon}\label{gpdcatdabreg}
		In a 2-abelian $\Gpd$-category, cofaithful and fully cofaithful arrows
		are stable under pullback and faithful and fully cofaithful arrows
		are stable under pushout.
	\end{pon}
	
		\begin{proof}
			We prove stability under pullback; the proof for pushouts
			is dual.
			Let us assume that diagram \ref{pfregul} is a pullback.
			Then
			$\left(\vervec{f_1}{f_2}\col A\ra B_1\oplus B_2,\bar{\gamma}\right)$
			is a kernel of $(g_1\; g_2^*)$, by Proposition \ref{tradpfnoy}.  
			
			Let us consider the following diagram. If $g_1\simeq (g_1\; g_2^*)i_1$
			is cofaithful or fully cofaithful, then $(g_1\; g_2^*)$ is cofaithful
			and is thus the cokernel of its kernel $\vervec{f_1}{f_2}$.  Since the column
			is an extension (by Corollary \ref{biprodextens}),
			we are in the situation of the previous lemma.  Therefore, if $g_1$ is cofaithful,
			$f_2\simeq p_2\vervec{f_1}{f_2}$ is cofaithful, and if $g_1$ is fully
			cofaithful, $f_2$ is fully cofaithful.\qedhere
			\begin{xym}\xymatrix@=40pt{
				&B_1\ar[d]_{i_1}\ar@/^1.89pc/[dd]^(0.36)0\ar@/^1pc/[dr]^-{g_1}
				\\ A\ar[r]^-{\vervec{f_1}{f_2}}\ar@/_1.87pc/[rr]_(0.45)0\ar@/_1pc/[dr]_-{f_2}
				&B_1\oplus B_2\ar[r]^(0.6){(g_1\; g_2^*)}\ar[d]^(0.62){p_2}
					\ar@{}[d]_(0.25){\bar{\gamma}\,\dir{=>}}
					\ar@{}[r]^(0.18){\txt{$\overset{\eta_{12}\;\;}{\underset{~}{\dir{=>}}}$}}
				&C
				\\ &B_2
			}\end{xym}
		\end{proof}
	
	An important consequence of this proposition is that 2-abelian $\Gpd$-categories
	are abelian and that we can thus prove in them the different
	snake lemmas and construct the long exact sequence of homology.  Since the previous proposition
	uses biproducts, the reasoning we use does not work for 2-Puppe-exact and Puppe-exact
	$\Gpdp$-categories.
		
	\begin{coro}\label{twoabimplab}
		Every 2-abelian $\Gpd$-category is also abelian.
	\end{coro}
	
		\begin{proof}
			Since in a 2-abelian $\Gpd$-category faithful arrows are
			the kernel of their cokernel, they are 0-monomorphisms. So
			the faithful arrows and the 0-monomorphisms coincide and
			condition 1 of Definition \ref{dfgpdab} hold.  Dually, condition 2 hold.
			Moreover, conditions 3 and 4 hold by the previous proposition.
		\end{proof}

\chapter{Examples}

\begin{quote}{\it
	In this chapter, we study a few examples of (good) 2-abelian $\Gpd$-categories.
	First, the $\Gpd$-category of symmetric 2-groups
	and the $\Gpd$-categories of 2-modules on a 2-ring (or, more generally, 
	of additive $\Gpd$-functors from a preadditive $\Gpd$-category to $\CGS$).
	Next, we study the $\Gpd$-category of morphisms, commutative squares and homotopies
	in an abelian category $\C$ (a special case is the notion of Baez-Crans 2-vector space)
	and we prove that it is 2-abelian if and only
	if the axiom of choice holds in $\C$ (Theorem \ref{caracdabflabplus}).
}\end{quote}

\section{Symmetric 2-groups: $\CGS$ is 2-abelian}\label{sectcgsdab}

	In this section we assume that all symmetric monoidal functors are
	described in a “normalised” way (with $FI\equiv I$ and $\varphi^F_0\eqdef 1_I$).

\subsection{Construction of limits and colimits}

	There are two inclusions of the category $\Ab$ of abelian groups in
	the $\Gpd$-category $\CGS$. There is the inclusion as discrete object:
	\begin{eqn}\index{dis@$(-)\dis$}
		(-)\dis\col\Ab\ra\CGS,
	\end{eqn}
	which maps an abelian group $A$ to $A\dis$, which is the set $A$ seen as a
	discrete groupoid, equipped with the product and the unit of $A$.  A homomorphism of
	groups is mapped to itself seen as a symmetric monoidal functor.

	Next, there is the inclusion as connected object:
	\begin{eqn}\index{con@$(-)\con$}
		(-)\con\col\Ab\ra\CGS,
	\end{eqn}
	which maps an abelian group $A$ to $A\con$, which is the one-object symmetric 2-group
	$I$ such that $A\con(I,I)=A$.  A homomorphism $f\col A\ra B$ is mapped to
	the symmetric monoidal functor whose action on arrows
	is defined by $(f\con)_{I,I}\eqdef f\col A\ra B$.
	
	Each of these $\Gpd$-functors has an adjoint.
	The left adjoint of $(-)\dis$ is the $\Gpd$-functor 
	\begin{eqn}\index{p0@$\pi_0$!for symmetric 2-groups}
		\pi_0\col \CGS\ra\Ab,
	\end{eqn}
	which maps a symmetric 2-group $\G$ to the group $\pi_0(\G)$ whose
	objects are those of $\G$, whose equality is defined by $A\simeq B$,
	and whose product is that of $\G$. A symmetric monoidal functor $F\col \G\ra\H$
	is mapped to the group homomorphism $\pi_0(F)\col \pi_0(\G)\ra\pi_0(\H)$
	that it induces.  And if $\alpha\col F\Ra F'$ is a
	monoidal natural transformation, then $\pi_0(F)=\pi_0(F')$, which defines
	$\pi_0$ on 2-arrows.
	
	The right adjoint of $(-)\con$ is the $\Gpd$-functor 
	\begin{eqn}\index{p1@$\pi_1$!for symmetric 2-groups}
		\pi_1\col \CGS\ra\Ab,
	\end{eqn}
	which maps a symmetric 2-group $\G$ to the group $\pi_1(\G)\eqdef\G(I,I)$.
	A symmetric monoidal functor $F\col \G\ra\H$ is mapped to the group homomorphism
	$\pi_1(F)\eqdef F_{I,I}\col \pi_1(\G)\ra\pi_1(\H)$. And if $\alpha\col F\Ra F'$
	is a monoidal natural transformation, then $\pi_1(F)=\pi_1(F')$.
	
	There is a difference of presentation between the situation described above
	and that of diagram \ref{diagequcondis}.
	Here, the two functors $\pi_0$ and $\pi_1$
	go to $\Ab$ (which is by definition $\caspar{Dis}(\CGS)$).
	We can recover the functors $\Omega$, $\Sigma$, $\pi_0$ and $\pi_1\col\CGS\ra\CGS$
	by defining:
	\begin{align*}
		\pi_0&\eqdef (\pi_0 -)\dis; &\pi_1&\eqdef (\pi_1 -)\con;\\ 
		\Sigma &\eqdef (\pi_0 -)\con; &\Omega &\eqdef (\pi_1 -)\dis.
	\end{align*}

	\bigskip
	Let us now recall the constructions in $\CGS$ of biproducts,
	zero object, kernels, cokernels,
	pips, copips, roots and coroots, described in \cite{Kasangian2000a}
	and \cite{Dupont2003a}.
	
	First, the biproduct of two symmetric 2-groups $\A$ and $\B$ is simply
	the cartesian product $\A\times\B$, equipped with the structure of symmetric 2-group
	defined componentwise \cite[Section 4]{Bourn2002a}.%
	\index{biproduct!in 2-sgp@in $\CGS$}
	We need symmetry to prove
	that it is also a coproduct.  The symmetric 2-group with one object and one arrow
	is a zero object.
	
	The kernel and the cokernel of symmetric 2-groups has been defined by Vitale
	\cite{Vitale2002a}.  The kernel is constructed as in $\Gpdp$ (see the construction
	given after Definition \ref{defkernel}).
	
	\begin{df}\label{dfdescker}\index{kernel!in 2-sgp@in $\CGS$}
		The kernel $\Ker F$ of $F\col \A\ra\B$ in $\CGS$ is defined in the following way.
		\begin{itemize}
			\item {\it Objects.} These are the pairs $(A,b)$, where
				$A\col \A$ and $b\col FA\ra I$ in $\B$.
			\item {\it Arrows.} A morphism $(A,b)\ra (A',b')$ is an arrow
				$f\col A\ra A'$ such that $b'(Ff) = b$.  The composition and the identities
				are those of $\A$.
			\item {\it Product.} The product $(A,b)\tens(A',b')$ is $(A\tens A',b'')$,
				where $b''$ is the composite
				\begin{eqn}
					F(A\tens A')\xrightarrow{\varphi^{-1}}FA\tens FA'
					\xrightarrow{b\tens b'}I\tens I\xrightarrow{l_I=r_I}I.
				\end{eqn}
				The product $f\tens f'$ is defined as in $\A$; this is an
				arrow in $\Ker F$ because $\varphi^F$ is natural.
			\item {\it Unit.} This is $(I,1_I)$.
			\item {\it Associativity, neutrality, symmetry.} The natural transformations
				$a$, $l$, $r$ and $c$ are defined as in $\A$
				(for example $l_{(A,b)}\eqdef l_A$).  These are arrows of $\Ker F$ 
				thanks to the axioms of symmetric monoidal functor and they are
				natural because they are in $\A$.  The axioms of symmetric monoidal groupoid
				hold because they do in $\A$.
			\item {\it Inverses.}  The inverse of $(A,b)$ is $(A^*,\tilde{b})$,
				where $\tilde{b}$ is the composite 
				\begin{eqn}
					F(A^*)\simeq (FA)^*\xrightarrow{(b^*)^{-1}}I^*\simeq I.
				\end{eqn}
				The arrow $\varepsilon_{(A,b)}$ is $\varepsilon_A$.
		\end{itemize}
		There is a functor $K_F\col \Ker F\ra \A$, which maps $(A,b)$ to $A$ and
		$f\col (A,b)\ra (A',b')$ to $f\col A\ra A'$.  This is a symmetric monoidal functor,
		with $\varphi^{K_F}_{(A,b),(A',b')}\eqdef 1_{A\tens A'}$ and $\varphi_0^{K_F}\eqdef 1_I$.
		And the monoidal natural transformation
		$\kappa_F\col FK_F\Ra 0$ is defined by $(\kappa_F)_{(A,b)}\eqdef b \col  FA\ra I$.
	\end{df}
	
	We need symmetry to construct the cokernel in the following way.
	
	\begin{df}\label{dfdesccoker}\index{cokernel!in 2-sgp@in $\CGS$}
		The cokernel $\Coker F$ of $F\col \A\ra\B$ in $\CGS$ is defined in the following way.
		\begin{itemize}
			\item {\it Objects.} These are the objects of $\B$.
			\item {\it Arrows.} A morphism $B_0\ra B_1$ is a pair $(A,g)$,
				where $A\col \A$ and $g\col B_0\ra FA\tens B_1$.
			\item {\it Composition.} The composite of $(A_0,g_0)\col B_0\ra B_1$
				and $(A_1,g_1)\col B_1\ra B_2$ is
				$(A_0\tens A_1,g')$, where $g'$
				is the composite
				\begin{multline}\stepcounter{eqnum}
					B_0\overset{g_0}\longrightarrow FA_0\tens B_1\xrightarrow{1\tens g_1}
					FA_0\tens (FA_1\tens B_2)\\ \simeq (FA_0\tens FA_1)\tens B_2
					 \xrightarrow{\varphi\tens 1} F(A_0\tens A_1)\tens B_2.
				\end{multline}
			\item {\it Equality.} Two arrows $(A,g),(A',g')\col B_0\ra B_1$
				are equal if there exists $a\col A\ra A'$ such that the following diagram
				commutes.
				\begin{xym}\xymatrix@=40pt{
					B_0\ar[r]^-{f}\ar[dr]_-{f'}
					&FA\tens B_1\ar[d]^{Fa\tens 1}
					\\ &FA'\tens B_1
				}\end{xym}
			\item {\it Product.} The product $B\tens B'$ is defined as in 
				$\B$.  If we have $(A,g)\col B_0\ra B_1$ and $(A',g')\col B'_0\ra B'_1$,
				the product $(A,g)\tens (A',g')$ is defined by $(A\tens A', g'')$,
				where $g''$ is the following composite.
				\begin{multline}\stepcounter{eqnum}
					B_0\tens B'_0\xrightarrow{g\tens g'}
					(FA\tens B_1)\tens (FA'\tens B'_1)\overset{\bar{c}}\longrightarrow
					(FA\tens FA')\tens (B_1\tens B'_1)\\ \xrightarrow{\varphi\tens 1}
					F(A\tens A')\tens (B_1\tens B'_1)
				\end{multline}
			\item {\it Unit.} The unit is the unit $I$ of $\B$.
			\item {\it Natural transformations of the symmetric monoidal structure.}
				They are given by the object $I$ and the arrow
				$l$ composed with the corresponding transformations in $\B$.
		\end{itemize}
		The cokernel is equipped with a functor $Q_F\col \B\ra\Coker F$, which maps
		$B$ to $B$ and an arrow $g\col B_0\ra B_1$ to $(I,\bar{g})$, where
		$\bar{g}$ is the composite $B_0\overset{g}\longrightarrow B_1\overset{l_{B_1}^{-1}}
		\longrightarrow I\tens B_1\equiv FI\tens B_1$.
		
		The monoidal natural transformation $\zeta_F\col Q_F F\Ra 0$ is defined at
		$A\col \A$ by $\zeta_A \eqdef (A,r_{FA}^{-1})$.
	\end{df}
	
		\begin{proof}
			We only give the construction of the factorisation through the cokernel.
			Let be $G\col \B\ra\cat{Y}$ and $\gamma\col GF\Ra 0$.
			We define a functor $H\col \Coker F\ra\cat{Y}$ by setting
			$H(B)\eqdef G(B)$ and, if $(A,g)\col B_0\ra B_1$, by setting $H(A,g)$ equal to the
			following composite.\qedhere
			\begin{eqn}
				GB_0\overset{Gg}\longrightarrow G(FA\tens B_1)\xrightarrow{\varphi^{-1}}
				GFA\tens GB_1\xrightarrow{\gamma_A\tens 1} I\tens GB_1
				\xrightarrow{l_{GB_1}}GB_1
			\end{eqn}
		\end{proof}
	
	We can check that the kernel of $0\ra\A$ is $\Omega\A = (\pi_1\A)\dis$
	(defined above), and that the cokernel of $\A\ra 0$ is $\Sigma\A = (\pi_0\A)\con$.
	
	As Proposition \ref{kerpep} tells us, we can define the pip of $F$
	as $\Omega\Ker F$, i.e.
	\begin{eqn}\label{eqdescpep}\index{pip!in 2-sgp@in $\CGS$}
		\Pip F \eqdef (\pi_1\Ker f)\dis.
	\end{eqn}
	Thus its objects are the arrows $a\col I\ra I$ in $\A$ such that
	$Fa=1_{FI}$; the only arrows are the identities.  The product is the composition
	in $\A$.  It is equipped with a natural transformation
	$\pi_F\col 0\Ra 0\col \Pip_F\ra\A$, whose component at $a$ is $a$ itself.
	
	Dually, the copip of $F$ is $\Sigma\Coker F$, i.e.
	\begin{eqn}\label{eqdesccopep}\index{copip!in 2-sgp@in $\CGS$}
		\Copip F \eqdef (\pi_0\Coker f)\con.
	\end{eqn}
	So it has a unique object $I$ and the arrows from $I$ to $I$ are the objects $B\col \B$,
	two such objects $B,B'$ being equal if there exist $A\col \A$ and $b\col B\ra FA\tens B'$.
	The natural transformation $\rho_F\col 0\Ra 0\col \B\ra\Copip F$ 
	is defined at $B$ by $B$ itself.
	
	It remains to describe the root and the coroot of a 2-arrow $\alpha
	\col 0\Ra 0\col \A\ra\B$.
	Let us first describe elementarily the functors $\bar{\alpha}\col\Sigma\A\ra\B$ and
	$\tilde{\alpha}\col\A\ra\Omega \B$ corresponding to it (we follow the notations
	of diagram \ref{diagnotationsomegsigm}).  
	\begin{enumerate}
		\item The functor $\bar{\alpha}\col(\pi_0\A)\con\ra\B$ maps the unique object $I$
			to $I$ and an object $A\col\A$ (seen as an arrow $I\ra I$ in $(\pi_0\A)\con$)
			 to $\alpha_A\col I\ra I$ in $\B$.
		\item The functor $\tilde{\alpha}\col\A\ra(\pi_1\B)\dis$ maps an object
			$A$ to the arrow $\alpha_A$ seen as an object of $(\pi_1\B)\dis$.
	\end{enumerate}
	
	Then, by the dual of Proposition \ref{coraccoker}, the root of $\alpha$%
	\index{root!in 2-sgp@in $\CGS$} 
	is the kernel of $\tilde{\alpha}$, in other words its objects are the pairs
	$(A,b)$ where $A\col\A$ and $b\col \alpha_A\ra 1_I$ (i.e.\ $\alpha_A=1_I$
	in $\pi_1\B$) and the arrows $(A,b)\ra (A',b')$ are the arrows
	$a\col A\ra A'$ such that $b'\circ Fa =b$, which is always true
	since $(\pi_1\B)\dis$ is discrete.  Thus we get the following simplified description.
	
	\begin{pon}
		In $\CGS$, the root of $\alpha\col 0\Ra 0\col \A\ra\B$ is $\Root \alpha $,
		the full sub-2-group of $\A$ whose 
		objects are the objects $A\col \A$ such that $\alpha_A=1_I$.
	\end{pon}
		
	In the same way, by Proposition \ref{coraccoker}, the coroot of $\alpha $%
	\index{coroot!in 2-sgp@in $\CGS$} is the cokernel
	of $\bar{\alpha}$, in other words its objects are those of $\B$, an arrow $B_0\ra B_1$
	is a pair $(A,g)$ where $A\col(\pi_0\A)\con$ (so $A$ can only be $I$)
	and $g\col B_0\ra A\tens B_1$.  The equality between arrows is defined using $\alpha$.
	 By simplifying this description, we get
	the following proposition.
	
	\begin{pon}\label{descorac}
		In $\CGS$, the coroot of $\alpha\col 0\Ra 0\col \A\ra\B$ can be described in the
		following way: the objects, the arrows and the tensor of $\Coroot \alpha $
		are those of $\B$; two arrows $g,g'\col B_0\ra B_1$ are equal
		if there exists an object $A\col\A$ such that the following diagram commutes.
		\begin{xym}\xymatrix@=40pt{
			B_0\ar[d]_g
			& B_0\tens I\ar[l]_{r}\ar[d]^{g'\tens \alpha_A}
			\\ B_1
			&B_1\tens I\ar[l]^r
		}\end{xym}
	\end{pon}
	
\subsection{(Co)faithful, fully (co)faithful and full arrows in $\CGS$}\label{sectcarcflechcgs}

In this subsection, we will check, on the one hand, that the faithful, fully faithful and full arrows in $\CGS$ are the symmetric monoidal functors which have the usual properties with the same name (this justifies the terminology) and, on the other hand,
that the (fully) cofaithful arrows are the (full and) surjective symmetric monoidal functors. In each case, the method will be the same: we will use the characterisation of these kinds of arrows in terms of the triviality of the kernel, cokernel, pip or copip.

In parallel, we will check that in $\CGS$ every (fully) 0-faithful arrow is (fully) faithful and that every (fully) 0-cofaithful arrow is (fully) cofaithful. Therefore, this will be the case in all $\CGS$-categories .

The equivalence between (fully) faithful arrows in the $\Gpd$-categorical sense and in the elementary sense is proved in \cite{Kasangian2000a}.

	\begin{pon}\label{caracfidcgs}\index{faithful arrow!characterisation in $\CGS$}
		Let be $F\col \A\ra\B$ in $\CGS$.  The following properties are equivalent:
		\begin{enumerate}
			\item $F$ is faithful (in the sense of Definition \ref{defsortfleches});
			\item $F$ is 0-faithful (in the sense of Definition \ref{caraczfid});
			\item $F$ is 0-faithful (in the elementary sense of Proposition \ref{fidsuro});
			\item $F$ is faithful (in the elementary sense).
		\end{enumerate}
	\end{pon}
	
		\begin{proof}
			{\it 1 $\Rightarrow$ 2. }This is obvious.
			
			{\it 2 $\Rightarrow$ 3. }Let us assume that $F$ is 0-faithful as an arrow
			in a $\Gpdp$-category.  By Proposition \ref{pepclaszfid},
			$\pi_F = 1_0$, where $\pi_F\col 0\Ra 0\col \Pip F\ra \A$ is the pip of $F$.
			By the description of the pip given above (equation \ref{eqdescpep}),
			this means that, for every $a\col I\ra I$ such that $Fa=1_I$,
			$a = (\pi_F)_a = 1_I$.  So $F$ is 0-faithful in the elementary sense.
						
			{\it 3 $\Rightarrow$ 4. }Let be $f\col A\ra A$ in $\A$ such that $Ff=1_{FA}$.
			We define $\hat{f}$ to be equal to the composite
			\begin{eqn}
				I\overset{\eta_A}\longrightarrow A^*\tens A\xrightarrow{1_{A^*}\tens f}
				A^*\tens A\overset{\eta_A^{-1}}\longrightarrow I.
			\end{eqn}
			Then $F\hat{f}=F\eta_A^{-1}\circ\varphi^{-1}_{A^*A}\circ 
			(1_{FA^*}\tens Ff)\circ	\varphi_{A^*A}\circ F\eta_A = 1_I$, 
			because $Ff=1_{FA}$. So, since $F$ is 0-faithful (in the elementary sense), $\hat{f}=1_I$.
			Finally, since $f = r_A\circ (1_A\tens \hat{f})\circ r_A^{-1}$ (by the triangular
			identities that $\eta_A$ and $\varepsilon_A\col A\tens A^*\ra I$ satisfy),
			$f=1_A$.
			
			{\it 4 $\Rightarrow$ 1. }Let be $\cat{X}$, $G,H\col \cat{X}\ra \A$
			and $\alpha,\alpha'\col G\Ra H$ in $\CGS$ such that $F\alpha= F\alpha'$.
			For every $X\col\cat{X}$, $F\alpha_X=F\alpha'_X$ and, since $F$
			is faithful in the elementary sense, $\alpha_X=\alpha'_X$. So $\alpha=\alpha'$.
		\end{proof}

	Since faithful and 0-faithful arrows are defined through the representables, we get
	the following proposition.
	\begin{pon}\label{caracfidcgscat}
		Let $\C$ be a $\CGS$-category and $f\col \flc$.  Then $f$ is faithful
		if and only if $f$ is 0-faithful.
	\end{pon}

	Let us turn now to fully (0-)faithful arrows.

	\begin{pon}\label{caracplfidcgs}\index{fully faithful arrow!characterisation in $\CGS$}
		Let be $F\col \A\ra\B$ in $\CGS$.  The following properties are equivalent:
		\begin{enumerate}
			\item $F$ is fully faithful (in the sense of Definition \ref{defsortfleches});
			\item $F$ is fully 0-faithful (in the sense of Definition \ref{caracplzfid});
			\item $F$ is fully 0-faithful
				(in the elementary sense of Proposition \ref{carplfidsurzer});
			\item $F$ is fully faithful (in the elementary sense).
		\end{enumerate}
	\end{pon}
	
		\begin{proof}
			{\it 1 $\Rightarrow$ 2. }This is obvious.
			
			{\it 2 $\Rightarrow$ 3. }Let us assume that $F$ is fully 0-faithful as
			an arrow in a $\Gpdp$-category.  By Proposition \ref{claspropker},
			there exists a monoidal natural transformation $\omega\col K_F\Ra 0$ such that
			$F\omega = \kappa_F$, where $(K_F,\kappa_F)$ is the kernel of $F$.  By
			the description of Definition \ref{dfdescker}, this means that, for
			every $b\col FA\ra I$, there is an arrow $\omega_{(A,b)}\col A\ra I$
			such that $F\omega_{(A,b)}=b$.  This proves condition 1(a) of
			Proposition \ref{carplfidsurzer}.
			
			To prove condition 1(b), let $a\col I\ra I$ be an arrow in $\A$
			such that $Fa=1_I$.  Then $a$ is an arrow $(I,1_I)\ra (I,1_I)$ in $\Ker F$.
			The naturality of $\omega$ tells us then
			that $\omega_{(I,1_I)}\circ a = \omega_{(I,1_I)}$ and, so, that $a=1_I$.
									
			{\it 3 $\Rightarrow$ 4. }To prove that $F$ is full, let be
			$b\col FA'\ra FA$ in $\B$.  We set $\hat{b}$ equal to the composite
			\begin{eqn}
				F(A'\tens A^*)\overset{\varphi^{-1}}\longrightarrow FA'\tens FA^*
				\xrightarrow{b\tens 1_{FA^*}}FA\tens FA^*\overset{\varphi}\longrightarrow
				F(A\tens A^*)\overset{F\varepsilon_A}\longrightarrow FI\equiv I.
			\end{eqn}
			As $F$ is fully 0-faithful in the elementary sense, there exists an
			arrow $\hat{a}\col A'\tens A^*\ra I$ such that $\hat{b}= F\hat{a}$.
			We set then $a$ equal to the composite
			\begin{eqn}
				A'\overset{r^{-1}_{A'}}\longrightarrow A'\tens I
				\xrightarrow{1_{A'}\tens\eta_A}A'\tens(A^*\tens A)\overset{a}\longrightarrow
				(A'\tens A^*)\tens A\xrightarrow{\hat{a}\tens 1_A}I\tens A\overset{l_A}
				\longrightarrow A.
			\end{eqn}
			We check that $b=Fa$ by using the axioms of monoidal functor and the
			triangular identities.
			
			Moreover, $F$ is faithful by the implication {\it 3 $\Rightarrow$ 4}
			of Proposition \ref{caracfidcgs}.
			
			{\it 4 $\Rightarrow$ 1. }Let be $\cat{X}$, $G,H\col \cat{X}\ra \A$
			and $\beta\col FG\Ra FH$ in $\CGS$.  Since $F$ is full, 
			for every $X\col\cat{X}$, there exists $\alpha_X\col GX\ra HX$
			such that $\beta_X=F\alpha_X$. We check that to give $\alpha_X$ for every $X$
			defines a monoidal natural transformation,
			by using faithfulness of $F$.  If $\alpha'\col G\Ra H$ is another
			monoidal natural transformation such that $\beta=F\alpha'$,
			then, for every $X\col\cat{X}$, $F\alpha'_X=F\alpha_X$ and thus $\alpha'_X
			=\alpha_X$, because $F$ is faithful. So $\alpha'=\alpha$.
		\end{proof}

	Again, the equivalence between fully faithful and fully 0-faithful arrows
	generalises to $\CGS$-categories .

	\begin{pon}\label{caracplfidcgscat}
		Let $\C$ be a $\CGS$-category and $f\col \flc$.  Then $f$ is fully faithful
		if and only if $f$ is fully 0-faithful.
	\end{pon}

The equivalence between cofaithful and surjective has been proved in \cite{Dupont2003a}.
 
 \begin{pon}\label{caraccofidcgs}\index{cofaithful arrow!characterisation in $\CGS$}
 	Let be $F\col \A\ra\B$ in $\CGS$.  The following properties are equivalent:
	\begin{enumerate}
		\item $F$ is cofaithful;
		\item $F$ is 0-cofaithful;
		\item $F$ is surjective.
	\end{enumerate}
 \end{pon}
 
 	\begin{proof}
		{\it 1 $\Leftrightarrow$ 2. }We have proved above
		(Proposition \ref{caracfidcgscat}) that the faithful and 0-faithful arrows coincide in
		$\CGS$-categories, so in particular this is the case in $\CGS\op$.

		{\it 2 $\Rightarrow$ 3. }Let us assume that $F$ is 0-cofaithful. By the dual of Proposition \ref{pepclaszfid},
		$\rho_F = 1_0$, where $\rho_F\col 0\Ra 0\col \B\ra\Copip F$ is the copip of $F$.
		By the description of the copip given above (equation \ref{eqdesccopep}),
		this means that, for every $B\col \B$, $B = I$ in $\Copip F$,
		i.e.\ that there exist $A\col \A$ and $b\col B\ra FA\tens I$.
		So $F$ is surjective.

		{\it 3 $\Rightarrow$ 2. }Let be $\gamma\col 0\Ra 0\col \B\ra\cat{Y}$ in $\CGS$
		such that $\gamma F=1_0$.  For every $B\col\B$, there exist $A\col\A$
		and $b\col FA\ra B$.  Since $\gamma$ is natural, we have
		$\gamma_B=\gamma_{FA}=1_I$.
	\end{proof}

The equivalence between fully cofaithful and full and surjective has been proved in \cite{Kasangian2000a}.

\begin{pon}\label{caracplcofidcgs}\index{fully cofaithful arrow!characterisation in $\CGS$}
	Let be $F\col \A\ra\B$ in $\CGS$. The following properties are equivalent:
	\begin{enumerate}
		\item $F$ is fully cofaithful;
		\item $F$ is fully 0-cofaithful;
		\item $F$ is full and surjective.
	\end{enumerate}
\end{pon}

	\begin{proof}
		{\it 1 $\Leftrightarrow$ 2. }It suffices to apply Proposition
		\ref{caracplfidcgscat} to $\CGS\op$.
		
		{\it 2 $\Rightarrow$ 3. }Let us assume that $F$ is fully 0-cofaithful.  By the dual of
		Proposition \ref{claspropker},
			there exists a monoidal natural transformation $\omega\col Q_F\Ra 0$ such that
			$\omega F= \zeta_F$, where $(Q_F,\zeta_F)$ is the cokernel of $F$.  By
			the description of Definition \ref{dfdesccoker}, this means that,
		for every $B\col \B$, there exists an arrow $\omega_B = (A_B,f_B)\col  B\ra I$
		in $\Coker F$, where $A_B\col \A$ and $f_B\col B\ra FA_B\tens I$.
		So $F$ is surjective.
		
		Next, let be $g\col FA\ra FA'$ in $\B$.  By the naturality of $\omega$,
		we have $\omega_{FA'}\circ Q_F g = \omega_{FA}$ in $\Coker F$. But
		$\omega F =\zeta_F$, thus this equality becomes
		$(A',r_{FA'}^{-1}g) = (A,r_{FA}^{-1})$ in $\Coker F$. Therefore, by the definition
		of equality between the arrows of $\Coker F$, there exists
		$f\col A\ra A'$ such that $g=Ff$. So $F$ is full.
		
		{\it 3 $\Rightarrow$ 2. }Let be $V,V'\col \B\ra\cat{Y}$ and 
		$\gamma\col VF\Ra V'F$ in $\CGS$.  Let be $B\col\B$.  Since $F$ is surjective,
		there exist $A_B\col\A$ and $f_B\col B\to FA_B$ in $\B$.  We set $\alpha_B$
		equal to the composite $VB\xrightarrow{Vf_B}VFA_B\xrightarrow{\gamma_{A_B}}
		V'FA_B\xrightarrow{V'f^{-1}_B}V'B$.  To prove the naturality of $\alpha$,
		let be $b\col B\to B'$ in $\B$.  Since $F$ is full, there exists
		$a_b\col A_B\to A_{B'}$ such that $Fa_b$ is equal to the composite
		$FA_B\xrightarrow{f^{-1}_B}B\overset{b}\longrightarrow B'\xrightarrow{f_{B'}}FA_{B'}$.
		Then $V'b\circ \alpha_B = \alpha_{B'}\circ Vb$ thanks to the naturality of $\gamma$.
		Finally, to prove $\gamma=\alpha F$, let be $A\col\A$. Since $F$ is full, 
		there exists $g_{A}\col A\to A_{FA}$ such that
		$Fg_A=f_{FA}\col FA\to FA_{FA}$; thus $\alpha_{FA}=\gamma_A$, thanks to the
		naturality of $\gamma$.
	\end{proof}

Finally, we prove that the full arrows are exactly the full functors.  Let us recall that 
$\mu_F\eqdef \zeta_F K_F \circ Q_F\kappa_F^{-1}$.

\begin{pon}\label{caracfullcgs}\index{full arrow!characterisation in $\CGS$}
	Let be $F\col \A\ra\B$ in $\CGS$. The following properties are equivalent:
	\begin{enumerate}
		\item $F$ is full (in the sense of Definition \ref{deffullgpdcat});
		\item $\mu_F = 1_0$ ($\zeta_F K_F = Q_F\kappa_F$);
		\item $F$ is full (in the elementary sense).
	\end{enumerate}
\end{pon}

	\begin{proof}
		{\it 1 $\Rightarrow$ 2. }Condition 2 is a special case of condition 1.

		{\it 2 $\Rightarrow$ 3. }Elementarily, condition 2 means that,
		for every $b\col FA\ra I$ (i.e.\ $(A,b)\col \Ker F$),
		$(Q_F\kappa_F)_{(A,b)} = Q_F(b) = (b,I)$
		is equal in $\Coker F$ to $(\zeta_F)_{K_F(A,b)} = (\zeta_F)_A = (1_{FA},A)$.
		This means, by the definition of equality between arrows in
		$\Coker F$, that there exists $a\col A\ra I$ such that
		$Fa \circ 1_{FA}=b$.  To sum up, for every $b\col FA\ra I$, there exists $a\col A\ra I$
		such that $b=Fa$. We deduce that $F$ is full by following the reasoning
		used to prove the implication {\it 3 $\Rightarrow$ 4} of Proposition
		\ref{caracplfidcgs}.
		
		{\it 3 $\Rightarrow$ 1. }Let be $\cat{X}$ and $\cat{Y}$,
		$U,U'\col\cat{X}\ra\A$ and $V,V'\col\B\ra\cat{Y}$, $\alpha\col FU\Ra FU'$
		 and $\beta\col VF\Ra V'F$ in $\CGS$.  Since $F$ is full in the elementary sense,
		 for every $X\col\cat{X}$, there exists $\gamma\col UX\ra U'X$ such that $\alpha_X
		 =F\gamma_X$.  Then, for every $X\col\cat{X}$, we have:
		 \begin{eqn}
		 	V'\alpha_X\circ\beta_{UX} = V'F\gamma_X\circ\beta_{UX}
			= \beta_{U'X}\circ VF\gamma_X = \beta_{U'X}\circ V\alpha_X,
		 \end{eqn}
		 thus $F$ is full in the sense of Definition \ref{deffullgpdcat}.
	\end{proof}

\subsection{The factorisations $\spf$ and $\psf$ on $\CGS$}\label{sectdeffactcgs}

	The symmetric monoidal functors between symmetric 2-groups factor
	either as a surjective functor followed by a fully faithful functor, or
	as a full and surjective functor followed by a faithful functor \cite{Kasangian2000a}.
	  We first describe these factorisations in the $\Gpd$-category of
	groupoids $\Gpd$.
		
	\begin{pon}\label{descrfacpl}
		Every functor $F\col \A\ra\B$ in $\Gpd$ factors
		as the following composite, where $\hat{E}_F$ is surjective, $\hat{\Omega}_F$
		is an equivalence, and $\hat{M}_F$ is full and faithful:
		\begin{eqn}
			\A\overset{\hat{E}_F}\longrightarrow\mathrm{Im}_{\mathrm{pl}}^1\, F\overset{\hat{\Omega}_F}
			\longrightarrow\mathrm{Im}_{\mathrm{pl}}^2\, F\overset{\hat{M}_F} \longrightarrow\B.
		\end{eqn}
		We call the groupoid $\mathrm{Im}_{\mathrm{pl}}^1\, F\simeq\mathrm{Im}_{\mathrm{pl}}^2\, F$
		the \emph{full image} of $F$.\index{full image!in 2-sgp@in $\CGS$}
	\end{pon}
	
		\begin{proof}
			Let us first describe the two variants of the full image of $F$.
			
			The groupoid $\mathrm{Im}_{\mathrm{pl}}^1\,F$ is described in the following way.
			\begin{itemize}
				\item {\it Objects.} These are the objects of $\A$.
				\item {\it Arrows.} $(\mathrm{Im}_{\mathrm{pl}}^1\,F)(A,A')=\B(FA,FA')$.
					The composition and the identities are those of $\B$.
			\end{itemize}
			
			The groupoid $\mathrm{Im}_{\mathrm{pl}}^2\,F$ is described in the following
			way\footnote{It is clearly equivalent to the full subgroupoid
			of $\B$ whose objects are the objects isomorphic to $FA$ for some $A\col \A$. 
			The advantage
			of the description given here is that $\hat{\Omega}^{-1}_F$ is
			defined constructively.}.
			\begin{itemize}
				\item {\it Objects.} These are the triples $(A,\varphi,B)$,
					where $A\col \A$, $B\col \B$ and $\varphi\col FA\ra B$.
				\item {\it Arrows.} $(\mathrm{Im}_{\mathrm{pl}}^2\,F)((A,\varphi,B),(A',\varphi',B'))
					=\B(B,B')$. The composition and the identities are those of $\B$.
			\end{itemize}
	
			Then we define the functor $\hat{\Omega}_F$.  It maps $A$ to
			$(A,1_{FA},FA)$ and $g\col FA\ra FA'$ to $g$.  We can also define
			$\hat{\Omega}^{-1}_F\col \mathrm{Im}_{\mathrm{pl}}^2\,F\ra\mathrm{Im}_{\mathrm{pl}}^1\,F$, which maps
			$(A,\varphi,B)$ to $A$ and $g\col B\ra B'$ to the composite
			\begin{eqn}
				FA\overset{\varphi} \longrightarrow B\overset{g} \longrightarrow
				 B'\overset{\varphi'^{-1}} \longrightarrow FA'.
			\end{eqn}
			It is then obvious that $\hat{\Omega}^{-1}_F\circ\hat{\Omega}_F\equiv 1$ and
			that $\hat{\Omega}_F\circ\hat{\Omega}^{-1}_F\simeq 1$.
			So $\hat{\Omega}_F$ is an equivalence.
			
			Finally, we define the surjective functor $\hat{E}_F$,
			which maps $A$ to $A$ and $f\col A\ra A'$ to $Ff$,
			and the full and faithful functor $\hat{M}_F$,
			which maps $(A,\varphi,B)$ to $B$ and is the identity on arrows.
			Then we have $F\equiv \hat{M}_F\circ\hat{\Omega}_F\circ \hat{E}_F$.
		\end{proof}

	\begin{pon}\label{factimfid}
		Every functor $F\col \A\ra\B$ in $\Gpd$ factors
		as the following composite, where $E_F$ is full and surjective, $\Omega_F$
		is an equivalence, and $M_F$ is faithful:
		\begin{eqn}
			\A\overset{E_F} \longrightarrow\im^1F\overset{\Omega_F} 
			\longrightarrow\im^2F\overset{M_F} \longrightarrow\B.
		\end{eqn}
		We call the groupoid $\im^1F\simeq\im^2F$
		the \emph{image} of $F$ (or \emph{faithful image} of $F$).
	\end{pon}
	
		\begin{proof}
			Let us first describe the two variants of the image of $F$.
			
			The groupoid $\im^1F$ is described in the following way\footnote{We can
			remark that $(\im^1F)(A,A')=\im^1F_{A,A'}$,
			the image of $F_{A,A'}\col \A(A,A')\ra\B(FA,FA')$ in $\Ens$.}.
			\begin{itemize}
				\item {\it Objects.} These are the objects of $\A$.
				\item {\it Arrows.} These are the arrows of $\A$. The composition
					and the identities are those of $\A$.
				\item {\it Equality.} Two arrows $f,f'\col A_0\ra A_1$ are
					equal ($f\simeq f'$) if $Ff=Ff'$.
			\end{itemize}
			
			The groupoid $\im^2F$ is described in the following way.
			\begin{itemize}
				\item {\it Objects.} These are the triples $(A,\varphi,B)$,
					where $A\col \A$, $B\col \B$ and $\varphi\col FA\ra B$.
				\item {\it Arrows.} The arrows of $(A,\varphi,B)$ to
					$(A',\varphi',B')$ are the pairs $(f,g)$, where $f\col A\ra A'$
					and $g\col B\ra B'$, such that the following diagram commutes.
					\begin{xym}\xymatrix@=40pt{
						FA\ar[r]^{\varphi}\ar[d]_{Ff}
						&B\ar[d]^g
						\\ FA'\ar[r]_{\varphi'}
						&B'
					}\end{xym}
					The identity on $(A,\varphi,B)$ is $(1_A,1_B)$. Composition is defined componentwise. The inverse of $(f,g)$
					is $(f^{-1},g^{-1})$.
				\item {\it Equality.} Two arrows $(f,g), (f',g')\col (A,\varphi,B)
					\ra (A',\varphi',B')$ are equal if $g=g'$ (or, equivalently, $Ff=Ff'$).
			\end{itemize}
			
			Then we define the functor $\Omega_F$.  It maps $A$ to
			$(A,1_{FA},FA)$ and $f\col A\ra A'$ to $(f,Ff)$.  We can also define
			$\Omega^{-1}_F\col \im^2F\ra\im^1F$, which maps
			$(A,\varphi,B)$ to $A$ and $(f,g)$ to $f$; it is well defined, because
			$(f,g)=(f',g')$ in $\im^2F$ if and only if $Ff=Ff'$,
			i.e.\ if and only if $f\simeq f'$ in $\im^1F$.
			It is then obvious that $\Omega^{-1}_F\circ\Omega_F\equiv 1$.
			Besides, $\Omega_F\circ\Omega^{-1}_F\simeq 1$, because the following diagram
			gives an isomorphism $(A,1_{FA},FA)\simeq (A,\varphi,B)$.
			\begin{xym}\xymatrix@=40pt{
				FA\ar@{=}[r]\ar@{=}[d]
				&FA\ar[d]^{\varphi}
				\\ FA\ar[r]_{\varphi}
				&B
			}\end{xym}
			So $\Omega_F$ is an equivalence.
			
			Finally, we define the full and surjective functor $E_F$,
			which maps $A$ to $A$ and $f\col A\ra A'$ to $f$.
			And we define $M_F$, which maps $(A,\varphi,B)$ to $B$ and $(f,g)$
			to $g$.  This functor is faithful, because, by definition, $(f,g)=(f',g')$
			in $\im^2F$ if and only if $g=g'$ in $\B$.
			It is then clear that $F\equiv M_F\circ\Omega_F\circ E_F$.
		\end{proof}

	Let us turn now to $\CGS$.  We review the constructions of the following propositions
	and add to them the structures of symmetric 2-group,
	symmetric monoidal functor, and monoidal natural transformation.

	\begin{pon}\label{precedente}
		Every symmetric monoidal functor $F\col \A\ra\B$ in $\CGS$ factors
		as the following composite, where $\hat{E}_F$ is surjective, $\hat{\Omega}_F$
		is an equivalence, and $\hat{M}_F$ is full and faithful:
		\begin{eqn}
			\A\overset{\hat{E}_F} \longrightarrow\mathrm{Im}_{\mathrm{pl}}^1\,F\overset{\hat{\Omega}_F}
			\longrightarrow\mathrm{Im}_{\mathrm{pl}}^2\,F\overset{\hat{M}_F} \longrightarrow\B.
		\end{eqn}
	\end{pon}
	
		\begin{proof}
			We define the symmetric 2-group structure on
			$\mathrm{Im}_{\mathrm{pl}}^1\,F$.  If $A,A'\col \mathrm{Im}_{\mathrm{pl}}^1\,F$, then $A\tens A'$
			is defined as in $\A$; if we have $g\col FA_0\ra FA_1$ and $g'\col FA'_0\ra FA'_1$,
			which are arrows in $\mathrm{Im}_{\mathrm{pl}}^1\,F$,
			respectively $A_0\ra A_1$ and $A'_0\ra A'_1$,
			then $g\tens g'\col A_0\tens A'_0\ra A_1\tens A'_1$
			is defined by the following composite:
			\begin{eqn}
				F(A_0\tens A'_0)\xrightarrow{\varphi^{-1}}FA_0\tens FA'_0
				\xrightarrow{g\tens g'}FA_1\tens FA'_1\overset{\varphi} \longrightarrow
				F(A_1\tens A'_1).
			\end{eqn}
			It is easy to check that $-\tens -$ is a functor.
			We take as unit the object $I$ of $\A$.
			
			The associativity, neutrality and symmetry natural transformations
			are defined as the image by $F$ of the corresponding transformations
			in $\A$: $\tilde{a}_{A,A',A''}\eqdef F(a_{A,A',A''})$,
			$\tilde{l}_{A}\eqdef F(l_A)$, $\tilde{r}_A\eqdef F(r_A)$ and $\tilde{c}_{A,A'}\eqdef
			F(c_{A,A'})$.  Their naturality follows from the compatibility of $\varphi^F$
			with these transformations (because $F$ is monoidal symmetric)
			and from the naturality of the corresponding transformations in $\B$.
			
			The axioms of symmetric monoidal groupoid follow from those of
			$\A$ by applying to them the functor $F$.  Every object $A$ has an inverse,
			which is nothing else than the inverse $A^*$ in $\A$; we can define 
			as above $\tilde{\varepsilon}_A\eqdef F(\varepsilon_A)\col A\tens A^*\ra I$.
			
			Next, let us define the structure of symmetric 2-group on
			$\mathrm{Im}_{\mathrm{pl}}^2\,F$.  The tensor product
			is defined by $(A,\gamma,B)\tens (A',\gamma',B')\eqdef(A\tens A',\gamma'',
			B\tens B')$, where $\gamma''$ is the composite
			\begin{eqn}
				F(A\tens A')\xrightarrow{\varphi^{-1}}FA\tens FA'\xrightarrow{\gamma
				\tens\gamma'}B\tens B'.
			\end{eqn}
			On arrows, $g\tens g'$ is defined as in $\B$, and this 
			defines a functor because $-\tens-$ is a functor in $\B$.
			The unit is $(I,1_I,I)$ and the natural transformations
			of symmetric monoidal groupoid are those of $\B$ (for example,
			$l_{(A,\gamma,B)}\eqdef l_B$), and are natural because they are in $\B$.
			The axioms hold because they do in $\B$.
			The inverse of $(A,\gamma,B)$ is $(A^*,\tilde{\gamma},B^*)$, where
			$\tilde{\gamma}$ is the composite $F(A^*)\simeq (FA)^*
			\xrightarrow{(\gamma^*)^{-1}} B^*$; we define $\varepsilon_{(A,\gamma,B)}
			\eqdef\varepsilon_B$.
			
			The monoidal structure of $\hat{E}_F$ is given by
			$\varphi^{\hat{E}_F}_{A,A'}\eqdef 1_{F(A\tens A')}$, which is natural
			 because $\varphi^F$ is natural. The axioms
			of symmetric monoidal functor are trivially true.
			For $\hat{\Omega}_F$, we set $\varphi^{\hat{\Omega}_F}_{A,A'}
			\eqdef\varphi^F_{A,A'}$.
			The axioms
			of symmetric monoidal functor follow from these axioms for $F$.
			For $\hat{M}_F$, $\varphi^{\hat{M}_F}$ is the identity
			and the naturality and the axioms are trivially true.
		\end{proof}

	\begin{pon}\label{factimfidcgs}
		Every symmetric monoidal functor $F\col \A\ra\B$ in $\CGS$ factors
		as the following composite, where $E_F$ is full and surjective, $\Omega_F$
		is an equivalence, and $M_F$ is faithful:
		\begin{eqn}
			\A\overset{E_F} \longrightarrow\im^1F\overset{\Omega_F} \longrightarrow
			\im^2F\overset{M_F} \longrightarrow\B.
		\end{eqn}
	\end{pon}
	
		\begin{proof}
			The tensor product of $\im^1F$ is defined on objects and on
			arrows as in $\A$; it preserves equality between arrows
			thanks to the naturality of $\varphi^F$.  The natural transformations
			of symmetric monoidal groupoid are defined as in $\A$ and are
			natural and satisfy the axioms because they do in $\A$.  The inverse
			is defined as in $\A$, as well as $\varepsilon_A\col A\tens A^*\ra I$.
			
			In $\im^2F$, the tensor product
			of $(A,\gamma,B)$ and $(A',\gamma',B')$ is defined as
			for $\mathrm{Im}_{\mathrm{pl}}^2\,F$ (see Proposition \ref{precedente}), whereas
			the tensor product of $(f,g)$ and $(f',g')$ is simply
			$(f\tens f',g\tens g')$, which is an arrow of $\im^2F$,
			thanks to the naturality of $\varphi^F$ and because $(f,g)$ and $(f',g')$
			are themselves morphisms.
			
			The natural transformations $a$, $l$, $r$, $c$ are each defined
			as being the pair of the corresponding natural transformations
			in $\A$ and in $\B$ (for example, $r_{(A,\gamma,B)}\eqdef(r_A,r_B)$);
			these are arrows in $\im^2F$, thanks to the compatibility
			of $\varphi^F$ with  this natural transformation, and to the naturality
			of the corresponding natural transformation in $\B$, and they
			are natural because they are in $\B$.  In the same way, the axioms
			of symmetric monoidal groupoid hold because they do in
			$\B$. The inverse of $(A,\gamma, B)$ is defined as in $\mathrm{Im}_{\mathrm{pl}}^2\,F$
			and $\varepsilon_{(A,\gamma,B)}=(\varepsilon_A,\varepsilon_B)$.
			
			For $\Omega_F$, the natural transformation $\varphi^{\Omega_F}_{A,A'}$
			is defined by $(1_{A\tens A'},\varphi^F_{A,A'})$. For
			$E_F$ and $M_F$, the definitions are as in the case of the full image
			(Proposition \ref{precedente}).
		\end{proof}

	We can check that $(\Surj,\PlFid)$ and $(\PlSurj,\Fid)$ are factorisation
	systems on $\CGS$ (see \cite{Kasangian2000a}).  Propositions
	\ref{caraccofidcgs} and \ref{caracplcofidcgs} tell us that these
	factorisation systems are coupled
	in the sense of Definition \ref{defsyfcouples}: let be $F\col\A\ra\B$ a symmetric
	monoidal functor between symmetric 2-groups; then:
	\begin{enumerate}
		\item $F$ is surjective if and only if, for every $\cat{Y}\col\CGS$,
			$-\circ F\col [\cat{Y},\B]\ra[\cat{Y},\A]$ is faithful;
		\item $F$ is full and surjective if and only if, for every $\cat{Y}\col\CGS$,
			$-\circ F\col [\cat{Y},\B]\ra[\cat{Y},\A]$ is full and faithful.
	\end{enumerate}

\subsection{$\CGS$ is 2-abelian}

	The goal of this subsection is to recall that in $\CGS$ the factorisations
	described in the previous subsection can be computed,
	for the full image, by taking the cokernel of the kernel
	\cite{Kasangian2000a} or the root of the copip \cite{Dupont2003a}
	and, for the faithful image, by taking the coroot of the pip \cite{Dupont2003a}
	or the kernel of the cokernel \cite{Kasangian2000a}.  This will prove
	that $\CGS$ is 2-Puppe-exact.
	Let us begin by the exactness properties on the “left” side of the two
	factorisations in $\CGS$.

	\begin{pon}
		For each symmetric monoidal functor $F\col \A\ra\B$ in $\CGS$,
		$\hat{E}_F\col\A\ra\mathrm{Im}_{\mathrm{pl}}^1\, F$ is the cokernel of the kernel of $F$.
	\end{pon}
	
		\begin{proof}
			First, there is a monoidal natural transformation
			$\delta\col \hat{E}_FK_F\Ra 0$, where $\delta_{(A,b)}\eqdef b\col FA\ra I\equiv FI$.
			We will prove that the functor $\Phi\col \Coker K_F\ra \mathrm{Im}_{\mathrm{pl}}^1\,F$
			induced by the universal property of the cokernel is an equivalence.
			
			Let us describe $\Coker K_F$.
			\begin{itemize}
				\item {\it Objects.} These are the objects of $\A$.
				\item {\it Arrows.} An arrow $A_0\ra A_1$ is given
					by $(N,b,f)$, where $N\col \A$, $b\col FN\ra I$ in $\B$
					and $f\col A_0\ra N\tens A_1$ in $\A$.
				\item {\it Equality.} Two arrows $(N,b,f)$ and $(N',b',f')$
					are equal if there exists $n\col N\ra N'$ such that the following diagrams
					commute.
					\begin{xyml}\label{eqcokker}\begin{gathered}\xymatrix@R=20pt@C=40pt{
						&N\tens A_1\ar[dd]^{n\tens 1}
						\\ A_0\ar[ur]^-{f}\ar[dr]_-{f'}
						\\ &N'\tens A_1
					}\end{gathered}\;\;\;\begin{gathered}\xymatrix@R=20pt@C=40pt{
						FN\ar[dd]_{Fn}\ar[dr]^-b
						\\ &I
						\\ FN'\ar[ur]_-{b'}
					}\end{gathered}\end{xyml}
			\end{itemize}
			The functor $\Phi$ maps $A$ to $A$, the arrow $(N,b,f)\col A_0\ra A_1$
			is mapped to $\Phi(N,b,f)\col A_0\ra A_1$ in $\mathrm{Im}_{\mathrm{pl}}^1\,F$, which
			is the following composite in $\B$:
			\begin{eqn}
				FA_0\overset{Ff} \longrightarrow F(N\tens A_1)\overset{\varphi^{-1}} \longrightarrow
				FN\tens FA_1\xrightarrow{b\tens 1}I\tens FA_1\overset{l_{FA_1}} \longrightarrow FA_1.
			\end{eqn}
			The functor $\Phi$ is of course surjective.  It remains to prove
			that it is full and faithful. Let $g\col FA_0\ra FA_1$ be an arrow
			of $\mathrm{Im}_{\mathrm{pl}}^1\,F(A_0,A_1)$.  We set $N\eqdef A_0\tens A_1^*$,
			$b$ equal to the following composite: 
			\begin{eqn}
				F(A_0\tens A_1^*)\simeq FA_0\tens F(A_1^*)\xrightarrow{g\tens 1}
				FA_1\tens F(A_1^*)\simeq F(A_1\tens A_1^{*})
				\xrightarrow{F\varepsilon_{A_1}}FI\equiv I
			\end{eqn}
			and $f$ equal to the following composite: 
			\begin{eqn}
				A_0\simeq A_0\tens I\xrightarrow{1\tens\eta_{A_1}}
				A_0\tens(A_1^*\tens A_1)\simeq (A_0\tens A_1^*)\tens A_1.
			\end{eqn}
			Then $\Phi(N,b,f)=g$. So $\Phi$ is full.
			
			Finally, $\Phi$ is faithful, because if $\Phi(N,b,f)=\Phi(N',b',f')$,
			we set $n$ equal to the composite
			\begin{multline}\stepcounter{eqnum}
				N\simeq N\tens I\xrightarrow{1\tens\varepsilon^{-1}_{A_1}}N\tens (A_1\tens A_1^*)
				\simeq (N\tens A_1)\tens A_1^*\xrightarrow{f^{-1}\tens 1}
				A_0\tens A_1^* \\
				\xrightarrow{f'\tens 1} (N'\tens A_1)\tens A_1^*	\simeq
				N'\tens (A_1\tens A_1^*)\xrightarrow{1\tens\varepsilon_{A_1}}
				N'\tens I\simeq N'.
			\end{multline}
			Then $n$ satisfies equations \ref{eqcokker}
			and thus $(N,b,f)=(N',b',f')$.
		\end{proof}
	
	The following corollary has been proved in \cite{Kasangian2000a}.
		
	\begin{coro}
		In $\CGS$, every symmetric monoidal surjective functor is canonically
		the cokernel of its kernel.
	\end{coro}
	
	\begin{pon}
		For each symmetric monoidal functor $F\col \A\ra\B$ in $\CGS$,
		$E_F\col\A\ra\im^1 F$ is the coroot of the pip of $F$.
	\end{pon}
	
		\begin{proof}
			Let us describe the coroot of the pip of $F$.
			\begin{itemize}
				\item {\it Objects.} These are the objects of $\A$.
				\item {\it Arrows.} These are the arrows of $\A$.
				\item {\it Equality.} Two arrows $f,f'\col A_0\ra A_1$ are equal
					if there exists $a\col I\ra I$ in $\A$ such that $Fa=1_{FI}$
					and $f=r_{A_1}(f'\tens a) r^{-1}_{A_0}$.
			\end{itemize}
			On the one hand, if $f=f'$ in $\Coroot(\pi_F)$, 
			$Ff=r_{FA_1}(Ff'\tens Fa)r_{FA_0}^{-1}= Ff'$, because $Fa=1_I$ 
			and thus $f=f'$ in $\im^1F$.
			On the other hand, if $Ff=Ff'$, we set $a$ equal to the composite
			\begin{eqn}
				I\xrightarrow{\eta_{A_0}} A_0^*\tens A_0
				\xrightarrow{(f'^*)^{-1}\tens f}A_1^*\tens A_1
				\xrightarrow{\eta_{A_1}^{-1}} I.
			\end{eqn}
			Then $Fa=1_I$, because $Ff=Ff'$, and it is obvious that
			$f=r_{A_1}(f'\tens a) r^{-1}_{A_0}$.
			So the equality of $\Coroot(\pi_F)$ is equivalent
			to the equality of $\im^1F$, and $\Coroot(\pi_F)\equiv\im^1F$.
		\end{proof}

	The following corollary appears in \cite{Dupont2003a}.
	
	\begin{coro}
		In $\CGS$, every full and surjective functor is canonically the coroot
		of its pip.
	\end{coro}
	
	Let us prove now the dual properties, i.e.\ that the
	factorisations in $\CGS$ can also be constructed by taking a coquotient
	of a cokernel, namely the root of the copip for the full image, and the kernel
	of the cokernel for the faithful image.  See \cite{Kasangian2000a} and \cite{Dupont2003a}.
	
	\begin{pon}
		Let $F\col \A\ra\B$ be a symmetric monoidal functor in $\CGS$.
		Then $M_F\col\im^2 F\ra\B$ is the kernel of the cokernel of $F$.
	\end{pon}
	
		\begin{proof}
			Let us first describe $\Ker Q_F$.
			\begin{itemize}
				\item {\it Objects.} These are the triples $(B,A,b)$, where
					$B\col \B$, $A\col \A$ and $b\col B\ra FA\tens I$.
				\item {\it Arrows.} An arrow $(B,A,b)\ra (B',A',b')$
					is an arrow $g\col B\ra B'$ equipped with $f\col A\ra A'$
					such that the following diagram commutes.
					\begin{xym}\xymatrix@=40pt{
						B\ar[r]^-{b}\ar[d]_g
						&FA\tens I\ar[d]^{Ff\tens 1_I}
						\\ B'\ar[r]_-{b'} &FA'\tens I
					}\end{xym}
				\item {\it Equality.} $(g,f)=(g',f')$ if and only if $g=g'$.
			\end{itemize}
			There is a symmetric monoidal functor (induced by the universal property
			of the kernel) $\Phi\col \im^2F\ra\Ker Q_F$, which maps $(A,b,B)$
			(where $b\col FA\ra B$) to $(B,A,\tilde{b})$, where $\tilde{b}$ is the composite
			$B\xrightarrow{b^{-1}} FA\xrightarrow{r_{FA}^{-1}} FA\tens I$.
			If $(f,g)\col (A,b,B)\ra (A',b',B')$, then $\Phi(f,g)=(g,f)$; $\Phi$
			preserves equality. $\Phi$ is obviously an equivalence.
		\end{proof}
	
	The following corollary has been proved in \cite{Kasangian2000a}.
	
	\begin{coro}
		In $\CGS$, every faithful functor is canonically the kernel of its
		cokernel.
	\end{coro}
	
	\begin{pon}
		Let $F\col \A\ra\B$ in $\CGS$.  Then $\hat{M}_F\col\mathrm{Im}_{\mathrm{pl}}^2\,F\ra\B$
		is the root of the copip of $F$.
	\end{pon}
	
		\begin{proof}
			$\Root(\zeta_F)$ is the full sub-2-group
			of $\B$ whose objects are the objects $B$ such that $B = I$ in $\Copip F$,
			i.e.\ such that there exist $A\col \A$ and $g\col B\ra FA\tens I$.
			This is thus exactly $\mathrm{Im}_{\mathrm{pl}}^2\,F$.
		\end{proof}

	The following corollary appears in \cite{Dupont2003a}.
		
	\begin{coro}
		In $\CGS$, every full and faithful functor is canonically the root
		of its copip.
	\end{coro}
	
	Finally, if we combine all corollaries of this subsection with
	the equivalences of subsection \ref{sectcarcflechcgs}, we get the following proposition.
	
	\begin{pon}\label{theocgsdab}
		$\CGS$ is a 2-abelian $\Gpd$-category.
	\end{pon}
	
	We can improve this result a little.
	
	\begin{pon}
		$\CGS$ is a good 2-abelian $\Gpd$-category.
	\end{pon}
	
		\begin{proof}
			We prove condition 2 of Proposition \ref{defcaracgood}.
			Let $F\col\A\ra\B$ be a fully faithful functor
			in $\CGS$.  Then the homomorphism $\pi_0F\col\pi_0\A\ra\pi_0\B$ is an
			injection: if we have $A,A'\col\A$ such that $FA\simeq FA'$, then $A\simeq A'$,
			since $F$ is full.
			
			Next, let $F\col\A\ra\B$ be a fully cofaithful functor (i.e.
			full and surjective, by Proposition \ref{caracplcofidcgs}).  The homomorphism
			$\pi_1F\col\pi_1\A\ra\pi_1\B$ is surjective since, if we have
			$b\col I\ra I$ in $\B$, as $I\equiv FI$, there exists $a\col I\ra I$
			such that $b=Fa$, because $F$ is full.
		\end{proof}

\section{2-modules}\label{sectdmod}

In this section, we will prove that the $\Gpd$-category of 2-modules on a 2-ring $\mathbb{R}$ (one-object $\CGS$-category) form a good 2-abelian $\Gpd$-category.  These are the additive $\Gpd$-functors from $\mathbb{R}$ to $\CGS$.  This is a special case of the “module categories” on a “ring category” of Kapranov and Voevodsky \cite{Kapranov1994a}.

More generally, we will prove that if $\C$ is (good) 2-abelian, then the $\Gpd$-category of additive $\Gpd$-functors from a preadditive $\Gpd$-category to $\C$ is (good) 2-abelian.

\subsection{Definition of the $\Gpd$-categories of additive $\Gpd$-functors}

	We define now the additive functors, natural transformations and modifications
	between preadditive $\Gpd$-categories .
	
	\begin{df}\label{dfgpdfctadd}\index{additive!Gpd-functor@$\Gpd$-functor}
	\index{Gpd-functor@$\Gpd$-functor!additive}\index{2-sgp-functor@$\CGS$-functor}
		Let $\C$, $\D$ be two preadditive $\Gpd$-categories .  A $\Gpd$-functor
		$F\col \C\ra\D$ equipped with a transformation $\mu_{fg}\col Ff+Fg\Ra F(f+g)$ natural
		at $f$ and $g$ is \emph{additive} (or is a \emph{$\CGS$-functor})
		if the following conditions hold:
		\begin{enumerate}
			\item for all $A,B\col \C$, the functor $F_{A,B}\col \C(A,B)\ra\D(FA,FB)$,
				equipped with $\mu_{fg}$, is monoidal symmetric;
			\item for all $A,B,C\col \C$, the natural transformation
				\begin{xym}\xymatrix@R=40pt@C=50pt{
					{\C(A,B)\times\C(B,C)}\ar[r]^-{F_{AB}\times F_{BC}}
						\ar[d]_{\mathrm{comp}}\drtwocell\omit\omit{\nu}
					&{\D(FA,FB)\times\D(FB,FC)}\ar[d]^{\mathrm{comp}}
					\\ {\C(A,C)}\ar[r]_{F_{AC}}
					&{\D(FA,FC)}
				}\end{xym}
				is bimonoidal.
		\end{enumerate}
	\end{df}

	These conditions can be translated in elementary terms by the following conditions:
	\begin{enumerate}
		\item \begin{xym}\xymatrix@R=40pt@C=50pt{
				{}\save[]+<1.5cm,0cm>*{(Ff + Fg) + Fh\,}="m"
					\ar[d]_-{\mu_{fg}+1}\restore
				&&{}\save[]+<-1.5cm,0cm>*{\,Ff+(Fg+Fh)}="n"
				\ar[d]^-{1+\mu_{gh}}\restore
				\ar "m";"n" ^-{\alpha}
				\\ F(f+g)+Fh
				&& Ff+F(g+h)
				\\ {}\save[]+<1.5cm,0cm>*{F((f+g)+h)\,}="a"
					\ar@{<-}[u]^-{\mu_{f+g,h}}\restore
				&&{}\save[]+<-1.5cm,0cm>*{\,F(f+(g+h))}="b"
					\ar@{<-}[u]_-{\mu_{f,g+h}}\restore
					\ar "a";"b"_-{F\alpha}
			}\end{xym}
			\begin{xym}\xymatrix@=40pt{
				Ff+Fg\ar[r]^-{\mu_{fg}}\ar[d]_\gamma
				&F(f+g)\ar[d]_{F\gamma}
				\\ Fg+Ff\ar[r]_-{\mu_{gf}}
				&F(g+f)
			}\end{xym}
		\item \begin{xym}\xymatrix@R=40pt@C=50pt{
				{Fg_1Ff+Fg_2Ff}\ar[d]_-{\psi^{Fg_1,Fg_2}_{Ff}}
					\ar[r]^-{\nu_{g_1f}+\nu_{g_2f}\,}
				&{F(g_1f)+F(g_2f)}\ar[d]^-{\mu_{g_1f,g_2f}}
				\\ (Fg_1+Fg_2)Ff\ar[d]_{\mu_{g_1g_2}Ff}
				& F(g_1f+g_2f)
				\\ {F(g_1+g_2)Ff}\ar[r]_-{\nu_{g_1+g_2,f}}
				&{F((g_1+g_2)f)}\ar@{<-}[u]_-{F\psi^{g_1,g_2}_f}
			}\end{xym}
			\begin{xym}\xymatrix@R=40pt@C=50pt{
				{FgFf_1+FgFf_2}\ar[d]_-{\varphi_{Ff_1,Ff_2}^{Fg}}
					\ar[r]^-{\nu_{gf_1}+\nu_{gf_2}\,}
				&{F(gf_1)+F(gf_2)}\ar[d]^-{\mu_{gf_1,gf_2}}
				\\ Fg(Ff_1+Ff_2)\ar[d]_{(Fg)\mu_{f_1f_2}}
				& F(gf_1+gf_2)
				\\ {FgF(f_1+f_2)}\ar[r]_-{\nu_{g,f_1+f_2}}
				&{F(g(f_1+f_2))}\ar@{<-}[u]_-{F\varphi_{f_1,f_2}^g}
			}\end{xym}
	\end{enumerate}
	
	In the case where $\C$ and $\D$ only have one object, we recover the 
	$\mathrm{Ann}$-functors between $\mathrm{Ann}$-categories  of
	\cite{Quang2007c} and the 2-homomorphisms
	between categorical rings of \cite{Jibladze2007a}.
	
	\begin{df}\index{Gpd-natural transformation@$\Gpd$-natural transformation!additive}%
	\index{additive!Gpd-natural transformation@$\Gpd$-natural transformation}
		Let $F,G\col \C\ra\D$ be additive $\Gpd$-functors between preadditive 
		$\Gpd$-categories.  A natural transformation $\kappa\col F\Ra G$ is \emph{additive}
		if, for all $A,B\col \C$, the natural transformation
		\begin{xym}\xymatrix@=40pt{
			{\C(A,B)}\ar[r]^-{F_{AB}}\ar[d]_{G_{AB}}
				\drtwocell\omit\omit{\;\;\;\;\,\kappa_{AB}}
			&{\D(FA,FB)}\ar[d]^{\kappa_B\circ-}
			\\ {\D(GA,GB)}\ar[r]_-{-\circ\kappa_A}
			&{\D(FA,GB)}
		}\end{xym}
		is monoidal.
	\end{df}
	
	This condition amounts to the commutativity of the following diagram for all
	$f_1,f_2\col A\ra B$.
	\begin{xym}\xymatrix@R=40pt@C=50pt{
		{\kappa_B Ff_1+\kappa_B Ff_2}\ar[d]_-{\varphi_{Ff_1,Ff_2}^{\kappa_B}}
			\ar[r]^-{\kappa_{f_1}+\kappa_{f_2}\,}
		&{(Gf_1)\kappa_A+(Gf_2)\kappa_A}\ar[d]^-{\psi^{Gf_1,Gf_2}_{\kappa_A}}
		\\ \kappa_B(Ff_1+Ff_2)\ar[d]_{\kappa_B\mu_{f_1f_2}}
		& (Gf_1+Gf_2)\kappa_A
		\\ {\kappa_B F(f_1+f_2)}\ar[r]_-{\kappa_{f_1+f_2}}
		&{G(f_1+f_2)\kappa_A}\ar@{<-}[u]_-{\mu_{f_1f_2}\kappa_A}
	}\end{xym}
	
	The $\Gpd$-category of additive $\Gpd$-functors between two preadditive $\Gpd$-categories 
	had been introduced in \cite{Drion2002a}, but without
	condition 2 of additive $\Gpd$-functor, and without additivity conditions
	for the natural transformations.
	
	\begin{df}\index{Add(C,D)@$\caspar{Add}(\C,\D)$}
		Let $\C$, $\D$ be two preadditive $\Gpd$-categories . The $\Gpd$-category
		$\caspar{Add}(\C,\D)$ has as objects the additive $\Gpd$-functors from $\C$
		to $\D$, as arrows the additive $\Gpd$-natural transformations
		between them and as 2-arrows the modifications between them.
	\end{df}
	
		\begin{proof}
			We have to check that the composite of two additive natural transformations
			is additive.
			The transformation which expresses the naturality of $\kappa\circ\theta$
			is the following composite,
			which is monoidal because all the involved transformations
			are monoidal, by the definition of preadditive $\Gpd$-categories 
			and of additive natural transformations.
		\end{proof}
			\begin{xym}\xymatrix@=40pt{
				{\C(A,B)}\ar[rr]^{F_{AB}}\ar[dr]^{G_{AB}}\ar[dd]_{H_{AB}}
					\drrtwocell\omit\omit{_<-1>{\;\;\;\;\;\theta_{AB}}}
					\ddrtwocell\omit\omit{_<3.3>{\;\;\;\;\;\kappa_{AB}}}
				&&{\D(FA,FB)}\ar[d]_{\theta_B\circ-}
					\dduppertwocell\omit{_<-7.6>}
					\ar@/^4.1pc/[dd]^{(\kappa_B\theta_B)\circ-}
				\\ &{\D(GA,GB)}\ar[r]^{-\circ\theta_A}\ar[d]_{\kappa_B\circ-}
					\drtwocell\omit\omit
				&{\D(FA,GB)}\ar[d]_{\kappa_B\circ-}
				\\ {\D(HA,HB)}\ar[r]^{-\circ\kappa_A}\rrlowertwocell\omit{_<3>}
					\ar@/_2.2pc/[rr]_{-\circ(\kappa_A\theta_A)}
				&{\D(GA,HB)}\ar[r]^{-\circ\theta_A}
				&{\D(FA,HB)}
			}\end{xym}
	
	In the case where $\C$ has one object (is a “2-ring”), an additive $\Gpd$-functor
	$\C\ra\CGS$ is what could be called a “2-module”\index{2-module}
	on $\C$.  The $\Gpd$-category
	$\dMod_\C\eqdef\caspar{Add}(\C,\CGS)$ is the $\Gpd$-category of 2-modules on $\C$.
	
	Like in dimension 1, we can prove the following results,
	which can be found (modulo the above-mentioned differences in definitions)
	in \cite{Drion2002a}.
	
	\begin{pon}
		Let $\C$, $\D$ be $\Gpd$-categories .  If $\D$ is preadditive,
		then $[\C,\D]$ is also preadditive.
	\end{pon}
	
		\begin{proof}
			Let us first define the structure of symmetric 2-group on the groupoid
			$[\C,\D](F,G)$.  Let be $\theta,\kappa\col F\Ra G$.  We define
			$\theta+\kappa$ on objects by $(\theta+\kappa)_C\eqdef\theta_C+\kappa_C$
			and on arrows by $(\theta+\kappa)_{AB}\eqdef$
			\begin{xym}\xymatrix@=60pt{
				{\C(A,B)}\ar[r]^-{F_{AB}}\ar[d]_{G_{AB}}
					\drtwocell\omit\omit{\substack{\theta_{AB}+\kappa_{AB} \\ {} 
					\\ {}}\;\;\;\;\;\;\;\;\;\;\;\;\;\;\;\;\;\;\;\;\;\;\;\;}
				&{\D(FA,FB)}\ar[d]_{\substack{(\theta_B\circ-)\\ +(\kappa_B\circ-)}}
					\duppertwocell\omit{_<-2.5>\psi^{-1}}
					\ar@/^2pc/[d]^{(\theta_B+\kappa_B)\circ-}
				\\{\D(GA,GB)}\ar[r]^-{(-\circ\theta_A)+(-\circ\kappa_A)}
					\rlowertwocell_{-\circ(\theta_A+\kappa_A)}<-8>{_<2>\varphi}
				&{\D(FA,GB)}
			}\end{xym}
			The associativity $\alpha$ and the neutrality of the identity
			$\lambda$ and $\rho$ are defined pointwise as the corresponding transformations
			of $\D(FC,GC)$.  These are modifications
			and the axioms hold because they do
			in $\D(FC,GC)$.
			
			In the same way the modifications expressing the distributivity
			$\varphi_{\theta,\kappa}^\pi\col \pi\theta+\pi\kappa\Rrightarrow \pi(\theta+\kappa)$ and
			$\psi_\theta^{\xi,\pi}\col \xi\theta+\pi\theta\Rrightarrow(\xi+\pi)\theta$ are
			defined pointwise as in $\D$ and satisfy the required
			axioms because it is the case in $\D$. 
		\end{proof}

	\begin{pon}
		For all preadditive $\Gpd$-categories  $\C$ and $\D$, the $\Gpd$-ca\-te\-go\-ry
		$\caspar{Add}(\C,\D)$ is preadditive.
	\end{pon}
	
		\begin{proof}
			Everything goes as in the previous proposition.  The only thing to check
			is that $\theta+\kappa$ is additive.  This is the case, because all the
			transformations of the previous diagram are monoidal, by the definition
			of preadditive $\Gpd$-categories  and by the additivity of $\theta$
			and $\kappa$.
		\end{proof}

	In \cite{Drion2002a}, since $\varphi$ and $\psi$ were not assumed to be
	monoidal, we cannot prove that $\kappa+\theta$ is additive.  This explains
	why he does not require that the natural transformations between
	additive $\Gpd$-functors satisfy some conditions.

\subsection{Limits and colimits of (additive) $\Gpd$-functors}

	The goal of this section is to prove that limits are constructed pointwise in the
	$\Gpd$-categories  of $\Gpd$-functors (it suffices to do it
	for pullbacks and the zero object).  We will check that this remains true for
	additive $\Gpd$-functors.
	
	\begin{pon}\label{limpointparpoint}
		Let $\C$, $\D$ be $\Gpd$-categories, $F,G,H\col\C\ra\D$ be $\Gpd$-func\-tors
		and $\alpha\col F\Ra H$, $\beta\col G\Ra H$ be $\Gpd$-natural transformations.
		If, for every $C\col\C$, there exist an object $PC$, arrows $\gamma_C$
		and $\delta_C$ and a 2-arrow $\Xi_C$, as in the following diagram,
		\begin{xym}\label{prodfibauds}\xymatrix@=40pt{
			PC\ar[r]^{\gamma_C}\ar[d]_{\delta_C}\drtwocell\omit\omit{\;\;\;\,\Xi_C}
			&FC\ar[d]^{\alpha_C}
			\\ GC\ar[r]_{\beta_C}
			&HC
		}\end{xym}
		which form a pullback, then there exist a $\Gpd$-functor $P\col\C\ra\D$,
		$\Gpd$-natural transformations $\gamma$ and $\delta$ and a modification
		$\Xi$, as in the following diagram, which at an object $C\col\C$
		are defined by diagram \ref{prodfibauds}.
		\begin{xym}\xymatrix@=40pt{
			P\ar[r]^{\gamma}\ar[d]_{\delta}\drtwocell\omit\omit{\;\,\Xi}
			&F\ar[d]^{\alpha}
			\\ G\ar[r]_{\beta}
			&H
		}\end{xym}
		Moreover, for each diagram in $[\C,\D]$ of the form of the previous diagram,
		if at each point $C\col\C$ diagram \ref{prodfibauds} is a pullback
		in $\D$, then $\Xi$ is a pullback in $[\C,\D]$.
	\end{pon}
	
		\begin{proof}
			{\it Constructions. }$P$, $\gamma$, $\delta$ and $\Xi$ are already defined
			on objects.
			
			If $C\overset{c}\ra C'$ is an arrow in $\C$, by the universal property
			of the pullback, there exists an arrow $Pc$ and 2-arrows $\gamma_c$
			and $\delta_c$ such that
			\begin{xyml}\begin{gathered}\xymatrix@R=20pt@C=10pt{
				PC\ar[rr]^-{\gamma_C}\ar[dd]_{\delta_C}
					\ddrrtwocell\omit\omit{\;\;\;\;\Xi_C}
				&&FC\ar[dr]^-{Fc}\ar[dd]^{\alpha_C}
				\\ &{}\ddrrtwocell\omit\omit{_<-5>{\;\;\alpha_c}}
					\ddrrtwocell\omit\omit{_<5>{\;\;\;\;\,\beta_c^{-1}}}
				&&FC'\ar[dd]^{\alpha_{C'}}
				\\ GC\ar[dr]_-{Gc} \ar[rr]_{\beta_C}
				&&HC\ar[dr]^(0.44){Hc}
				\\ &GC'\ar[rr]_-{\beta_{C'}}
				&&HC'
			}\end{gathered} \;\;= \;\;
			\begin{gathered}\xymatrix@R=20pt@C=10pt{
				PC\ar[rr]^-{\gamma_C}
					\ar[dd]_{\delta_C}\ar@{-->}[dr]^-{Pc}
					\ddrrtwocell\omit\omit{_<-5>{\;\;\;\;\gamma_c^{-1}}}
					\ddrrtwocell\omit\omit{_<5>{\;\delta_c}}
				&&FC\ar[dr]^-{Fc}
				\\ &PC'\ar[rr]_-{\gamma_{C'}}\ar[dd]^{\delta_{C'}}
					\ddrrtwocell\omit\omit{\;\;\;\;\;\Xi_{C'}}
				&&FC'\ar[dd]^{\alpha_{C'}}
				\\ GC\ar[dr]_-{Gc} &&{}
				\\ &GC'\ar[rr]_-{\beta_{C'}}
				&&HC'
			}\end{gathered}.\end{xyml}
			This defines $P$, $\gamma$ and $\delta$ on arrows and prove
			that $\Xi$ is a modification.
			
			Next, if $\chi\col c\Ra c'\col C\ra C'$ is a 2-arrow in $\C$,
			by the universal property of the pullback, there exists a unique 2-arrow
			$P\chi\col Pc\Ra Pc'$ such that $\gamma_{c'}\circ\gamma_{C'}(P\chi) = 
			(F\chi)\gamma_C\circ\gamma_c$ and $\delta_{c'}\circ\delta_{C'}(P\chi) = 
			(G\chi)\delta_C\circ\delta_c$, which defines $P$ on 2-arrows and proves
			that $\gamma_c$ and $\delta_c$ are natural.
			
			By using again the universal property of the pullback, we get
			2-arrows $\varphi^P_{c',c}:\allowbreak (Pc')(Pc)\Ra P(c'c)$ and $\varphi^P_C
			\col 1_{PC}\Ra P1_C$ satisfying
			conditions expressing the $\Gpd$-naturality of $\gamma$ and $\delta$.
						
			We check the naturality of $\varphi^P$ and the $\Gpd$-functoriality
			of $P$ by using the fact that, for every $C\col\C$,
			the arrows $\gamma_C$ and $\delta_C$ are jointly faithful.
			
			{\it Universal property. }For the first part
			of the universal property, let be the following diagram in $[\C,\D]$.
			\begin{xym}\xymatrix@=40pt{
				Q\ar[r]^{\varepsilon}\ar[d]_{\eta}\drtwocell\omit\omit{\;\Theta}
				&F\ar[d]^{\alpha}
				\\ G\ar[r]_{\beta}
				&H
			}\end{xym}
			For every $C\col\C$, by the universal property of the pullback
			in $\D$, there exist an arrow $\chi_C\col QC\ra PC$ and 2-arrows
			$\Upsilon_C\col\varepsilon_C\Ra \gamma_C\chi_C$ and $\Upsilon'_C\col\eta_C
			\Ra\delta_C\chi_C$ such that $\Xi_C\chi_C\circ\alpha_C\Upsilon_C
			= \beta_C\Upsilon'_C\circ\Theta_C$.  By using again the universal property
			of the pullback, we define for each $C\overset{c}\ra C'$
			a 2-arrow $\chi_c$ satisfying a condition expressing that
			$\Upsilon$ and $\Upsilon'$ are modifications. We check that $\chi$
			is a $\Gpd$-natural transformation by using the fact that
			$\gamma_C$ and $\delta_C$ are jointly faithful.
			
			For the second part of the universal property, let be
			$\Theta\col\gamma\chi\Rrightarrow\gamma\chi'$ and
			$\Theta'\col\delta\chi\Rrightarrow\delta\chi'$
			such that $\Xi\chi'\circ\alpha\Theta=\beta\Theta'\circ\Xi\chi$.
			Then by the universal property at each point $C\col\C$, there exists
			a unique $\Upsilon_C\col\chi_C\Ra \chi'_C$ such that $\gamma_C\Upsilon_C=\Theta_C$
			and $\delta_C\Upsilon_C=\Theta'_C$.  We check that $\Upsilon$ is a
			modification by using the fact that $\gamma_C$ and $\delta_C$ are jointly
			faithful.
		\end{proof}
		
	\begin{pon}\label{limpointparpointadd}
		Let $\C$, $\D$ be preadditive $\Gpd$-categories ,
		$F,G,H\col\C\ra\D$ be additive $\Gpd$-functors
		and $\alpha\col F\Ra H$, $\beta\col G\Ra H$ be additive $\Gpd$-natural transformations.
		If for every $C\col\C$, there exist an object $PC$, arrows $\gamma_C$
		and $\delta_C$ and a 2-arrow $\Xi_C$, as in the following diagram,
		\begin{xym}\label{prodfibaudsu}\xymatrix@=40pt{
			PC\ar[r]^{\gamma_C}\ar[d]_{\delta_C}\drtwocell\omit\omit{\;\;\;\,\Xi_C}
			&FC\ar[d]^{\alpha_C}
			\\ GC\ar[r]_{\beta_C}
			&HC
		}\end{xym}
		which form a pullback, then there exist an additive $\Gpd$-functor
		$P\col\C\ra\D$,
		additive $\Gpd$-natural transformations $\gamma$ and $\delta$ and a
		modification $\Xi$, as in the following diagram, which, at an object $C\col\C$
		are defined by diagram \ref{prodfibaudsu}.
		\begin{xym}\xymatrix@=40pt{
			P\ar[r]^{\gamma}\ar[d]_{\delta}\drtwocell\omit\omit{\;\,\Xi}
			&F\ar[d]^{\alpha}
			\\ G\ar[r]_{\beta}
			&H
		}\end{xym}
		Moreover, for each diagram in $\caspar{Add}(\C,\D)$ of the form of the previous
		 diagram, if at each point $C\col\C$ diagram \ref{prodfibaudsu} is a pullback
		in $\D$, then $\Xi$ is a pullback in $\caspar{Add}(\C,\D)$.
	\end{pon}
	
		\begin{proof}
			{\it Constructions. } Everything goes as in the proof of the previous
			 proposition.  We just need to check that $P$, $\gamma$ and $\delta$ are additive.
			
			By the universal property of the pullback,
			we get, from the 2-arrows $\mu^F_{f,g}$ and $\mu^G_{f,g}$,
			a 2-arrow $\mu^P_{f,g}\col Pf+Pg\Ra P(f+g)$ satisfying two conditions
			expressing the additivity of $\gamma$ and $\delta$.  
			Then $\mu^P$ is natural because $\gamma_c$, $\delta_c$, $\mu^F$, $\mu^G$,
			$\varphi$ and $\psi$ are.
			
			We check that the functor $P_{C,C'}$,
			equipped with $\mu^P$, is symmetric monoidal, and that the axioms of additive 
			$\Gpd$-functor hold by using the fact that $\gamma_{C'}$ and $\delta_{C'}$
			are jointly faithful.
			
			{\it Universal property. }  The only new thing to check with respect to
			 the proof of the previous proposition is that the $\Gpd$-natural transformation
			$\chi\col Q\Ra P$ is additive.  We check the additivity axiom
			by using the fact that $\gamma_C$ and $\delta_C$ are jointly faithful.
		\end{proof}
	
	The following proposition is very easy to prove.
			
	\begin{pon}\label{objzerofonct}
		Let $\C$ be a $\Gpd$-category and $\D$ be a $\Gpdp$-category with zero object.
		Then the constant $\Gpd$-functor $0\col\C\ra\D$ which maps every object to $0$,
		every arrow to $1_0$ and every 2-arrow to $1_{1_0}$ is a zero object
		in $[\C,\D]$.  If moreover $\C$ and $\D$ are preadditive, then
		$0$ is a zero object in $\caspar{Add}(\C,\D)$.
	\end{pon}
	
	It follows from the two previous propositions that, in $[\C,\D]$
	(or in $\caspar{Add}(\C,\D)$, if $\C$ and $\D$ are preadditive), the products
	(case where $H$ is $0$), the coproducts, the kernels (case where $G$ is $0$) and
	cokernels are constructed pointwise.  Given that the pips, roots, copips and
	coroots can be constructed with kernels and cokernels, we can also construct them
	pointwise.  This allows to prove the following corollary.
	
	\begin{coro}\label{caraczfidfonct}
		Let $\C$ and $\D$ be two $\Gpd$-categories .  Let us assume that $\D$ has all kernels.
		\begin{enumerate}
			\item An arrow $\alpha\col F\Ra G$ in $[\C,\D]$ is 0-faithful if and only
				if for every $C\col\C$, $\alpha_C\col FC\ra GC$ is 0-faithful in $\D$.
			\item An arrow $\alpha\col F\Ra G$ in $[\C,\D]$ is fully
				0-faithful if and only if for every $C\col\C$, $\alpha_C\col FC\ra GC$
				is fully 0-faithful in $\D$.
		\end{enumerate}
		If $\C$ and $\D$ are preadditive, the same properties hold
		in $\caspar{Add}(\C,\D)$.
	\end{coro}
	
		\begin{proof}
			We prove point 1. The proof of point 2 is similar.
			Let $\alpha\col F\Ra G$ be a $\Gpd$-natural transformation.  Since $\D$
			has all kernels, $\D$ has all pips and so, by Proposition
			\ref{limpointparpoint}, the pip of $\alpha$ exists (let us denote it by
			$\varpi\col 0\Rrightarrow 0\col P\Ra F$)
			and is computed pointwise.  If $\alpha$ is 0-faithful, $\varpi=1_0$,
			by Proposition \ref{pepclaszfid}.
			So, for every $C\col\C$, $\varpi_C=1_0$ and $\alpha_C$ is 0-faithful,
			again by Proposition \ref{pepclaszfid}. Conversely, if for every
			$C\col\C$, $\varpi_C=1_0$, then $\varpi=1_0$ and $\alpha$ is 0-faithful.
		\end{proof}
	
\subsection{$[\C,\D]$ and $\caspar{Add}(\C,\D)$ are 2-abelian if $\D$ is}

	\begin{pon}\label{pondmoddab}~
		\begin{enumerate}
			\item Let $\C$ be a $\Gpd$-category and $\D$ be a 2-abelian $\Gpd$-category.
				Then $[\C,\D]$ is 2-abelian; if moreover $\D$ is good, then 
				$[\C,\D]$ is also good.
			\item Let $\C$ be a preadditive $\Gpd$-category
				and $\D$ be a 2-abelian $\Gpd$-category.
				Then $\caspar{Add}(\C,\D)$ is 2-abelian; if  moreover $\D$ is good, then 
				$\caspar{Add}(\C,\D)$ is also good.
		\end{enumerate}
	\end{pon}
	
		\begin{proof}
			We prove point 1. Point 2 is proved in the same way, by
			using Proposition \ref{limpointparpointadd}.
			
			First, by Proposition \ref{limpointparpoint} and Proposition
			\ref{objzerofonct}, $[\C,\D]$ has a zero object, all finite
			products and coproducts,
			and all kernels and cokernels, because $\D$ has these limits.
			
			Next, let $\alpha\col F\Ra G$ be a 0-faithful $\Gpd$-natural transformation,
			and let $\rho\col G\Ra Q$,
			with $\Upsilon\col \rho\alpha\Rrightarrow 0$ be the cokernel of $\alpha$, which is
			computed pointwise.  By Corollary \ref{caraczfidfonct},
			for every $C\col\C$, $\alpha_C$ is 0-faithful. So, since $\D$ is 2-abelian,
			$(\alpha_C,\Upsilon_C)=\Ker \rho_C$.  Then, by Proposition
			\ref{limpointparpoint}, $(\alpha,\Upsilon)=\Ker \rho$.
			
			We prove in the same way that every fully 0-faithful arrow is
			canonically the root of its copip.  The dual properties are proved dually.
			So $[\C,\D]$ is 2-abelian.
			
			Let us assume now that $\D$ is a good 2-abelian $\Gpd$-category.
			We use condition 4 of Proposition \ref{defcaracgood}.
			Let be $\alpha\col F\Ra G$ in $[\C,\D]$.
			Diagram \ref{diagbonnquatr} for $\alpha$ is constructed with  kernels
			and cokernels (thus pointwise).  Let us denote by $\gamma_\alpha$
			the comparison arrow $\Ker \alpha\Ra\Ker\pi_0\alpha$.  Since
			$\D$ is a good 2-abelian $\Gpd$-category, for every $C\col\C$,
			$(\gamma_\alpha)_C$ is cofaithful and thus 0-cofaithful. So, by Corollary
			\ref{caraczfidfonct}, $\gamma_\alpha$ is 0-cofaithful and thus cofaithful, 
			since we have already proved that $[\C,\D]$ is 2-abelian.  We prove dually
			that $\delta_\alpha$ is faithful.
		\end{proof}
		
	In particular, this proposition shows that if $\mathbb{R}$ is a 2-ring
	(a one-object $\CGS$-category), the $\Gpd$-category of 2-modules on $\mathbb{R}$,
	which is 
	\begin{eqn}
		\dMod_{\mathbb{R}}\eqdef\caspar{Add}(\mathbb{R},\CGS),
	\end{eqn}
	is 2-abelian.  We will prove
	that the abelian category of discrete objects (or of connected objects) in
	$\dMod_{\mathbb{R}}$ is nothing else than the category
	of modules on $\pi_0(\mathbb{R})$ (the ring with the same objects as $\mathbb{R}$,
	where two objects are equal if they are isomorphic in $\mathbb{R}$):
	\begin{eqn}\label{deuxmodmod}
		\Dis(\dMod_{\mathbb{R}})\simeq \caspar{Mod}_{\pi_0(\mathbb{R})}.
	\end{eqn}
	In particular, if $\mathbb{R}$ is $R\dis$ (the discrete 2-ring whose ring
	of objects is a ring $R$), the discrete 2-modules on $R$ are the ordinary
	modules on $R$.
	
	If $\C$ is a $\Gpd$-category or a $\CGS$-category,
	let us recall that the we denote by $\mathrm{Ho}(\C)$ the result of the application
	of $\pi_0$ to the groupoids (or symmetric 2-groups) of arrows between two objects
	(equation \ref{defhochomotop}).
	The category
	(or $\Ab$-category) $\mathrm{Ho}(\C)$\index{Ho(C)@$\mathrm{Ho}\,\C$}
	can be described in the following way.
	\begin{itemize}
		\item {\it Objects. } These are the objects of $\C$.
		\item {\it Arrows. } These are the arrows of $\C$.
		\item {\it Equality. } Two arrows $c,c'\col C\ra C'$ in $\pi_0(\C)$ are equal
			if there exists a 2-arrow $\gamma\col c\Ra c'$ in $\C$.
	\end{itemize}
	If $\C$ is a $\CGS$-category, the 2-arrows of the preadditive structure
	become equalities showing that $\mathrm{Ho}(\C)$ is a preadditive category.
	
	\begin{lemm}~
		\begin{enumerate}
			\item If $\C$ is a $\Gpd$-category and $\D$ is a $\Ens$-category, then 
				\begin{eqn}
					[\C,\D]\simeq [\mathrm{Ho}(\C),\D].
				\end{eqn}
			\item If $\C$ is a $\CGS$-category and $\D$ is an $\Ab$-category, then 
				\begin{eqn}
					\caspar{Add}(\C,\D)\simeq \caspar{Add}(\mathrm{Ho}(\C),\D).
				\end{eqn}
		\end{enumerate}
	\end{lemm}
	
		\begin{proof}
			Let $\Phi\col [\mathrm{Ho}(\C),\D]\ra [\C,\D]$ be the $\Gpd$-functor
			which maps a $\Gpd$-functor $F\col\mathrm{Ho}(\C)
			\ra\D$ to the $\Gpd$-functor $\Phi(F)$, which maps an object $C$ to $FC$,
			an arrow $C\overset{c}\ra C'$ to $Fc$, and a 2-arrow $\gamma\col c\Ra c'$
			to the equality $Fc = Fc'$ (because $c=c'$ in $\mathrm{Ho}(\C)$ and thus $Fc=Fc'$
			in the $\Ens$-category $\D$).  A $\Gpd$-natural transformation
			$\alpha\col F\Ra F'$ is mapped to the transformation $\Phi(\alpha)$
			which takes the value $\alpha_C$ at the point $C\col\C$.
			
			The $\Gpd$-functor $\Phi$ is surjective because, if $G\col\C\ra\D$ is a $\Gpd$-functor,
			we can define $\hat{G}\col\mathrm{Ho}(\C)\ra\D$, which maps $C$ to $GC$ and
			$C\overset{c}\ra C'$ to $Gc$;	$\hat{G}$ preserves equality between arrows,
			because if $c=c'$ in $\mathrm{Ho}(\C)$, we have $\gamma\col c\Ra c'$ in $\C$
			which is mapped by $G$ to the equality $Gc=Gc'$.  It is clear
			that $\Phi(\hat{G})\equiv G$. We easily check
			that $\Phi$ is full and faithful.
			
			If $\C$ and $\D$ are preadditive, we must check that $\Phi$ restricts
			to an equivalence between additive functors.  It is obvious
			that the additivity of $F$ implies the additivity of $\Phi(F)$ and that
			the additivity of $G$ implies that of $\hat{G}$.
		\end{proof}
	
	By applying the following proposition to $\C=\mathbb{R}$ and $\D=\CGS$, we get
	equation \ref{deuxmodmod}.
	
	\begin{pon}
		Let $\C$ be a preadditive $\Gpd$-category and $\D$ a 2-abelian $\Gpd$-category.
		Then
		\begin{eqn}
			\Dis(\caspar{Add}(\C,\D))\simeq\caspar{Add}(\mathrm{Ho}(\C),\Dis(\D)).
		\end{eqn}
	\end{pon}

		\begin{proof}
			An additive $\Gpd$-functor $F\col\C\ra\D$ is discrete in
			$\caspar{Add}(\C,\D)$ if the arrow $F\Ra 0$ is faithful.  Since
			$\caspar{Add}(\C,\D)$ is 2-abelian, this is the case if and only if it
			is 0-faithful, which is the case if and only if for every $C\col \C$,
			the arrow $FC\ra 0$ is 0-faithful, by Corollary
			\ref{caraczfidfonct}.  So the discrete objects in
			$\caspar{Add}(\C,\D)$ are the additive $\Gpd$-functors whose image
			lies in $\Dis(\D)$.  We have thus the equivalence
			\begin{eqn}
				\Dis(\caspar{Add}(\C,\D))\simeq\caspar{Add}(\C,\Dis(\D)).
			\end{eqn}
			The conclusion follows by the previous lemma.
		\end{proof}

\section{Baez-Crans 2-vector spaces}\label{catintcatab}

	The goal of this section is to study the $\Gpd$-category of 2-vector spaces on 
	a field $K$ in the sense of Baez-Crans \cite{Baez2004c}.%
	\index{2-vector space, Baez-Crans}
	This is the $\Gpd$-category of internal groupoids, internal functors
	and internal natural transformations in the category of vector spaces on $K$.
	More generally, we will study the $\Gpd$-category of internal groupoids
	in any abelian category $\C$, which is equivalent to the $\Gpd$-category
	of chain complexes of length 1 in $\C$, with the morphisms of chain complexes
	and the homotopies of chain complexes (which is studied in part	in \cite{Grandis2001b}).  We will see that this $\Gpd$-category is
	2-abelian if and only if the axiom of choice holds in $\C$.  The problem is
	that the arrows which are faithful, full and cofaithful are not in general 
	equivalences.

\subsection{Definition}

	Let us fix an abelian abelian category $\C$. We won't work with the $\Gpd$-category
	$\Gpd(\C)$ of internal groupoids, internal functors and internal
	natural transformations in $\C$, but with the $\Gpd$-category $\flcp$, which is equivalent.\index{C2+@$\flcp$}
	
	\begin{df}
		Let $\C$ be an $\Ab$-category.  The $\Gpdp$-category
		$\flcp$ is defined in the following way.
		\begin{enumerate}
			\item {\it Objects.} These are the arrows $A_1\overset{f}\ra A_0$.
			\item {\it Arrows.} The arrows from $A_1\overset{f}\ra A_0$
				to $B_1\overset{g}\ra B_0$ are the commutative squares
				\begin{xym}\label{flechedeflc}\xymatrix@=30pt{
					A_1\ar[r]^{u_1}\ar[d]_f
					&B_1\ar[d]^g
					\\ A_0\ar[r]_{u_0} &B_0.
				}\end{xym}
			\item {\it 2-arrows.} A 2-arrow
				$\alpha\col(u_1,u_0)\Ra (u'_1,u'_0)\col f\ra g$ is an arrow
				$\alpha\col A_0\ra B_1$ in $\C$ (see the following diagram) such that
				\begin{align}\begin{split}\stepcounter{eqnum}
					u_1-u'_1 &=\alpha f;\\
					u_0-u'_0 &= g\alpha.
				\end{split}\end{align}
				In particular, a 2-arrow $(u_1,u_0)\Ra (0,0)$
				is an arrow $\alpha\col A_0\ra B_1$ such that
				$u_1=\alpha f$ and $u_0=g\alpha$, and a 2-arrow
				$(0,0)\Ra (0,0)\col f\ra g$ is an arrow $\alpha$ such that 
				$\alpha f=0$ and $g\alpha =0$.
			\item {\it Composition.} Composition of 1-arrows is defined
				componentwise:
				$(v_1,v_0)\circ(u_1,u_0)=(v_1u_1,v_0u_0)$.  The identity on
				$A_1\overset{f}\ra A_0$ is $(1_{A_1},1_{A_0})$. 
				For 2-arrows, the composite of
				$(u_1,u_0)\overset{\alpha}\Longrightarrow (u'_1,u'_0)
				\overset{\alpha'}\Longrightarrow (u''_1,u''_0)$ is
				$\alpha'+\alpha$; the identity on $(u_1,u_0)$ is $0$; and each
				2-arrow $\alpha$ has $-\alpha$ as its inverse.
				It remains to define the horizontal composition of 2-arrows.
				If we have $\alpha\col (u_1,u_0)\Ra (u'_1,u'_0)$ and 
				$\beta\col (v_1,v_0)\Ra (v'_1,v'_0)$ as in the following diagram,
				we define
				\begin{eqn}
					\beta * \alpha = v'_1\alpha+\beta u_0 = \beta u'_0 + v_1\alpha.
				\end{eqn}
				Consequently
				\begin{align}\begin{split}\stepcounter{eqnum}
					(v_1,v_0)*\alpha &= v_1\alpha;\\
					\beta * (u_1,u_0) &= \beta u_0.
				\end{split}\end{align}
				\begin{xym}\xymatrix@=40pt{
					A_1\ar@<1.2mm>[r]^{u_1}\ar@<-1.2mm>[r]_{u'_1}\ar[d]_f
					&B_1\ar[d]^g\ar@<1.2mm>[r]^{v_1}\ar@<-1.2mm>[r]_{v'_1}
					&C_1\ar[d]^h
					\\ A_0\ar@<1.2mm>[r]^{u_0}\ar@<-1.2mm>[r]_{u'_0}
						\ar[ur]_(0.6)\alpha
					&B_0\ar@<1.2mm>[r]^{v_0}\ar@<-1.2mm>[r]_{v'_0}\ar[ur]_(0.6)\beta
					&C_0
				}\end{xym}
			\item {\it Zero arrows.} The zero arrow between two arrows
				is simply the pair $(0,0)$.  Moreover, the arrow
				$0\ra 0$ is a zero object.
		\end{enumerate}
	\end{df}

\subsection{Kinds of arrows}

	To each arrow $(u_1,u_0)\col f\ra g$ in $\flcp$
	(diagram \eqref{flechedeflc}) corresponds a sequence whose composite is $0$:
	\begin{eqn}\label{carrecomsuit}
		A_1\xrightarrow{\vervec{-f}{u_1}} 
		A_0\oplus B_1\xrightarrow{(u_0\;\, g)} B_0.
	\end{eqn} 
	
	We will characterise the (fully) faithful, (fully) cofaithful
	and full morphisms of $\flcp$ in terms of the exactness of this sequence.
	
	\begin{pon}\label{caracfidflc}\index{faithful arrow!in C2+@in $\flcp$}
		Let $(u_1,u_0)\col f\ra g$ be a morphism in $\flcp$.  The following conditions
		are equivalent:
		\begin{enumerate}
			\item $(u_1,u_0)$ is faithful;
			\item $(u_1,u_0)$ is 0-faithful;
			\item $u_1$ and $f$ are jointly monomorphic;
			\item $\vervec{-f}{u_1}$ is a monomorphism;
			\item the sequence $0\longrightarrow A_1\xrightarrow{\vervec{-f}{u_1}} A_0\oplus B_1$
				is exact.
		\end{enumerate}
	\end{pon}
	
		\begin{proof}
			It is clear that condition 1 implies condition 2 and that conditions
			3, 4 and 5 are equivalent.
			
			{\it 2 $\Rightarrow$ 3. }Let be $\alpha,\beta\col X\ra A_1$
			such that $u_1\alpha= u_1\beta$
			and $f\alpha=f\beta$.  Then $\alpha-\beta$ is a 2-arrow
			$(0,0)\Ra (0,0)\col 0_X\ra f$ and
			$(u_1,u_0)*(\alpha-\beta) = 1_{(0,0)}$.  Then, by
			condition 2, $\alpha-\beta = 1_{(0,0)} = 0$, thus $\alpha = \beta$.
			
			{\it 3 $\Rightarrow$ 1. }Let be $\alpha, \beta \col (v_1,v_0)\Ra (v'_1,v'_0)\col 
			x\ra f$ in $\flcp$
			such that $(u_1,u_0)*\alpha = (u_1,u_0)*\beta$.
			We have thus $v_1-v'_1=\alpha x=\beta x$, $v_0-v'_0=f\alpha=f\beta$,
			and $u_1\alpha = u_1\beta$.  So, by condition 3,
			$\alpha = \beta$.
		\end{proof}

	The discrete objects being those for which the unique arrow to $0$ is
	faithful, we get the following corollary.
	
	\begin{coro}\label{caracdiscreteflc}\index{discrete object!in C2+@in $\flcp$}
		Let $A_1\overset{f}\longrightarrow A_0$ be an object of $\flcp$. Then $f$ is discrete
		if and only if $f$ is a monomorphism in $\C$.  Dually,
		$f$ is connected if and only if $f$ is an epimorphism in $\C$.
	\end{coro}

	Let us turn now to fully faithful arrows.

	\begin{pon}\label{caracplfidflc}\index{fully faithful arrow!in C2+@in $\flcp$}
		Let $(u_1,u_0)\col f\ra g$ be a morphism in $\flcp$.  The following conditions
		are equivalent:
		\begin{enumerate}
			\item $(u_1,u_0)$ is fully faithful;
			\item $(u_1,u_0)$ is fully 0-faithful;
			\item the square \eqref{flechedeflc} is a pullback;
			\item $\vervec{-f}{u_1}$ is the kernel of $(u_0\;\, g)$;
			\item the sequence 	$0\longrightarrow A_1\xrightarrow{\vervec{-f}{u_1}}
				A_0\oplus B_1\xrightarrow{(u_0\;\, g)} B_0$ is exact.
		\end{enumerate}
	\end{pon}
	
		\begin{proof}
			It is obvious that condition 1 implies condition 2 and that conditions
			3 to 5 are equivalent, since $\C$ is abelian.
			
			{\it 2 $\Rightarrow$ 3. }Let be $A_0\overset{v}\leftarrow X\overset{\gamma}
			\ra B_1$ such that $u_0v=g\gamma$.  We have then an object $0\overset{0_X}\ra X$,
			an arrow $(0,v)\col 0_X\ra f$ and a 2-arrow
			$\gamma\col(u_1,u_0)(0,v)\Ra (0,0)$ in $\flcp$.
			Since $(u_1,u_0)$ is fully 0-faithful,
			there exists a unique $\alpha\col X\ra A_1$ such that
			$f\alpha=v$ and $u_1\alpha=\gamma$.
			
			{\it 3 $\Rightarrow$ 1. }Let be $X_1\overset{x}\ra X_0$ an object of
			$\flcp$, arrows $(v_1,v_0)$ and $(v'_1,v'_0)\col x\ra f$, and
			a 2-arrow $\gamma\col(u_1,u_0)(v_1,v_0)\Ra (u_1,u_0)(v'_1,v'_0)$
			(i.e.\ $\gamma\col X_0\ra B_1$ such that
			$u_1v_1-u_1v'_1=\gamma x$ and $u_0v_0-u_0v'_0=g\gamma$).  Then
			$A_0\xleftarrow{v_0-v'_0}X_0\overset{\gamma}\longrightarrow B_1$ is a rival of
			the pullback and there exists a unique $\alpha\col X_0\ra A_1$ such
			that $f\alpha = v_0-v'_0$ and $u_1\alpha=\gamma$.  We check that
			$v_1-v'_1=\alpha x$ by testing this equality with  $f$ and $u_1$, which are
			jointly monomorphic.
		\end{proof}
	
	We can also characterise the full arrows in $\flcp$.  To do this,
	we use anticipatively the construction of the kernel and of the cokernel
	(described in diagrams \ref{diagnumeroun} and \ref{diagnumerodeux}).

	\begin{pon}\index{full arrow!in C2+@in $\flcp$}
		Let $(u_1,u_0)\col f\ra g$ be a morphism in $\flcp$.  The following conditions
		are equivalent:
		\begin{enumerate}
			\item $(u_1,u_0)$ is full;
			\item $(u_1,u_0)$ is 0-full;
			\item in diagram \ref{diagnumerodeux}, $\zeta k=q\kappa$;
			\item the square \eqref{flechedeflc} is exact;
			\item the sequence 	$A_1\xrightarrow{\vervec{-f}{u_1}}
				A_0\oplus B_1\xrightarrow{(u_0\;\, g)} B_0$ is exact.
		\end{enumerate}
	\end{pon}

		\begin{proof}
			It is obvious that condition 1 implies condition 2 and that
			conditions 4 and 5 are equivalent, because $\C$ is abelian.
			
			{\it 2 $\Rightarrow$ 3. }It suffices to apply condition 2 to diagram
			\ref{diagnumerotrois}.
			
			{\it 3 $\Rightarrow$ 1. }Let $X_1\overset{x}\ra X_0$ and $Y_1\overset{y}\ra Y_0$
			be two objects of $\flcp$, $(a_1,a_0)$ and $(b_1,b_0)\col x\ra f$, $(c_1,c_0)$
			and $(d_1,d_0)\col g\ra y$ be two arrows in $\flcp$, and let
			$\alpha\col (u_1,u_0)(a_1,a_0)\Ra(u_1,u_0)(b_1,b_0)$ and $\beta\col
			(c_1,c_0)(u_1,u_0)\Ra (d_1,d_0)(u_1,u_0)$ be two 2-arrows in $\flcp$.
			We must prove that $d_1\alpha+\beta a_0=\beta b_0+c_1\alpha$.
			
			Since $\alpha\col X_0\ra B_1$ is a 2-arrow of $\flcp$, we have
			$u_0(a_0-b_0)=g\alpha$.  There is thus a factorisation
			$\bar{\alpha}\col X_0\ra K$ such that $k\bar{\alpha}=a_0-b_0$
			and $\kappa\bar{\alpha}=\alpha$.  Dually, we have a factorisation
			$\bar{\beta}\col Q\ra Y_1$ such that $\bar{\beta}q=c_1-d_1$
			and $\bar{\beta}\zeta=\beta$.  Then, by condition 3, the following diagram
			commutes and we have the required equation.
			\begin{xym}\xymatrix@=20pt{
				&&B_1\ar[dr]^{q}\ar@/^0.8pc/[drr]^{c_1-d_1}
				\\ X_0\ar[r]^{\bar{\alpha}}\ar@/^0.8pc/[urr]^{\alpha}
					\ar@/_0.8pc/[drr]_{a_0-b_0}
				&K\ar[ur]^{\kappa}\ar[dr]_k
				&&Q\ar[r]^{\bar{\beta}}
				&Y_1
				\\ &&A_0\ar[ur]_{\zeta}\ar@/_0.8pc/[urr]_(0.4){\beta}
			}\end{xym}
			
			{\it 3 $\Leftrightarrow$ 5. }The sequence is exact if and only if we have
			 $(\zeta\;\, q)\vervec{-k}{\kappa}=0$ in diagram \ref{diagnumeroun},
			which is the case if and only if $\zeta k =q\kappa$.
		\end{proof}
	
	We notice that, in $\flcp$, an arrow is fully faithful if and only
	if it is full and faithful. 
	We can now combine the previous propositions.
	
	\begin{coro}\label{pasequ}
		Let $(u_1,u_0)\col f\ra g$ be a morphism in $\flcp$.  The following conditions
		are equivalent:
		\begin{enumerate}
			\item $(u_1,u_0)$ is faithful and fully cofaithful;
			\item $(u_1,u_0)$ is faithful, full and cofaithful;
			\item $(u_1,u_0)$ is fully faithful and cofaithful;
			\item the square \eqref{flechedeflc} is cartesian;
			\item the sequence 	$0\longrightarrow A_1\xrightarrow{\vervec{-f}{u_1}}
				A_0\oplus B_1\xrightarrow{(u_0\;\, g)} B_0 \longrightarrow 0$ is exact.
		\end{enumerate}
	\end{coro}
	
	By comparing the previous corollary and the following proposition, it becomes clear
	that in general $\flcp$ is not 2-abelian,
	because a faithful and fully cofaithful arrow is not always an equivalence
	(which must be the case in a 2-abelian $\Gpd$-category). The equivalences
	correspond to the cases where the exact sequence of condition 5 of the previous
	corollary \emph{splits}.
	
	\begin{pon}\label{caracequflc}
		Let $(u_1,u_0)\col f\ra g$ be a morphism in $\flcp$.  The following conditions
		are equivalent:
		\begin{enumerate}
			\item $(u_1,u_0)$ is an equivalence;
			\item there exist arrows $v_1, v_0, \varepsilon, \eta$, as
				in the following diagram:
				\begin{xym}\xymatrix@=40pt{
					A_1\ar[r]_{u_1}\ar[d]^f
					&B_1\ar[d]_g\ar@/_1pc/[l]_{v_1}
					\\ A_0\ar[r]^{u_0}\ar@/^1pc/[u]^\eta
					&B_0\ar@/_1pc/[u]_\varepsilon\ar@/^1pc/[l]^{v_0}
				}\end{xym}
				such that
				\begin{align*}
					gu_1&=u_0f, &fv_1&=v_0g,\\
					\varepsilon u_0 &= u_1\eta,  &\eta v_0 &= v_1\varepsilon,
				\end{align*}\begin{align*}
					\eta f+v_1u_1&=1_{A_1}, &f\eta + v_0 u_0 &= 1_{A_0},\\
					\varepsilon g+u_1 v_1 &=1_{B_1}, &g\varepsilon + u_0v_0 &=1_{B_0};
				\end{align*}
			\item the sequence $0\longrightarrow A_1\xrightarrow{\vervec{-f}{u_1}}
				A_0\oplus B_1\xrightarrow{(u_0\;\, g)} B_0\longrightarrow 0$ is split exact.
		\end{enumerate}
	\end{pon}

		\begin{proof}
			To give an inverse $(v_1,v_0)\col g\ra f$ and 2-arrows
			$\eta\col 1_f\Ra(v_1,v_0)\circ(u_1,u_0)$ and
			$- \varepsilon\col (u_1,u_0)\circ(v_1,v_0)\Ra 1_g$ satisfying the
			triangular identities amounts exactly to give the data of condition 2.
			
			Moreover, to give arrows
			$B_0\xrightarrow{\vervec{v_0}{-\varepsilon}}
			A_0\oplus B_1\xrightarrow{(\eta\;\, v_1)} A_1$
			which forming with sequence \eqref{carrecomsuit} a biproduct
			also amounts to give the data of condition 2.
		\end{proof}

\subsection{Construction of limits and colimits}

	We will now construct the (co)kernel, (co)pip, (co)root, as well as
	$\Omega$, $\Sigma$, $\pi_0$ and $\pi_1$ in $\flcp$, and characterise the
	arrows which are the kernel of their cokernel and the root of their copip.
	As the characterisation of equivalences suggests, what we must add
	to (co)faithful and fully (co)faithful, to get normal (co)faithful and normal fully
	(co)faithful, is that the corresponding exact sequence split.
	
	First, if $A_1\overset{f}\ra A_0$ and $B_1\overset{g}\ra B_0$ are
	two objects of $\flcp$, their biproduct exists: it is computed pointwise
	(the object is $A_1\oplus B_1\xrightarrow{f\oplus g}A_0\oplus B_0$).
	
	Let us give now a construction of the kernel and of the cokernel
	(which appear in \cite{Grandis2001b}).
	In the following diagram, $\vervec{-k}{\kappa}$ is
	the kernel of $(u_0, g)$ and $(\zeta, q)$ is the cokernel of
	$\vervec{-f}{u_1}$.  Moreover, there exists
	comparison arrows $k'$ and $q'$ making the triangles commute.
	
	\begin{xym}\label{diagnumeroun}\xymatrix@=40pt{
			&K\ar[d]^{\vervec{-k}{\kappa}}
			\\ A_1\ar[r]_-{\vervec{-f}{u_1}}\ar[ur]^{k'}
			&A_0\oplus B_1\ar[d]_-{(u_0\;\, g)}
				\ar[r]^-{(\zeta\;\,q)}
			&Q\ar[dl]^{q'}
			\\ &B_0
	}\end{xym} 
	
	This gives the following sequence of arrows and 2-arrows.  In this diagram,
	the left square is the kernel of $(u_1,u_0)$ and the right square
	is its cokernel.  We can also notice that $(K,k,\kappa)$ is the pullback
	 of $u_0$ and $g$, whereas dually $(Q,q,\zeta)$ is the pushout of $u_1$ and $f$.
	 \index{kernel!in C2+@in $\flcp$}\index{cokernel!in C2+@in $\flcp$}
	\begin{xym}\label{diagnumerodeux}\xymatrix@=40pt{
		A_1\ar@{=}[r]\ar[d]_{k'}
		&A_1\ar[d]_(0.3)f\ar[r]^{u_1}
		&B_1\ar[d]^(0.7)g\ar[r]^q
		&Q\ar[d]^{q'}
		\\ K\ar[r]_k\ar[urr]^(0.3)\kappa
		&A_0\ar[r]_{u_0}\ar[urr]_(0.7)\zeta
		&B_0\ar@{=}[r]
		&B_0
	}\end{xym}
	We can also present this in the following form.
	\begin{xym}\label{diagnumerotrois}\xymatrix@=40pt{
		k'\ar[r]^{(1,k)}\rrlowertwocell_0<-10>{_<2.8>\kappa}
		& f\ar[r]^{(u_1,u_0)}\rruppertwocell^0<10>{^<-2.8>\zeta}
		& g\ar[r]_{(q,1)}
		& q'
	}\end{xym}
	
	It is then possible to make explicit the definition of exact sequence and the construction of homology.
	This has been done by Grandis \cite{Grandis2001b}; the equivalence between the two
	constructions of homology is a form of the two-square lemma, as it was stated in \cite{Fay1989a}.
	
	Let us now construct the kernel of the cokernel of an arrow $(u_1,u_0)$.
	In the following diagram, the right square is the cokernel of $(u_1,u_0)$,
	the right front face of the prism is the kernel of this cokernel, and the left front face is the comparison morphism between $(u_1,u_0)$ and the kernel of its cokernel.
	\begin{xym}\xymatrix@=20pt{
		A_1\ar[dd]_f\ar[rr]^{u_1}\ar[dr]_{u_1}
		&&B_1\ar[dd]^g\ar[rr]^q
		&&Q\ar[dd]^{q'}
		\\ &B_1\ar[dd]^(0.3)q\ar@{=}[ur]
		\\ A_0\ar'[r]^{u_0}[rr]\ar[dr]_\zeta
		&&B_0\ar@{=}[rr]
		&&B_0
		\\ &Q\ar[ur]_{q'}
	}\end{xym}
	
	\begin{pon}\label{flcnoyconoy}
		Let $(u_1,u_0)\col f\ra g$ be a morphism in $\flcp$.  The following conditions
		are equivalent:
		\begin{enumerate}
			\item $(u_1,u_0)$ is normal faithful;
			\item $u_1$ and $f$ are “jointly split monomorphic”,
				i.e.\ there exist $A_0\overset{\eta}\ra A_1$ and
				$B_1\overset{v_1}\ra A_1$ such that
				\begin{eqn}
					\eta f+v_1u_1=1_{A_1};
				\end{eqn}
			\item $\vervec{-f}{u_1}$ is a split monomorphism;
			\item the sequence $0\longrightarrow A_1\xrightarrow{\vervec{-f}{u_1}} A_0\oplus B_1$
				is split exact.
		\end{enumerate}
	\end{pon}

		\begin{proof}
			The morphism $(u_1,u_0)$ is normal faithful if and only
			if the comparison arrow $(u_1,\zeta)\col f\ra q$ is an equivalence.
			By Proposition \ref{caracequflc}, this is the case if and only
			if the sequence
			\begin{eqn}
				0\longrightarrow A_1\xrightarrow{\vervec{-f}{u_1}}
				A_0\oplus B_1\xrightarrow{(\zeta\;\, q)} Q\longrightarrow 0
			\end{eqn}
			is split exact, which is the case if and only if
			$\vervec{-f}{u_1}$ is a split monomorphism, since this sequence
			is always right exact.  Condition 2 is a translation of
			condition 3.
		\end{proof}

	The constructions of $\Omega$ and $\Sigma$ are special cases of those of the kernel and of the cokernel. First, $\Omega(A_1\overset{f}\ra A_0)$ is 
	given by the following diagram.
	\begin{xym}\xymatrix@=30pt{
		0\ar[r]\ar[d]
		&A_1\ar[d]^f
		\\ {\Ker f}\ar[ur]^-{k_f}\ar[r]_-0
		&A_0
	}\end{xym}
	And, dually,
	$\Sigma(A_1\overset{f}\ra A_0)$ is the arrow $\Coker f\ra 0$, with
	$q_f\col 0\Ra 0\col f\ra \Sigma f$.
	
	So $\pi_1 (f)$ is the arrow $\Ker f\ra 0$ and $\pi_0 (f)$ is
	the arrow $0\ra\Coker f$.  In the following diagram, the left commutative square
	is $\varepsilon_f$ and the right one is $\eta_f$.
	\begin{xym}\xymatrix@=30pt{
		{\Ker f}\ar[r]^-{k_f}\ar[d]
		&A_1\ar[d]^f\ar[r]
		&0\ar[d]
		\\ 0\ar[r]
		&A_0\ar[r]_-{q_f}
		&{\Coker f}
	}\end{xym}
	
	From the previous proposition follows a corollary characterising the
	objects canonically equivalent to their $\pi_0$.  These are not, in general, all
	discrete objects
	(which are the $f\col\flcp$ which are monomorphisms in $\C$, by
	Corollary \ref{caracdiscreteflc}), but
	only the split ones.
	
	\begin{coro}\label{flcaracpiz}
		Let be $A_1\overset{f}\ra A_0$ in $\flcp$. The following conditions
		are equivalent:
		\begin{enumerate}
			\item $\eta_f$ is an equivalence $f\simeq\pi_0 f$;
			\item $f$ is a split monomorphism in $\C$;
			\item $f$ is equivalent to $0\ra B$ for some $B$.
		\end{enumerate}
	\end{coro}
	
	Let us turn now to the pip and the root.  Let be an arrow $(u_1,u_0)\col f\ra g$
	in $\flcp$.  Its pip is the diagonal of the left square of the following diagram.
	\begin{xym}\xymatrix@=40pt{
		0\ar[r]\ar[d]
		&A_1\ar[d]^f\ar[r]^{u_1}
		&B_1\ar[d]^g
		\\ {\Ker\vervec{-f}{u_1}}\ar[ur]^-{\pi}\ar[r]_0
		&A_0\ar[r]_-{u_0}
		&B_0
	}\end{xym}
	
	Let be now $\alpha\col (0,0)\Ra (0,0)\col f\ra g$.  The root of $\alpha$
	is the left square of the following diagram, where $f'$ is induced
	by the fact that $\alpha f=0$.
	\begin{xym}\xymatrix@=40pt{
		A_1\ar@{=}[r]\ar[d]_{f'}
		&A_1\ar[d]_f\ar[r]^0
		&B_1\ar[d]^g
		\\ {\Ker\alpha}\ar[r]_{k_\alpha}
		&A_0\ar[r]_0\ar[ur]^\alpha
		&B_0
	}\end{xym}
	
	The following diagram describes the root of the copip of $(u_1,u_0)$.
	The diagonal of the right square is its copip, the right front face of the prism
	is the root of the copip; we get $\im(u_0\;\,g)$, because
	it is the kernel of $\rho$, which is itself the cokernel of $(u_0\;\,g)$.
	The left front face of the prism is the comparison morphism between $(u_1,u_0)$
	and the root of its copip.
	\begin{xym}\xymatrix@=20pt{
		A_1\ar[dd]_f\ar[rr]^{u_1}\ar[dr]_{u_1}
		&&B_1\ar[dd]^g\ar[rr]^-0
		&&{\Coker(u_0\;\,g)}\ar[dd]
		\\ &B_1\ar[dd]^(0.3){g'}\ar@{=}[ur]
		\\ A_0\ar[dr]_{u'_0}\ar'[r]^{u_0}[rr]\ar[dr]
		&&B_0\ar[rr]\ar[uurr]_{\rho}
		&&0
		\\ &{\im (u_0\;\,g)}\ar[ur]_-{k}
	}\end{xym}
	
	\begin{pon}\label{flcraccopep}
		Let $(u_1,u_0)\col f\ra g$ be a morphism in $\flcp$.  The following conditions
		are equivalent:
		\begin{enumerate}
			\item $(u_1,u_0)$ is normal fully faithful;
			\item the square \eqref{flechedeflc} is a “split pullback”
				(i.e.\ it is a pullback and 
				$u_1$ and $f$ are “jointly split monomorphic”);
			\item $\vervec{-f}{u_1}$ is the kernel of $(u_0\;\, g)$
				and is split;
			\item the sequence 	$0\longrightarrow A_1\xrightarrow{\vervec{-f}{u_1}}
				A_0\oplus B_1\xrightarrow{(u_0\;\, g)} B_0$ is split exact.
		\end{enumerate}
	\end{pon}
	
		\begin{proof}
			The morphism $(u_1,u_0)$ is normal fully faithful if and only
			if the comparison arrow $(u_1,u'_0)\col f\ra g'$ of the previous diagram
			is an equivalence. By Proposition \ref{caracequflc},
			this is the case if and only if the upper sequence
			of the following diagram is split exact, which is equivalent
			to the lower sequence being split exact, because $k$ is a
			monomorphism.
		\end{proof}
			\begin{xym}\xymatrix@=40pt{
				0\ar[r]
				&A_1\ar[r]^-{\vervec{-f}{u_1}}
				&A_0\oplus B_1\ar@{->>}[r]^-{(u'_0\;\,g')}\ar[dr]_-{(u_0\;\,g)}
				&{\im(u_0\;\,g)}\ar@{>->}[d]^k\ar[r]
				&0
				\\ &&&B_0
			}\end{xym}

\subsection{Characterisation of 2-abelianness of $\flcp$}\label{secaxchflc}

	We say that an arrow $A_1\overset{f}\ra A_0$ in $\C$
	is \emph{split} \cite{Borceux2004a}\index{arrow!split}\index{split arrow}
	if there exists $A_0\overset{g}\ra A_1$ such that
	\begin{eqn}
		fgf = f \text{ and } gfg = g.
	\end{eqn}
	Then, a monomorphism (or an epimorphism) is split (in the usual sense) if and only if
	it is split in this sense.
	
	We can characterise the split arrows as objects of the
	$\Gpdp$-category $\flcp$.
	
	\begin{pon}\label{caracflechscind}
		Let be $A_1\overset{f}\ra A_0$ in $\flcp$. The following conditions
		are equivalent:
		\begin{enumerate}
			\item $f$ is split;
			\item there exists $A_0\overset{g}\ra A_1$ such that $fgf=f$;
			\item $f$ is equivalent in $\flcp$ to an arrow $0$;
			\item $f$ is equivalent in $\flcp$ to $\Ker f\overset{0}\ra\Coker f$;
			\item $f$ is equivalent to $\pi_0 f\oplus \pi_1 f$.
		\end{enumerate}
	\end{pon}
	
		\begin{proof}
			{\it 1 $\Leftrightarrow$ 2. }Clearly condition
			1 implies condition 2 and, conversely,
			if there exists $g$ such that $fgf=f$, then $g'\eqdef gfg$ will be such that $fg'f=f$
			and $g'fg'=g'$.
			
			{\it 4 $\Rightarrow$ 3. }This is obvious.
			
			{\it 3 $\Rightarrow$ 2. }If we are in the situation of condition 2
			of Proposition \ref{caracequflc}, with $g=0$, we have
			\begin{eqn}
				f\eta f=(1-v_0u_0)f=f-v_0u_0f=f-v_00u_1=f.
			\end{eqn}
			
			{\it 2 $\Rightarrow$ 4. }We start from the situation
			of the following diagram, with $f\eta f=f$, $v_1=\Ker f$ and $u_0=\Coker f$.
			\begin{xym}\xymatrix@=30pt{
				A_1\ar[d]^f
				&{\Ker f}\ar[l]_-{v_1}\ar[d]^0
				\\ A_0\ar@/^1pc/[u]^\eta\ar[r]_-{u_0}
				&{\Coker f}
			}\end{xym}
			Since $f(1-\eta f)=0$, there is a factorisation $u_1\col A_1\ra\Ker f$ such
			that $v_1u_1 = 1-\eta f$. Dually, we have $v_0\col \Coker f\ra A_0$
			such that $v_0u_0=1-f\eta$.  Finally, since $v_1$ is a monomorphism and
			$v_1u_1\eta f=(1-\eta f)\eta f=\eta f-\eta f\eta f=0=v_10$,
			we have $u_1\eta f =0$, which induces a factorisation
			$\varepsilon\col\Coker f\ra\Ker f$ such that $\varepsilon u_0=u_1\eta$.
			It is easy to check the
			other conditions of condition 2 of Proposition \ref{caracequflc},
			which shows that we have constructed an equivalence between $f$
			and $\Ker f\overset{0}\ra\Coker f$.
			
			{\it 4 $\Leftrightarrow$ 5. }Condition 5 is a direct translation of condition 4, by using
			the construction of biproducts in $\flcp$.
		\end{proof}
	
	Remark: condition 5 has the advantage of being expressed in purely $\Gpdp$-categorical
	terms, it has a meaning independently of the special structure
	of $\flcp$.  In the context of 2-groups, Elgueta \cite[Section 2.8]{Elgueta2006a} defines
	a \emph{split} 2-group as a 2-group $\cat{G}$ such that the following sequence
	is split exact:
	\begin{eqn}
		0\longrightarrow\pi_1\G\overset{\varepsilon_\G}\longrightarrow\G\overset{\eta_\G}\longrightarrow\pi_0\G\longrightarrow 0.
	\end{eqn}
	
	We can now give the following characterisation of the abelian categories
	$\C$ such that $\flcp$ is 2-abelian. These are the abelian categories
	satisfying the \emph{Von Neumann axiom} (condition 1),\index{axiom!Von Neumann}%
	\index{Von Neumann axiom} in the sense
	of Borceux-Rosický \cite{Borceux2004a}, which is equivalent to the axiom of choice%
	\index{choice, axiom of}\index{axiom!of choice}
	(condition 3) for an abelian category.  Most of the equivalences
	are well-known. 
	
	Let us assume that the axiom of choice holds in the category of sets.
	If $\C$ is $\caspar{Mod}_R$ for a ring $R$,
	$\C$ satisfies the axiom of choice if and only if the ring $R$ is
	semisimple (see \cite[Corollary 13.10]{Anderson1992a}
	or \cite[Theorem 5.2.13]{Hazewinkel2004a}).
	In particular, this is the case for a field $k$. So the Baez-Crans 2-vector spaces
	form a 2-abelian $\Gpdp$-category if we accept
	the axiom of choice.
	
	\begin{thm}\label{caracdabflabplus}
		Let $\C$ be an abelian category. The following conditions are
		equivalent.
		\begin{enumerate}
			\item Every arrow in $\C$ splits.
			\item Every monomorphism in $\C$ splits.
			\item Every epimorphism in $\C$ splits.
			\item Every short exact sequence in $\C$ splits.
			\item Every object of $\C$ is projective.
			\item Every object of $\C$ is injective.
			\item $\flcp$ is 2-abelian.
		\end{enumerate}
	\end{thm}
	
		\begin{proof}
			{\it 1 $\Rightarrow$ 2,3. }This is obvious.
			
			{\it 2 $\Leftrightarrow$ 3 $\Leftrightarrow$ 4. }This follows from the fact that a short exact sequence splits if and only
			if one of the two sides splits.

			{\it 2,3 $\Rightarrow$ 1. }In the factorisation $f=me$ of $f$ as an
			epimorphism followed by a
			monomorphism, $m$ and $e$ split, by conditions 2 and 3:
			$ea=1$ and $bm=1$, and by setting $g\eqdef ab$, we have $fgf=meabme=me=f$.
			
			{\it 4 $\Leftrightarrow$ 5. }This follows from the fact that an
			object $P$ is projective if and only if every short exact sequence of the form
			$0\ra A\ra B\ra P\ra 0$ splits.
			
			{\it 4 $\Leftrightarrow$ 6. }The proof is dual.

			{\it 7 $\Rightarrow$ 2. }By Proposition \ref{caracdispreadd}, if $\flcp$
			is 2-abelian, every discrete object in $\flcp$ (i.e.
			every monomorphism) is equivalent to its $\pi_0$ (i.e.
			is a split monomorphism, by Corollary \ref{flcaracpiz}).
			
			{\it 2,3 $\Rightarrow$ 7. }	By Proposition \ref{caracfidflc}, if an arrow
			$(u_1,u_0)\col f\ra g$ is 0-faithful, $\vervec{-f}{u_1}$ is a monomorphism.
			Then, by condition 2, $\vervec{-f}{u_1}$ is a split monomorphism and,
			by Proposition \ref{flcnoyconoy}, 
			$(u_1,u_0)$ is normal faithful.  We proceed in the same way
			to prove that every fully 0-faithful arrow is normal
			by using Propositions \ref{caracplfidflc} and \ref{flcraccopep}.
			Dually, we prove that every 0-cofaithful arrow is normal and that
			every fully 0-cofaithful arrow is normal
			by using condition 3.
		\end{proof}
	
	In the case where these equivalent conditions hold, we also get
	that $\Con(\flcp)\simeq\Dis(\flcp)\simeq\C$, because the discrete objects in $\flcp$
	are then equivalent to the arrows with zero domain, by Proposition
	\ref{caracflechscind}.
	
	\begin{pon}
		If $\C$ is an abelian category in which the axiom of choice holds, then $\flcp$
		is a good 2-abelian $\Gpd$-category.
	\end{pon}
	
		\begin{proof}
			Let $(u_1,u_0)$ be an arrow in $\flcp$. We must prove that
			the comparison arrow $\pi_0\Ker(u_1,u_0)\ra\Ker\pi_0(u_1,u_0)$
			is an epimorphism (and dually, the comparison arrow $\Coker\pi_1(u_1,u_0)
			\ra\pi_1\Coker(u_1,u_0)$ must be a monomorphism).  Since
			every object in $\flcp$ is split, by Proposition \ref{caracflechscind},
			we can assume that the domain
			and the codomain of $(u_1,u_0)$ are zero morphisms.
			
			Let us consider the following diagram. The first two rows
			show the arrow $(u_1,u_0)$ and its kernel. The third row is the result
			of the application of $\pi_0$ to these first two rows ($\pi_0$
			being simply the cokernel).  So $\Ker u_0\oplus\Coker u_1$
			is $\pi_0\Ker(u_1,u_0)$ and $\Ker u_0$ is $\Ker\pi_0(u_1,u_0)$;
			the comparison arrow is $p_1$, which is an epimorphism.
		\end{proof}
			\begin{xym}\xymatrix@=30pt{
				A_1\ar@{=}[r]\ar[d]_{\vervec{0}{u_1}}
				&A_1\ar[d]^(0.6)0\ar[r]^{u_1}
				&B_1\ar[d]^0
				\\ {\Ker u_0\oplus B_1}\ar[urr]^-(0.3){p_2}
					\ar[r]_-{kp_1}\ar[d]_{1\oplus q}
				&A_0\ar[r]_{u_0}\ar@{=}[d]
				&B_0\ar@{=}[d]
				\\ {\Ker u_0\oplus\Coker u_1}\ar[r]^-{kp_1}\ar@{->>}[d]_-{p_1}
				&A_0\ar[r]_{u_0}
				&B_0
				\\ {\Ker u_0}\ar@{>->}[ur]_-k
			}\end{xym}

\subsection{Discretely presentable objects}\label{secobjdiscpres}

	To understand why $\flcp$ is not in general 2-abelian, we consider
	the case where $\C$ is $\caspar{Dis}(\D)$, for a 2-abelian $\Gpd$-category $\D$.
	We can restrict the equivalence between
	$\Fid(\D)$ and $\Cofid(\D)$ (given by the kernel and the cokernel)
	to the faithful arrows with discrete codomain
	and to the cofaithful arrows with discrete domain.  By Proposition \ref{fiddisc},
	a faithful arrow with discrete codomain also has a discrete domain and so
	is simply an arrow in $\caspar{Dis}(\D)$.  We have thus an equivalence between
	categories 
	\begin{xym}\label{diagequfldisddis}\xymatrix@=40pt{
		{\fl{\caspar{Dis}(\D)}}\ar@<+2mm>[r]^-{\Coker}\ar@{}[r]|-{\simeq}
		&{\D^{\mathrm{dis}},}\ar@<+2mm>[l]^-{\Ker}
	}\end{xym}
	where $\D^{\mathrm{dis}}$ is the full sub-$\Gpd$-category of $\Cofid(\D)$ whose
	objects are the cofaithful arrows
	with discrete domain; it is a category, because there is at most one 2-arrow between
	two arrows, by the discreteness of the domains and the cofaithfulness of the arrows.  We
	call the objects of $\D^{\mathrm{dis}}$ the \emph{discretely presented objects}%
	\index{object!discretely presented}\index{discretely presented object}
	of $\D$, because they are the quotients of an arrow between discrete objects.
	These equivalent categories are abelian.
	 	
	In $\CGS$, a discretely presented object is a surjective symmetric monoidal functor
	$A\col A_0\twoheadrightarrow\A$ where $A_0$ is an abelian group seen as a discrete
	2-group and where $\A$ is a symmetric 2-group.
	By taking the image $\mathrm{Im}_{\mathrm{pl}}^1\, A$, we can see the surjective functor
	as being the identity at the level of objects.  In other words, to give a discretely presented object in $\CGS$ amounts to give an abelian group $A_0$ and a symmetric 2-group
	$\A$ whose objects are those of $A_0$ and the structure of symmetric 2-group
	is that of $A_0$
	(it is thus strictly described); this is what we generally call a strict symmetric
	2-group.  
	The above equivalence becomes the known equivalence between strict symmetric 2-groups
	and arrows in $\Ab$.  The cokernel in $\CGS$
	of an arrow in $\Ab$ is the “realisation” of the arrow as a symmetric 2-group.
	
	There is a forgetful functor $U\col\D^{\mathrm{dis}}\ra\D$, which forgets the domain.
	This forgetful functor factors through two images:
	\begin{eqn}
		\D^{\mathrm{dis}}\longrightarrow\D^{\mathrm{dis}}_{+}\longrightarrow
		\D^{\mathrm{dis}}_{++}\longrightarrow\D.
	\end{eqn}
	The first, that the we denote by $\D^{\mathrm{dis}}_{+}$, has as objects and arrows those of 
	$\D^{\mathrm{dis}}$ and as 2-arrows the arrows of $\D$.  In other words,
	$\D^{\mathrm{dis}}_{+}$ can be described in the following way:
	\begin{itemize}
		\item {\it objects: } these are the arrows $A_0\overset{a}\twoheadrightarrow A$
			in $\D$, where $A_0$ is discrete and $a$ is cofaithful;
		\item {\it arrows: } an arrow from $A_0\overset{a}\twoheadrightarrow A$
			to $B_0\overset{b}\twoheadrightarrow B$ is given by arrows
			$f_0\col A_0\ra B_0$ and $f\col A\ra B$ and a 2-arrow 
			$\varphi\col fa\Ra bf_0$:
		\item {\it 2-arrows: } a 2-arrow $(f_0,\varphi,f)\Ra (f'_0,\varphi',f')$
			is simply a 2-arrow $\alpha\col f\Ra f'$ in $\D$.
	\end{itemize}
	There is no compatibility condition with  $\varphi$
	and $\varphi'$ for $\alpha$ (it is this condition which makes the 2-arrows in
	$\D^{\mathrm{dis}}$ unique).
	
	We can transfer $\D^{\mathrm{dis}}_{+}$
	to the other side of the equivalence \ref{diagequfldisddis}: actually, the $\Gpd$-category
	$\D^{\mathrm{dis}}_{+}$ is equivalent to $\caspar{Dis}(\D)^\deux_{+}$.  To
	see this, it suffices to add to equivalence \ref{diagequfldisddis}
	a correspondance at the level of 2-arrows: by the universal property
	of the cokernel, a homotopy in $\caspar{Dis}(\D)^\deux_{+}$ induces a natural transformation
	in $\D^{\mathrm{dis}}_{+}$ and, by the universal property of the kernel,
	a natural transformation induces a homotopy.  If $K$ is a field
	and $K\dis$ is the field $K$ seen as a discrete 2-ring, we can take
	$\D\eqdef \dMod_{K\dis}$.  Then $\D^{\mathrm{dis}}_{+}\simeq\Dis(\D)^{\deux}_+$ is the
	$\Gpd$-category of Baez-Crans 2-vector spaces on $K$.
	
	The reason for which $\D^{\mathrm{dis}}_{+}$ 
	(and thus $\caspar{Dis}(\D)^\deux_{+}$) is not in general
	2-abelian is clear, as we will see.
	It is easy to check that an arrow $(f_0,\varphi,f)$ in $\D^{\mathrm{dis}}_{+}$
	is faithful [resp.\ fully faithful, cofaithful, fully cofaithful] if and only
	if $f$ is faithful [resp.\ fully faithful, cofaithful, fully cofaithful] in $\D$.
	Therefore $(f_0,\varphi,f)$ is fully faithful and cofaithful if and only if $f$
	is fully faithful and cofaithful in $\D$, i.e.\ if and only if $f$
	is an equivalence in $\D$.  But this does not imply in general that $(f_0,\varphi,f)$
	is an equivalence in $\D^{\mathrm{dis}}_{+}$, even if $f$ has an inverse $g$, because, 
	to define an inverse for $(f_0,\varphi,f)$, we also need a compatible arrow
	$g_0$ at the level of discrete objects.
	
	In $\CGS$, this is simply the fact that a strict symmetric monoidal functor
	between strict symmetric 2-groups can be an equivalence without its
	inverse being strict.
	This is the reflection in $\CGS^{\mathrm{dis}}_+$ of the gap between Corollary \ref{pasequ} and
	Proposition \ref{caracequflc}.  We can use these propositions to give an
	example of strict monoidal functor which is an equivalence in $\CGS$ but not in
	$\CGS^{\mathrm{dis}}_+$: the exact sequence
	\begin{eqn}
		0\ra\Z\xrightarrow{2\cdot-}\Z\overset{q}\longrightarrow\Z_2\ra 0
	\end{eqn}
	does not split and so, in $\Ab^{\deux}_{+}$, the corresponding cartesian square
	\begin{xym}\xymatrix@=30pt{
		{\Z}\ar[r]\ar[d]_{2\cdot -}
		&0\ar[d]
		\\ {\Z}\ar[r]_-q
		&{\Z_2}
	}\end{xym}
	is not an equivalence.  If we translate this in $\CGS^{\mathrm{dis}}_+$
	by computing the cokernels in $\CGS$ of these arrows, we get a functor
	\begin{eqn}\label{exampleequpasequ}
		\Coker(2\cdot-)\ra(\Z_2)\dis
	\end{eqn}
	where the left 2-group has as objects the integers and where there is exactly
	one arrow from $m$ to $n$ if they have the same parity, otherwise there is no arrow.
	The functor maps $n$ to $n\mod 2$.  This is a strict symmetric monoidal equivalence,
	but there is no strict inverse, because for that we would need a homomorphism
	$\alpha\col\Z_2\ra\Z$ such that $q\alpha = 1$, which does not exist.

\bigskip
Let us turn now to the second image of the forgetful functor $U$, which we denote by
$\D^{\mathrm{dis}}_{++}$, whose objects are those of $\D^{\mathrm{dis}}$ and whose
arrows and 2-arrows are those of $\D$.  In other words, $\D^{\mathrm{dis}}_{++}$
is the full sub-$\Gpd$-category of $\D$ whose objects are the \emph{discretely presentable} objects.%
\index{object!discretely presentable}\index{discretely presentable object}

\begin{pon}\label{dispresdab}
	Let $\D$ be a 2-abelian $\Gpd$-category. Then $\D^{\mathrm{dis}}_{++}$
	is 2-abelian.
\end{pon}

	\begin{proof}
		We check that this full sub-$\Gpd$-category of $\D$ is stable
		under the limits which interest us.  First, since $0$
		is discrete, the object $0$ is discretely presentable, through the identity
		$0\overset{1_0}\ra 0$.  Next, since $\mathrm{Dis}(\D)$ is a reflective
		sub-$\Gpd$-category of $\D$, the biproduct of two discrete objects
		is always discrete; moreover the biproduct of two cofaithful arrows
		is always cofaithful; thus the biproduct of two discretely presentable objects
		is discretely presentable.
		
		Next, every subobject of a discretely presentable object is discretely
		presentable, because if $A_0\overset{a}\twoheadrightarrow A$ is a cofaithful arrow
		with discrete domain and $B\overset{m}\rightarrowtail A$ is a faithful arrow,
		then the pullback of $a$ along $m$ is a cofaithful arrow
		(by the regularity of 2-abelian $\Gpd$-categories)
		$B_0\overset{b}\ra B$ with discrete domain (which we can check by using
		the facts that $A_0$ is discrete, that $m$ is faithful, and that the projections of the
		pullback are jointly faithful).  In particular, the kernel of an arrow
		between discretely presentable objects is discretely presentable.
		
		Finally, every quotient of a discretely presentable object is discretely
		presentable, because if $A_0\overset{a}\twoheadrightarrow A$ is a cofaithful arrow
		with discrete domain and $A\overset{p}\twoheadrightarrow B$ is a cofaithful arrow,
		then $A_0\xrightarrow{pa}B$ is a cofaithful arrow with discrete domain.
		In particular, the cokernel of an arrow between discretely presentable objects
		is discretely presentable.
		
		Therefore, the (fully) (co)faithful arrows in $\D^{\mathrm{dis}}_{++}$
		are the (fully) (co)faithful arrows of $\D$ between discretely
		presentable objects, since we can characterise them in terms of (co)kernel or (co)pip.
		And the conditions of 2-abelian $\Gpd$-category hold because they do
		in $\D$.
	\end{proof}

Now, we wish to complete the following diagram by defining internally in $\caspar{Dis}(\D)$ a $\Gpd$-category $\caspar{Dis}(\D)^\deux_{++}$ equivalent
to $\D^{\mathrm{dis}}_{++}$.
\begin{xym}\xymatrix@=30pt{
	{\caspar{Dis}(\D)^\deux}\ar@{-}[d]_{\wr}\ar[r]
	&{\caspar{Dis}(\D)^\deux_{+}}\ar[r]\ar@{-}[d]_{\wr}
	&{\caspar{Dis}(\D)^\deux_{++}}\ar[r]\ar@{-}[d]^{\wr}
	&{\D}\ar@{=}[d]
	\\ {\D^{\mathrm{dis}}}\ar[r]
	&{\D^{\mathrm{dis}}_{+}}\ar[r]
	&{\D^{\mathrm{dis}}_{++}}\ar[r]
	&{\D}
}\end{xym}
The solution should be the analogue for chain complexes of length 1 of the internal anafunctors and ananatural transformations. These notions have been introduced 
by Michael Makkai \cite{Makkai1996a} in the case of internal categories in $\Ens$, and
generalised to internal categories in more general categories by Toby Bartels \cite{Bartels2006a}.
Bartels needed to use internal anafunctors to solve a problem very similar to the one which interests us: he wanted to define a 2-category
of 2-spaces as the 2-category of internal categories, internal functors and
internal natural transformations in a category of spaces, but this 2-category has not the desired properties,
because the internal functors which are full, faithful and surjective are not always
equivalences.

This suggests the following program: 
\begin{enumerate}
	\item to define internally in any abelian category $\C$ a
		$\Gpd$-category $\C^\deux_{++}$;
	\item to check that $\caspar{Dis}(\D)^\deux_{++}\simeq\D^{\mathrm{dis}}_{++}$,
		for any 2-abelian $\Gpd$-category $\D$;
	\item to check that $\C^\deux_{++}$ is 2-abelian.
\end{enumerate}
 If everything works, this would give for each abelian category $\C$ a 2-abelian $\Gpd$-category $\D$ such that $\C\simeq\caspar{Dis}(\D)\simeq\Con(\D)$.

\chapter{Towards Tierney theorem}

An important property of abelian categories is that a category is abelian (in the sense that it has a zero object, finite products and coproducts, kernels and cokernels, that every monomorphism is normal and every epimorphism is normal) if and only if it is additive and Barr-exact (Tierney equation).  This is due to the fact that, in an additive category $\C$ with kernels, there is an equivalence between the category of equivalence relations in $\C$ and the category of monomorphisms in $\C$ and that this equivalence commutes with, on the one hand, the kernel and the kernel relation and, on the other hand, the cokernel and the quotient.  In other words, there is an equivalence between the kernel-quotient system $\Coker\!\adj\Ker$ and the kernel-quotient system $\Quot\adj\KRel$, allowing to translate all notions expressed in terms of one of these systems into notions in terms of the other.

Here is the program to give a 2-dimensional version of this theorem:
\begin{enumerate}
	\item  to give a 2-dimensional version of the kernel-quotient system
	 	$\Quot\adj\KRel$, with the equivalence 2-relations as congruences;
	\item to prove that equivalence 2-relations in a $\CGS$-category $\C$ with
		all the necessary limits are equivalent to reflexive 2-relations in $\C$
		and that the latter are equivalent to faithful arrows in $\C$, these equivalences
		commuting with the appropriate kernels; then the kernel-quotient system $\Coker\adj\Ker$ 
	will be equivalent to the system $\Quot\adj 2\text{-}\!\KRel$;
\item to use this equivalence to translate the conditions expressed in terms of the system
	 $\Coker\!\adj\Ker$ into conditions expressed in terms of the other kernel-quotient system.
\end{enumerate}
We must do the same three stages for a non-pointed version of the kernel-quotient system $\Coroot\adj\Pip$.

The last two stages are in progress, but the first stage is already completed. Here is a summary of this stage.

	Ross Street has defined a notion of exactness for 2-categories \cite{Street1982b}.
	This notion does not work as such for $\Gpd$-categories because in $\Gpd$ not every
	congruence (in the sense of Street) is the kernel of its quotient.  The problem
	is that the appropriate congruences for 2-categories are in some way order 2-relations,
	whereas the appropriate congruences for $\Gpd$-categories are equivalence 2-relations.
	
	\begin{df}\index{2-relation}
		Let $\C$ be a $\Gpd$-category and $A,B$ be objects of $\C$.  A \emph{2-relation}
		$R\col A\rel B$
		is a jointly faithful span $A\overset{d_0}\longleftarrow R\overset{d_1}\longrightarrow B$.
	\end{df}
	
	Toby Bartels \cite{Bartels2006a} defines 2-relations as being jointly fully faithful spans, which is not general enough for kernel 2-relations to be 2-relations.  The full faithfulness allows him to define in a simplified way the equivalence 2-relations, since in this case the arrows $r$, $t$, $s$ of the following definition are unique and the 2-arrows $\upsilon_0$,
	$\upsilon_1$, $\alpha$ and $\gamma$ necessarily exist.
	For a 2-relation $C\overset{d_0}\longleftarrow R\overset{d_1}\longrightarrow C$, 
	let us denote the pullback
	of $d_1$ and $d_0$ by $\varphi: d_1 p_0\Ra d_0 p_1$, where 
	$R\overset{p_0}\longleftarrow R^2\overset{p_1}\longrightarrow R$, and let us denote the pullback of $p_1$ and $p_0$ by $\psi: p_1 q_0\Ra p_0 q_1$, where 
	$R^2\overset{q_0}\longleftarrow R^3\overset{q_1} \longrightarrow R^2$.
		
	\begin{df}\index{2-relation!equivalence}%
	\index{equivalence 2-relation}
		Let $\C$ be a $\Gpd$-category with all pullbacks.
		An \emph{equivalence 2-relation} on $C\col \C$
		is a 2-relation $C\overset{d_0}\longleftarrow R\overset{d_1}\longrightarrow C$
		equipped with arrows $r, t, s$ and 2-arrows $\rho_0,\rho_1,\tau_0,\tau_1, \sigma_0,\sigma_1$,
		as in the following diagrams,
		\begin{xyml}\begin{gathered}\xymatrix@R=20pt@C=40pt{
			&C\ar@{=}[dl]\ar@{=}[dr]\ar[dd]_r\dltwocell\omit\omit{_<-3.5>{\rho_0\;}}
				\drtwocell\omit\omit{^<3.5>{\rho_1\;}}
			\\ C &&C
			\\ &R\ar[ul]^{d_0}\ar[ur]_{d_1}
		}\end{gathered}\;\;\;\;\begin{gathered}\xymatrix@=40pt{
			R\ar[d]_{d_0}
			&R^2\ar[l]_{p_0}\ar[r]^{p_1}\ar[d]_t\dltwocell\omit\omit{\tau_0\,}
				\drtwocell\omit\omit{^{\tau_1\;}}
			&R\ar[d]^{d_1}
			\\ C
			&R\ar[l]^{d_0}\ar[r]_{d_1}
			&C
		}\end{gathered}\end{xyml}
		\begin{xym}\xymatrix@R=20pt@C=40pt{
			&R\ar[dl]_{d_1}\ar[dr]^{d_0}\ar[dd]_s
				\dltwocell\omit\omit{_<-3.5>{\sigma_0\;}}
				\drtwocell\omit\omit{^<3.5>{\sigma_1\;}}
			\\ C &&C
			\\ &R\ar[ul]^{d_0}\ar[ur]_{d_1}
		}\end{xym}
		such that there exist 2-arrows
		\begin{xym}\xymatrix@=40pt{
			R\ar@{=}[dr]\ar[r]^{(1_R,rd_1)}\drtwocell\omit\omit{_<-3>{\;\,\upsilon_0}}
			&R^2\ar[d]^t
			&R\ar@{=}[dl]\ar[l]_{(rd_0,1_R)}\dltwocell\omit\omit{^<3>{\upsilon_1}}
			\\ &R
		}\end{xym}
		\begin{xyml}\begin{gathered}\xymatrix@=40pt{
			R^3\ar[r]^{(tq_0,p_2)}\ar[d]_{(p_0,tq_1)}\drtwocell\omit\omit{\alpha}
			&R^2\ar[d]^t
			\\ R^2\ar[r]_t &R
		}\end{gathered}\;\;\text{ and }\;\;\begin{gathered}\xymatrix@=40pt{
			R\ar[r]^{(1_R,s)}\ar[d]_{d_0}\drtwocell\omit\omit{\gamma}
			&R^2\ar[d]^t
			\\ C\ar[r]_r &R
		}\end{gathered}\end{xyml}
		satisfying the following conditions:
		\begin{enumerate}
			\item $d_0\upsilon_0 = d_0\pi_0\circ \tau_0(1_R,rd_1)$;
			\item $d_1\upsilon_0 = \rho_1 d_1\circ d_1\pi_1\circ\tau_1(1_R,rd_1)$;
			\item $d_0\upsilon_1 = \rho_0 d_0\circ d_0\pi_0\circ\tau_0(rd_0,1_R)$;
			\item $d_1\upsilon_1 = d_1\pi_1\circ \tau_1(rd_0,1_R)$;
			\item $d_0\pi_0^{-1}\circ d_0\pi_0\circ\tau_0(p_0,tq_1)\circ d_0\alpha 
				= \tau_0 q_0\circ d_0\pi_0\circ\tau_0(tq_0,p_2)$;
			\item $\tau_1 q_1\circ d_1\pi_1\circ\tau_1(p_0,tq_1)\circ d_1\alpha
				= d_1\pi_1^{-1}\circ d_1\pi_1\circ\tau_1(tq_0,p_2)$;
			\item $\rho_0 d_0\circ d_0\gamma = d_0\pi_0\circ\tau_0(1_R,s)$;
			\item $\rho_1d_0\circ d_1\gamma = \sigma_1\circ d_1\pi_1\circ\tau_1(1_R,s)$.
		\end{enumerate}
	\end{df}
	
	Thanks to the faithfulness, the 2-arrows $\upsilon_0$, $\upsilon_1$, $\alpha$
	and $\gamma$ are unique and we do not need to ask that they satisfy coherence conditions.
	In a $\Ens$-category seen as a locally discrete $\Gpd$-category, an equivalence 2-relation
	is nothing else than an internal groupoid.

		\begin{df}
			Let $\C$ be a $\Gpd$-category and $(R,d_0,d_1,\ldots)$ be an equivalence 2-relation.  A \emph{quotient}
			of $R$ consists of an arrow $q\col C\ra Q$ and a 2-arrow
			$\xi\col qd_0\Ra qd_1$, such that the following conditions hold:
			\begin{xyml}\label{condquotun}
			\begin{gathered}\xymatrix@=20pt{
				&&C\ar[dr]^q
				\\ C\ar[r]^r\ar@{=}@/^2.5ex/[urr]\ar@{=}@/_2.5ex/[drr]
					\urrtwocell\omit\omit{{\;\;\;\;\,\rho_0^{-1}}}
					\drrtwocell\omit\omit{\;\,\rho_1}
				& R\ar[ur]_{d_0}\ar[dr]^{d_1}\rrtwocell\omit\omit{\xi}
				&& Q
				\\ &&C\ar[ur]_q
			}\end{gathered}\;\;=\;\;
			\begin{gathered}\xymatrix@=30pt{C\ar[r]^q &Q}\end{gathered}
			\end{xyml}
			\begin{xyml}\label{condquotdeux}
			\begin{gathered}\xymatrix@!@R=15pt@C=10pt{
				&R\ar[rr]^{d_0}\ar[dr]^{d_1}
				&&C\ar[dr]^{q}\ar@{}[dl]|{\dir{=>}\;\xi}
				\\ R^2\ar[ur]^{p_0}\ar[dr]_{p_1}\rrtwocell\omit\omit{\varphi}
				&&C\ar[rr]^q\ar@{}[dr]|{\dir{=>}\xi}
				&&Q
				\\ &R\ar[ur]_{d_0}\ar[rr]_{d_1}
				&&C\ar[ur]_q
			}\end{gathered}\;\;=\;\;\begin{gathered}\xymatrix@!@R=15pt@C=10pt{
				&R\ar[rr]^{d_0}\ar@{}[dr]|{\dir{=>}\tau_0^{-1}}
				&&C\ar[dr]^q
				\\ R^2\ar[ur]^{p_0}\ar[dr]_{p_1}\ar[rr]^t
				&&R\ar[ur]^{d_0}\ar[dr]_{d_1}\rrtwocell\omit\omit{\xi}
					\ar@{}[dl]|{\dir{=>}\;\tau_1}
				&&Q
				\\ &R\ar[rr]_{d_1}
				&&C\ar[ur]_q
			}\end{gathered}
			\end{xyml}
			\begin{xyml}\label{condquottrois}\begin{gathered}\xymatrix@!@=20pt{
				&&C\ar[dr]^q
				\\ R\ar[r]^s\ar@/^2.5ex/[urr]^{d_1}\ar@/_2.5ex/[drr]_{d_0}
					\urrtwocell\omit\omit{\;\;\;\;\,\sigma_0^{-1}}
					\drrtwocell\omit\omit{\;\,\sigma_1}
				& R\ar[ur]_{d_0}\ar[dr]^{d_1}\rrtwocell\omit\omit{\xi}
				&& Q 
				\\ &&C\ar[ur]_q
			}\end{gathered}\;\;=\;\;\begin{gathered}\xymatrix@=20pt{
				&C\ar[dr]^q
				\\ R\ar[ur]^{d_1}\ar[dr]_{d_0}\rrtwocell\omit\omit{\;\;\;\;\xi^{-1}}
				&&Q
				\\ &C\ar[ur]_q
			}\end{gathered}\end{xyml}
			and such that
			\begin{enumerate}
				\item for every other pair $(q',\xi')$ satisfying the same
					 conditions, there exists a factorisation $q''\col Q\ra Q'$
					with a 2-arrow $\chi\col q''q\Ra q'$ such that
					$\chi d_1\circ q''\xi =  \xi'\circ \chi d_0$;
				\item for every pair $u,v \col  Q\ra Y$ and for every 2-arrow
					$\beta\col uq\Ra vq$ such that
					\begin{xym}\label{propunivquot}\xymatrix@!{
						&&Q\ar[dr]^u
						&&&C\ar[dr]^q
						\\ &C\ar[ur]^q\ar[dr]_q\rrtwocell\omit\omit{\beta}
						&&Y
						&R\ar[ur]^{d_0}\ar[dr]_{d_1}\rrtwocell\omit\omit{\xi}
						&&Q\ar[dr]^u
						\\ R\ar[ur]^{d_0}\ar[dr]_{d_1}\rrtwocell\omit\omit{\xi}
						&&Q\ar[ur]_v \ar@{}[rrr]|-{}="a" \save "a"*{=}\restore
						&&&C\ar[ur]^q\ar[dr]_q\rrtwocell\omit\omit{\beta}
						&&Y
						\\ &C\ar[ur]_q
						&&&&&Q\ar[ur]_v
					}\end{xym}
					there exists a unique $\alpha\col u\Ra v$ such that $\beta = \alpha q$.
			\end{enumerate}
		\end{df}

	Actually condition \ref{condquottrois} follows from the others and can be removed.  The quotient is thus the codescent object \cite{Lack2002a} of the diagram of the equivalence 2-relation.

If $A\overset{f}\ra B$ is an arrow in a $\Gpd$-category $\C$, we construct its \emph{kernel 2-relation} by taking the pullback of $f$ with itself:
	\begin{xym}\xymatrix@=40pt{
		\rf\ar[r]^-{d_0}\ar[d]_{d_1}\drtwocell\omit\omit{\kappa}
		& A\ar[d]^f
		\\ A\ar[r]_f
		&B.
	}\end{xym}
Then the universal property of the pullback induces arrows $r$, $t$, $s$ and the associated 2-arrows, as well as 2-arrows $\upsilon_0$, $\upsilon_1$, $\alpha$ and $\gamma$ satisfying the conditions of the definition of equivalence 2-relation. 

In a $\Gpd$-category with all kernel 2-relations and quotients of equivalence 2-relations, the quotient is left adjoint to the kernel 2-relation and they form a kernel-quotient system.  In $\Gpd$, every equivalence 2-relation is canonically the kernel 2-relation of its quotient, and every surjective functor is canonically the quotient of its kernel 2-relation (see \cite{Dupont2008b}).

\backmatter

\cleardoublepage
\printindex
\cleardoublepage

\end{document}